\documentclass[10pt,twoside,a4paper]{ClasseThese}

\usepackage[english]{babel}
\usepackage[latin1]{inputenc}
\usepackage{amsfonts,amsthm,amssymb,amsmath,latexsym,wasysym,stmaryrd,mathrsfs}
\usepackage{multirow,rotating,subfigure,picins,graphicx,bigstrut,lscape,appendix}
\usepackage{hyperref}
\usepackage{makeidx}
\usepackage[centerlast]{caption}
\usepackage[usenames]{color}
\usepackage[all,dvips,color]{xy}
\usepackage{fancyhdr}
\usepackage{indentfirst}

\graphicspath{{Figures/}}

\newtheorem {theo}{Theorem}[section]
\newtheorem {theo*}{Theorem}[]
\newtheorem {lemme}[theo]{Lemma}
\newtheorem {prop}[theo]{Proposition}
\newtheorem {cor}[theo]{Corollary}

\theoremstyle{definition}
\newtheorem {defi}[theo]{Definition}
\newtheorem {nota}[theo]{Notation}

\theoremstyle{remark}
\newtheorem {remarque}[theo]{Remark}

\def\e{\varepsilon}
\def\h{\mathfrak{h}}
\def\C{{\mathcal C}}
\def\D{{\mathcal D}}
\def\F{{\mathcal F}}
\def\FF{{\mathbb F}}

\def\HH{\mathcal{H}}
\def\Hom{{\textnormal{Hom}}}
\def\I{{\mathcal I}}

\def\K{{\mathcal K}}
\def\M{{\mathcal M}}
\def\N{{\mathbb N}}
\def\O{{\mathbb O}}
\def\P{\mathfrak{P}}
\def\R{{\mathbb R}}
\def\SS{{\mathfrak S}}
\def\Sei{{\textnormal{Sei}}}
\def\T{{\mathcal T}}
\def\X{{\mathbb X}}
\def\Z{{\mathbb Z}}
\def\2Z{{\fract{\Z}/{2\Z}}}

\def\gr{{gr}}
\def\ib{{\overline{i}}}
\def\kb{{\overline{k}}}
\def\p{\partial}

\def\Cone{{\textnormal{Cone}}}
\def\Im{{\textnormal{Im}}}
\def\Ker{{\textnormal{Ker}}}
\def\Id{{\textnormal{Id}}}
\def\Int{{\textnormal{Int}}}
\def\Hex{{\textnormal{Hex}^\circ}}
\def\Pent{{\textnormal{Pent}^\circ}}
\def\Pol{{\textnormal{Pol}}}
\def\PPol{{\P\textnormal{ol}}}
\def\Rect{{\textnormal{Rect}^\circ}}
\def\Sn{{\widetilde{\SS}_n}}
\def\Spike{{\textnormal{Spike}}}
\def\deg{{\textnormal{deg}}}
\def\ie{{\it i.e. }}
\def\qdim{{\textnormal{{\it q}-dim}}}
\def\rk{{\textnormal{rk}}}

\def\siecle#1{\textsc{\romannumeral #1}\textsuperscript{e}~siècle}

\newcommand{\noi}{\noindent}
\newcommand{\disp}{\displaystyle}

\def\fract#1/#2{\hbox{\leavevmode
  \kern.1em \raise .25ex \hbox{\the\scriptfont0 $#1$}\kern-.1em }\big/
  {\hbox{\kern-.15em \lower .5ex \hbox{\the\scriptfont0 $#2$}} }}

\newcommand{\dessin}[2]{
  \vcenter{\hbox{\includegraphics[height=#1]{#2}}}}
\newcommand{\dessinH}[2]{
  \vcenter{\hbox{\includegraphics[width=#1]{#2}}}}

\newcommand{\func}[3]{
  #1\colon #2\longrightarrow#3}
\newcommand{\plong}[3]{
  #1\colon #2\hookrightarrow#3}

\newcommand{\Gsig}[1]{
   {\sigma_{#1}}}

\title{\LARGE Généralisation de l'homologie d'Heegaard-Floer\\aux entrelacs singuliers\\\&\\ Raffinement de l'homologie de Khovanov\\aux entrelacs restreints}
\author{\normalsize Benjamin \textsc{Audoux}}
\date{\normalsize\today}

\renewcommand{\chaptermark}[1]{\markboth{\it{#1}}{}}

\makeindex

\begin{document}

\thispagestyle{empty}
\setcounter{page}{-1}

\begin{center}
{\Large \sffamily \bfseries DOCTORAT DE L'UNIVERSIT\'E DE TOULOUSE\\
délivré par l'Université Toulouse III -- Paul Sabatier\\
\vspace{5pt} en MATH\'EMATIQUES PURES\\}
\vspace{10pt}
{\large \sffamily \bfseries par\\}
\vspace{20pt}
{\huge \sffamily \bfseries Benjamin Audoux \\}
\vspace{20pt}
{\large \sffamily \bfseries intitulé\\}
\vspace{20pt}
{\Large \sffamily \bfseries  Généralisation de l'homologie d'Heegaard-Floer\\aux entrelacs singuliers\\\&\\ Raffinement de l'homologie de Khovanov\\aux entrelacs restreints\\}

\vspace{75pt}

\noindent {\large \sffamily \bfseries soutenu le 5 décembre 2007}
{\large \sffamily \bfseries devant le jury composé de\\}

\vspace{10pt}
{\sffamily \large
\begin{tabular}{llr}

 {\bfseries C. Blanchet} & Université Paris VII  & Rapporteur\\
 {\bfseries V. Colin} & Université de Nantes  & Examinateur\\
 {\bfseries T. Fiedler} & Université Toulouse III  & Directeur\\
 {\bfseries S. Orevkov} & Université Toulouse III  & Examinateur\\
 {\bfseries P. Ozsv\'ath} & Columbia University  & Rapporteur\\
 {\bfseries V. Vershinin} & Université Montpellier II  & Président\\
\end{tabular}
}
\vfill
\vspace{10pt}
{\sffamily \textnormal
Institut de Mathématiques de Toulouse, UMR 5580, UFR MIG\\
Laboratoire \'Emile Picard,\\
Université Paul Sabatier 31062 TOULOUSE Cédex 9
}
\end{center}

\newpage
\thispagestyle{empty}
$\ $
% \newpage
% \thispagestyle{empty}
% $\ $
% \vspace{.5cm}

% $$
% \dessin{15cm}{Chevet}
% $$

% \vspace{1.5cm}

% \hfill {\Large Si se taire...}

% \vfill
% \noi Gilles {\sc Audoux}\\
% Abbaye d'Auberive

% \newpage
% \thispagestyle{empty}
% $\ $
\newpage

% Résumé
{\Large Résumé~:}\\

La catégorification d'un invariant polynomial d'entrelacs $I$ est un invariant de type homologique dont la caractéristique d'Euler graduée est égale à $I$.
On pourra citer la catégorification originelle du polynôme de Jones par M. Khovanov ou celle du polynôme d'Alexander par P. Ozsv\'ath et Z. Szab\'o.
Outre leur capacité accrue à distinguer les n\oe uds, ces nouveaux invariants de type homologique semblent drainer beaucoup d'informations d'ordre géométrique.\\
D'autre part, suite aux travaux de I. Vassiliev dans les années 90, un invariant polynomial d'entrelacs peut être étudié à l'aune de certaines propriétés, dites de type fini, de son extension naturelle aux entrelacs singuliers, c'est-à-dire aux entrelacs possédant un nombre fini de points doubles transverses.\\
La {\bf première partie} de cette thèse s'intéresse aux liens éventuels entre ces deux procédés, dans le cas particulier du polynôme d'Alexander. 
Dans cette optique, nous donnons d'abord une description des entrelacs singuliers par diagrammes en grilles.
Nous l'utilisons ensuite pour généraliser l'homologie de Ozsv\'ath et Szab\'o aux entrelacs singuliers.
Outre la cohérence de sa définition, nous montrons que cet invariant devient acyclique sous certaines conditions annulant naturellement sa caractéristique d'Euler. Ce travail s'insère dans un programme plus vaste de catégorification des théories de Vassiliev.\\

Dans une {\bf seconde partie}, nous nous proposons de raffiner l'homologie de Khovanov aux entrelacs restreints.
Ces derniers correspondent aux diagrammes d'entrelacs quotientés par un nombre restreint de mouvements de Reidemeister.
Les tresses fermées apparaissent notamment comme sous-ensemble de ces entrelacs restreints.
Un tel raffinement de l'homologie de Khovanov offre donc un nouvel outil pour une étude plus ciblée des n\oe uds et de leurs déformations.

\vspace{1cm}

{\Large Abstract~:}\\

A categorification of a polynomial link invariant is an homological invariant which contains the polynomial one as its graded Euler characteristic.
This field has been initiated by Khovanov categorification of the Jones polynomial. Later, P. Ozsv\'ath and Z. Szab\'o gave a categorification of Alexander polynomial.
Besides their increased abilities for distinguishing knots, this new invariants seem to carry many geometrical informations.\\
On the other hand, Vassiliev works gives another way to study link invariant, by generalizing them to singular links \ie links with a finite number of rigid transverse double points.\\
The {\bf first part} of this thesis deals with a possible relation between these two approaches in the case of the Alexander polynomial.
To this purpose, we extend grid presentation for links to singular links.
Then we use it to generalize Ozsv\'ath and Szab\'o invariant to singular links.
Besides the consistency of its definition, we prove that this invariant is acyclic under some conditions which naturally make its Euler characteristic vanish.
This work can be considered as a first step toward a categorification of Vassiliev theory.\\

In a {\bf second part}, we give a refinement of Khovanov homology to restricted links.
Restricted links are link diagrams up to a restricted set of Reidemeister moves.
In particular, closed braids can be seen as a subset of them.
Such a refinement give then a new tool for studying knots and their deformations.

%%% Local Variables: 
%%% mode: latex
%%% TeX-master: "These"
%%% End: 

\tableofcontents

% Remerciements
\chapter*{Remerciements}
\label{chap:Thanks}
\addcontentsline{toc}{part}{Remerciements}

Une page peut-elle mesurer le champ de mes reconnaissances ?
La tentation est grande de les masquer naïvement par des drapés de vaines logorrhées dont la louable dithyrambie ne ferait qu'escamoter, maladroitement, une sincérité pourtant tangible, bien que pudique.\\
Je ne le ferais pas.\\

Puisque la recherche est un héritage, je tiens néanmoins à remercier Thomas Fiedler, pour son rôle de directeur de thèse~;
Michel Boileau, pour la direction de mon mémoire de DEA~;
Christian Blanchet et Peter Ozsv\'ath, pour avoir rapporté ma thèse en dépit de leurs calendriers~;
ainsi que Vincent Colin, Stepan Orevkov et Vladimir Vershinin, pour avoir accepté de former mon jury.\\

Mais c'est également des équipes, ouvertes.
En cela, je sais gré au GDR Tresses de m'avoir, deux fois par an, ouvert ses portes pluridisciplinaires~;
à l'IMM de Montpellier d'avoir accueilli un pauvre toulousain au sein de son groupe de travail AlgHom/GQ~;
à toutes les éditions de la Llagonne de m'avoir initié aux couleurs de la cuisine catalane, et de leurs mathématiques aussi~;
ainsi qu'à tous les laboratoires qui m'ont invité, de leur accueil, de leur écoute et de leurs conseils.\\

Le protocole n'est cependant pas tout.
Et que ce soit pour les discussions professionellement enrichissantes, parfois, ou pour les moments de détente, souvent, je serais bien ingrat si je ne citais pas tous mes compagnons de route~: Etienne Galais, Jérome Petit, Emmanuel Wagner, Anne-Laure Thiel, Anne, Cécile, Guitta, Iman, Eva, Julien, Landry, Alexandre, Kuntal, Guillaume, Mathieu, etc.
Sans oublier Vincent Florens qui, lui, l'a essarté, la route.\\

Moi, pour qui les voix de l'administration sont et resteront souvent impénétrables, je ne peux que rendre grâce à Agnès Requis, Yveline Pannabière, Jocelyne Picard et à leur compétence toujours souriante~;
et, étrangement, à Julien Roques, pour ses nombreux conseils et aides.\\

Si vous lisez ces lignes, c'est que j'ai pu les taper.
Que l'informatique n'en soit pas remercié mais plus certainement ses mercenaires, Maxime et Gachette, pour toutes les heures passées à installer/réparer/programmer/compiler~;
ainsi que Marc, pour son efficacité et sa discrétion en \LaTeX.\\

Nous parlions d'héritage.
Pourrais-je taire ma gratitude envers mes parents, qui, tout au long des trentes dernières années, m'ont toujours soutenu dans mes choix.
Etait-ce bien raisonable ?
Envers mes grands-parents, et notamment mon grand-père paternel, pour avoir partagé sa passion pour les sciences.
Envers mon frère, oui, car on se ressemble finalement plus qu'il n'y parait.
Et envers ma tante, pour être présente de chaque côté de l'Atlantique.\\

Ne nous mentons, toutefois, pas.
De trois années de thèse, le pot est, sans aucun doute, l'élément qui reste le plus longtemps dans les mémoires.
A ce titre, Alain \& Gwenola, pour les desserts~;
Nicolas \& Céline, pour les cakes~;
Brigitte, pour les tartes et le vin~;
et enfin Romain, pour sa journée du 1er décembre, chacun est passible d'une part sensible du mérite.\\

Malheureusement, mes craintes premières se concrétisent.
Comment vous remercier tous, vous dont le prénom commence par A, Z, E, R, T, Y, U, I, O, P, Q, S, D, F, G, H, J, K, L, M, W, X, C, V, B ou N, et qui saurez, en lisant ces mots, tolérer mon indolence, comprenant que je préfère que vous m'en vouliez tous, un peu, plutôt que certains, beaucoup.\\

Enfin, pour donner l'illusion d'une conclusion, je témoignerai sobrement de ma reconnaissance envers Séb \& Solenn, Bianca et Céline, notamment pour leur décence de ne pas m'obliger à dire pourquoi.\\

Mais l'illusion s'évanouit vite et l'horizon, déjà, s'éclaire de mes silences.
Et tous ces "entre autre", implicites à chaque phrase, resteront les garants diaphanes du poids de ma gratitude.

%%% Local Variables: 
%%% mode: latex
%%% TeX-master: "These"
%%% End: 

% Introduction
% \chapter*{Préface historique}
% \label{chap:PreIntro}
% \addcontentsline{toc}{part}{Introduction}

% Selon toute vraisemblance, l'origine du boulghour se fondrait dans les racines de l'Histoire de l'Homme. Toutefois, quelques références bibliques ne font débuter son histoire écrite qu'autour du bassin méditerranéen. A son instar, si les premiers entrelacs noués par l'homme se perdent dans la deliquescence des fibres végétales, les marins de la {\it mare nostrum} ont su leur promettre un avenir éternel. Il faudra néanmoins attendre le \siecle{19} pour que les n\oe uds deviennent une véritable discipline scientifique.\\

% Gauss, Lord Kelvin, Tait, Reidemeister, Alexander, Jones, Khovanov, Ozsvath et Szabo.

\chapter*{Introduction}
\label{chap:Intro}
\addcontentsline{toc}{part}{Introductions}
\addcontentsline{toc}{section}{Introduction (français)}

Dégageons-nous de leur origine marine et donnons aux n\oe uds une définition mathématique\footnote[1]{Pour un exposé plus approfondi, on pourra se référer, par exemple, à \cite{Burde}.}. Pour cela, on considère les plongements lisses $\plong{\nu}{\bigsqcup^\ell S^1}{\R^3}$ de $\ell\in\N^*$ cercles orientés disjoints dans l'espace ambiant. Deux plongements $\nu_1$ et $\nu_2$ sont dits \emph{isotopes ambiants}\index{isotopy!ambient} s'il existe une famille continue d'homéomorphismes $(\func{h_t}{\R^3}{\R^3})_{t\in[0,1]}$ tels que $h_0\equiv \Id$ et $\nu_2\equiv \nu_1\circ h_1$. Un \emph{entrelacs} (\emph{link})\footnote[2]{Les traductions anglaises des termes importants sont données entre parenthèses.}\index{link} est une classe de tels plongements à isotopie ambiante près. Lorsque $\ell$ vaut un, on parle alors de \emph{n\oe ud} (\emph{knot})\index{knot}. La classe d'un plongement planaire, \ie dont l'image est contenue dans un plan, est appelée \emph{n\oe ud} ou \emph{entrelacs trivial} (\emph{trivial knot} or \emph{link})\index{trivial!knot}\index{trivial!link}. On le note $U_\ell$.\\

Cette définition des entrelacs est directement inspirée de la perception sensitive que nous en avons. Il existe cependant de nombreuses descriptions alternatives. Il s'agit, le plus souvent, d'objets combinatoires que l'on considère à certaines opérations élémentaires près. Pour définir un invariant d'entrelacs, il suffit alors de le faire pour chacun de ces objets, puis de vérifier que l'invariance par chacune des opérations élémentaires.\\
Les \emph{diagrammes d'entrelacs} (\emph{diagrams})\index{diagram} demeurent sans doute l'approche de ce type la plus naturelle et la plus usitée. Il s'agit, pour un entrelacs donné, de la projection sur un plan $P\subset\R^3$ d'un de ses représentants en position générale par rapport à $P$. Par position générale, on entend un représentant tel que les singularités de sa projection sur $P$ se limitent à un nombre fini de points doubles transverses. Les entrelacs ne possédant pas de tels représentants sont dits sauvages et ne seront pas étudiés dans cette thèse. Chaque point double est appelé \emph{croisement} (\emph{crossing})\index{crossing}. En de tels points, les ombres de deux brins de l'entrelacs se rencontrent. Un choix d'orientation sur l'orthogonal de $P$ permet de déterminer leurs hauteurs relatives par rapport au plan $P$. Pour les différencier, on interrompt légèrement le tracé du brin passant en-dessous.
$$
\dessin{5cm}{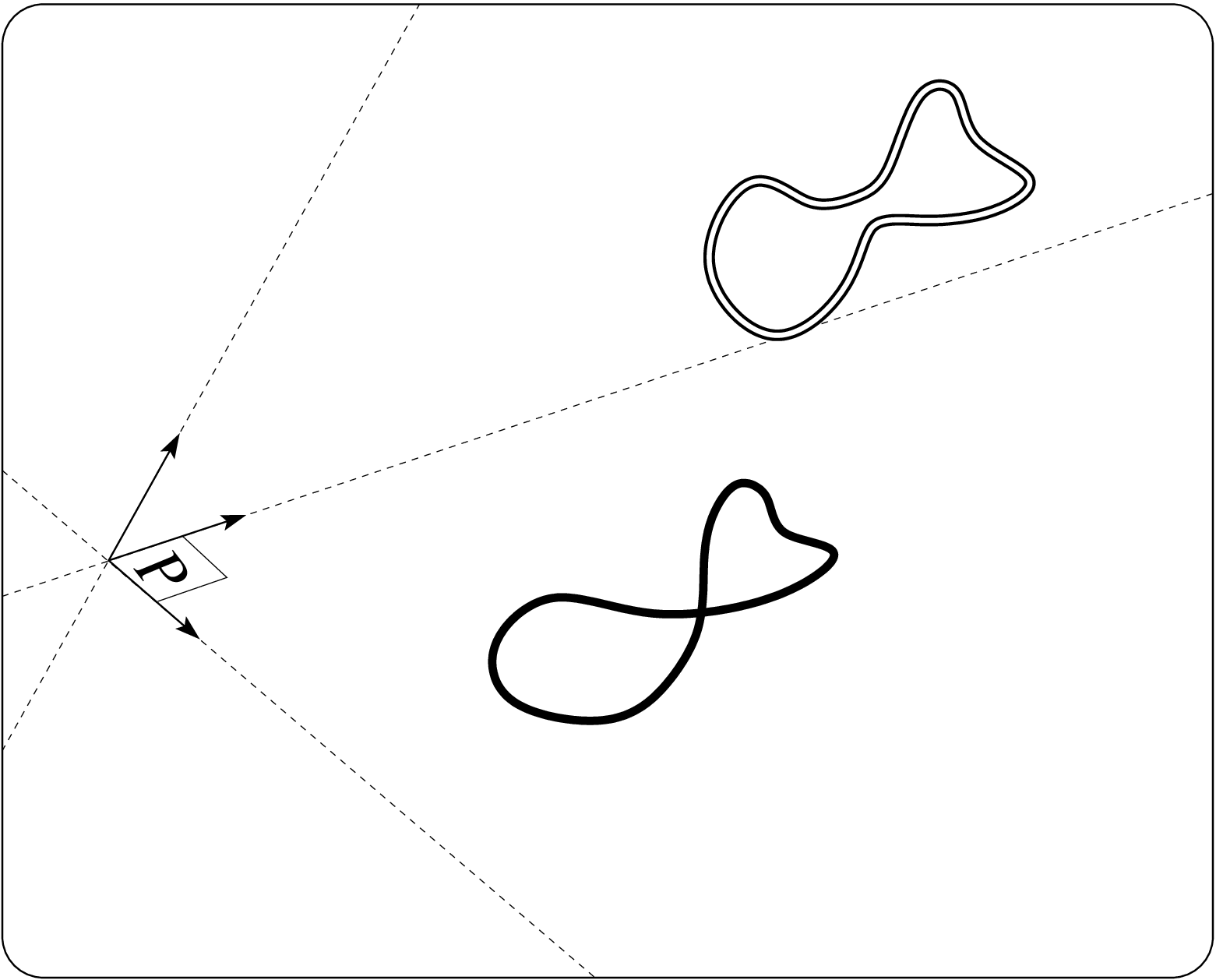} \hspace{.5cm} \leadsto \hspace{.4cm} \dessin{1.4cm}{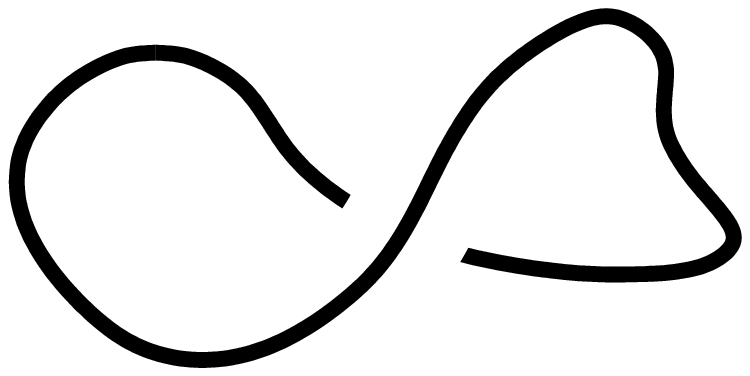}
$$
Bien entendu, un même entrelacs est représenté par plusieurs diagrammes. Notamment, il est clair qu'une déformation d'un diagramme n'affectant pas ses croisements préserve l'entrelacs sous-jacent. Il en est de même pour les trois mouvements suivants, dits \emph{mouvements de Reidemeister} (\emph{Reidemeister moves})\index{moves!Reidemeister}:
\begin{eqnarray}
\begin{array}{c}
  \dessinH{1cm}{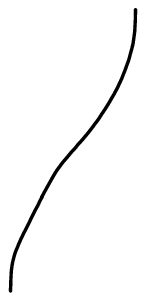} \leftrightarrow \ \dessinH{1cm}{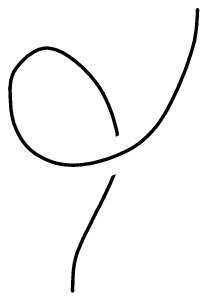}\\[.5cm]
\textrm{type I}
\end{array}
\hspace{1.2cm}
\begin{array}{c}
  \dessinH{1cm}{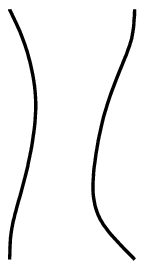} \leftrightarrow \dessinH{1cm}{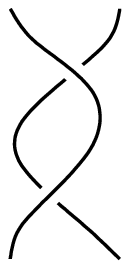}\\[.5cm]
\textrm{type II}
\end{array}
\hspace{1.2cm}
\begin{array}{c}
  \dessinH{1.4cm}{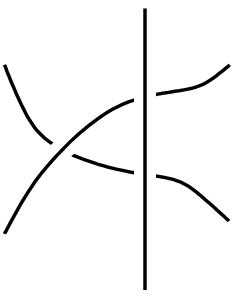} \leftrightarrow \dessinH{1.4cm}{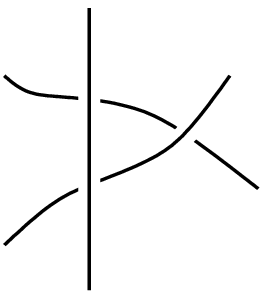}\\[.5cm]
\textrm{type III}
\end{array}.
\label{eq:ReidMoves}
\end{eqnarray}

Réciproquement, Kurt Reidemeister a démontré que deux diagrammes représentant le même entrelacs sont toujours reliés par une suite finie de mouvements de Reidemeister \cite{Reidemeister}.\\

La notion d'entrelacs revêt donc désormais un aspect plus combinatoire. En particulier, les croisements d'un diagramme offrent, pour l'étude des entrelacs, un nouvel outil dont l'apparente simplicité n'est qu'un masque de prélat vénitien. Selon l'orientation des brins, on peut distinguer deux types de croisement que l'on nomme, par convention, \emph{positif} (\emph{positive})\index{crossing!positive} et \emph{négatif} (\emph{negative})\index{crossing!negative} selon le modèle suivant:
$$
\begin{array}{c}
  \dessin{.9cm}{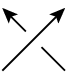}\\[.5cm]
\textrm{croisement positif}
\end{array}
\hspace{2cm}
\begin{array}{c}
  \dessin{.9cm}{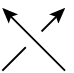}\\[.5cm]
\textrm{croisement negatif}
\end{array}.
$$
On désigne par \emph{inversion de croisement} (\emph{switch})\index{crossing!switch}\index{switch|see{crossing}} l'opération qui change le type d'un croisement. Cela correspond à laisser l'entrelacs se traverser lui-même en un point. Pour tout diagramme $D$, on appelle \emph{nombre d'entortillement} (\emph{writhe})\index{writhe} le nombre total $w(D)$ de croisements positifs diminué du nombre total de croisements négatifs. Si ce nombre est préservé par les mouvements de Reidemeister de type II et III, il varie par contre de un à chaque mouvement de type I. Il n'est donc pas canoniquement associé à un entrelacs mais à chacun de ses diagrammes. Il se révèle toutefois d'une importance capitale dans de nombreuses constructions.\\
Deux n\oe uds orientés $K_1$ et $K_2$ peuvent être fusionés. Il s'agit de la \emph{somme connexe} (\emph{connected sum})\index{connected sum}, notée $K_1\# K_2$. Les deux n\oe uds sont ouverts chacun en un point, puis les quatres extrémités, ainsi créées, recollées de manière à obtenir une unique composante connexe orientée de façon cohérente.
$$
\dessin{1.35cm}{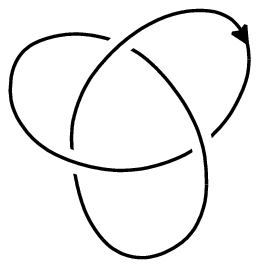}\ \# \ \dessin{1.8cm}{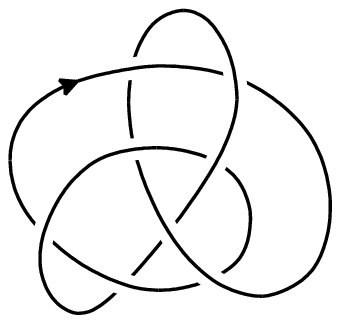}\ = \ \dessin{1.8cm}{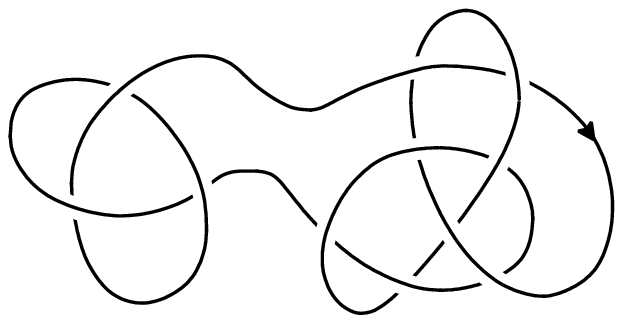}
$$
Le n\oe ud résultant ne dépend pas du choix de ces deux points. Dans le cas d'entrelacs à plusieurs composantes, il est nécessaire de préciser lesquelles sont ainsi fusionées.\\

Les deux invariants polynomiaux $\triangle$ et $V$, appelés respectivement \emph{polynôme d'Alexander} et \emph{polynôme de Jones} (\emph{Alexander} and \emph{Jones polynomials}), ont certainement été les avancées les plus marquantes en théorie des n\oe uds. Si leurs définitions originelles sont très géométriques et algébriques, on leur connait désormais des descriptions purement combinatoires basées sur les diagrammes. Intéressons-nous plus particulièrement au polynôme de Jones \cite{StateJones}.\\
Soit $D$ un diagramme d'entrelacs avec $k$ croisements. Chacun de ses croisements peut être lissé de façon à le faire disparaitre. Il y a pour cela deux choix possibles que, par convention, on nomme \emph{$A$} et \emph{$A^{-1}$--lissage} (\emph{$A$}\index{smoothing!Asmoothing@$A$--smoothing} or \emph{$A^{-1}$--smoothing})\index{smoothing!Bsmoothing@$A^{-1}$--smoothing} selon le modèle suivant:
$$
\xymatrix@!0 @R=.8cm @C=2.5cm {
& **[r] \dessin{.9cm}{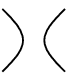} \hspace{.5cm} A\textrm{--lissage}\\
\dessin{1.15cm}{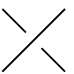} \ar[ur]!L \ar[dr]!L & \\
& **[r] \dessin{.9cm}{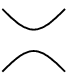} \hspace{.5cm} A^{-1}\textrm{--lissage}\\
}.
$$
On dit également que le croisement est \emph{$A$} ou \emph{$A^{-1}$--lissé} (\emph{$A$} or \emph{$A^{-1}$--smoothed}). Un seul de ces lissages permet d'induire canoniquement une orientation sur le diagramme lissé. On l'appelle \emph{lissage de Seifert} (\emph{Seifert smoothing})\index{smoothing!Seifert}. En opérant un lissage de Seifert sur tous les croisements d'un diagramme, on obtient canoniquement un nouveau diagramme sans croisement, appelé \emph{état de Seifert} (\emph{Seifert state})\index{smoothing!Seifert state}\index{Seifert state|see{smoothing}}. Il n'est toutefois pas isolé. Il y a, au total, $2^k$ façons de lisser entièrement $D$. Puisqu'il ne possède plus de croisement, chaque lissage complet $s$ correspond à un jeu de cercles non orientés plongés dans le plan. On note $d(s)$ le nombre de cercles et $\sigma (s)$ le nombre de croisements $A$--lissés diminué du nombre de croisements $A^{-1}$--lissés. On peut alors définir le polynôme de Jones par

\begin{eqnarray}
  V(D):=(-A)^{-3w(D)}\sum_{\substack{s\textrm{ lissage}\\[.1cm] \textrm{complet de }D}} A^{\sigma(s)}(-A^2-A^{-2})^{d(s)}.
\label{eq:Jones} 
\end{eqnarray}
\index{polynomial!Jones}
\index{V@$V$|see{polynomial, Jones}}
De cette définition, l'invariance par tous les mouvements de Reidemeister dérive naturellement. Le polynôme de Jones est également caractérisé par sa \emph{relation d'écheveau} (\emph{skein relation})\index{skein relation}, c'est-à-dire par sa valeur sur le n\oe ud trivial et par une relation linéaire entre ses valeurs sur trois diagrammes qui ne diffèrent que dans un disque contenant alternativement un croisement positif, un croisement négatif et leur lissage de Seifert. En effet, tout entrelacs pouvant être dénoué, \ie rendu trivial, par un nombre fini d'inversions de croisement, une telle relation permet un calcul récursif de $V$. Le polynôme de Jones satisfait
\begin{gather*}
  A^4V(\dessin{.5cm}{Pos})-A^{-4}V(\dessin{.5cm}{Pos})=(A^2-A^{-2})V(\dessin{.5cm}{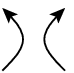});\\[.3cm]
  V(U_\ell)=(-A^2-A^{-2})^\ell.
\end{gather*}
Dans \cite{Formal}, Louis Kauffman donne une description de ce type pour le polynôme d'Alexander. Nous n'en ferons toutefois pas usage dans cette thèse. Nous nous contenterons donc de donner sa relation d'écheveau:
\begin{gather*}
  \triangle(\dessin{.5cm}{Pos})-\triangle(\dessin{.5cm}{Pos})=(t^\frac{1}{2}-t^{-\frac{1}{2}})\triangle(\dessin{.5cm}{Seif});\\[.3cm]
  \triangle(U_1)=1 \hspace{1.2cm} \triangle(U_\ell)=0, \forall \ell \geq 2.
\end{gather*}
\index{polynomial!Alexander}
\index{Delta@$\triangle$|see{polynomial, Alexander}}

Dans un \siecle{20} moribond, le polynôme de Jones s'est offert une seconde noce. En interprétant ses coefficients comme les dimensions de modules libres dont les générateurs sont donnés par les indices de sommation d'une reformulation de (\ref{eq:Jones}), Mikhail Khovanov l'a élevé à l'éden fonctoriel \cite{Khovanov}. Ces modules peuvent en effet être agencés en un complexe de chaîne gradué sur lequel une différentielle vient naturellement s'apposer. Bien que contenant strictement plus d'informations, les groupes d'homologies associés, appelés \emph{homologie de Khovanov} (\emph{Khovanov homology})\index{homology!Khovanov} sont alors eux-même invariants par tous les mouvements de Reidemeister. On dit qu'ils \emph{catégorifient} (\emph{categorify}) \index{categorification} le polynôme de Jones dans la mesure où ce dernier est récupéré comme la caractéristique d'Euler graduée de cette homologie.\\

Comme l'a souligné Oleg Viro \cite{Viro}, une fois les modules mis en place, la différentielle de Khovanov est imposée au vu des graduations à respecter. Le polynôme d'Alexander fut plus long à mûrir. Ce siècle avait deux ans lorsque l'homologie d'Heegaard-Floer perça sous les travaux de Peter Ozsv{\'a}th et Zolt{\'a}n Szab{\'o} \cite{OS1}. Il s'agit d'une homologie associée à toute variété fermée de dimension $3$ à l'aide d'une technologie analytique et géométrique fine\index{homology!Heegaard--Floer}. Un an plus tard, P. Ozsv{\'a}th, Z. Szab{\'o}  et, indépendamment, Jacob Rasmussen, montrèrent que tout n\oe ud plongé dans une telle variété induit une filtration, dite d'Alexander, sur son homologie d'Heegard-Floer \cite{OS2} \cite{Rasmussen}. L'homologie graduée associée, appelée \emph{homologie d'Heegaard-Floer pour les noeuds} (\emph{knot Floer homology}), donne alors un invariant du n\oe ud catégorifiant le polynôme d'Alexander. La construction fut ensuite étendue au cas des entrelacs \cite{OS3}. Si cette approche semble être une lucarne particulièrement cristalline vis-à-vis des nombreuses propriétés géométriques de cet invariant, elle n'est, malheureusement, que peu adaptée aux calculs.\\
Au prix d'une plus grande opacité, il existe cependant une autre approche purement combinatoire, basée sur une remarque de Sucharit Sarkar \cite{MOS} et utilisant une description des entrelacs par diagrammes en grille développée par Peter Cromwell \cite{Cromwell} puis Ivan Dynnikov \cite{Dynnikov}. C'est cette approche que nous développerons et utiliserons dans cette thèse.\\

Il est possible de généraliser légèrement la notion d'entrelacs. Pour cela, on ne considére plus seulement les plongements mais également les immersions $\func{\nu}{\bigsqcup^\ell S^1}{\R^3}$ possédant un nombre fini de points doubles transverses, \ie de points doubles tels que les deux vecteurs tangents ne sont pas colinéaires. On demande à toute isotopie ambiante de préserver à chaque instant cette condition de transversalité\index{isotopy!ambiant}. On parle alors de \emph{n\oe uds} ou d'\emph{entrelacs singuliers} (\emph{singular knots}\index{knot!singular} or \emph{links})\index{link!singular}. La notion de diagramme est conservée, les points doubles étant représentés sans interruption du tracé. Pour garder une correspondance bijective avec les entrelacs singuliers, il est toutefois nécessaire d'ajouter deux nouveaux types de mouvements de Reidemeister \cite{Kauffman}:
$$
\begin{array}{c}
  \dessinH{1cm}{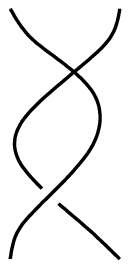} \leftrightarrow \dessinH{1cm}{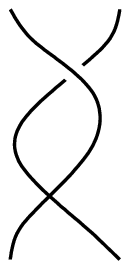}\\[.5cm]
\textrm{type IV}
\end{array}
\hspace{1.2cm}
\begin{array}{c}
  \dessinH{1.4cm}{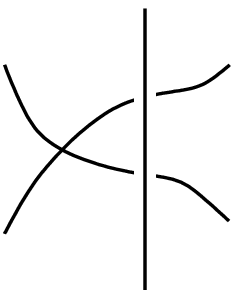} \leftrightarrow \dessinH{1.4cm}{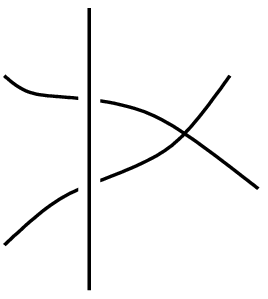}\\[.5cm]
\textrm{type V}
\end{array}
$$\index{moves!Singular Reidemeister}
La somme connexe\index{connected sum} $K_1{}_{p_1}\#_{p_2}$ de deux n\oe uds singuliers $K_1$ et $K_2$ est encore définie, mais il est nécessaire de préciser en quels points $p_1\in K_1$ et $p_2\in K_2$ la fusion doit s'effectuer.\\ 

Pour désingulariser un point double d'un entrelacs orienté, il existe deux possibilités que le produit vectoriel des deux vecteurs tangents permet de distinguer de façon intrinsèque. Par convention, on les appelle \emph{$0$} et \emph{$1$--résolution} (\emph{$0$} and  \emph{$1$--resolution})\index{resolution!Aresolution@$0$--resolution}\index{resolution!Bresolution@$1$--resolution} selon le modèle suivant:
$$
\xymatrix@!0 @R=1.2cm @C=3.5cm {
& **[r] \dessin{1.6cm}{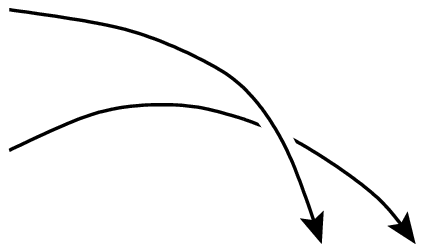} \hspace{.5cm} 0\textrm{--résolution}\\
\dessin{2.05cm}{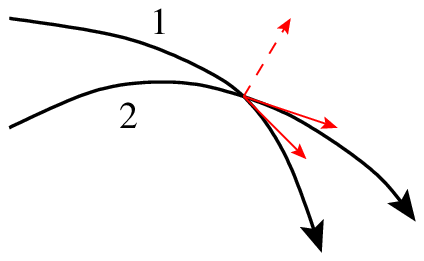} \ar[ur]!L \ar[dr]!L & \\
& **[r] \dessin{1.6cm}{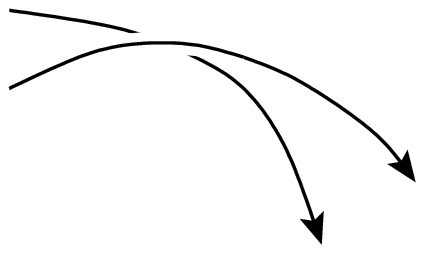} \hspace{.5cm} 1\textrm{--résolution}\\
}.
$$
On dit également que le point double est \emph{$0$} ou \emph{$1$--résolu} (\emph{$0$} or \emph{$1$--resolved}). Au niveau des diagrammes, ils correspondent respectivement à la création d'un croisement positif et négatif.\\
Tout invariant $I$ à valeurs réelles~--- ou plus généralement à valeurs dans un $\Z$--module~--- défini sur les entrelacs peut être naturellement prolongé aux entrelacs singuliers {\it via} la formule
$$
I(\dessin{.5cm}{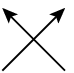}) := I(\dessin{.5cm}{Pos}) - I(\dessin{.5cm }{Neg}).
$$
Cette extension de $I$ permet ainsi de mesurer son comportement vis-à-vis d'une inversion de croisement.\\
Quel que soit l'invariant initial, on a, par construction:
\begin{eqnarray}
I(L{}_*\#_p\dessin{.4cm}{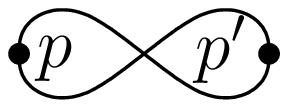}\hspace{-.05cm}{}_{p'}\#_* L')=\left(\ \dessin{1.1cm}{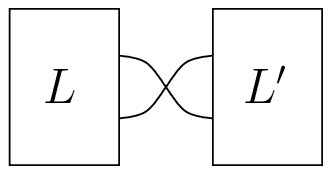}\right)=0
\label{eq:Trivial}
  \end{eqnarray}
où $L$ et $L$ sont deux entrelacs potentiellement singuliers.\\
Pour tout entier naturel $m$, on dit que $I$ est un \emph{invariant de Vassiliev d'ordre fini inférieur à m} (\emph{Vassiliev invariant of order $\leq m$})\index{Vassiliev invariants} s'il s'annule pour tout entrelacs singulier possédant au moins $m+1$ points doubles. On dit également qu'il est \emph{de type fini} (\emph{of finite type})\index{finite type invariants|see{Vassiliev invariants}}. Une large majorité des invariants d'entrelacs connus se décomposent en somme d'invariants de type finis. Joan Birman et Xiao-Song Lin \cite{Birman} ont notamment montré que les coefficients des polynômes d'Alexander et de Jones, après un changement de variable approprié, sont de type fini.\\
Dans les travaux de Maxim Kontsevitch \cite{Kontsevich} puis de Dror Bar-Natan \cite{BarNatanVass}, les invariants de Vassiliev restreints aux n\oe uds et d'un ordre fixé trouvent une interprètation comme dual d'un certain module de type fini. Conjecturalement, ils seraient pourtant denses parmi les invariants d'entrelacs et, pris dans leur ensemble, permettraient de distinguer tous les n\oe uds.\\

Un invariant d'entrelacs polynomial peut donc être catégorifié d'une part et généralisé aux entrelacs singuliers d'autre part. Il est naturel de s'interroger sur les liens éventuels entre ces deux procédés.
$$
\xymatrix@!0@R=1.1cm@C=3.5cm{& {\begin{array}{c} \textrm{Catégorification $\I$~:}\\[.1cm] I=\displaystyle{\sum_{i,j}} (-1)^i \rk \I_i^j q^j \end{array}} \ar[]!R;[rd]!UL& \\
  **[l] \textrm{Invariant polynomial } I \ar[]!UR;[ru]!L \ar[]!DR;[rd]!L & & **[r] ? \\
  & {\begin{array}{c}\textrm{Théorie de Vassiliev~:} \\[.1cm] I(\dessin{.5cm}{Double}):= I(\dessin{.5cm}{Pos}) - I(\dessin{.5cm }{Neg}) \end{array}} \ar[]!R;[ru]!DL & \\}
$$
Mais auparavant, il est nécessaire de généraliser ces nouveaux invariants de type homologique aux entrelacs singuliers. Malheureusement, l'ensemble des complexes de chaîne n'est pas naturellement muni d'une structure de $\Z$--module. Néanmoins, à tout morphisme entre deux complexes de chaîne, on peut en associer un troisième, appelé \emph{cône de recollement} (\emph{mapping cone}), dont la caractéristique d'Euler est obtenue comme différence des caractéritiques d'Euler des complexes de départ et d'arrivée. Le morphisme nul est cependant le seul à être défini sans équivoque entre deux complexes quelconques. Et si l'homologie $H$ associée au cône de recollement de cette application nulle, \ie  la somme des deux homologies intiales, est bien un invariant pour un noeud singulier, elle ne vérifie pas, en général, la relation
\begin{eqnarray}
H\left(\ \dessin{1.1cm}{Torsion}\right)\cong 0
\label{eq:Trivial2}  
\end{eqnarray}
liée à l'égalité (\ref{eq:Trivial}), relation qu'il serait pourtant raisonnable d'exiger. Cela impose donc, pour chaque invariant de type homologique, une construction au cas par cas d'un morphisme plus spécifique entre les complexes de chaîne associés à chacune des résolutions d'un point double. C'est l'approche préconisée par Nadya Shirokova dans \cite{Shiro}, qu'elle applique ensuite à l'homologie de Khovanov puis, en colla\-boration avec Ben Webster, à l'homologie de Khovanov-Rozansky \cite{ShiroWeb}. Cela se fait toutefois au prix d'une perte de généralité dans la construction, l'invariance par choix du représentant d'un entrelacs singulier et la relation (\ref{eq:Trivial2}) n'ayant alors plus rien d'automatique.\\

Dans {\bf la première partie} de cette thèse, nous proposons de définir une généralisation de l'homologie d'Heegaard-Floer aux entrelacs singuliers, que l'on notera $\widehat{HF}$ ou, dans sa forme la plus générale, $H^-$. Pour cela, nous étendrons d'abord la description par diagrammes en grille au cas singulier. Nous décrirons à cette occasion l'ensemble $\M$ des mouvements élémentaires sur les grilles singulières permet\-tant d'obtenir la correspondance bijective suivante:
$$
\{\textrm{Entrelacs singuliers}\} \ \longleftrightarrow \ \fract{\{\textrm{Grilles singulières}\}}/{\M}.
$$
A l'aide de ces grilles, nous définirons ensuite un \emph{morphisme d'inversion} (\emph{switch morphism})
$$
\func{f}{C^-(\dessin{.5cm}{Pos})}{C^-(\dessin{.5cm}{Neg})}
$$
où $C^-(L)$ représente un complexe de chaîne pour l'homologie d'Heegaard-Floer de l'entrelacs régulier $L$. Nous montrerons alors que l'homologie associée au cone de recollement de cette application est un invariant de l'entrelacs singulier sous-jacent. Par le biais d'un cube de résolution, cette construction sera ensuite étendue à des entrelacs possédant plusieurs points doubles. L'homologie dépend alors d'un choix d'orientation pour chacun des points doubles de l'entrelacs, \ie d'un choix d'orientation pour chaque plan engendré par les deux vecteurs tangents à l'entrelacs en un point double.\\
L'homologie d'Heegaard-Floer pour les entrelacs réguliers est munie de nombreuses filtrations dont les homologies graduées associées sont autant d'invariants pour les entrelacs. Il existe toutefois des relations fortes entre elles. Nous montrerons que toutes ces variantes ainsi que leurs relations sont conservées dans le cas singulier. Plusieurs choix arbitraires sont également faits durant la construction. Les différentes alternatives, leurs pertinences et leurs relations seront discutées. Par ailleurs, cette discussion sera l'occasion de donner au passage une preuve supplémentaire de la relation d'écheveau vérifiée par la caractéristique d'Euler de ces homologies.\\
Un programme en Ocaml a été implémenté et certains résultats sont donnés en annexe. Sur les quelques exemples calculés, la relation (\ref{eq:Trivial2}) a toujours été vérifiée. Cependant, nous ne prouverons que le cas particulier où l'entrelacs $L'$ est trivial. Cela correspond aux entrelacs possédant une boucle singulière.\\
La question de comportements ``de type fini'' reste enfin ouverte. Quelques pas seront toutefois faits dans cette direction en exhibant un inverse à homotopie près pour le morphisme d'inversion. Cet inverse ne respecte bien entendu pas la filtration d'Alexander mais cette défaillance reste minime et contrôlée dans la mesure où ce morphisme est filtré de degré $1$.\\
Dans \cite{OSSing}, P. Ozsv\'ath, Andr\'as Stipsicz and Z. Szab\'o donnent une autre généralisation $HFS$ de l'homologie d'Heegard-Floer aux entrelacs singuliers. Leur approche et leurs objectifs sont cependant assez éloignés des nôtres. Notamment, un point double, dans leur papier, doit être compris comme une arête épaissie au sens des travaux de \cite{Yamada}. Cela aboutit aux suites exactes suivantes:
$$
\xymatrix@!0@C=1.7cm@R=1.7cm{
HFS\left(\dessin{.5cm}{ARes}\right)\{1\} \ar[rr] && HFS\left(\dessin{.5cm}{Pos}\right) \ar[dl]\\
& \widetilde{HFS}\left(\dessin{.5cm}{Double}\right) \ar[ul]&
} \hspace{1.5cm}
\xymatrix@!0@C=1.5cm@R=1.5cm{
HFS\left(\dessin{.5cm}{Neg}\right) \ar[rr] && HFS\left(\dessin{.5cm}{ARes}\right)\{-1\} \ar[dl]\\
& \widetilde{HFS}\left(\dessin{.5cm}{Double}\right) \ar[ul]&
}.
$$
où $\widetilde{HFS}$ est une version modifée nécessitant un quotient supplémentaire. Dans notre construction, un point double est perçu comme la transcription d'une inversion de croisement et cela conduit à la suite exacte:
$$
\xymatrix@!0@C=1.7cm@R=1.7cm{
\widehat{HF}\left(\dessin{.5cm}{Pos}\right) \ar[rr] && \widehat{HF}\left(\dessin{.5cm}{Neg}\right) \ar[dl]\\
& \widehat{HF}\left(\dessin{.5cm}{Double}\right) \ar[ul]&
}.
$$

{\bf La seconde partie} traite de l'homologie de Khovanov et de son comportement vis-à-vis de certains mouvements de Reidemeister. Nous nous intéresserons donc à ces derniers et à leurs relations. Comme explicité précédemment en (\ref{eq:ReidMoves}), on distingue généralement trois types de mouvements en fonction de la singularité des courbes planes impliquée. Cependant, en tenant compte de l'orientation, les mouvements de types II et III peuvent être subdivisés en sous-catégories:
$$
\xymatrix@!0 @R=.9cm @C=2.5cm {
& \dessin{.6cm}{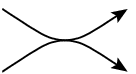}\\
\dessin{.75cm}{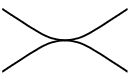} \ar[ur]!L \ar[dr]!L & \\
& \dessin{.6cm}{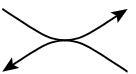}\\
} \hspace{2cm}
\xymatrix@!0 @R=.9cm @C=2.5cm {
& \dessin{.9cm}{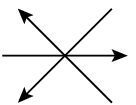}\\
\dessin{1.1cm}{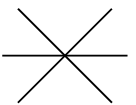} \ar[ur]!L \ar[dr]!L & \\
& \dessin{.9cm}{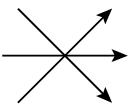}\\
}.
$$
D'autre part,comme le montre la figure \ref{fig:Trick} (\S\ref{par:relationReidMoves}), tous ces mouvements ne sont pas indépendants. On peut notamment s'intéresser aux classes d'équivalence induites par les mouvements de type II parmi les mouvements de type III. En étudiant le graphe de ces corrélations donné dans la figure \ref{fig:TypeIII} (\S\ref{par:relationReidMoves}), on remarque:
\begin{itemize}
\item[-] que les mouvements correspondant à la singularité $\dessin{.5cm}{sing1bis}$ séparent les mouvements de type III en deux classes dont les éléments de l'une sont les images miroir des éléments de l'autre;
\item[-] que les mouvements correspondant à la singularité $\dessin{.5cm}{sing1}$ relient tous les mouvements correspondant à la singularité $\dessin{.75cm}{Sing2bis}$ mais isolent les deux mouvements correspondant à $\dessin{.75cm}{sing2}$.
\end{itemize}
Les mouvements correspondant aux singularités $\dessin{.5cm}{sing1}$ et $\dessin{.75cm}{Sing2bis}$ forment donc une certaine unité. Ils correspondent exactement aux mouvements locaux intervenant lors d'isotopies de diagrammes de tresses. Cela justifie leur dénomination anglophone \emph{braid-like}.\\
Nous nous intéresserons aux diagrammes d'entrelacs considérés aux mouvements braid-like uniquement. Cela définit un nouvel objet généralisant la notion de tresse fermée dans un tore solide. Nous montrerons qu'une telle restriction sur les isotopies permet de raffiner le polynôme de Jones en scindant son indéterminé en deux. Nous montrerons également que ce raffinement se propage à l'homologie de Khovanov dont on définira une version trigraduée. Tous ces raffinements peuvent être vus comme des généralisations des constructions faites, notamment dans le cas des tresses fermées, par J\'ozef Przytycki et ses collaborateurs \cite{Hoste}\cite{Asaeda}.\\
De grands progrès ont été faits dans la compréhension des n\oe uds en utilisant le lien fort existant entre n\oe uds et tresses fermées. Par notre généralisation, nous offrons un nouvel outil fourni avec un équipement d'invariants de type homologique. Cela permet de disséquer certains mécanismes internes d'invariance de l'homologie de Khovanov lors des mouvements non braid-like.\\
A noter que la technologie développée dans le cas braid-like s'adapte également aux isotopies de diagrammes ne faisant intervenir que des singularités de type $\dessin{.5cm}{sing1}$ et $\dessin{.75cm}{sing2}$. Là encore, cela permet de cibler l'étude d'un entrelacs en restreignant sa mobilité.

%%% Local Variables: 
%%% mode: latex
%%% TeX-master: "These"
%%% End: 

% Introduction anglaise
\chapter*{Short introduction}
\label{chap:ShortIntro}
\addcontentsline{toc}{section}{Short introduction (english)}

A \emph{link} is defined as a smooth embedding in $\R^3$ of $\ell$ disjoint copies of the oriented circle, up to ambient isotopies. If $\ell=1$, we say it is a \emph{knot}. A link can be described by a generic projection on a $\R^2$-plane with the data at each crossing of which strand is overpassing the second. Such a projection is called a \emph{diagram}. To recover the notion of links, one has to consider them up to the Reidemeister moves shown in Figure \ref{fig:EnReidMoves}.\\

\begin{figure}[!h]
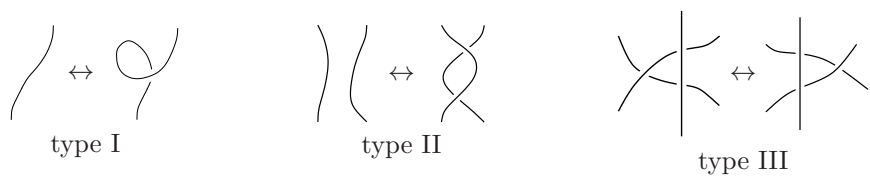

  $$
  \begin{array}{c}
    \dessinH{1cm}{MoveI2} \leftrightarrow \ \dessinH{1cm}{MoveI1}\\[.5cm]
    \textrm{type I}
  \end{array}
  \hspace{1.2cm}
  \begin{array}{c}
    \dessinH{1cm}{MoveII1} \leftrightarrow \dessinH{1cm}{MoveII2}\\[.5cm]
    \textrm{type II}
  \end{array}
  \hspace{1.2cm}
  \begin{array}{c}
    \dessinH{1.4cm}{MoveIII1} \leftrightarrow \dessinH{1.4cm}{MoveIII2}\\[.5cm]
    \textrm{type III}
  \end{array}
  $$
  \caption{Reidemeister moves}
  \label{fig:EnReidMoves}
\end{figure}

A \emph{link invariant} is an application defined on the set of links. Most known invariants, such as Alexander or Jones polynomials, can be defined combinatorially using diagrams.\\

Now, we consider a polynomial link invariant $I$, \ie an application which associates a Laurent polynomial to every link.\\
One can try to categorify $I$. Practically, such a \emph{categorifcation} means defining a graded chain complexes $\I=(\I_i^j)_{i,j\in\Z}$ associated to any diagram $D$ such that
\begin{enumerate}
\item the homology $H_*(\I)$ depends only of the underlying link $K$;
\item the graded Euler characteristic $\xi_\gr (\I) = \displaystyle{\sum_{i,j}}(-1)^i \rk \I_i^jq^j$ is equal to $I(K)$ with undetermined $q$.
\end{enumerate}
Categorification is worthwhile since it defines a new link invariant which usually distinguishes more links than the initial polynomial invariant, and which is generally endowed with ``good'' functorial properties with regard to the category of cobordisms.\\
In 1999, M. Khovanov initiated this field by categorifying the Jones polynomial. Four years later, P. Ozsv\'ath, Z. Sz\'abo and, independently, J. Rasmussen were defining the link Floer homology which categorifies the Alexander polynomial. The latter construction was geometric, using some kind of Floer theory on Heegaard splitting compatible with a given knot. However, during the summer 2006, in collaboration with C. Manolescu, S. Sarkar and D. Thurston, they gave an equivalent construction based on grid diagram presentation, which is a way to describe link diagrams using combinatorial square grids.\\     

On the other hand, according to the \emph{Vassiliev theory}, one can generalize any polynomial invariant $I$ to singular links, \ie immersions of oriented circles in $\R^3$ with a finite number of rigid transverse double points. This can be done using the following iterative formula:
\begin{equation}
I(\dessin{.5cm}{Double}) := I(\dessin{.5cm}{Pos}) - I(\dessin{.5cm }{Neg}).
  \label{eq:Vassiliev}  
\end{equation}
Then one defines \emph{finite type invariants of order $k$} for every integer $k$ as the invariants which vanish on every knots with, at least, $k+1$ doubles points. It defines a filtration on polynomial invariants. Most known invariants are combinations of finite type invariants. Moreover the question is still open to know if the union of all finite type invariants is strong enough to distinguish all knots.\\

Thus, a given polynomial invariant $I$ can be studied trough the Vassiliev theory in one hand, or through a categorification $\I$ on the other. It is natural to raise the question of a possible relation between these two constructions.
$$
\xymatrix@!0@R=1.1cm@C=3.5cm{& {\begin{array}{c} \textrm{Categorification $\I$:}\\[.1cm] I=\displaystyle{\sum_{i,j}} (-1)^i \rk \I_i^j q^j \end{array}} \ar[]!R;[rd]!UL& \\
  **[l] \textrm{Invariant polynomial } I \ar[]!UR;[ru]!L \ar[]!DR;[rd]!L & & **[r] ? \\
  & {\begin{array}{c}\textrm{Vassiliev theory:} \\[.1cm] I(\dessin{.5cm}{Double}):= I(\dessin{.5cm}{Pos}) - I(\dessin{.5cm }{Neg}) \end{array}} \ar[]!R;[ru]!DL & \\}
$$
In other words, can the formula (\ref{eq:Vassiliev}) be categorified ?\\
There is a natural candidate.
Actually, any chain complex morphism $\func{f}{C_1}{C_2}$ gives raise to a third chain complex, called $\Cone (f)$, of which the Euler charateristic is the difference between the Euler characteristic of $C_1$ and $C_2$.
In the Khovanov and in the link Floer cases, several chain maps can be considered. Unfortunatly, they do not behave reasonably when applied to a knot with a small singular loop. As a matter of fact, the associated homology is usually non trivial. But yet, in this instance, the polynomial case clearly satisfies:
$$
I\left(\ \dessin{.9cm}{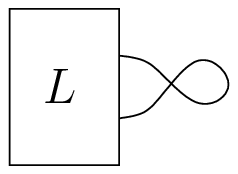}\right)=0.
$$ 
For a categorification $H$ of (\ref{eq:Vassiliev}), it is then natural to require the following condition:
\begin{equation}
H\left(\ \dessin{.9cm}{SmallLoop}\right) \equiv 0.
\label{eq:Condition}
\end{equation}

The {\bf first part} of this thesis is dedicated to the construction of such a singular generalization of the link Floer homology. For this purpose,
\begin{itemize}
\item we generalize the grid diagram presentation to singular links, stating and proving a Reidemeister-like result for a given set of elementary moves;
\item and we apply it to the the construction of a homological invariant $H^-$ generalizing the link Floer homology to singular links, categorifying (\ref{eq:Vassiliev}) and satisfying (\ref{eq:Condition}).
\end{itemize}
This invariant depends only on the singular link and on a choice of orientation for its double points \ie a choice of orientation for all the plan spanned by the two tangent vectors at a double point.\\
Then we discuss the construction and give a few properties of $H^-$.\\

The {\bf second part} concerns the Khovanov homology and some of its internal behavior with regard to oriented Reidemeister moves.\\
Actually, as shown in Figure \ref{fig:OrReidMoves}, when paying attention to the orientations of the involved strands, one can distinguish two kind of Reidemeister moves of type II and III.
\begin{figure}[!h]
$$
  \xymatrix@!0 @R=.9cm @C=2.5cm {
& \dessin{.6cm}{sing1}\\
\dessin{.75cm}{SingII} \ar[ur]!L \ar[dr]!L & \\
& \dessin{.6cm}{sing1bis}\\
} \hspace{2cm}
\xymatrix@!0 @R=.9cm @C=2.5cm {
& \dessin{.9cm}{sing2}\\
\dessin{1.1cm}{SingIII} \ar[ur]!L \ar[dr]!L & \\
& \dessin{.9cm}{Sing2bis}\\
}
$$
  \caption{Distinction between oriented Reidemeister moves}
  \label{fig:OrReidMoves}
\end{figure}
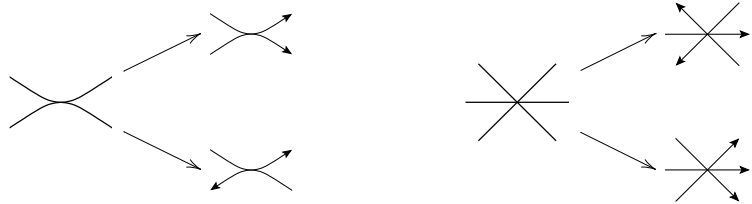
If restricting to the case of closed braids, only the moves corresponding to the singularities of type $\dessin{.5cm}{sing1}$ or $\dessin{.75cm}{Sing2bis}$ can occur during an isotopy. We call them \emph{braid-like moves} and we define the set $\K_{br}$ of \emph{braid-like links} as the set of link diagrams up to braid-like moves. They can be seen as \emph{universal planar transverse} links.\\
Actually, if $\xi$ is any given vector field in the plane, one can define $\K_\xi$ as the set of knot diagrams which are transverse to $\xi$ up to the isotopies which preserve this transversality condition\footnote[1]{Here, the transversality condition lives in the plane at the level of diagrams. The relation with 3--dimensional transverse conditions still need to be clarified.}. Then, there is a natural map
$$
\func{\phi_\xi}{\K_\xi}{\K_{br}}.
$$
Then, any invariant defined on the latter induces an invariant on the former.\\
In the specific case of closed braids, which corresponds to a radial vector field, this map is injective. Braid-like links are thus a generalization of closed braids.\\

In this thesis, we define
\begin{itemize}
\item a refinement $V_{br}$ of the Jones polynomial which is only invariant under braid-like moves;
\item and a bigraded categorification $\HH_{br}$ of $V_{br}$ which is also invariant under braid-like moves.
\end{itemize}
By spreading out the Khovanov homology in a second grading, the braid-like refinement gives a point of view on its internal mechanisms, for instance when performing a non braid-like Reidemeister move of type II. This is of particular interest since the Khovanov--Rozansky generalizations of Khovanov homology fail in being invariant under this move.\\
Although there is no geometrical interpretation for it, to date, the constructions given in the braid-like case can be modified to be invariant only under the moves corresponding to the singularities $\dessin{.5cm}{sing1}$ or $\dessin{.75cm}{sing2}$. Here again, it brings to light some particularities of the invariance phenomena.

%%% Local Variables: 
%%% mode: latex
%%% TeX-master: "These"
%%% End: 

% User's guide
\chapter*{User's guide}
\label{chap:Guide}
\addcontentsline{toc}{part}{User's guide}

This thesis is composed of two independent parts. Each part is divided in chapters, sections and subsections. The part and the chapter are not mentioned in the numbering of sections, subsections and propositions.\\

In order to lighten the presentation, paragraphs have been numbered and titled. The numbering is continuous from the begining until the end. It should not be understood as part of the hierarchical text structure but as a labeling which clarify the semantic one and an help for the localization of propositions. With this intention, each time we refer to a proposition, or something equivalent, which does not belong to the current paragraph, we indicate between parenthesis the paragraph where the reader would find it. Moreover, it lifts all the ambiguities occurring because of the numbering of the propositions which does not mention the chapter.\\

Some definitions are given in the body text and some in the introduction. To deal with this, an index is given at the end of the document.\\

In order to make the navigation inside the document easier, references and citations are enhanced with hypertext links.\\

The next two pages are devoted to a schematic plan for the logical organization of the thesis. Framed zones corresponds to some remainder ou prerequisites.

\newpage

\begin{figure}[m]
  $$
  \begin{xy}
    (-44,0)*{\begin{xy} (0,25)*+{\begin{array}{c} \\[-.4cm]
            \textrm{Grid diagrams}\\[.2cm]
            \dessin{2.2cm}{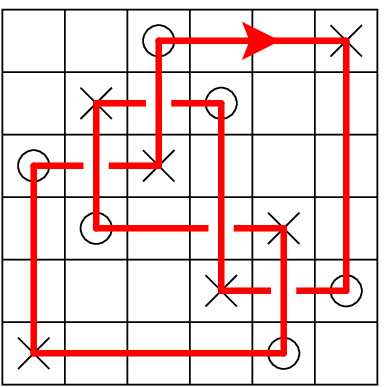}\end{array}}="A";
        (0,-25)*+{\begin{array}{c} \textrm{Link Floer homology}\\[.2cm]
            C^-\left(\dessin{1.2cm}{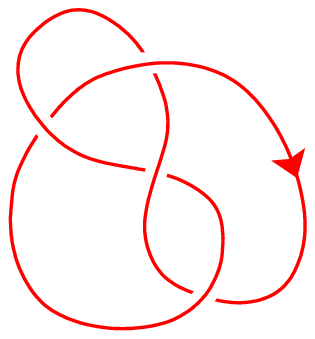}\right) \leadsto
            H^-\left(\dessin{1.2cm}{Knot}\right)\\[-.1cm] \
          \end{array}}="B"; **\frm{-} \ar@{<-}"B";"A"^(.47)*[@]{\hbox
          to 0pt{\hss\txt{\scriptsize \cite{MOST}}\hss}}
        \ar@{}"B"!<-2.4cm,1.9cm>;"A"!<-2.4cm,1.9cm>^*[@]{\hbox to
          0pt{\hss\txt{Chapter I}\hss}}
      \end{xy}}="D"; (54,21)*{\begin{array}{c} \\[-.4cm]
        \textrm{Singular grid diagrams}\\[.2cm]
        \dessin{1.83cm}{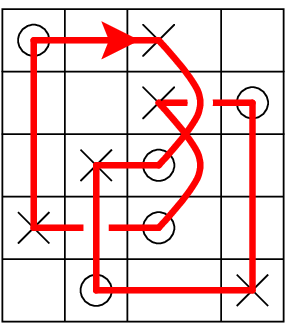}\end{array}}="E";
    (38,-45)*{\begin{array}{c} \textrm{Switch morphism}\\[.2cm]
        \func{f}{C^-(\dessin{.5cm}{Pos})}{C^-(\dessin{.5cm}{Neg})} \end{array}}="F";
    (-40,-93)*{\begin{array}{c} \textrm{Singular link}\\\textrm{Floer homology}\\[.2cm]
        H^-\left(\dessin{1.2cm}{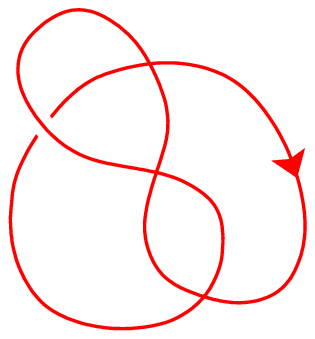}\right) = \Cone\left( \dessin{2.5cm}{}\right) \\[-.1cm] \
      \end{array}}="G";
    (52,-83)*{\begin{array}{c} \\[-.4cm]
        \textrm{Discussion}\\
        \textrm{on definitions}\end{array}}="H";
    (49,-116)*{H^-\left( \dessin{.9cm}{SmallLoop} \right)\cong 0}="I";
    (14,-139)*{\begin{array}{c} \\[-.4cm]
        \textrm{Filtrated inverse}\\
        \textrm{of degree }1\textrm{ for }f\end{array}}="J"
    
    \ar"D"!<4.2cm,1.9cm>;"E"!<-2.4cm,-.2cm>^{\txt{\scriptsize Section
        II.1}} \ar"D"!<3.4cm,-3cm>;"F"!UL|{}="aa"
    \ar"E"!D!<-.15cm,-.7cm>;"F"!U!<.2cm,.7cm>|(.6){}="bb"
    \ar@{.}@/^.7cm/"aa";"bb"^(.48)*[@]{\hbox to
      0pt{\hss\txt{\scriptsize Section II.2}\hss}}
    \ar@{<-}"G"!UR!<-1cm,.3cm>;"F"!DL!<0cm,-.4cm>^(.52)*[@]{\hbox
      to 0pt{\hss\txt{\scriptsize Section II.3}\hss}}_(.52)*[@]{\hbox
      to 0pt{\hss\txt{\scriptsize and II.4}\hss}}
    \ar"G"!R!<.4cm,.2cm>;"H"!L!<-.6cm,-.2cm>^*[@]{\hbox
      to 0pt{\hss\txt{\scriptsize Section III.1}\hss}}_(.52)*[@]{\hbox
      to 0pt{\hss\txt{\scriptsize and III.4}\hss}}
    \ar"G"!R!<.2cm,-1.2cm>;"I"!L!<-.6cm,.2cm>^*[@]{\hbox
      to 0pt{\hss\txt{\scriptsize Section III.3}\hss}}
    \ar"G"!D!<3.2cm,0cm>;"J"!UL!<1cm,.3cm>^(.49)*[@]{\hbox
      to 0pt{\hss\txt{\scriptsize Section III.2}\hss}}
    \ar"F"!D!<-.2cm,-.4cm>;"H"!U!<-.9cm,.4cm>^*[@]{\hbox
      to 0pt{\hss\txt{\scriptsize Section III.1}\hss}}_(.52)*[@]{\hbox
      to 0pt{\hss\txt{\scriptsize and III.4}\hss}}
  \end{xy}
  $$
\vspace{1cm}
\begin{center}
  {\large Part I}
\end{center}
\end{figure}

\newpage

\begin{figure}[m]
  $$
  \begin{xy}
    (0,0)*{\begin{xy}
        (0,112)*+{\begin{array}{c} \\[-.4cm]
            \textrm{Link diagrams}\\[.2cm]
            \dessin{2cm}{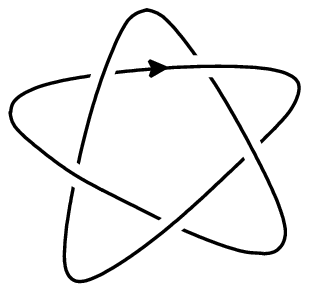}\end{array}}="A";
        (0,70)*+{\begin{array}{c} \\[-.4cm]
            \textrm{Link isotopies}\\[.2cm]
            \dessin{.6cm}{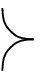}\hspace{.3cm}
            \dessin{.6cm}{SingII}\hspace{.3cm}
            \dessin{.9cm}{SingIII}\end{array}}="B2";
        (0,35)*+{\begin{array}{c} \\[-.4cm]
            \textrm{Jones polynomial}\\[.2cm]
            V(D)\in\Z [A,A^{-1}]\end{array}}="C2";
        (0,1)*+{\begin{array}{c} \\[-.4cm]
            \textrm{Khovanov homology}\\[.2cm]
            \HH(D)= \oplus_{i,j}\HH_i^j\\[-.2cm] \ \end{array}}="D2";
        \ar@{<-}"B2"!U!<0cm,.2cm>;"A"!D^(.53)*[@]{\hbox
          to 0pt{\hss\txt{\scriptsize \cite{Reidemeister}}\hss}}
        \ar@{<-}"C2"!U!<0cm,.2cm>;"B2"!D!<0cm,-.2cm>^(.53)*[@]{\hbox
          to 0pt{\hss\txt{\scriptsize \cite{StateJones}}\hss}}
        \ar@{<-}"D2"!U!<0cm,.2cm>;"C2"!D!<0cm,-.2cm>^(.53)*[@]{\hbox
          to 0pt{\hss\txt{\scriptsize \cite{Khovanov}}\hss}}
      \end{xy}}="M" *\frm{-};
    (-56,11)*{\begin{array}{c} \\[-.4cm]
        \textrm{Braid-like isotopies}\\[.2cm]
        \dessin{.6cm}{sing1}\hspace{.4cm}
        \dessin{.9cm}{Sing2bis}\end{array}}="B1";
    (56,11)*{\begin{array}{c} \\[-.4cm]
        \textrm{Star-like isotopies}\\[.2cm]
        \dessin{.6cm}{sing1}\hspace{.4cm}
        \dessin{.9cm}{sing2}\end{array}}="B3";
    (-56,-25)*{\begin{array}{c} \\[-.4cm]
        \textrm{Braid-like}\\\textrm{Jones polynomial}\\[.2cm]
        V_{br}(D)\in\Z [A,A^{-1},H,H^{-1}] \end{array}}="C1";
    (56,-25)*{\begin{array}{c} \\[-.4cm]
        \textrm{Star-like}\\\textrm{Jones polynomial}\\[.2cm]
        V_{st}(D)\in\Z [A,A^{-1},H,H^{-1}] \end{array}}="C3";
    (-56,-60)*{\begin{array}{c} \\[-.4cm]
        \textrm{Braid-like}\\\textrm{Khovanov homology}\\[.2cm]
        \HH_{br}(D)= \oplus_{i,j,k}\HH_i^{j,k} \end{array}}="D1";
    (56,-60)*{\begin{array}{c} \\[-.4cm]
        \textrm{Star-like}\\\textrm{Khovanov homology}\\[.2cm]
        \HH_{st}(D)= \oplus_{i,j,k}\HH_i^{j,k} \end{array}}="D3";
    \ar@{<-}"C1"!U!<0cm,.2cm>;"B1"!D!<0cm,-.2cm>^(.53)*[@]{\hbox
      to 0pt{\hss\txt{\scriptsize Section IV.2}\hss}}
    \ar@{<-}"C3"!U!<0cm,.2cm>;"B3"!D!<0cm,-.2cm>^(.53)*[@]{\hbox
      to 0pt{\hss\txt{\scriptsize Section IV.2}\hss}}
    \ar@{<-}"D1"!U!<0cm,.2cm>;"C1"!D!<0cm,-.2cm>^(.53)*[@]{\hbox
      to 0pt{\hss\txt{\scriptsize Chapter V}\hss}}
    \ar@{<-}"D3"!U!<0cm,.2cm>;"C3"!D!<0cm,-.2cm>^(.53)*[@]{\hbox
      to 0pt{\hss\txt{\scriptsize Chapter V}\hss}}
    \ar@{<-}"B1"!UR!<-.9cm,.3cm>;"M"!UL!<-.5cm,-2.2cm>^(.53)*[@]{\hbox
      to 0pt{\hss\txt{\scriptsize Section IV.1}\hss}}
    \ar"M"!UR!<.5cm,-2.2cm>;"B3"!UL!<.9cm,.3cm>^(.47)*[@]{\hbox
      to 0pt{\hss\txt{\scriptsize Section IV.1}\hss}}
  \end{xy}
  $$
\vspace{1cm}
\begin{center}
  {\large Part II}
\end{center}
\end{figure}

%%% Local Variables: 
%%% mode: latex
%%% TeX-master: "These"
%%% End: 

% 1ere partie : Homologie de Floer pour les entrelacs singuliers
\part{Singular link Floer homology}

% Chapitre de rappels
\chapter{Knot Floer homology}
\label{chap:Floer}

\section{Some basics in algebra}
\label{sec:Algebra}

In this first section, we set down some basic notation and definitions. For a thorough treatment, we refer the reader to \cite{Cartan}, \cite{Gelfand} or \cite{Weibel}.\\

Unless explicitly stated, by module we will mean a $\Z$--module.

\subsection{Structures on modules}
\label{ssec:Grad&Filt}

\subsubsection{Graded modules}
\label{par:GrModules}

Let $r$ be a positive integer.\\

A \emph{$r$--grading} on a module $M$ is a decomposition of $M$ into direct summands $\oplus_{\ib\in\Z^r}M_\ib$ where $M_\ib$ is the \emph{homogeneous module of elements of degree $\ib$}. We say that $M$ is \emph{$r$--graded}.\\\index{module!graded}

Let $M$ and $M'$ be two $r$--graded modules.\\

A \emph{graded map}\index{map!graded} $\func{f}{M}{M'}$ of degree $\kb\in \Z^r$ is a sequence ${(\func{f_\ib}{M_\ib}{M_{\ib+\kb}})}_{\ib\in\Z^r}$ of linear maps. When omitted, the degree of a graded map is zero \ie it preserves the grading.\\

When $r=1$, we say that $M$ is graded. The grading is then denoted by $\Z$--indices.\\
If $M_i$ is finitely generated for all integer $i$, we define the \emph{graded dimension}\index{dimension!graded}\index{dimension!qdim@$\qdim$|} as the formal series
$$
\qdim M=\disp{\sum_{i\in\Z}}\rk M_i.q^i\in \Z[[q,q^{-1}]].
$$
For all integer $l$, we define also the \emph{shift operation of height $l$}\index{shift operation
} which associates to M a new graded module $M\{l\}$ defined by $M\{l\}_i:=M_{i-l}$ for all integer $i$.\\

When $r=2$, we say that $M$ is bigraded. The first grading is then denoted by $\Z$--indices whereas the second is denoted by $\Z$--exponents.\\
Under the same finite generation condition, the \emph{bigraded dimension}\index{dimension!bigraded} is the formal series
$$
\qdim M=\disp{\sum_{i,j\in\Z}}\rk M_i^j.t^iq^j\in \Z[[t,t^{-1},q,q^{-1}]].
$$

When $r=3$, we say that $M$ is trigraded. The first grading is then denoted by $\Z$--indices whereas the lasttwo are denoted by $\Z^2$--exponents.\\
Under the same finite generation condition, the \emph{trigraded dimension}\index{dimension!trigraded} is the formal series
$$
\qdim M=\disp{\sum_{i,j,k\in\Z}}\rk M_i^{j,k}.t^iA^jH^k\in \Z[[t,t^{-1},A,A^{-1},H,H^{-1}]].
$$

In each  case , if $M$ is finitely generated, then the (tri,bi)graded dimension is a polynomial with non negative integer coefficients.\\

The module $M\otimes M'$ is naturally $r$--graded by the convention $\deg(x\otimes x')=\deg(x)+\deg(x')$ for all homogeneous elements $x\in M$ and $x'\in M'$.\\

Suppose now that $M$ is endowed with another $r'$--grading. Then $M$ is naturally $(r+r')$--graded by considering the intersection of homogeneous modules.

\subsubsection{Filtrated modules}
\label{par:FiModules}

Let $r$ be a positive integer.\\

  A \emph{filtration}\index{module!filtrated} on a ($r$--graded) module $M$ is an increasing sequence of ($r$--graded) modules
$$
{(\F_iM|\F_iM \subset \F_{i+1}M)}_{i\in\Z}
$$
such that $\cup_i \F_iM=M$ and  $\cap_i \F_iM=\emptyset$. We say that $M$ is \emph{filtrated}.\\
The filtration is \emph{bounded below}\index{module!filtrated!bounded below} if there exists an integer $i_0$ such that $\F_{i_0}M=\emptyset$.\\ 

 A \emph{filtrated (graded) map}\index{map!filtrated} $\func{f}{(M,\F)}{(M',\F')}$ of degree $k\in\N$ is a (graded) map $f$ such that $\Im f_{|\F_iM}\subset \F'_{i+k}M'$ for all integer $i$. When omitted, the degree of a filtrated map is zero \ie it preserves the filtration.\\

A filtration is naturally defined on any graded module $M$ by
$$
\forall i\in\Z, \F_iM=\bigoplus_{r\leq i}M_r.
$$
A graded map becomes then a filtrated one of the same degree.\\

Conversely, being given a filtration $\F$ on a ($r$--graded) module $M$, one can define a ($(r+1)$--)graded module by
$$
\forall i\in\Z, M_i:=\fract{\F_iM}/{\F_{i-1}M}.
$$
If the filtration is bounded below, then the filtrated and the graded modules are isomorphic as modules.\\
To any filtrated map $f$ of degree $k$, one can associate a graded one $f_\gr$ by composing it with the surjections ${(\func{s_i}{\F_iM}{\fract{\F_iM}/{\F_{i-1}M}})}_{i\in\Z}$.

\subsection{Chain complexes}
\label{ssec:Chain}

\subsubsection{Chain complexes}
\label{par:Chain}

  A \emph{chain complex}\index{chain complex} $(C,\p)$ is a graded module $C=\oplus_{i\in\Z}C_i$ together with a graded map $\p$ of degree $-1$, called \emph{differential}\index{differential}, such that $\p^2_{|C_i}=\p_{i-1}\p_i\equiv 0$ for all integer $i$.\\
When not pertinent, the differential will be omitted and only the underlying module will be specified.\\

Let $(C,\p)$ and $(C',\p')$ be two chain complexes. A \emph{chain map}\index{map!chain} $\func{f}{(C,\p)}{(C',\p')}$ is a graded map ${(f_i)}_{i\in\Z}$ such that for all integer $i$, $f_{i-1}\p_i+\p'_i f_n\equiv 0$, \ie such that every square in the following diagram is anti-commuting:
$$
\xymatrix @!0 @C=2cm @R=2cm {
\cdots & C_{i-1} \ar[l]_{\p_{i-1}} \ar[d]^{f_{i-1}} & C_i \ar[l]_(.45){\p_i} \ar[d]^{f_i} & C_{i+1} \ar[l]_{\p_{i+1}} \ar[d]^{f_{i+1}} & \cdots \ar[l]_(.47){\p_{i+2}}\\
\cdots & C'_{i-1} \ar[l]_{\p'_{i-1}} & C'_i \ar[l]_(.45){\p'_i} & C'_{i+1} \ar[l]_{\p'_{i+1}} & \cdots \ar[l]_(.47){\p'_{i+2}} .
}
$$
Usually, chain maps are required to commute with the differentials, but within sight of our use of them, anti-commutativity is more convenient. Nevertheless, the usual composition of two chain maps is thus not more a chain map. However, by multiplying $f_i$ by $(-1)^i$, it is straightforward to turn an anti-commuting map into a commuting one and vice versa. Hence, this inconvenient is not annoying at all.\\

For any integer $l$, we define the \emph{shift operation of height $l$}\index{shift operation} which associates to any chain complex $(C,\p)$ a new chain complex $(C[l],\p[l])$ defined by $C[l]:=C\{l\}$ and $\p[l]_i:=\p_{i-l}$ for all integer $i$.\\

When converging, the \emph{Euler characteristic} $\xi(C)$\index{Euler characteristic} of a chain complex $C$ is defined as the graded dimension of $C$ evaluated at $-1$ \ie
$$
\xi(C):=\big(\qdim C\big)(-1)=\sum_{i\in\Z}(-1)^i\rk C_i.
$$
\index{Xi@$\xi$|see{Euler characteristic}}
\subsubsection{Graded chain complexes}
\label{par:GrChain}

  A \emph{graded chain complex}\index{chain complex!graded} is a chain complex $(C,\p)$ where $C_i$ is graded and $\p_i$ respects this grading for all integer $i$. The module $C$ is then bigraded. By convention, the chain complex grading is considered as the first one. The map $\p$ is hence graded of degree (-1,0).\\

A \emph{graded chain map}\index{map!chain!graded} of degree $k$ for any integer $k$ is a chain map $f$ such that $f_i$ is graded of degree $k$ for all integer $i$.\\

For any integer $l$, the shift operation $.\{l\}$ can be extended to graded chain complexes by applying it simultaneously to all the graded modules.\\

When converging, the \emph{graded Euler characteristic}\index{Euler characteristic!graded} $\xi_\gr(C)$ of a graded chain complex $C$ is defined as the bigraded dimension of $C$ with the first variable evaluated at $-1$ \ie
$$
\xi_\gr(C)=\big(\qdim C\big)(-1,q)=\sum_{i,j\in\Z}(-1)^i\rk C_i^j.q^j.
$$
\index{Xigr@$\xi_{gr}$|see{Euler characteristic, graded}}

If $C$ is finitely generated, then $\xi_\gr(C)$ is a polynomial with integer coefficients.\\

One can extend the notion of $r$--grading to higher values of r.

\subsubsection{Filtrated chain complexes}
\label{par:FiChain}

  A \emph{filtrated chain complex}\index{chain complex!filtrated} is a ($r$--graded) chain complex $(C,\p)$ where $C_i$ is filtrated and $\p_i$ respects this filtration for all integer $i$. We say that the filtration on $(C,\p)$ is \emph{bounded below}\index{chain complex!filtrated!bounded below} if it is bounded below on $C_i$ for all integer $i$.\\

A \emph{filtrated chain map}\index{map!chain!filtrated} of degree $k$ for any non negative integer $k$ is a (graded) chain map $f$ such that $f_i$ is filtrated of degree $k$ for all integer $i$.

\begin{lemme}
  Let $C=\oplus_{i,j\in\Z}C_i^j$ be a bigraded module and $\func{\p}{C}{C}$ a graded map of degree $-1$ with regard to the first grading and which preserves the filtration $\F$ associated to the second one. If $\p^2\equiv 0$, then
  \begin{itemize}
  \item[-] $\p_\gr$, the graded part of $\p$ with regard to $\F$, satisfies $\p_\gr^2\equiv 0$ as well;
  \item[-] $\big((C,\F),\p\big)$ is a filtrated chain complex;
  \item[-] $(C,\p_\gr)$ is a graded chain complex.
  \end{itemize}
\end{lemme}

\subsubsection{Homologies}
\label{par:Homologies}

The condition $\p^2\equiv 0$ implies that $\Im \p_{i+1} \subset \Ker \p_i$ for all integer $i$.

\begin{defi}
  The \emph{homology}\index{homology} $H(C,\p)$ of a ($r$--graded, filtrated) chain complex $(C,\p)$ is the graded ($(r+1)$--graded, filtrated) module ${\big(H_i:=\fract{\Ker\p_i}/{\Im\p_{i+1}}\big)}_{i\in\Z}$.\\
The modules ${(H_i)}_{i\in\Z}$ are also called \emph{homology groups}\index{homology!groups}. If they are all null, we say that $(C,\p)$ is \emph{acyclic}\index{chain complex!acyclic} \index{acyclic|see{chain complex}}.
\end{defi}

Homology groups can be seen as a ($r$--graded, filtrated) chain complex with a trivial differential.
 
\begin{prop}
  A chain map $\func{f}{C}{C'}$ induces a map $\func{f_*}{H(C)}{H(C')}$ on the associated homology groups.
\end{prop}

A chain map $f$ is a \emph{quasi-isomorphism}\index{map!chain!quasi-isomorphism}\index{quasi-isomorphism|see{map}} if the induced map $f_*$ is an isomorphism on homology groups.

\begin{prop}\label{Euler}
  Let $(C,\p)$ be a graded chain complex then
$$
\xi_\gr \big(H(C,\p)\big)=\xi_\gr(C,\p).
$$
\end{prop}

\subsubsection{Cohomologies}
\label{par:Cohomologies}

  A \emph{cochain complex}\index{cochain complex} $(C,\p)$ is a graded module $C=\oplus_{i\in\Z}C^i$ together with a graded map $\p$ of degree $1$ called \emph{codifferential}\index{codifferential} such that $\p^2_{|C_i}=\p^{i+1}\p^i\equiv 0$ for all integer $i$.\\

Mutatis mutandis, all the definitions and the statements made for chain complexes have their counterpart for cochain complexes. In particular, one can define bigraded cochain complexes as well as bigraded cohomologies.\\
Usual notation for bigraded cochain complexes is to exchange exponents and indices.\\ 

To any (bigraded) chain complex $(C,\p)$, we can naturally associated a dual (bigraded) cochain complex $(C^*,\p^*)$ defined by
\begin{gather*}
  \forall i\in\Z, C^{*,i}:=\Hom(C_i,\Z)\\[.2cm]
  \forall \phi\in C^*, \p^*(\phi)=\phi\circ \p.
\end{gather*}

The following proposition is a corollary of Universal Coefficient Theorem for cohomology:
\begin{prop}\label{UCT}
  For any bigraded chain complex $(C,\p)$ and for any integers $i$, $j$ and $k$ the homology groups $H^{*,i}_{j,k}(C^*)$ and $\big(H_i^{j,k}(C)/\T_i^{j,k}(C)\big)\oplus T_{i+1}^{j,k}(C)$ are isomorphic where $T(D)$ denotes the torsion part of $H(C)$.
\end{prop}

\subsection{Algebraic tools}
\label{ssec:Tools}

All definitions and statements in this section do hold, {\it mutatis mutandis}, for $r$--graded and filtrated chain complexes.

\subsubsection{Short and long exact sequences}
\label{par:ExSequences}

A \emph{long exact sequence}\index{exact sequence!long} is an acyclic  chain complex.\\

A \emph{short exact sequence}\index{exact sequence!short} of chain complexes is a triplet $(C,C',C'')$ of chain complexes together with an injective chain map $\func{f}{C}{C'}$ and a surjective one $\func{g}{C'}{C''}$. We denote it by
$$
\xymatrix{
0 \ar[r] & C \ar@{^(->}[r]^f & C' \ar@{->>}[r]^g & C'' \ar[r] & 0.
}
$$

\begin{theo}
  Let $\xymatrix{0 \ar[r] & C \ar@{^(->}[r]^f & C' \ar@{->>}[r]^g & C'' \ar[r] & 0}$ be a short exact sequence of chain complexes. Then there are natural maps $\func{d_i}{H_i(C'')}{H_{i-1}(C)}$ for all integer $i$ such that
$$
\xymatrix@!0@C=2.5cm{
\cdots & H_{i-1}(C) \ar[l]_(.55){{(f_*)}_{i-1}} & H_i(C'') \ar[l]_(.42){d_i} & H_i(C') \ar[l]_(.45){{(g_*)}_i} & H_i(C) \ar[l]_(.45){{(f_*)}_i} & H_{i+1}(C'') \ar[l]_{d_{i+1}} & \cdots \ar[l]_(.35){{(g_*)}_{i+1}}
},
$$
also denoted
$$
\xymatrix@!0@C=1.5cm@R=1.5cm{
H(C'') \ar[rr]^{d} && H(C) \ar[dl]^{f_*}\\
& H(C') \ar[ul]^{g_*}&
},
$$
is a long exact sequence.
\end{theo}

Exact sequences are very efficient for computing homology groups. Actually, if a short exact sequence involves an acyclic chain complex, then the two other ones share the same homology.\\
Moreover, the following proposition holds:

\begin{prop}[$5$--lemma]
  Let
$$
\xymatrix{
E \ar[d]^\cong_e & D \ar[l] \ar[d]^\cong_d & C \ar[l] \ar[d]_c & B \ar[l] \ar[d]^\cong_b & A \ar[l] \ar[d]^\cong_a\\
E' & D' \ar[l] & C' \ar[l] & B' \ar[l] & A' \ar[l]
}
$$
be a $5$ piece chain map between two long exact sequences. Moreover, suppose that $a$, $b$, $d$ and $e$ are isomorphisms, then $c$ is also an isomorphism.
\end{prop}

\subsubsection{Spectral sequences}
\label{par:SpecSequences}

Let $n$ be a non negative integer.\\

A \emph{spectral sequence}\index{spectral sequence} is the data of:
\begin{itemize}
\item[-] a sequence ${(E^s)}_{s\in\N}$ of bigraded modules;
\item[-] graded maps $\func{\p^s}{E^s}{E^s}$ of degree $(s,1-s)$ for all non negative integer $s$;
\item[-] graded isomorphisms between $E^s$ and $\fract{\Ker \p^{s-1}}/{\Im \p^{s-1}}$ for all positive integer $s$.
\end{itemize}
For all non negative integer $s$, $E^s$ is called the \emph{$s^{\textrm{th}}$ page}. We say the spectral sequence starts at $E_0$.\\

A spectral sequence \emph{converges}\index{spectral sequence!convergence} to a bigraded module $H$ if for any degree $(i,j)\in\Z^2$, there exists an non negative integer $s_{ij}$ such that for all $s\geq s_{ij}$, ${(E^s)}_i^j$ is isomorphic to $H_i^j$.\\

\begin{prop}
  A filtrated chain complex $\big((C,\F),\p\big)$ naturally determines a spectral sequence which starts at the associated graded chain complex $(C,\p_\gr)$ (after the change of indices $(p,q)=(j,-i-j)$).\\
Moreover, if the filtration is bounded below, then the spectral sequence converges to $H(C)$.
\end{prop}

\begin{cor}\label{Filt->Acycl}
  If a chain complex $C$ can be endowed with a filtration such that the associated graded chain complex is acyclic, then $C$ is acyclic. 
\end{cor}

\subsection{On chain maps}
\label{ssec:OnMaps}

All definitions and statements in this section do hold, {\it mutatis mutandis}, for graded and filtrated chain maps.

\subsubsection{Chain homotopies}
\label{par:HomChain}

A chain map $\func{f}{(C,\p)}{(C'\p')}$ between chain complexes is \emph{null homotopic}\index{map!chain!null homotopic}\index{null homotopix|see{map}} if there exists a graded map $\func{h}{C}{C'}$ of degree 1 such that
$$
f\equiv \p' h - h \p
$$
or such that
$$
f+ \p' h + h \p \equiv 0.
$$
Two chain maps are \emph{homotopic}\index{map!chain!homotopic}\index{homotopic|see{map}} if their sum or their difference is null homotopic.\\

\begin{prop}\label{induction}
  Two homotopic chain maps $f$ and $g$ induce the same map $f_*\equiv g_*$ on the homology groups.
\end{prop}

Two chain complexes $C$ and $C'$ are \emph{homotopy equivalent}\index{chain complex!homotopy equivalence}\index{homotopy equivalence|see{chain complex}} if there exist two chain maps $\func{f}{C}{C'}$ and $\func{g}{C'}{C}$ such that $fg$ and $gf$ are respectively homotopic to the identity map on $C$ and $C'$.\\

  According to proposition \ref{induction}, the maps $f_*$ and $g_*$ are inverse one to each other. Thus, the chain complexes $C$ and $C'$ share the same homology groups.\\
The converse is not necessarily true. In this thesis, we will deal only with homology groups and not with chain complexes up to homotopy equivalence. However, homotopies remain an effective tool to prove that a map is a quasi-isomorphism.

%   This notion of homotopy equivalence can be even tightened up by imposing that $gf$ {\bf is} actually the identity map on $C$. The chain map $g$ is then a \emph{strong deformation retract} of $C'$ into $C$.

\subsubsection{Mapping cones}
\label{par:MapCones}

Let $\func{f}{(C,\p)}{(C',\p')}$ a chain map. The \emph{mapping cone}\index{mapping cone} of $f$ is the chain complex $\Cone(f):=C[1]\oplus C'$ with differential $\p_c$ defined by
$$
\p_c(x)=\left\{
    \begin{array}{ll}
      \p(x) + f(x) & \textrm{if }x\in C[1]\\[.15cm]
      \p'(x)       & \textrm{if }x\in C'.
    \end{array}
\right.
$$
\begin{lemme}
  The map $\p_c$ is actually a differential \ie it satisfies $\p_c^2\equiv 0$.
\end{lemme}

By construction, it verifies the following short exact sequence:
$$
\xymatrix{
0 \ar[r] & C' \ar@{^(->}[r] & \Cone(f) \ar@{->>}[r] & C[1] \ar[r] & 0
}.
$$
This leads to the long exact sequence:
$$
\xymatrix @!0 @C=1.5cm @R=1.5cm{
H(C) \ar[rr]^{f_*} && H(C') \ar[dl]\\
& H\big(\Cone(f)\big) \ar[ul]&.
}
$$

\begin{cor}
  The mapping cones of two homotopic chain maps share the same homology.
\end{cor}

\begin{cor}\label{AcyclicCone}
  A chain map is a quasi-isomorphism if and only if its mapping cone is acyclic.
\end{cor}

\begin{cor}\label{QuasiIso}
  Let $f$ be a filtrated chain map. If the associated graded chain map $f_\gr$ is a quasi-isomorphism then $f$ is a quasi-isomorphism.
\end{cor}
\begin{proof}
  The mapping cone of $f$ inherits a filtration. Now, the first page of the spectral sequence associated to this filtration is the homology of the mapping cone of $f_\gr$ which is actually acyclic. The spectral sequence converges then to zero and the mapping cone of $f$ is acyclic.
\end{proof}

% Finally, under strong homotopy condition, the mapping cone can be translated.
% \begin{prop}[\cite{BarNatanKho}, lemma 4.5]
% Let $\func{h}{C_1}{C'_0}$ be a chain map.\\
% Let $\func{g}{C_0}{C_1}$ and $\func{g'}{C'_0}{C'_1}$ be strong deformation retracts.
% $$
% \xymatrix @!0 @C=2cm @R=2cm{
% &C_1 \ar[d]_(.45)h \ar@{-->}[dl]_(.45){g'h}  & C_0 \ar[l]_g^(.55){\textrm{\tiny s. d. r.}} \ar@{-->}[dl]^(.45){hg}\\
% C'_1 & C'_0 \ar[l]_(.45){g'}^{\textrm{\tiny s. d. r.}} &
% }
% $$
% Then $\Cone(h)$, $\Cone(hg)$ and $\Cone(g'h)$ are homotopy equivalent.
% \end{prop}

\subsubsection{Cubes of maps}
\label{par:MapsCubes}

Let $k$ be a non negative integer.\\
First, a few notation have to be set.

\begin{nota}
 Let $I=(i_1,\cdots,i_k)$ be an element of $\{0,1,\star\}^k$.\\
For all $j\in\llbracket 1,\cdots,k\rrbracket$ and $a\in\{0,1,\star\}$, we denote
 \begin{itemize}
 \item[-] by $0(I)$ the set $\big\{j\in\llbracket 1,k\rrbracket |i_j=0\big\}$;
 \item[-] by $I(j:a)$ the $k$--uple obtained from $I$ by inserting $a$ as the $j^{\textrm{th}}$ element.
 \end{itemize}
By abuse of notation, we will uncurrify them and write $I(i:a,j:b)$ instead of $\big(I(i:a)\big)(j:b)$ if $i<j$ and instead of $\big(I(i-1:a)\big)(j:b)$ if $i>j$.
\end{nota}

A \emph{$k$--dimensional cube of maps}\index{map!cube of} $\C$ is a set ${\{C_I\}}_{I\in\{0,1\}^k}$ of chain complexes associated with chain maps $\func{f_{I(j:\star)}}{C_{I(j:0)}}{C_{I(j:1)}}$ for all $I\in \{0,1\}^{k-1}$ and $j\in\llbracket 1,k\rrbracket$. The cube is \emph{straigth}\index{map!cube of!straigth} if
$$
f_{I(i:\star,j:1)}\circ f_{I(i:0,j:\star)}+f_{I(i:1,j:\star)}\circ f_{I(i:\star,j:0)}
$$
for all $I\in \{0,1\}^{k-2}$ and all $i\neq j\in\rrbracket 1,k\llbracket$.\\

If $\C$ is straight, then its \emph{mapping cone in direction $j\in\llbracket 1,k\rrbracket$} is the straight $(k-1)$--dimensional cube of maps defined by ${(\Cone(f_{I(j:\star)}))}_{I\in\{0,1\}^{k-1}}$ and the chain maps induced by the chain maps of $\C$.

\begin{lemme}
  The chain complex obtained by applying iteratively mapping cones to a cube of maps does not depend on the order in which the cones have been applied.
\end{lemme}

Now we can give a direct definition of $\Cone(\C)$, the \emph{generalized cone of $\C$}\index{mapping cone!generalized cone}, as the chain complex $\oplus_{I\in \{0,1\}^k} C_I[\# 0(I)]$ with differential $\p_c$ defined on $C_I[\# 0(I)]$ for all $I\in\{0,1\}^k$ by
$$
\p_c(x)= \p_I(c)+\sum_{j\in 0(I)} f_{I(j:\star)}(x).
$$

\subsubsection{Lame cube of maps}
\label{par:LameCube}
\index{map!cube of!lame}
Being straight is an essential condition. Otherwise, the map $\p_c$ defined in the previous paragraph does not satisfy $\p_c^2=0$. However, some anti-commutativity defects can be corrected by adding ``diagonal'' maps. In restriction to a vertex, the map $\p_c$ is then defined as the sum of all the arrows leaving this vertex. If this map satifies the differential relation $\p_c^2=0$, then the generalized cone can be defined as well.\\

It will be the case in paragraph \ref{par:Completion}.

%%% Local Variables:
%%% mode: latex
%%% TeX-master: "These"
%%% End:

\newpage

% Description des diagrammes en grille
\section{Grid diagrams}
\label{sec:Grid}

In this section, we recall grid presentations for knots and links.

\subsubsection{Grid diagrams}
\label{par:Grid}

A \emph{grid diagram}\index{diagram!grid} $G$ of size $n\in \N^*$ is a $(n\times n)$--grid  whose squares may be decorated by a $O$ or by an $X$ in such a way that each column and each row contains exactly one $O$ and one $X$.\\
We denote by $\O$ the set of $O$'s and by $\X$ the set of $X$'s.\\
A \emph{decoration}\index{decoration} is an element of $\O\cup \X$.\\

A link diagram can be associated to any grid diagram. For this purpose, one should join the $X$ to the $O$ in each column by a straight line and, then, join again the $O$ to the $X$ in each row by a straight line which underpasses all the vertical ones. Each decoration is then replaced by a right angled corner which can be smoothed.\\
By convention, the link will be oriented by running the horizontal strands from the $O$ to the $X$.\\

\begin{figure}[h]
$$
\dessin{2.3cm}{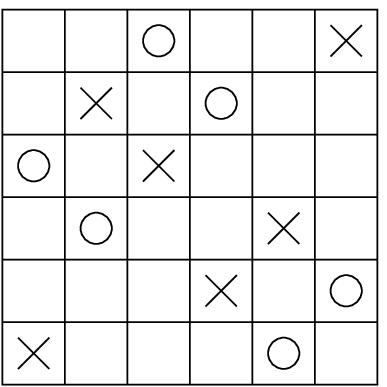} \hspace{.3cm} \leadsto \hspace{.3cm}\dessin{2.3cm}{GridKnot} \hspace{.3cm} \leadsto \hspace{.3cm} \dessin{2.3cm}{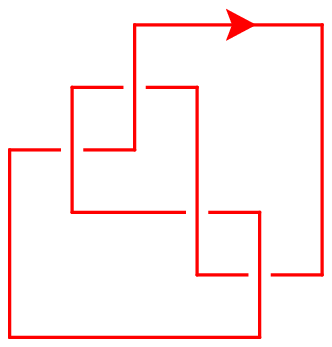} \hspace{.3cm} \leadsto \hspace{.3cm} \dessin{2.3cm}{Knot}
$$
\caption{From grid diagrams to knots}
\label{fig:G->K}
\end{figure}

Every link can be described by a grid diagram. Even better, every link diagram (up to isotopy) can be described in this way. To this end, rotate locally all the crossings of a given diagram in order to get only orthogonal intersections with vertical overpassing strands. Then, perform an isotopy to get only horizontal and vertical strands such that none of them are colinear, and incidently to get right angled corners. Now, according to the link orientation, turn the corners into $O$ or $X$ and draw the grid by separating pairs of decorations in columns and in rows. An illustration of this process is given by the Figure \ref{fig:G->K} read from right to left.

\subsubsection{Elementary moves}
\label{par:ElemMoves}

Of course, different grid diagrams can lead to the same link. However, this is controlled by the following theorem:

\begin{theo}[Cromwell \cite{Cromwell}, Dynnikov \cite{Dynnikov}]\label{def:Moves}
  Any two grid diagrams which describe the same link can be connected by a finite sequence of the following elementary moves:
  \begin{description}
  \item[Cyclic permutation]\index{moves!grid!Cyclic permutation} cyclic permutation of the columns (resp. rows);
  \item[Commutation]\index{moves!grid!Commutation} commutation of two adjacent columns (resp. rows) under the condition that all the decorations of one of the two commuting columns (resp. rows) are strictly above (resp. strictly on the right of) the decorations of the other one;
  \item[Stabilization/Destabilization]\index{moves!grid!(de)stabilization@stabilization} addition (resp. removal) of one column and one row by replacing (resp. substituting) locally a decorated square by (resp. to) a $(2\times 2)$--grid containing three decorations in such a way that it remains globally a grid diagram.
  \end{description}\index{moves!grid}
\end{theo}

The condition for a commutation move is slightly more restricting than usually stated. However, it is easily seen that, up to cyclic permutations, both definitions are equivalent.

\begin{figure}[h]
  \begin{gather*}
\hspace{-.6cm}
\begin{array}{ccc}
  \xymatrix@C=2cm{\dessin{2.1cm}{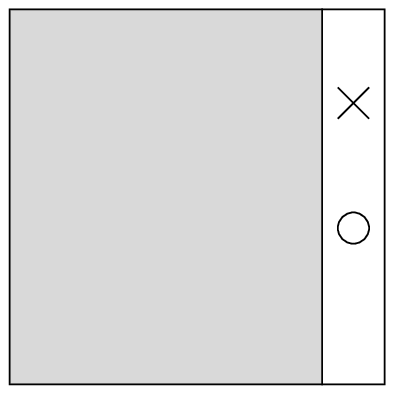} \ar@{<->}[r]^{\dessin{.9cm}{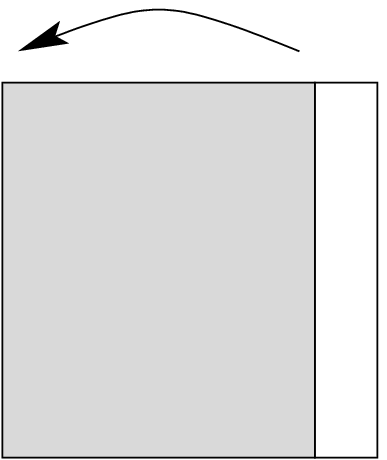}} & \dessin{2.1cm}{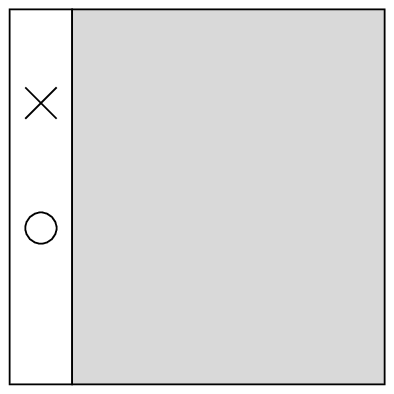}}
&  &
  \xymatrix@C=2cm{\dessin{2.1cm}{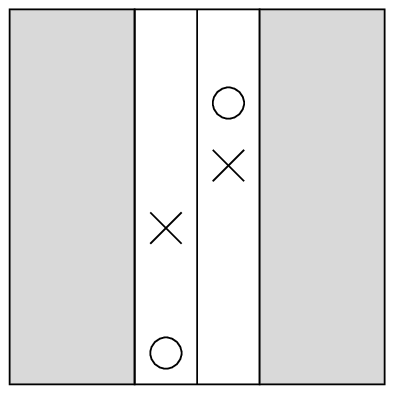} \ar@{<->}[r]^{\dessin{.9cm}{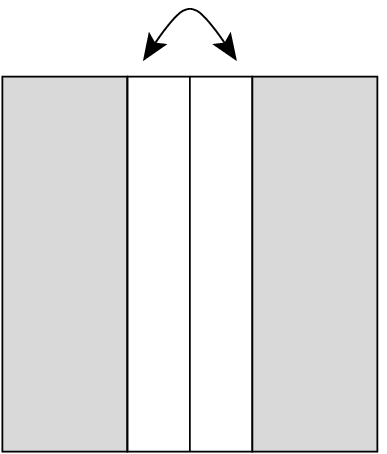}} & \dessin{2.1cm}{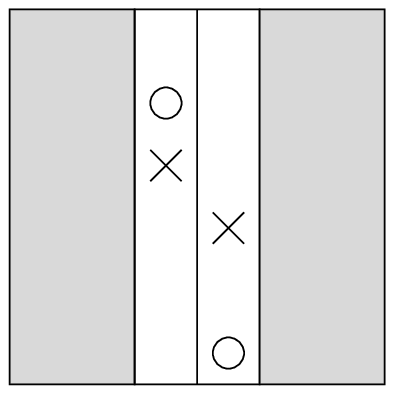}}\\[1cm]
\textrm{Cyclic permutation} && \textrm{Commutation}
\end{array}\\[.6cm]
\begin{array}{c}
  \xymatrix@C=2cm{\dessin{2.1cm}{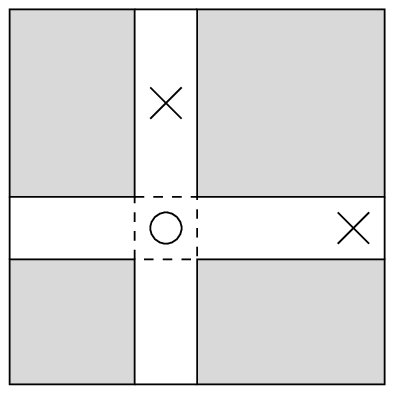} \ar@{<->}[r]^{\dessin{.9cm}{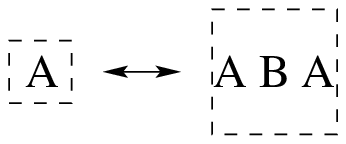}} & \dessin{2.1cm}{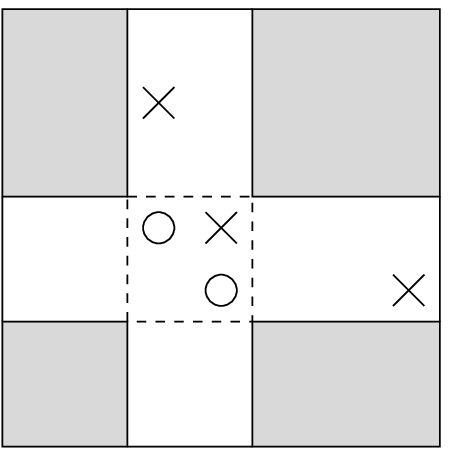}}\\[1cm]
  \textrm{Stabilization/Destabilization}
\end{array}
\end{gather*}
  \caption{Elementary grid diagram moves}
  \label{fig:Moves}
\end{figure}

\begin{proof}[Sketch of proof]
  As pointed out above, any link diagram can be described by a grid. Moreover, for any isotopy of diagram, one can require that all the crossings are rigidly kept orthogonal with vertical overpassing strands except for a finite number of exceptional times when a crossing is turned over around itself. Then, isotopies of diagram can be realized by elementary grid moves; the most tricky move being the exceptional half rolling up:
$$
\xymatrix{\dessin{1cm}{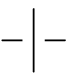} \ \ar@{<->}[r] & \dessin{2cm}{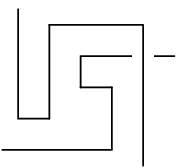} \ar@{<->}[r]  & \dessin{2cm}{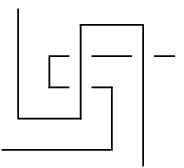} \ar@{<->}[r]& \dessin{2cm}{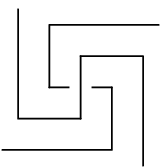}}.
$$
Now, any grid diagram can be modified using elementary grid moves in such a way that the Reidemeister moves can all be realized as follows:
\begin{center}
  \begin{tabular}{p{4cm}p{6cm}p{1.5cm}}
    \begin{tabular}{l} Reidemeister move I \\ (using a (de)stabilization)\end{tabular}: & \centering{$\xymatrix{\dessin{1.8cm}{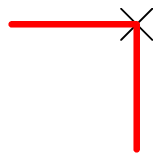} \ \ar@{<->}[r] & \ \dessin{1.8cm}{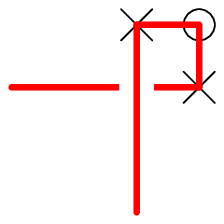}}$;}&
  \end{tabular}
  \begin{tabular}{p{4cm}p{6cm}p{1.5cm}}
    \begin{tabular}{l} Reidemeister move II \\ (using a commutation)\end{tabular}: & \centering{$\xymatrix{\dessin{1.8cm}{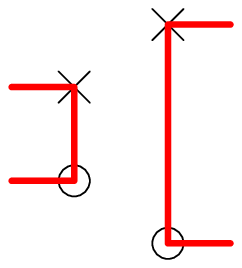} \ \ar@{<->}[r] & \ \dessin{1.8cm}{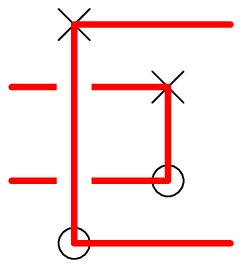}}$;}&
  \end{tabular}
  \begin{tabular}{p{4cm}p{6cm}p{1.5cm}}
    \begin{tabular}{l} Reidemeister move III \\ (using a commutation)\end{tabular}: & \centering{$\xymatrix{\dessin{2cm}{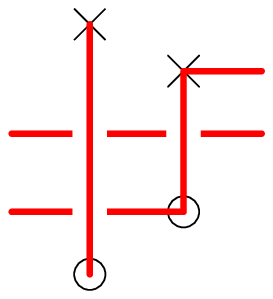} \ \ar@{<->}[r] & \ \dessin{2cm}{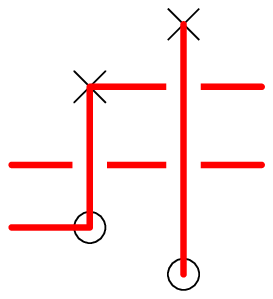}}$.}&
  \end{tabular}
\end{center}
\vspace{-.5cm}
\end{proof}

%%% Local Variables: 
%%% mode: latex
%%% TeX-master: "These"
%%% End:

% Description de l'homologie d'Heegaard-Floer combinatoire
\section{Combinatorial link Floer homology}
\label{sec:Combi}

Now, we will recall briefly the combinatorial construction for link Floer homology given in \cite{MOST}. Nevertheless, we will consider only the case filtrated by $\Z$ since the multi-filtrated refinement does not suit our singular generalization.\\

Let $G$ be a regular grid diagram of size $n$ which represents a link $L$ with $\ell$ connected components. To such a grid, we associate a graded chain complex $C^-(G)$.\\
We will first describe the generators of a module $\widehat{C}^-(G)$. Then we will define two gradings on these generators. At this point, we will be in position to define $C^-(G)$ and, finally, to set a differential $\partial^-$ on it.

\subsection{Generators}
\label{ssec:Generators}

\subsubsection{Generators as sets of dots}
\label{par:DotsGen}

Basically, $\widehat{C}^-(G)$ is generated by all the possible one-to-one correspondances between rows and columns of $G$. Such a generator can be depicted on the grid by drawing a dot at the bottom left corner of the common squares of associated row and column. Then, generators are sets of $n$ dots arranged on the grid lines' intersections such that every line contains exactly one point, except the rightmost and the uppermost ones which do not contain any.

\subsubsection{Generators as permutations}
\label{par:PermGen}

However, as soon as columns have been numbered from the left to the right and rows from the bottom to the top, generators can also be understood as elements of $\SS_n$, the group of permutations of the $n$ first integers.\\

The notation $x$ will denote the first description \ie a set of $n$ dots while $\Gsig{x}$ will denote the associated permutation.\\

\begin{figure}[h]
  $$
    \xymatrix@!0@R=.7cm{&&\dessin{2.3cm}{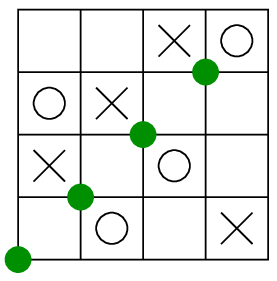}\\ x_1 \ar[ur] \ar[dr] &&\\ &&  \Gsig{x_1}=\textcolor{red}{Id}} \hspace{2cm}  \xymatrix@!0@R=.7cm{&&\dessin{2.3cm}{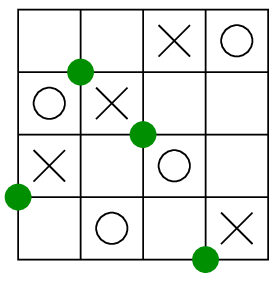}\\ x_2 \ar[ur] \ar[dr] &&\\ && \Gsig{x_2}= \textcolor{red}{\tau_{1,4}\tau_{1,2}}}
  $$
  \caption{Descriptions of the generators}
  \label{fig:Descript}
\end{figure}

\subsection{Gradings}
\label{ssec:Grad}

To define the gradings, we consider the grid as embedded in the $\R^2$--plane, the horizontal and the vertical lines being parallel respectively to the $x$--axis and the $y$--axis. The decorations as well as the dots defining generators are then assimilated to the coordinates of their gravity center.

\subsubsection{Maslov and Alexander gradings}
\label{par:MaslAlex}

First, a few notation are required.

\begin{nota}
  Let $A$ and $B$ be two finite subsets of $\R^2$.\\
  We define $\I(A,B)$ as the number of pairs $((a_1,a_2),(b_1,b_2))\in
  A\times B$ satisfying $a_1<b_1$ and $a_2<b_2$ \ie
$$
\I(A,B):=\#\{(a,b)\in A\times B|a \textrm{ lies in the open
  south-west quadrant of }b\}.
$$
Then, we set $M_B(A):=\I(A,A)-\I(A,B)-\I(B,A)+\I(B,B)+1$.
\end{nota}

Now, we can define a Maslov grading $M$, which will be used as the homological degree, and an Alexander grading $A$ which will induce a filtration.\\

Let $x$ be a generator of $C^-(G)$, we set
\begin{itemize}
\item[-] $M(x):=M_\O(x)$;
\item[-] $A(x):=\frac{1}{2}(M_\O(x)-M_\X(x)) -\frac{n-\ell}{2}$.
\end{itemize}
\index{grading!Maslov}\index{grading!Alexander}\index{filtration!Alexander}
\begin{prop}[\cite{MOST}]\label{PermPreserves}
  The gradings $M$ and $A$ are invariant under cyclic permutations of rows or columns.
\end{prop}

As a corollary, the gradings $M$ and $A$ are well defined even if the grid is considered as a torus one \ie even if we identify its right boundary with its left one and its top boundary with its bottom one.\\

For any grid $G'$, we denote by $\T_{G'}$ this associated torus. However, the local notions of left, right, above, below, horizontal and vertical, as well as the orientation inherited from the planar grid $G'$ are kept.

\subsection{Definition of $C^-(G)$}
\label{ssec:Action}

We label the elements of $\O$ by integers from $1$ to $n$ but this numbering will be totally transparent in the following construction. For each $O_i$ in $\O$, we defined an indetermined $U_{O_i}$, also denoted $U_i$.

\subsubsection{$C^-(G)$ as a $\Z[U_1,\cdots,U_n]$--module}
\label{par:Z[X]-act}

We endow $\Z[U_1,\cdots,U_n]$ with a bigrading by setting
$$
\deg (U_1^{k_1}\cdots U_n^{k_n})=(-2\sum k_i,-\sum k_i)
$$
for all $(k_1,\cdots,k_n)\in \N^n$.\\

Now we define $C^-(G)$ as the bigraded module $\widehat{C}^-(G)\otimes \Z[U_1,\cdots,U_n]$. It can be seen as a bigraded module over the bigraded ring $\Z[U_1,\cdots,U_n]$ where $\deg(U_i)=(-2,-1)$ for all $i\in\llbracket1,n\rrbracket$. The generators are then those given in paragraph \ref{par:DotsGen} or \ref{par:PermGen}.

\subsection{Differential}
\label{ssec:Diff}

The differential can be described in several ways, depending on the point of view we have on the generators of $C^-(G)$. In the light of this thesis, it will be more convenient to define it by counting rectangles on $\T_G$.

\subsubsection{Rectangles}
\label{par:Rect}

  Let $x$ and $y$ be two generators of $C^-(G)$. A \emph{rectangle $\rho$ connecting $x$ to $y$}\index{polygon!rectangle}\index{rectangle|see{polygon}} is an embedded rectangle in $\T_G$ which satisfies:
  \begin{itemize}
  \item[-] edges of $\rho$ are embedded in the grid lines;
  \item[-] opposite corners of $\rho$ are respectively in $x\setminus y$ and $y\setminus x$;
  \item[-] except on $\partial \rho$, the sets $x$ and $y$ coincide;
  \item[-] according to the orientation of $\rho$ inherited from the one of $\T_G$, horizontal boundary components of $\rho$ are oriented from points of $x$ to points of $y$.
  \end{itemize}

\begin{remarque}
  If it a rectangle connecting $x$ to $y$ does exist, then $\Gsig{x}\circ \Gsig{y}^{-1}$ is a transposition.
\end{remarque}

A rectangle $\rho$ is \emph{empty}\index{polygon!empty}\index{empty|see{polygon}} if $\Int(\rho)\cap x=\emptyset$.\\
We denote by $\Rect(G)$ the set of all empty rectangles on $G$ and by $\Rect(x,y)$ the set of those which connect $x$ to $y$.\\

\begin{figure}[h]
  \begin{center}
    $$
  \begin{array}{ccccc}
    \dessin{3cm}{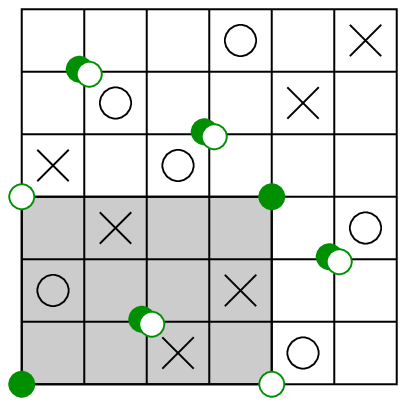} & & \dessin{3cm}{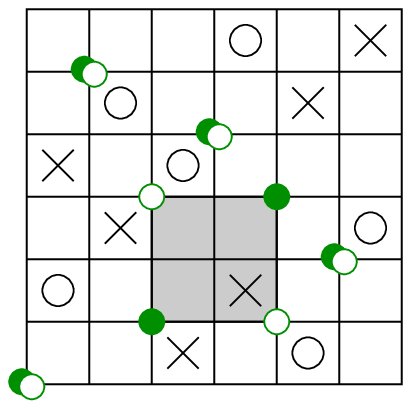} & & \dessin{3cm}{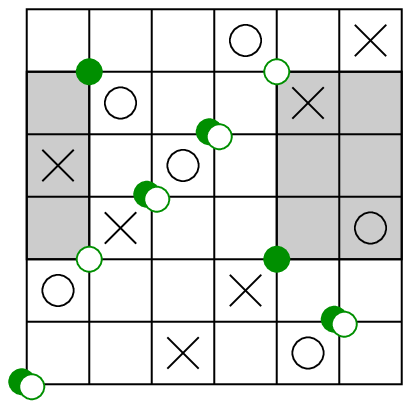}\\[1cm]
    \textrm{a non empty rectangle} && \textrm{an empty rectangle} && \textrm{a torn empty rectangle}
  \end{array}
$$
\caption{Examples of rectangles: {\footnotesize dark dots describe
    the generator $x$ while hollow ones describe $y$. Rectangles are
    depicted by shading. Since a rectangle is embedded in the torus
    and not only in the rectangular grid, it may be ripped in several
    pieces as in the case on the right.}}
\end{center}
  \label{fig:Rect}
\end{figure}

A rectangle $\rho\in \Rect(G)$ is \emph{torn}\index{polygon!torn} if $\rho\cap l\neq \emptyset$ where $l$ is the leftmost vertical line of the grid $G$. The line $l$ may intersect $\rho$ on its boundary.

\begin{prop}[Manolescu-Ozsv\'ath-Szab\'o-Thurston \cite{MOST}]\label{RectGradings}$\ $\\
  Let $x$ and $y$ be two generators of $C^-(G)$ and $\rho$ an element of $\Rect(x,y)$. Then
  \begin{gather*}
    M(x) - M(y) = 1-2\#(\rho\cap \O)\\[.3cm]
    A(x) - A(y) = \#(\rho\cap\X)-\#(\rho\cap \O)
  \end{gather*}
\end{prop}

{\noindent \it Sketch of proof.}  Let $x$ and $y$ be two generators connected by an empty rectangle $\rho$. Since the gradings are invariant by cyclic permutations of the rows and the columns (Prop. \ref{PermPreserves}, \S\ref{par:MaslAlex}), we can assume that $\rho$ is not ripped in pieces when seen on the grid only.\\
{\parpic[r]{$\dessin{2.4cm}{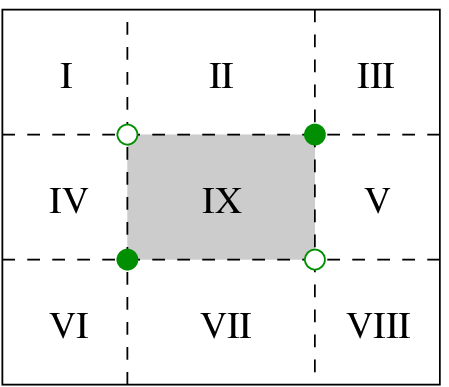}$}
\vspace{-.15cm}The integers $M(x)$, $M(y)$, $A(x)$ and $A(y)$ are computed by counting pairs of points. Their values change only because of pairs involving corners of $\rho$.\\
The only case when both points are corners of $\rho$ gives the $1$--term in $M(x)-M(y)$. Otherwise, we divide the grid in nine areas as shown on the right. The number of pairs involving corners of $\rho$ depends then only on the area where the second point lies. Actually, only points in area IX gives surviving terms. Since they lie in $\rho$, they must be decorations and they appear twice.\hfill $\square$

}

\subsubsection{Differential}
\label{par:Diff}

To define the differential, a map
$$
\func{O_i(\ .\ )}{\{\textrm{embedded polygons in }\T_G\}}{\{0,1\}}
$$
has to be set first for every $O_i$ in $\O$, sending a polygon $\pi$ to $1$ if $O_i \in \pi$ and to $0$ otherwise.\\

We set the map $\func{\p_G^-}{C^-(G)}{C^-(G)}$ as the morphism of $\Z[U_{O_1},\cdots,U_{O_n}]$--modules defined on the generators by
$$
\p_G^-(x)=\sum_{y \textrm{ generator}\textcolor{white}{R}} \sum_{\rho\in \Rect(x,y)} \e(\rho)U_{O_1}^{O_1(\rho)}\cdots U_{O_n}^{O_n(\rho)}\cdot y,
$$
where $\func{\e}{\Rect(G)}{\{\pm 1\}}$ is a map which will be defined in paragraph \ref{par:OrientMap}.

\begin{remarque}
  One can define $\e$ as the constant map which sends every rectangle to $1$ but then, the $\Z$--coefficients should be turned into $\FF_2=\fract \Z / {2\Z}$ ones. 
\end{remarque}

\begin{theo}[Manolescu-Ozsv\'ath-Szab\'o-Thurston \cite{MOST}]\label{D0Preserves}$\ $\\
  The map $\p_G^-$ is a differential which decreases the Maslov grading by $1$ and preserves the filtration induced by the Alexander grading.
\end{theo}
\begin{proof}[Sketch of proof]
  The fact that it decreases the Maslov grading and preserves the Alexander filtration is a consequence of Proposition \ref{RectGradings} (\S\ref{par:Rect}).\\
  To prove that $\p_G^-$ is a differential, we show that all terms in ${\p_G^-}^2$ arise in cancelling pairs. Those terms are juxtapositions of two rectangles.\\
When the rectangles have disjoint sets of corners, the two decompositions differ from the order in which pairs of dots are moved.\\
When they do share a corner, they form a L--shape. As shown in Figure \ref{fig:Bidecomp}, it can be then broken up into rectangles in two different ways.\\

\begin{figure}
  $$  
  \dessin{2.95cm}{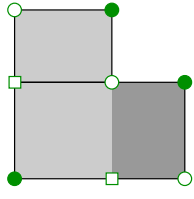}  
  $$
  \caption{The two L--decompositions: {\footnotesize dark dots describe the initial generator while hollow ones describe the final one. The squares describe intermediate states. One decomposition is given by its border and the second by shading of different intensity.}}
  \label{fig:Bidecomp}
\end{figure}

When they do share more than one corner then they share all the corners and their union is an annulus around the torus $\T_G$ which must be of height or width one since it can't contain any dot of the initial generator. Then vertical and horizontal annuli cancel together when they share the same $\O$--decoration.
\end{proof}

\subsection{Orientation}
\label{ssec:Orient}

We will use the description of the sign refinement given by E. Gallais in the last part of \cite{Gallais}.

\subsubsection{Sign assignment}
\label{par:SignAssignment}

A \emph{sign assignment}\index{sign assignment} is a function $\func{\e}{\Rect(G)}{\{\pm 1\}}$ satisfying:
\begin{itemize}
\item[(Sq)] for any four rectangles $\rho_1\in\Rect(x,z)$, $\rho_2\in\Rect(z,y)$, $\rho'_1\in\Rect(x,z')$ and $\rho'_2\in\Rect(z',y)$ where $x$, $y$, $z$ and $z'$ are generators of $C^-(G)$, satisfying 
$$
\rho_1\neq\rho'_1 \ \textrm{ and }\ \rho_1\cup \rho_2=\rho'_1\cup \rho'_2,
$$
then
$$
\e(\rho_1)\e(\rho'_1)\e(\rho_2)\e(\rho'_2)=-1\ ;
$$
\item[(V)] for any two rectangles$\rho_1\in\Rect(x,z)$ and $\rho_2\in\Rect(z,x)$ such that $\rho_1\cup \rho_2$ is a vertical annulus then
$$
\e(\rho_1)\e(\rho_2)=-1\ ;
$$
\item[(H)] for any two rectangles $\rho_1\in\Rect(x,z)$ and $\rho_2\in\Rect(z,x)$ such that $\rho_1\cup \rho_2$ is an horizontal annulus then
$$
\e(\rho_1)\e(\rho_2)=1.
$$
\end{itemize}

\begin{theo}[Manolescu-Ozsv\'ath-Szab\'o-Thurston \cite{MOST}]\label{SignAssigment}$\ $\\
If $\func{\e}{\Rect(G)}{\{\pm 1\}}$ is a sign assignment, then Theorem \ref{D0Preserves} (\S\ref{par:Diff}) holds. Moreover, if $\e$ and $\e'$ are two sign assignments, then the associated chain complexes are isomorphic.
\end{theo}

\subsubsection{Spin extension of $\SS_n$}
\label{par:SpinExt}

If we want to deal with $\Z$--coefficients, we need to give some signs to permutations. That is why we consider $\Sn$ the \emph{spin extension of $\SS_n$}\index{spin extension of Sn@spin extension of $\SS_n$}\index{Sn@$\Sn$|see{spin extension of $\SS_n$}} defined as:
$$
\Sn = \left\langle z,\widetilde{\tau}_{ij} \textrm{ for } 1\leq i\neq j\leq j \left|
    {\scriptsize \begin{tabular}{ll}
        \begin{tabular}{l}
          $z^2=1$ \\[.05cm]
          $z\widetilde{\tau}_{ij}= \widetilde{\tau}_{ij}z=\widetilde{\tau}_{ji}, \widetilde{\tau}_i^2=z$\\[.05cm]
          $\widetilde{\tau}_{ij}\widetilde{\tau}_{i'j'}=z\widetilde{\tau}_{i'j'}\widetilde{\tau}_{ij}$\\[.05cm]
          $\widetilde{\tau}_{ij}\widetilde{\tau}_{jk}\widetilde{\tau}_{ij}=\widetilde{\tau}_{jk}\widetilde{\tau}_{ij}\widetilde{\tau}_{jk}=\widetilde{\tau}_{ik}$
        \end{tabular}
        &
        \begin{tabular}{l}
          $1\leq i\neq j\neq k \leq n$\\[.1cm]
          $1\leq i'\neq j'\leq n$ \\[.1cm]
          $\{i,j\}\cap \{i',j'\}=\emptyset$
        \end{tabular}      
      \end{tabular} } \right.
\right\rangle.
$$

The generators $\widetilde{\tau}_{ij}$ are called \emph{transpositions}\index{transposition}.\\

We denote by $\iota_n$ the natural injection of $\Sn$ into $\widetilde{\SS}_{n+1}$.\\

\begin{prop}
  For $n\geq 4$, the group $\Sn$ is a non-trivial extension of $\SS_n$ by $\fract \Z / {2\Z}$
$$
\xymatrix{1 \ar[r] & \fract \Z / {2\Z} \ar@{^(->}[r]^i & \Sn \ar@{->>}[r]^p & \SS_n \ar[r]& 1}
$$
where $i$ maps $1$ to $z$ and where $p$ maps $\widetilde{\tau}_{ij}$ to the transposition $\tau_{i,j}$ for all $1\leq i \neq j\leq n$ and $z$ to $Id$.
\end{prop}

\begin{lemme}\label{SpinProp}
  For all distinct positive integers $i$, $j$, $k$ and $l$, the transpositions satisfy:
  \begin{enumerate}
  \item[i)] $\widetilde{\tau}_{ij}\widetilde{\tau}_{ji}=1$;
  \item[ii)] $\widetilde{\tau}_{ij}\widetilde{\tau}_{ki}\widetilde{\tau}_{i,j}=\widetilde{\tau}_{li}\widetilde{\tau}_{kj}\widetilde{\tau}_{li}$;
  \item[iii)] $\widetilde{\tau}_{ij}\widetilde{\tau}_{jk}\widetilde{\tau}_{ki}=\widetilde{\tau}_{kj}$. 
  \end{enumerate}
\end{lemme}
\begin{proof}
  Only the third point may need a proof:
  \begin{eqnarray*}
    \widetilde{\tau}_{ij}\widetilde{\tau}_{jk}\widetilde{\tau}_{ki} & = & \widetilde{\tau}_{ij}\widetilde{\tau}_{jk}\widetilde{\tau}_{ij}\widetilde{\tau}_{ji}\widetilde{\tau}_{ki}\\
    & = & \widetilde{\tau}_{ik}\widetilde{\tau}_{ji}\widetilde{\tau}_{ki}\\
    & = & z\widetilde{\tau}_{ik}\widetilde{\tau}_{ji}\widetilde{\tau}_{ik} = z\widetilde{\tau}_{jk}=\widetilde{\tau}_{kj}.
  \end{eqnarray*}
\end{proof}

\subsubsection{Section $s$}
\label{par:Section}

Now, we define a section $\func{s:=s_n}{\SS_n}{\Sn}$
$$
\xymatrix{1 \ar[r] & \fract \Z / {2\Z} \ar@{^{(}->}[r]^i & \Sn \ar@{->>}[r]^p & \SS_n \ar[r] \ar@/^/[l]^{s_n} & 1}.
$$
by induction on $n$.
  \begin{description}
  \item[Case $n=1$] $s_1(Id)=1$.
  \item[Case $n>1$]$\ $
\begin{itemize}
    \item[-] $\forall i\in \llbracket1,n-1\rrbracket$, $s_n\big(\tau_{in}\big)=\widetilde{\tau}_{i,n}$;
    \item[-] $\forall \sigma\in \SS_n$, $s_n(\sigma)=\iota_{n-1}\big(s_{n-1}(\sigma\tau_{\sigma^{-1}(n),n})\big)s_n(\tau_{\sigma^{-1}(n),n})$ where the permutation $\sigma\tau_{\sigma^{-1}(n),n}$ is seen as an element of $\SS_{n-1}$ since it lets $n$ fixed.
    \end{itemize}
  \end{description}

\begin{prop}
  The map $s$ is a section \ie $p\circ s=\textnormal{Id}$.\\
In particular, for every pair of permutations $(\sigma,\sigma')\in\SS_n^2$, the term $s(\sigma)^{-1}s(\sigma')s(\sigma^{-1}\sigma')^{-1}$ is a power of $z$.
\end{prop}

\subsubsection{Orientation map}
\label{par:OrientMap}

We define the map $\func{c}{\SS_n\times\SS_n}{\fract \Z /{2\Z}}$ by
$$
\forall (\sigma,\sigma')\in\SS_n^2,s(\sigma)^{-1}s(\sigma')s(\sigma^{-1}\sigma')^{-1}=z^{c(\sigma,\sigma')}. 
$$

\begin{remarque}
  Actually, $c$ is a $2$-cocycle in $C^2(\SS_n,\fract \Z / {2\Z})$.
\end{remarque}

The orientation map $\e$ will depend as well on the way a rectangle is split into pieces when seen on the grid $G$ and not on the torus $\T_G$.\\

 If $\rho$ is an empty rectangle which connects $x$ to $y$, then we set
$$
\e(\rho)=\left\{\begin{array}{ll}(-1)^{1+c(\Gsig{x},\Gsig{y})}&\textrm{if }\rho\textrm{ is torn}\\(-1)^{c(\Gsig{x},\Gsig{y})}&\textrm{otherwise}\end{array}\right.
$$

\begin{theo}[Gallais \cite{Gallais}]
  The map $\func{\e}{\Rect(G)}{\{\pm 1\}}$ is a sign assigment as defined in paragraph \ref{par:SignAssignment}.
\end{theo}

\subsubsection{Alternative description}
\label{par:AltDiff}

Since we will use it in the paragraph \ref{par:SemiSingMap}, we should say a few words about the alternative description of $\p_G^-$ given in \cite{Gallais}.\\

Let assume that $C^-(G)$ is no more generated by elements of $\SS_n$ but by elements of $\Sn$ up to the relation $z\sim -1$. As a set, these generators are straightly in bijection with $\SS_n$ but the group law is now endowed with some twisted signs.\\

Let $\widetilde{\sigma}_x$ and $\widetilde{\sigma}_y$ be two such generators of $C^-(G)$. Furthermore, suppose that $\widetilde{\sigma}_x^{-1}\widetilde{\sigma}_y$ is a transposition $\widetilde{\tau}_{ij}$ for some $i\neq j\in\llbracket1,n\rrbracket$.\\
Let $\rho\in\Rect(x,y)$ be the rectangle joining $x$ to $y$ of which the bottom left corner is the $i^\textrm{th}$ dot of $x$ counting from left to right. Then we set
$$
\Rect(\widetilde{\sigma}_x,\widetilde{\sigma}_y)=\left\{
  \begin{array}{ll}
    \{\widetilde{\tau}_{ij}\} & \textrm{ if }\rho\textrm{ is empty and not torn}\\[.1cm]
      \{\widetilde{\tau}_{ji}\} & \textrm{ if }\rho\textrm{ is empty and torn}\\[.1cm]
    \emptyset & \textrm{ otherwise.}
  \end{array}
\right.
$$

Now, the differential $\p_G^-$ can be defined as well on generators by
$$
\p_G^-(\widetilde{\sigma}_x)=\sum_{\widetilde{\sigma}_y \textrm{ generator}\textcolor{white}{R}} \sum_{\widetilde{\tau}\in \Rect(\widetilde{\sigma}_x,\widetilde{\sigma}_y)} U_{O_1}^{O_1(\rho)}\cdots U_{O_n}^{O_n(\rho)}\cdot \widetilde{\sigma}_x\widetilde{\tau}.
$$

% \subsubsection{Twisting signs}
% Let $p$ be any point of the grid $G$.\\
% A sign assignmen $\e$ being given, one can \emph{twist it around $p$}. Actually, let $p_1$, $p_2$, $p_3$ and $p_4$ be the four corners of a small square neighborhood of $p$ with edges paralel to the $x$ and $y$--axes, then we can define $\func{\e_p}{\Rect(G)}{\{\pm 1\}}$ by
% $$
% \e_p(\rho)=(-1)^{\#\rho\cap\{p_1,p_2,p_3,p_4\}}\e(\rho).
% $$
% It is direct to check that $\e_p$ is still a sign assignment. According to Theorem \ref{SignAssigment} (\S\ref{par:SignAssignment}), it gives rise to an isomrophic chain complex. Howver, twisting signs around $p$ is non trivial only if $p$ is a grid lines intersection.

\subsection{Link Floer homologies}
\label{ssec:Homology}

\subsubsection{Filtrations and gradings}
\label{par:FloerHomologies}

The chain complex $C^-(G)$ is endowed with several filtrations: the Alexander one and, moreover, for each $O\in \O$, setting $\F_i^OC^-(G):=U_O^i.C^-(G)$ for all non negative integer $i$ defines another filtration on $C^-(G)$ which is preserved by the differential. We call it \emph{$O$--filtration}\index{filtration!Ofiltration@$O$--filtration}. For each of them, one can consider the associated graded chain complex.\\

It is a corollary of Proposition \ref{RectGradings} (\S\ref{par:Rect}) that the graded differential associated to the Alexander filtration is obtained from the definition of $\p_G^-$ given in the paragraph \ref{par:Diff} (or \ref{par:AltDiff}), adding the condition that the rectangles do not contain any $X$.\\

Concerning the $O$--filtration associated to $O\in\O$, it can be read directly in the definition of $\p_G^-$ that the associated graded differential is obtained when considering only those rectangles which do not contain $O$. Equivalently, the differential is obtained when sending the variable $U_O$ to zero.\\

By construction, all the filtrations are bounded below.

\subsubsection{Table of homologies}
\label{par:TableHomologies}
Here is a table of the notation for the different homologies thus obtained.\\

For the sake of clarity, we assume that the elements of $\O$ are numbered in such a way that the $\ell$ first ones belong to the $\ell$ different components of $L$.

\begin{center}
  \begin{tabular}{|c|cccc|}
    \cline{2-5}
    \multicolumn{1}{c|}{} & \multicolumn{4}{|c|}{Alexander filtration}\\
    \cline{2-5}
    \multicolumn{1}{c|}{} & filtrated & & & graded\\ 
    \hline
{\Large \strut} -- & $H_*^-(G)$ & \multicolumn{2}{c}{$\leadsto$} & ${HL}_*^-(G)$ \\ 
    \hline 
 {\Large \strut}    $U_1=\cdots=U_\ell=0$ & $\widehat{H}_*(G)$ & \multicolumn{2}{c}{$\leadsto$} & ${\widehat{HL}}_*(G)$ \\
    \hline
 {\Large \strut}    $U_1=\cdots=U_n=0$ & $\widetilde{H}_*(G)$ & \multicolumn{2}{c}{$\leadsto$} & ${\widetilde{HL}}_*(G)$ \\
    \hline
  \end{tabular}
\end{center}
\index{homology!knot Floer}\index{homology!link Floer}

\subsubsection{Link invariants}
\label{par:LinkInv}

Let $V$ be a bigraded module generated by two elements, one of degree $(0,0)$ and one of degree $(-1,-1)$.
\begin{prop}[\cite{MOST}]\label{UiUj}
   $\widetilde{HL}_*(G)\equiv \widehat{HL}_*\otimes V^{\otimes (n-\ell)}$.
\end{prop}
\begin{proof}[Sketch of proof.]
  First, we prove that, for any $i,j\in\llbracket 1,n\rrbracket$, the multiplications by $U_i$ and by $U_j$ are homotopic maps as soon as $O_i$ and $O_j$ belong to the same component of the link. This can be done gradually. First we consider $i$ and $j$ such that there is an $X$--decoration, denoted by $X_*$, which lies in the same row than $O_i$ and the same column than $O_j$. Then we can set the $\Z[U_{O_1},\cdots,U_{O_n}]$--linear map
$$
\func{H}{C^-(G)}{C^-(G)}
$$
defined on the generators by
$$
\p_G^-(x)=\sum_{y \textrm{ generator}\textcolor{white}{R}} \sum_{\substack{\rho\in \Rect(x,y)\\ X_*\in\rho}} \e(\rho)U_{O_1}^{O_1(\rho)}\cdots U_{O_n}^{O_n(\rho)}\cdot y.
$$
Most terms in $\p^-_G\circ H- H\circ\p^-_G$ vanish for the same reasons than in the proof of \ref{D0Preserves} (\S\ref{par:Diff}). However, among vertical and horizontal thin annuli, only those which contain $X_*$ do arise. As a matter of fact, $\p_G^-\circ H - H\circ \p_G^-$ is equal to the multiplication by $U_i-U_j$.\\
Such maps can now be added in order to prove that multiplications by $U_i$ and $U_j$ are homotopic for any suitable couple $(i,j)$.\\

Now, some algebra says that $\fract {C^-(G)}/{(U_i,U_j)}$ is quasi-isomorphic to the mapping cone of the multiplication by $\func{U_j}{\left(\fract {C^-(G)}/{U_i}\right)[-2]\{-1\}}{\fract {C^-(G)}/{U_i}}$. But since this last map is homotopic to the zero map, we obtain that $\fract {C^-(G)}/{(U_i,U_j)}$ is quasi-isomorphic to $\fract {C^-(G)}/{U_i}\otimes V$.\\
Similarly, each time we make the quotient by $U_k$ for some $k\in\llbracket 1,n\rrbracket$ such that $O_k$ belongs to the same component than $O_i$, the homology is tensorized by $V$.
\end{proof}

\begin{theo}[Manolescu-Ozsv\'ath-Sarkar \cite{MOS}]$\ $\\
The data $(C^-(G),\p_G^-)$ is a chain complex for the Heegaard-Floer homology $CF^-(S^3)$ over a Heegaard diagram with $n$ basepoints and with grading induced by $M$. Moreover the filtration induced by $A$ coincides with the link filtration of $CF^-(S^3)$.
\end{theo}

\begin{cor}
   $\widetilde{HL}_*(G)\equiv \widehat{HFK}_*(L)\otimes V^{\otimes (n-\ell)}$.
\end{cor}

\begin{cor}
  The homology $\widehat{HL}_*(G)$ is an invariant of the link associated to $G$. 
\end{cor}

\begin{theo}[\cite{MOST}]
  The bigraded homology $\widehat{HL}_*(G)$ is finitely generated over $\Z$ and categorifies the Alexander polynomial in the sense that
$$
\Delta(L)(q) = \xi_\gr(\widehat{HL}_*(G)).
$$
\end{theo}

%%% Local Variables: 
%%% mode: latex
%%% TeX-master: "These"
%%% End: 

% Transition
\vspace{2cm}
Now the background has been set, we can generalize it to the singular case.

%%% Local Variables: 
%%% mode: latex
%%% TeX-master: "These"
%%% End: 

% Definition de l'homologie d'Heegaard-Floer pour les enterelacs singuliers
\chapter{Singular link Floer homology}
\label{chap:Singular}

% {%\small
% Ce chapitre généralise la construction précédente au cas singulier.\\

% Nous commençons par généraliser les diagrammes en grilles au cas des entrelacs singuliers. Pour cela, tout en conservant la condition sur les lignes, nous rajoutons des colonnes singulières contenant deux ronds et deux croix de telle sorte que les deux décorations du milieu soient chacune encadrée par des décorations de types différents. Pour obtenir l'entrelacs associé, on relie alors la première et la troisième décorations d'une colonne singulière ainsi que la deuxième et la quatrième par des trait légèrement incurvé simultanément vers la droite ou la gauche, créant ainsi autant de points doubles qu'il y a de colonnes singulières. On poursuit ensuite le protocole décrit dans le cas régulier.}

% \newpage

% Présentation en grilles singulières
 \section{Singular grid diagrams}
\label{sec:SingularGrid}

We generalize grid diagrams to singular link cases.

\subsection{Singular columns}
\label{ssec:SingColumns}

\subsubsection{Singular grid diagrams}
\label{par:SGrid}

  A \emph{singular grid diagram}\index{diagram!grid!singular} $G$ of size $(n,k) \in \N^*\times\N$ is a $(n\times (n+k))$--grid  whose squares may be decorated by a $O$ or by an $X$ in such a way that each column and each row contain exactly one $O$ and one $X$, except for $k$ columns which contain exactly two $O$'s and two $X$'s. Moreover, in these $k$ singular columns, two decorations surrounding a third one must be of different kinds.\\

As in the regular case, every singular grid diagram gives rise to a singular link. The process is almost identical. First join the decorations in regular columns. For singular ones, connect the uppermost decoration to the third one and the second to the lowermost by vertical lines slightly bended to the right (or, equivalently to the left) in such a way that the two curves intersect in one singular double point. Then join again the decorations in rows, taking care to underpass vertical strands when necessary.

\begin{figure}[h]
   $$
  \dessin{2cm}{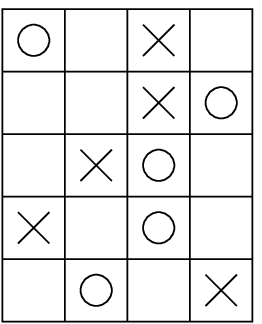} \hspace{.5cm} \leadsto \hspace{.5cm}\dessin{2cm}{SGridKnot} \hspace{.5cm} \leadsto \hspace{.5cm} \dessin{2.3cm}{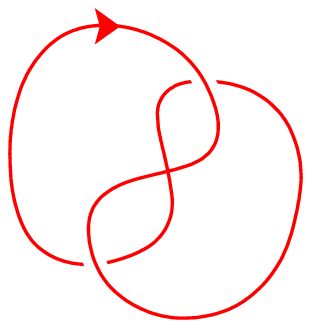}
  $$
  \caption{From singular grid diagrams to singular knots}
  \label{fig:SG->SK}
\end{figure}

\begin{prop}
  Every singular link can be described by a singular grid diagram.
\end{prop}

\begin{proof}
  Consider a planar diagram for a given singular link with one double point and choose a way to desingularize it. Now, according to the process given previously, consider a grid diagram which corresponds to this regular diagram.\\

{\parpic[r]{$\dessin{2.15cm}{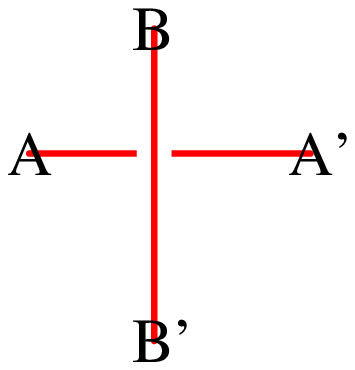}$\hspace{.65cm}}
{\vspace{.5cm}Then the singular point appears as a regular crossing, \ie with four decorations arranged within the cross pattern depicted on the right where $A$ and $A'$, as well as $B$ and $B'$, are decorations of different kinds.\vspace{.4cm}}

\parpic[r]{$\dessin{2.15cm}{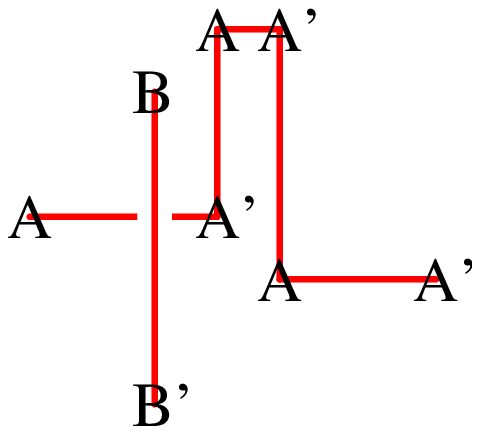}$ \hspace{.5cm}}
{\vspace{.7cm}Perform two stabilizations and a few commutations.\vspace{1.2cm}}

\parpic[r]{$\dessin{2.15cm}{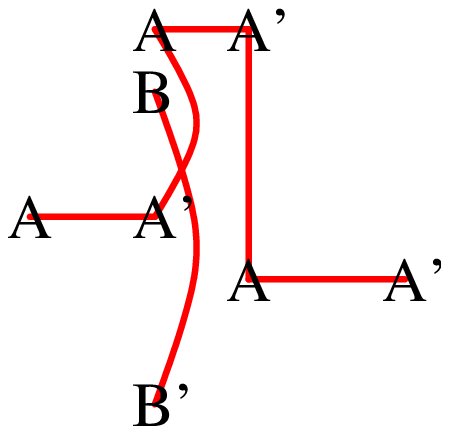}$ \hspace{.5cm}}
{\vspace{.7cm}Finally, re-singularize the link by merging the $B$ and the leftmost $A$--columns.\vspace{.6cm}}

}

{\parpic[l]{\hspace{.5cm} $\dessin{2.15cm}{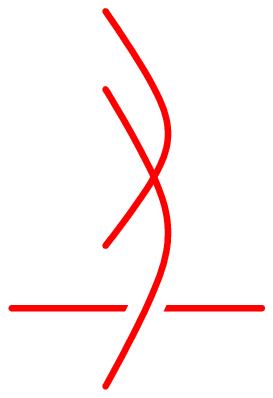}$}
{\vspace{.6cm}\noi In the case of a singular link with more than one doubles points, it may occur that the vertical strand is already part of a singular column.\vspace{.7cm}}

\parpic[l]{\hspace{.5cm} $\dessin{2.15cm}{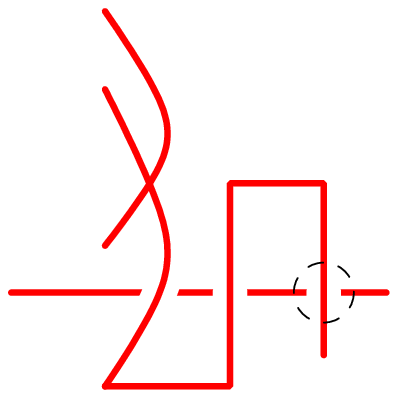}$}
{\vspace{.6cm}\noi And then, we are back to the precedent case by performing, here again, two stabilizations and a few commutations.\vspace{1cm}}

}
\end{proof}

\subsubsection{Singular grid and elementary regular grid moves}
\label{par:SingGridEtElemMoves}

Obviously, circular permutations and commutations of rows or columns, as defined in paragraph \ref{par:ElemMoves}, leave invariant the associated singular link, even if it involves a singular column.\\

Concerning (de)stabilizations, things are not that easy. Actually, as shown below, some destabilizations, involving a singular column may modify the associated link.
$$
\xymatrix@C=1.5cm{\dessin{2.5cm}{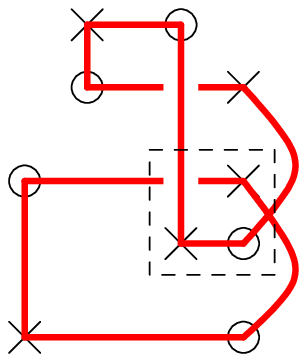}\ar@{<->}[r]|{\dessin{.5cm}{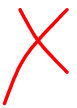}}&\dessin{2.5cm}{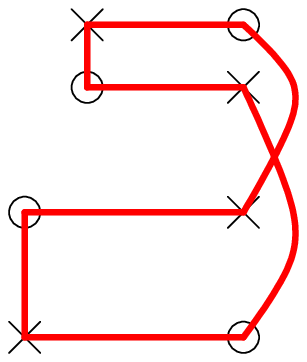}}
$$
In order to avoid this phenomenon, we require that the intersection of the $(2\times 2)$--square involved in a (de)stabilization with a singular column contains at most one decoration.
In other words, two adjacent decorations in the $(2\times 2)$--square cannot belong to a singular column.

\parpic[r]{\xymatrix@C=.5cm{\dessin{2.2cm}{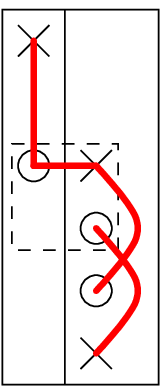}\ar@{<->}[r]&\dessin{2.2cm}{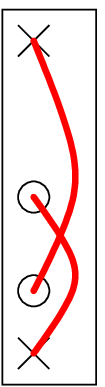}} \hspace{.5cm}}
{\vspace{.2cm}
\begin{remarque}
  Nevertheless, some stabilizations which are henceforth banned can leave the link invariant. An example is given on the right. However, they are less natural in the sense that they involve (trivially) a decoration which does not belong to the strand we are putting a bend in.
\end{remarque}

}
\parpic[l]{\xymatrix@C=.5cm{\dessin{2.2cm}{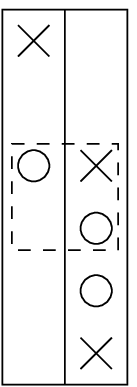}\ar@{<->}[r]&\dessin{2.2cm}{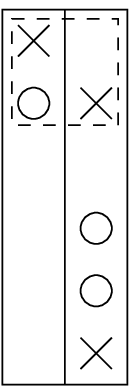}} \hspace{.5cm}}
{\vspace{.2cm}But actually, as shown on the left, those moves can be replaced by more justified ones using a few rows commutations.\\
Even if we will not use it in this thesis, the same trick can be used to replace a stabilization affecting a $O$ decoration by a stabilization affecting an $X$ one.

}
\vspace{.3cm}

\subsubsection{Flip moves and co}
\label{par:Flip}

Now, we introduce the \emph{flip move}\index{moves!grid!flip} which is a grid move involving a singular column. It corresponds to the vertical reversing of a special $(4\times5)$--subgrid as shown in Figure \ref{fig:FlipMove}.

\begin{figure}[!h]
  $$
  \begin{array}{c}
    \xymatrix@C=2cm{\dessin{2.5cm}{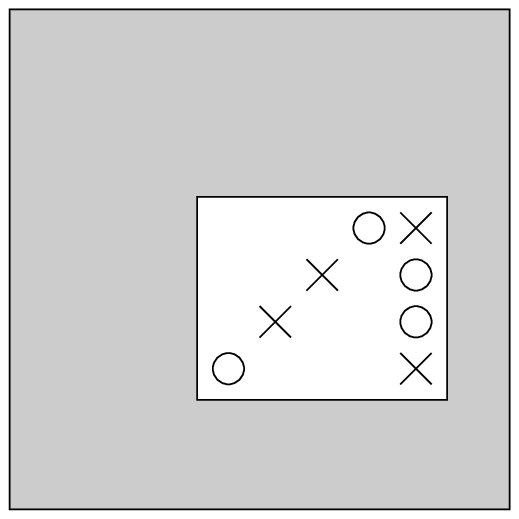} \ar@{<->}[r]^{\dessin{.9cm}{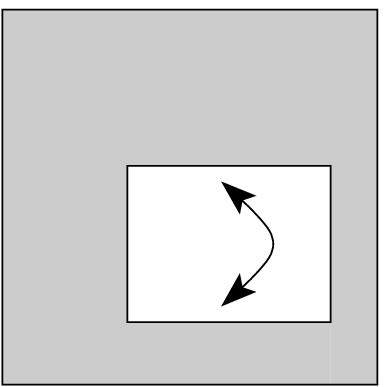}} & \dessin{2.5cm}{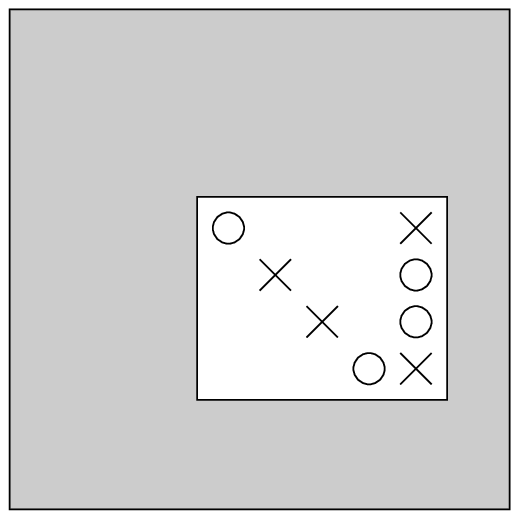}}\\[1.5cm]
    \xymatrix@C=2cm{\dessin{2.5cm}{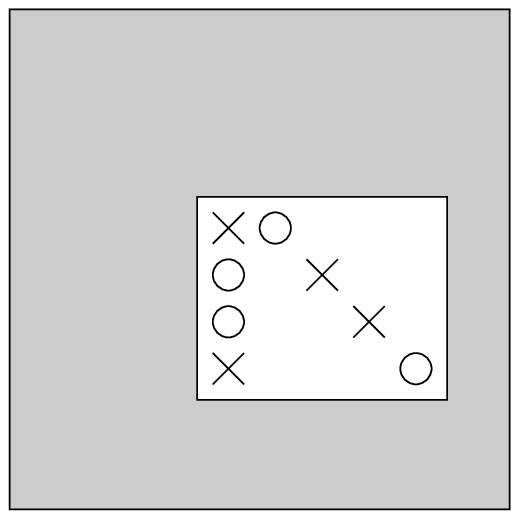} \ar@{<->}[r]^{\dessin{.9cm}{FlipAct}} & \dessin{2.5cm}{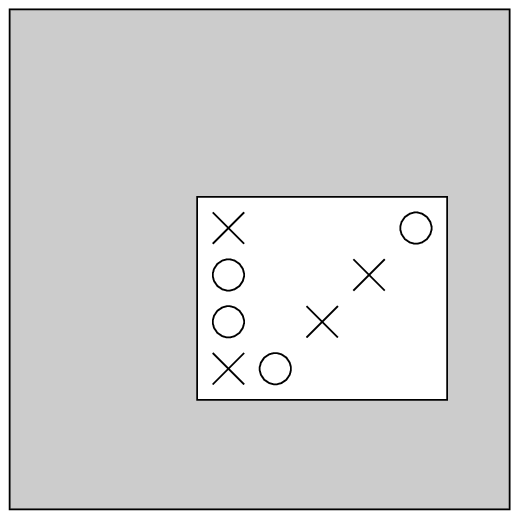}}
  \end{array}
  $$
  \caption{Flip moves}
  \label{fig:FlipMove}
\end{figure}

Actually, there are many moves which are equivalent to flip moves. The following lemma gives some of them.

\begin{lemme}\label{EquivalentMoves}
  Up to regular grid moves, all the following moves are equivalent:
$$
\xymatrix@!0@C=6.3cm@R=3.3cm{
\fbox{$\dessin{1.9cm}{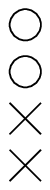} \sim \dessin{1.9cm}{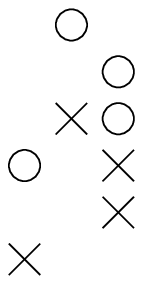}$} \ar@{=>}[r]^{\textrm{\ref{subfig:A}}} &
\fbox{$\dessin{1.9cm}{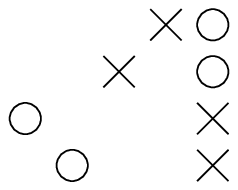} \sim \dessin{1.9cm}{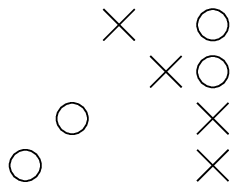}$} \ar@{=>}[d]^{\textrm{\ref{subfig:B}}} \\
\fbox{$\dessin{1.9cm}{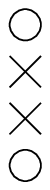} \sim \dessin{1.9cm}{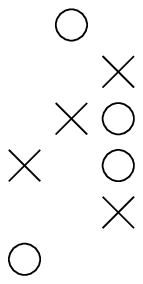}$} \ar@{=>}[u]_{\textrm{\ref{subfig:D}}} \ar@{=>}[d]^{\textrm{\ref{subfig:E}}} &
\fbox{$\dessin{1.9cm}{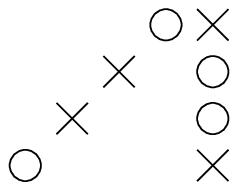} \sim \dessin{1.9cm}{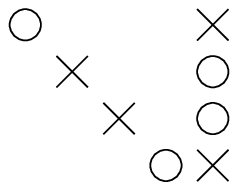}$} \ar@{=>}[l]_{\textrm{\ref{subfig:C}}} \\
\fbox{$\dessin{1.9cm}{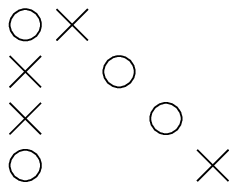} \sim \dessin{1.9cm}{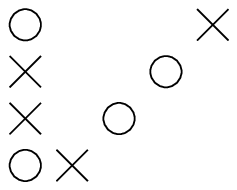}$} \ar@{=>}[r]^{\overline{\textrm{\ref{subfig:C}}}}  &
\fbox{$\dessin{1.9cm}{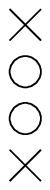} \sim \dessin{1.9cm}{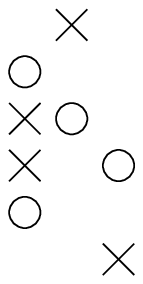}$} \ar@{=>}[u]_{\overline{\textrm{\ref{subfig:E}}}}  \ar@{=>}[d]^{\overline{\textrm{\ref{subfig:D}}}}  \\
\fbox{$\dessin{1.9cm}{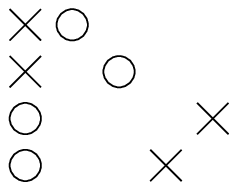} \sim \dessin{1.9cm}{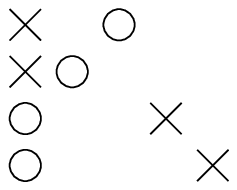}$} \ar@{=>}[u]_{\overline{\textrm{\ref{subfig:B}}}} &
\fbox{$\dessin{1.9cm}{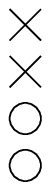} \sim \dessin{1.9cm}{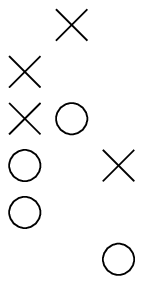}$} \ar@{=>}[l]_{\overline{\textrm{\ref{subfig:A}}}}}
$$
  Another set of equivalent moves is obtained by swapping the nature of every decoration.
\end{lemme}

Every logical implication is labeled by the figure illustrating its proof. Overlining means that the proof pictures have to be horizontally reflected and the nature of decorations swapped.

\begin{figure}[p]
  \centering
  \subfigure[]{
    $$
    \xymatrix@!0@C=2.5cm@R=3.1cm{
      \dessin{2cm}{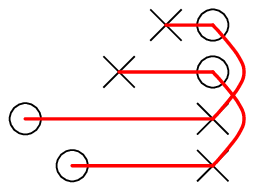} \ar@{<=>}[dr]|(.55){}="Nya1" \ar@{<=>}[rr]|{}="Nya2"&& \dessin{2cm}{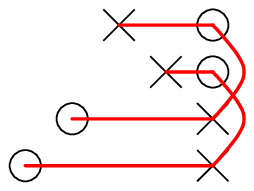} \\
      & \dessin{2cm}{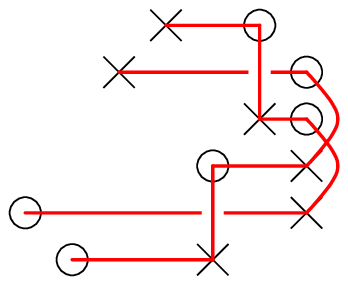} \ar@{<->}[ur] &
      \ar@{.>}@/_.4cm/"Nya1"!<.2cm,.2cm>;"Nya2"!<0cm,-.2cm>|(.55){\ \Uparrow\ }}
    $$
    \label{subfig:A}}
  \subfigure[]{
  $$
  \xymatrix@!0@C=1.6cm@R=1.2cm{
    &&& \dessin{2cm}{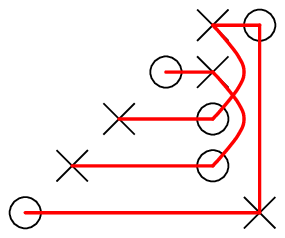} \ar@{<->}[rrrd] &&&  \\
    \dessin{2cm}{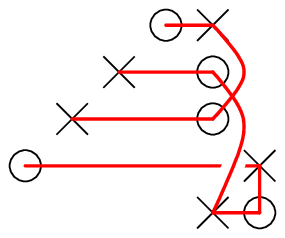} \ar@{<->}[rrru]|{\ r \ }  \ar@{<->}[rrdd] &&&&&& \dessin{2cm}{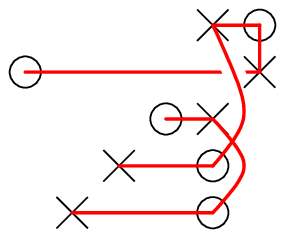} \ar@{<->}[lldd]\\
    &&&&&& \\
    && \dessin{2cm}{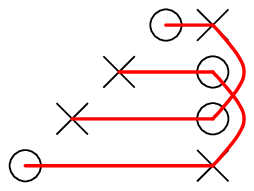}  \ar@{<=>}[lldd]|{}="Nya1" && \dessin{2cm}{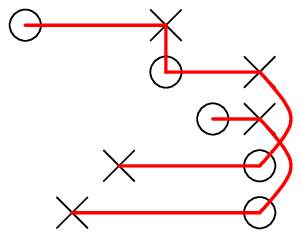}  \ar@{red}@{<=>}[rrdd]|{}="Nya2" && \\
    &&&&&& \\
    \dessin{2cm}{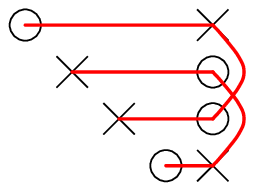}  \ar@{<->}[rrrdd] &&&&&& \dessin{2cm}{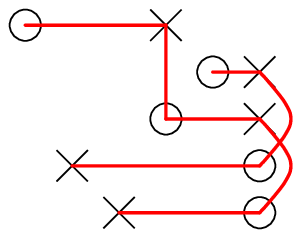}  \ar@{<->}[llldd]|{\ r \ } \\
    &&&&&& \\
    &&&\dessin{2cm}{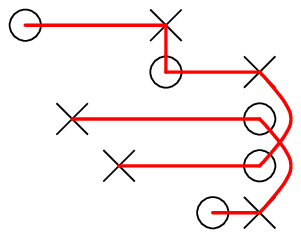} &&&
  \ar@{<.}@/_.9cm/"Nya1"!<.18cm,-.13cm>;"Nya2"!<-.18cm,-.13cm>|{\ \Leftarrow\ }}
  $$
    \label{subfig:B}}
  \caption{Implications between singular grid moves: {\footnotesize Simple arrows stand for a combination of commutations and cyclic permutations of the rows when they are $r$--labeled, of commutations and cyclic permutations of the columns when they are $c$--labeled and of commutations and (de)stabilizations when they are unlabeled. Double arrows stand for singular grid moves.}}
\end{figure}

\addtocounter{figure}{-1}
\addtocounter{subfigure}{2}
\begin{figure}[p]
  \centering
  \subfigure[]{
    $$
    \xymatrix@!0@C=1.8cm@R=1.4cm{
      &\dessin{2cm}{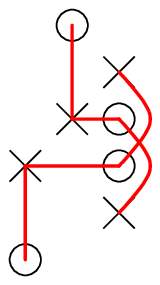} \ar@{<=>}[rr]|{}="Nya1" && \dessin{2cm}{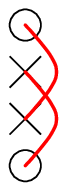} \ar@{<->}[rdd]& \\
      &&&&\\
      \dessin{2cm}{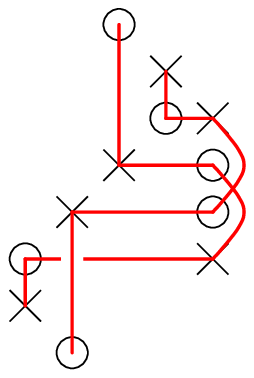} \ar@{<->}[uur] &&&& \dessin{2cm}{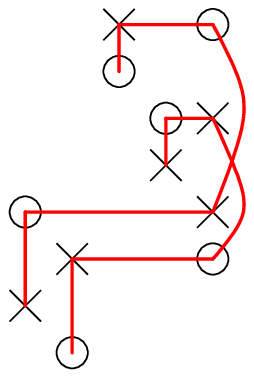} \ar@{<->}[dll]|{\ r\ } \\
      && \dessin{2cm}{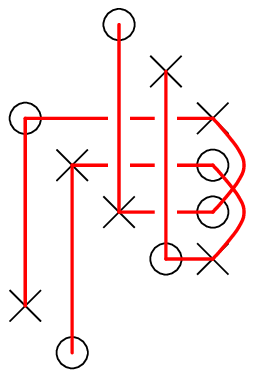} \ar@{<=>}[ull]|{}="Nya2" &&
      \ar@{<.}@/^.5cm/"Nya1"!<0cm,-.2cm>;"Nya2"!<.2cm,.2cm>|{\ \Uparrow\ }}
    $$
    \label{subfig:C}}
  \subfigure[]{
  $$
  \xymatrix@!0@C=1.6cm@R=1.2cm{
    && \dessin{2cm}{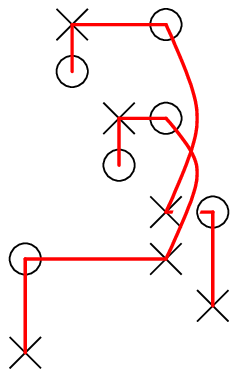} \ar@{<->}[rr]|{\ r \ }  &&  \dessin{2cm}{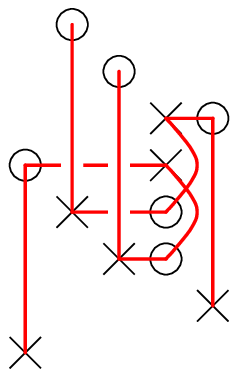} \ar@{<->}[rrd]|{\ r \ } &&  \\
    \dessin{2cm}{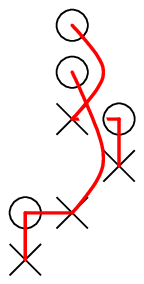} \ar@{<->}[rru] \ar@{<->}[rrdd] &&&&&& \dessin{2cm}{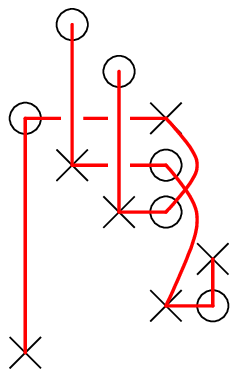} \ar@{<->}[lldd]\\
    &&&&&& \\
    && \dessin{2cm}{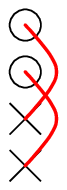}  \ar@{<=>}[lldd]|{}="Nya1" && \dessin{2cm}{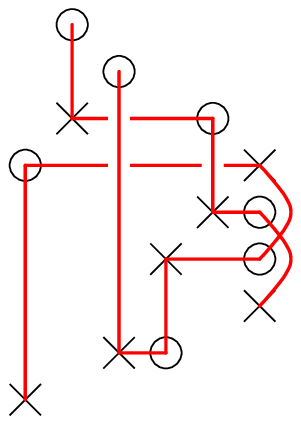}  \ar@{red}@{<=>}[rrdd]|{}="Nya2" && \\
    &&&&&& \\
    \dessin{2cm}{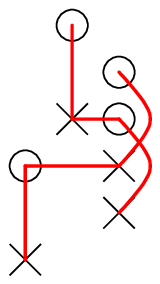}  \ar@{<->}[rrrdd] &&&&&& \dessin{2cm}{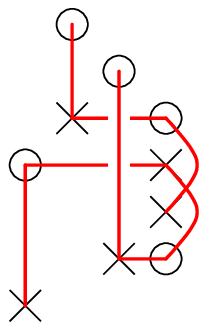}  \ar@{<->}[llldd]|{\ r \ } \\
    &&&&&& \\
    &&&\dessin{2cm}{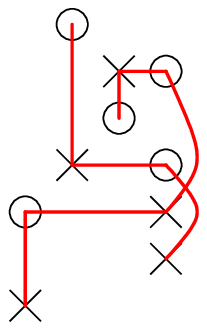} &&&
  \ar@{<.}@/_.9cm/"Nya1"!<.18cm,-.13cm>;"Nya2"!<-.18cm,-.13cm>|{\ \Leftarrow\ }}
  $$
    \label{subfig:D}}
  \caption{Implications between singular grid moves: {\footnotesize Simple arrows stand for a combination of commutations and cyclic permutations of the rows when they are $r$--labeled, of commutations and cyclic permutations of the columns when they are $c$--labeled and of commutations and (de)stabilizations when they are unlabeled. Double arrows stand for singular grid moves.}}
\end{figure}

\addtocounter{figure}{-1}
\addtocounter{subfigure}{4}
\begin{figure}[h!]
  \centering
  \subfigure[]{
  $$
  \xymatrix@!0@C=1.6cm@R=1.2cm{
    && \dessin{2cm}{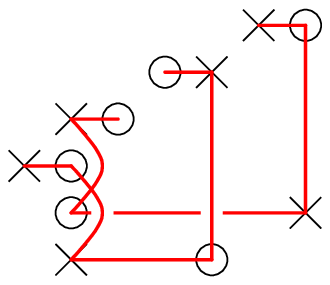} \ar@{<->}[rr]|{\ c \ }  &&  \dessin{2cm}{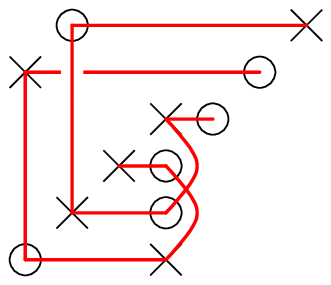} \ar@{<->}[rrd] &&  \\
    \dessin{2cm}{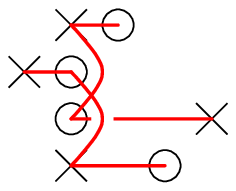} \ar@{<->}[rru] \ar@{<->}[rrdd]|{\ r \ } &&&&&& \dessin{2cm}{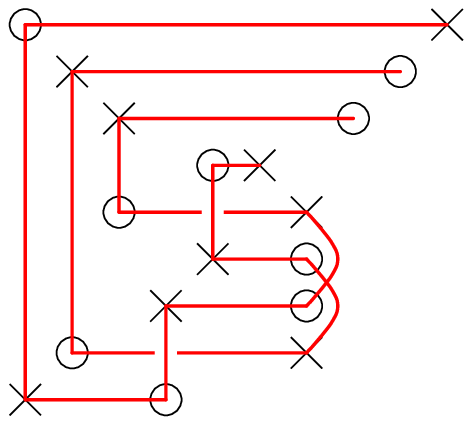} \ar@{<->}[lldd]\\
    &&&&&& \\
    && \dessin{2cm}{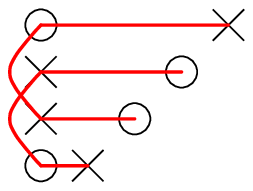}  \ar@{<=>}[lldd]|{}="Nya1" && \dessin{2cm}{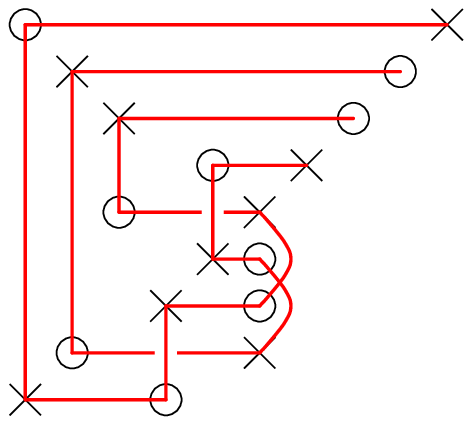}  \ar@{red}@{<=>}[rrdd]|{}="Nya2" && \\
    &&&&&& \\
    \dessin{2cm}{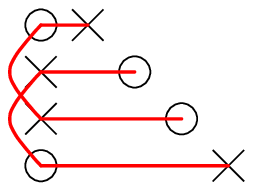}  \ar@{<->}[rrrdd] &&&&&& \dessin{2cm}{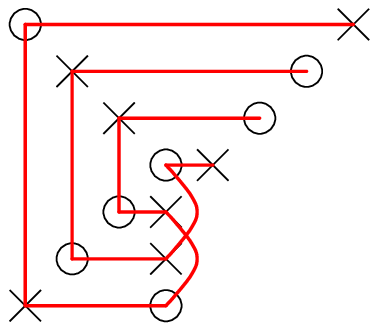}  \ar@{<->}[llldd] \\
    &&&&&& \\
    &&&\dessin{2cm}{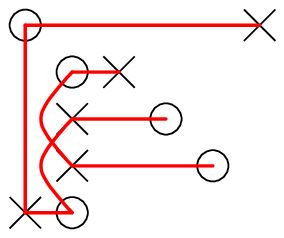} &&&
  \ar@{<.}@/_.9cm/"Nya1"!<.18cm,-.13cm>;"Nya2"!<-.18cm,-.13cm>|{\ \Leftarrow\ }}
  $$
    \label{subfig:E}}
  \caption{Implications between singular grid moves: {\footnotesize Simple arrows stand for a combination of commutations and cyclic permutations of the rows (and maybe some (de)stabilizations) when they are $r$--labeled, of commutations and cyclic permutations of the columns when they are $c$--labeled and of commutations and (de)stabilizations when they are unlabeled. Double arrows stand for singular grid moves.}}
\end{figure}

\subsubsection{Singular elementary moves}
\label{par:SElemMoves}

Now, we can state the singular counterpart of Theorem \ref{def:Moves} (\S\ref{par:ElemMoves}) about elementary moves decomposition.

\begin{theo}\label{SingElemMoves}
  Any two singular grid diagrams which describe the same singular link can be connected by a finite sequence of
  \begin{itemize}
  \item[-] cyclic permutations of rows or columns (possibly singular);
  \item[-] commutations of rows or columns (possibly singular);
  \item[-] stabilizations and destabilizations (under the restrictions given in paragraph \ref{par:SingGridEtElemMoves});
  \item[-] flips (or any equivalent moves).
  \end{itemize}
\end{theo}

\begin{proof}
Let $G_1$ and $G_2$ be two singular grid diagrams which describe the same link.\\

First, we assume that the decorations in any singular columns are in adjacent cases, the two circles being above the two crosses. Moreover, we assume that every decoration which share a row with one of the two middle decorations of a singular column is located on the left of this column. All the doubles points of the associated links are thus in the following position:
$$
\dessin{2cm}{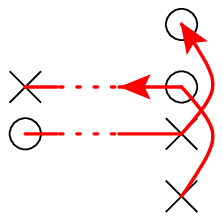}.
$$

We suppose now that the diagrams $D_1$ and $D_2$ associated to the grids $G_1$ and $G_2$, using the above convention for bending the singular vertical strands, are isotopic. As in the regular case, one can require that the isotopy connecting the two diagrams rigidly preserves a neigborhood of the crossings and of the double points except for a finite number of exceptional times when a crossing is turned over or a double point fully rolled up around itself. Compared with the proof of Theorem \ref{def:Moves} (\S \ref{par:ElemMoves}), only the latter need some attention. However, it can also be easily realized as
$$
\dessin{1.8cm}{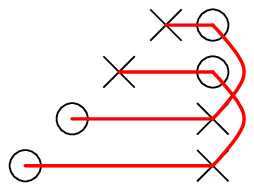}\ \sim \dessin{2.6cm}{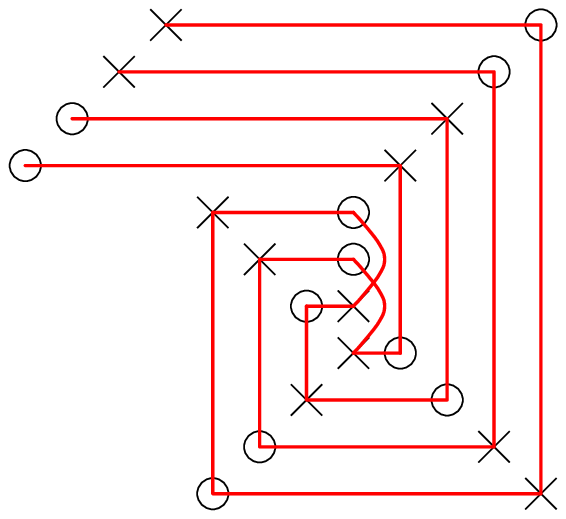}
$$
using stabilizations and commutations. The grids $G_1$ and $G_2$ can thus be connected by a sequence of regular grid moves.\\

If the associated diagrams are not isotopic, then we need to realize the Reidemeister moves. For the regular ones, we refer to the proof of Theorem \ref{def:Moves} (\S \ref{par:ElemMoves}). Because of the required rigidity condition on double points, each of the last two Reidemeister moves splits into four cases. Figures \ref{fig:Reid4} and \ref{fig:Reid5} handle with all of them.\\

\begin{figure}[b!]
  $$
  \begin{array}{lc}
    \fbox{$\dessin{1.3cm}{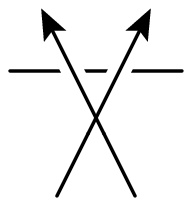}\sim\dessin{1.3cm}{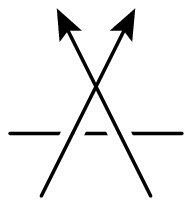}$}\ \ : \hspace{1cm} &
    \xymatrix{\dessin{2.45cm}{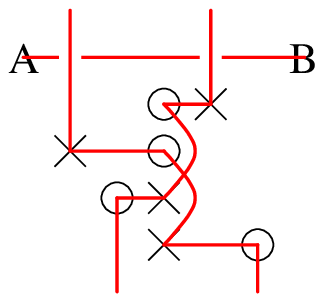} \ar@{<->}[r] & \dessin{2.45cm}{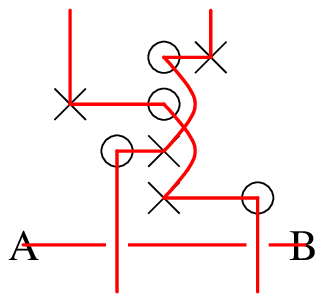}} \\[.4cm]
    \fbox{$\dessin{1.3cm}{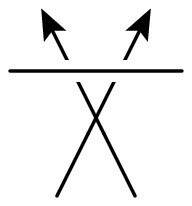}\sim\dessin{1.3cm}{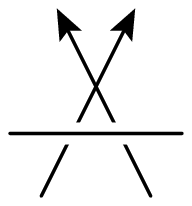}$}\ \ :  \hspace{1cm} &
    \xymatrix{\dessin{2.45cm}{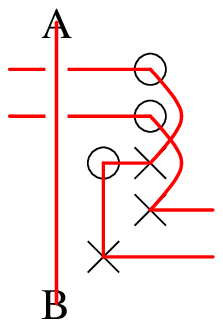} \ar@{<->}[r] & \dessin{2.45cm}{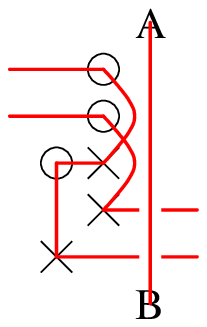}} \\[.4cm]
    \fbox{$\dessin{1.3cm}{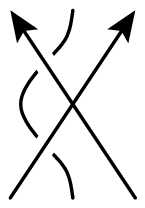}\sim\dessin{1.3cm}{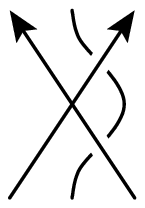}$}\ \ : \hspace{1cm} &
    \xymatrix{\dessin{2.45cm}{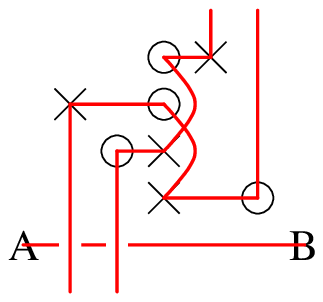} \ar@{<->}[r] & \dessin{2.45cm}{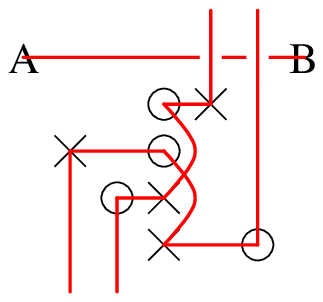}} \\[.4cm]
    \fbox{$\dessin{1.3cm}{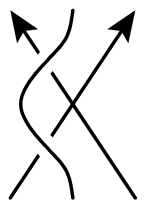}\sim\dessin{1.3cm}{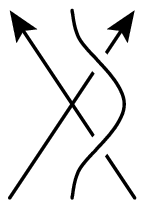}$}\ \ :  \hspace{1cm} &
    \xymatrix{\dessin{2.45cm}{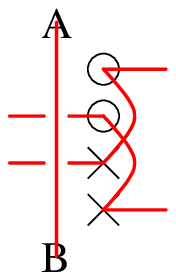} \ar@{<->}[r] & \dessin{2.45cm}{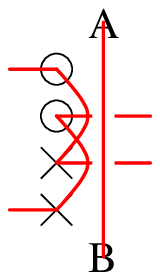}}
  \end{array}
  $$
  \caption{Realization of singular Reidemeister moves V: {\footnotesize Arrows stand for a sequence of  commutations.}}
  \label{fig:Reid5}
\end{figure}

\begin{figure}[p]
  $$
  \begin{array}{lc}
    \fbox{$\dessin{1.3cm}{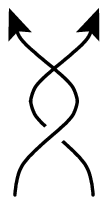}\sim\dessin{1.3cm}{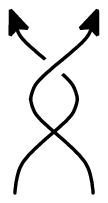}$}\ \ : \hspace{1cm} &
    \xymatrix{\dessin{2.1cm}{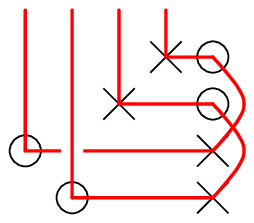} \ar@{<=>}[r] & \dessin{2.1cm}{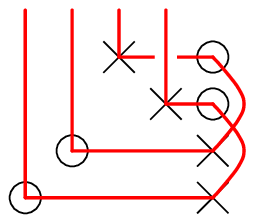}} \\[1.8cm]
    \fbox{$\dessin{1.3cm}{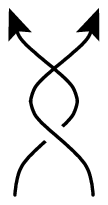}\sim\dessin{1.3cm}{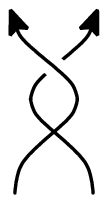}$}\ \ :  \hspace{1cm} &
    \xymatrix{\dessin{2.45cm}{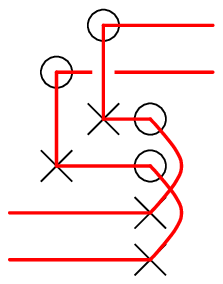} \ar@{<=>}[r] & \dessin{2.45cm}{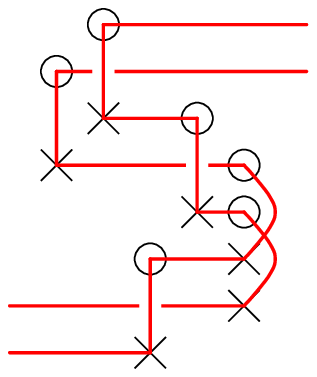} \ar@{<->}[r] & \dessin{2.45cm}{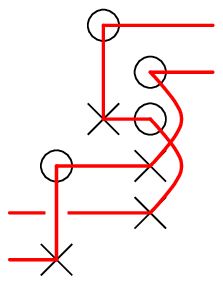}}\\[2cm]
    \fbox{$\dessin{1.3cm}{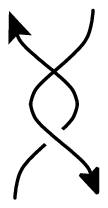}\sim\dessin{1.3cm}{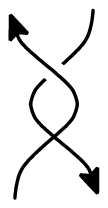}$}\ \ : \hspace{1cm} &
    \vcenter{\hbox{\xymatrix@!0@C=1.35cm@R=1.2cm{\dessin{2.45cm}{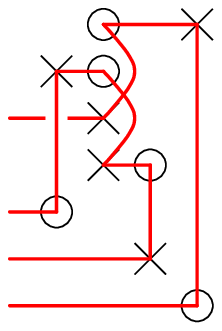} \ar@{<->}[drr] &&&&&& \dessin{2.45cm}{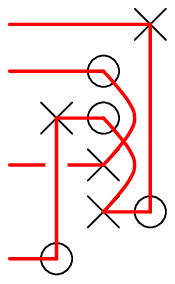}\\
      && \dessin{2.45cm}{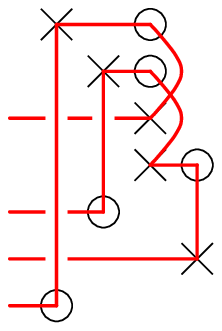} \ar@{<->}[ddll]|{\ r\ } && \dessin{2.45cm}{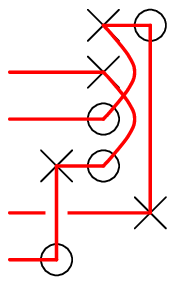} \ar@{<->}[urr]|{\ r\ } &&\\
      &&&&&&\\
      \dessin{2.45cm}{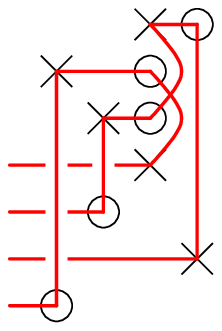} \ar@{<->}[drrr] &&&&&& \dessin{2.45cm}{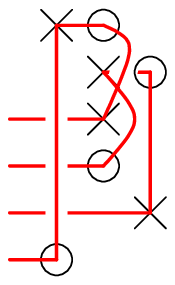} \ar@{<->}[uull]|{\ r\ }\\
      &&& \dessin{2.45cm}{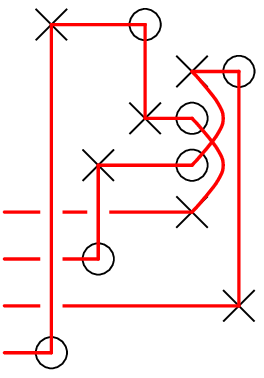} \ar@{<=>}[urrr] &&&}}} \\[4.5cm]
    \fbox{$\dessin{1.3cm}{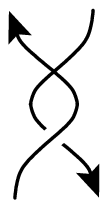}\sim\dessin{1.3cm}{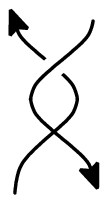}$}\ \ :  \hspace{1cm} &
    \vcenter{\hbox{
        \begin{minipage}{.5\linewidth}
          \begin{center}
            reflect vertically the pictures of the previous case and swap the nature of all the decorations
          \end{center}
        \end{minipage}}}
  \end{array}
  $$
  \caption{Realization of singular Reidemeister moves IV: {\footnotesize Simple arrows stand for a combination of commutations and cyclic permutations of the rows (and maybe some (de)stabilizations) when they are $r$--labeled and of commutations and (de)stabilizations when they are unlabeled. Double arrows stand for singular grid moves.}}
  \label{fig:Reid4}
\end{figure}

Finally, if the two grids are in general position, then we can modify them using regular grid moves in order to reach the conditions required above.
\end{proof}

\subsection{Singular rows}
\label{ssec:SingRows}

Of course, considering singular rows instead of singular columns give a totally similar description of singular links. But one can also mix them.

\subsubsection{Mixed singular grids}
\label{par:SingRC}

In this subsection, we consider \emph{mixed singular grids}\index{diagram!grid!mixed singular} which allow singular columns and singular rows. It means that some columns and some rows may contain exactly two decorations of each kind, in such an arragngement that the two middle decorations are surrounded by decorations of different kinds. The associated singular link is then obtained by connecting the decorations of the singular columns by slightly bended strands, then the decorations of the regular columns by straight lines, then the decorations of the regular rows and finally the decorations of the singular rows. In this process, when a line is crossing another one, it underpasses it.

\begin{figure}[h]
   $$
  \dessin{2cm}{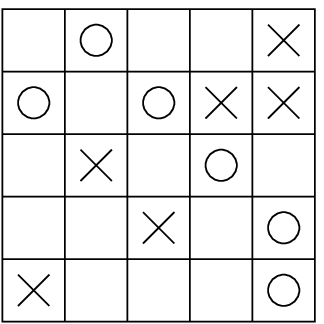} \hspace{.5cm} \leadsto \hspace{.5cm}\dessin{2cm}{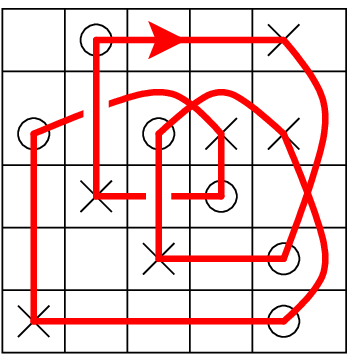}
  $$
  \caption{From mixed singular grid diagrams to singular knots}
  \label{fig:SSG->SSK}
\end{figure}

\subsubsection{Mixed elementary grid moves}
\label{par:MixedElemMoves}

A double point can be represented by mean of a singular column or of a singular row. It is thus necessary to introduce a new elementary move which connect the two possibilities. For this purpose, we define the \emph{rotation moves}\index{moves!grid!rotation} which replaces a specific $(3\times 4)$-subgrid by a $(4\times 3)$ one, up to the adding or the removal of some empty pieces of rows or columns. Pictures of the moves are depicted in Figure \ref{fig:RotationMove}.

\begin{figure}[!h]
$$
\begin{array}{ccc}
  \xymatrix{\dessin{2.5cm}{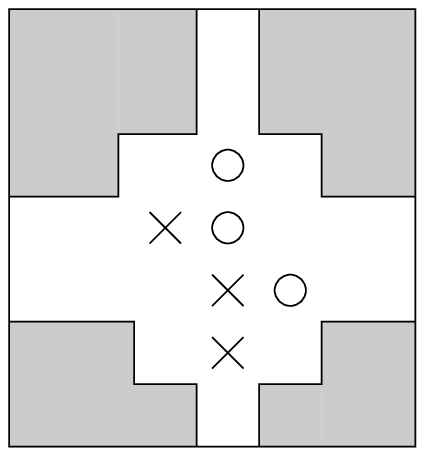} \ar@{<->}[r] & \dessin{2.5cm}{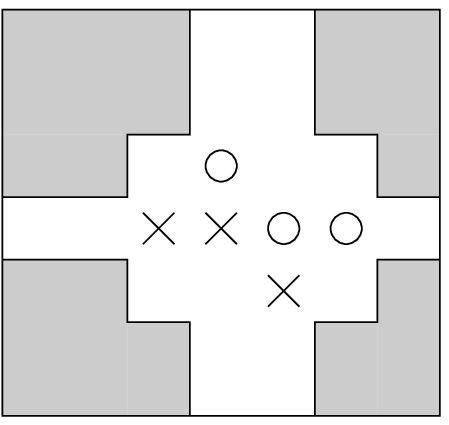}} & \hspace{.5cm} &  \xymatrix{\dessin{2.5cm}{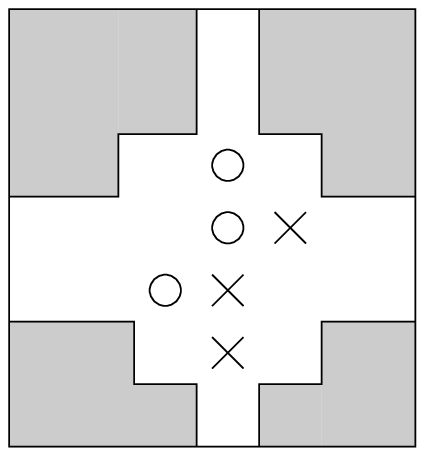} \ar@{<->}[r] & \dessin{2.5cm}{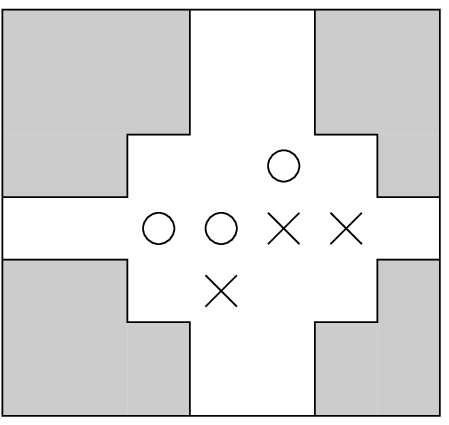}}\\
\textrm{Left rotation} && \textrm{Right rotation}
\end{array}
$$
  \caption{Rotation moves}
  \label{fig:RotationMove}
\end{figure}

Using these moves, any mixed singular grid can be connected to a grid with no more singular rows. Then, the Theorem \ref{SingElemMoves} (\S \ref{par:SElemMoves}) can be applied. However, it can be sharpened. As a matter of fact, Figure \ref{fig:Rot->Flip} and its image under horizontal reflecting illustrate how flip moves can be replaced by rotation moves. Moreover, any singular column involved in a commutation or a cyclic permutation move can be laid down ; the singular commutation is then replaced by regular ones.

\begin{figure}[t!]
$$
\xymatrix@!0@C=1.5cm@R=2cm{
&& \dessin{2cm}{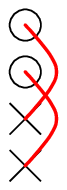} \ar@{<=>}[rrrr]|{}="Ska0" &&&& \dessin{2cm}{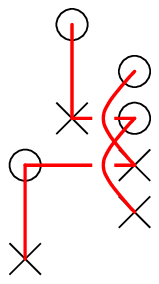} \ar@{<->}[drr] && \\
\dessin{2cm}{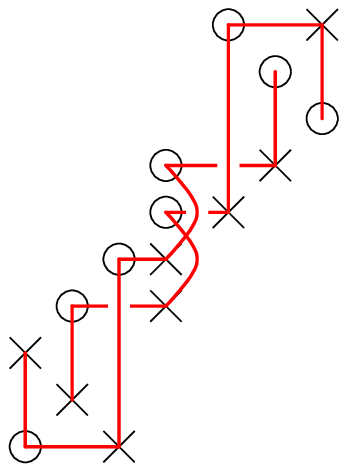} \ar[urr] &&&&&&&& \dessin{2cm}{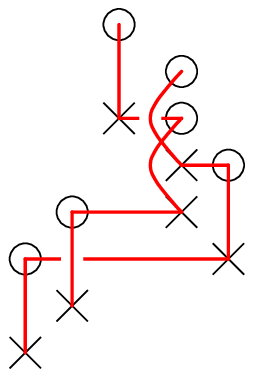} \ar@{<=>}[dlll]|{}="Ska2" \\
&&& \dessin{2cm}{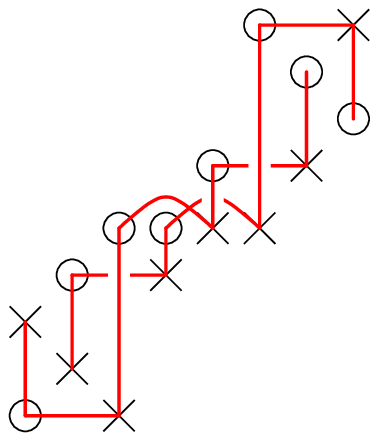} \ar@{<=>}[ulll]|{}="Ska1" && \dessin{2cm}{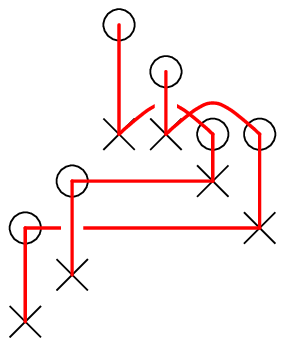} \ar@{<->}[drr] \\
& \dessin{2cm}{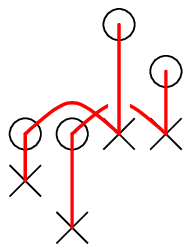} \ar@{<->}[urr] &&&&&& \dessin{2cm}{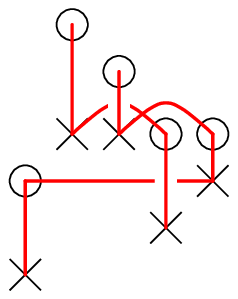} \ar@{<->}[dlll]|c \\
&&&& \dessin{2cm}{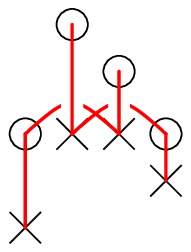} \ar@{<->}[ulll]|c
\ar@{<.>}@/_.4cm/"Ska1"!<.2cm,.2cm>;"Ska0"!<-.1cm,-.2cm>|(.7){}="SKA3" \ar@{<.>}@/^.4cm/"Ska2"!<-.2cm,.2cm>;"Ska0"!<.1cm,-.2cm>|(.7){}="SKA4"
\ar@{.}@/_.4cm/"SKA3";"SKA4"|\Uparrow
}
$$
  \caption{From rotations to flips: {\footnotesize Simple arrows stand for a combination of commutations and cyclic permutations of the columns when they are $c$--labeled and of commutations and (de)stabilizations when they are unlabeled. Double arrows stand for singular grid moves.}}
  \label{fig:Rot->Flip}
\end{figure}

These remarks prove the following theorem:
\begin{theo}
  Any two mixed singular grid diagrams which describe the same singular link can be connected by a finite sequence of
  \begin{itemize}
  \item[-] cyclic permutations of regular rows or columns;
  \item[-] commutations of regular rows or regular columns;
  \item[-] stabilizations and destabilizations (under the restrictions given in paragraph \ref{par:SingGridEtElemMoves});
  \item[-] rotations.
  \end{itemize}
\end{theo}

The main part of this thesis will not use this refined mixed description. However, Appendix \ref{appendix:MixedSimplification} shows how the construction can be extended to mixed grids and how it simplifies the proofs of invariance.

%%% Local Variables: 
%%% mode: latex
%%% TeX-master: "These"
%%% End: 

% Morphisme de croisement
 \section{Switch morphism}
\label{sec:Wall}

Let $G^+$ and $G^-$ be two regular grid diagrams of size $n$ which differ from a commutation of two adjacent columns which is \textbf{not} a commutation move as defined in paragraph \ref{par:ElemMoves}. If the two uppermost decorations of the two commuting columns are of the same kind, then, among the two grids, $G^+$ is the diagram where the line passing through these two decorations has a negative slope. Otherwise, $G^+$ is the diagram where this line has a positive slope.

\subsection{Grid and link configurations}
\label{ssec:PosNegResol}

\subsubsection{Standard configuration}
\label{par:PosNegCrossing}

\begin{figure}[h]
$$
\begin{array}{ccc}
  \dessin{2.5cm}{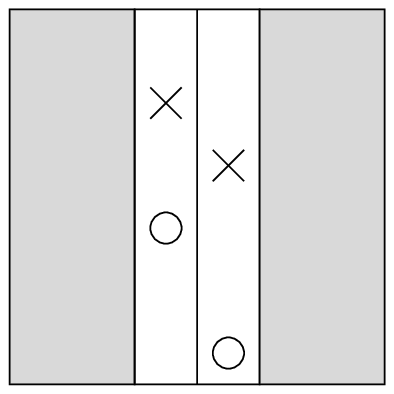} & \hspace{1cm} & \dessin{2.5cm}{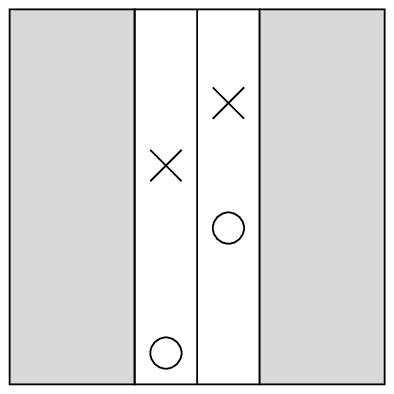}\\[1cm]
G^+ && G^-
\end{array}
$$  
  \caption{The standard configuration for $G^+$ and $G^-$}
  \label{fig:StandardConfiguration}
\end{figure}

\begin{prop}\label{Standard}
 The possible arrangements of the decorations are the \emph{standard configuration}\index{standard configuration} shown in Figure \ref{fig:StandardConfiguration} and its images by cyclic permutations of the rows. In particular, when the two commuting columns are merged and considered cyclically by identifying their top and bottom boundary, decorations of the same kind are vertically side by side.
\end{prop}
\begin{proof}
  By hypothesis, the decorations appear alternatively on each column. Every decoration is hence vertically surrounded by the decorations of the other column and consequently decorations of different kinds. One of them is thus of the same kind than the surrounded one. This proves the second part of the statement.\\
Now, everytime a cyclic permutation of the rows affects the uppermost decorations, it changes simultaneously their possible equalness of nature and the slope sign of the line passing through them. This concludes the proof.
\end{proof}

\subsubsection{Associated links}
\label{par:AssLinks}

\begin{prop}\label{PosNeg}
  The two links $L^+$ and $L^-$ respectively associated to $G^+$ and $G^-$ differ only from a crossing which is positive in $L^+$ and negative in $L^-$.
\end{prop}
\begin{proof}
  Since $L^+$ (resp. $L^-$) is invariant under cyclic permutations of the rows of $G^+$ (resp. $G^-$) and according to Proposition \ref{Standard} (\S\ref{par:PosNegCrossing}), it is sufficient to check it for the standard configuration.\\
  Moreover, even if it means to perform a few cylic permutations of the columns first, we can assume that the two decorations belonging to the same rows than the middle decorations of the commuting columns are located on the left side of these two columns.\\
  Now, the statement is easy to check:
$$
\xymatrix@!0@R=1.5cm@C=2.6cm{&\dessin{1.7cm}{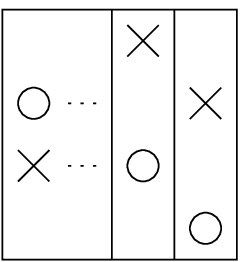}\ar@{^<-_>}[r]&\dessin{1.7cm}{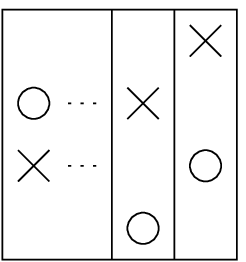}&\\
\dessin{1.7cm}{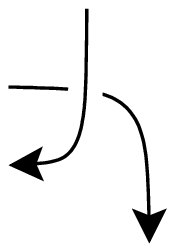} \ar@{<-}[ur]&&&\dessin{1.7cm}{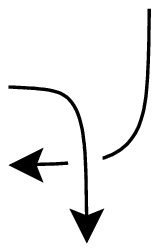} \ar@{<-}[ul]}
$$
\end{proof}

\subsection{More polygons on torus}
\label{ssec:Polygons}

\subsubsection{Combined grid}
\label{par:CombGrid}

The aim of the next sections is to define a filtrated chain map $\func{f}{C^-(G^+)}{C^-(G^-)}$.\\

For this purpose, it will be convenient to draw at the same time $G^+$ and $G^-$ on a combined grid $G_{Comb}$.
\begin{eqnarray}\label{fig:CombGrid}
  \hspace{.5cm} \dessin{3.8cm}{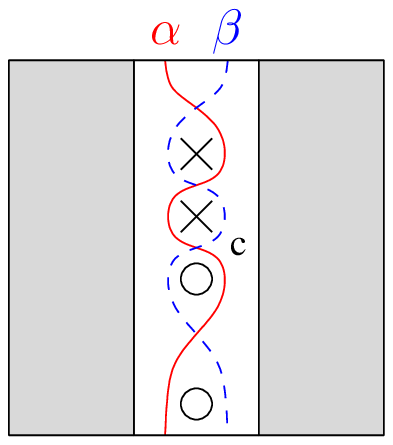} \hspace{.5cm}
\begin{array}{cl}
  \xymatrix{\ar@*{[red]}@{-}[r]&} & G^+\\[3mm]
  \xymatrix{\ar@*{[blue]}@{--}[r]&} & G^-
\end{array}
\end{eqnarray}

\begin{remarque}\label{Dependance}
  We fix that $\alpha$ and $\beta$ never intersect each other on a horizontal grid line. There are several ways to do it and, even if we omit to denote it in the notation, all the following constructions will depend on this choice.
\end{remarque}

In this context, a generator of $C^-(G^+)$ (resp. $C^-(G^-)$) is represented by $n$ dots, one of which is located on $\alpha$ (resp. $\beta$ ).

\subsubsection{Pentagons}
\label{par:Pentagons}

As for the definition of the differential, we consider the torus $\T_{Comb}:=\T_{G_{Comb}}$. We denote by $c\in\alpha\cap\beta$ the intersection point located just below an $X$ and just above a $O$ (see (\ref{fig:CombGrid})).\\

  Let $x$ be a generator of $C^-(G^+)$ and $y$ a generator of  $C^-(G^-)$. A \emph{pentagon connecting $x$ to $y$}\index{polygon!pentagon}\index{pentagon|see{polygon}} is an embedded pentagon $\pi$ in $\T_{Comb}$ which satisfies:
  \begin{itemize}
  \item[-] edges of $\pi$ are embedded in the grid lines (including $\alpha$ and $\beta$);
  \item[-] the point $c$ is a corner of $\pi$;
  \item[-] starting at $c$ and running positively along the boundary of $\pi$, according to the orientation of $\pi$ inherited from the one of $\T_{Comb}$, the corners of $\pi$ are successively and alternatively a point of $x$ and a point of $y$;
  \item[-] except on $\p \pi$, the sets $x$ and $y$ coincide;
  \item[-] the interior of $\pi$ does not intersect $\alpha\cup\beta$ in a neighborhood of $c$.
  \end{itemize}
The corner $c$ is called the \emph{peak of $\pi$}\index{polygon!peak}\index{peak|see{polygon}}.\\

Note that a pentagon $\pi$ connecting $x$ to $y$ has a unique oriented vertical edge which joins a point of $x$ to a point of $y$. We say that $\pi$ \emph{points toward the right}\index{polygon!to point toward}\index{point toward (to)|see{polygon}} if this edge is oriented toward the bottom of the grid. Otherwise, we say that it \emph{points toward the left}. For a generic choice of arcs $\alpha$ and $\beta$, it corresponds to the direction toward which the peak is pointing.\\

A pentagon $\pi$ is \emph{empty} if $\Int(\pi)\cap x=\emptyset$.\\
We denote by $\Pent(G_{Comb})$ the set of all empty pentagons on $G_{Comb}$ and by $\Pent(x,y)$ the set of those which connect $x$ to $y$.

\begin{figure}[h]
 $$
  \begin{array}{ccc}
 \dessin{3.05cm}{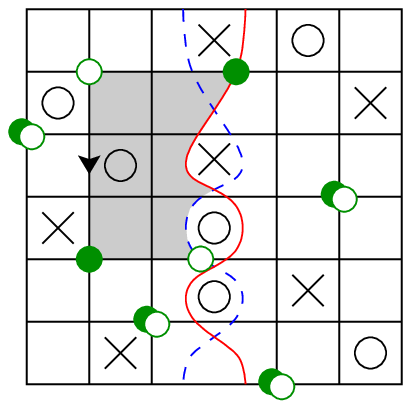} & \hspace{1cm} & \dessin{3.05cm}{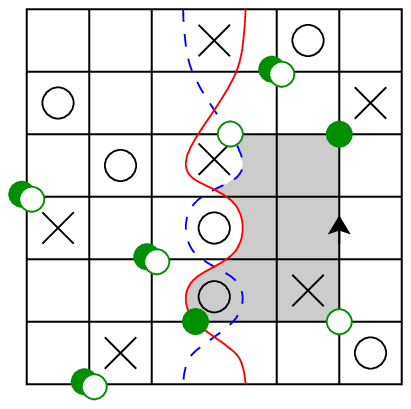}\\[1cm]
\textrm{an empty pentagon} && \textrm{an empty pentagon}\\
\textrm{pointing to the right} && \textrm{pointing to the left}
  \end{array}
$$ 
  \caption{Examples of pentagons: {\footnotesize dark dots describe the generator $x$ while hollow ones describe $y$. Pentagons are depicted by shading.}}
    \label{fig:Pentagons}
\end{figure}

\subsubsection{Spikes}
\label{par:Spikes}

    Let $x$ be a generator of $C^-(G^+)$ and $y$ a generator of  $C^-(G^-)$. A \emph{spike}\index{polygon!spike}\index{spike|see{polygon}} connecting $x$ to $y$ is a triangle $\tau$ in $\T_{Comb}$, possibly crossed, which satisfies:
  \begin{itemize}
  \item[-] edges of $\tau$ are embedded in the grid lines (including $\alpha$ and $\beta$);
  \item[-] the point $c$ is a corner of $\tau$;
  \item[-] starting at $c$ and running positively along the boundary of $\tau$, according to the orientation of $\tau$ inherited near $c$ from the one of $\T_{Comb}$, the corners of $\tau$ are a point of $y$ and then of $x$;
  \item[-] except on $\p \tau$, the sets $x$ and $y$ coincide.
  \end{itemize}
The grid line which contains the horizontal edge of $\tau$ is called the \emph{support of $\tau$}\index{polygon!spike!support}\index{support|see{polygon}}.

\begin{figure}[h]
 $$
  \begin{array}{ccc}
 \dessin{3.05cm}{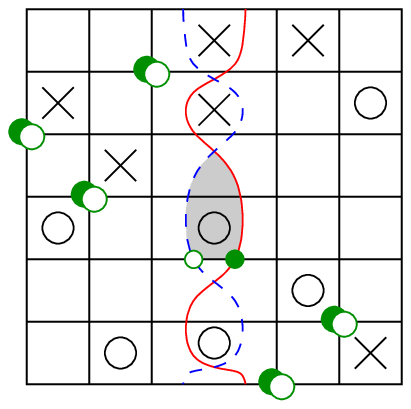} & \hspace{1cm} &\dessin{3.05cm}{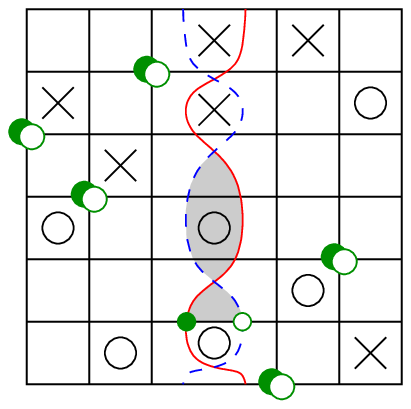}\\[1cm]
\textrm{an embedded spike} && \textrm{a crossed spike}
  \end{array}
$$ 
  \caption{Examples of spikes: {\footnotesize dark dots describe the generator $x$ while hollow ones describe $y$. Spikes are depicted by shading.}}
    \label{fig:Spikes}
\end{figure}

\begin{remarque}
  Essentially, a spike moves a dot from $\alpha$ to $\beta$ along its support.
\end{remarque}

\subsubsection{Spikes and pentagons}
\label{par:SpiPent}

\begin{prop}\label{Pent->Rect}
  For a given pentagon $\pi\in\Pent(G_{Comb})$, there is a unique spike $\tau$ such that $\pi\cup \tau$ is a rectangle $\rho\in \Rect(G^+)$.
\end{prop}

This defines a map $\func{\phi}{\Pent(G_{Comb})}{\Rect(G^+)}$.\\

\vspace{-.45cm}
{\parpic[r]{$\dessin{3cm}{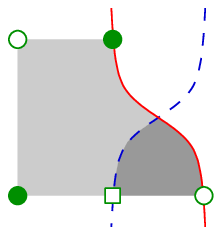}$\hspace{.5cm}}
\vspace{.45cm}{\noindent \it Proof.} Let consider a pentagon $\pi$ which connects $x\in C^-(G^+)$ to $y\in C^-(G^-)$. There is a unique spike $\tau$ which has two common vertices with $\pi$ such that one of them belongs to $y$. By glueing it to $\pi$ along their common edge, we embank the peak of $\pi$ and move a dot of $y$ from $\beta$ to $\alpha$. Hence, we get a rectangle in $\Rect(G^+)$ which connects $x$ to a new generator $z$ of $C^-(G^+)$.

}\hfill{$\square$}

\begin{cor}\label{PentDegree}
   Let $x$ be a generator of $C^-(G^+)$, $y$ a generator of  $C^-(G^-)$ and $\pi$ a pentagon which connects $x$ to $y$. Then
$$
M(x)-M(y)=-2\#(\pi\cap\O),
$$
$$
A(x)-A(y)=\#(\pi\cap\X)-\#(\pi\cap\O).
$$
\end{cor}
\begin{proof}
 We keep the notation of the precedent proof.\\
We will begin by three remarks.
\begin{itemize}
\item[-] Since propositions \ref{PermPreserves} (\S\ref{par:MaslAlex}) and \ref{Standard} (\S\ref{par:PosNegCrossing}) do hold, it is sufficient to deal with the standard configuration.
\item[-] As sets of dots on the regular grids $G^+$ and $G^-$, the generators $z$ and $y$ are identical; only the decorations are changing.
\item[-] The decorations contained in $\tau$ determine the vertical position of its support compared with the four commuting decorations.
\end{itemize}
Now, for each position of the support of $\tau$ compared with the decorations, it is easy to compute $M(y)$ (resp. $A(y)$) in fonction of $M(z)$ (resp. $A(z)$). For instance, if the support is between the two $O$'s, then
$$
M(y)=M(z)-1 \hspace{1cm} \textrm{and} \hspace{1cm} A(y)=A(z)-1.
$$

Finally, we conclude by using the Proposition \ref{RectGradings} (\S\ref{par:Rect}) and the decomposition of $\rho=\pi\cup \tau$.
\end{proof}

\subsection{First step toward singular link Floer homology}
\label{ssec:FirstStep}

\subsubsection{Switch map}
\label{par:WallMap}

Now, we consider the map $\func{f}{C^-(G^+)}{C^-(G^-)}$ which is the morphism of $\Z[U_{O_1},\cdots,U_{O_n}]$--modules defined on the generators by
$$
f(x)=\sum_{\substack{y \textrm{ generator}\textcolor{white}{R}\\\textrm{of }C^-(G^-)\textcolor{white}{I}}} \sum_{\pi\in \Pent(x,y)} \e(\pi)U_{O_1}^{O_1(\pi)}\cdots U_{O_n}^{O_n(\pi)}\cdot y,
$$
where $\func{\e}{\Pent(G_{Comb})}{\{\pm 1\}}$ is defined by
$$
\e(\pi)=\left\{\begin{array}{rl}\e(\phi(\pi))&\textrm{if }\pi\textrm{ is a pentagon pointing to the right}\\-\e(\phi(\pi))&\textrm{if }\pi\textrm{ is a pentagon pointing to the left.}\end{array}\right.
$$
\index{switch!morphism@morphism $f$}

\begin{figure}[b]
  $$
  \dessin{3cm}{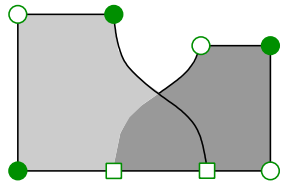}
  $$
  \caption{Alternative decompositions for pentagons sharing their peak with a rectangle:  {\footnotesize dark dots describe the initial generator while hollow ones describe the final one. The squares describe intermediate states. One decomposition is given by its border and the second by shading of different intensity.}}
  \label{fig:AlternativePent}
\end{figure}

\begin{prop}\label{D1Preserves}
  The map $f$ preserves the Maslov grading as well as the Alexander filtration. Moreover, it anti-commutes with the differentials $\p_{G^\pm}^-$.
\end{prop}

\begin{proof}
  The fact that $f$ preserves the Maslov graduation and Alexander filtration is immediate after Corollary \ref{PentDegree} (\S\ref{par:SpiPent}).\\
The proof that it anti-commutes with $\p_{G^\pm}^-$ is very similar to the proof of lemma $3.1$ in \cite{MOST}. As in the proof of Theorem \ref{D0Preserves} (\S\ref{par:Diff}), it consists mainly in showing that every juxtaposition of a pentagon and a rectangle arising in $f\circ\p^-_{G^+}$ has an alternative decomposition arising in $\p^-_{G^-}\circ f$. Most cases are similar to those occuring in the proof of theroem \ref{D0Preserves}. However, when the peak of the pentagon is involved in a common edge with the rectangle, the alternative decomposition is given in Figure \ref{fig:AlternativePent}. The minus sign comes from the fact that pentagons are pointing in opposite directions.\\

There are also special cases which have only one decomposition: one of the two lateral thin annuli surrounding the arcs $\alpha$ and $\beta$ may be filled by a pentagon and a rectangle. But then, for a given initial generator, one can check that the right and left such domains do contain the same decorations. In particular, they contain simultaneously the possible $O$ decoration which may produce a $U_O$ factor. Moreover the sign is positive if the thin annulus we are dealing with is the left one and negative otherwise.\\
Finally, all the special terms cancel by pairs (see Figure \ref{fig:Cancel} for an example).
\end{proof}

\begin{figure}
$$
\xymatrix@!0@R=1.5cm@C=3cm{& \dessin{3.1cm}{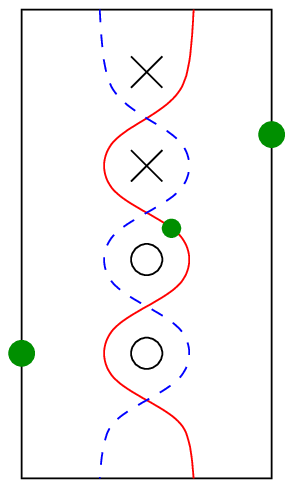}\ar[dl] \ar[dr] &\\
\dessin{3.2cm}{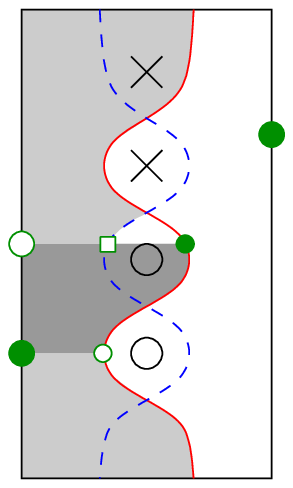} && \dessin{3.2cm}{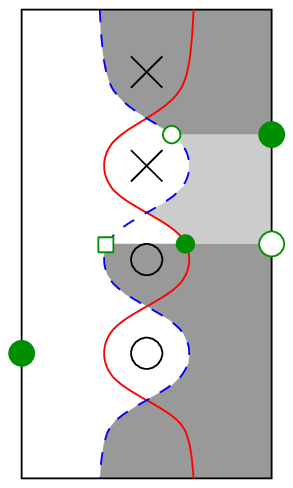}}
$$  
  \caption{Special cases: {\footnotesize dark dots describe the initial generator and hollow ones the intermediate state. The square is the special dot which distinguishes the final generator from the initial one. Pentagons are depicted by shading and rectangles by darker shading. Note that the involved pentagons are pointing in opposite directions, the associate two terms differ hence from a minus sign.}}
  \label{fig:Cancel}
\end{figure}

\subsubsection{Ambiguity}
\label{par:UpToHom}

As pointed out in the Remark \ref{Dependance} (\S\ref{par:CombGrid}), the map $f$ depends on the choice of arcs $\alpha$ and $\beta$. More precisely, it depends on the relative position of the intersection $c\in\alpha \cap \beta$ compared with the horizontal grid lines. Now, we deal with this dependancy.\\

Let  $(\alpha,\beta)$ and  $(\alpha',\beta')$ be two sets of arcs which are identical except around an horizontal grid line denoted by $l$. We assume that the arcs $\alpha$ and $\beta$ intersect in $c$ just above this horizontal grid line, whereas $\alpha'$ and $\beta'$ intersect just below. 

\begin{eqnarray}\label{AboveBelow}
\dessin{2.5cm}{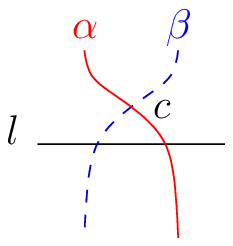}  \hspace{2cm}  \dessin{2.5cm}{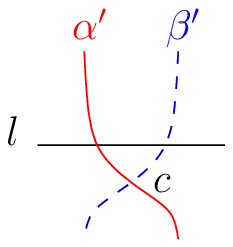}
\end{eqnarray}

Actually, the maps $f$ and $f'$ associated to each set of arcs are homotopic.\\

To define the homotopy map, first note that, for every generator $x\in C^-(G^+)$, there exists at most one spike of which $c$ is a corner and $l$ the support and which connects $x$ to a generator of $C^-(G^-)$. Such a spike is called a \emph{small spike}\index{polygon!spike!small}.\\

Now, we set
$$
h(x)=\left\{\begin{array}{cl}
y & \textrm{if }x\textrm{ and }y\textrm{ are connected by a small spike}\\
0 & \textrm{otherwise}
\end{array}\right.
$$
and we extend $h$ to $C^-(G^+)$ by $\Z[U_{O_1},\cdots,U_{O_n}]$--linearity.\\

\begin{prop}\label{Homotopie}
  The map $h$ preserves the Alexander filtration and increases the Maslov grading by one. Moreover, it satisfies
$$
f-f'=h\circ \p_{G^+}^-\p_{G^-}^-\circ h.
$$
\end{prop}
\begin{proof}
The first part of the statement follows from the easiest computation made in the proof of Corollary \ref{PentDegree} (\S\ref{par:SpiPent}).\\
For the second part, note that the only surviving term in $f-f'$ are pentagons which are rectangles minus a small spike (compare Proposition \ref{Pent->Rect}, \S\ref{par:SpiPent}). Those are terms of $h\circ \p_{G^+}^-$ if the pentagon is above $l$ and pointing to the right or below $l$ and pointing to the left. In the two other cases, they are terms of $\p_{G^-}^-\circ h$. Because of the definition of the orientation map $\e$ on pentagons, the signs coincide.\\
The remaining terms in $h\circ \p_{G^+}^-$ cancel with terms in  $\p_{G^-}^-\circ h$.
\end{proof}

\begin{cor}
  The map $\func{f}{C^-(G^+)}{C^-(G^-)}$ is well defined up to homotopy.
\end{cor}

\subsubsection{Link Floer homology for a single double point}
\label{par:1SingFlHom}

According to Proposition \ref{PosNeg} (\S\ref{par:AssLinks}), the links $L^+$ and $L^-$ differ only from a crossing. So, they can be read as the two desingularizations of a link $L^0$ with a single double point. Then, the homology of the mapping cone of $f$ is already  an invariant of $L^0$. This is proved is paragraphs \ref{par:ArcsIsotopies}---\ref{par:SingleComm}.\\

For singular links with more double points, things are slightly more complicated.

\subsection{Variants}
\label{ssec:Variants}

\subsubsection{Different peaks}
\label{par:DifferentPeaks}

We can play the same game after changing the point $c$ for its cousin $c'$, the arcs intersection located just above the two commuting $\X$--decorations and below the two $\O$ ones (up to cyclicity due to the torus point of view).

$$
\dessin{3.5cm}{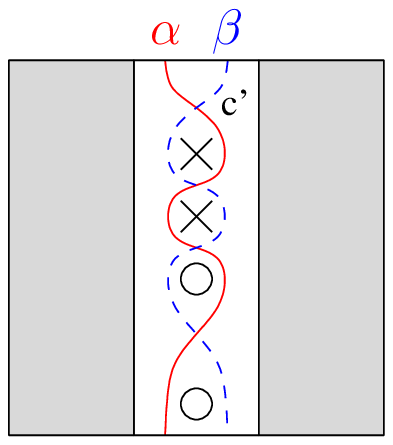}
$$

All the proofs of this section can be adapted straighforwardly to this case. Actually, the only point which depends on the nature of the decorations is the Corollary \ref{PentDegree} (\S\ref{par:SpiPent}). But the same reasoning leads to the same statement.\\
This gives rise to a second map $f'$.\\

The two constructions only differ from the location of peaks. In one case, they are all located on $c$--like arcs intersections, below an $\X$--decoration and above a $\O$ one. Then we say that the peaks are \emph{of type $\dessin{.4cm}{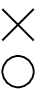}$}\index{zzzaaa@$\dessin{.4cm}{Cstyle}$}. In the other case, the peaks are on $c'$--like arcs intersections, above an $\X$--decoration and below a $\O$ one, we say they are \emph{of type $\dessin{.4cm}{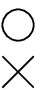}$}\index{zzzaab@$\dessin{.4cm}{Cpstyle}$}. This terminology may also be applied to the associated constructions.\\

Naturally, one can also consider the sum $f_{tot}=f+f'$ but we will see in section \ref{sec:Sum} that this is of less interest.

%%% Local Variables: 
%%% mode: latex
%%% TeX-master: "These"
%%% End: 

% Cube of resolution
 \section{Singular link Floer homology}
\label{sec:SingFloerHom}

The map $f$ defined in the previous section is the key ingredient for generalizing Heegaard-Floer homology to singular links.\\

Let $G$ be a singular grid diagram of size $(n,k)$. We label the $k$ singular columns by integers from $1$ to $k$, but this numbering will be totally transparent in the following construction.

\subsection{Cube of resolution}
\label{ssec:Cube}

\subsubsection{Grid resolutions}
\label{sec:GridResol}

There are essentially two ways to desingularize a singular column with respect to the connections between decorations illustrated in Figure \ref{fig:Desing}, which are imposed by the associated singular link $L$.

\begin{figure}[h]
$$
\xymatrix@!0@R=1.2cm@C=3.1cm{\ \textcolor{green}{\huge \checked}\ \dessin{2cm}{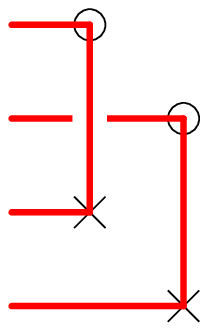}&& \dessin{2cm}{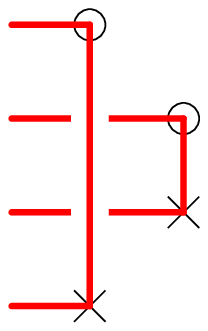}\\
&\dessin{2cm}{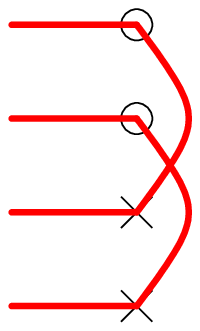} \ar[ul] \ar[ur]|{\dessin{.5cm}{wrong}} \ar[dl] \ar[dr]|{\dessin{.5cm}{wrong}} &\\
\ \textcolor{green}{\huge \checked}\ \dessin{2cm}{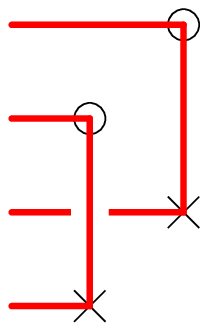}&& \dessin{2cm}{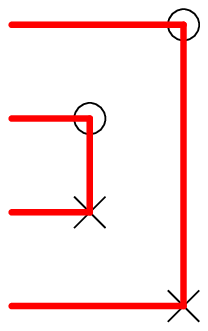} }
$$
 \caption{Desingularization of a singular column}
  \label{fig:Desing}
\end{figure}

According to Proposition \ref{PosNeg} (\S\ref{par:AssLinks}), one of them corresponds to the $0$--resolution of the corresponding double point of $L$; thus we call it the \emph{$0$--resolution}\index{resolution!Aresolution@$0$--resolution} of the column. Accordingly, the other one is called \emph{$1$--resolution}\index{resolution!Bresolution@$1$--resolution}. We also say that the singular column has been \emph{positively} or \emph{negatively resolved}.\\

\subsubsection{Cube of resolution}
\label{par:ResolCube}

\begin{figure}[p]
$$
\xymatrix@C=3.4cm@R=3.4cm@!0{
& \dessin{2.2cm}{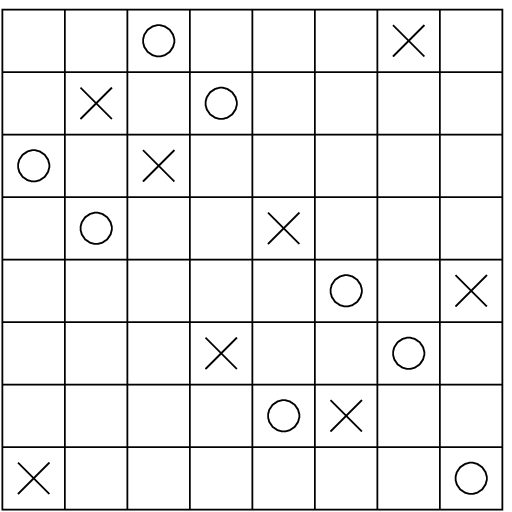} \ar[r] \ar[dr]|!{[d];[r]}\hole&\dessin{2.2cm}{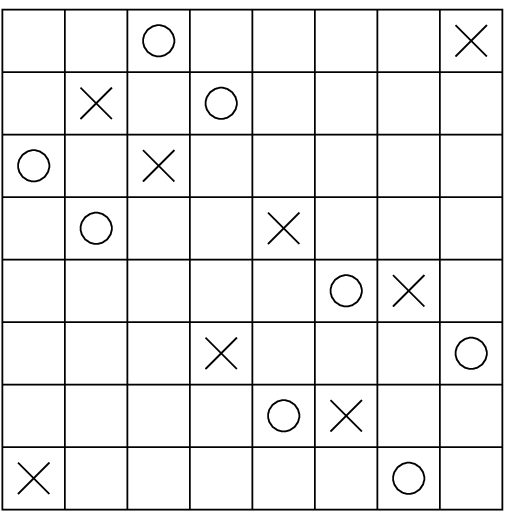} \ar[dr]&\\
\dessin{2.2cm}{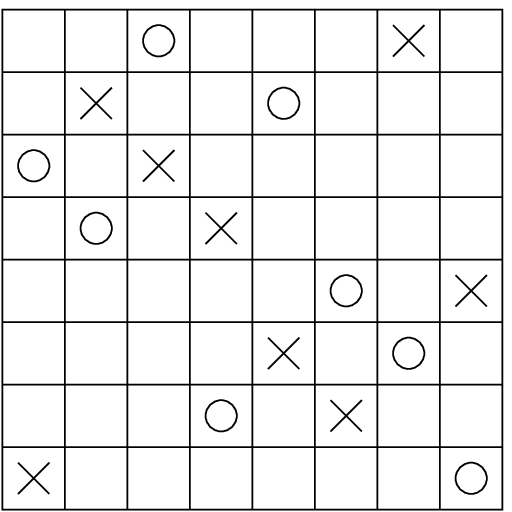} \ar[ur] \ar[r] \ar[dr] & \dessin{2.2cm}{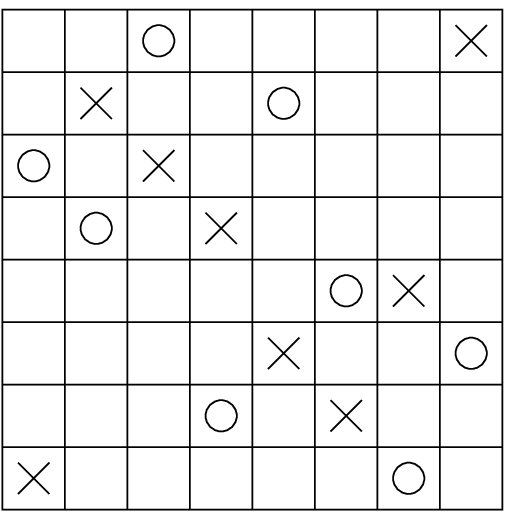} \ar[ur] \ar[dr] & \dessin{2.2cm}{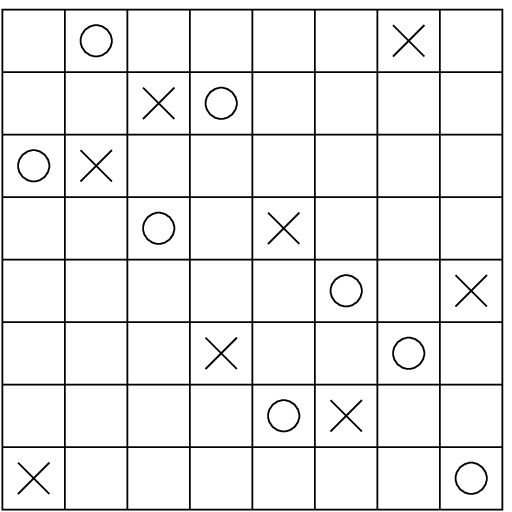} \ar[r] & \dessin{2.2cm}{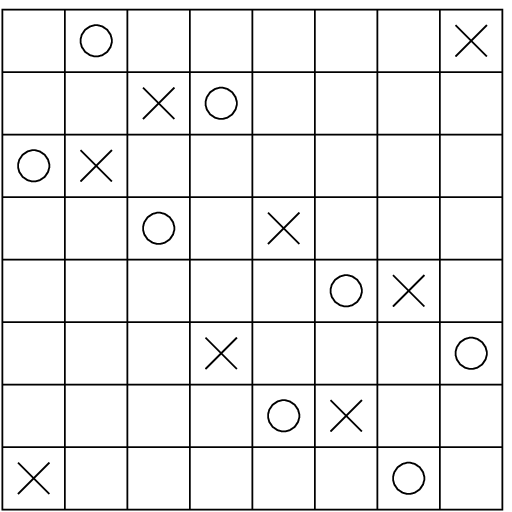}\\
&\dessin{2.2cm}{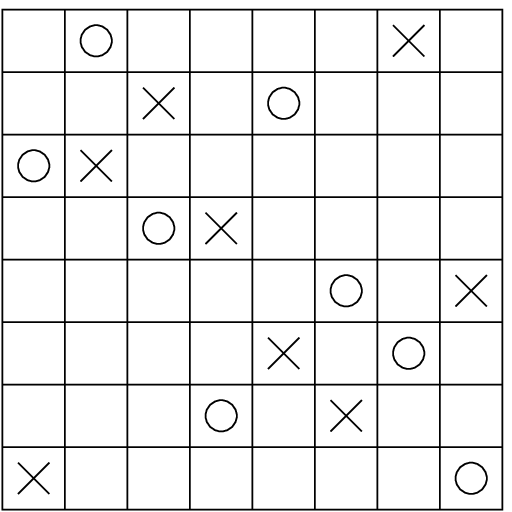} \ar[r] \ar[ur]|!{[u];[r]}\hole & \dessin{2.2cm}{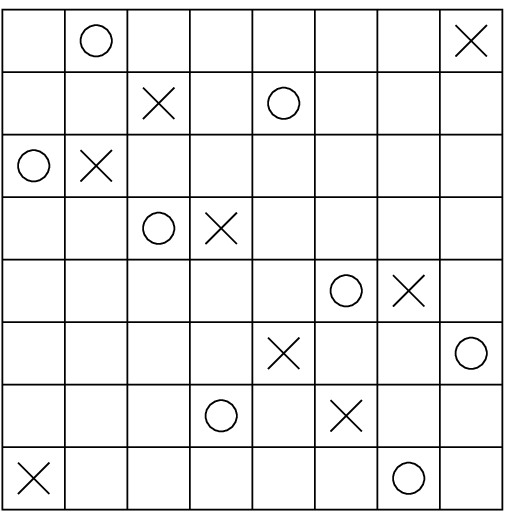} \ar[ur] & }
$$
\vspace{.1cm}
$$
\rotatebox{270}{$\leadsto$}
$$
\vspace{-2.6cm}
$$
\hspace{.5cm}\dessin{3cm}{Sing8} \hspace{9cm} \dessin{3.5cm}{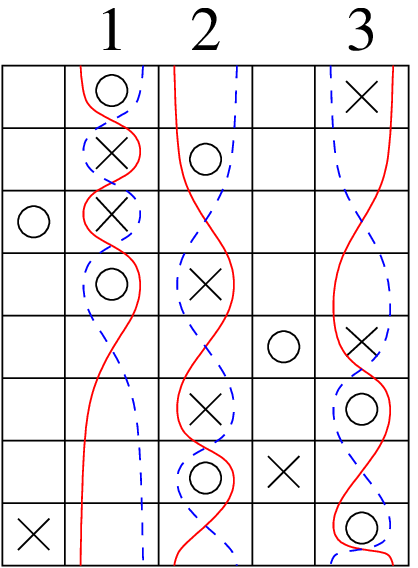}
$$
\vspace{-1cm}
$$
\xymatrix@C=3.4cm@R=3.4cm@!0{
  & C^-\left(\ \dessin{1.1cm}{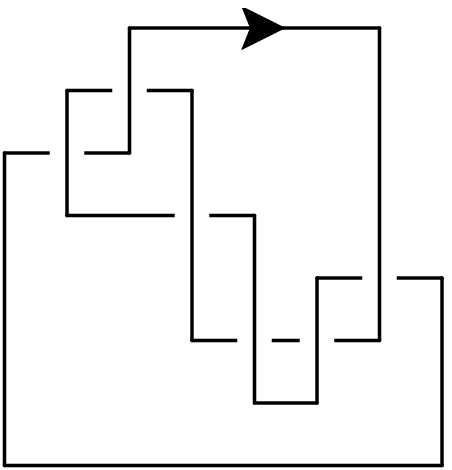}\ \right) \ar[r]^{f_{01\star}} \ar[dr]^(.3){f_{\star 10}}|!{[d];[r]}\hole &  C^-\left(\ \dessin{1.1cm}{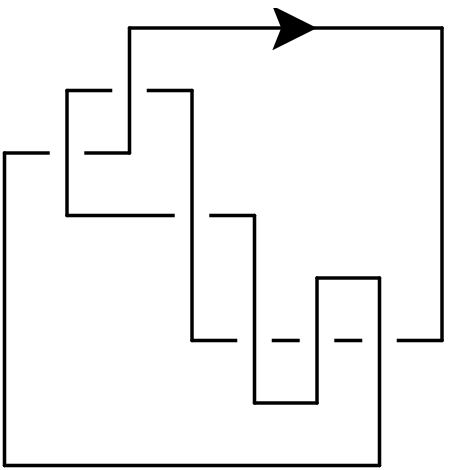}\ \right) \ar[dr]^{f_{\star 11}}\\
 C^-\left(\ \dessin{1.1cm}{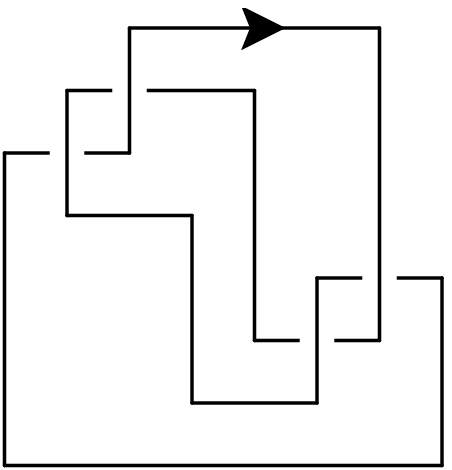}\ \right) \ar[ur]^{f_{0\star 0}} \ar[r]^{f_{00\star}} \ar[dr]_{f_{\star 00}} &  C^-\left(\ \dessin{1.1cm}{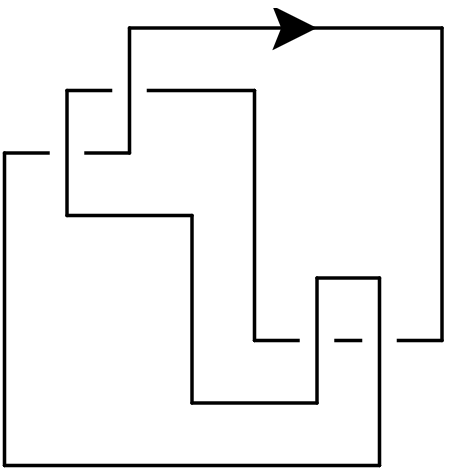}\ \right) \ar[ur]^(.3){f_{0\star 1}} \ar[dr]^(.3){f_{\star 01}} & C^-\left(\ \dessin{1.1cm}{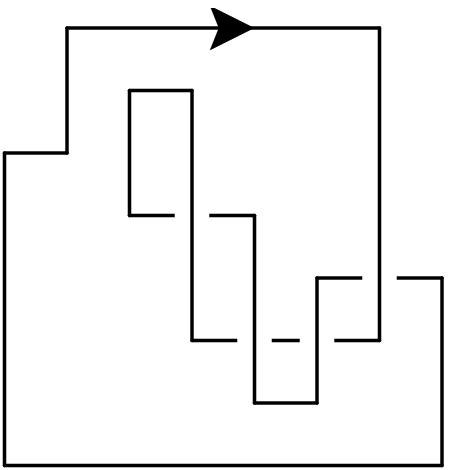}\ \right) \ar[r]^{f_{11\star}} & C^-\left(\ \dessin{1.1cm}{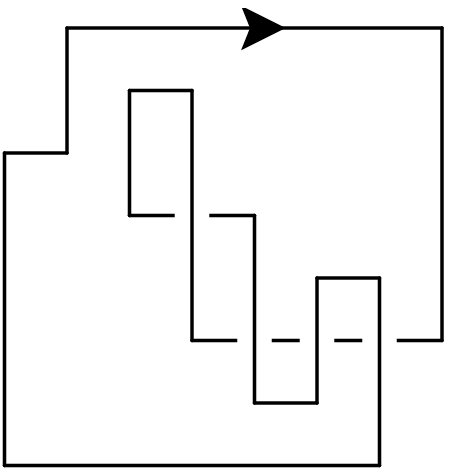}\ \right)\\
& C^-\left(\ \dessin{1.1cm}{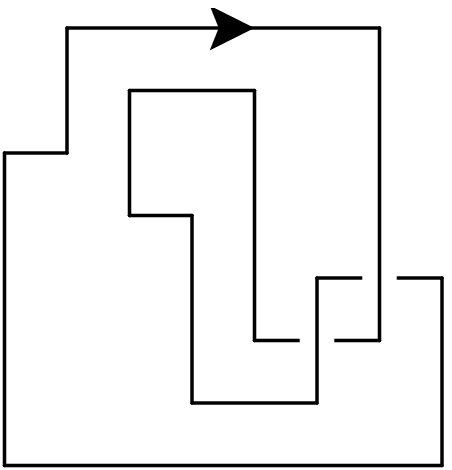}\ \right) \ar[ur]^(.3){f_{1\star 0}}|!{[u];[r]}\hole \ar[r]_{f_{10\star}} &  C^-\left(\ \dessin{1.1cm}{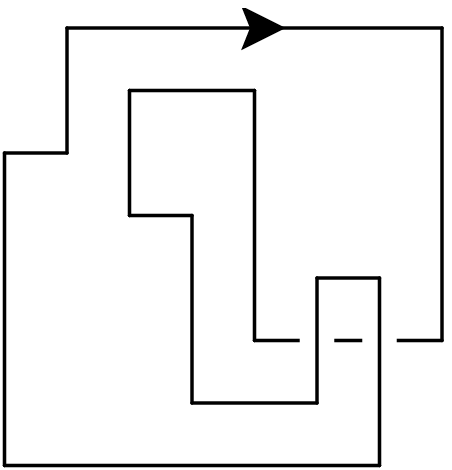}\ \right) \ar[ur]_{f_{1\star 1}} & }
$$
  \caption{Cube of resolution for a three time singularized figure eight knot}
  \label{fig:Cube}
\end{figure}

For every singular column, we choose a set of two arcs $\alpha_i$ and $\beta_i$, where $i$ is the label of the considered column, which corresponds to the construction given in paragraphs \ref{par:CombGrid}---\ref{par:WallMap}. We denote by $c_i\in\alpha_i\cap\beta_i$ the distinctive intersection of type $\dessin{.4cm}{Cstyle}$.\\

For all $I=(i_1,\cdots,i_k)\in\{0,1\}^k$, we denote by $G_I$ the regular grid obtained from $G$ by performing a $i_j$--resolution to the $j^{\textrm{th}}$ singular column for all $j\in \llbracket1,k\rrbracket$.\\

We are now in position to define a $k$--dimensional cube of maps $\C$ by considering the chain complexes $C^-(G_I)$ for all $I\in\{0,1\}^k$ and the chain maps $\func{f_{I(j:\star)}}{C^-(G_{I(j:0)})}{C^-(G_{I(j:1)})}$ defined in paragraph \ref{par:WallMap} using the arcs $\alpha_j$ and $\beta_j$  for all $I\in\{0,1\}^{k-1}$ and $j\in\llbracket 1,k\rrbracket$.\\

The cube of maps $\C$ is called the \emph{cube of resolution of G}\index{resolution!cube of}.

\begin{remarque}
If $k>1$, the cube of resolution $\C$ is \textbf{not} straight.
\end{remarque}

Actually, most of the terms in $f_{I(j:1,i:\star)}\circ f_{I(j:\star,i:0)} + f_{I(j:\star,i:1)}\circ f_{I(j:0,i:\star)}$, for all $I\in \{0,1\}^{k-2}$ and any pair of distinct integer $i\neq j\in\rrbracket 1,k \llbracket$, cancel by pairs. But the configurations of the following kind
\begin{eqnarray}\label{UniqueDec}
\dessin{3.3cm}{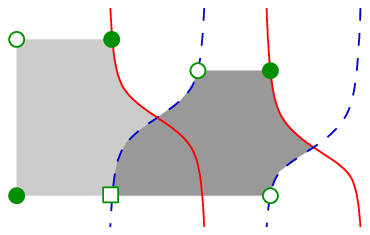}
\end{eqnarray}
have a unique decomposition.\\
As sketched in paragraph \ref{par:LameCube}, we will mend this anti-commutativity defect.

\subsubsection{Hexagons}
\label{par:Hexagons}

As usual, we identify the boundaries of $G$ in order to get a torus $\T_G$.\\

Let $I\in\{0,1\}^{k-2}$ and $i\neq j\in\llbracket 1,k\rrbracket$.\\
Let $x$ be a generator of $C^-(G_{I(i:0,j:0)})$ and $y$ a generator of  $C^-(G_{I(i:1,j:1)})$. An \emph{hexagon connecting $x$ to $y$}\index{polygon!hexagon}\index{hexagon|see{polygon}} is an embedded hexagon $\eta$ in $\T_{G}$ which satisfies:
\begin{itemize}
\item[-] edges of $\eta$ are embedded in the grid lines (including $\alpha_l$'s and $\beta_l$'s);
\item[-] the points $c_i$ and $c_j$ are corners of $\eta$;
\item[-] starting $c_i$ (resp. $c_j$) and running positively along the boundary of $\eta$, according to the orientation of $\eta$ inherited from the one of $\T_{G}$, the next three corners of $\eta$ are, successively and in this order, a point of $x$, a point of $y$ and $c_j$ (resp. $c_i$);
\item[-] except on $\p \eta$, the sets $x$ and $y$ coincide;
\item[-] the interior of $\eta$ does not intersect $\alpha_i\cup\beta_i\cup\alpha_j\cup\beta_j$ in a neighborhood of $c_i\cup c_j$.
\end{itemize}
The index $I(i:0,j:0)$ is called the \emph{origin of $\eta$}\index{polygon!hexagon!origin}\index{origin|see{polygon}}.

\begin{remarque}
  This definition of an hexagon \textbf{does not} coincide with the one given in \cite{MOST} (see the clear shape in Figure \ref{fig:Hex->Rect}). 
\end{remarque}

An hexagon $h$ is \emph{empty} if $\Int(\eta)\cap x=\emptyset$.\\
We denote by $\Hex(G)$ the set of all empty hexagons and by $\Hex(x,y)$ the set of those which connect $x$ to $y$.

\subsubsection{Spikes and hexagons}
\label{par:SpiHex}

\begin{prop}\label{Hex->Rect}
  Let $I\in\{0,1\}^k$. For a given hexagon $\eta\in\Hex(G)$ with origin $I$, there is a unique pair of spikes $(\tau,\tau')$ such that $\eta\cup \tau \cup \tau'$ is a rectangle $\rho\in \Rect(G_I)$.
\end{prop}

The proof is analogous to the corresponding proof for pentagons.

\begin{figure}[h]
  $$
  \dessin{3cm}{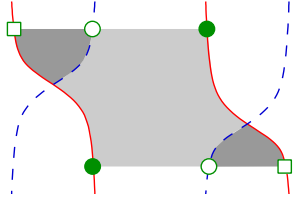}
  $$
  \caption{From hexagon to rectangle: {\footnotesize dark dots describe the initial generator of the hexagon while hollow ones describe the final one. The squares describe the initial generator for the two spikes.}}
  \label{fig:Hex->Rect}
\end{figure}

This defines a map $\func{\phi}{\Hex(G)}{\Rect(G)}$.

\subsubsection{Completion of the resolution cube}
\label{par:Completion}

Now, for all $I\in\{0,1\}^{k-2}$ and all $i\neq j\in\llbracket 1,k\rrbracket$, we can define the map
$$
\func{f_{I(i:\star,j:\star)}}{C^-(G_{I(i:0,j:0)})}{C^-(G_{I(i:1,j:1)})}
$$
as the morphism of $\Z[U_{O_1},\cdots,U_{O_{n+k}}]$--modules defined on the generators by
$$
f_{I(i:\star,j:\star)}(x)= \sum_{\substack{y \textrm{ generator}\textcolor{white}{R}\\ \textrm{of }C^-(G_{I(i:1,j:1)})\textcolor{white}{I}}} \sum_{\eta \in \Hex(x,y)} -\e\big(\phi(\eta)\big)U_{O_1}^{O_1(\eta)}\cdots U_{O_{n+k}}^{O_{n+k}(\eta)}\cdot y.
$$

\begin{prop}\label{D2Preserves}
  For all $I\in\{0,1\}^{k-2}$ and $i\neq j\in\llbracket 1,k\rrbracket$, the map $f_{I(i:\star,j:\star)}$ preserves the Alexander filtration and increases the Maslov grading by one.
\end{prop}

The proof is analogous to the corresponding proof for pentagons.\\

Now the maps $\{f_{I(i:\star,j:\star)}\}$ enable us to correct the defect of commutativity in $\C$ by adding large diagonals to the faces of $\C$:
$$
\xymatrix@!0@C=1.9cm@R=1.9cm{
  &C^-(G_{I(i:1,j:0)}) \ar[dr]^{f_{I(i:1,j:\star)}}&\\
  C^-(G_{I(i:0,j:0))}) \ar[ur]^{f_{I(i:\star,j:0)}}\ar[dr]_{f_{I(i:0,j:\star)}} \ar@[red][rr]^[red]{f_{I(i:\star,j:\star)}}&&C^-(G_{I(i:1,j:1)})\\
  &C^-(G_{I(i:0,j:1)}) \ar[ur]_{f_{I(i:\star,j:1)}}&
}.
$$ 

\subsection{Singular link Floer homology}
\label{ssec:SingFlHom}

The Heegaard-Floer complex for a singular grid can now be defined as the generalized cone of its completed cube of resolution.

\subsubsection{Definitions}
\label{par:MainDefinition}

The $\Z[U_{O_1},\cdots,U_{O_{n+k}}]$--module $C^-(G)$ is defined by
$$
C^-(G)= \bigoplus_{I\in \{0,1\}^k}C^-(G_I)[\# 0(I)].
$$

\begin{remarque}\label{3emeFiltration}
  The module $C^-(G)$ is naturally endowed with a third grading of which $\# O(I)$ is the degree.
\end{remarque}

We also set
$$
\p_G^-=\p_0^-+\p_1^-+\p_2^-
$$
where the three maps $\p^-_i,i=1,2,3$ are morphisms of $\Z[U_{O_1},\cdots,U_{O_{n+k}}]$--modules defined on $C^-(G_I)$ for all $I\in\{0,1\}^k$ by

\begin{gather*}
  \p_0^-(x)=\p^-_{G_I}(x)\\[.3cm]
  \p_1^-(x)=\sum_{i\in 0(I)} f_{I_{i\leftarrow\star}}(x)\\[.3cm]
  \p_2^-(x)=\sum_{\substack{i,j\in 0(I)\\i\neq j}} f_{I_{i,j\leftarrow\star}}(x),
\end{gather*}
where $I_{i\leftarrow\star}$ (resp. $I_{i,j\leftarrow\star}$) is obtained from $I$ by turning the $i^\textrm{th}$ element (resp. the $i^\textrm{th}$ and $j^\textrm{th}$ elements) into $\star$.\\
\index{homology!knot Floer!singular}\index{homology!link Floer!singular}

This map does not respect the third grading defined in Remark \ref{3emeFiltration}. But however, it does respect the induced filtration. This filtration will have a key role in the invariance proofs of section \ref{sec:Invariance}.

\subsubsection{Consistency}
\label{par:Consistency}

\begin{prop}\label{Cchain}
  The couple $\big(C^-(G),\p^-_G\big)$ is a filtrated chain complex \ie the map $\p_G^-$
  \begin{itemize}
  \item[-] decreases the Maslov grading by one;
  \item[-] preserves the Alexander filtration;
  \item[-] satisfies ${\p_G^-}^2=0$.
  \end{itemize}
\end{prop}

\begin{figure}
$$
\hspace{-.1cm} \begin{array}{ccccc}
\dessin{3cm}{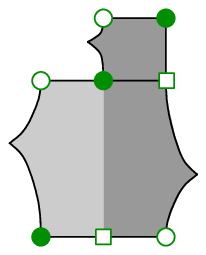} & \hspace{1cm} & \dessin{3cm}{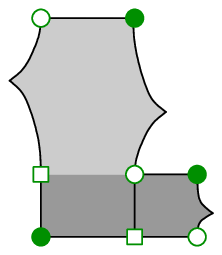} & \hspace{1cm} &\dessin{3cm}{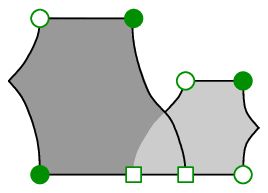}\\[1.5cm]
a) && b) && c)
  \end{array}
$$
  \caption{Pentagons and hexagons: {\footnotesize dark dots describe the initial generator while hollow ones describe the final one. The squares describe intermediate states. One decomposition is given by its border and the second by shading of different intensity. In case $a)$ and $b)$, the usual anti-commutativity of signs for rectangles make the two terms cancel since the pentagons are simultaneously pointing to the same direction. In case $c)$, the order of application of the polygons is the same, but pentagons are pointing in opposite directions.}}
  \label{fig:PentvsHex}
\end{figure}

\begin{proof}
  The first two points are direct consequences of propositions \ref{D0Preserves} (\S\ref{par:Diff}), \ref{D1Preserves} (\S\ref{par:WallMap}) and \ref{D2Preserves} (\S\ref{par:Completion}). These propositions also state that ${\p^-_0}^2=0$ and $\p^-_0\circ\p^-_1+\p^-_1\circ\p^-_0=0$. Hence, it is sufficient to prove
\begin{gather}
\p^-_0\circ\p^-_2+{\p^-_1}^2+\p^-_2\circ\p^-_0=0\label{Formula1}\\[.3cm]
\p^-_1\circ\p^-_2+\p^-_2\circ\p^-_1=0\label{Formula2}\\[.3cm]
{\p^-_2}^2=0.\label{Formula3}
\end{gather}

The formula (\ref{Formula3}) is a corollary of the anti-commutativity of the signs for rectangles since the sets of hexagon corners must be disjoint. This holds everytime the two involved polygons have disjoint sets of corners.\\

For the formula (\ref{Formula2}), note that a pentagon and a hexagon can share at most one corner. Then, three configurations can occur and, for each of them, there are two cancelling decompositions. They are described in Figure \ref{fig:PentvsHex}.\\

The formula (\ref{Formula1}) finally holds for the same reasons. All the possible configurations are obtained by embanking a peak in one of the three cases in Figure \ref{fig:PentvsHex}. We can note that the minus sign in the definition of $f_{I(i:\star,j:\star)}$ is essential for all the cases derivated from $c)$ and those derivated from $a)$ by embanking a peak pointing to the left.\\
Note that the previous configurations contain the cases of thin horizontal annuli which may arise two times as a juxtaposition of a rectangle and a hexagon and two times as a juxtaposition of two pentagons. Hexagons and pairs of pentagons cannot be involved in the filing of a vertical annulus since sets of arcs $\{(\alpha_i,\beta_i)\}$ are necessarily separated by a regular vertical grid line.
\end{proof}

\subsubsection{Singular link Floer homology as a mapping cone}
\label{par:HomologyAsCone}

Even if the complex defined above is not derived from a straight cube of maps, it can be seen as a sequence of consecutive mapping cones.\\

Actually, let choose a distinguished singular column in the grid $G$, then the differential $\p_G^-$ can be split into $\p'_G$, the sum over polygons which have no peak on this distinguished singular column, and $f$, the sum over those which do have. It is clear from the proof of Proposition \ref{Cchain} (\S\ref{par:Consistency}) that $\p'_G$ is still a differential and, since ${\p^-}^2\equiv 0$, the map $f$ is indeed a chain map.\\
The chain complex $C^-(G)$ can be seen as the mapping cone of $f$.\\

We can now process recursively on each side of $f$ to get a description of $C^-(G)$ as a sequence of mapping cones. However, because of the terms from $\p^-_2$, this description strongly depends on the underlying ordering on the singular columns.

\subsubsection{Invariance}
\label{par:Invariance}

In the case of a regular grid, the associated homology is trivially the combinatorial link Floer homology of the associated link $L$ and, consequently, is an invariant of $L$. This fundamental property extends to the singular cases.

\begin{theo}\label{Invariance}
  The homology of $\big(C^-(G),\p^-_G\big)$ is an invariant of the singular link associated to $G$.
\end{theo}

The next section is devoted to the proof of this theorem.\\

%%% Local Variables: 
%%% mode: latex
%%% TeX-master: "These"
%%% End: 

% Invariance par mouvements de Reidemeister
\section{Invariance}
\label{sec:Invariance}

The Theorem \ref{Invariance} (\S\ref{par:Invariance}) can be divided in seven points:
\begin{enumerate}
\item[i)] invariance under isotopies of arcs $\alpha_l$'s and $\beta_l$'s;
\item[ii)] invariance under cyclic permutations of the rows or of the columns;
\item[iii)] invariance under stabilization/destabilization;
\item[iv)] invariance under commutation of two rows or two columns;
\item[v)] invariance under commutation of a regular column with a singular one;
\item[vi)] invariance under commutation of two singular columns;
\item[vii)] invariance under flip.\\
\end{enumerate}

If working with coefficient in $\fract \Z / {2\Z}$, the point ii) is trivial since the construction of the chain complex uses only polygons embedded in the torus. Moreover, the sign refinement is clearly invariant by cylic permutations of the rows.\\
Concerning cyclic permutations of the columns, it is proven in \cite{MOST} and \cite{Gallais} that the sign assignment for rectangles is essentially unique. At least, it gives isomorphic chain complexes.\\
Since the sign assignment for other polygons depends on its definition for rectangles and since a cyclic permutations of the columns will not change the direction toward which a pentagon is pointing, the invariance still holds for our construction.\\

Unfortunatly, because of the number of cases which need to be treated, proofs for points vi) and vii) have not been completed. However, we give a strategy to achieve them.

\subsection{Isotopies of arcs}

\subsubsection{Moving peaks}
\label{par:ArcsIsotopies}

To prove the invariance under isotopies of arcs $\alpha_l$'s and $\beta_l$'s, it is sufficient to deal with the intersection $c_i\in\alpha_i\cap\beta_i$ going through a horizontal grid line for a given $i\in \llbracket 1,k\rrbracket$.\\

Let $(\alpha_i,\beta_i)$ and $(\alpha'_i,\beta'_i)$ be two such sets of arcs (see (\ref{AboveBelow}), \S\ref{par:UpToHom}). We denote by $\p$ and $\p'$ the corresponding differentials.\\

Now, we consider the map $h$ defined in paragraph \ref{par:UpToHom} which associates generators connected by a small spike.\\
We set $\func{\varphi}{C^-(G)}{C^-(G)}$ as the morphism of $\Z[U_{O_1},\cdots,U_{O_{n+k}}]$--modules defined on the generators by
$$
\forall I\in\{0,1\}^{k-1}, \varphi(x) = \left\{\begin{array}{ll} x - h(x) & \textrm{if }x\in C^-(G_{I(i:0)})\\[2mm] x  & \textrm{otherwise}  \end{array} \right..
$$

\begin{lemme}
  The map $\func{\varphi}{(C^-(G),\p)}{(C^-(G),\p')}$ is an isomorphism of chain complexes.  
\end{lemme}

Contrary to most maps in this thesis, the map $\varphi$ is indeed commuting with differentials and not anti-commuting.

\begin{proof}
  Using a simple calculus, proving that the map $\varphi$ commutes with the differentials can be reduced to proving that for all $I\in \{0,1\}^{k-1}$ and all generator $x\in C^-(G_{I(i:0)})$
  $$
  \p(x)-\p'(x)=h\circ\p(x)-\p'\circ h(x).
  $$ 
  
  Then, the proof is similar to the proof of Proposition \ref{Homotopie} (\S\ref{par:UpToHom}). When sharing a corner, a rectangle and a small spike give a pentagon, a pentagon and a small spike give an hexagon and, finally, hexagons and spikes cannot share a corner.\\
  
  The fact that $\varphi$ is a bijection is clear.
\end{proof}

The two chain complexes are isomorphic, so they share the same homology.

\subsection{Stabilization/Destabilization}

Now, we consider the following stabilization move:

\begin{eqnarray}\label{thedot}
  \begin{array}{ccc}
    \dessin{1.7cm}{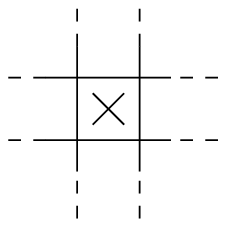} & \longrightarrow & \dessin{1.7cm}{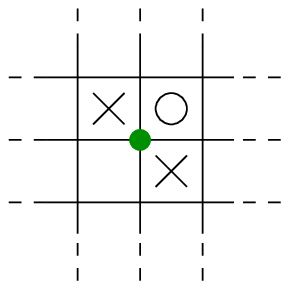}.\\[1cm]
    G && G_s\\
  \end{array}
\end{eqnarray}

We label by $1$ the new $\O$--decoration and by $2$ the one which is lying on the same row than the initial $\X$--decoration.\\

According to the nature of the initial decoration and to the square which is left empty after stabilization, there are seven others cases. Nevertheless, {\it mutatis mutandis}, the proof remains valid for all of them.\\

According to the restrictions on (de)stabilizations given in paragraph \ref{par:SingGridEtElemMoves}, $O_1$ does not belong to a singular column.\\

Here, the proof follows the same lines than in section $3.2$ of \cite{MOST}. We recall the broad outlines of it.

\subsubsection{Invariance for  regular  link Floer homology}
\label{par:RegStab}

\begin{description}
\item[i) Description of $C^-(G)$ using $G_s$] Every generator of $C^-(G)$ can be seen as drawn on $G_s$ by adding $x_0$, the dot located at the south-west corner of $O_1$ (see the dot in (\ref{thedot})). The gradings are the same and the differential is given by ignoring the conditions involving $O_1$ and $x_0$ \ie rectangles may contain $x_0$ in their interior (but not in their boundaries) and there is no multiplication by $U_1$.
\item[ii) Description of $H^-(G)$ involving $U_1$] The chain map $C^-(G)$ is quasi-isomorphic to the mapping cone $C=C_1[1]\oplus C_2$ of the map
  $$
  \xymatrix@C=2cm{C_1\simeq \left(C^-(G)\otimes\Z[U_1]\right)\{-1\}[-2]\ar[r]^(.55){\times (U_2-U_1)} & C^-(G)\otimes\Z[U_1]\simeq C_2}.
  $$
  Hence, it is sufficient to define a quasi-isomorphism from $C^-(G_s)$ to $C$.
\item[iii) Simplifying filtration] There are filtrations on $C^-(G_s)$ such that the associated graded differential is the sum only over thin rectangles which are contained in the row or in the column through $O_1$ and which do not contain $O_1$ or any $X$.
\item[iv) Graded quasi-isomorphism] The associated graded chain complex has then the following decompostion in subcomplexes (details about the notation used here are given in Appendix \ref{appendix:Subcomplexes}):
  \begin{gather*}
    \xymatrix@C=2cm{
      \left\{\dessin{1.5cm}{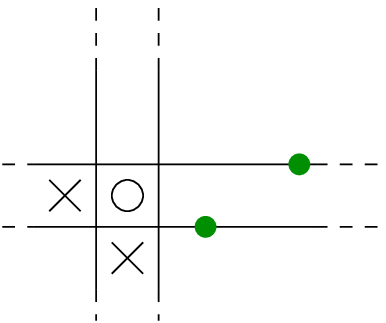}\right\} \ar[r]^{\dessin{.9cm}{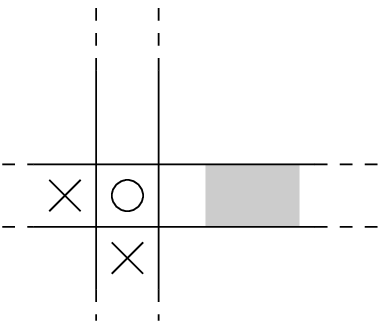}}_\sim \ar@{}|{}="Nya"  \ar@(dl,ul)"Nya"!<-1.4cm,-.2cm>;"Nya"!<-1.4cm,.4cm>^{\dessin{1.1cm}{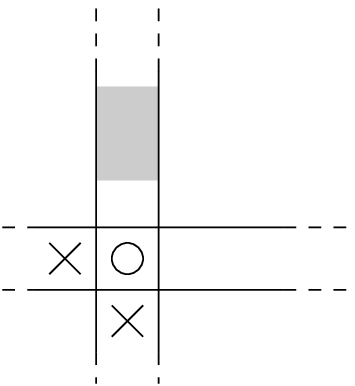}} & \left\{\dessin{1.5cm}{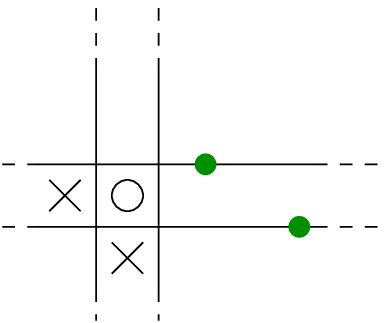}\right\}\ar@{}|{}="Nya2"  \ar@(dr,ur)"Nya2"!<1.4cm,-.2cm>;"Nya2"!<1.4cm,.4cm>_{\dessin{1.1cm}{DiffVert}}}\\
    \oplus\\
    \xymatrix{\left\{\dessin{1.5cm}{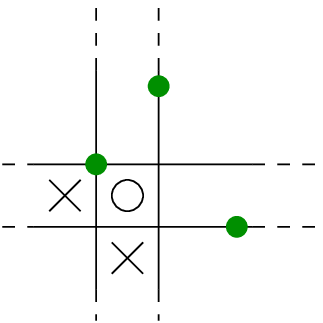}\ -\ \dessin{1.5cm}{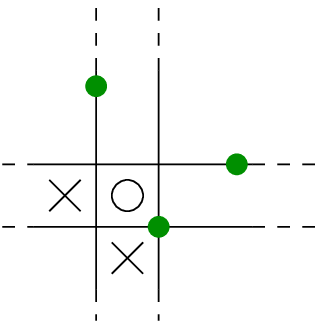}\right\} \ar@{}|{}="Nya" \ar@(dl,ul)"Nya"!<-2.4cm,-.2cm>;"Nya"!<-2.4cm,.4cm>^0 \ar@*{[white]}@(dr,ur)"Nya"!<2.4cm,-.2cm>;"Nya"!<2.4cm,.4cm>^{\textcolor{white}{0}}}\\
    \oplus\\
    \xymatrix{\left\{\dessin{1.5cm}{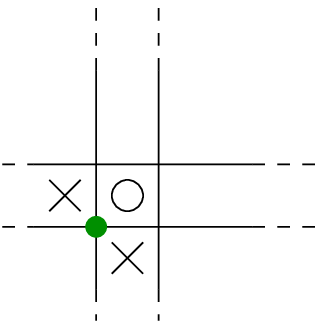}\right\}. \ar@{}|{}="Nya" \ar@(dl,ul)"Nya"!<-1.3cm,-0.2cm>;"Nya"!<-1.3cm,.4cm>^0 \ar@*{[white]}@(dr,ur)"Nya"!<1.3cm,-0.2cm>;"Nya"!<1.3cm,.4cm>_{\textcolor{white}{0}}}
  \end{gather*}
  This allows us to define a graded chain map $\func{F_{gr}}{C^-(G_s)}{C}$ by
  $$
  F_{gr}:=\left\{\begin{array}{l}
      \xymatrix@C=2cm@R=0.5cm{\dessin{1.3cm}{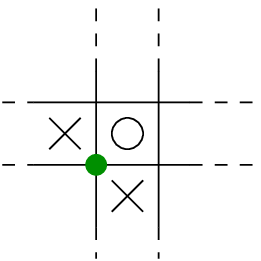} \ar@{|->}[r]^(.45){Id}& \dessin{1.3cm}{Gen1}\in C_1[1]\\
        \dessin{1.3cm}{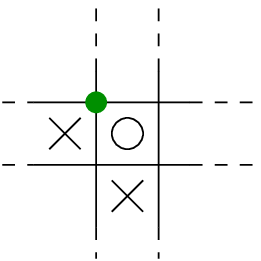} \ar@{|->}[r]^(.45){\dessin{.9cm}{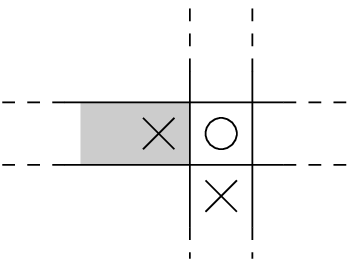}} & \dessin{1.3cm}{Gen1}\in C_2 \hspace{.2cm}\\
        \dessin{1.3cm}{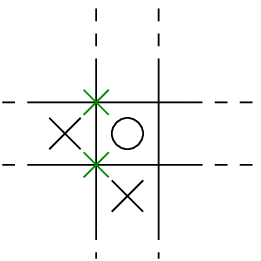} \ar@{|->}[r] & \hspace{.5cm} zero \hspace{1.2cm}}\end{array} \right..
  $$
  It is obviously a quasi-isomorphism for the associated graded chain complexes.
\item[v) Filtrated extension of $F_{gr}$] The map $F_{gr}$ can be extended to a map $\func{F}{C^-(G_s)}{C}$ of filtrated chain complexes. Essentially, commutativity of $F$ with rectangles which are empty except concerning $x_0$ which is actually contained in their interior (\ie rectangles embedded in the grid $G_s$ which are involved in the differential associated to $G$ as described in i) but not in the differential associated to $G_s$) imposes inductively additional terms in the definition of $F$.
\item[vi) Filtrated quasi-isomorphism] Corollary. \ref{QuasiIso} (\S\ref{par:MapCones}) completes the proof.
\end{description}

Points i), ii) and v) extend trivially to the singular case.

\subsubsection{Crushing filtration}
\label{par:CrushingFiltration}

Concerning the filtration of the point iii), we give an alternative construction than the one given in \cite{MOST}.\\
Let $x$ be a generator of $C^-(G_s)$. By construction, the extra column and row in $G_s$ can be crushed in order to get back to $G$. If doing this with $x$ drawn on the grid, it gives a set $\widetilde{x}$ of dots which is almost a generator of $C^-(G)$ except that one horizontal and one vertical lines have two dots on it. It may happen that two dots merge, the resulting dot is then counted with multiplicity two.\\
We perform a few cyclic permutations of the rows and columns in such a way that these singular vertical and horizontal grid lines are on the border of the grid. The upper right corner of the grid is then filled with an $X$ denoted by $X_*$.\\

Contrary to the convention used until this point, we draw dots of the extremal grid lines on the rightmost and uppermost ones.\\

Now, we can consider
$$
M_G(x):= M_{\O_G}(\widetilde{x})+ \#0(I)
$$
where $\O_G$ is the set of $\O$--decorations of $G$, $M_{\O_G}$ is defined in paragraph \ref{par:MaslAlex} and $I$ is the element of $\{0,1\}^k$ such that $x\in C^-(G_I)$.\\

\begin{figure}[h]
  $$
    \dessinH{3.2cm}{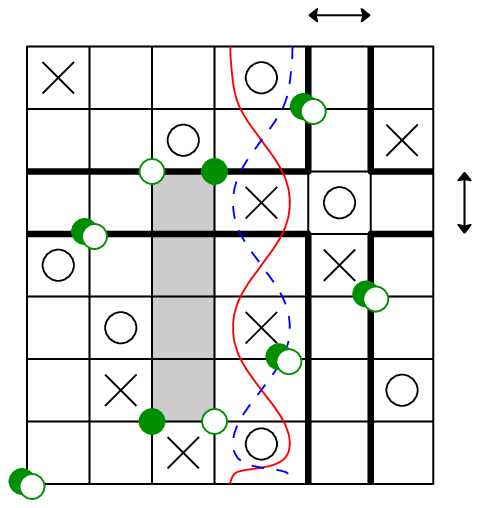}\ \begin{array}{c}\textrm{\scriptsize crushing}\\[-.1cm] \leadsto\\[-.1cm] \textrm{\scriptsize process} \end{array} \dessinH{3.2cm}{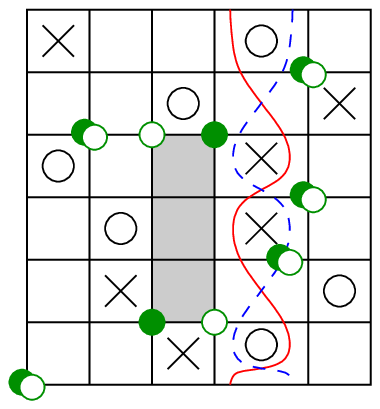} \begin{array}{c}\textrm{\scriptsize cyclic}\\[-.1cm] \leadsto\\[-.1cm] \textrm{\scriptsize perm.} \end{array} \dessinH{3.2cm}{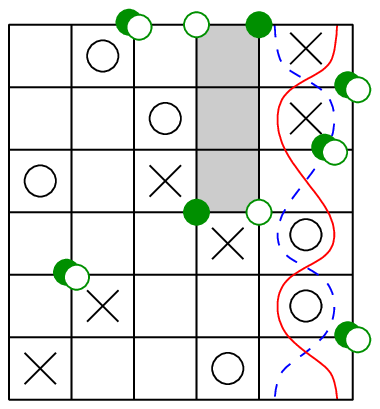}  \leadsto \ \ \begin{array}{l} M_G(x)=-3 \\[.3cm] M_G(y)=-4 \end{array}
  $$
  $$
    \dessinH{3.2cm}{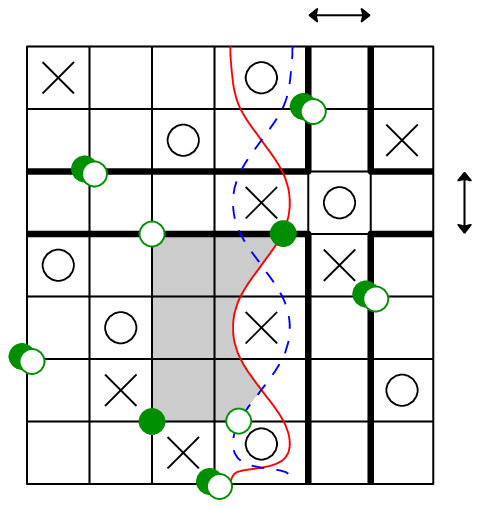}\ \begin{array}{c}\textrm{\scriptsize crushing}\\[-.1cm] \leadsto\\[-.1cm] \textrm{\scriptsize process} \end{array} \dessinH{3.2cm}{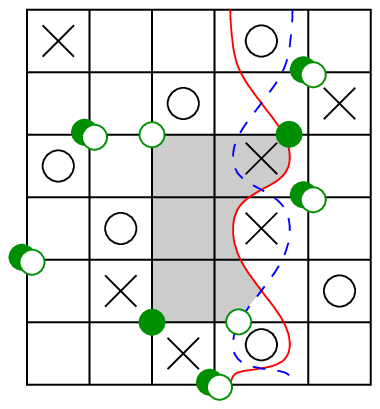} \begin{array}{c}\textrm{\scriptsize cyclic}\\[-.1cm] \leadsto\\[-.1cm] \textrm{\scriptsize perm.} \end{array} \dessinH{3.2cm}{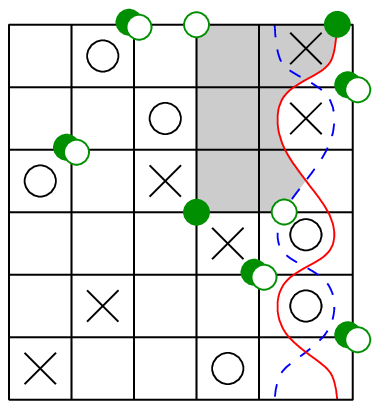}  \leadsto \ \ \begin{array}{l} M_G(x')=-5 \\[.3cm] M_G(y')=-7 \end{array}
  $$
  \caption{Crushing row and column:  {\footnotesize dark dots describe the initial generators $x$ and $x'$ while hollow ones describe the final ones $y$ and $y'$. Polygons are depicted by shading.}}
  \label{fig:Crushing}
\end{figure}

Then, we consider the graded differential $\widetilde{\p}$ associated to the Alexander filtration and to the $O$--filtration for all $O\in\O$ (see \S\ref{par:FloerHomologies}).

\begin{prop}\label{CrushingFilt}
  $M_G$ defines a filtration on $\big(C^-(G_s),\widetilde{\p}\big)$. Moreover, the associated graded differential corresponds to the sum over rectangles which are contained in the row or in the column through $O_1$ and which do not contain $O_1$ or any $X$.\index{filtration!crushing}
\end{prop}
\begin{proof}
Let $x$ and $y$ be two generators of $C^-(G_s)$.\\
In the crushing process, any empty rectangle $\rho$ on $G_s$ containing no decoration gives rise to an empty rectangle $\widetilde{\rho}$ on $G$ which is also empty of decoration except, possibly, $X_*$. It may happen that a dot is pushed into its border, but not into its interior.\\
If $\rho$ is connecting $x$ to $y$, then $\widetilde{\rho}$ is connecting $\widetilde{x}$ to $\widetilde{y}$.\\

It may happen that $\rho$ is totally flattened during the crushing process. Then $\widetilde{x}=\widetilde{y}$ and $M_G(x)=M_G(y)$. We suppose now that  $\widetilde{\rho}$ is not flat. Since $X_*$ does not interfere with $M_G$, the fact that $\widetilde{\rho}$ contains it or not, does not matter.\\

Here, the value $M_G$ is not invariant under cyclic permutations of rows and columns. Then we cannot assume that $\widetilde{\rho}$ is not ripped. Anyway, we can check independantly all the cases.\\

{\parpic[r]{$\dessin{2.2cm}{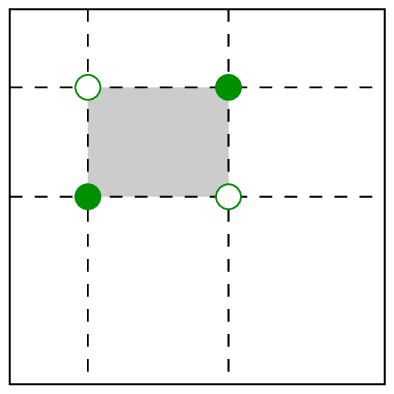}\ $}
\vspace{.3cm}
If $\widetilde{\rho}$ is not ripped and does not have any extra dot on its border, then we can apply the same arguments than in the proof of Proposition \ref{RectGradings} (\S\ref{par:Rect}).\\
Hence, we have $M_G(y)=M_G(x)-1$.\\

}

{\parpic[l]{$\ \dessin{2.2cm}{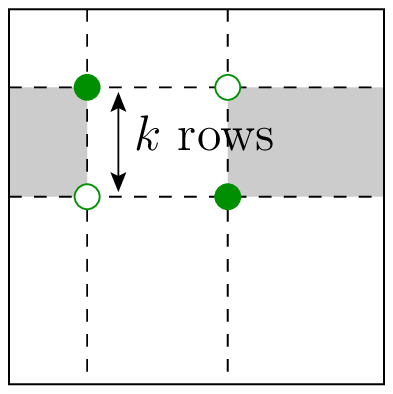}$}
\vspace{.1cm}
If $\widetilde{\rho}$ is horizontally ripped in two then we can use the same reasoning but applied to its horizontal complement. Since $\widetilde{\rho}$ is empty, containing no decorations, its complement must contain $k$ $O$'s and $k-1$ dots, where $k$ is the height of $\widetilde{\rho}$.\\ 
Then, we obtain $M_G(y)=M_G(x)+1+2(k-1)-2k=M_G(x)-1$.\\

}

{\parpic[r]{$\dessin{2.2cm}{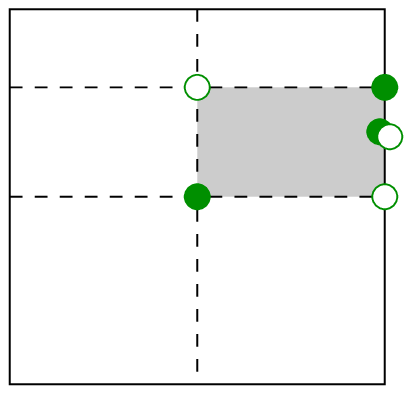}\ $}
\vspace{.4cm}
If $\widetilde{\rho}$ has an extra dot on its right border, then there is an extra term in $\I(\widetilde{x},\widetilde{x})$ which does not appear in $\I(\widetilde{y},\widetilde{y})$.\\ 
As a result, $M_G(y)=M_G(x)-2$.\\

}

{\parpic[l]{\ $\dessin{2.2cm}{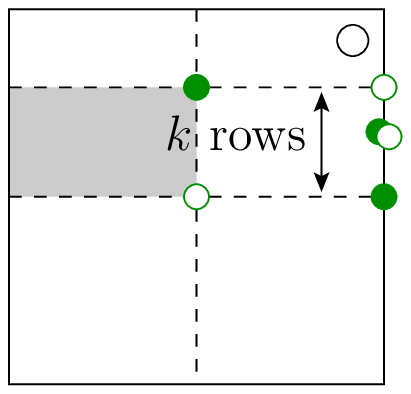}$}
\vspace{0.1cm}
If $\widetilde{\rho}$ has an extra dot on its left border, we consider, here again, its horizontal complement $\widetilde{\rho}'$. Now, the extra term appears in $\I(\widetilde{y},\widetilde{y})$ but, on the other hand, $\widetilde{\rho}$ contains only $k-2$ dots in its interior.\\ 
Finally, $M_G(y)=M_G(x)+2+2(k-2)-2k=M_G(x)-2$.\\

}

Vertical ripping can be treated in the same way.\\
At last, because of $O_1$, a rectangle cannot be ripped in four pieces.\\

Concerning pentagons and hexagons, problems occur when the crushed column and row do contain possible peaks. However, because of the restrictions given in paragraph \ref{par:SingGridEtElemMoves}, the crushed column cannot contain any of them and, if necessary, it is always possible to choose arcs $\alpha$'s and $\beta$'s such that no peak belongs to the crushed row. This precaution being taken, everything works the same.\\ 

The graduation $M_G$ induces then a filtration on $\big(C^-(G_s),\widetilde{\p}\big)$. Furthermore, it is clear that $M_G$ is only preserved by rectangles
 which are flattened during the crushing process. This corresponds exactly to rectangles contained in the row or in the column through $O_1$.
\end{proof}

The point iv) holds without any change.

\subsubsection{Extension of $F_{gr}$}
\label{par:FgrExtension}

Concerning the extension of $F_{gr}$ to $F$, the ins and outs are, once again, the same. This leads to a similar map than in \cite{MOST}, except that some peaks may have been added to the vertical boundary lines. However, this does not change the combinatorics of such domains as soon as we add a minus sign for each peak pointing toward the left.\\
Hence, the chain map $F$ is defined as the sum over polygons described in Figure \ref{fig:F}.

\begin{figure}
$$
\begin{array}{cccc}
  \dessin{4.8cm}{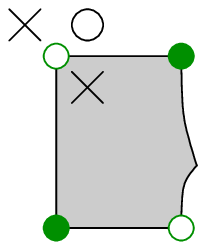} & \dessin{4.8cm}{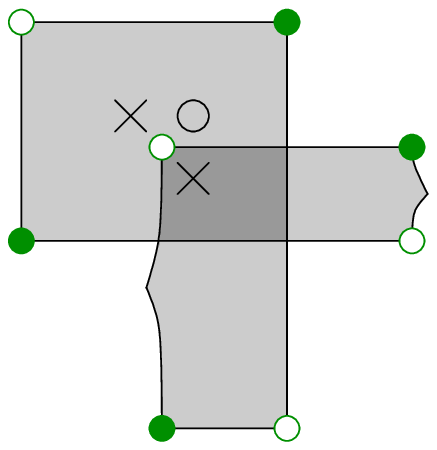} & \dessin{4.8cm}{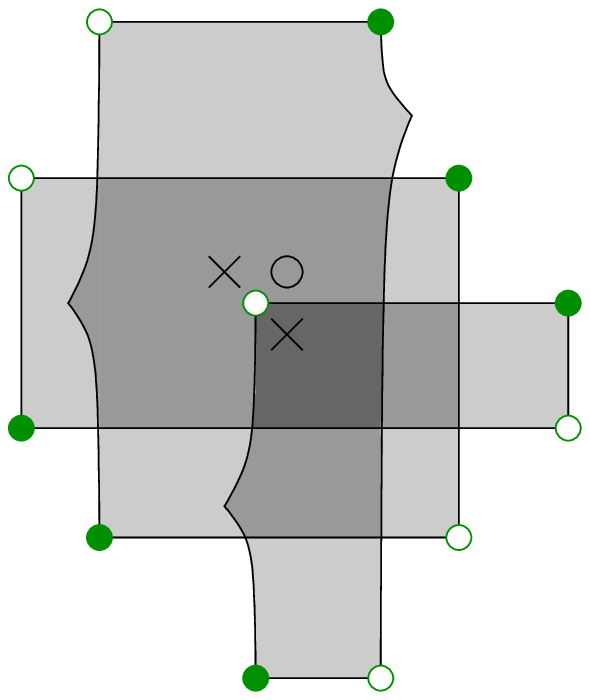} & \ \cdots \\[2cm]
  \dessin{4.8cm}{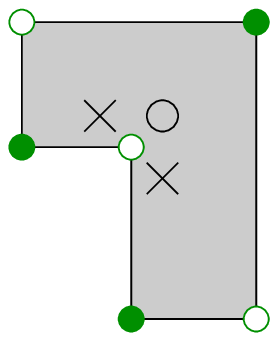} & \dessin{4.8cm}{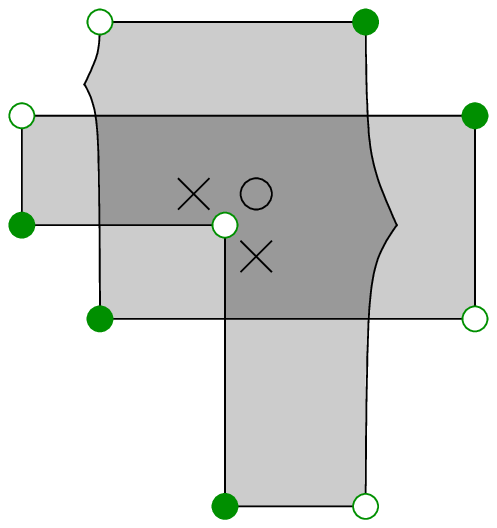} & \dessin{4.8cm}{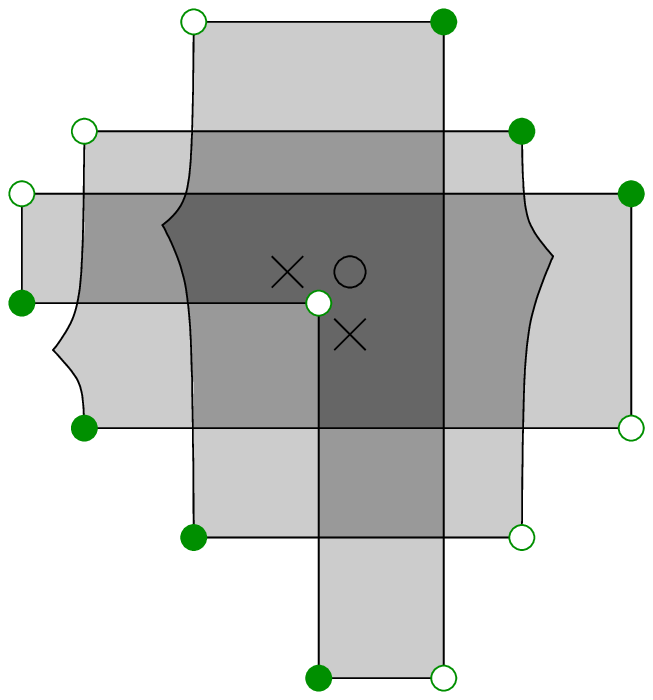} & \ \cdots 
\end{array}
$$
  \caption{Function F: {\footnotesize dark dots describe the initial generator while hollow ones describe the final one. Polygons are depicted by shading. Generators of $C^-(G_s)$ are connected to generators of $C_2$ by polygons in the top row and to generators of $C_1[1]$ by polygons in the second row. The second row contains also the trivial domain which does not move any dot. There is a canonical way to decompose any such domain $\D$ in a juxtaposition of rectangles, pentagons and hexagons. Signs are then given by multiplying the signs assigned to all these polygons and then switching it $m$ times where $m$ is the number of peaks pointing to the left in $\D$.}}
  \label{fig:F}
\end{figure}

\subsection{Regular commutation}

By a regular commutation, we mean a commutation move which involves only regular objects like rows or regular columns.\\

The proof will be a combination of constructions and proofs made earlier in this paper or in \cite{MOST}.

\subsubsection{Commutation map}
\label{par:RegCommutation}

Since the decorations of one of the two commuting columns are strictly above the decorations of the other one, the move can be seen as replacing a distinguished vertical grid line $\alpha$ by a different one $\beta$, as pictured below:

\begin{eqnarray}
\dessin{3.5cm}{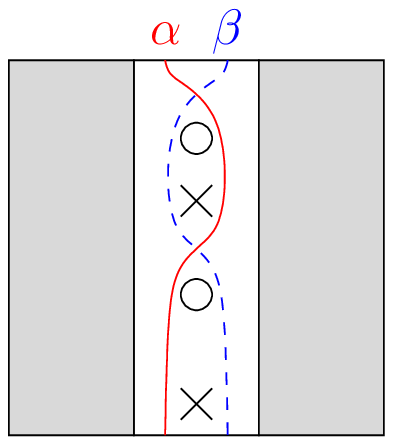}.
\label{fig:CommRegReg}
\end{eqnarray}

We denote by $G_\alpha$ and $G_\beta$ the corresponding grids.
Using the striking similarity between pictures (\ref{fig:CommRegReg}) and (\ref{fig:CombGrid}) (\S\ref{par:CombGrid}), we can define \emph{pentagons}\index{polygon!pentagon} and \emph{hexagons}\index{polygon!hexagon} as in paragraphs \ref{par:Pentagons} and \ref{par:Hexagons}, using hence the same terminology, but with the condition that at least one of their peaks relies on $\alpha\cap\beta$. However, there is no ambiguity here of which point in $\alpha\cap\beta$ should be used as a peak.\\

For all generators $x$ of $C^-(G_\alpha)$ and $y$ of $C^-(G_\beta)$, we denote by $\Pol^\circ(x,y)$ the set of such empty pentagons and hexagons connecting $x$ to $y$.\\
Then, we can set the map $\func{\phi_{\alpha\beta}}{C^-(G_\alpha)}{C^-(G_\beta)}$ as the morphism of $\Z[U_{O_1},\cdots,U_{O_n}]$--modules defined on the generators by
$$
\phi_{\alpha\beta}(x)=\sum_{\substack{y \textrm{ generator}\textcolor{white}{R}\\\textrm{of }C^-(G_\beta)\textcolor{white}{I}}} \sum_{\pi\in\Pol^\circ(x,y) } \e(\pi)U_{O_1}^{O_1(\pi)}\cdots U_{O_n}^{O_n(\pi)}\cdot y,
$$
where the sign $\e(\pi)$ is defined as in section \ref{par:WallMap}.\\

Contrary to the definition of $f$, the arcs $\alpha$ and $\beta$ play symmetric roles.\\

The map $\phi_{\alpha\beta}$ is a filtrated chain map. The points which should be checked to prove the anti-commutativity with the differentials and to prove that it preserves Maslov grading and Alexander filtration have already been checked in the proof of \ref{Cchain} (\S\ref{par:Consistency}).

\subsubsection{Resolution filtration}
\label{par:LastFiltration}

Now, we can consider the filtration induced by the third grading on $C^-(G)$ defined in Remark \ref{3emeFiltration} (\S\ref{par:MainDefinition}) which counts the singular columns positively resolved. The differentials and the map $\phi_{\alpha\beta}$ clearly preserve this filtration. Moreover, the associated graded chain complexes are the direct sums of the chain complexes associated to every desingularized grid, and, restricted to any of these subcomplexes, the graded map associated to $\phi_{\alpha\beta}$ is the eponyme morphism defined in \cite{MOST}. In the latter, it is proved that it is a quasi-isomorphism. The proof follows the same lines than in paragraph \ref{par:Candidate}, except that, in the present case, the inverse map does preserve the Alexander filtration.\\

Finally, Corollary \ref{QuasiIso} (\S\ref{par:MapCones}) completes the proof.\\

The same arguments can be applied to the case of rows commutations.\\

\subsubsection{Case of a single double point}
\label{par:SingleComm}

The result for a regular commutation is sufficient to deal with links with a single double point. Actually, a commutation of the singular column with a regular one can be replaced by regular commutations and cyclic permutations:
$$
\xymatrix{\dessin{2.5cm}{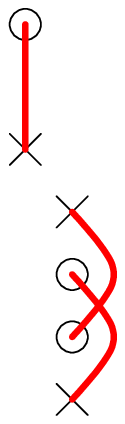} \ar@{<->}[r] & \dessin{2.5cm}{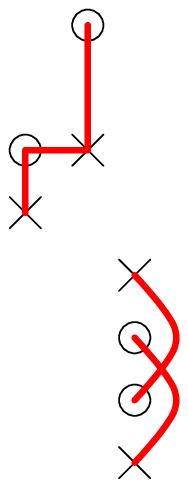} \ar@{<->}[r] & \dessin{2.5cm}{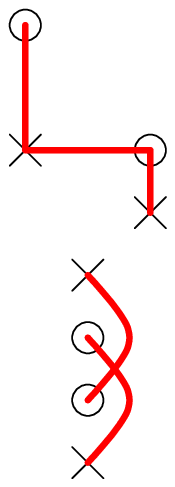} \ar@{<->}[r] & \dessin{2.5cm}{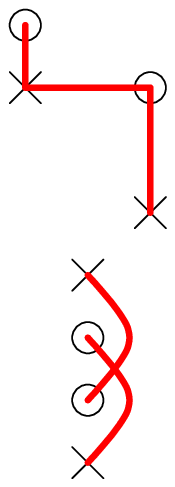} \ar@{<->}[r] & \dessin{2.5cm}{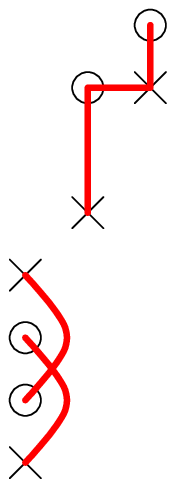} \ar@{<->}[r] & \dessin{2.5cm}{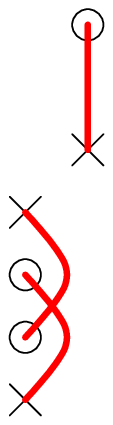}}
$$

\subsection{Semi-singular commutation}

Now, we consider the commutation of a regular column with a singular one.

\subsubsection{Reduction of the cases}
\label{par:ReductionCases}

Even if it means to perform first a few cyclic permutations of the rows, we can assume that the decorations of the regular column are all above the decorations of the singular one. It will be clear in the proof that the order in which the regular $X$ and the regular $O$ are placed has no incidence.\\

On the other hand, the proof will depend on the order in which the singular decorations are placed. Nevertheless, as shown in the Figure \ref{fig:PermutingDecorations}, up to rows and regular columns commutations, cyclic permutations and (de)stabilizations, it is possible to replace a semi-singular commutation move by another one with the four singular decorations cyclically permuted. Hence, it is sufficient to deal with the following case:
$$
\begin{array}{ccc}
  \dessin{2.2cm}{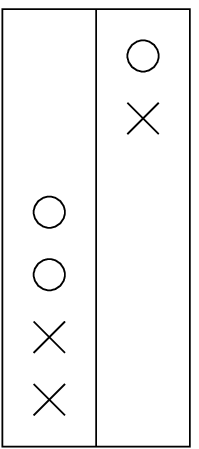} & \ \longleftrightarrow \  & \dessin{2.2cm}{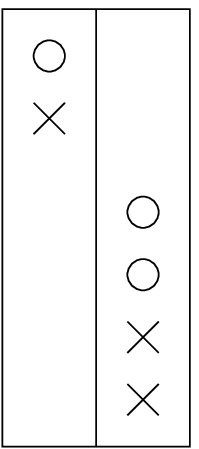}\\[1.3cm]
G_1 && G_2
\end{array}
$$

\begin{figure}
  $$
  \hspace{-.2cm} \xymatrix@!0@C=1.3cm@R=1.15cm{
    &&\dessin{2.5cm}{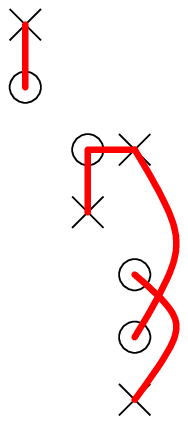}\ar@{<->}[rr]&&\dessin{2.5cm}{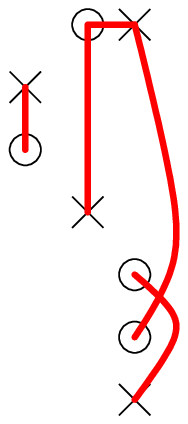}\ar@{<->}[rr]&&\dessin{2.5cm}{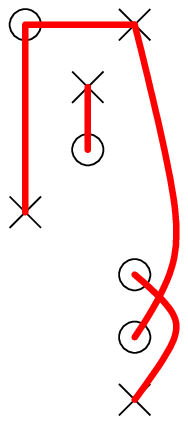}\ar@{<->}[rrd]|{\ r\ }&&\\
    \dessin{2.5cm}{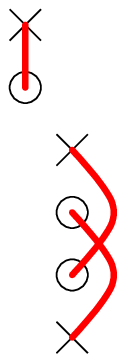}\ar@{<->}[rru]\ar@{<=>}[dddr]|{}="Nya1" &&&&&&&& \dessin{2.5cm}{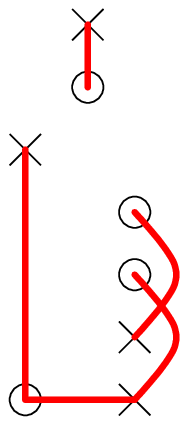}\ar@{<=>}[dddl]|{}="Nya2" \\
    &&&&&&&& \\
    &&&&&&&& \\
    &\dessin{2.5cm}{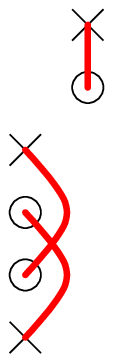}\ar@{<->}[rr]&&\dessin{2.5cm}{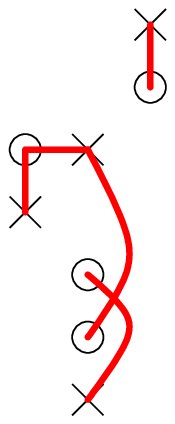}\ar@{<->}[rr]&&\dessin{2.5cm}{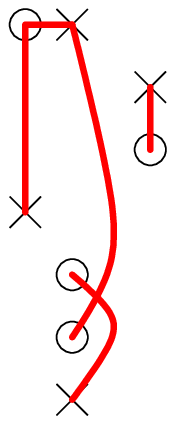}\ar@{<->}[rr]|{\ r\ }&&\dessin{2.5cm}{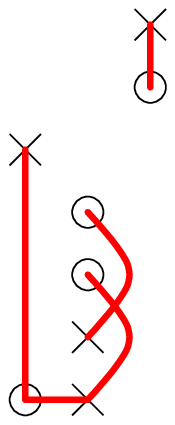}&
    \ar@{<.}@/^.7cm/"Nya1"!<.23cm,.07cm>;"Nya2"!<-.23cm,.07cm>|{\ \Leftarrow \ }
}
  $$
  \caption{Cyclic permutations of the four singular decorations: {\footnotesize Simple arrows stand for a combination of regular commutations and cyclic permutations of the rows when they are $r$--labeled and of regular commutations and (de)stabilizations when they are unlabeled. Double arrows stand for non regular commutations.}}
  \label{fig:PermutingDecorations}
\end{figure}

\subsubsection{Multi-combined grid}
\label{par:MultiCombGrid}

Here again, the commutation move can be seen as the replacement of distinguished vertical grid lines, but in a more sophisticated way. As illustrated in Figure \ref{fig:CommRegSing}, we choose a set of arcs such that $G_1$ (resp. $G_2$) is obtained by considering the arcs indexed by $1$ (resp. $2$) only.\\
\begin{figure}[!h]
  $$
  \dessin{5.5cm}{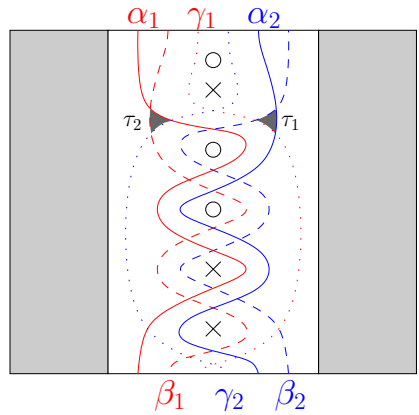}
  $$
    \caption{Commutation of a regular and a singular columns: {\footnotesize the arcs are choosen in such a way that the two dark shaded triangles $\tau_1$ and $\tau_2$ as well as the losange located between them are lying between the same two consecutive horizontal grid lines.}}
  \label{fig:CommRegSing}
\end{figure}
We denote respectively by $\p^-_1$ and $\p^-_2$ the differentials of the corresponding chain complexes. The map $f_{\alpha_1}^{\beta_1}$ (resp. $f_{\alpha_2}^{\beta_2}$) is the part of $\p^-_1$ (resp. $\p^-_2$) which corresponds to pentagons and hexagons with at least one peak on $\alpha_1\cap \beta_1$ (resp. $\alpha_2\cap \beta_2$).\\

Now, we also consider the two intermediate states defined by the sets of arcs $(\alpha_1,\alpha_2)$ and $(\beta_1,\beta_2)$. We denote by, respectively, $\p^-_\alpha$ and $\p^-_\beta$ the associated differentials. Thess intermediate states can be seen as the result of commutations of regular columns. With this point of view, we consider the maps $\phi_{\gamma_1\alpha_2}$, $\phi_{\alpha_1\gamma_2}$, $\phi_{\gamma_1\beta_2}$ and $\phi_{\beta_1\gamma_2}$ defined in paragraph \ref{par:RegCommutation}. They anti-commmute with the differentials but \textbf{not} with $f_{\alpha_1}^{\beta_1}$ and  $f_{\alpha_2}^{\beta_2}$. Figure \ref{fig:SemiSingTable} summarizes all the chain complexes and the chain maps between them.\\

\begin{figure}[h]
  \centering
  \hspace{-.2cm}
  \xymatrix@!0@C=3.5cm@R=3.5cm{\dessin{2cm}{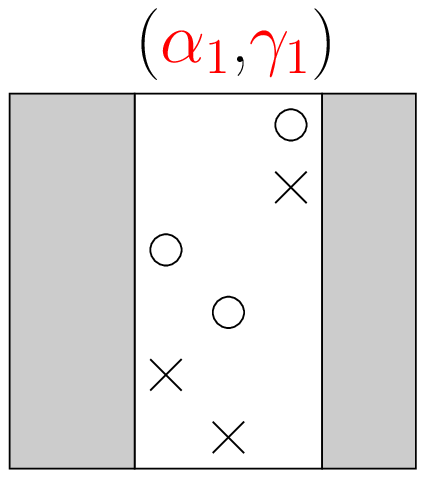} \ar[r]^{\phi_{\textcolor{red}{\gamma_1}\textcolor{blue}{\alpha_2}}} \ar[d]^{f_{\textcolor{red}{\alpha_1}}^{\textcolor{red}{\beta_1}}}  \ar@{}|{}="Nya"  \ar@(dl,ul)"Nya"!<-1.1cm,-.4cm>;"Nya"!<-1.1cm,.2cm>^{\p_1^0:=\p^-_1-f_{\textcolor{red}{\alpha_1}}^{\textcolor{red}{\beta_1}}} \ar@<.18cm>@{-->}[rrd]^{\phi_R} \ar@{-->}[rrd]|{\ \ \phi_{ex}\ } \ar@<-.18cm>@{-->}[rrd]_{\phi_L}& \dessin{2cm}{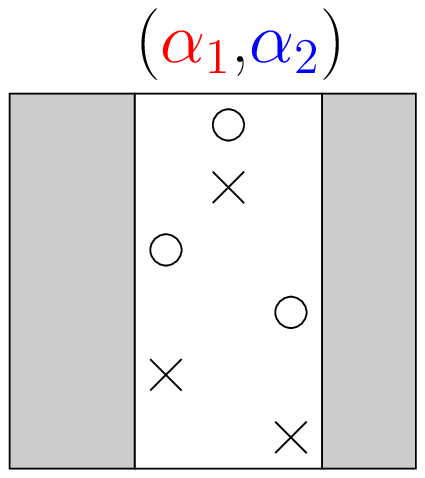} \ar[r]^{\phi_{\textcolor{red}{\alpha_1}\textcolor{blue}{\gamma_2}}} \ar@{}|{}="Nyabe"  \ar@(ul,ur)"Nyabe"!<-.3cm,1.1cm>;"Nyabe"!<.3cm,1.1cm>^(.15){\p^-_\alpha}  &\dessin{2cm}{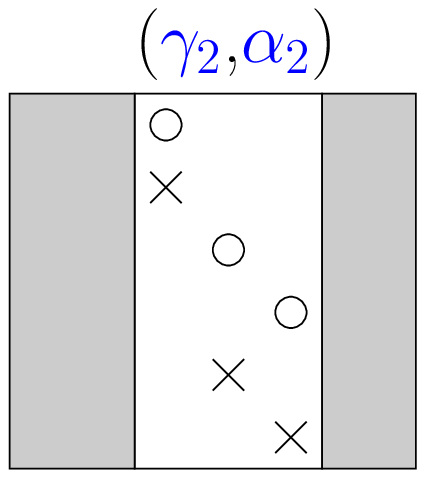} \ar[d]^{f_{\textcolor{blue}{\alpha_2}}^{\textcolor{blue}{\beta_2}}} \ar@{}|{}="Nya3"  \ar@(dr,ur)"Nya3"!<1.05cm,-.4cm>;"Nya3"!<1.05cm,.2cm>_{\p^-_2-f_{\textcolor{blue}{\alpha_2}}^{\textcolor{blue}{\beta_2}}=:\p_2^0}\\
    \dessin{2cm}{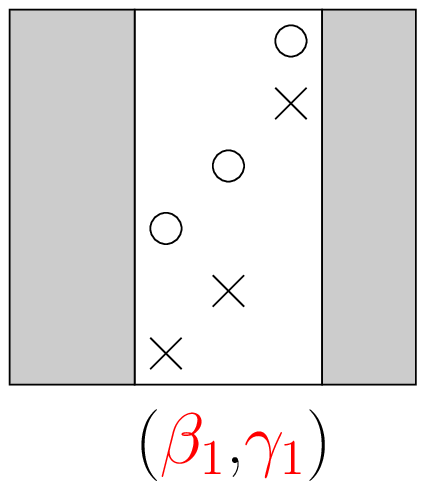}
    \ar[r]^{\phi_{\textcolor{red}{\gamma_1}\textcolor{blue}{\beta_2}}}\ar@{}|{}="Nya2"
    \ar@(dl,ul)"Nya2"!<-1.1cm,-.1cm>;"Nya2"!<-1.1cm,.5cm>^{\p^1_1:=\p^-_1}&
    \dessin{2cm}{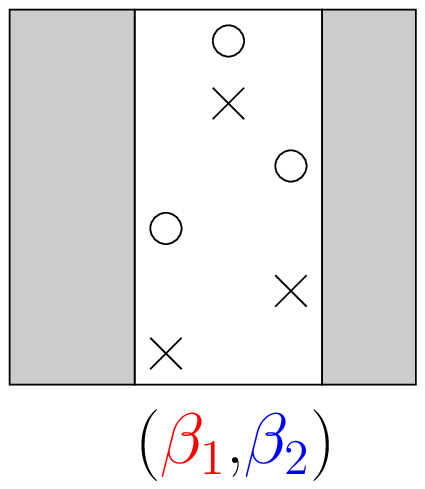}
    \ar[r]^{\phi_{\textcolor{red}{\beta_1}\textcolor{blue}{\gamma_2}}}
    \ar@{}|{}="Nyabee"
    \ar@(dl,dr)"Nyabee"!<-.3cm,-1.2cm>;"Nyabee"!<.3cm,-1.2cm>_(.15){\p^-_\beta}
    &\dessin{2cm}{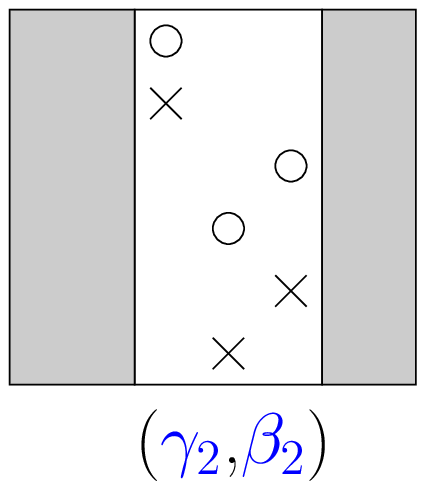}\ar@{}|{}="Nya4"
    \ar@(dr,ur)"Nya4"!<1.05cm,-.1cm>;"Nya4"!<1.05cm,.5cm>_{\p^-_2=:\p^1_2.}}
  \caption{Decomposition of a semi-singular commutation into regular ones}
\label{fig:SemiSingTable}
\end{figure}

As in paragraph \ref{par:Completion}, we will adjust the commutativity defect of this diagram by adding large diagonal maps $\phi_R$, $\phi_L$ and $\phi_{ex}$.

\subsubsection{Pentagons and hexagons again}
\label{par:AgainPentHex}

To define these maps, we consider combinations of polygons. We denote by $G_{Comb}$ the multi-combined grid with all arcs $\alpha$'s, $\beta$'s and $\gamma$'s and by $\T_{Comb}:=\T_{G_{Comb}}$ the torus obtained by identifying its boundary components.\\

We denote by $E$ the set of $n$ dots arranged on the intersections of lines and arcs of $G_{Comb}$.\\

Let $x$ and $y$ be two elements of $E$. A \emph{pentagon} (resp. \emph{hexagon}) \emph{of type $R$ connecting $x$ to $y$}\index{polygon!pentagon!of type R@of type $R$} is a pentagon (resp. hexagon)\index{polygon!of type R@of type $R$} connecting $x$ to $y$ as defined in paragraph \ref{par:Pentagons} (resp. \ref{par:Hexagons}) with  $\alpha_1\cap\gamma_2\cap\tau_2$ as a peak.\\
{\it Mutatis mutandis}, we define \emph{pentagons} and \emph{hexagons of type $L$}\index{polygon!of type L@of type $L$} by substituing $\gamma_1\cap\beta_2\cap\tau_1$ to $\alpha_1\cap\gamma_2\cap\tau_2$

We denote by $\Pol^\circ_R(x,y)$ and $\Pol^\circ_L(x,y)$ the sets of empty pentagons and hexagons of type $R$ or $L$ connecting $x$ to $y$.

\begin{figure}[h]
  $$
  \begin{array}{ccc}
    \dessinH{4.2cm}{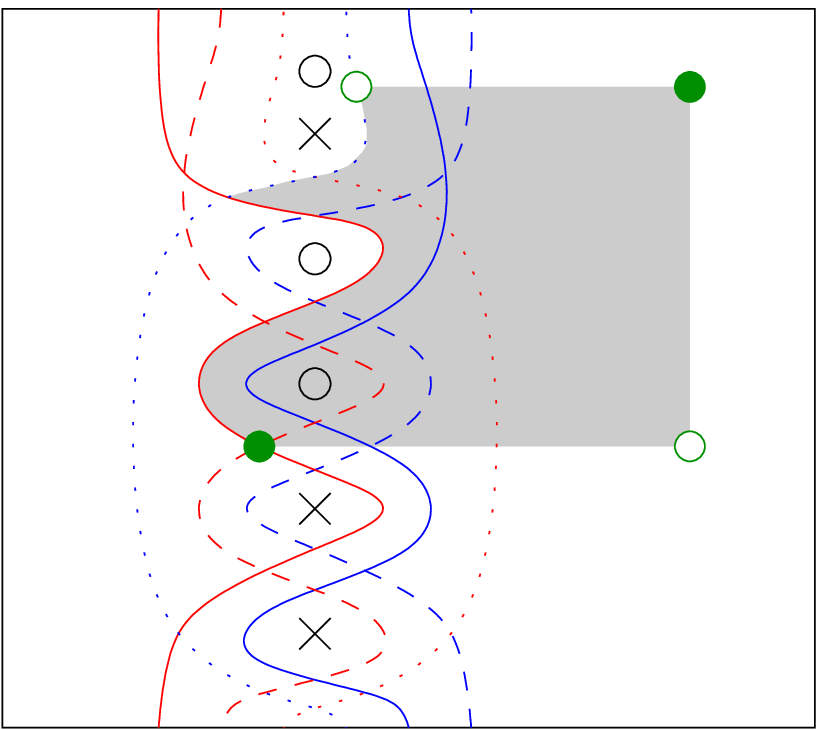} & \hspace{1.5cm}& \dessinH{4.2cm}{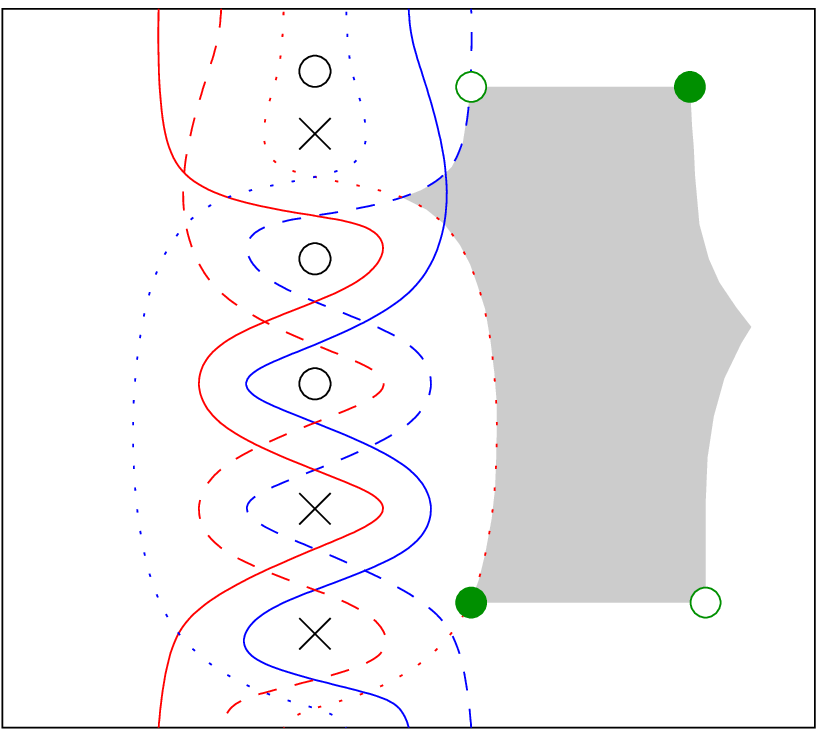}\\[1cm]
    \textrm{an empty pentagon} && \textrm{an empty hexagon}\\
    \textrm{of type } R && \textrm{of type } L
  \end{array}
  $$
  \caption{Examples of polygons: {\footnotesize dark dots describe the generator $x$ while hollow ones describe $y$. Polygons are depicted by shading.}}
  \label{fig:ExamplesPolygons}
\end{figure}

\subsubsection{New kinds of polygons}
\label{par:NewPolygons}

Now, we need to define some new hexagons which are more in the spirit of hexagons described in \cite{MOST}. To distinguish them, we will write it $\h$exagons.\\

Let $x$ and $y$ be elements of $E$. An \emph{$\h$exagon of type $R$}\index{polygon!hexagoo@$\h$exagon} (resp. \emph{of type $L$}) \emph{connecting $x$ to $y$} is an embedded hexagon $\eta$ in $\T_{Comb}$ which satisfies:
\begin{itemize}
\item[-] edges of $\eta$ are embedded in the grid lines (including $\alpha_l$'s, $\beta_l$'s and $\gamma_l$'s);
\item[-] the intersection $\alpha_2\cap\tau_1$ (resp. $\beta_1\cap\tau_2$) is a whole edge $e$ of $\eta$;
\item[-] running positively along the boundary of $\eta$, according to the orientation of $\eta$ inherited from the one of $\T_{Comb}$, the next four corners of $\eta$ after the edge $e$ are, alternatively, points of $x$ and $y$;
\item[-] except on $\p \eta$, the sets $x$ and $y$ coincide;
\item[-] the interior of $\eta$ does not intersect $\alpha_2$ (resp. $\beta_1$) in a neighborhood of $e$.
\end{itemize}

\begin{remarque}
  Since an $\h$exagon is embbeded and has more than three vertices, the triangles $\tau_1$ or $\tau_2$ cannot lie inside. Moreover, it must contain a unique vertical edge connecting an element of $x$ to an element of $y$.
\end{remarque}

By adding a third peak, belonging to another remote singular column, on the free vertical edge of an $\h$exagon $\eta$, we define \emph{$\h$eptagons}\index{polygon!heptagoo@$\h$eptagon} of the same type and connecting the same elements of $E$ than $\eta$.\\

An $\h$exagon or an $\h$eptagon $\eta$ connecting $x$ to $y$ is \emph{empty} if $\Int(\eta)\cap x=\emptyset$.\\
We denote respectively by $\PPol^\circ_R(x,y)$ and $\PPol^\circ_L(x,y)$ the sets of empty $\h$exagons and $\h$eptagons of type $R$ and $L$ connecting $x$ to $y$.

\begin{figure}[!h]
  $$
  \begin{array}{ccc}
    \dessinH{4.2cm}{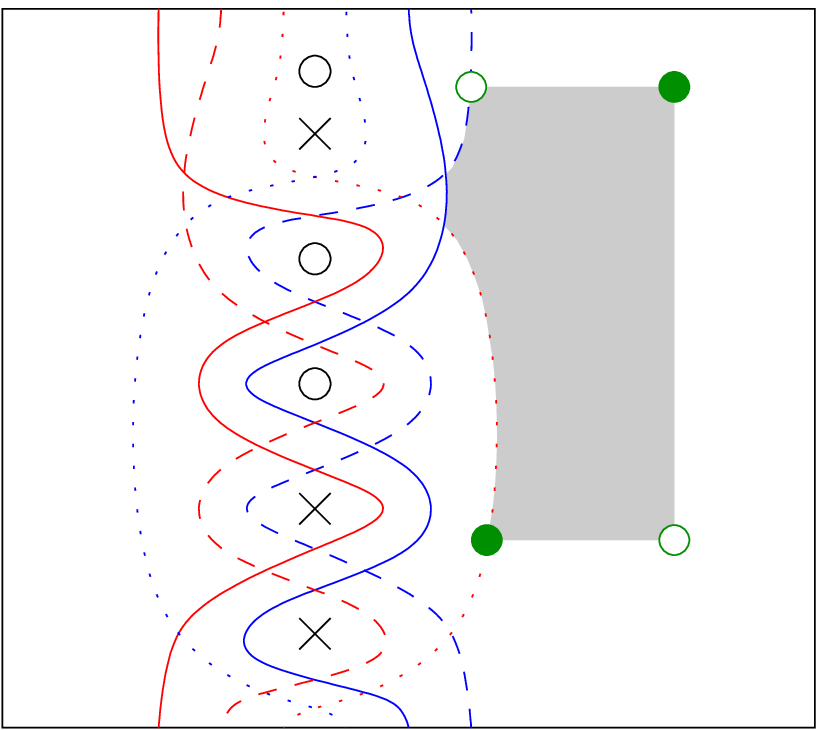} & \hspace{1.5cm}& \dessinH{4.2cm}{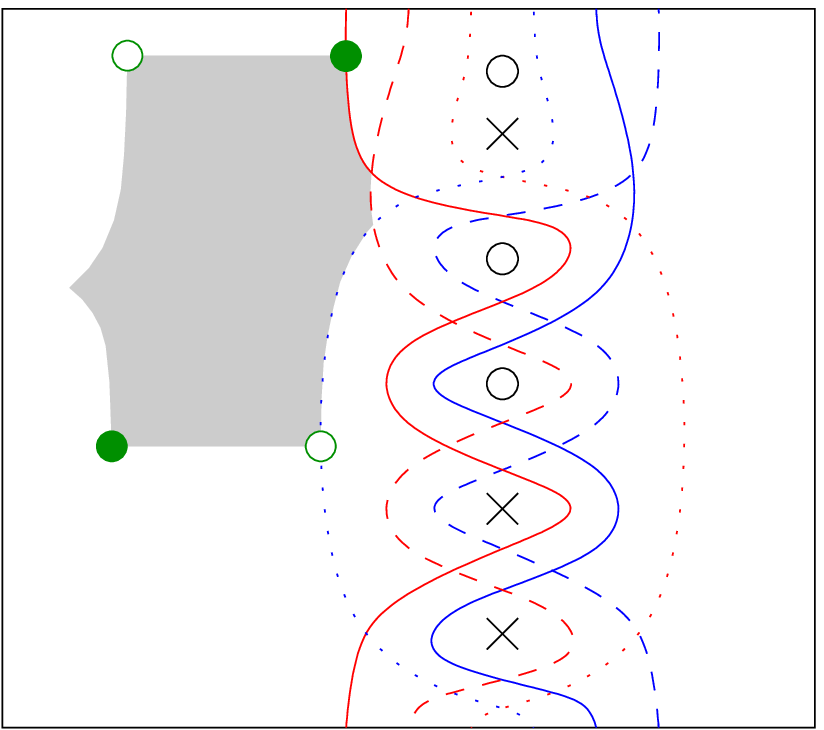}\\[1cm]
    \textrm{an empty }\h\textrm{exagon} && \textrm{an empty }\h\textrm{eptagon}\\
    \textrm{of type } R && \textrm{of type } L
  \end{array}
  $$
  \caption{Examples of $\mathfrak{p}$olygons: {\footnotesize dark dots describe the generator $x$ while hollow ones describe $y$. Polygons are depicted by shading.}}
  \label{fig:ExamplesPolygons2}
\end{figure}

Finally, we define $\HH$\index{polygon!hexagon!H@$\HH$}\index{H@$\HH$|see{polygon}} as the following exceptional thin hexagon:
\begin{eqnarray}\label{ExceptHex}
  \dessinH{4.2cm}{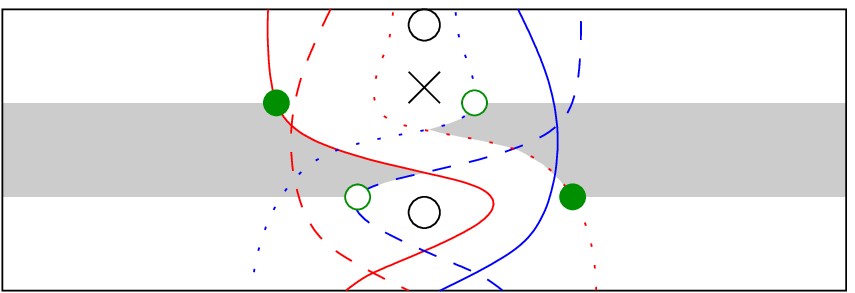}.
\end{eqnarray}
For any generators $x\in C^-(G_{\gamma_1,\alpha_1})$ and $y\in C^-(G_{\gamma_2,\beta_2})$, we say that \emph{$\HH$ connects $x$ to $y$} if
\begin{itemize}
\item[i)] $x$ and $y$ coincide except on $\p \HH$;
\item[ii)] $x\cap\HH$ and $y\cap\HH$ are respectively the dark and the hollow dots in the picture (\ref{ExceptHex}).
\end{itemize}
We also say that \emph{$x$ is a starting point for $\HH$ to $y$}. The generator $y$ is then determined by $x$.

\subsubsection{Signs for polygons}
\label{par:SignsPolygons}

Under the condition that $y$ is a generator of $C^-(G_{\gamma_2\beta_2})$, any element $\pi$ of $\Pol^\circ_R(x,y)$ or $\Pol^\circ_L(x,y)$ can be filled with one spike in case of a pentagon or two if it is an hexagon, in order to get a rectangle $\phi(\pi)$ in $G_{\gamma_2\beta_2}$ (compare Prop. \ref{Pent->Rect}, \S\ref{par:SpiPent} or Prop. \ref{Hex->Rect}, \S\ref{par:SpiHex}).\\

Likewise, if $x$ is a generator of $C^-(G_{\alpha_1\gamma_1})$, any element $\pi$ of $\PPol^\circ_R(x,y)$ or $\PPol^\circ_L(x,y)$ can be filled to a rectangle $\phi(\pi)$ in $G_{\alpha_1\alpha_2}$.\\

In all cases, we can associate a sign $\e(\pi)=\e(\phi(\pi))$ if $\pi$ has an even number of peaks pointing toward the left and $-\e(\phi(\pi))$ otherwise.\\

Since it has a unique peak pointing to the left, we set $\e(\HH)=-1$.

\subsubsection{Combinations of polygons}
\label{par:MendingMap}

Now, we can set the maps $\func{\phi_R,\phi_L,\phi_{ex}}{C^-(G_{\gamma_1,\alpha_1})}{C^-(G_{\gamma_2,\beta_2})}$ which are the morphisms of $\Z[U_{O_1},\cdots,U_{O_n}]$--modules defined on the generators by

$$
\begin{array}{p{.8cm}p{.1cm}p{13cm}}
  $\phi_R(x)$ & $=$ &
  $\sum_{\substack{y \textrm{ generator}\\[.05cm]
      \textrm{of }C^-(G_{\gamma_2,\beta_2})\\[.05cm]
      M(y)=M(x)\\[.05cm]
      A(y)\leq A(x)}}
  \sum_{\textcolor{white}{I}z \in E\textcolor{white}{I}}
  \sum_{\substack{p_1\in \Pol^\circ_R(x,z)\\[.05cm]
      p_2\in \PPol^\circ_R(z,y)}}
  \e(p_1)\e(p_2)U_{O_1}^{O_1(p_1)+O_1(p_2)} \cdots U_{O_n}^{O_n(p_1)+O_n(p_2)}\cdot y$,
\end{array}
$$
\vspace{.3cm}
$$
\begin{array}{p{.8cm}p{.1cm}p{13cm}}
  $\phi_L(x)$ & $=$ &
  $\sum_{\substack{y \textrm{ generator}\\[.05cm]
      \textrm{of }C^-(G_{\gamma_2,\beta_2})\\[.05cm]
      M(y)=M(x)\\[.05cm]
      A(y)\leq A(x)}}
  \sum_{\textcolor{white}{I}z \in E\textcolor{white}{I}}
  \sum_{\substack{p_1\in \Pol^\circ_L(x,z)\\[.05cm]
      p_2\in \PPol^\circ_L(z,y)}}
  \e(p_1)\e(p_2)U_{O_1}^{O_1(p_1)+O_1(p_2)} \cdots U_{O_n}^{O_n(p_1)+O_n(p_2)}\cdot y$,
\end{array}
$$
\vspace{.3cm}
$$
\begin{array}{p{.8cm}p{.1cm}p{13cm}}
  $\phi_{ex}(x)$ & $=$ &
  $\left\{\begin{array}{cl} \e(\HH)U_{O_1}^{O_1(\HH)} \cdots U_{O_n}^{O_n(\HH)}y & \textrm{if }x\textrm{ is a starting point to }y \\ 0 & \textrm{otherwise} \end{array}\right.$.
\end{array}
$$
Moreover, we denote by, respectively, $F_T$ and $F_B$ the compositions $\phi_{\alpha_1\gamma_2}\circ \phi_{\gamma_1\alpha_2}$ and $\phi_{\beta_1\gamma_2}\circ \phi_{\gamma_1\beta_2}$.

\begin{figure}[p]
\begin{center}
  \begin{tabular}{ccc}
    \begin{tabular}{c}
      \subfigure[]{$\dessinH{4.2cm}{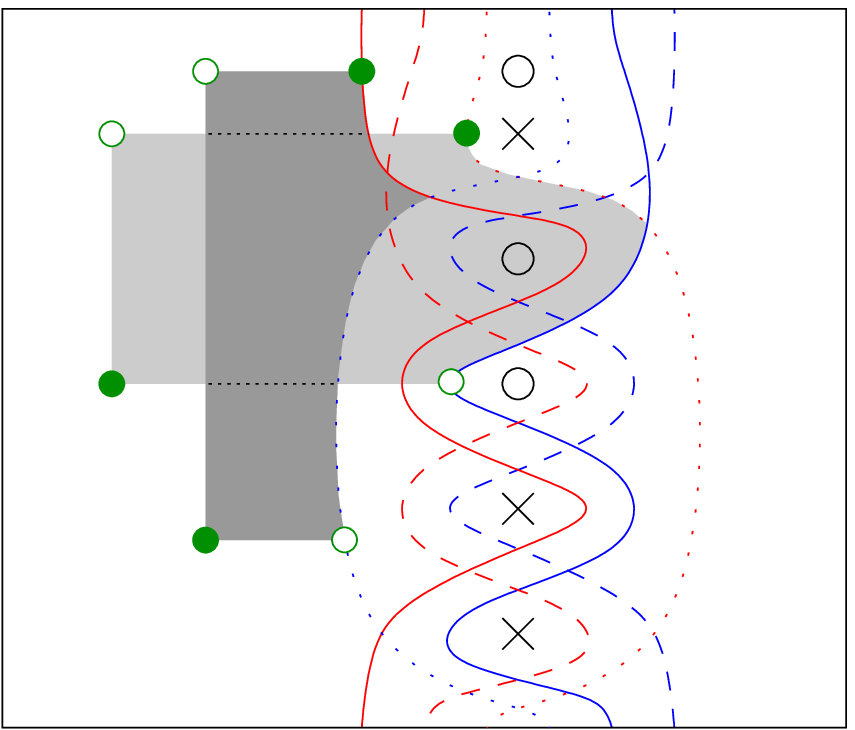}$}
    \end{tabular}
    &
    \begin{tabular}{c}
      \subfigure[]{$\dessinH{4.2cm}{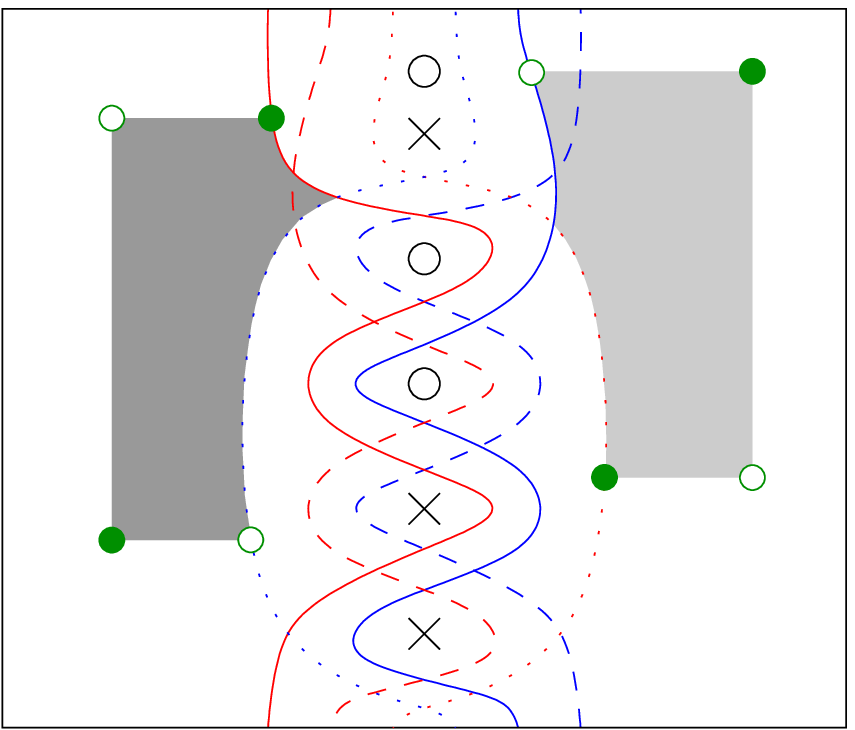}$}
    \end{tabular}
    &
    \begin{tabular}{c}
      \subfigure[]{$\dessinH{4.2cm}{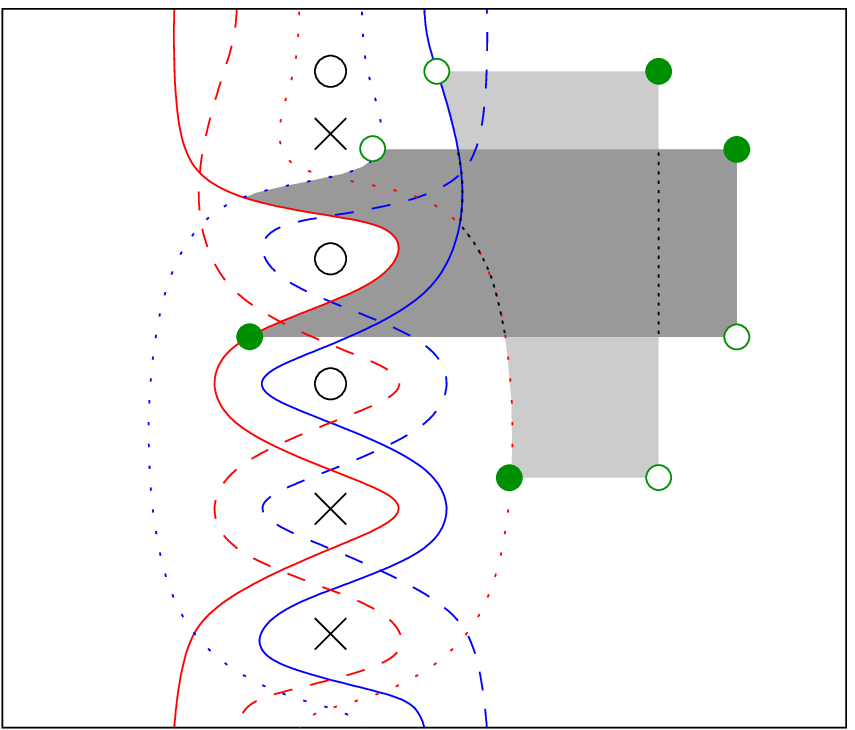}$}
    \end{tabular}
  \end{tabular}
  \begin{tabular}{ccc}
    \begin{tabular}{c}
      \subfigure[]{$\dessinH{4.2cm}{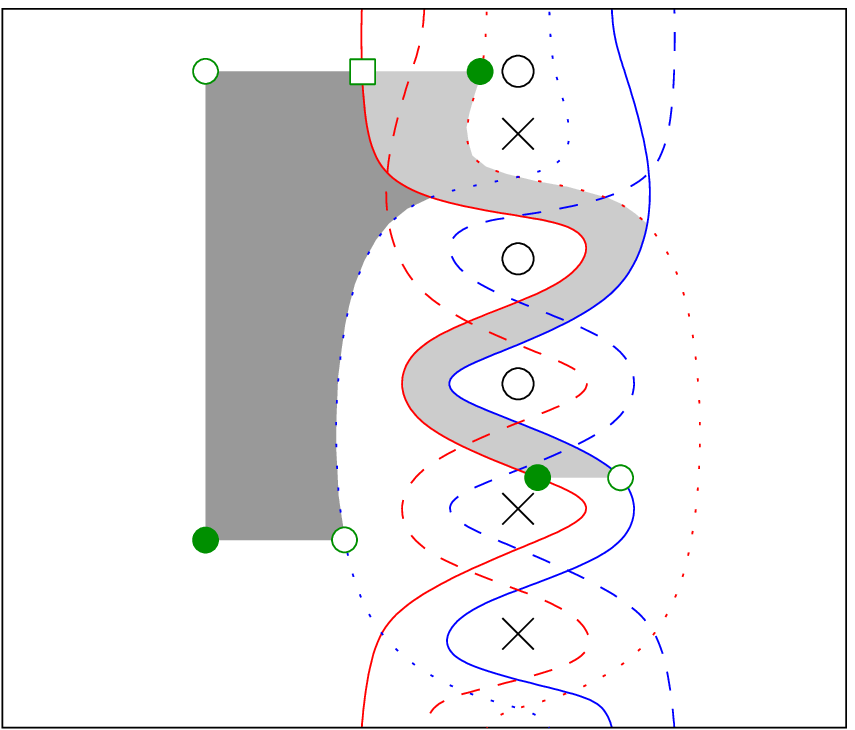}$}
    \end{tabular}
    &
    \begin{tabular}{c}
      \subfigure[]{$\dessinH{4.2cm}{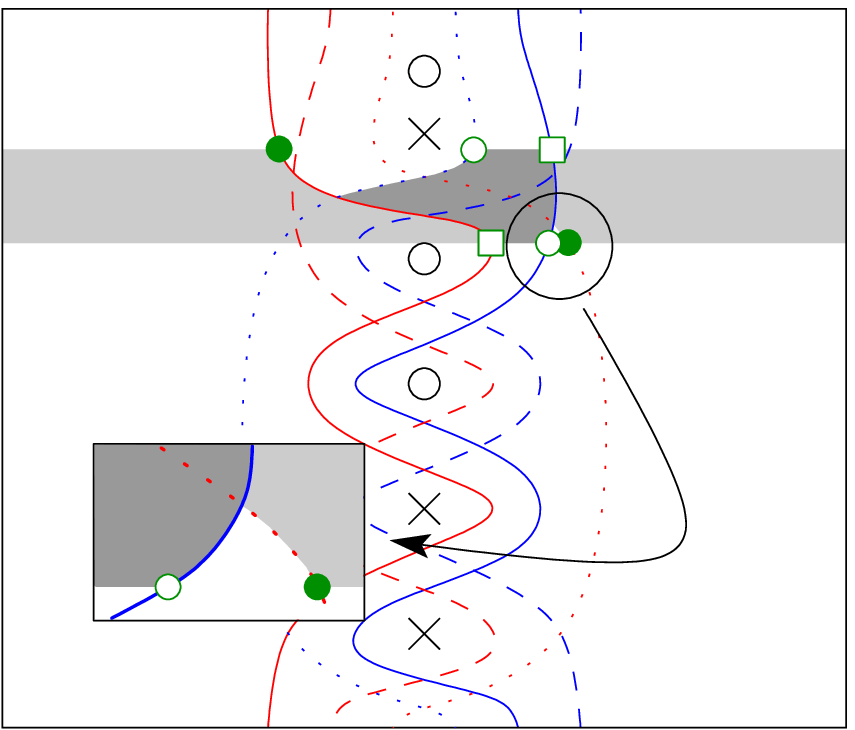}$}
    \end{tabular}
    &
    \begin{tabular}{c}
      \subfigure[]{$\dessinH{4.2cm}{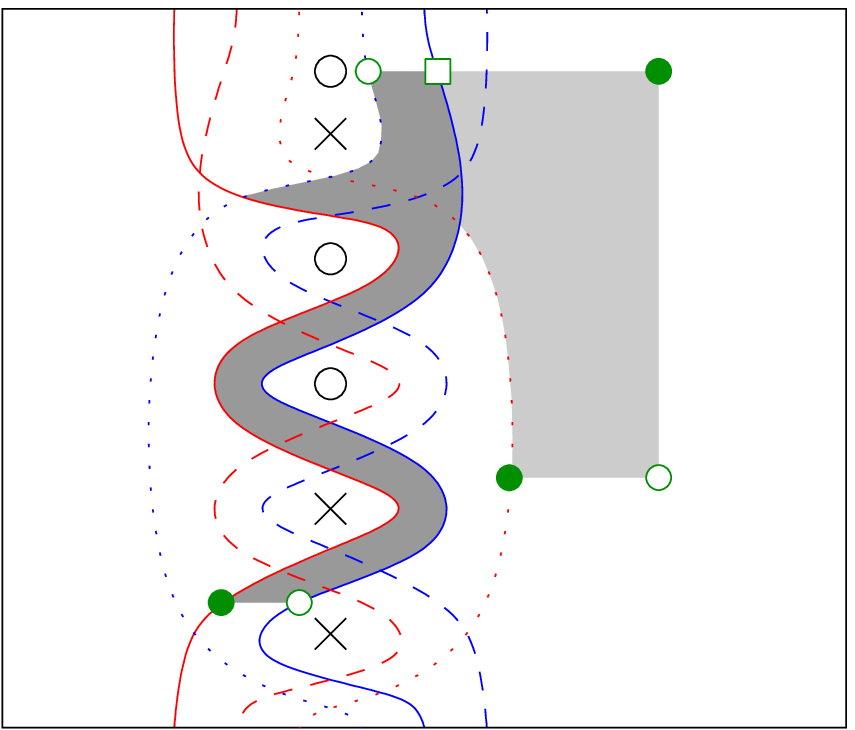}$}
    \end{tabular}
  \end{tabular}
  \begin{tabular}{ccc}
    \begin{tabular}{c}
      \subfigure[]{$\dessinH{4.2cm}{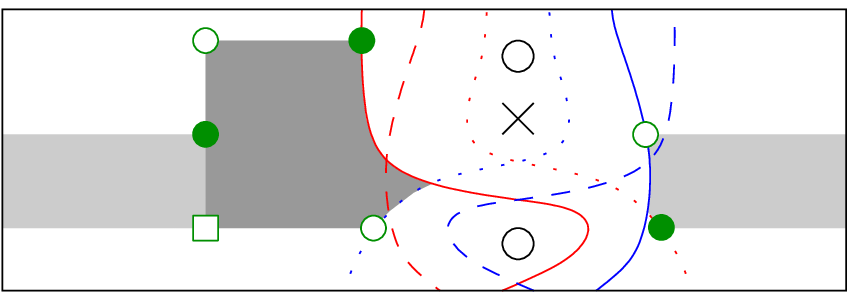}$}
    \end{tabular}
    &
    \begin{tabular}{c}
      \subfigure[]{$\dessinH{4.2cm}{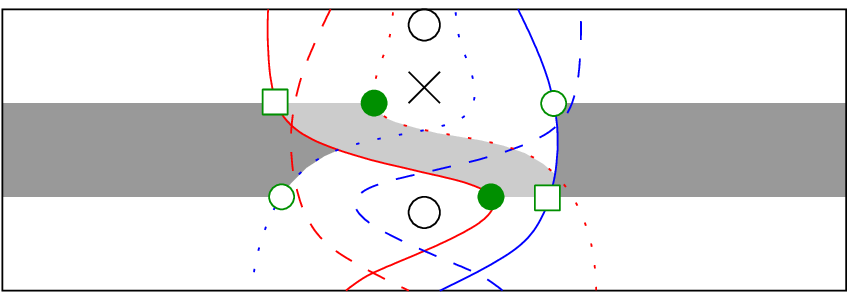}$}
    \end{tabular}
    &
    \begin{tabular}{c}
      \subfigure[]{$\dessinH{4.2cm}{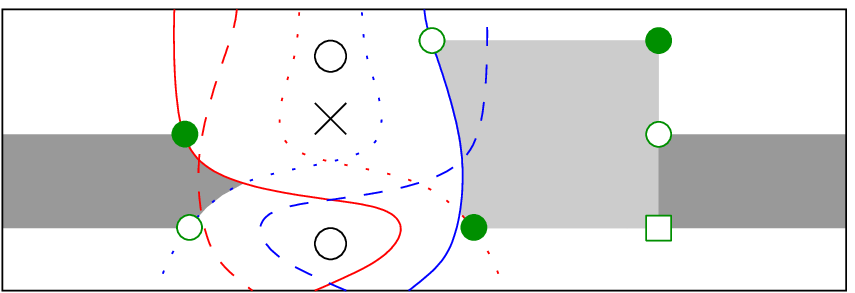}$}
    \end{tabular}
  \end{tabular}
\end{center}
\begin{center}
  A. Configurations occuring in $F_H$
\end{center}
\vspace{.5cm}
\addtocounter{subfigure}{-9}
\begin{center}
  \begin{tabular}{ccc}
    \begin{tabular}{c}
      \subfigure[]{$\dessinH{4.2cm}{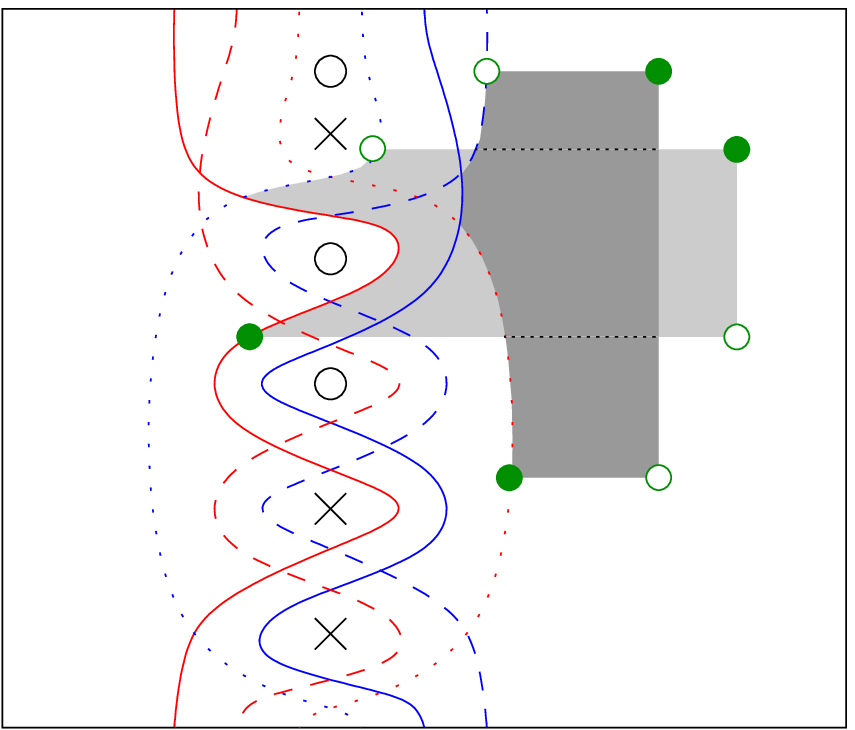}$}
    \end{tabular}
    &
    \begin{tabular}{c}
      \subfigure[]{$\dessinH{4.2cm}{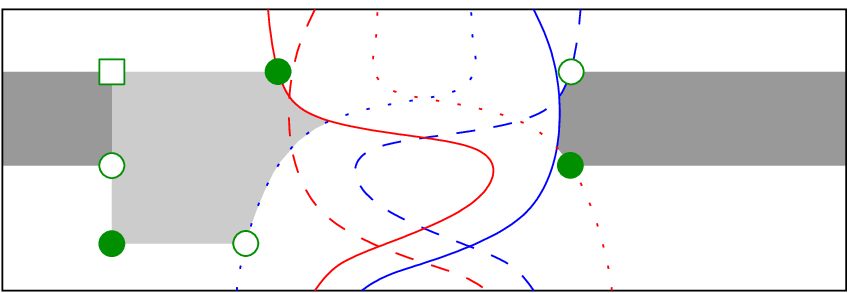}$}\\[1.1cm]
      \subfigure[]{$\dessinH{4.2cm}{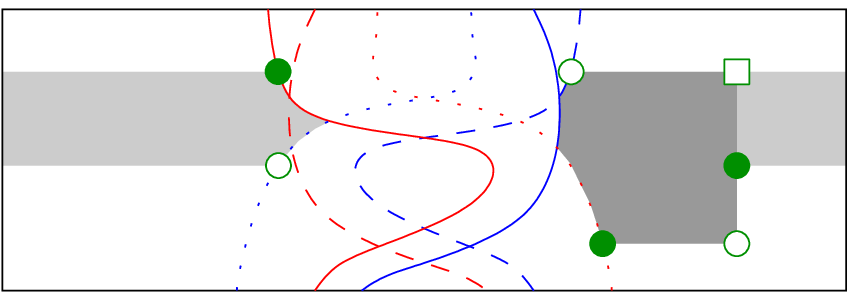}$}
    \end{tabular}
    &
    \begin{tabular}{c}
      \subfigure[]{$\dessinH{4.2cm}{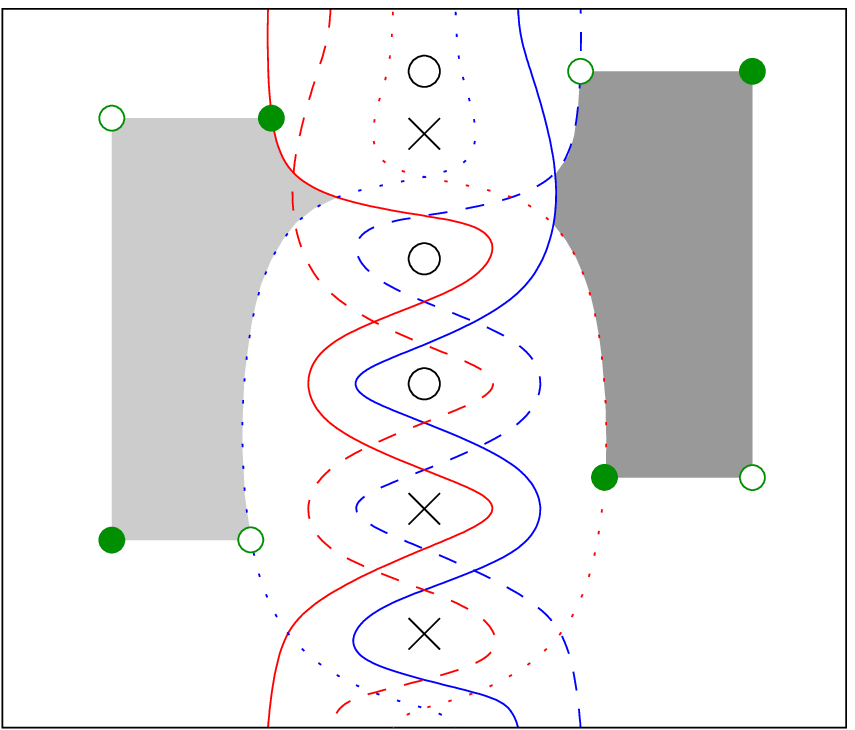}$}
    \end{tabular}
  \end{tabular}
\end{center}
\begin{center}
  B. Configurations occuring in $\phi_R$
\end{center}
\caption{Configurations of polygons: {\footnotesize Dark dots describe the initial generator while hollow ones describe the final one. Squares describe intermediate states. Polygons are depicted by shading. The lightest is the first to occur whereas the darkest one is the last. Polygons are stacked in an opaque way.}}
\label{fig:ConfigFH}
\end{figure}

\addtocounter{figure}{-1}
\begin{figure}[p]
  \begin{center}
  \begin{tabular}{ccc}
    \begin{tabular}{c}
      \subfigure[]{$\dessinH{4.2cm}{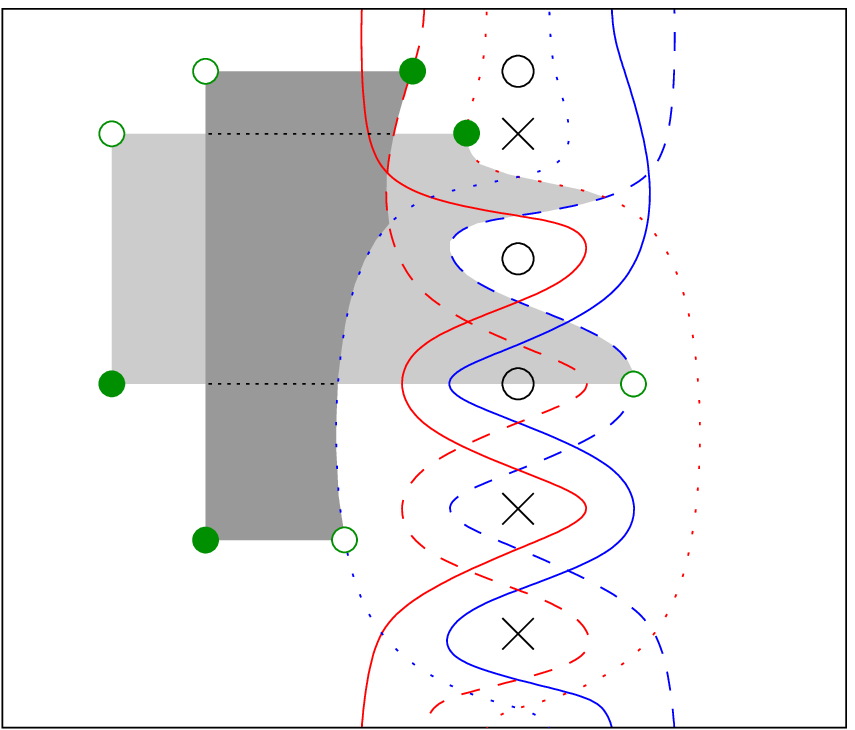}$}
    \end{tabular}
    &
    \begin{tabular}{c}
      \subfigure[]{$\dessinH{4.2cm}{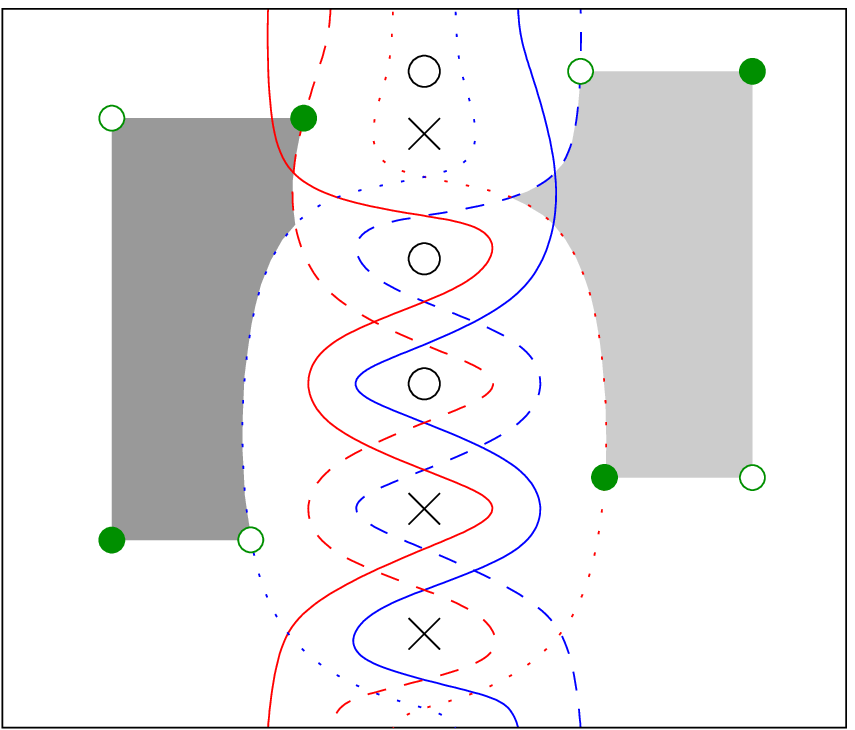}$}
    \end{tabular}
    &
    \begin{tabular}{c}
      \subfigure[]{$\dessinH{4.2cm}{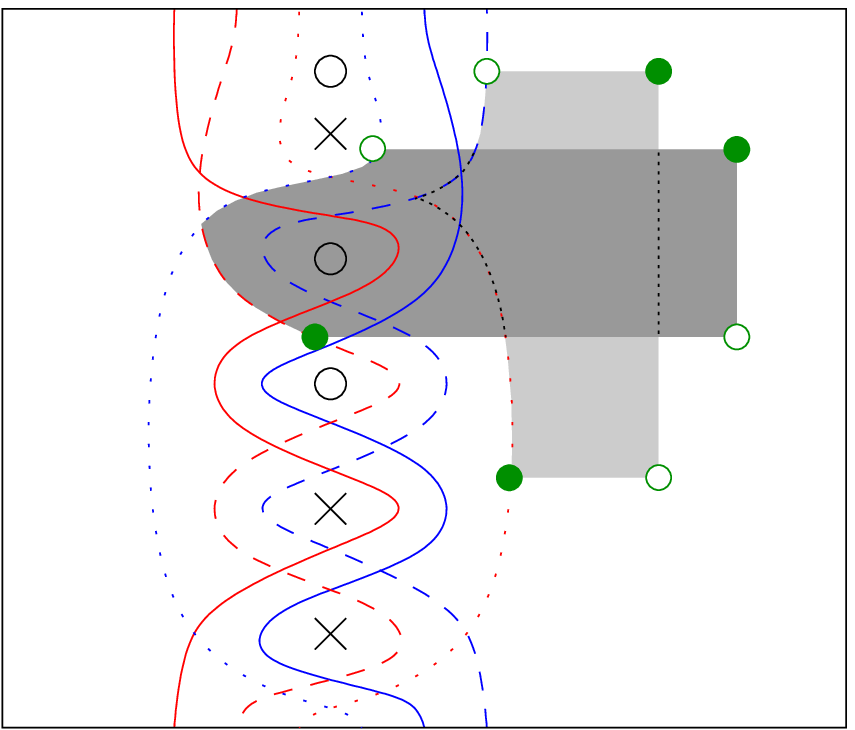}$}
    \end{tabular}
  \end{tabular}
  \begin{tabular}{ccc}
    \begin{tabular}{c}
      \subfigure[]{$\dessinH{4.2cm}{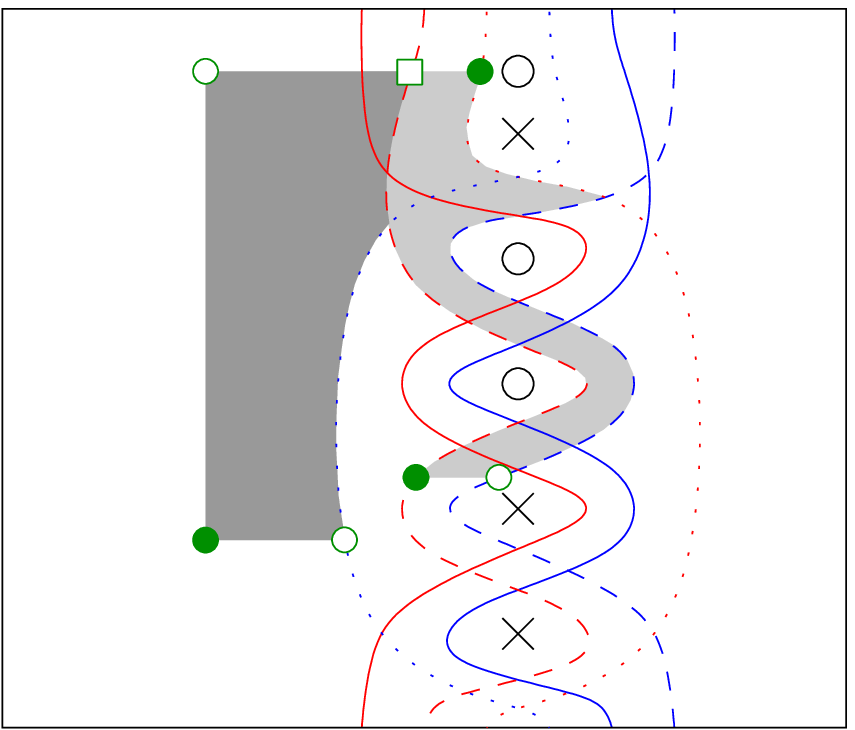}$}
    \end{tabular}
    &
    \begin{tabular}{c}
      \subfigure[]{$\dessinH{4.2cm}{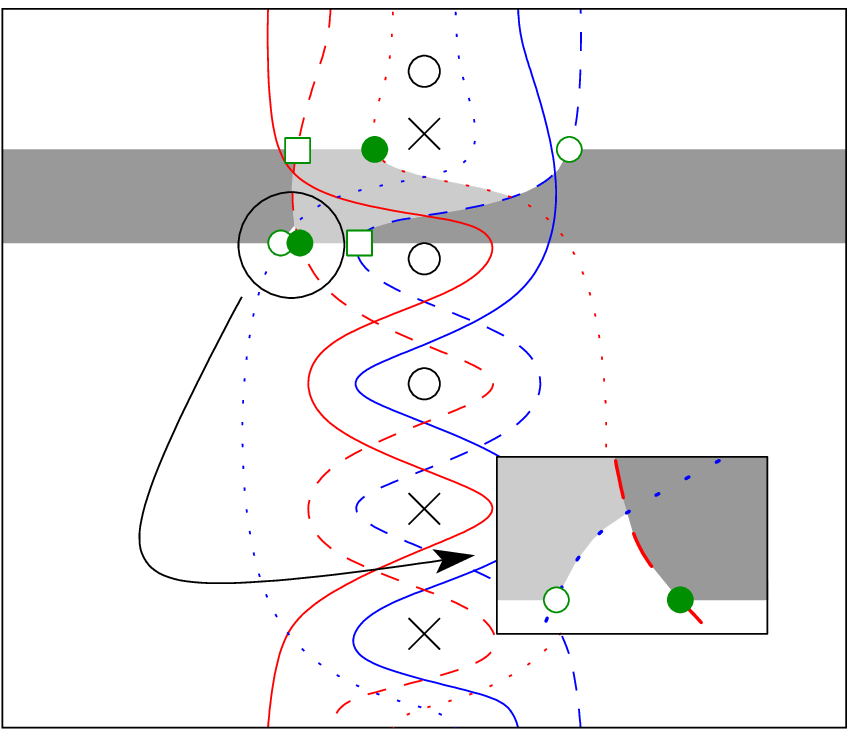}$}
    \end{tabular}
    &
    \begin{tabular}{c}
      \subfigure[]{$\dessinH{4.2cm}{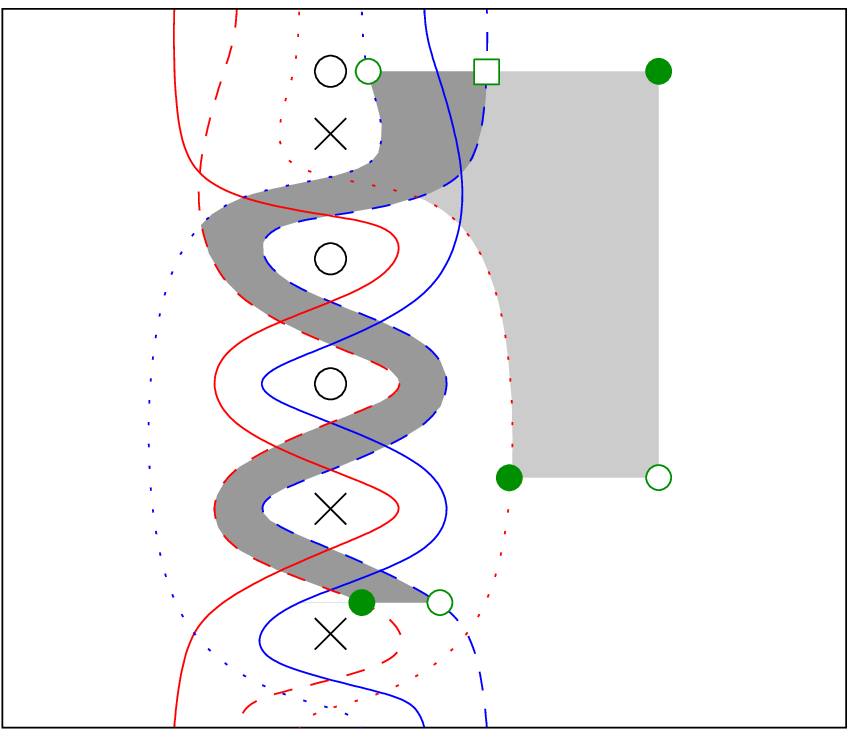}$}
    \end{tabular}
  \end{tabular}
  \begin{tabular}{ccc}
    \begin{tabular}{c}
      \subfigure[]{$\dessinH{4.2cm}{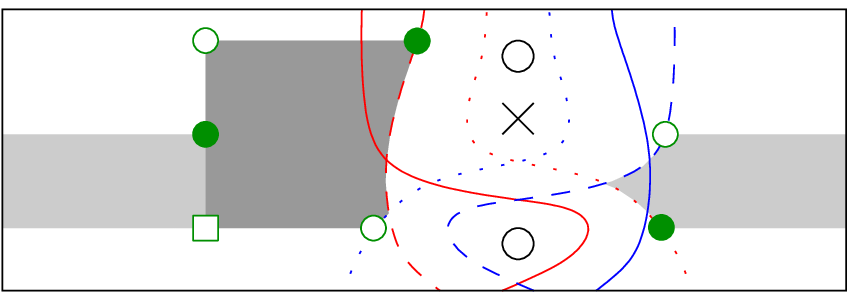}$}
    \end{tabular}
    &
    \begin{tabular}{c}
      \subfigure[]{$\dessinH{4.2cm}{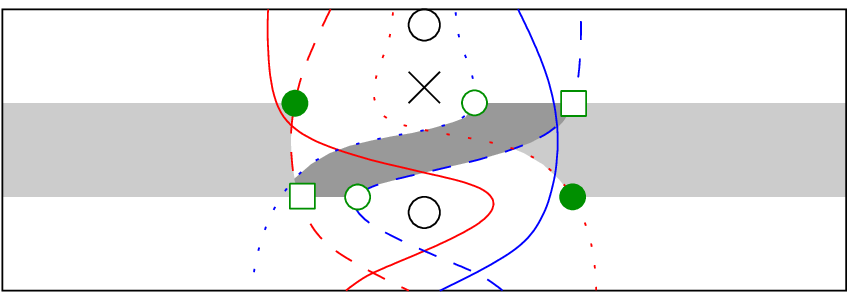}$}
    \end{tabular}
    &
    \begin{tabular}{c}
      \subfigure[]{$\dessinH{4.2cm}{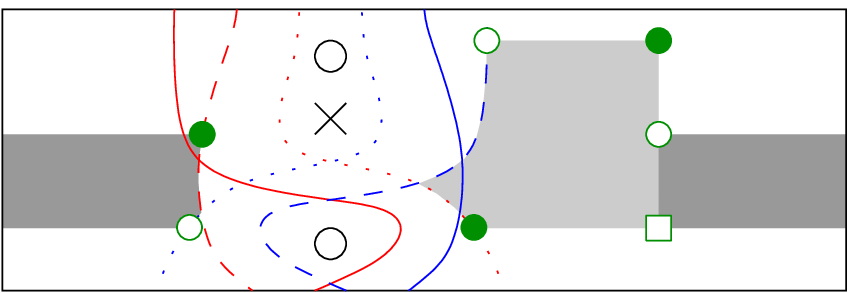}$}
    \end{tabular}
  \end{tabular}
\end{center}
\begin{center}
  C. Configurations occuring in $F_B$
\end{center}
\vspace{.5cm}
\addtocounter{subfigure}{-9}
\begin{center}
  \begin{tabular}{ccc}
    \begin{tabular}{c}
      \subfigure[]{$\dessinH{4.2cm}{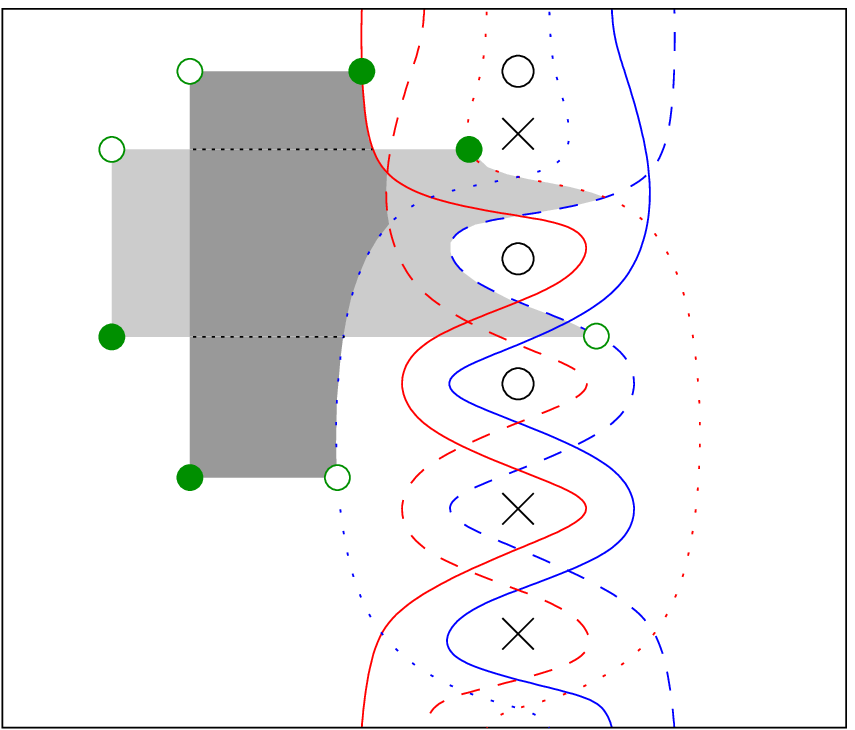}$}
    \end{tabular}
    &
    \begin{tabular}{c}
      \subfigure[]{$\dessinH{4.2cm}{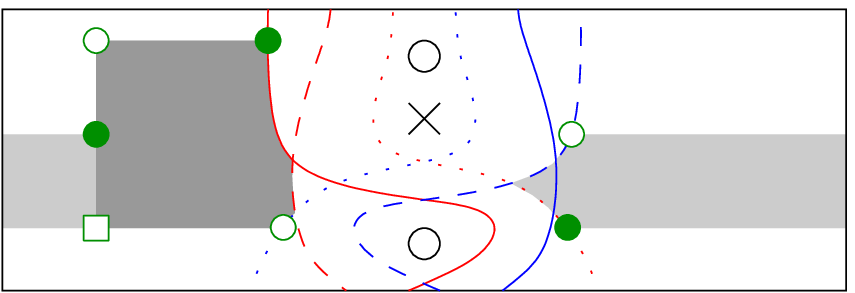}$}\\[1.1cm]
      \subfigure[]{$\dessinH{4.2cm}{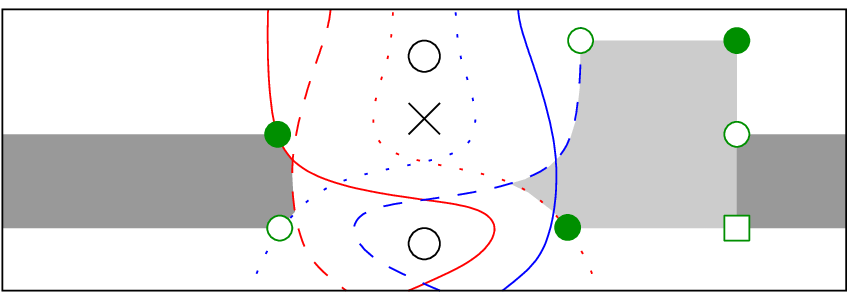}$}
    \end{tabular}
    &
    \begin{tabular}{c}
      \subfigure[]{$\dessinH{4.2cm}{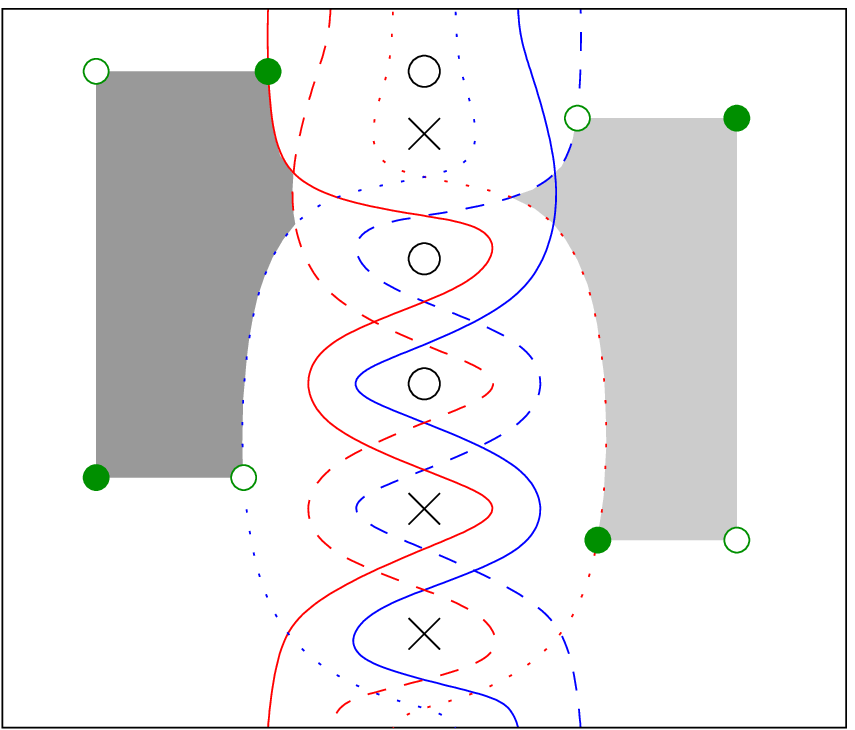}$}
    \end{tabular}
  \end{tabular}
\end{center}
\begin{center}
  D. Configurations occuring in $\phi_L$
\end{center}
  \caption{Configurations of polygons: {\footnotesize Dark dots describe the initial generator while hollow ones describe the final one. Squares describe intermediate states. Polygons are depicted by shading. The lightest is the first to occur whereas the darkest one is the last. Polygons are stacked in an opaque way.}}
  \label{fig:ConfigFB}
\end{figure}

\subsubsection{Semi-singular commutation chain map}
\label{par:SemiSingMap}

Having set down all the notation, we can now define $\func{\psi}{C^-_1}{C^-_2}$ as
$$
\psi(x)=\left\{
  \begin{array}{cl}
    F_T(x) - \phi_R(x) + \phi_L(x) - 2\phi_{ex}& \textrm{if }x\in C^-(G_{\gamma_1\alpha_1})\\[.2cm]
    F_B(x) & \textrm{if }x\in C^-(G_{\gamma_1\beta_1}).
  \end{array}\right.
$$

\begin{prop}
  The map $\psi$ is a filtrated chain map.
\end{prop}

Contrary to most maps in this thesis, the map $\psi$ is indeed commuting with differentials and not anti-commuting.

\begin{proof}
  The fact that $\psi$ respects the grading and the filtration holds by construction.\\
  
  The goal is now to prove
  $$
  \p^-_2\circ \psi-\psi\circ \p^-_1=0.
  $$
We already know that
$$
\begin{array}{ccccc}
  F_T\circ \p_1^0  & = & \phi_{\alpha_1\gamma_2}\circ\phi_{\gamma_1\alpha_2}\circ\p^0_1 &&\\[.1cm]
  & = & - \phi_{\alpha_1\gamma_2}\circ\p_\alpha^-\circ\phi_{\gamma_1\alpha_2} &&\\[.1cm]
  & = & \p_2^0\circ\phi_{\alpha_1\gamma_2}\circ\phi_{\gamma_1\alpha_2} &  = & \p_2^0\circ F_T.
\end{array}
$$
Similarly, $F_B\circ\p_1^1=\p_2^1\circ F_B.$\\

It remains to check terms which cross the diagram from $C^-(G_{\alpha_1\gamma_1})$ to $C^-(G_{\gamma_2\beta_2})$. First we set up a list of the possible configurations for $F_T$, $F_B$, $\phi_R$ and $\phi_L$. This is done in Figure \ref{fig:ConfigFB}.\\

Then, it remains to compose these maps with $f_{\alpha_1}^{\beta_1}$, $f_{\alpha_2}^{\beta_2}$, $\p^0_1$ or $\p_2^1$. Most terms cancel two by two for similar reasons than in the proof of Theorem \ref{D0Preserves} (\S\ref{par:Diff}). But besides this, there are 144 special cases which can be gathered in 72 cancelling {\it par nobile fratrum}.  An example for each cancelling process is given in Figure \ref{fig:AllCases}.\\

However, signs should be checked. There are several ways for a minus sign to occur :
\begin{description}
\item[Order] when one side is an element of $\p^-_2\circ \psi$ whereas the other one belongs to $\psi\circ \p^-_1$;
\item[Configuration] when the signs associated to the two configurations of polygons differ. To compute it, one can use the description of the differential given in paragraph \ref{par:AltDiff}. But in most cases, it is more convenient to add spikes in order to get rectangles and then translate them in a common grid. This can be done since the sign does not depend on the decorations. Then, we can consider paths of configuration  using the elementary moves occuring in the proof of Theorem \ref{D0Preserves} (\S\ref{par:Diff}), \ie
  \begin{itemize}
  \item[-] commuting the order of occurence of two consecutive polygons with disjoint sets of corners;
  \item[-] commuting two L--decompositions.
  \end{itemize}
Each of them switch signs;
\item[Peaks] when the parities of the numbers of peaks pointing to the left differ;
\item[Maps] when $\phi_R$ or $\phi_{ex}$ is involved since they appear with a minus sign in the definition of $\psi$.
\end{description}

\begin{figure}[p]
  \begin{center}
    \subfigure[]{\fbox{\xymatrix@R=.6cm{
          \dessinH{5.7cm}{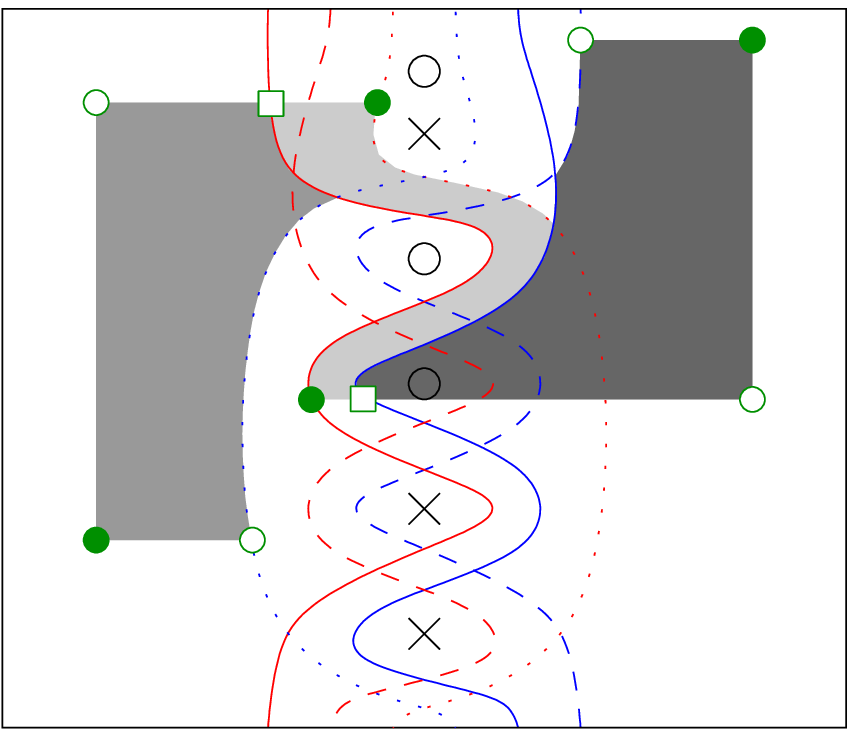} \ar@{^<-_>}[d]_(.46){\p_2^-\circ F_T\ }^(.53){\ \phi_R\circ \p_1^-} \\  \dessinH{5.7cm}{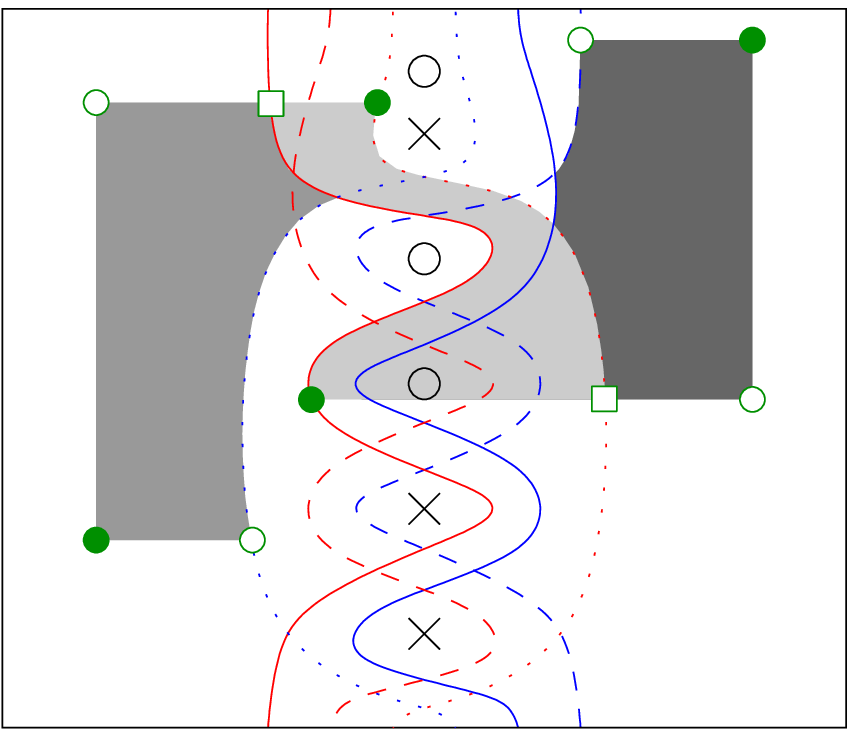}
        }}} \hspace{1cm}
    \subfigure[]{\fbox{\xymatrix@R=.6cm{
          \dessinH{5.7cm}{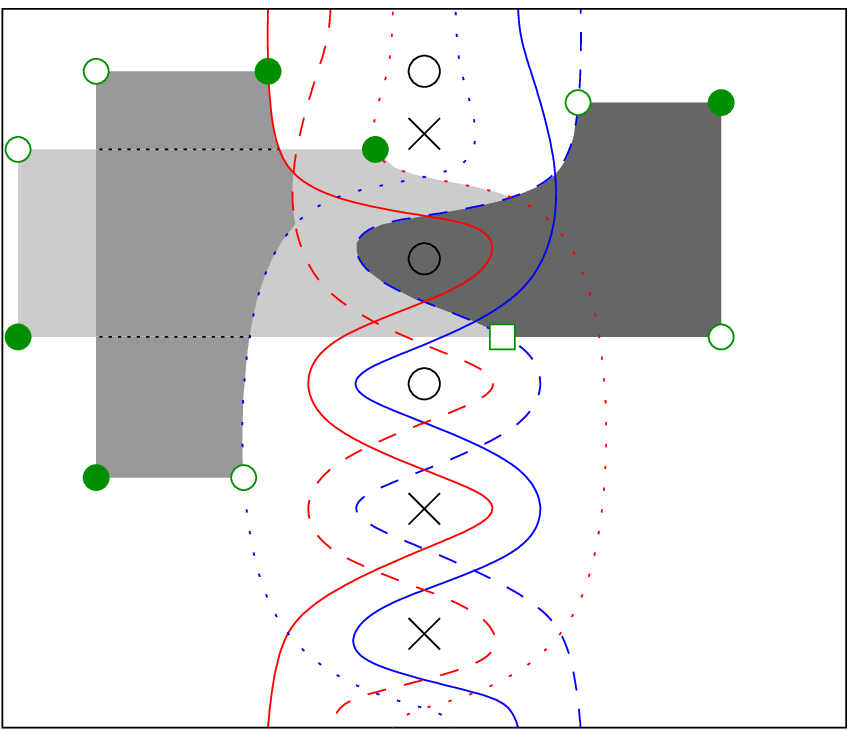} \ar@{^<-_>}[d]_(.46){\p_2^-\circ \phi_L\ }^(.53){\ \phi_L\circ\p_1^-} \\  \dessinH{5.7cm}{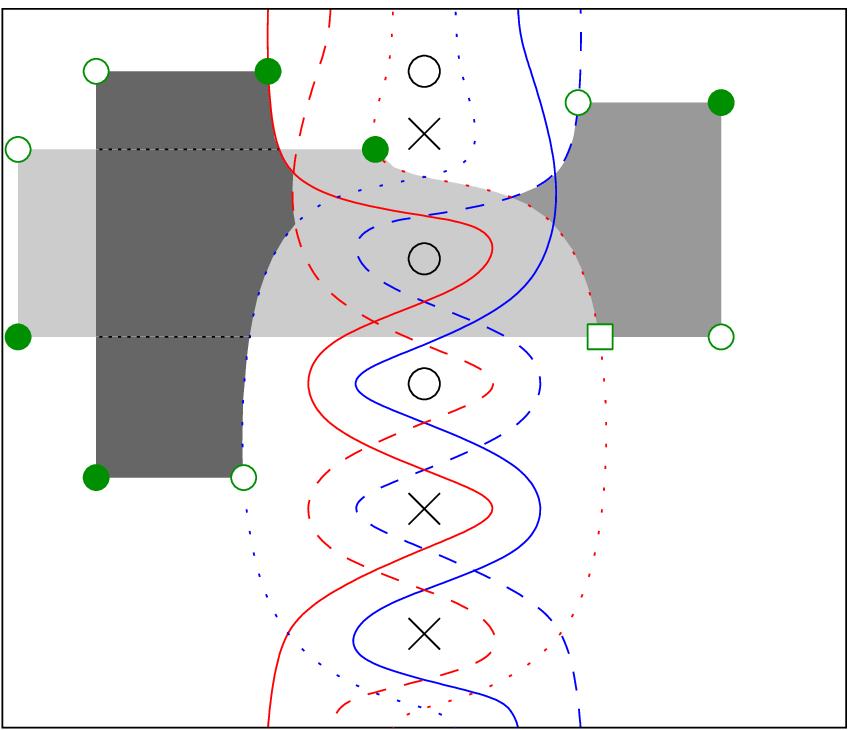}
        }}}
    \\
    \subfigure[]{\fbox{\xymatrix@R=.6cm@C=1.45cm{
          \dessinH{5.7cm}{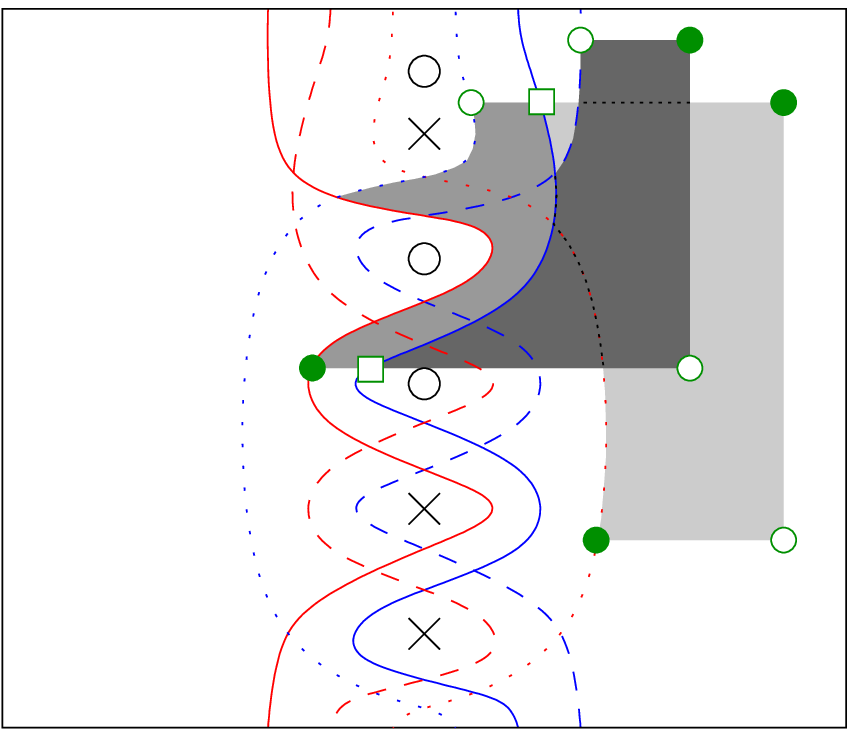} \ar@{^<-_>}[r]_(.46){\p_2^-\circ F_T}^(.54){\p_2^-\circ\phi_R} & \dessinH{5.7cm}{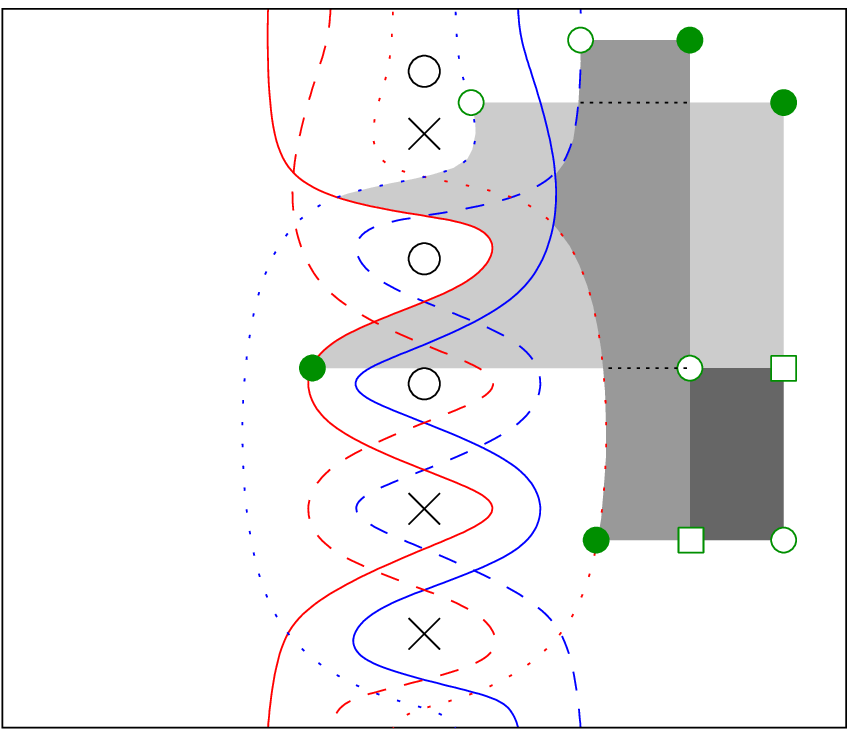}
        }}}
  \end{center}
  \caption{Special cancelling pairs: {\footnotesize Dark dots describe the initial generator while hollow ones describe the final one. Squares describe intermediate states. Polygons are depicted by shading. The lightest is the first to occur whereas the darkest one is the last. Polygons are stacked in an opaque way. For each configuration, we indicate to which part of $\p^-_2\circ \psi-\psi\circ \p^-_1$ it belongs.}}
  \label{fig:AllCases}
\end{figure}

\addtocounter{figure}{-1}
\addtocounter{subfigure}{3}

\begin{figure}[p]
  \begin{center}

    \subfigure[]{\fbox{\xymatrix@R=.6cm{
          \dessinH{5.7cm}{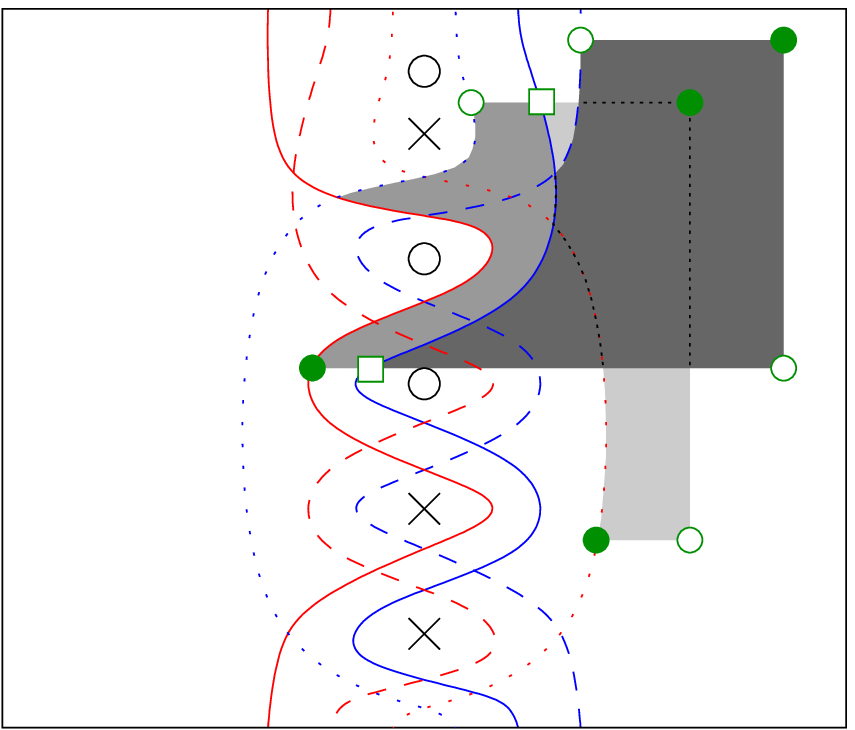} \ar@{^<-_>}[d]_(.46){\p_2^-\circ F_T\ }^(.53){\ \phi_R\circ\p_1^-} \\  \dessinH{5.7cm}{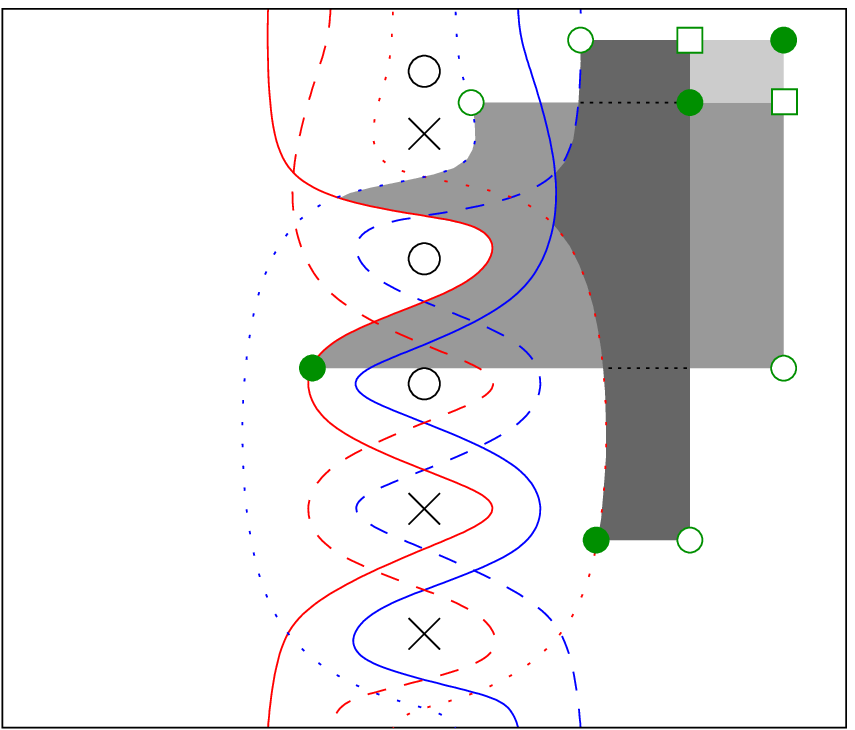}
        }}} \hspace{1cm}
    \subfigure[]{\fbox{\xymatrix@R=.6cm{
          \dessinH{5.7cm}{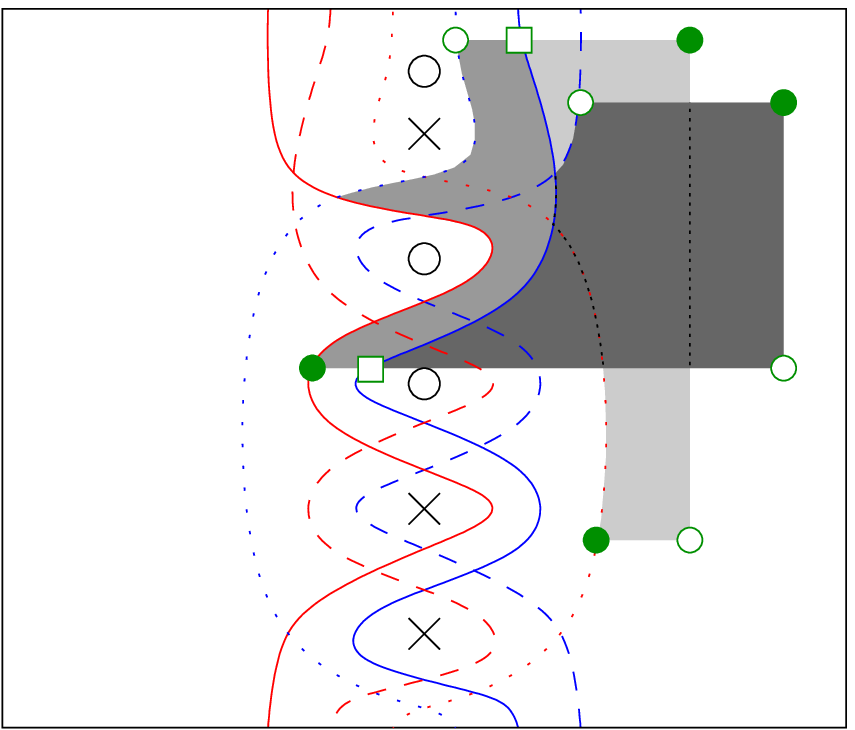} \ar@{^<-_>}[d]_(.46){\p_2^-\circ F_T\ }^(.53){\ \p_2^-\circ\phi_R} \\  \dessinH{5.7cm}{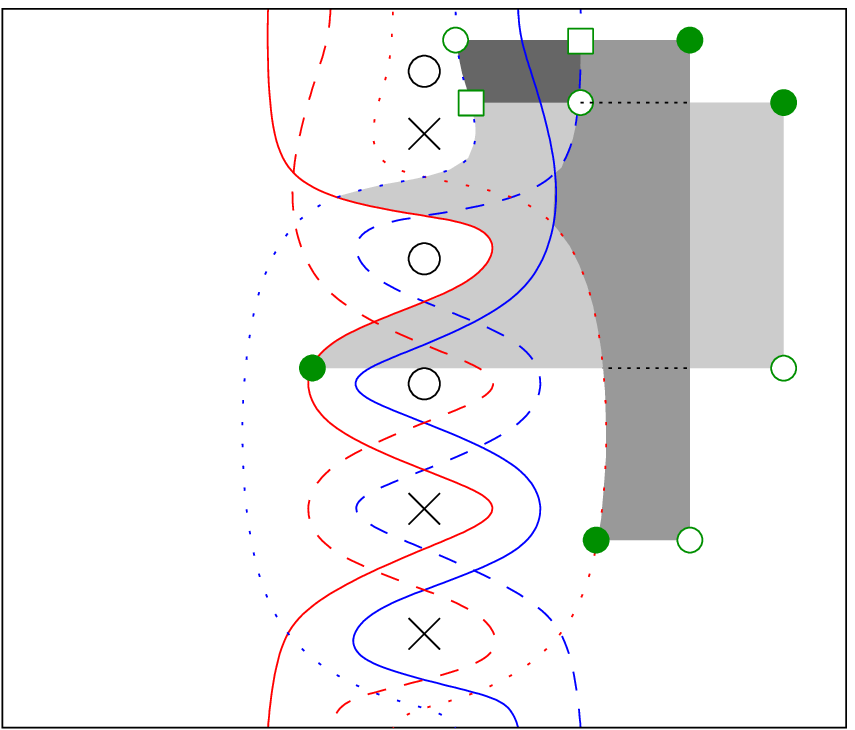}
        }}} \\
    \subfigure[]{\fbox{\xymatrix@R=.6cm@C=1.45cm{
          \dessinH{5.7cm}{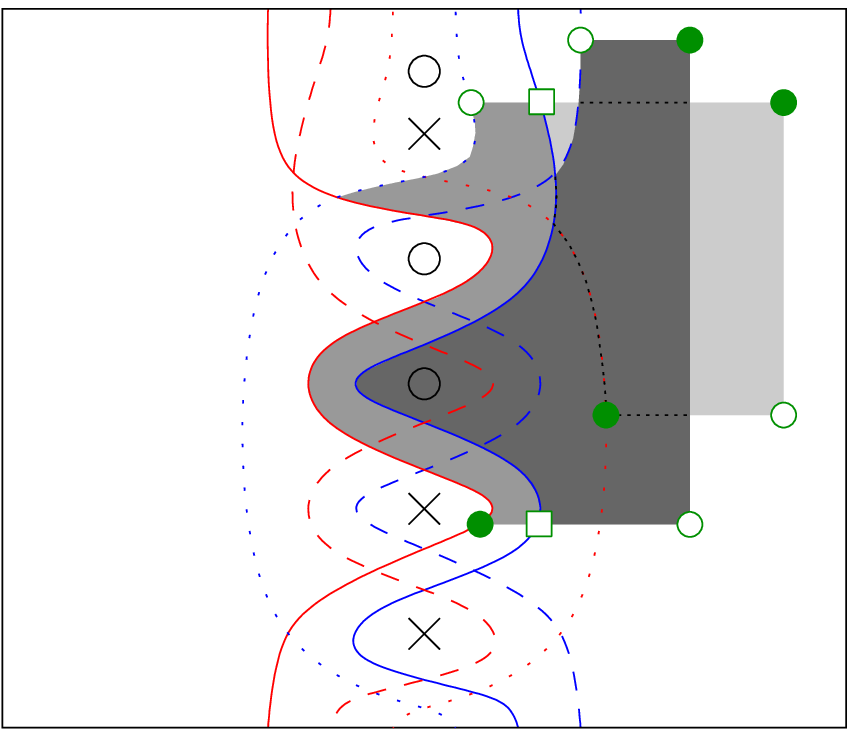} \ar@{^<-_>}[r]_(.46){\p_2^-\circ F_T}^(.54){\phi_R\circ\p_1^-} & \dessinH{5.7cm}{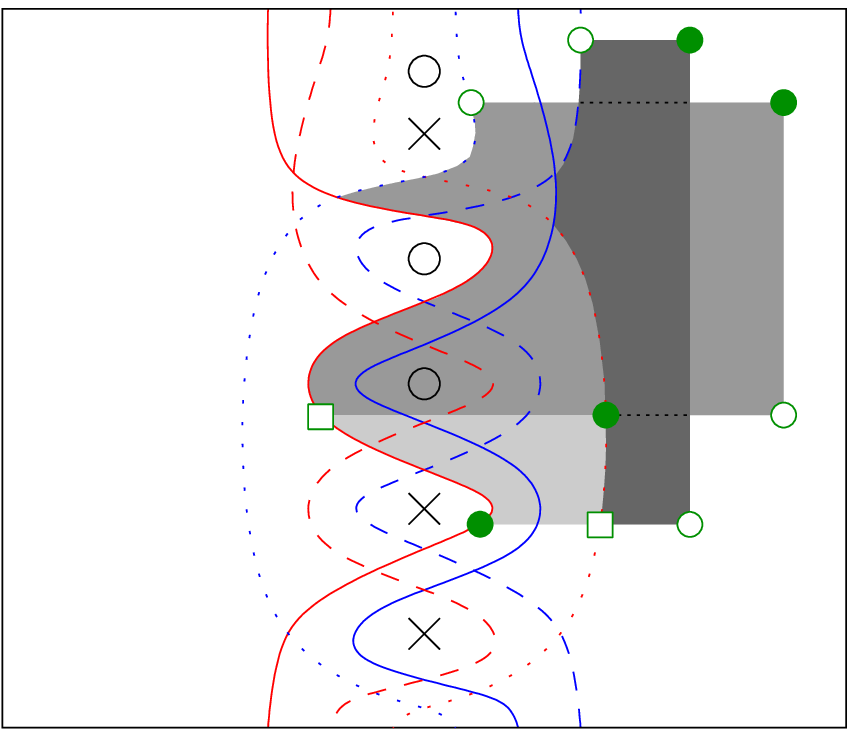}
        }}}
  \end{center}
  \caption{Special cancelling pairs: {\footnotesize Dark dots describe the initial generator while hollow ones describe the final one. Squares describe intermediate states. Polygons are depicted by shading. The lightest is the first to occur whereas the darkest one is the last. Polygons are stacked in an opaque way. For each configuration, we indicate to which part of $\p^-_2\circ \psi-\psi\circ \p^-_1$ it belongs.}}
\end{figure}

\addtocounter{figure}{-1}
\addtocounter{subfigure}{6}

\begin{figure}[p]
  \begin{center}
    \subfigure[]{\fbox{\xymatrix@R=.6cm{
          \dessinH{5.7cm}{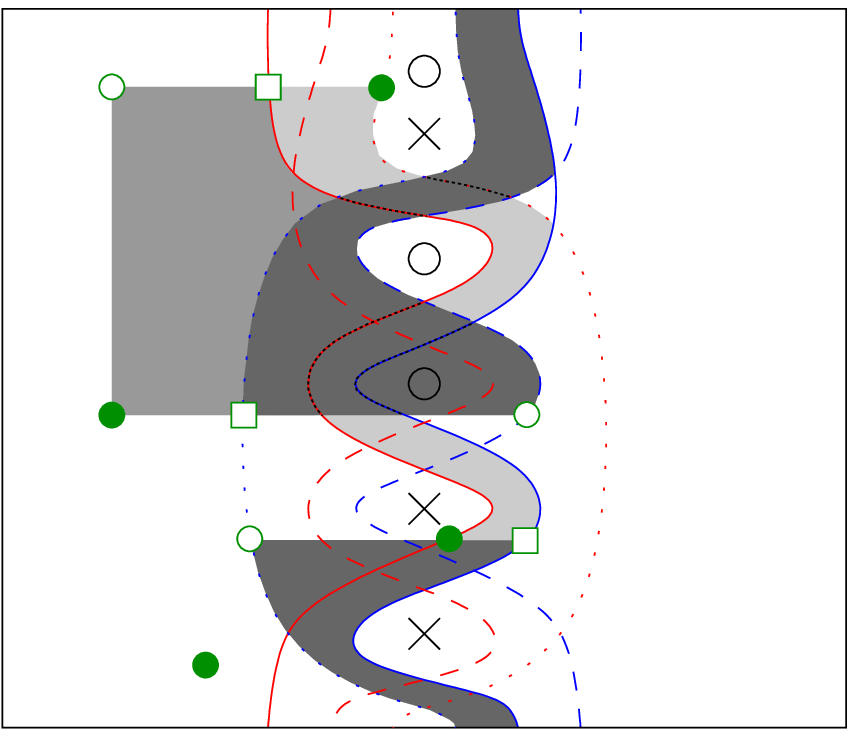} \ar@{^<-_>}[d]_(.46){\p_2^-\circ F_T\ }^(.53){\ \phi_L\circ\p_1^-} \\  \dessinH{5.7cm}{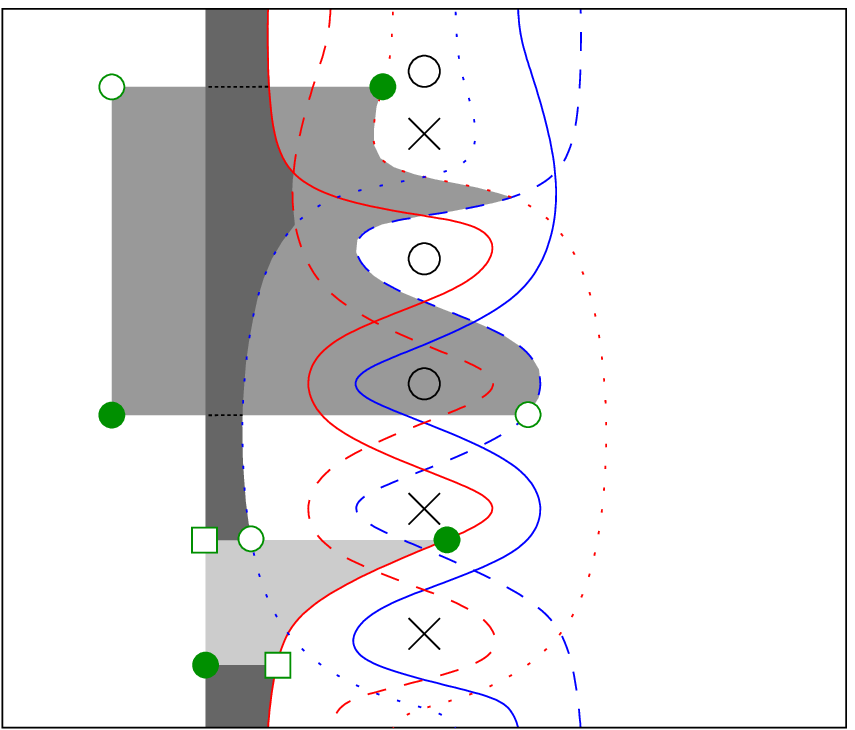}
        }}} \hspace{1cm}
    \subfigure[]{\fbox{\xymatrix@R=.6cm{
          \dessinH{5.7cm}{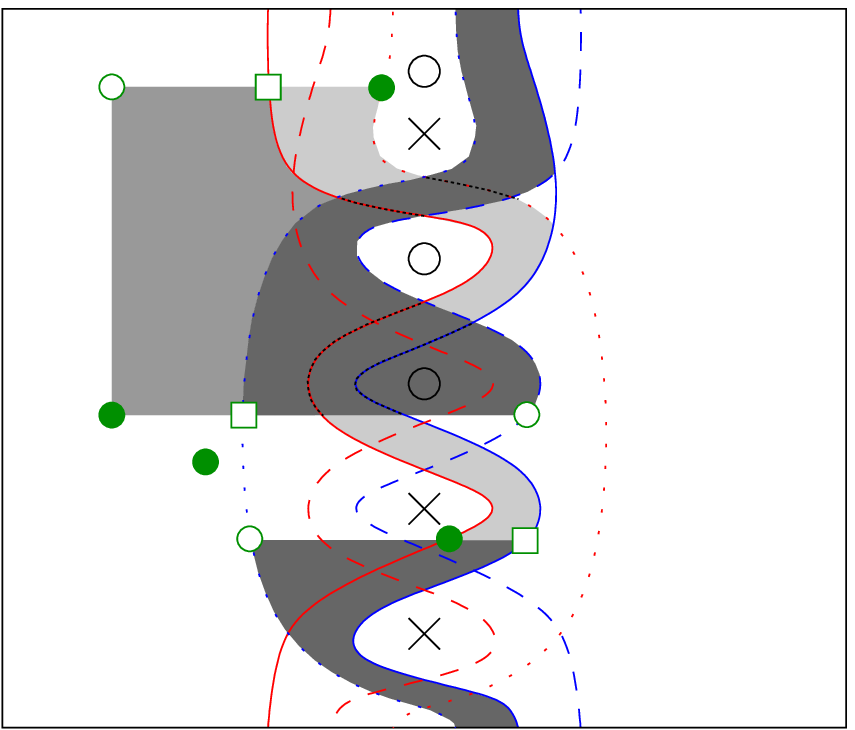} \ar@{^<-_>}[d]_(.46){\p_2^-\circ F_T\ }^(.53){\ \p_2^-\circ\phi_L} \\  \dessinH{5.7cm}{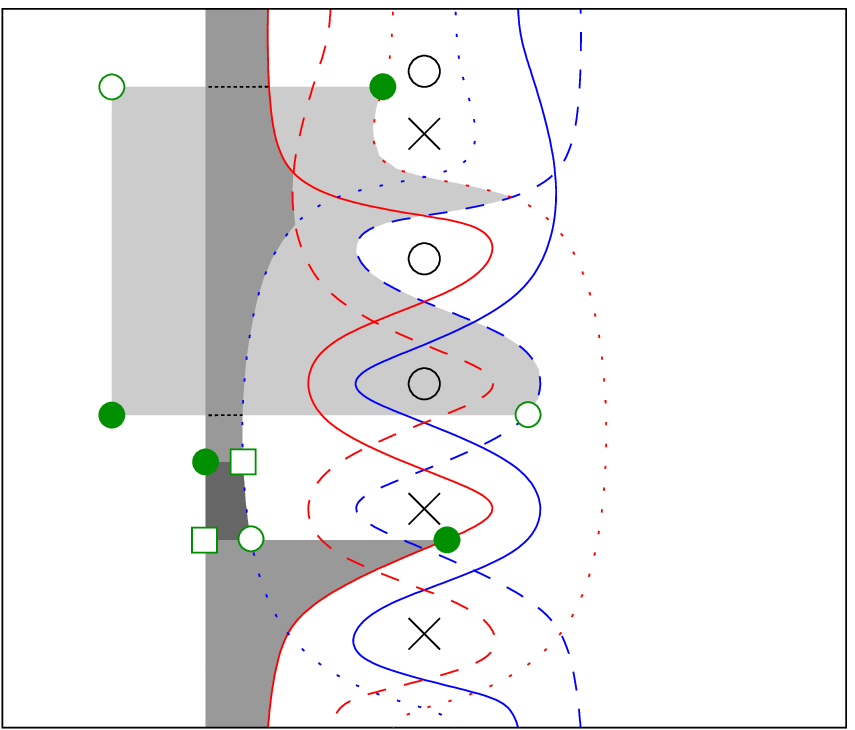}
        }}}
    \\
    \subfigure[]{\fbox{\xymatrix@R=.6cm@C=1.45cm{
          \dessinH{5.7cm}{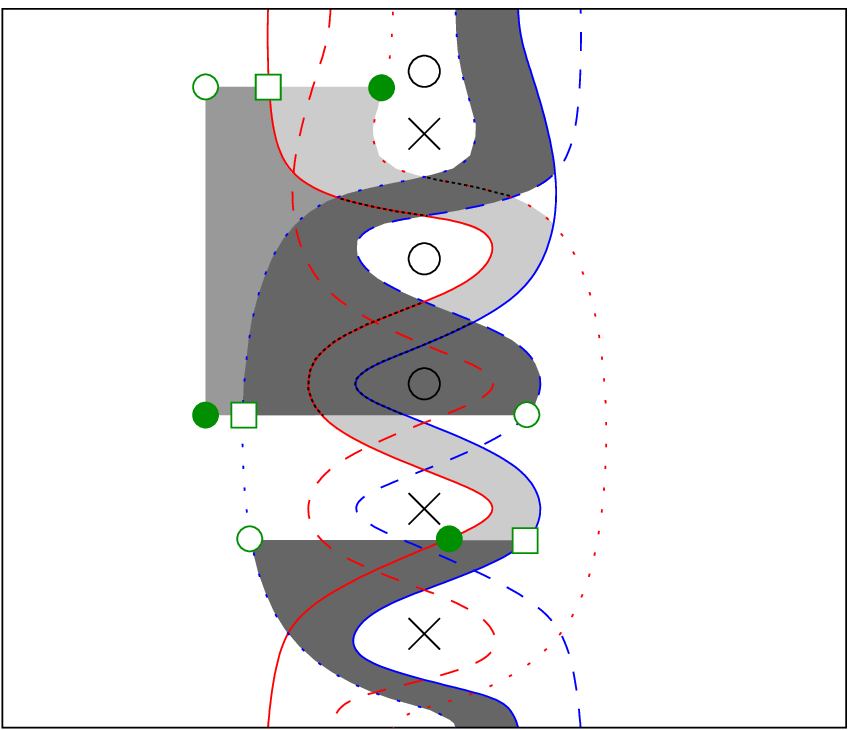} \ar@{^<-_>}[r]_(.46){\p_2^-\circ F_T}^(.54){F_B\circ\p_1^-} & \dessinH{5.7cm}{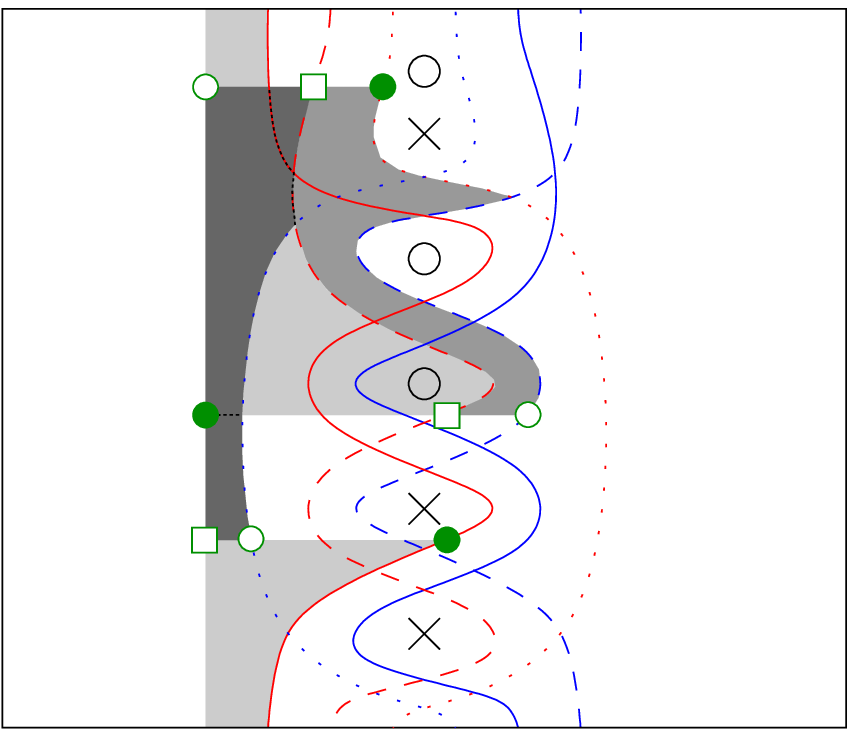}
        }}}
  \end{center}
  \caption{Special cancelling pairs: {\footnotesize Dark dots describe the initial generator while hollow ones describe the final one. Squares describe intermediate states. Polygons are depicted by shading. The lightest is the first to occur whereas the darkest one is the last. Polygons are stacked in an opaque way. For each configuration, we indicate to which part of $\p^-_2\circ \psi-\psi\circ \p^-_1$ it belongs.}}
\end{figure}

\addtocounter{figure}{-1}
\addtocounter{subfigure}{9}

\begin{figure}[p]
  \begin{center}
    \subfigure[]{\fbox{\xymatrix@R=.6cm{
          \dessinH{5.7cm}{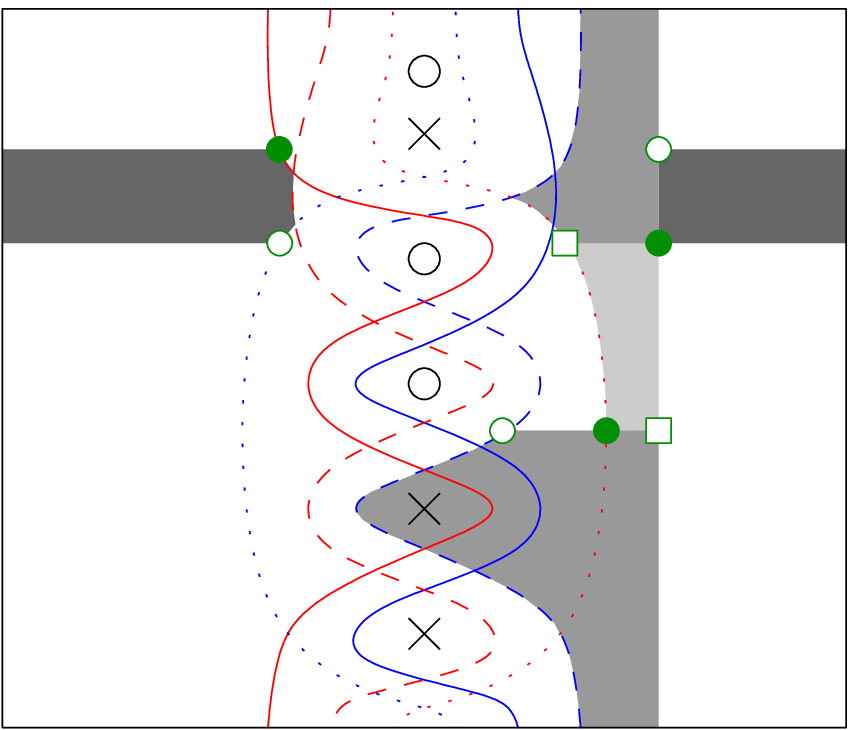} \ar@{^<-_>}[d]_(.46){\phi_L\circ \p_1^-\ }^(.53){\ \p_2^-\circ\phi_R} \\  \dessinH{5.7cm}{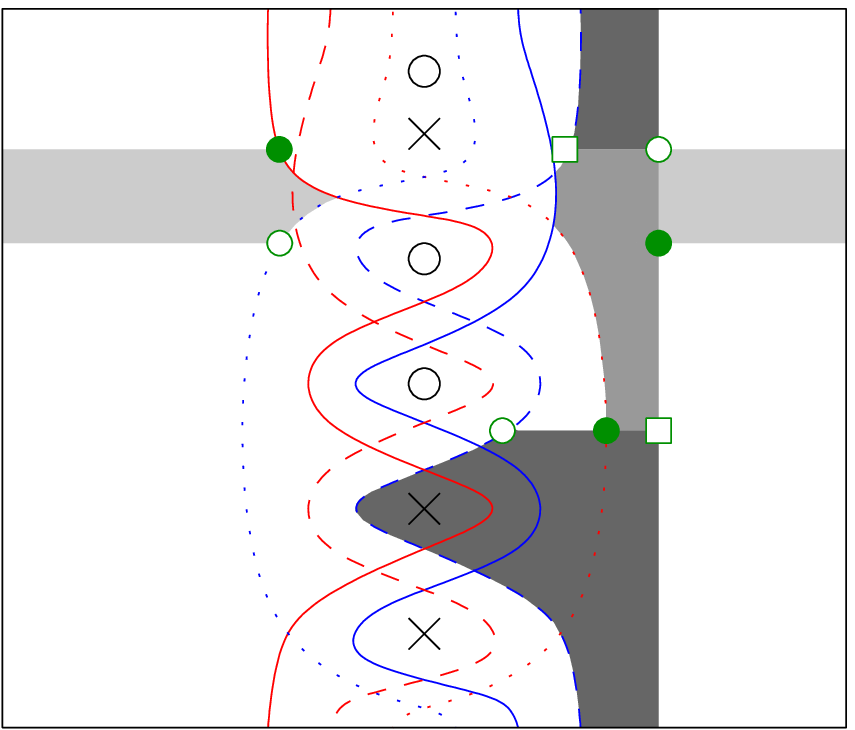}
        }}} \hspace{1cm}
    \subfigure[]{\fbox{\xymatrix@R=.6cm{
          \dessinH{5.7cm}{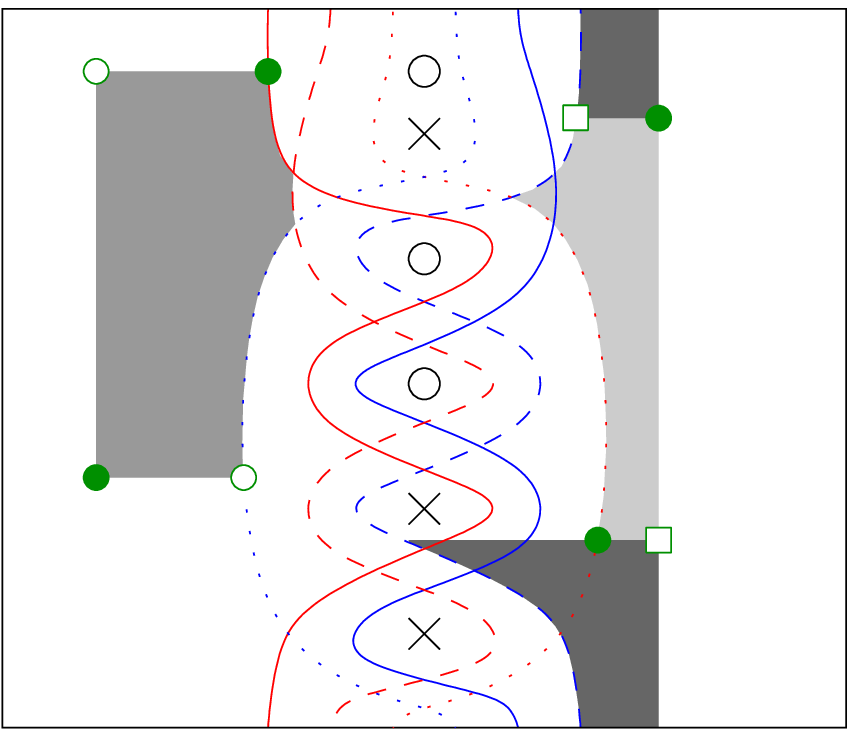} \ar@{^<-_>}[d]_(.46){\p_2^-\circ \phi_L\ }^(.53){\ \p_2^-\circ\phi_R} \\  \dessinH{5.7cm}{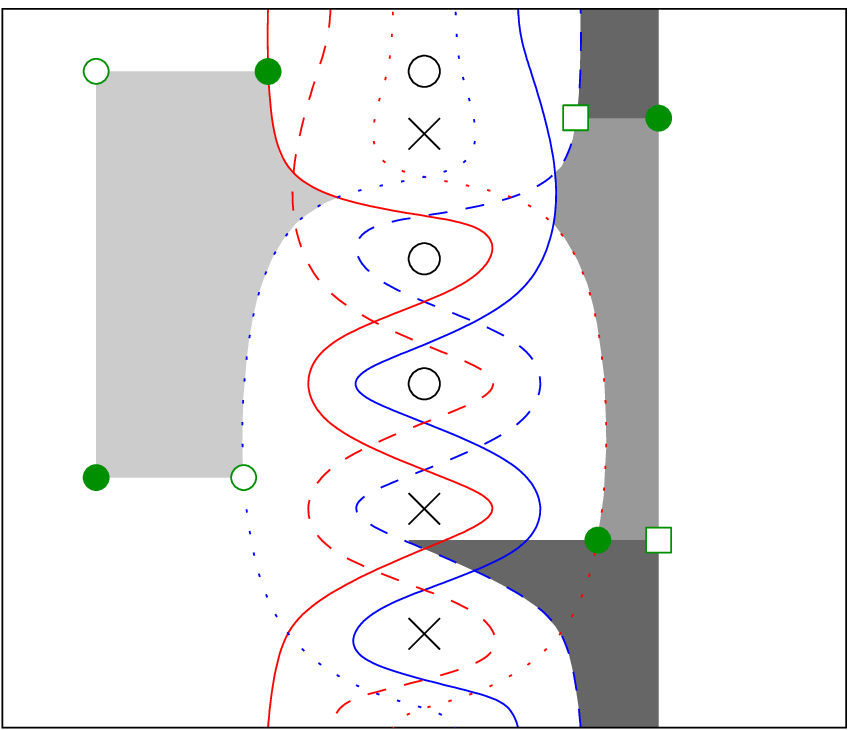}
        }}}
    \\
    \subfigure[]{\fbox{\xymatrix@R=.6cm@C=1.45cm{
          \dessinH{5.7cm}{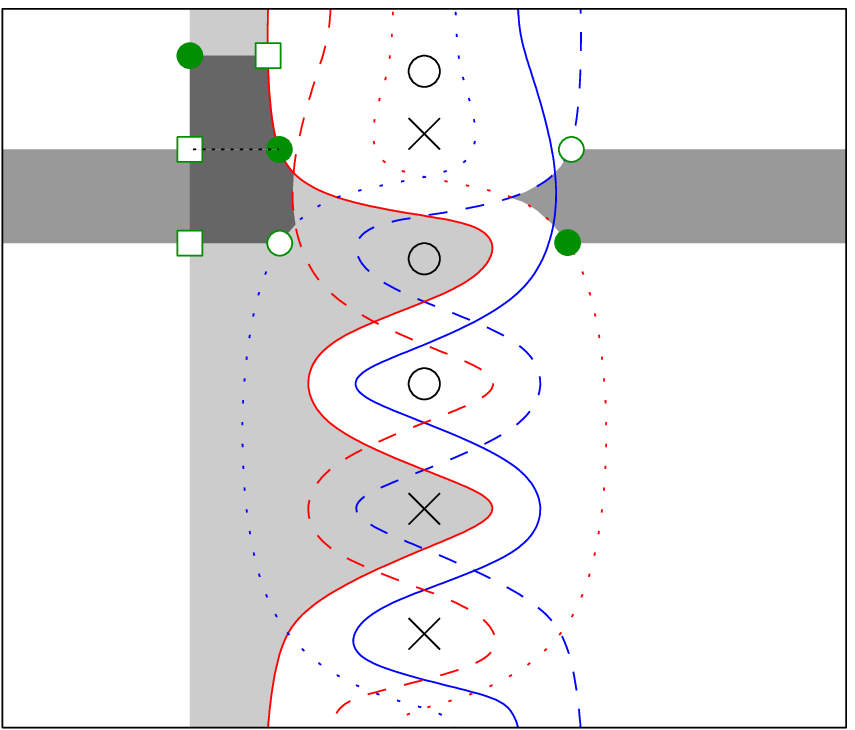} \ar@{^<-_>}[r]_(.46){\phi_L\circ \p_1^-}^(.54){\p_2^-\circ\phi_L} & \dessinH{5.7cm}{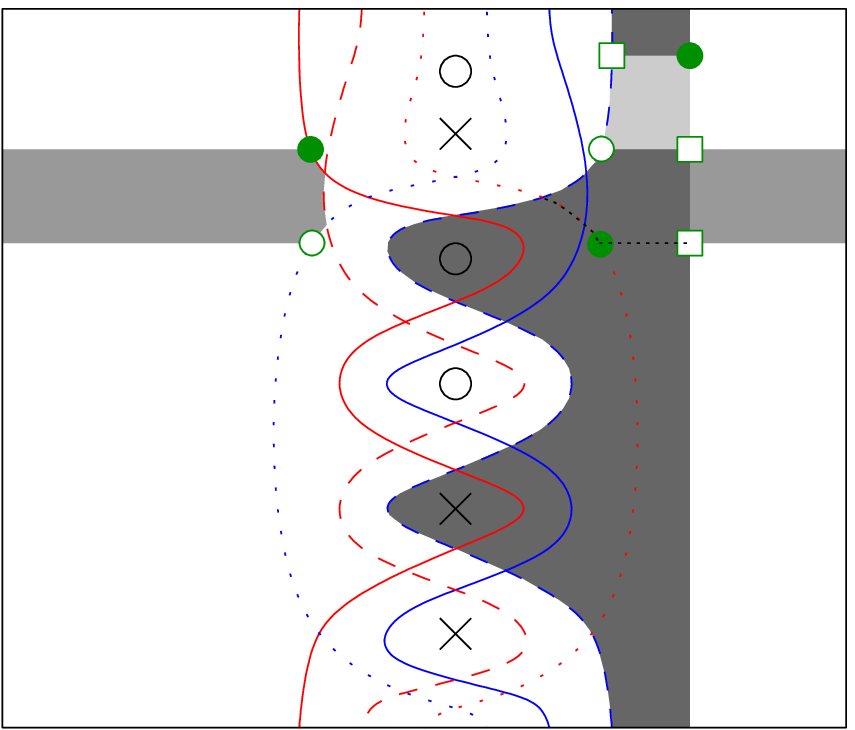}
        }}}
  \end{center}
  \caption{Special cancelling pairs: {\footnotesize Dark dots describe the initial generator while hollow ones describe the final one. Squares describe intermediate states. Polygons are depicted by shading. The lightest is the first to occur whereas the darkest one is the last. Polygons are stacked in an opaque way. For each configuration, we indicate to which part of $\p^-_2\circ \psi-\psi\circ \p^-_1$ it belongs.}}
\end{figure}

\addtocounter{figure}{-1}
\addtocounter{subfigure}{12}

\begin{figure}[p]
  \begin{center}
 \hspace{1cm}

    \subfigure[]{\fbox{\xymatrix@R=.6cm{
          \dessinH{5.7cm}{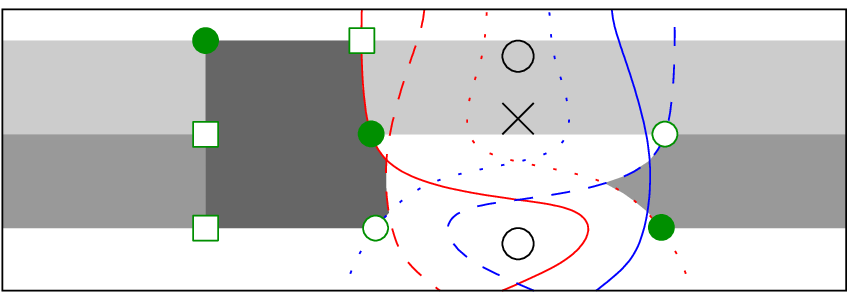} \ar@{^<-_>}[d]_(.44){\phi_L\circ \p_1^-\ }^(.57){\ \p_2^-\circ\phi_L} \\  \dessinH{5.7cm}{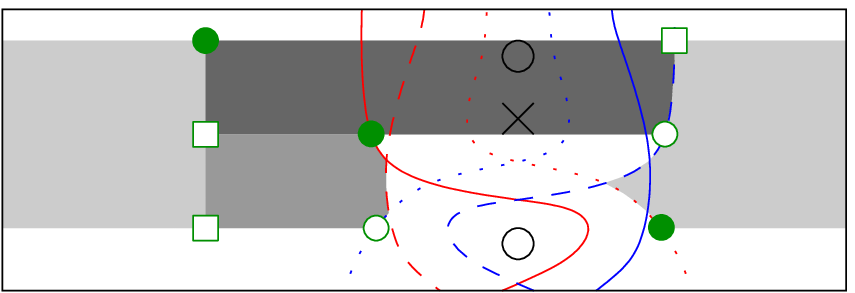}
        }}} \hspace{1cm}
    \subfigure[]{\fbox{\xymatrix@R=.6cm{
          \dessinH{5.7cm}{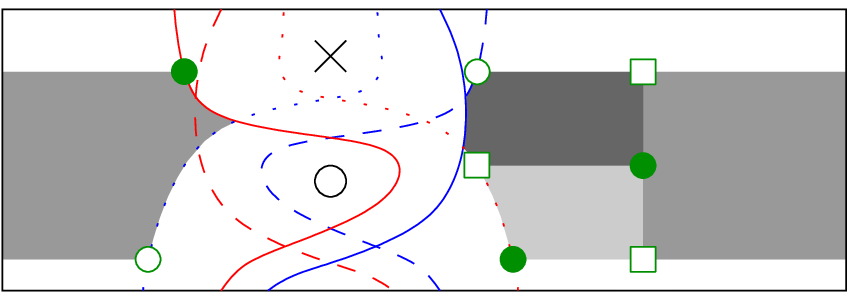} \ar@{^<-_>}[d]_(.44){\phi_R\circ \p_1^-\ }^(.57){\ \p_2^-\circ\phi_R} \\  \dessinH{5.7cm}{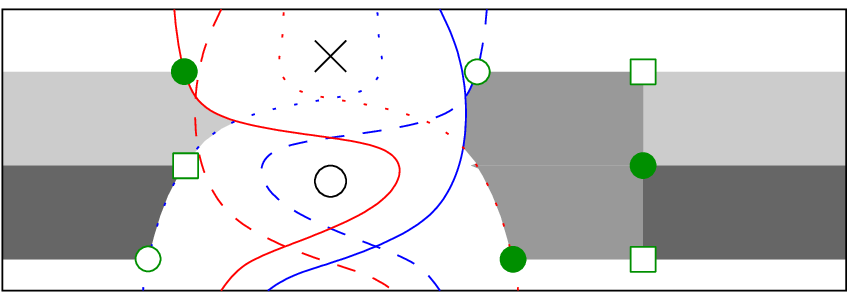}
        }}}
    \\
    \subfigure[]{\fbox{\xymatrix@R=.6cm@C=1.45cm{
          \dessinH{5.7cm}{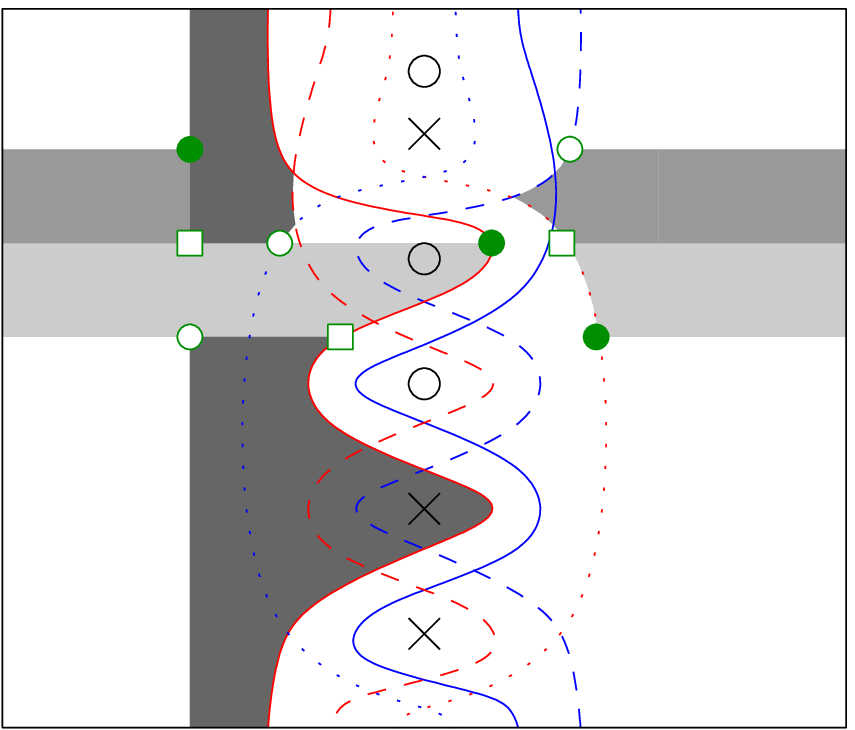} \ar@{^<-_>}[r]_(.46){\phi_L\circ \p_1^-}^(.54){\phi_R\circ\p_1^-} & \dessinH{5.7cm}{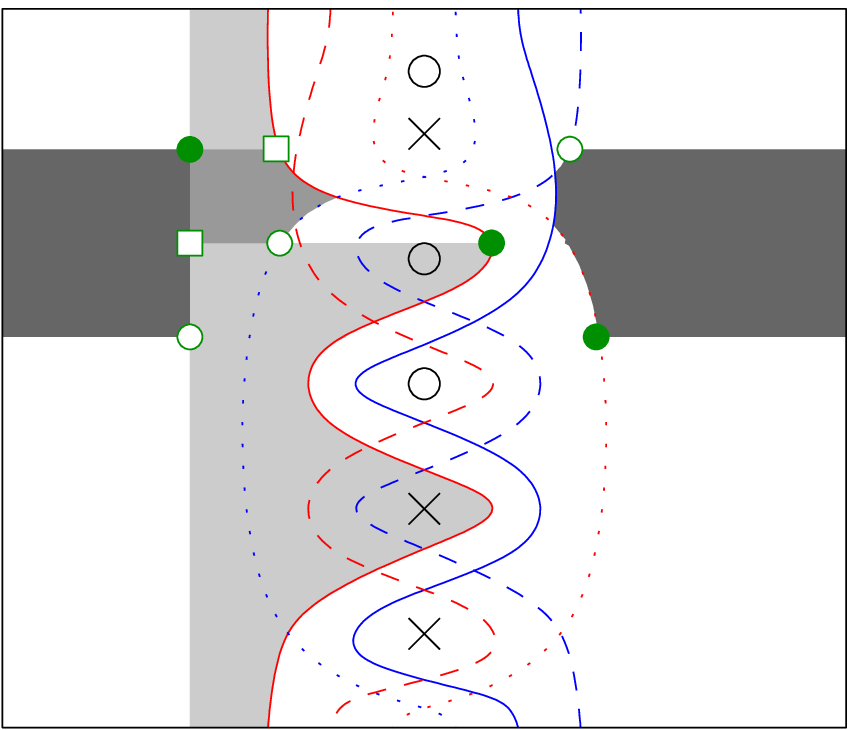}
        }}}
    \\
    \vspace{1.5cm}
    \subfigure[]{\fbox{\xymatrix@R=.6cm@C=1.45cm{
          \dessinH{5.7cm}{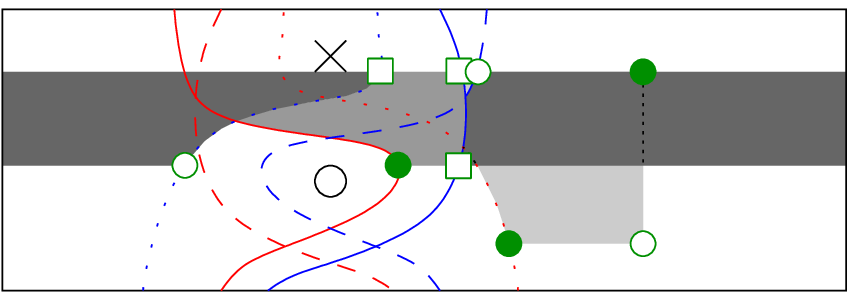} \ar@{^<-_>}[r]_(.46){\p_2^-\circ F_T}^(.54){\phi_R\circ\p_1^-} & \dessinH{5.7cm}{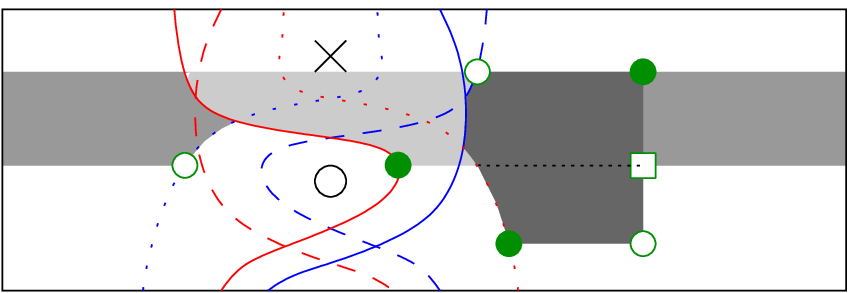}
        }}}
  \end{center}
  \caption{Special cancelling pairs: {\footnotesize Dark dots describe the initial generator while hollow ones describe the final one. Squares describe intermediate states. Polygons are depicted by shading. The lightest is the first to occur whereas the darkest one is the last. Polygons are stacked in an opaque way For each configuration, we indicate to which part of $\p^-_2\circ \psi-\psi\circ \p^-_1$ it belongs..}}
\end{figure}

\addtocounter{figure}{-1}
\addtocounter{subfigure}{16}

\begin{figure}[p]
  \begin{center}
    \subfigure[]{\fbox{\xymatrix@R=.6cm{ \dessinH{5.7cm}{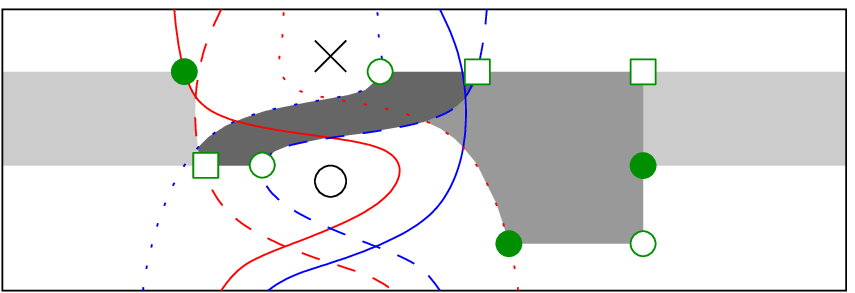} \\
          \dessinH{5.7cm}{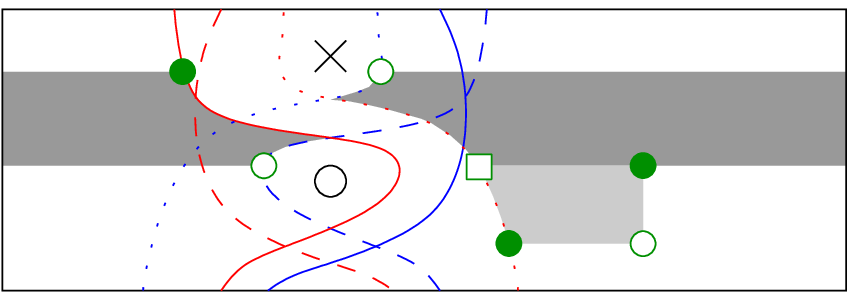} \ar@{^<-_>}[u]^(.57){F_B\circ\p_1^-\ }_(.44){\ \phi_{ex}\circ \p_1^-} \ar@{^<-_>}[d]^(.57){\p_2^-\circ \phi_R\ }_(.44){\ \phi_{ex}\circ \p_1^-} \\ \dessinH{5.7cm}{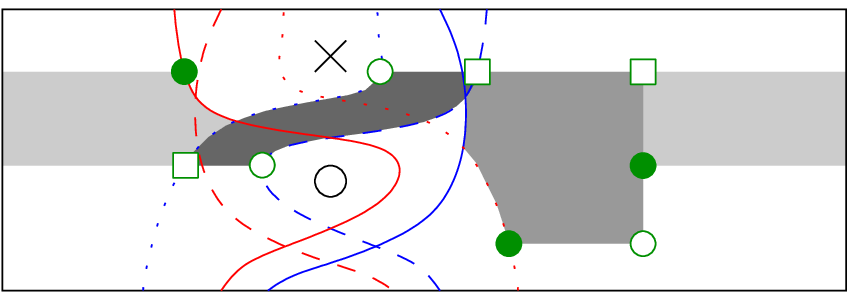}
        }}} \hspace{1cm}
    \subfigure[]{\fbox{\xymatrix@R=.6cm{ \dessinH{5.7cm}{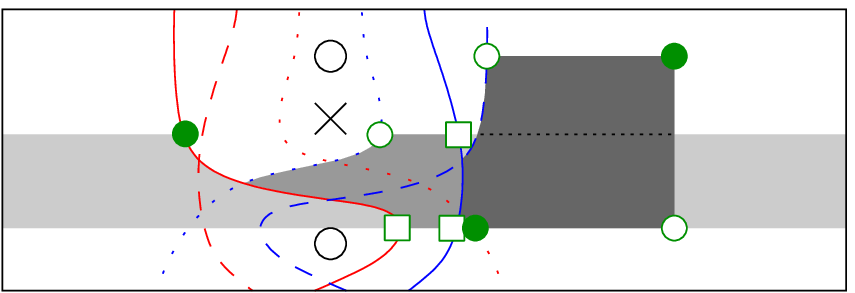} \\
          \dessinH{5.7cm}{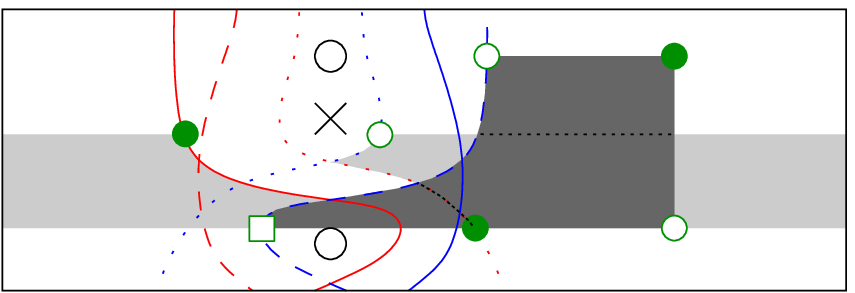} \ar@{^<-_>}[u]^(.57){\ \p_2^-\circ F_T}_(.44){\ \p_2^-\circ \phi_{ex}} \ar@{^<-_>}[d]^(.57){\ \p_2^-\circ \phi_L}_(.44){\ \p_2^-\circ \phi_{ex}} \\ \dessinH{5.7cm}{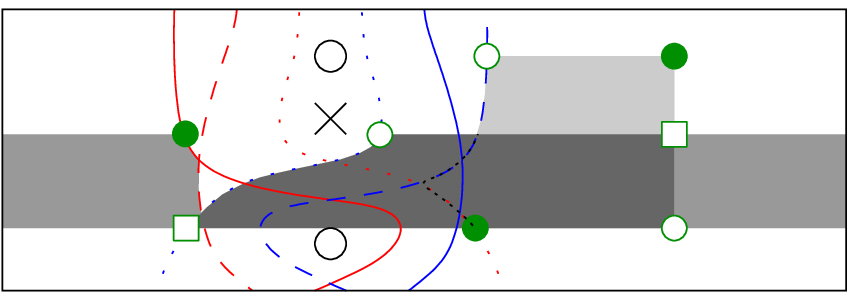}
        }}}
%     \subfigure[]{\fbox{\xymatrix@R=.6cm{
%           \dessinH{5.7cm}{302b} \\ \dessinH{5.7cm}{302} \ar@{^<-_>}[u]^(.57){F_B\circ \p_1^-\ }_(.44){\ \p_2^-\circ\phi_{ex}} \ar@{^<-_>}[d]^(.57){\p^-_2\circ\phi_R\ }_(.44){\ \phi_{ex}\circ \p_1^-} \\ \dessinH{5.7cm}{302a}
%         }}}
    \\
    \subfigure[]{\fbox{\xymatrix@R=.6cm{ \dessinH{5.7cm}{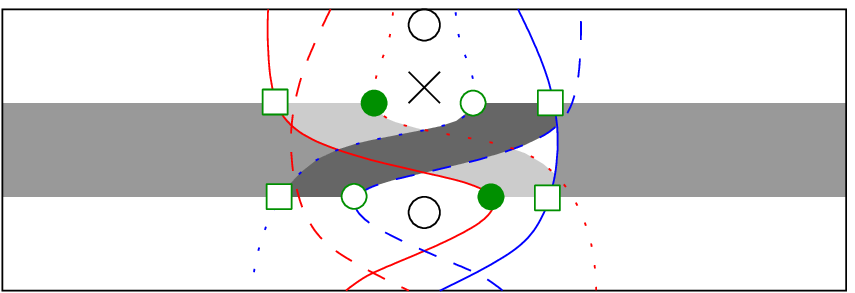} \\
          \dessinH{5.7cm}{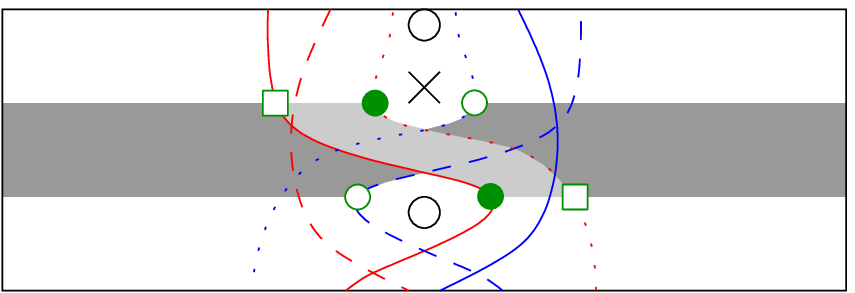} \ar@{^<-_>}[u]^(.57){\p_2^-\circ F_T\ }_(.44){\ \phi_{ex}\circ \p_1^-} \ar@{^<-_>}[d]^(.57){\ F_B\circ\p_1^-}_(.44){\ \phi_{ex}\circ \p_1^-} \\ \dessinH{5.7cm}{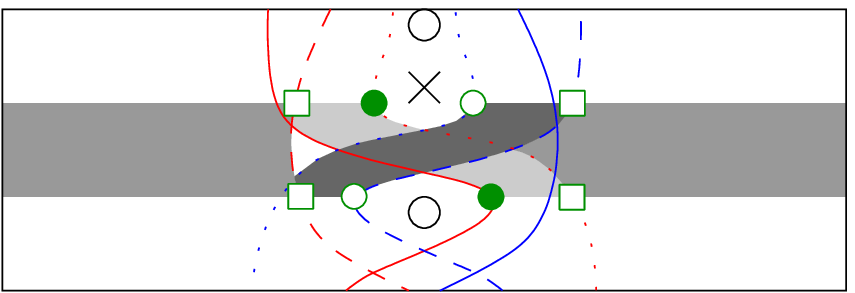}
        }}} \hspace{1cm}
    \subfigure[]{\fbox{\xymatrix@R=.6cm{ \dessinH{5.7cm}{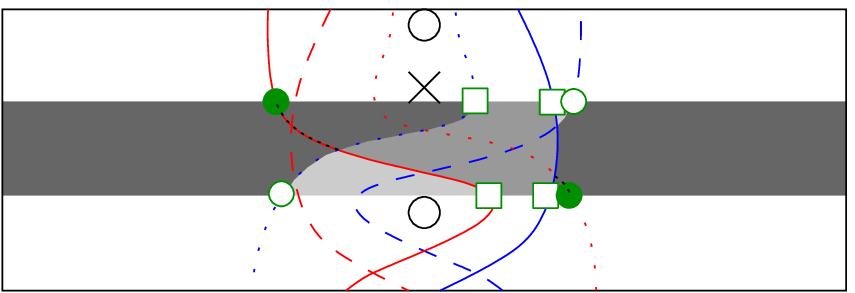} \\
          \dessinH{5.7cm}{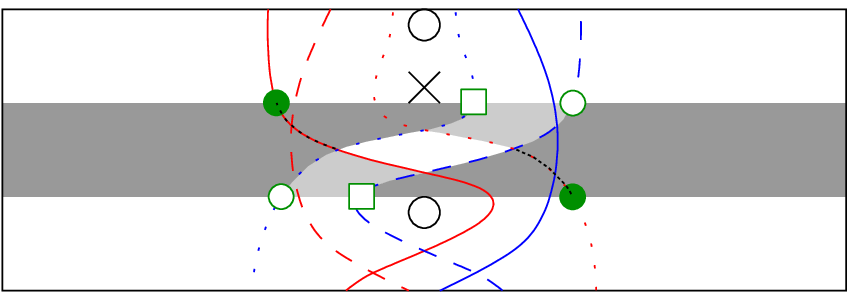} \ar@{^<-_>}[u]^(.57){\ \p_2^-\circ F_T}_(.44){\ \p_2^-\circ \phi_{ex}} \ar@{^<-_>}[d]^(.57){\ F_B\circ\p_1^-}_(.44){\ \p_2^-\circ \phi_{ex}} \\ \dessinH{5.7cm}{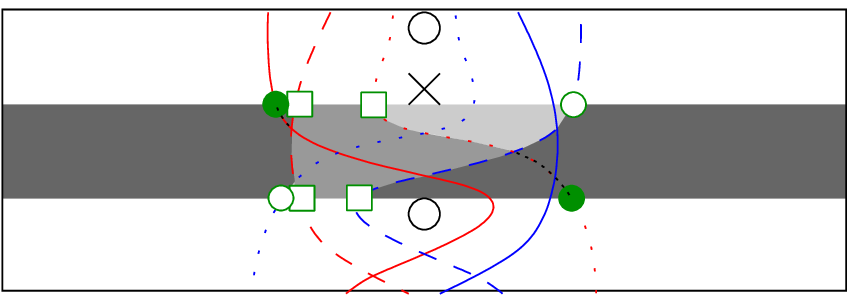}
        }}}
  \end{center}

  \caption{Exceptional cancelling pairs: {\footnotesize Dark dots
      describe the initial generator while hollow ones describe the
      final one. Squares describe intermediate states. Polygons are
      depicted by shading. The lightest is the first to occur
      whereas the darkest one is the last. Polygons are stacked in
      an opaque way. For each configuration, we indicate to which part of $\p^-_2\circ \psi-\psi\circ \p^-_1$ it belongs.}}
\end{figure}

Now, we can draw up a summary table:

\begin{center}
  \begin{tabular}{|l|c|c|c|c|c|c|c|c|c|c|c|c|}
    \cline{2-13}
    \multicolumn{1}{c|}{}&\rotatebox{90}{fig.(a)$\ $}&\rotatebox{90}{fig.(b)$\ $}&\rotatebox{90}{fig.(c)$\ $}&\rotatebox{90}{fig.(d)$\ $}&\rotatebox{90}{fig.(e)$\ $}&\rotatebox{90}{fig.(f)$\ $}&\rotatebox{90}{fig.(g)$\ $}&\rotatebox{90}{fig.(h)$\ $}&\rotatebox{90}{fig.(i)$\ $}&\rotatebox{90}{fig.(j)$\ $}&\rotatebox{90}{fig.(k)$\ $}&\rotatebox{90}{fig.(l)$\ $}\\
    \hline
    Order         &-&-&+&-&+&-&-&+&-&-&+&+\\
    \hline
    Configuration &+&-&+&-&+&-&+&-&+&+&-&+\\
    \hline
    Peaks         &-&-&+&+&+&+&+&+&+&-&-&-\\
    \hline
    Maps          &-&+&-&-&-&-&+&+&+&-&-&+\\
    \hline
  \end{tabular}
\end{center}
\begin{center}
  \begin{tabular}{|l|c|c|c|c|c|c|c|c|c|c|c|c|}
    \cline{2-13}
    \multicolumn{1}{c|}{}&\rotatebox{90}{fig.(m)$\ $}&\rotatebox{90}{fig.(n)$\ $}&\rotatebox{90}{fig.(o)$\ $}&\rotatebox{90}{fig.(p)$\ $}&\rotatebox{90}{fig.(q), top$\ $}&\rotatebox{90}{fig.(q), bottom$\ $}&\rotatebox{90}{fig.(r), top$\ $}&\rotatebox{90}{fig.(r), bottom$\ $}&\rotatebox{90}{fig.(s), top$\ $}&\rotatebox{90}{fig.(s), bottom$\ $}&\rotatebox{90}{fig.(t), top$\ $}&\rotatebox{90}{fig.(t), bottom$\ $}\\
    \hline
    Order         &-&-&-&-&+&-&+&+&-&+&+&-\\
    \hline
    Configuration &+&+&+&-&-&-&+&+&+&+&+&+\\
    \hline
    Peaks         &+&+&-&+&-&-&+&+&-&+&+&-\\
    \hline
    Maps          &+&+&-&-&-&+&-&-&-&-&-&-\\
    \hline
  \end{tabular}
\end{center}
  Every column has an odd number of minus signs.\\

  The \textbf{Configuration} row may need some details. As a model, we give the computations for the case (e). After having added spikes, the configurations of rectangles correspond to the two extremal terms in the following equality :
  $$
  \dessin{1.8cm}{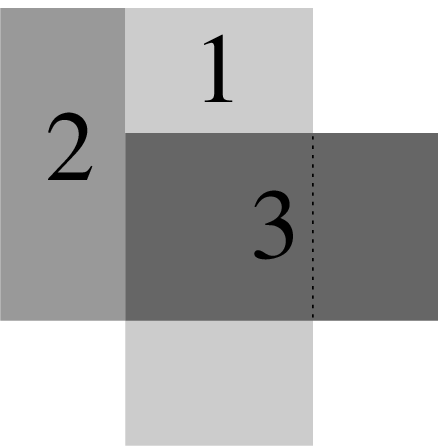}\ =\ -\ \dessin{1.8cm}{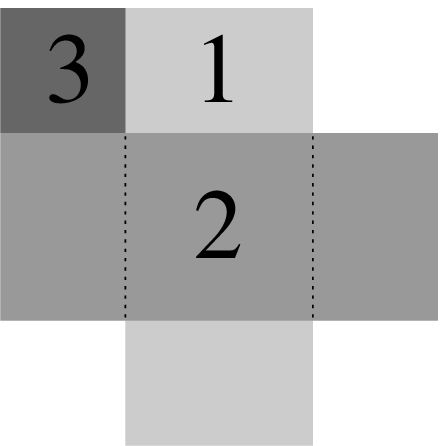}\ =\ \dessin{1.8cm}{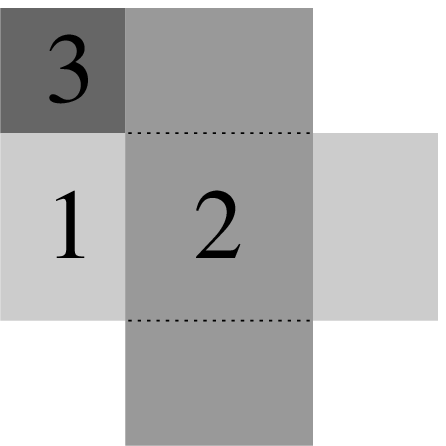}.
  $$

Nevertheless, cases (g), (h), (i), (l), (q), (r), (s) and (t) cannot be treated in this way. However, computations can be made in the group $\Sn$. Then, cases (r), (s) and (t) are consequences of Lemma \ref{SpinProp}, i) (\S\ref{par:SpinExt}), cases (g) and (h) of ii) and case (q) of iii). Cases (i) and (l) are direct applications of the definition of $\Sn$.
\end{proof}

\begin{remarque}
  If we restrict ourself to $\FF_2$--coefficients or to the graded versions, then the exceptional term $\phi_{ex}$ can be swept.
\end{remarque}

\subsubsection{Graded quasi-isomorphisms}
\label{par:Isomorphism}

Now we can consider the filtration introduced in paragraph \ref{par:MainDefinition} and already used in paragraph \ref{par:LastFiltration}. The associated graded map $\psi_{gr}$ corresponds then to the horizontal arrows in Figure \ref{fig:SemiSingTable}. It has been proved in paragraph \ref{par:LastFiltration} that they are quasi-isomorphisms. We thus conclude using Corollary \ref{QuasiIso} (\S\ref{par:MapCones}).

\subsection{Hint for the remaining cases}
\label{SingSing}

In this section, we gives some hints for direct proofs of the invariance under the last two moves. However, using mixed singular grids, these moves can be replaced by a third one. The proof of invariance under this move is then much easier. Details are written in Appendix \ref{appendix:MixedSimplification}.

\subsubsection{Strategy}
\label{par:Strategy}

Now, we must deal with the fully singular commutation move and the flip move. Since they imply handling an indecent number of cases, we will not give complete proofs but a strategy in four steps. This is actually the strategy which has been used in most of the previous invariance proofs:
\begin{enumerate}
\item[I.] Find a filtration for which the graded parts of differentials are simpler;
\item[II.] Find a graded quasi-isomorphism;
\item[III.] Extend this graded quasi-isomorphism to a filtrated one;
\item[IV.] Apply Corollary \ref{QuasiIso} (\S\ref{par:MapCones}).
\end{enumerate}

Unfortunatly, the third step requires to check an impressive number of distincts cases. Due to the lack of time, we have to elude this step.\\

The fourth one is self-contained.\\

For the remaining two ones, we will essentially consider the filtration associated to the number of singular columns positively resolved\footnote[1]{Actually, it is even sufficient to count positively resolved columns among the singular columns involved in the move we are dealing with.} (see paragraph \ref{par:LastFiltration}). It thus reduces the problem to regular grids.\\

Now the issue is to project the considered move into the regular world.

\subsubsection{Fully singular commutation}
\label{par:SingSingComm}

Like the semi-singular commutation is an avatar of the regular one, the fully singular commutation can be seen as the composition of two semi-singular ones:
$$
\xymatrix@!0@C=3.5cm@R=3.5cm{
  \dessin{2cm}{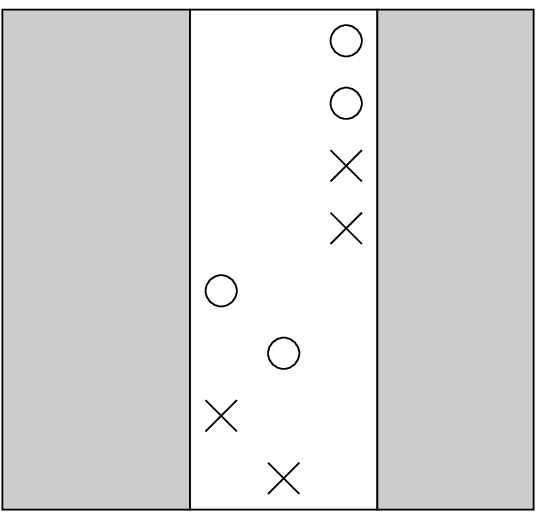} \ar[r] \ar@{-->}[d] &  \dessin{2cm}{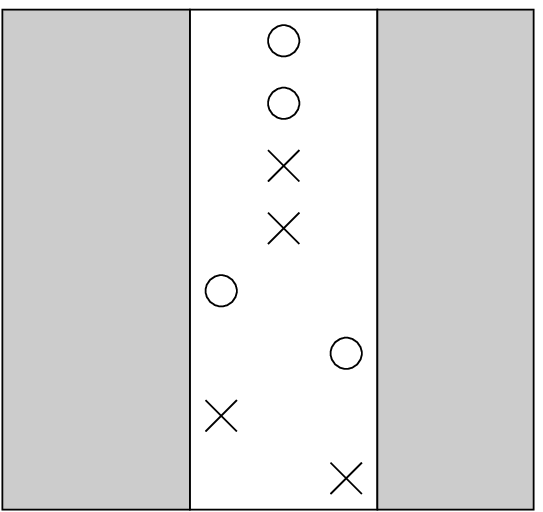} \ar[r] &  \dessin{2cm}{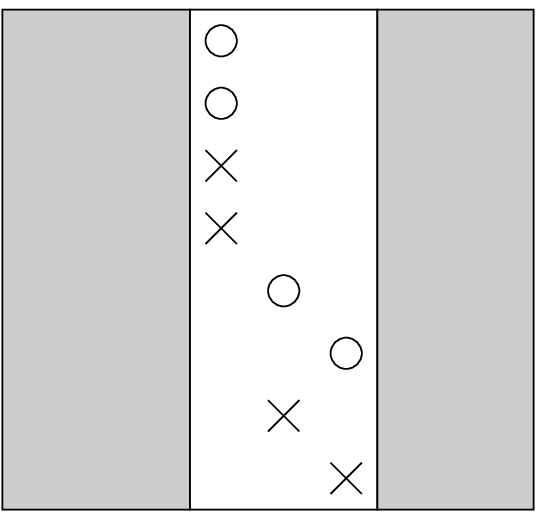} \ar@{-->}[d] \\
  \dessin{2cm}{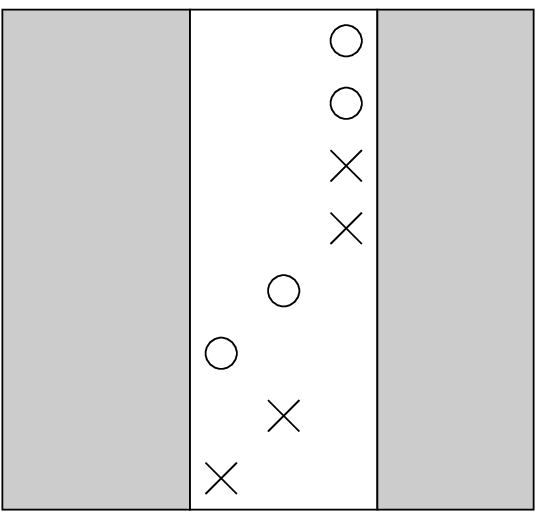} \ar[r] &  \dessin{2cm}{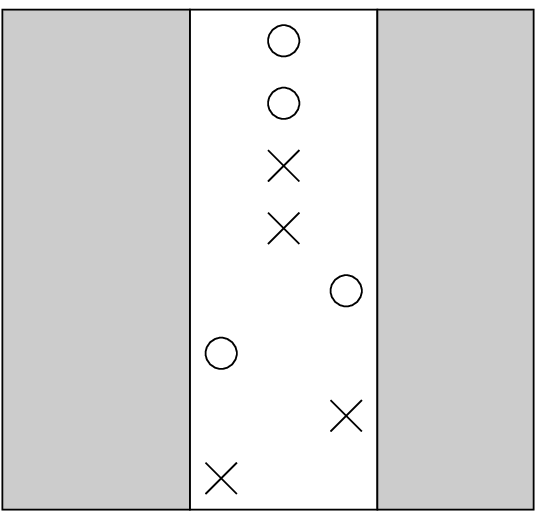} \ar[r] &  \dessin{2cm}{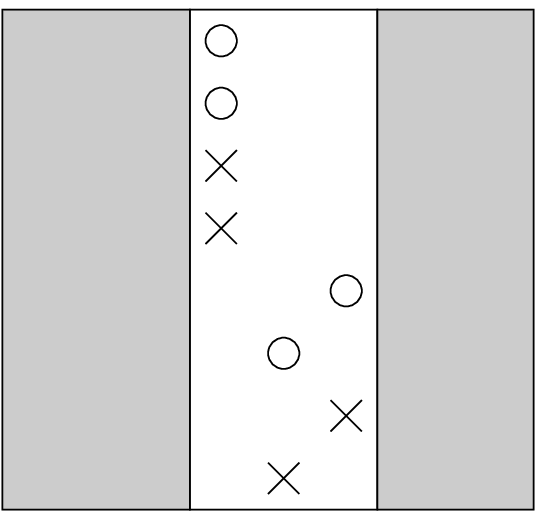}
}.
$$
Now we can iterate the process in order to get only regular grids but we can also only consider the filtration associated to the resolution of the commuting singular column with the lowest decorations. Then, we can use four copies of the maps defined in paragraph \ref{par:SemiSingMap} which have already been proved to be quasi-isomorphisms. Step II is then completed.

\subsubsection{Flip}
\label{par:FlipInvariance}

Instead of proving directly the invariance under the flip move, we consider the following uneven move
$$
\xymatrix{\dessin{1.3cm}{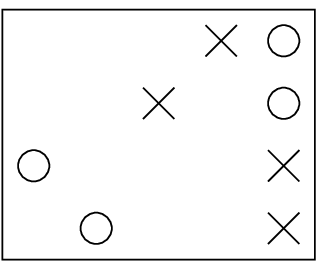} \ar@{<->}[r] & \dessin{1.3cm}{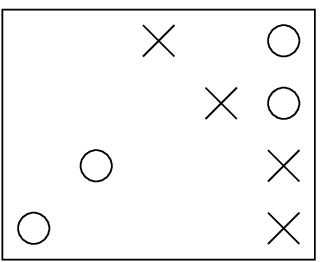}}.
$$
According to Lemma \ref{EquivalentMoves} (\S\ref{par:Flip}), it is equivalent to the flip move.\\
At the level of regular grid, this move can be decomposed into regular commutations:
$$
\xymatrix@!0@C=3.5cm@R=3.5cm{
  \dessin{2cm}{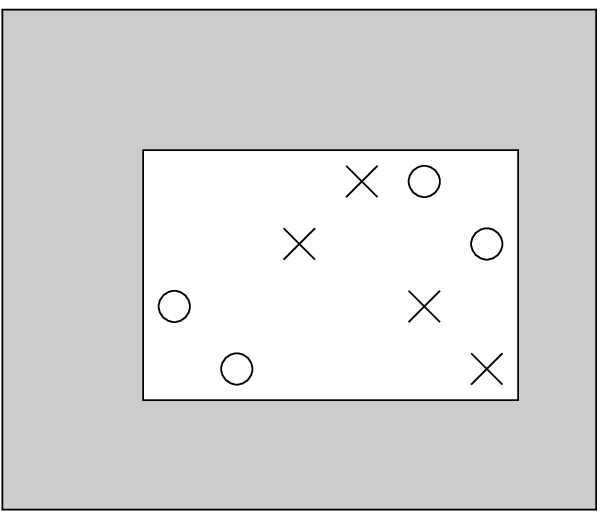} \ar[r] \ar@{-->}[d] &  \dessin{2cm}{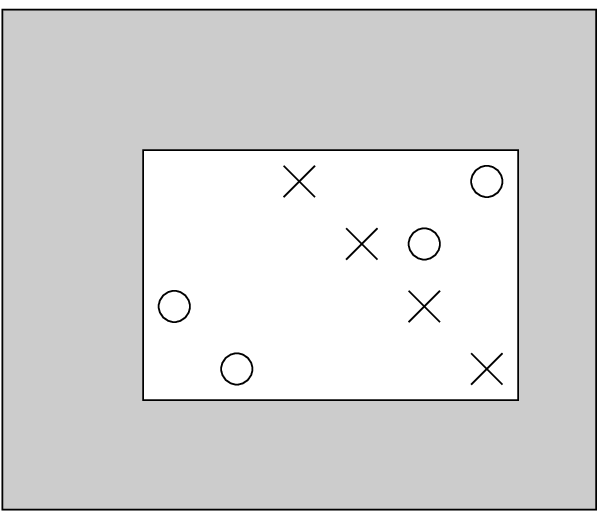} \ar[r] &  \dessin{2cm}{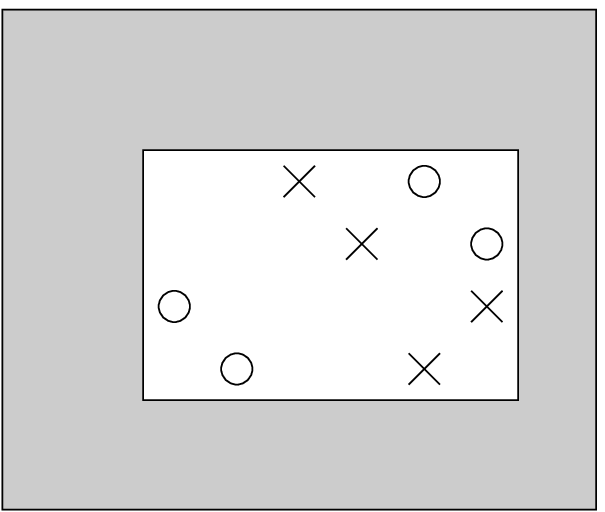} \ar[r] &  \dessin{2cm}{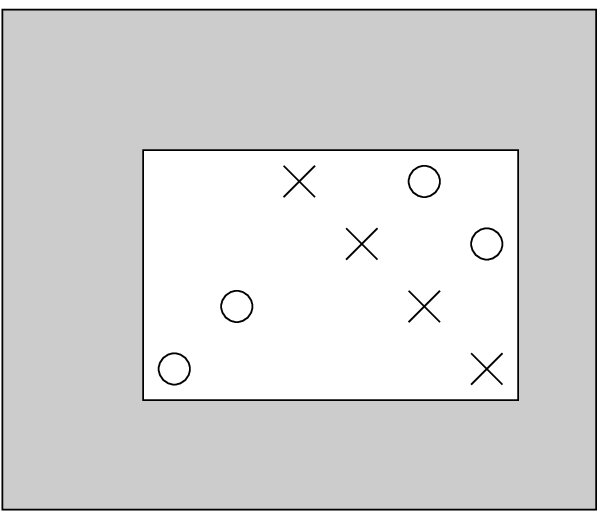}\ar@{-->}[d] \\
  \dessin{2cm}{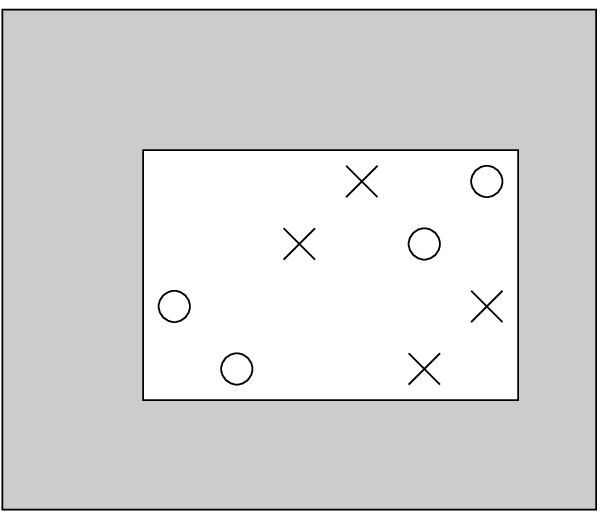} \ar[r] &  \dessin{2cm}{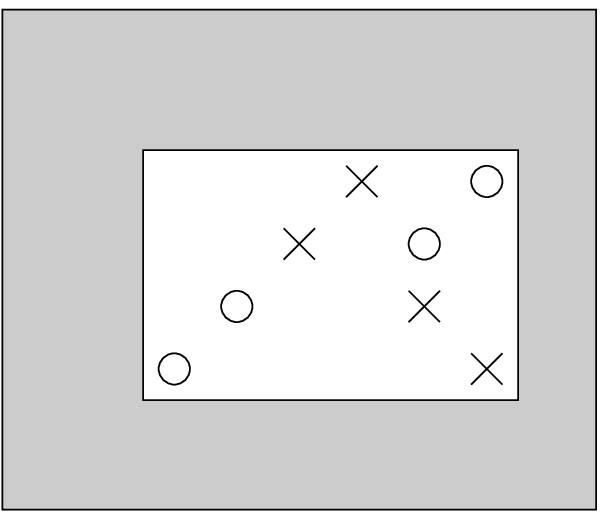} \ar[r] &  \dessin{2cm}{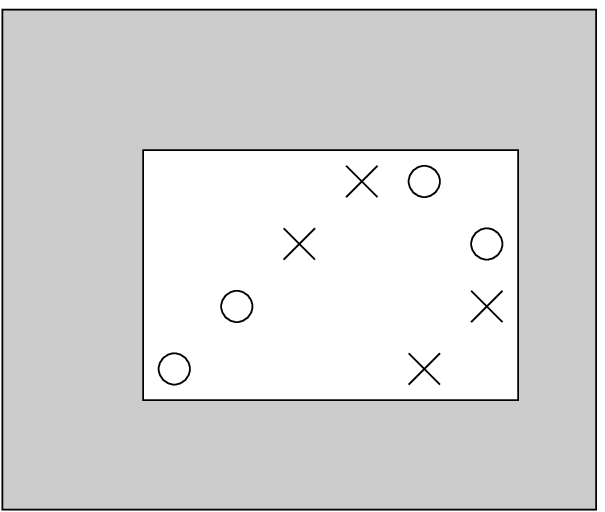} \ar[r] &  \dessin{2cm}{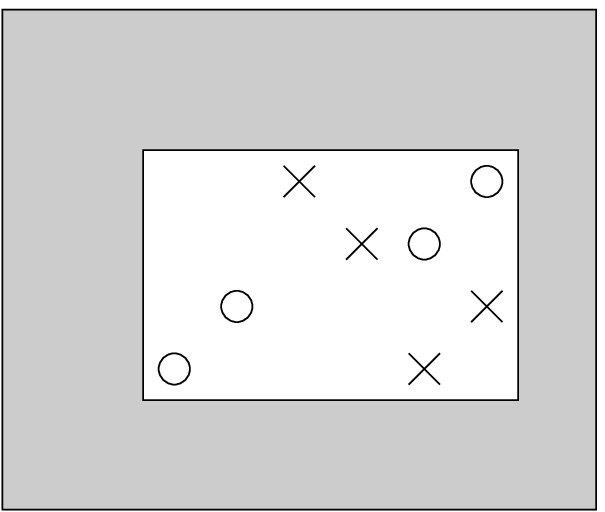}
}.
$$
For each of them, we have already defined a quasi-isomorphism. Now it remains to compose them to accomplish the step II.

%%% Local Variables: 
%%% mode: latex
%%% TeX-master: "These"
%%% End: 

% Transition
\vspace{2cm}
The singular definitions are now given. In the next chapter, we will discuss them and prove a few properties.

%%% Local Variables: 
%%% mode: latex
%%% TeX-master: "These"
%%% End: 

% Applications
\chapter[Discussion on singular link Floer homology]{Discussion on singular\\ link Floer homology}
\label{chap:Applications}

% Variantes
\section{Variants}
\label{sec:Variants}

\subsection{Graded homologies}
\label{sec:GradedHomologies}

\subsubsection{Graded singular homologies}
\label{sec:GradSingHomologies}

As in the regular case, $C^-(G)$ is associated with an Alexander filtration and $O$--filtration for all $O\in\O$. And as in the reguler case, the associated graded objects are the interesting ones.

\begin{theo}
  All the graded homologies derivated from the Alexander filtration and/or $O$--filtration for some $O\in\O$ on $C^-(G)$ are invariant under the elementary singular grid moves except the stabilization which may tensorize the homology by the module $V$ defined in paragraph \ref{par:LinkInv}\footnote[1]{This phenomenon occurs if and only if the two variables $U_1$ and $U_2$ involved in the point ii) in paragraph \ref{par:RegStab} are send to zero.}.
\end{theo}

Since the rules given in paragraph \ref{par:FloerHomologies} still hold, it is straighforward to check that the previous proofs given for the filtrated case in paragraphs \ref{par:ArcsIsotopies}---\ref{par:FlipInvariance} remain valid as soon as we restrict ourself to the polygons which does not contains the decorations which are forbidden\footnote[2]{depending on the filtrations we are turning into gradings}.\\

Since the singular link Floer homology can be seen as an extension of link Floer homology to singular links, we keep the notation given in the table of paragraph \ref{par:TableHomologies}.

\subsubsection{Relation between graded singular homologies}
\label{par:RelationGraded}

In the regular case, Proposition \ref
{UiUj} (\S\ref{par:LinkInv}) gives a strong relation between  $\widehat{HL}_*(G)$ and  $\widetilde{HL}_*(G)$. This relation does still hold in the singular case.

\begin{prop}\label{Dunealautre}
  Let $G$ be a singular grid $G$ of size $(n,k)\in \N^*\times\N$ which is a presentation for a link with $\ell\in\N^*$ components. Then
$$
\widetilde{HL}_*(G)\equiv \widehat{HL}_*\otimes V^{\otimes (n-\ell)}
$$
where $V$ is defined in paragraph \ref{par:LinkInv}.
\end{prop}

Actually, the proof given in the regular case can be transposed to the singular one without any trouble. Only the special cases represented in Figure \ref{fig:Cancel} could potentially be worrying, but it can be easily seen if a given decoration does simultaneously intervenes or not in two such terms.

\begin{defi}
  Let $L$ be a singular link and $G$ a grid presentation for it. The homology groups $\widehat{HL}_*(G)$ are called \emph{singular link Floer homology for $L$}\index{homology!knot Floer!singular}\index{homology!link Floer!singular} and are denoted $\widehat{HL}_*(L)$.
\end{defi}

\subsection{Another choice for peaks}
\label{sec:AnotherChoice}

\subsubsection{Another singular link Floer homology}
\label{par:H'}

All along the previous chapter, we have used functions of type $\dessin{.4cm}{Cstyle}$ as key ingredients for constructing cubes of resolution and then defining homologies. But we have made the choice indeed arbitrary to consider the arcs intersections of type $\dessin{.4cm}{Cstyle}$, as kings of the dance. As pointed out in the paragraph \ref{ssec:Variants}, there is another consistent choice.
$$
\dessin{3.65cm}{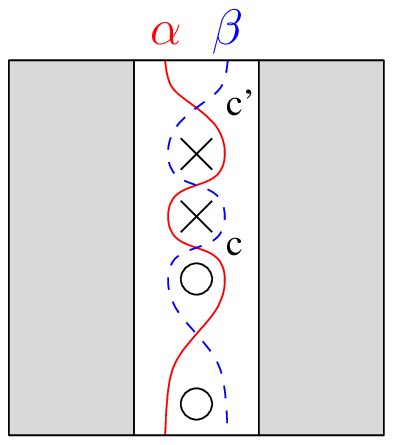}
$$
Actually, it is easy to check that we observe the same behavior when considering arcs intersections of type $\dessin{.4cm}{Cpstyle}$ instead.\\

This defines a second singular link Floer homology which is also an invariant of the singular link. The only point which may deserve discussion is the proof of invariance under commutation moves involving singular columns. But using the trick depicted in Figure \ref{fig:PermutingDecorations}, one can exchange the location of $\X$ and $\O$--decorations in any singular column. Then, the same proof remains valid without the slightest modification.\\

We denote this second homology by $H^-_{\dessin{.2cm}{Cpstyle}}$.

\subsection{Columns against rows}
\label{ssec:ColumnsVsRows}

\subsubsection{Flipping grids}
\label{par:FlippingGrids}

Another arbitrary choice is to consider grids with singular columns instead of singular rows. However, reflecting a grid along the line parametrized by $y=x$ swaps singular rows and singular columns. Moreover, it defines a bijection between generators which preserves Maslov and Alexander gradings. This map clearly commutes with differentials. The induced sign refinement is not exactly a sign assigment as defined in paragraph \ref{par:SignAssignment}, since it sends vertical annuli to $1$ and horizontal ones to $-1$, but replacing $U_i$ by $-U_i$ for all $i\in\llbracket 1,n\rrbracket$ corrects this.\\

Concerning the associated link, reflecting a grid only reverses its orientation.\\

In conclusion, defining a singular link Floer homology using singular rows gives the same result as composing the singular link Floer homology developped in this thesis with the operation which reverses the orientation.

\begin{remarque}
  The construction given in this thesis can be extended to mixed singular grids. It reduces the number of elementary moves and simplifies the proofs of invariance. Details are written in Appendix \ref{appendix:MixedSimplification}. Moreover, it clarifies the relation between the homologies of a singular link and of its orientation reversed image. It clarifies as well the relation between the homologies of a link and of its mirror image.
\end{remarque}

%%% Local Variables: 
%%% mode: latex
%%% TeX-master: "These"
%%% End: 

% Inverse pour f
\section{Almost an inverse}
\label{sec:AlmostInverse}

Let $G$ be a grid with a distinguished singular column $S$.\\

As pointed out in paragraph \ref{par:HomologyAsCone}, the complex $C^-(G)$ can be understood as a mapping cone. The underlying map $f_S$ is defined as the sum over pentagons with at least one peak in S. Naturally, $f_S$ is not, in general, a quasi-isomorphism. However, it is not far from it.\\

We denote by $G^+$ (resp. $G^-$) the grid obtained from $G$ by resolving $S$ positively (resp. negatively) and by $\alpha_S$ and $\beta_S$ the two arcs chosen in $S$ to define $C^-(G)$.

\subsubsection{More and more peaks}
\label{par:MorePeaks}

By construction, all the polygons involved in $f_S$ have their peak on $S$, located between an $X$ and a $O$. We can define \emph{empty inversed polygons}\index{polygon!inversed} by modifying, in the definition, the point stating the location of the peak on $S$ and asking that it lies on the intersection of $\alpha_S$ and $\beta_S$, which is located between the two $\O$--decorations.\\
We denote by $\Pol^{-1}_S(x,y)$ the set of such empty polygons which connect a generator $x$ to a generator $y$.\\
\begin{figure}[h]
$$
\dessin{3cm}{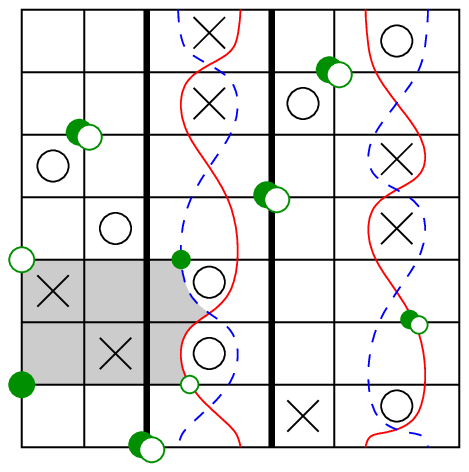} \hspace{2cm} \dessin{3cm}{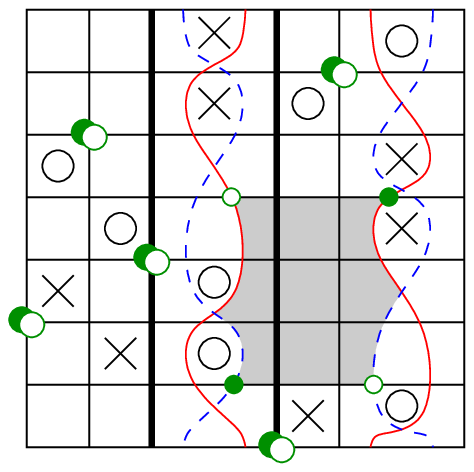}
$$
    \caption{Examples of empty inversed polygons: {\footnotesize dark dots describe the generator $x$ while hollow ones describe $y$. Polygons are depicted by shading. The distinguished singular column $S$ is depicted by wider vertical lines.}}
\end{figure}

Concerning the gradings, such inversed polygons behave slightly differently than the usual ones.

\begin{prop}\label{InversedGrading}
  Let $x$ and $y$ be respectively generators of $C^-(G^-)$ and $C^-(G^+)$.\\
If $\pi$ is an inversed polygon connecting $x$ to $y$, then
\begin{gather*}
  M_{G^-}(x)-M_{G^+}(y)= \eta(\pi)-2\#(\pi\cap\O)\\[.3cm]
  A_{G^-}(x)-A_{G^+}(y)= -1+\#(\pi\cap\X)-\#(\pi\cap\O),
\end{gather*}
where $\eta(\pi)$ is equal to $0$ for pentagons and $-1$ for hexagons.
\end{prop}

The proof is similar than for usual polygons.

\subsubsection{A candidate for an inverse...}
\label{par:Candidate}

Now, we can set $\func{g_S}{C^-(G^-)}{C^-(G^+)}$ the morphism of $\Z[U_{O_1},\cdots,U_{O_n}]$--modules defined on the generators by
$$
g_S(x)=\sum_{\substack{y \textrm{ generator}\textcolor{white}{R}\\\textrm{of }C^-(G^+)\textcolor{white}{I}}} \sum_{\pi\in \Pol^{-1}_S(x,y)} \e(\pi)U_{O_1}^{O_1(\pi)}\cdots U_{O_n}^{O_n(\pi)}\cdot y,
$$
where the sign $ \e(\pi)$ is defined as usual.

\begin{prop}
  The map $g_S$ is a chain map. Moreover $f_S\circ g_S$ and $g_S\circ f_S$ are homotopic to the identity map.
\end{prop}
\begin{proof}
  The proof that $g_S$ is a chain map is identical to the proofs of similar statements made earlier in this thesis.\\

To prove the homotopy equivalence, new kinds of polygons have to be introduced. The definitions are the same than for $\h$exagon and $\h$eptagons (\S\ref{par:NewPolygons}) except the distinguished edge $e$ which is required to be a piece of $\beta_S$ of which extremities are the two possible peaks for usual and inversed polygons. Moreover, the interior of such polygons is required to not intersect $\beta_S$ in a neigborhood of this two peaks.\\
We denote by $\PPol^\beta_S(x,y)$ the set of such empty polygons connecting a generator $x$ to another generator $y$.\\

Then, the $\Z[U_{O_1},\cdots,U_{O_n}]$--linear map $\func{H_\beta}{C^-(G^+)}{C^-(G^+)}$ is defined on generators by
$$
H_\beta(x)=\sum_{\substack{y \textrm{ generator}\textcolor{white}{R}\\\textrm{of }C^-(G^+)\textcolor{white}{I}}} \sum_{\pi\in \PPol^{-1}_S(x,y)} \e(\pi)U_{O_1}^{O_1(\pi)}\cdots U_{O_n}^{O_n(\pi)}\cdot y,
$$
where $\e(\pi)$ is defined as usual.\\
Now, we claim that
$$
g_S\circ f_S + \Id_{C^-(G^+)} \p^-_{G^+}\circ H_\beta + H_\beta \circ \p^-_{G^+} \equiv 0. 
$$
Here again, the cancelling pairs involved in the proof are essentially identical to those of the previous proofs. Nevertheless, Figure \ref{fig:Homotopy} points out the cases specific to this relation. The identity map cancels with vertical annuli along the right or the left border of $S$. The side is determined by the vertical location of the dot lying in $S$ with repect to the two possible peaks in $S$. Finally, being a vertical annulus brings a minus sign.\\

\begin{figure}[p]
  \begin{center}
    \subfigure[]{\fbox{\xymatrix@R=.6cm{
          \dessin{3.6cm}{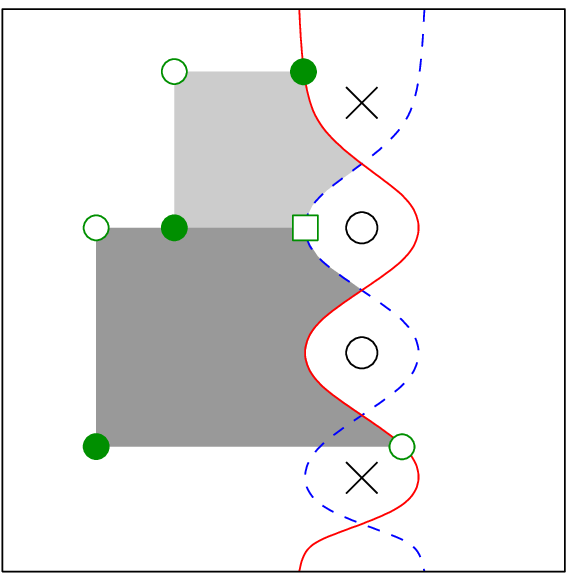} \ar@{^<-_>}[d]_(.45){g_S\circ f_S\ }^(.53){\ \p_{G^+}^-\circ H_\beta} \\  \dessin{3.6cm}{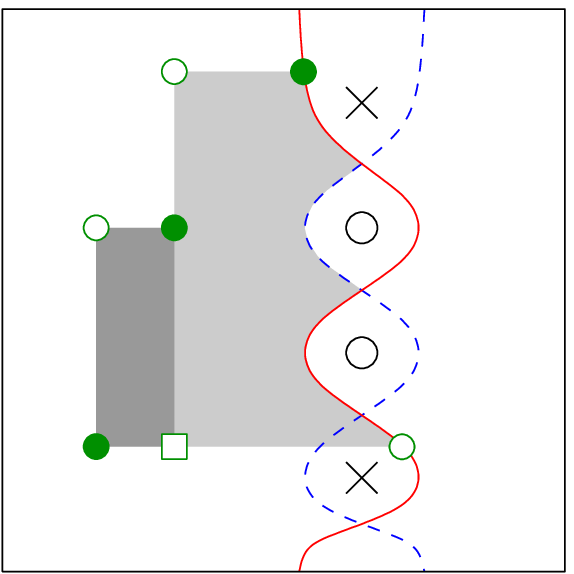}
        }}} \hspace{1cm}
    \subfigure[]{\fbox{\xymatrix@R=.6cm{
          \dessin{3.6cm}{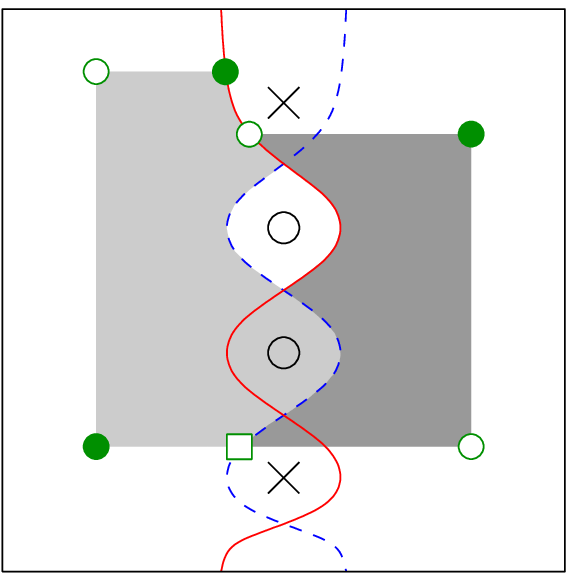} \ar@{^<-_>}[d]_(.45){g_S\circ f_S\ }^(.53){\ \p_{G^+}^-\circ H_\beta} \\  \dessin{3.6cm}{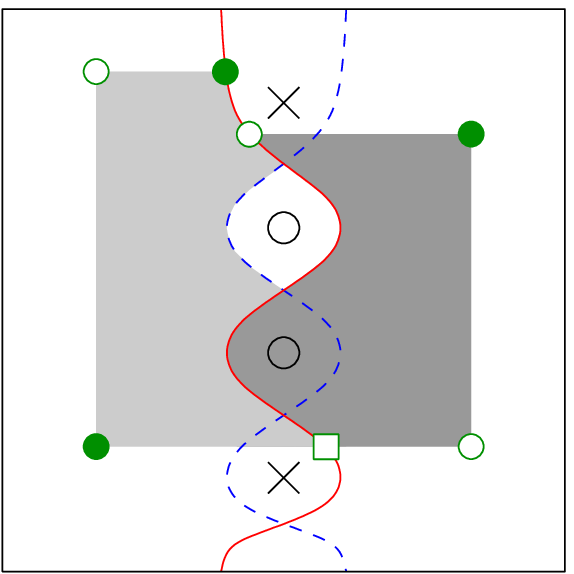}
        }}}
    \\
    \subfigure[]{\fbox{\xymatrix@!0@R=2cm@C=3.35cm{
          \dessin{3.6cm}{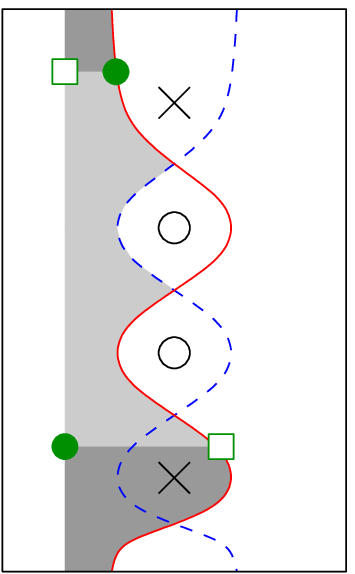} &&  \dessin{3.6cm}{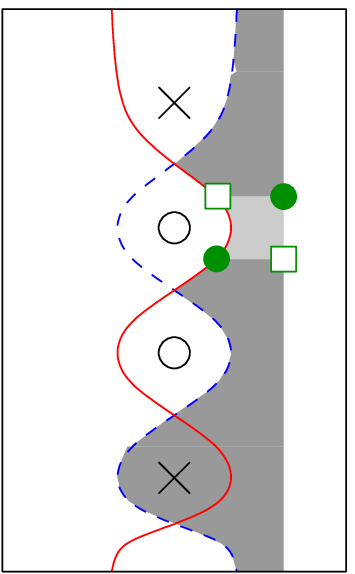}\\
          & Id \ar@{_<-^>}[lu]_{\p_{G^+}^-\circ H_\beta} \ar@{^<-_>}[ru]^{H_\beta\circ\p^-_{G^+}} \ar@{^<-_>}[ld]^{g_S\circ f_S} \ar@{_<-^>}[rd]_{g_S\circ f_S}&\\
          \dessin{3.6cm}{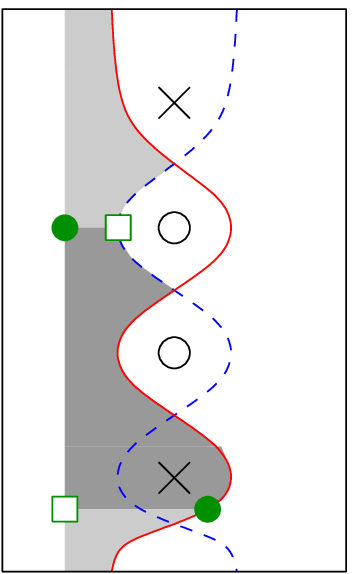} &&  \dessin{3.6cm}{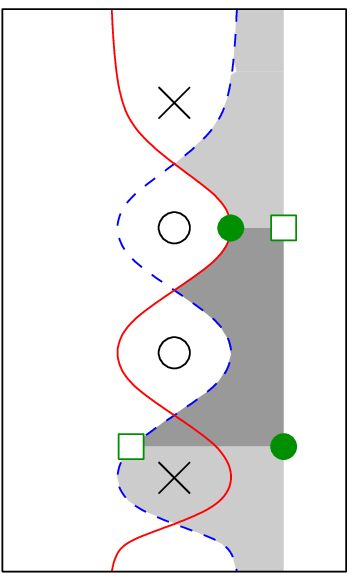}
        }}}
  \end{center}
  \caption{Cancelling pairs: {\footnotesize Dark dots describe the initial generator while hollow ones describe the final one if different. Squares describe intermediate states. Polygons are depicted by shading. The lightest is the first to occur whereas the darkest one is the last. For each configuration, we indicate to which part of the relation it belongs.}}
  \label{fig:Homotopy}
\end{figure}

A similar work can be done for $f_S\circ g_S$ by replacing $\beta_S$ with $\alpha_S$.
\end{proof}

\subsubsection{...which is not filtrated}
\label{par:NoAlexander}

Up to homotopy, $f_S$ has thus an inverse $g_S$. Nevertheless, as a corollary of Proposition \ref{InversedGrading} (\S\ref{par:MorePeaks}), the chain map $g_S$ is not filtrated since it can rise the Alexander degree by $1$. But it means that it is filtrated of degree $1$.\\

By the way, if we forget everything about the Alexander filtration, we have proved the following proposition:
\begin{prop}
  For any link $L$, the unfiltrated homology $H^-(L)$ is invariant under the switch operation. If $L$ is a regular link, then it depends only on the number of components. 
\end{prop}
This property was already well-known for regular links since the unfiltrated link Floer homology is nothing more than the multi-pointed Heegaard-Floer homology of the sphere $S^3$.

\subsubsection{Another candidate}
\label{par:AnotherCandidate}

A similar work can be done using the intersection of $\alpha_S$ and $\beta_S$ which is located between the two $\X$--decorations of $S$ instead of the two $\O$--ones.\\
The map $g_S$ so defined  preserves then the Alexander filtration and decreases the Maslov grading by two. Nevertheless, the compositions $g_S\circ f_S$ and $f_S\circ g_S$ are not homotopic to the identity map but to the multiplication by $U_O$ where $O$ is an $\O$--decoration from $S$.

%%% Local Variables: 
%%% mode: latex
%%% TeX-master: "These"
%%% End: 

% Acyclicité pour les boucles
\section{Acyclicity for singular loop}
\label{sec:Acyclicity}

A link $L$ being given, the simplest way to singularize it, is to add a contractible singular loop:
$$
\dessin{2cm}{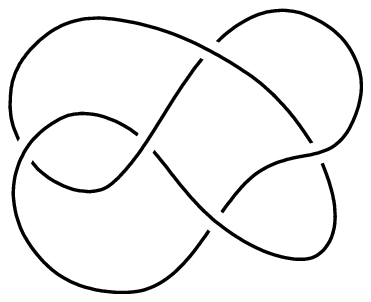} \ \leadsto \ \ \dessin{2cm}{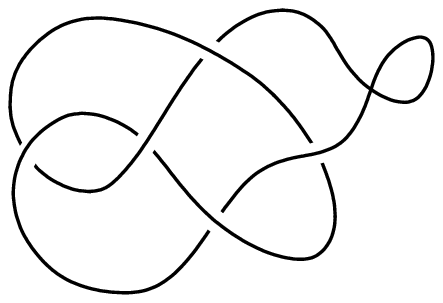}.
$$
Both resolutions lead then to the same initial link $L$. The least we can expect is that the link Floer homology for such a singular link is null. That is the point of the next theorem.
\begin{theo}\label{Loops}
  For any link $L$, possibly singular,
$$
 {\widetilde{HL}}\left(L{}_*\#_p \dessin{.4cm}{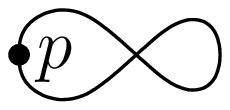}\right) = {\widetilde{HL}}\left(\ \dessin{1cm}{SmallLoop}\right)\cong 0.
$$
\end{theo}

This section is devoted to the proof of this statement.\\

Since $\widetilde{HL}$ is obtained from any other filtrated version as the graded homology associated to some filtration, the same statement holds for all of them.

\subsubsection{Representation of a small loop}
\label{par:RepresentationLoop}

Let $L$ be a link with $k$ singular double points and $G$ a grid presentation for it.\\
Adding a singular loop to $L$ can be seen as replacing a regular column of $G$ by the following adjacent two:
$$
\dessin{2.1cm}{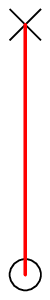} \ \leadsto \ \ \dessin{2.1cm}{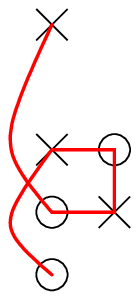}.
$$
Moreover, even if it means to perform a stabilization first, we can ask the four decorations of the singular column to be in adjacent rows.\\

We suppose we have now such a grid presentation $G_s$ for a link with a small singular loop. We denote by $S$ the singular column added with the loop and by $\alpha_S$ and $\beta_S$ the arcs lying in $S$.

\subsubsection{Another crushing filtration}
\label{par:AnotherCrushingFiltration}

Here, we can play the same game as in paragraph \ref{par:CrushingFiltration} and crush column and rows which have appeared when adding a loop.\\
If $x$ is a generator of $C^-(G_s)$ drawn on $G_s$, then, during the crushing process, we push the dot which belongs to $\alpha_S$ or $\beta_S$ to the first vertical grid line on its left. Thus, $x$ gives rise to a set of dots $\widetilde{x}$ on $G$ which is almost a generator of $C^-(G)$ except that exactly one horizontal and two adjacent vertical grid lines have more than one dots. Then we permute cyclically the rows and columns in such a way that the singular horizontal grid line is the bottommost and the two vertical ones the leftmost and the second rightmost. Now, the upper right corner of the grid is filled with a $O$, denoted by $O_*$, and the bottom right one by an $X$, denoted $X_*$.\\
Then, we define for all generator $x$ of $C^-(G_s)$:
$$
M_G(x):= M_{\O_G}(\widetilde{x})+ \#\big(O(I)\}\setminus\{i_S\}\big)
$$
where $\O_G$ is the set of $O$--decorations of $G$, $M_{\O_G}$ is defined in paragraph \ref{par:MaslAlex}, $I$ is an element of $\{0,1\}^k$ such that $x\in C^-(G_I)$ and $i_S$ is the label of $S$.

\begin{figure}[!h]
  $$
    \dessinH{3.2cm}{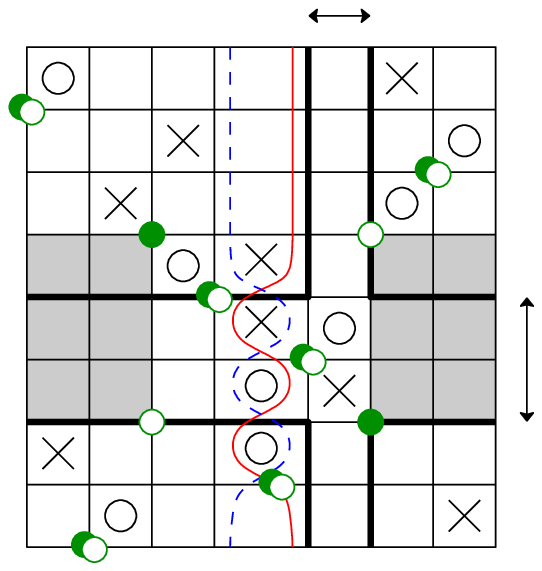}\ \begin{array}{c}\textrm{\scriptsize crushing}\\[-.1cm] \leadsto\\[-.1cm] \textrm{\scriptsize process} \end{array} \dessinH{3.2cm}{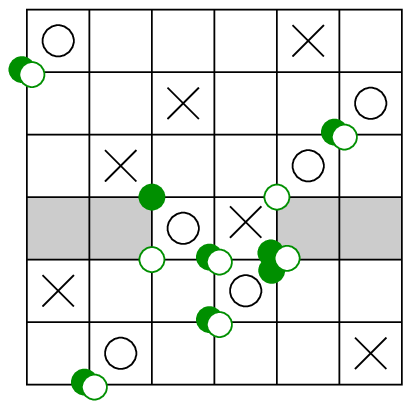} \begin{array}{c}\textrm{\scriptsize cyclic}\\[-.1cm] \leadsto\\[-.1cm] \textrm{\scriptsize perm.} \end{array} \dessinH{3.2cm}{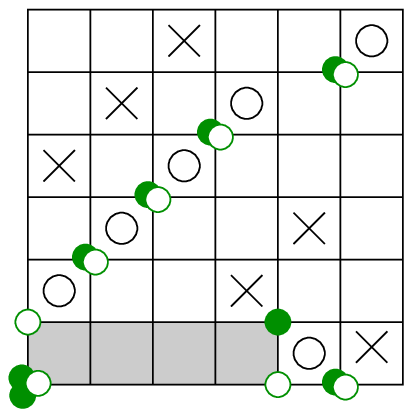}  \leadsto \ \ \begin{array}{l} M_G(x)=-6 \\[.3cm] M_G(y)=-8 \end{array}
  $$
  \caption{Crushing rows and column:  {\footnotesize dark dots describe the initial generator $x$ while hollow ones describe the final one $y$. Rectangles are depicted by shading.}}
  \label{fig:CrushingLoop}
\end{figure}

\begin{prop}\label{AcyclicDiff}
  $M_G$ defines a filtration on $\widetilde{CL}(G_s)$\index{filtration!crushing}. Moreover, the associated graded differential, denoted by $\widetilde{\p}_\gr$, corresponds to the sum over polygons contained in the crushed rows and column and which do not contain any decoration.
\end{prop}
\begin{proof}
  Arguments are essentially the same than in the proof of Proposition \ref{CrushingFilt} (\S\ref{par:CrushingFiltration}). Nevertheless, a few details are different.\\

An empty rectangle $\rho$ on $G_s$ containing no decoration gives rise to an empty rectangle $\widetilde{\rho}$ on $G$ but which may contain $O_*$ or $X_*$. The latter is not involved in $M_G$ but the first is.\\
However, the proof of Proposition \ref{CrushingFilt} reduces the reasoning to the study of non ripped rectangles \ie rectangles which do not intersect the leftmost column nor the uppermost row. Consequently, $O_*$ is not involved in the calculations. Moreover, such rectangles are not crossed by the singular grid lines and the number of dots they contain can be thus deduced as needed.\\

Because of the four crushed decorations, no rectangle can be ripped in four pieces.\\

Concerning polygons, we choose arcs $\alpha$'s and $\beta$'s such that their intersections do not belong to the crushed rows and column. This can always be done, except for $S$. Nevertheless, because of its two extremal decorations, any polygon $\pi$ with a peak on it and containing no decoration must lie entirely in the two crushed rows. If $\pi$ is connecting $x$ to $y$, then $\widetilde{x}=\widetilde{y}$ and since $S$ is not considered in the count of singular columns positively resolved, $M_G(x)=M_G(y)$.\\

The filtration induced by $M_G$ is then respected and the associated graded differential is as stated. 
\end{proof}

\subsubsection{Last filtration before exit}
\label{par:LastFiltrationExit}

In order to simplify the proof, we need to define a last filtration on $(\widetilde{CL}(G_s),\widetilde{\p}_\gr)$.\\

We denote by $\alpha'_S$ (resp. $\beta'_S$) the interior of the intersection of $\alpha_S$ (resp. $\beta_S$) with the complement of the two crushed rows. Then we can define a grading $\kappa$ for all generator $x$ of $\widetilde{CL}(G_s)$ by
$$
\kappa(x)=\#(x\cap\beta'_S) - \#(x\cap\alpha'_S).
$$

It is easy to check that the filtration associated to $\kappa$ is respected by $\widetilde{\p}_\gr$.
We denote by $\widetilde{d}_\gr$ the associated graded differential. It differs from $\widetilde{\p}_\gr$ only by forbidding the following two moves:
$$
\dessin{2.7cm}{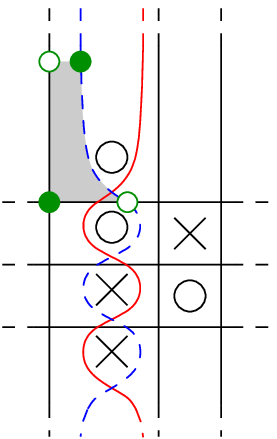} \hspace{2cm} \dessin{2.7cm}{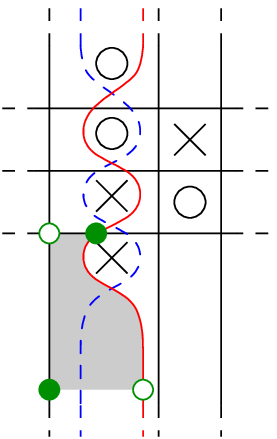}.
$$

\begin{prop}\label{AcyclicSubcomplexes}
  The chain complex $\big( \widetilde{CL}(G_s),\widetilde{d}_\gr\big)$ is acyclic.
\end{prop}

Actually, the graded chain complex $(\widetilde{CL}(G_s),\widetilde{d}_\gr)$ can be split into acyclic subcomplexes. Details are explicited in appendix \ref{appendix:Subcomplexes}.\\

Finally, Theorem \ref{Loops} can be deduced by applying Corollary \ref{Filt->Acycl} (\S\ref{par:SpecSequences}) twice.

%%% Local Variables: 
%%% mode: latex
%%% TeX-master: "These"
%%% End: 

% Description alternative
\section{What about the sum ?}
\label{sec:Sum}

As pointed out at the end of paragraph \ref{par:DifferentPeaks}, one can consider at the same time peaks of types $\dessin{.4cm}{Cstyle}$ and $\dessin{.4cm}{Cpstyle}$ for a given singular column. But then, the proofs given for invariance under semi-singular commutation and acyclicity for contractible singular loops are no more valid.\\

Actually, for the first one, the second peak adds terms in $\p^-_2\circ F_T$ and $F_B\circ \p^-_1$ which are refractory to cancellation. Nevertheless, even if the proof fails, we do not know any counterexample.\\

Concerning the acyclicity for small singular loops, polygons with a peak on the added singular column are no more necessarily contained in the two crushed rows. And furthermore, we will prove in this section that even if such an homology is well defined, it would be of less interest insofar as no singular link would be able to make it null.\\

To prove it, we give an alternative description of this homology for $\FF_2$--coefficients. This description is related to the Seifert smoothing of a singular double point instead of its positive and negative resolutions. It is inspired by the results proved in \cite{OSSkein}.\\

We consider its graded version associated to the Alexander filtration. We denote it by $(HL^-_\Sigma,\p^-_\Sigma)$.\\

As specified above, the following construction works for $\FF_2$--coefficients only. Nevertheless, we will still pay attention to signs in order to stress when they make the machinery crash.

\subsection{Three dimensional cube of maps}
\label{ssec:3dimCube}

\subsubsection{Multiple $\X$--decorations}
\label{par:MultiX}

{\parpic[r]{$\dessin{3.7cm}{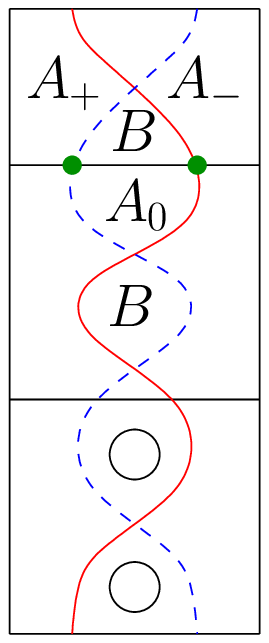}$}
  Let $G$ be a singular grid diagram with a distinguished singular column such that its two $\X$--decorations are in adjacent rows.\\
  Instead of considering, as previously, only the two natural resolutions, we consider now all the possible desingularizations of the distinguished singular column as depicted in Figure \ref{fig:Desing}. The picture on the right describes all of them depending on the vertical arc and the letters $A$ or $B$ we remove. Other letters are then considered as $\X$--decorations. According to the arc we are considering, $A_+$ or $A_-$ should be ignored.\\
  We assume that the two interesting arcs intersections are respectively in the same two rows as the two $B$'s.
  
}
After Figure \ref{fig:Desing}, it is easily checked that the $B$--decorations correspond to the two positive and negative resolutions of the associated double point whereas the $A$--decorations correspond to the Seifert smoothing.

\subsubsection{Subcomplexes}
\label{par:SubComplexes}

First, we state a lemma which will be usefull for defining subcomplexes.

\begin{lemme}
  Let $G$ be a (singular) grid with two $\X$--decorations in adjacent rows and adjacent (regular) columns. They are hence arranged in diagonal. Let $x$ be the grid line intersection located between them. Then the part of $\p^-_\Sigma$ which does not involve the point $x$ is still a differential.
\end{lemme}

{\parpic[l]{\hspace{.5cm}$\dessin{2.7cm}{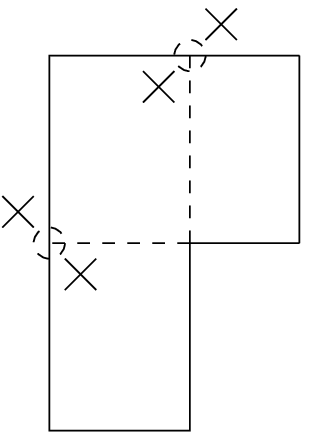}$}
  \vspace{.3cm}
  \textit{Proof.} Essentially, the non trivial part of the proof that $\p^2\equiv 0$ lies in the two decompositions into rectangles of a L--figure. A problem may occur if the dot $x$ is involved in one decomposition, but not in the other. But actually, the only two dots which behave in this way are contained in the interior of an edge of the L--figure. Hence, one of the two considered $\X$--decorations must lie inside the L--figure.\\
  
}
Since the two considered $\X$--decorations lie by assumption in regular columns, the same reasoning can be applied even if the grid is singular.
\hfill $\square$\\

Now, we can define several complexes since the previous lemma states that the part of the differentials which leaves the following subcomplexes stable is still a differential.\\
In this section, when talking about $\X$ or $\O$--decorations, we mean the four decorations represented in the pictures. By $x$, we mean the grid line intersection located between the two $\X$--decorations.

\begin{description}
  {\parpic[r]{$\dessin{1.8cm}{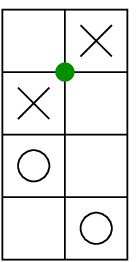}$\hspace{.5cm}}
    \vspace{.2cm}\item[$\textcolor{red}{X_A}$:] the subcomplex generated by the generators containing $x$, extracted from the complex associated to the grid obtained with the plain arc and the $A$--decorations, ;
    \vspace{.25cm}
    
  }
  {\parpic[l]{\hspace{-.5cm}$\dessin{1.8cm}{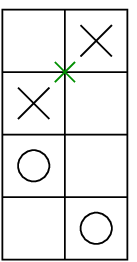}$\hspace{.7cm}}
    \vspace{.2cm}\item[$\textcolor{red}{Y_A}$:] the subcomplex generated by the generators which does not contain $x$, extracted from the complex associated to the grid obtained with the plain arc and the $A$--decorations, ;
    \vspace{.25cm}
    
  }
  {\parpic[r]{$\dessin{1.8cm}{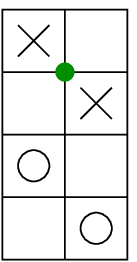}$\hspace{.5cm}}
    \vspace{.2cm}\item[$\textcolor{red}{X_B}$:] the subcomplex generated by the generators containing $x$, extracted from the complex associated to the grid obtained with the plain arc and the $B$--decorations, ;
    \vspace{.25cm}
    
  }
  {\parpic[l]{\hspace{-.5cm}$\dessin{1.8cm}{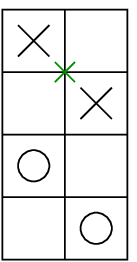}$\hspace{.7cm}}
    \vspace{.2cm}\item[$\textcolor{red}{Y_B}$:] the subcomplex generated by the generators which does not contain $x$, extracted from the complex associated to the grid obtained with the plain arc and the $B$--decorations, ;
    \vspace{.25cm}
    
  }
  {\parpic[r]{$\dessin{1.8cm}{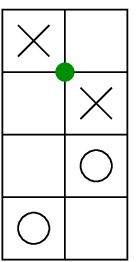}$\hspace{.5cm}}
    \vspace{.2cm}\item[$\textcolor{blue}{X'_A}$:] the subcomplex generated by the generators containing $x$, extracted from the complex associated to the grid obtained with the dashed arc and the $A$--decorations, ;
    \vspace{.25cm}
    
  }
  {\parpic[l]{\hspace{-.5cm}$\dessin{1.8cm}{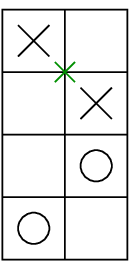}$\hspace{.7cm}}
    \vspace{.2cm}\item[$\textcolor{blue}{Y'_A}$:] the subcomplex generated by the generators which does not contain $x$, extracted from the complex associated to the grid obtained with the dashed arc and the $A$--decorations, ;
    \vspace{.25cm}
    
  }
  {\parpic[r]{$\dessin{1.8cm}{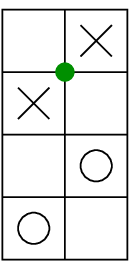}$\hspace{.5cm}}
    \vspace{.2cm}\item[$\textcolor{blue}{X'_B}$:] the subcomplex generated by the generators containing $x$, extracted from the complex associated to the grid obtained with the dashed arc and the $B$--decorations, ;
    \vspace{.25cm}
    
  }
  {\parpic[l]{\hspace{-.5cm}$\dessin{1.8cm}{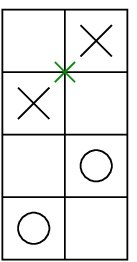}$\hspace{.7cm}}
    \vspace{.2cm}\item[$\textcolor{blue}{Y'_B}$:] the subcomplex generated by the generators which does not contain $x$, extracted from the complex associated to the grid obtained with the dashed arc and the $B$--decorations, .
    \vspace{.25cm}
    
  }
\end{description}

\subsubsection{Graded chain maps}
\label{par:SubChainMaps}

Now, we define chain maps between these complexes.

\begin{description}
\item[$\func{\p_A^x}{\textcolor{red}{Y_A}}{\textcolor{red}{X_A}}$:] the part of the differential which involves $x$. Since the whole differential \textbf{is} a differential, this map anti-commutes with the partial differentials;
\item[$\func{\p_A^x}{\textcolor{blue}{X'_A}}{\textcolor{blue}{Y'_A}}$:] idem;
\item[$\func{\p_B^x}{\textcolor{red}{X_B}}{\textcolor{red}{Y_B}}$:] idem;
\item[$\func{\p_B^x}{\textcolor{blue}{Y'_B}}{\textcolor{blue}{X'_B}}$:] idem;
\item[$\func{\Spike}{\textcolor{blue}{X'_A}}{\textcolor{red}{X_B}}$:] up to signs, the restriction to $\textcolor{blue}{X'_A}$ of the map $h$ defined in paragraph \ref{par:UpToHom} as the sum over spikes of which the support is the horizontal grid line containing $x$. Actually, it is only the isomorphism which moves the dot $x$ from the dashed arc to the plain one. It is clear that it commutes with the differentials except around the two $\O$--decorations. But since the differentials do not involve $x$, the polygons underlying the differentials cannot have an edge on the vertical grid line supporting $x$. Hence, if one of the two $\O$--decorations is inside such a polygon, then the empty square by its side is also contained in the polygon. This symetrizes the two differentials. Finally, as specified in paragraph \ref{par:Chain}, the map is made to anti-commute by adding a minus sign on odd homological degrees;
\item[$\func{\Spike}{\textcolor{red}{X_A}}{\textcolor{blue}{X'_B}}$:] idem;
\item[$\func{\Id}{\textcolor{red}{Y_A}}{\textcolor{red}{Y_B}}$:] up to sign, the identity map for sets of dots with no regard to the decorations. It commutes with the differentials for similar reasons than the map $\Spike$. In the same way, it is made to anti-commute;
\item[$\func{\Id}{\textcolor{blue}{Y'_A}}{\textcolor{blue}{Y'_B}}$:] idem;
\item[$\func{f}{\textcolor{red}{X_B}}{\textcolor{blue}{Y'_B}}$:] the map $f$ defined in paragraph \ref{par:HomologyAsCone} as the sum over pentagons and hexagons of which at least one peak belongs to the distinguished singular column. Since this peak is located in the same row as one of the two $\X$--decorations and since a polygon is not allowed to contain the other $\X$--decoration which is in the row beside, every such polygon does involve the dot $x$. Moreover, it has already been checked in paragraph \ref{par:Consistency} that it anti-commutes with the differentials;
\item[$\func{f}{\textcolor{red}{Y_B}}{\textcolor{blue}{X'_B}}$:] idem.
\end{description}

All the maps defined above are graded of some degree. This is clear for $\p_A^x$, $\p_B^x$ and $f$. For the map $\Spike$, it is a consequence of the proof of Corollary \ref{PentDegree} (\S\ref{par:SpiPent}). Finally, concerning the maps $\Id$, only the position of $x$ compared to the two $\X$--decorations changes when the two last ones are permuted. But, precisely, we are dealing with those generators which do not involve $x$.\\
For all of them, it is straightforward to check their degree.

\subsubsection{(Anti-)commutative diagram}
\label{par:CommDiag}

The following diagram summarizes all the previous chain complexes and graded chain maps. Every complex is shifted in such a way that the maps do preserve the gradings. We assume here that the Seifert smoothing has one component more as a link than the two positive and negative resolutions.

\begin{eqnarray}
  \vcenter{\hbox{\xymatrix@!0@C=3cm@R=2cm{
        \textcolor{blue}{X'_A}[-2]\{-1\} \ar[r]^(.55){\Spike}_(.55)\sim \ar[d]^{\p^x_A} & \textcolor{red}{X_B}[-1] \ar[r]^{\p^x_B} \ar[d]^{f} & \textcolor{red}{Y_B} \ar@{<-}[r]^{\textnormal{\Id}}_\sim \ar[d]^{f} & \textcolor{red}{Y_A} \ar[d]^{\p^x_A}\\
        \textcolor{blue}{Y'_A}[-1]\{-1\} \ar[r]^{\Id}_\sim &
        \textcolor{blue}{Y'_B}[-1] \ar[r]^{\p^x_B} &
        \textcolor{blue}{X'_B} \ar@{<-}[r]^{\Spike}_\sim &
        \textcolor{red}{X_A}[1] }}}
\end{eqnarray}

If the Seifert smoothing has one component less than the two other resolutions, then the Alexander grading of the four extremal complexes must be shifted by one.

\begin{lemme}
  The diagram above is commutative for $\FF_2$--coefficients.
\end{lemme}
\begin{proof}
  The mapping cones $\Cone(\func{\p_B^x}{\textcolor{red}{X_B}}{\textcolor{red}{Y_B}})$ and $\Cone(\func{\p_B^x}{\textcolor{blue}{Y'_B}}{\textcolor{blue}{X'_B}})$ are respectively the complexes associated to the positive and to the negative resolutions of the distinguished singular column. Then the anti-commutativity of the middle square is nothing more than arguing that $f$ is a chain map. This is implicit  in the proof of Proposition \ref{Cchain} (\S\ref{par:Consistency}).\\
  
  Now, we focus on the left square. As pointed out in the previous paragraph, any pentagon (resp. hexagon) involved in the map $f$ has $x$ as a corner. Moreover, when adding any spike involved in $h$, it results in a rectangle (resp. pentagon) which contains $B$, but no more $A$. It is thus a polygon involved in $\p_A^x$.
  $$
  \dessin{3cm}{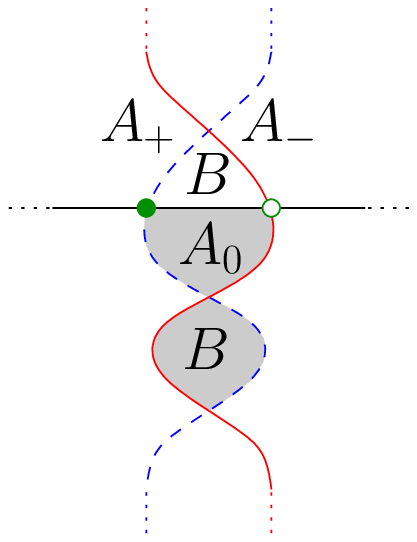}\ +\ \dessin{3cm}{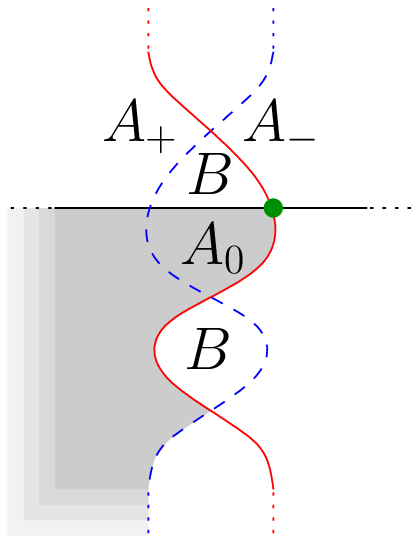}\ =\ \dessin{3cm}{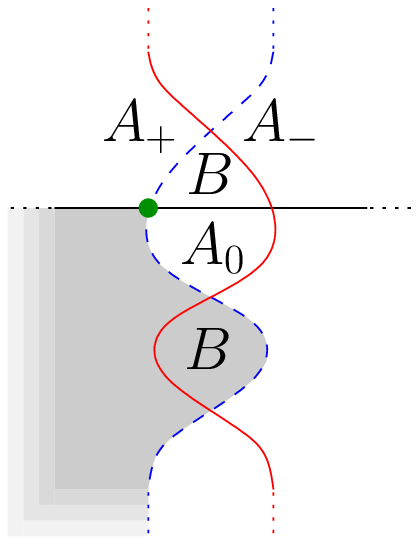}
  $$
  
  Conversely, when removing a spike, any polygon involved in $\p_A^x$ gives a polygon involved in $f$ .
  
  $$
  \dessin{3cm}{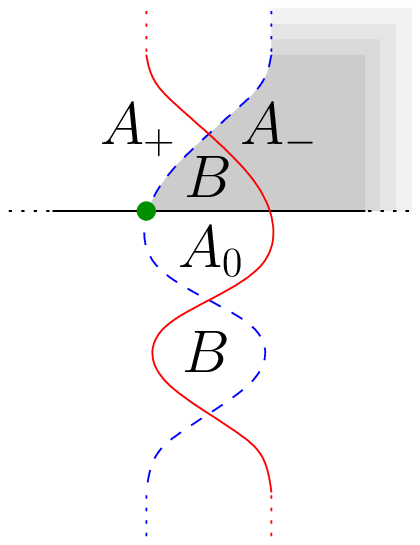}\ -\  \dessin{3cm}{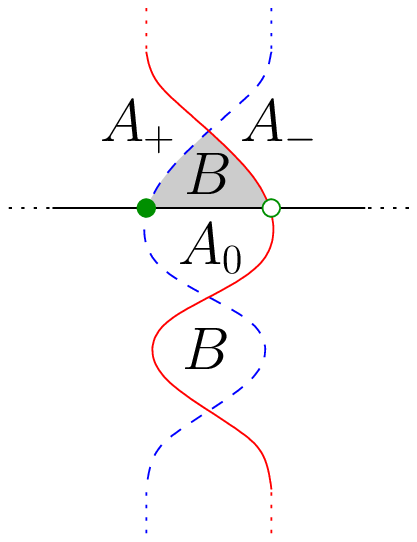}\ =\ \dessin{3cm}{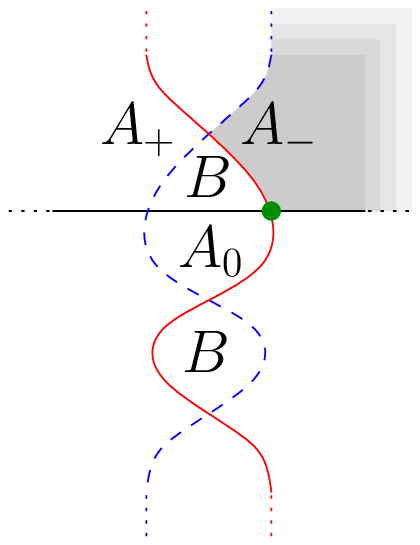}
  $$
  Up to signs, this proves the commutativity of the square. Now, the fact that $\p_A^x$ shifts the Maslov degree by one whereas $f$ preserves it, should make the square anti-commutes since the maps $\Spike$ and $\Id$ add a minus sign on odd homological degrees. However, $f$ also adds a minus sign for right pentagons when $\p_A^x$ does not. This is why we consider only $\FF_2$--coefficients.\\

  Similar considerations apply to the square on the right.
\end{proof}

\subsubsection{Cube of maps}
\label{par:3DCube}

Completing the above diagram with two maps
$$
\func{g:=\Id^{-1}\circ\p_B^x\circ\Spike}{\textcolor{blue}{X'_A}}{\textcolor{red}{Y_A}}
$$

$$
\func{g:=\Spike^{-1}\circ\p_B^x\circ\Id}{\textcolor{blue}{Y'_A}}{\textcolor{red}{X_A}},
$$
we get a three dimensional cube of maps:
$$
\xymatrix@!0@C=1.85cm@R=1.85cm{
  \textcolor{blue}{X'_A}[-2]\{-1\} \ar[rr]^(.55)g \ar[rd]^{\Spike}_\sim \ar[dd]^{\p_A^x} && \textcolor{red}{Y_A} \ar[dd]^(.3){\p_A^x}|!{[ld];[rd]}\hole \ar[rd]^{\Id}_\sim &\\
  & \textcolor{red}{X_B}[-1] \ar[rr]^(.35){\p_B^x} \ar[dd]^(.3)f && \textcolor{red}{Y_B} \ar[dd]^f\\
  \textcolor{blue}{Y'_A}[-1]\{-1\} \ar[rr]^(.65)g|!{[ru];[rd]}\hole \ar[rd]^(.45){\Id}_(.45)\sim && \textcolor{red}{X_A}[1] \ar[rd]^(.4){\Spike}_(.4)\sim &\\
  & \textcolor{blue}{Y'_B}[-1] \ar[rr]^{\p_B^x} && \textcolor{blue}{X'_B}.}
$$

The new faces are still commutative. This holds by definition for two of them and by composition of three commutative faces for the third one.

\subsection{Acyclicity and consequences}
\label{ssec:Acyclicity}

Let $L$ be the singular link associated to the singular grid $G$ considered in the previous section. Let $p$ be the distinguished double point of $L$.

\subsubsection{Notation for edges and faces}
\label{par:EdgesNotation}

By \emph{top}, \emph{bottom}, \emph{left}, \emph{right}, \emph{back} or \emph{front face}, we mean the obvious corresponding faces of the cube of maps.\\

A specific edge is now determined as the intersection of two faces. For instance, the front left and left front edge both denote the edge $\func{f}{\textcolor{red}{X_B[-1]}}{\textcolor{blue}{Y'_B[-1]}}$.

\subsubsection{Recognizing cones}
\label{par:RecognizingCones}

As corollaries of the definitions given in paragraphs \ref{par:SubComplexes} and \ref{par:SubChainMaps}, we can state the following lemmata:
\begin{lemme}
  The mapping cone of the front top edge is a link Floer complex for the positive resolution of $p$.
\end{lemme}
\begin{lemme}
  The mapping cone of the front bottom edge is a link Floer complex for the negative resolution of $p$.
\end{lemme}
\begin{lemme}
  The mapping cones of the back left and back right edges are distinct link Floer complexes for the Seifert smoothing of $p$.
\end{lemme}

And, finally:
\begin{lemme}
  The generalized cone of the front face is a singular link Floer complex for $L$. 
\end{lemme}

\subsubsection{Singular link Floer homology as a map between Seifert smoothing}
\label{par:SeifertSmoothing}

Moreover, since the top left, top right, bottom left and bottom right edges are all isomorphisms, the generalized cone of the whole cube of maps can be seen as the generalized cone of a square with acyclic complexes at its corners, it is then acyclic itself. It means the homologies of the back and the front faces are isomorphic. This proves the following theorem:

\begin{theo}\label{SeifertSmoothing}
  Let $L$ a singular link with $p$ as a double point. Let $L'$ be the link obtained by smoothing $p$ with respect to the orientation. If $L'$ has one component more than $L$, then, it exists a graded chain map
$$
\func{g}{C[-2]\{-1\}}{C'}
$$
where $C$ and $C'$ are link Floer complexes for $L'$ such that
$$
{HL}^-_\Sigma(G)\cong \Cone(g)
$$
for $\FF_2$--coefficients. If $L'$ has one component less than $L$, then the same statement is true but for a graded chain map
$$
\func{g}{C[-2]}{C'\{1\}}.
$$
\end{theo}

With any result of acyclicity for some singular link, this theorem would be very powerful since it would carry acyclicity when merging any two points of an acyclic link. But in this story, the murderer is the victim:
\begin{cor}
  Total singular link Floer homology $HL^-_\Sigma$ with $\FF_2$--coefficients never vanish.
\end{cor}
\begin{proof}
Suppose that $L$ is a singular link such that $HL^-_\Sigma(L)$ is null. Let $L'$ be the Seifert smoothing of one of the double point of $L$. According to Theorem \ref{SeifertSmoothing}, there is a graded chain map $g$ between shifted link Floer chain complexes for $L'$ such that ${HL}^-_\Sigma(G)\cong \Cone(g)$.\\
Now, if  $HL^-_\Sigma(L')$ is non zero, any element of lowest degree must be in the kernel of $g_*$ since $g$ is graded. But according to Corollary \ref{AcyclicCone} (\S\ref{par:MapCones}), the map $g_*$ is an isomorphism. Then $HL^-_\Sigma(L')\cong 0$.\\

We can iterate this process and find a regular link with null link Floer homology, but such a link does not exist.
\end{proof}

By the way, this theorem gives another proof that link Floer homology categorifies the Alexander polynomial multiplied by $t^\frac{1-\ell}{2}$ where $\ell$ is the number of components.

%%% Local Variables: 
%%% mode: latex
%%% TeX-master: "These"
%%% End: 

%%% Local Variables: 
%%% mode: latex
%%% TeX-master: "These"
%%% End: 

% 2nde partie : Homologie de Khovanov pour les noeuds restreints
\part[Khovanov homology for restricted links]{Khovanov homology\\ for restricted links}

% Invariants de Jones pour entrelacs restreints
\chapter{Jones invariants for restricted links}
\label{chap:Jones}

\section{Restricted links}
\label{sec:RestrictedLinks}

\subsection{Reidemeister moves}
\label{ssec:ReidMoves}

\subsubsection{Distinction between Reidemeister moves}
\label{par:DistinctionReidMoves}

Links in $3$--space are usually given by diagrams and isotopies of links by sequences of Reidemeister moves. It is commonly admitted that three different kinds of Reidemeister moves exist. Nevertheless, if one pays attention to orientations, each of them can be subdivided.\\
Actually, for relatively oriented links \ie up to global orientation reversing, there are exactly two different local Reidemeister moves of type I:
$$
\dessin{1.6cm}{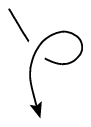} \ \longleftrightarrow \ \dessin{1.6cm}{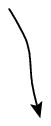} \ \longleftrightarrow \ \dessin{1.6cm}{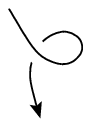},
$$
two of type II which correspond, with respect to the underlying planar curve, to the following singularities:
$$
\dessin{.8cm}{sing1} \hspace{2cm} \dessin{.8cm}{sing1bis},
$$
and eight of type III which are represented in Figure \ref{fig:TypeIII}. Six of them correspond to the singularity
$$
\dessin{1.2cm}{Sing2bis}
$$
whereas the two extremal ones correspond to
$$
\dessin{1.2cm}{sing2}.
$$

\subsubsection{Relation between Reidemeister moves}
\label{par:relationReidMoves}

\begin{figure}
$$
  \vcenter{\hbox{\xymatrix@!0@C=2.6cm@R=2.6cm{
        \dessin{1.6cm}{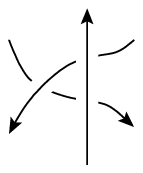} \ar@{<->}[d] \ar@{<=>}[r]|{}="Nya2" & \dessin{1.6cm}{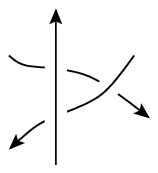} \ar@{<->}[d]\\
        **[l]\dessin{1.6cm}{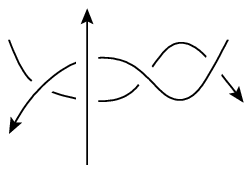} \ar@{<=>}[r]|{}="Nya1"  & **[r]\dessin{1.6cm}{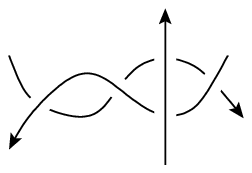}
        \ar@{.>}"Nya1"!<0cm,.2cm>;"Nya2"!<0cm,-.2cm>|{\Uparrow}}}}
$$
\caption{Two Reidemeister moves of type III which are equivalent up to Reidemeister II moves}
\label{fig:Trick}
\end{figure}

These moves are not all independant. For instance, the trick illustrated in figure \ref{fig:Trick} leads to the following lemma:
\begin{lemme}
  Let $I$ be a map defined on the set of diagrams. If $I$ is invariant under all Reidemeister moves of type II and under at least one of the eight moves of type III, then it is invariant under all moves of type III.
\end{lemme}
\begin{figure}
$$
\xymatrix @!0 @C=3.5cm @R=3.3cm {
& *+[F-:<5pt>]{\dessin{1.3cm}{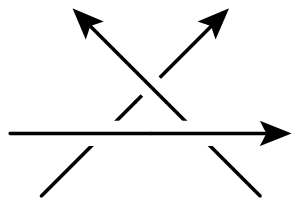}} \ar|{\dessin{.35cm}{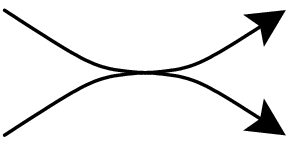}}@{<->}[r] \ar|{\dessin{.35cm}{sing1g}}@{<->}[dr] &
*+[F-:<5pt>]{\dessin{1.3cm}{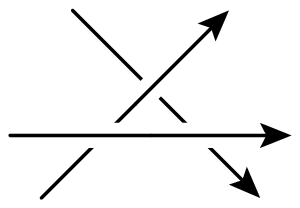}} \ar|{\dessin{.35cm}{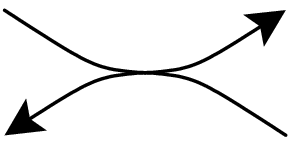}}@{<->}[dr] & \\
*+[F-:<5pt>]{\dessin{1.3cm}{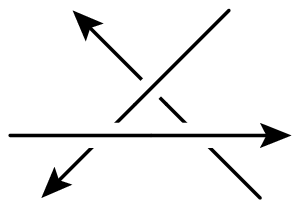}} \ar|{\dessin{.35cm}{sing1bisg}}@{<->}[ur] \ar|{\dessin{.35cm}{sing1bisg}}@{<->}[r] \ar|{\dessin{.35cm}{sing1bisg}}@{<->}[dr] &
*+[F-:<5pt>]{\dessin{1.3cm}{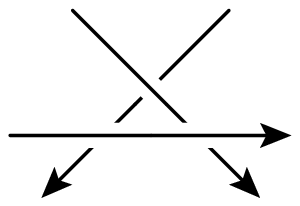}} \ar|{\dessin{.35cm}{sing1g}}@{<->}[ur] \ar|{\dessin{.35cm}{sing1g}}@{<->}[dr] &
*+[F-:<5pt>]{\dessin{1.3cm}{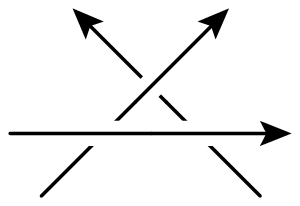}} \ar|{\dessin{.35cm}{sing1bisg}}@{<->}[r] &
*+[F-:<5pt>]{\dessin{1.3cm}{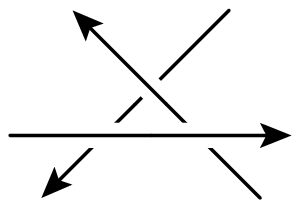}} \\
& *+[F-:<5pt>]{\dessin{1.3cm}{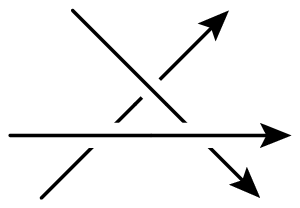}} \ar|{\dessin{.35cm}{sing1g}}@{<->}[ur] \ar|{\dessin{.35cm}{sing1g}}@{<->}[r] &
*+[F-:<5pt>]{\dessin{1.3cm}{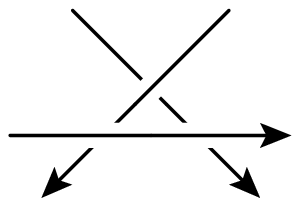}} \ar|{\dessin{.35cm}{sing1bisg}}@{<->}[ur]& \\
}
$$
\caption{Reidemeister III moves graph: {\footnotesize Each vertex is a Reidemeister move of type III, edges are labeled by Reidemeister moves of type II. An arrow between two vertices means that they can be replaced one by the other using the labeling Reidemeister II move as pictured in \ref{fig:Trick}. The mirror operation corresponds to reversing horizontally the graph.}}
\label{fig:TypeIII}
\end{figure}
According to Figure \ref{fig:TypeIII}, the lemma can even be refined.
\begin{lemme}
  Let $I$ be a map defined on the set of diagrams which is invariant under the Reidemeister moves of type II coresponding to the singularity $\dessin{.5cm}{sing1bis}$ and under at least one of the eight moves of type III. If, furthermore, a map $\phi$ satisfying $I\circ\mu=\phi\circ I$ exists, where $\mu$ denotes the mirror operation, then $I$ is invariant under all moves of type III.
\end{lemme}

It is thus natural to wonder to which extend certain Reidemeister moves can be replaced by other ones in an isotopy of links. Because of their characteristic behavior with respect to the writhe, the two movements of the type I cannot pretend to such substitutions.\\
According to the previous lemma, if one wants to distinguish Reidemeister moves of type III, it is necessary to exclude some Reidemesiter moves of type II and particularly those which correspond to the $\dessin{.5cm}{sing1bis}$ singularity.

\subsection{Braid-like isotopies}
\label{ssec:BraidLike}

\subsubsection{Braid-like links}
\label{par:BraidLikeLinks}

Figure \ref{fig:TypeIII} shows that, up the $\dessin{.5cm}{sing1}$ singularity, six Reidemeister moves are linked. Besides, it corresponds exactly to the Reidemeister moves which occur in isotopies of braid diagrams. This motivates the following definition:

\begin{defi}
A \emph{braid-like move}\index{moves!braid-like} is a Reidemeister move corresponding to one of the two following singularities of planar curve:
$$
\dessin{.75cm}{sing1} \hspace{2cm} \dessin{1.1cm}{Sing2bis}.
$$
A \emph{braid-like isotopy}\index{isotopy!braid-like} is a sequence of braid-like moves.\\
Two diagrams of links are \emph{braid-like isotopic} if they are related by a braid-like isotopy.\\
\emph{Braid-like links}\index{knot!braid-like}\index{link!braid-like} are braid-like isotopy classes.
\end{defi}

\begin{lemme}\label{jeudemove}
If $I$ is invariant under the braid-like moves of type II and under at least one braid-like move of type III, then it is also invariant under all other braid-like moves of type III.
\end{lemme}

Braid-like links can be seen as a refinement of usual links. Moreover, since a theorem of Artin says that two closed braids in the solid torus are isotopic if and only if they are braid-like isotopic, it can also be seen as a generalization of such objets.\\
The study of links via closed braids and Markov moves, which correspond to stabilzations of braids by adding once twisted strands, has led to considerable progress in knot theory (see {\it e.g.} \cite{Birman2}). It is then natural to ask what are the equivalent of Markov moves for braid-like links. It would lead to a new description of links for which a refined Khovanov homology is already defined in this thesis.

\subsection{Star-like isotopies}
\label{ssec:StarLike}

\subsubsection{Star-like links}
\label{par:StarLikeLinks}

It is now natural to consider the complementary set of braid-like moves of type III. This leads to the following definition:
\begin{defi}
A \emph{star-like move}\index{moves!star-like} is a Reidemeister move corresponding to one of the two following singularities of planar curve:
$$
\dessin{.75cm}{sing1} \hspace{2cm} \dessin{1.1cm}{sing2}.
$$
A \emph{star-like isotopy}\index{isotopy!star-like} is a sequence of star-like moves.\\
Two diagrams of links are \emph{star-like isotopic} if they are related by a star-like isotopy.\\
\emph{Star-like links}\index{knot!star-like}\index{link!star-like} are star-like isotopy classes.
\end{defi}

From now on, \emph{restricted}\index{restricted} stands for braid-like or star-like.

\section{Jones polynomial refinements}
\label{sec:JonesRefinement}

We denote by $\C$ the set of closed unoriented $1$--manifold, possibly empty, embedded in $\R^2$ up to ambiant isotopies. Elements of $\C$ can be seen as configurations of circles in the plane. Moreover, we define $\Gamma$ as the $\Z[A,A^{-1}]$--module over $\C$.\\

Let $D$ be a link diagram.

\subsection{Kauffman states}
\label{ssec:Kauffman}

\subsubsection{Break and Seifert points}
\label{par:BreakSeifert}

Our starting point is the Kauffman bracket for framed oriented links $K$ in $3$--space. We use Kauffman's notation and terminology as depicted in \cite{StateJones}.\\

For each crossing of $D$, there is two different ways to resolve it. A \emph{Kauffman state}\index{smoothing!Kauffman state}\index{Kauffman state|see{smoothing}} is a choice of smoothing for every crossing of $D$. It is hence an element of $\C$.\\

Each Kauffman state $s$ of $D$ has a piecewise smooth structure which is induced by the orientation of $D$. The crosssings of $D$ give rise to special points of $s$. If in such a point, the piecewise orientation changes, we say that it is a \emph{break point}\index{break point}. The number of break points on any circle of $s$ is always even. The remaining special points are called \textit{Seifert points}\index{Seifert point}. In Seifert points, the orientations of the two pieces fit together.\\

\begin{figure}
$$
\xymatrix@!0 @R=.8cm @C=3.5cm {
& \dessin{1cm}{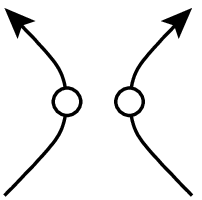} \hspace{.5cm} \textrm{Seifert points}\\
\dessin{1cm}{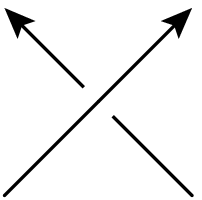} \ar[ur]!L \ar[dr]!L & \\
& \dessin{1cm}{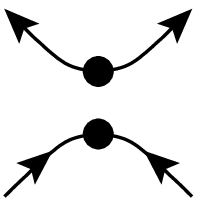} \hspace{.5cm} \textrm{break points}\\
}
$$
\caption{Break and Seifert points}
\label{fig:Seifert}
\end{figure}

\subsection{Braid-like Jones polynomial}
\label{ssec:BraidLikeJones}

\subsubsection{Definition}
\label{par:BraidLikeJonesDef}

Let $s$ be a Kauffman state for $D$.\\
Let $c$ be a connected component of $s$. It is \emph{of type $h$}\index{circle!of type h@of type $h$} if one half of the number of break points on $c$ is even. We say it is a \emph{$h$--circle}\index{hcircle@$h$--circle|see{circle, of type $h$}}. Otherwise, it is \emph{of type $d$}\index{circle!of type d@of type $d$} and we call it a \emph{$d$--circle}\index{dcircle@$d$--circle|see{circle, of type $d$}}.

Now we define some gradings on the set of Kauffman states by
\begin{eqnarray*}
  \sigma(s) & = & \#\{A\textrm{-smoothed crossing in }s\} - \#\{A^{-1}\textrm{-smoothed crossing in }s\},\\
  d(s) & = & \#\{\textrm{circles of type }d\textrm{ in }s\},\\
  h(s) & = & \#\{\textrm{circles of type }h\textrm{ in }s\},
\end{eqnarray*}
and $c(s) \in C$ is the configuration of only the $h$--circles.

\begin{defi}
The \emph{braid-like Kauffman bracket}\index{Kauffman bracket!braid-like} $\langle D \rangle_{br} \in \Gamma$ is defined by
$$
\langle D \rangle_{br}=\sum_{\substack{s\textrm{ Kauffman}\\[.1cm] \textrm{state of }D}} A^{\sigma(s)}(-A^2 - A^{-2})^{d(s)} c(s).
$$
Then, we define the \emph{braid-like Jones polynomial}\index{polynomial!Jones!braid-like} by
$$
V_{br}(D)=(-A)^{-3w(D)}\langle D \rangle_{br}.
$$
\end{defi}
\index{Vbr@$V_{br}$|see{polynomial, Jones, braid-like}}

From this definition, it follows immediatly:
\begin{prop}\label{BraidLikeSkein}
Let $v$ be a crossing of $D$. Let $D_0$ and $D_1$ be the diagrams obtained from $D$ by performing on $v$ respectively a $A$--smoothing and a $A^{-1}$--smoothing. Then, we have
$$
{\langle D \rangle}_{br} = A {\langle D_0 \rangle}_{br} + A^{-1} {\langle D_1 \rangle}_{br}.
$$
Moreover, we have
$$
{\langle D \ \dessin{.3cm}{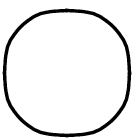} \rangle}_{br} = (-A^2 - A^{-2}) {\langle D \rangle}_{br}.
$$
\end{prop}

\subsubsection{Consistency}
\label{par:BraidLikeJonesCons}

\begin{theo}\label{InvBraidLike}
  The bracket $\langle \ .\ \rangle_{br}$ and the braid-like Jones polynomial are invariant under global orientation reversing and braid-like isotopies.
\end{theo}
\begin{proof}
  Since the writhe is invariant under braid-like isotopies, the statements for $V_{br}$ and the bracket are equivalent.\\
  The whole construction is clearly invariant under global orientation reversing.\\
  
  Now we consider the braid-like moves of type II.\\
  When calculating the Kauffman bracket, we obtain
  $$
  \dessin{.8cm}{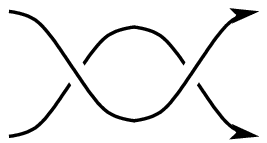} \ = \ A^2 \dessin{.8cm}{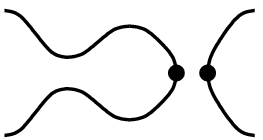} \ + \ \dessin{.8cm}{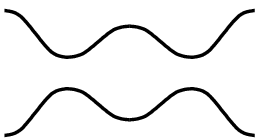} \ + \ \dessin{.8cm}{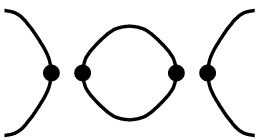}\ + \  A^{-2} \dessin{.8cm}{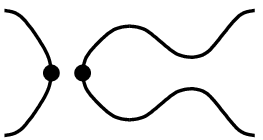} .
  $$
  The circle is of type $d$ and the usual identification $d = -A^2 - A^{-2}$ implies invariance.\\
  
  Then, according to Lemma \ref{jeudemove} (\S\ref{par:BraidLikeLinks}), the invariance under braid-like moves of type III is reduced to the invariance under the following move:
  $$
  \dessin{1.7cm}{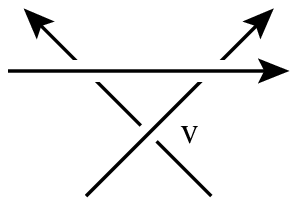} \ \longleftrightarrow \ \dessin{1.7cm}{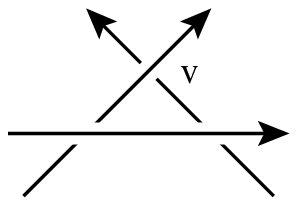}
  $$
  The $A$--smoothing of $v$ leads to isotopic diagrams. Hence, we only have to prove 
  $$
  \dessin{1.16cm}{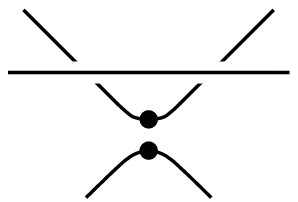}\ = \ \dessin{1.16cm}{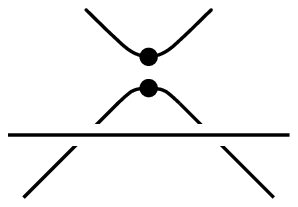}.
  $$
  The calculation on the left-hand side gives
  $$
  \begin{array}{rcl}
    \dessin{1.16cm}{fig3a} & = &
    A^2 \dessin{1.16cm}{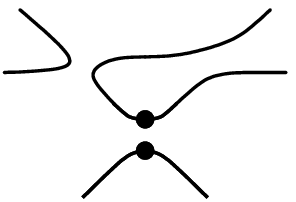} \ + \ \dessin{1.16cm}{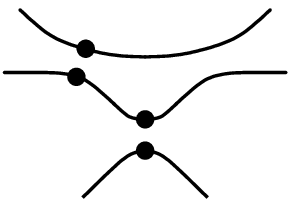} \ + \ \dessin{1.16cm}{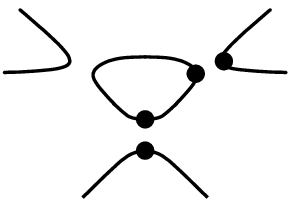} \ + \ A^{-2} \dessin{1.16cm}{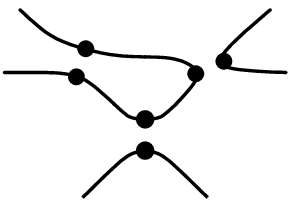}\\[7mm]
    & = & \dessin{1.16cm}{cac1},
  \end{array}
  $$
  whereas, on the right-hand side, it gives
  $$
  \begin{array}{rcl}
    \dessin{1.16cm}{fig3b} & = &
    A^2 \dessin{1.16cm}{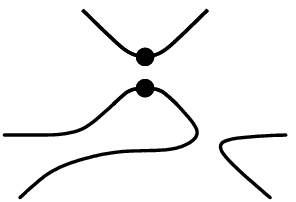} \ + \ \dessin{1.16cm}{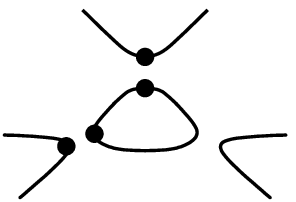} \ + \ \dessin{1.16cm}{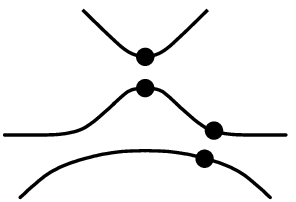} \ + \ A^{-2} \dessin{1.16cm}{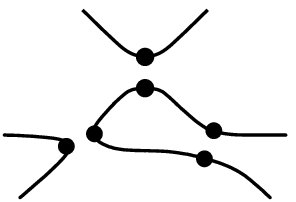}\\[7mm]
    & = & \dessin{1.16cm}{cbc3}.
  \end{array}
  $$
\end{proof}

\subsubsection{Properties}
\label{par:BraidLikeJonesProp}

It is a direct consequence of Proposition \ref{BraidLikeSkein} (\S\ref{par:BraidLikeJonesDef}) that $V_{br}$ satisfies the same skein relation than the Jones polynomial. Nevertheless, it differs from the Jones polynomial since it has infinitely initial conditions. Actually, as shown in Figure \ref{fig:IsotopicTrivialDiagrams}, there are infinitely regularly isotopic\footnote[1]{\ie with the same writhe and the same Whitney index} diagrams of the trivial knot which are pairwise non braid-like isotopic.\\

\begin{figure}[!h]
$$
\dessin{1.05cm}{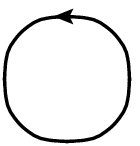}\ , \ \  \dessin{1.05cm}{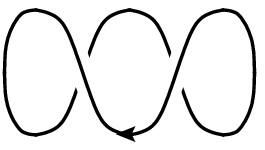}\ , \ \  \dessin{1.05cm}{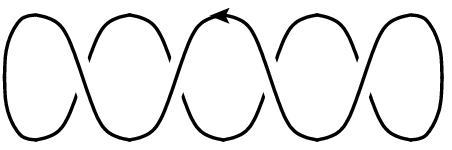}\ , \ \ \dots
$$

\caption{Regularly, but not braid-like, isotopic trivial knots}
\label{fig:IsotopicTrivialDiagrams}
\end{figure}

Of course, the Seifert state $\Sei(D)$ obtained by resolving all the crossings of $D$ with respect to its orientation is involved in the sum which defines ${\langle D \rangle}_{br}$. But even more, as implied by the following proposition, it survives as a leading term.
\begin{prop}
  The bracket $\langle D \rangle_{br}$ is equal to $A^{w(D)} \Sei(D)$ and a remainder involving terms with strictly fewer connected components.
\end{prop}

This proves that the Seifert state is already an invariant for braid-like links and, incidently, that all diagrams in Figure \ref{fig:IsotopicTrivialDiagrams} are pairwise non braid-like isotopic.

\begin{proof}
  Let $s$ be a Kauffman state which maximizes the number of $h$--circles among the states which are not the Seifert one. As a matter of fact, it contains some break points. Using moves $1$---$4$ in Figure \ref{fig:d-break}, we remove as many of them as possible. This operation does not modify the number of $h$--circles.

\begin{figure}[h]
$$
\begin{array}{ccccccc}
\dessin{1cm}{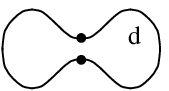} & \longrightarrow & \dessin{1cm}{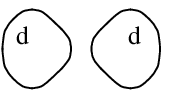} & \hspace{1.5cm} & \dessin{1cm}{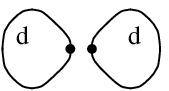} & \longrightarrow & \dessin{1cm}{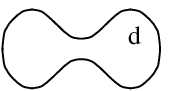}\\
&\textrm{move 1}& & & &\textrm{move 2}&\\[5mm]
\dessin{1cm}{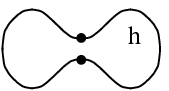} & \longrightarrow & \dessin{1cm}{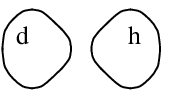} & & \dessin{1cm}{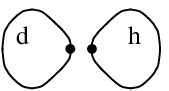} & \longrightarrow & \dessin{1cm}{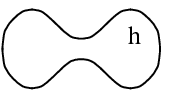}\\
&\textrm{move 3}& & & &\textrm{move 4}&\\[5mm]
\dessin{1cm}{rbp11} & \longrightarrow & \dessin{1cm}{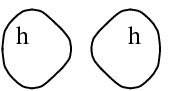} & & \dessin{1cm}{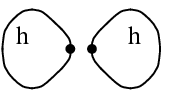} & \longrightarrow & \dessin{1cm}{rbp22}\\
&\textrm{move 5}& & & &\textrm{move 6}&
\end{array}
$$

\caption{Moves removing break points}
\label{fig:d-break}
\end{figure}

Then if it remains a $d$--circle, it means that a move $5$ can be performed and, as it increases the number of $h$--circles, the resulting state must be the Seifert one. The result is then proven.\\
Otherwise, we get a configuration of only $h$--circles $s'$ with the following properties:

\begin{enumerate}
\item[i)] no connected component contains an internal breaking-smoothed crossing i;e. any two break points linked by a crossing are on distinct connected components;
\item[ii)] it contains some break points since, as the last move to get the Seifert state must be a move $5$ in Figure \ref{fig:d-break}, it cannot be the Seifert state.
\end{enumerate}

Then we construct a graph $G$ by associating a vertex to each $h$--circle of $s'$ and an edge to each pair of break points linked by a cossing. Property $i)$ assures that $G$ is planar and propety $ii)$ that it contains a connected component not reduced to a point. We choose such a component $G'$ with $k \geq 2$ vertices and $n$ edges corresponding to a substate $s''$ of $s'$ with $k$ $h$--circle and $2n$ break points.\\
Now we change the smoothing of $(k-1)$ breaking-smoothed crossings of $s''$ in such a way that the number of circles decreases each time by one. This can be done because $G'$ is connected and the removal of two linked break points corresponds to the retraction of the corresponding edge in $G'$.\\
Then it remains one circle with $2(k-k+1)$ break points corresponding to a graph with a single vertex and $(n-k+1)$ loops. As this graph is still planar, we can find two adjacent break points linked by a crossing. Removing these two break points creates an $h$--circle with no break point. By repeating this operation, we create one such $h$--circle for each pair of linked break points, except for the last one which will create two such $h$--circles.\\
At last, we get a state with $(n-k+2)$ $h$--circles. Now, each $h$--circles in $s''$ was containing at least one break point, therefore at least four break points in order to be of type $h$. Thus, we have
$$
2n \geq 4k,
$$
and consequently
$$
n-k+2 \geq k+2.
$$
We have therefore increased the number of $h$--circles, so the last state must be the Seifert state.\\
This concludes the proof.
\end{proof}

In the case of closed braids in a solid torus, Hoste and Przytycki have already defined a refinement of Jones polynomial in \cite{Hoste}. The braid-like one can be seen as a generalization of it.
\begin{prop}\label{ClosedBraid}
  In the case of closed braids, $V_{br}$ coincides with the invariant of Hoste and Przytycki.
\end{prop}
\begin{proof}
  A connected component of a Kauffman state for a closed braid in a solid torus is contractible in the solid torus if and only if it is of type $d$. Actually, the critical points of the restriction of the radius function on a circle correspond exactly to the break points on this circle. Moreover, the number of such critical points is congruent to $2$ modulo $4$ if and only if the circle is contractible.
\end{proof}

\subsection{Star-like Jones polynomial}
\label{ssec:StarLikeJones}

\subsubsection{Definition}
\label{par:StarLikeJonesDef}
Except concerning the definitions of $h$ and $d$--circles, the star-like refinement of Jones polynomial is identical to the braid-like one.\\ 

Let $s$ be a Kauffman state for $D$.\\
Let $c$ be a connected component of $s$. It is \emph{of type $h$}\index{circle!of type h@of type $h$} if one half of the number of break points plus the number of Seifert points on $c$ is even. We say it is a \emph{$h$--circle}\index{hcircle@$h$--circle|see{circle, of type $h$}}. Otherwise, it is \emph{of type $d$}\index{circle!of type d@of type $d$} and we call it a \emph{$d$--circle}\index{dcircle@$d$--circle|see{circle, of type $d$}}.

Now, similarly to the braid-like case, we define some gradings on the set of Kauffman states by
\begin{eqnarray*}
  \sigma(s) & = & \#\{A\textrm{-smoothed crossing in }s\} - \#\{A^{-1}\textrm{-smoothed crossing in }s\},\\
  d(s) & = & \#\{\textrm{circles of type }d\textrm{ in }s\},\\
  h(s) & = & \#\{\textrm{circles of type }h\textrm{ in }s\},
\end{eqnarray*}
and $c(s) \in C$ is the configuration of only the $h$--circles.

\begin{defi}
The \emph{star-like Kauffman bracket}\index{Kauffman bracket!star-like} $\langle D \rangle_{st} \in \Gamma$ is defined by
$$
\langle D \rangle_{st}=\sum_{\substack{s\textrm{ Kauffman}\\[.1cm] \textrm{state of }D}} A^{\sigma(s)}(-A^2 - A^{-2})^{d(s)} c(s).
$$
Moreover, we define the \emph{star-like Jones polynomial}\index{polynomial!Jones!star-like} by
$$
V_{st}(D)=(-A)^{-3w(D)}\langle D \rangle_{st}.
$$
\end{defi}
\index{Vst@$V_{st}$|see{polynomial, Jones, star-like}}

Here again, it follows from the definition:
\begin{prop}\label{StarLikeSkein}
Let $v$ be a crossing of $D$. Let $D_0$ and $D_1$ be the diagrams obtained from $D$ by performing on $v$ respectively a $A$--smoothing and a $A^{-1}$--smoothing. Then, we have
$$
{\langle D \rangle}_{st} = A {\langle D_0 \rangle}_{st} + A^{-1} {\langle D_1 \rangle}_{st}.
$$
Moreover, we have
$$
{\langle D \ \dessin{.3cm}{cer} \rangle}_{st} = (-A^2 - A^{-2}) {\langle D \rangle}_{st}.
$$
\end{prop}

\subsubsection{Consistency}
\label{par:StarLikeJonesCons}

\begin{theo}\label{InvStarLike}
  The bracket $\langle \ .\ \rangle_{st}$ and the braid-like Jones polynomial are invariant under global orientation reversing and star-like isotopies.
\end{theo}
\begin{proof}
Arguments for the invariance under global orientation reversing and star-like moves of type II are similar than for the braid-like case.\\

For invariance under star-like moves of type III, there is no star-like equivalent of Lemma \ref{jeudemove} (\S\ref{par:BraidLikeLinks}). The moves
$$
\dessin{1.7cm}{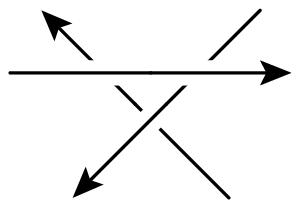} \ \longleftrightarrow \ \dessin{1.7cm}{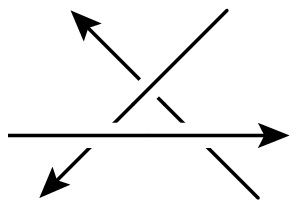}
$$
and
$$
\dessin{1.7cm}{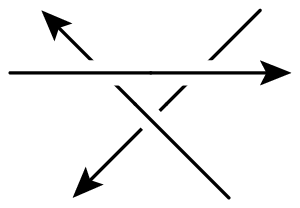} \ \longleftrightarrow \ \dessin{1.7cm}{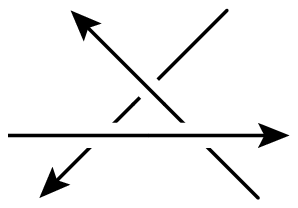}
$$
need hence to be treated separately.\\

The left hand side and the right hand side of the first one give respectively the following eight contributions to the bracket:
$$
\hspace{-.2cm}
(1)
\left\{\textrm{
    \begin{tabular}{p{2.1cm}p{.8cm}p{2.1cm}}
      \begin{tabular}{p{.3cm}p{1.5cm}}
        \centering{$A$} & $\dessin{1.16cm}{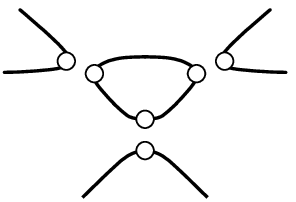}$ 
      \end{tabular}
      &&
      \begin{tabular}{p{.3cm}p{1.5cm}}
        \centering{$A^3$} & $\dessin{1.16cm}{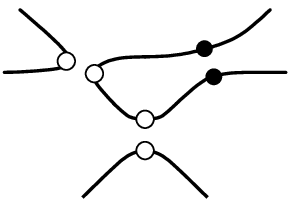}$ 
      \end{tabular}
      \\[.6cm]
      \begin{tabular}{p{.3cm}p{1.5cm}}
        \centering{$A^{-1}$} & $\dessin{1.16cm}{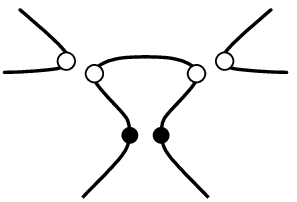}$ 
      \end{tabular}
      &&
      \begin{tabular}{p{.3cm}p{1.5cm}}
        \centering{$A^{-1}$} & $\dessin{1.16cm}{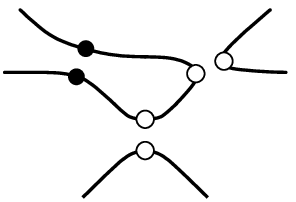}$ 
      \end{tabular}
    \end{tabular}}
\right.
\hspace{1.2cm}
(1')
\left\{\textrm{
    \begin{tabular}{p{5.8cm}}
      \centering{\begin{tabular}{p{.3cm}p{1.5cm}}
          \centering{$A^{-1}$} & $\dessin{1.16cm}{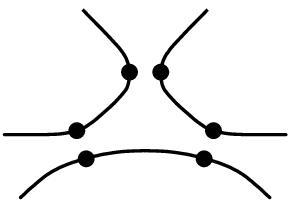}$ 
        \end{tabular}}
    \end{tabular}}   
\right.
$$
$$
\hspace{-.2cm}
(2)
\left\{\textrm{
    \begin{tabular}{p{5.8cm}}
      \centering{\begin{tabular}{p{.3cm}p{1.5cm}}
          \centering{$A^{-1}$} & $\dessin{1.16cm}{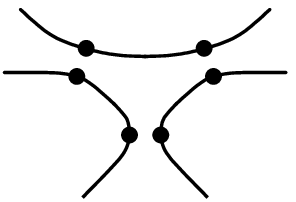}$ 
        \end{tabular}}
    \end{tabular}}   
\right.
\hspace{1.2cm}
(2')
\left\{\textrm{
    \begin{tabular}{p{2.1cm}p{.8cm}p{2.1cm}}
      \begin{tabular}{p{.3cm}p{1.5cm}}
        \centering{$A$} & $\dessin{1.16cm}{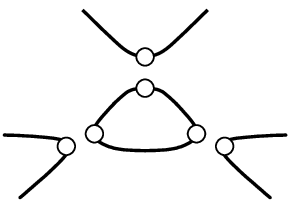}$ 
      \end{tabular}
      &&
      \begin{tabular}{p{.3cm}p{1.5cm}}
        \centering{$A^3$} & $\dessin{1.16cm}{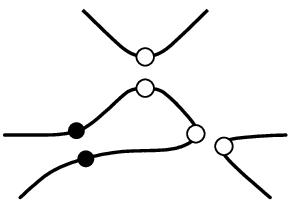}$ 
      \end{tabular}
      \\[.6cm]
      \begin{tabular}{p{.3cm}p{1.5cm}}
        \centering{$A^{-1}$} & $\dessin{1.16cm}{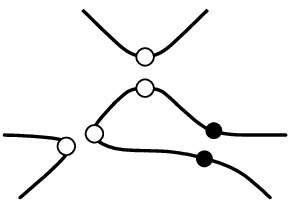}$ 
      \end{tabular}
      &&
      \begin{tabular}{p{.3cm}p{1.5cm}}
        \centering{$A^{-1}$} & $\dessin{1.16cm}{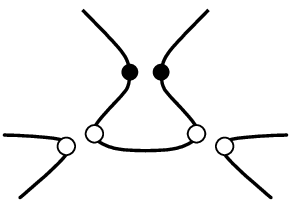}$ 
      \end{tabular}
    \end{tabular}}
\right.
$$
\vspace{.3cm}
$$
\hspace{-.2cm}
(3)
\left\{\textrm{
\begin{tabular}{p{2.1cm}p{.8cm}p{2.1cm}}
  \begin{tabular}{p{.3cm}p{1.5cm}}
    \centering{$A$} & \centering{$\dessin{1.16cm}{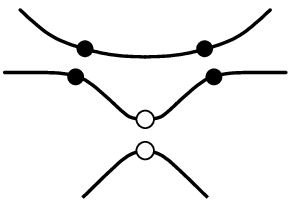}$} 
  \end{tabular}
  &&
  \begin{tabular}{p{.3cm}p{1.5cm}}
    \centering{$A^{-3}$} & \centering{$\dessin{1.16cm}{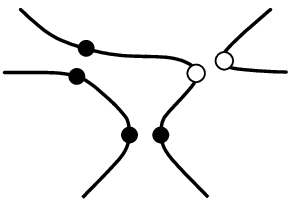}$} 
  \end{tabular}
  \\[.6cm]
  \multicolumn{3}{c}{\begin{tabular}{p{.3cm}p{1.5cm}}
      \centering{$A$} & $\dessin{1.16cm}{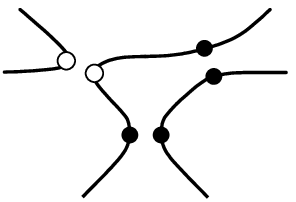}$ 
    \end{tabular}}
\end{tabular}}
\right.
\hspace{1.2cm}
(3')
\left\{\textrm{
\begin{tabular}{p{2.1cm}p{.8cm}p{2.1cm}}
  \begin{tabular}{p{.3cm}p{1.5cm}}
    \centering{$A$} & \centering{$\dessin{1.16cm}{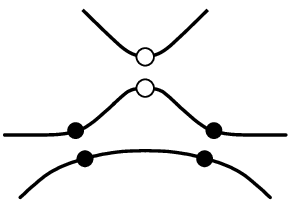}$} 
  \end{tabular}
  &&
  \begin{tabular}{p{.3cm}p{1.5cm}}
    \centering{$A^{-3}$} & \centering{$\dessin{1.16cm}{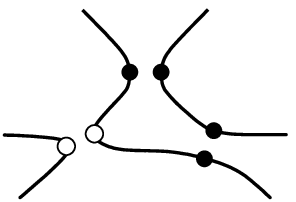}$} 
  \end{tabular}
  \\[.6cm]
  \multicolumn{3}{c}{\begin{tabular}{p{.3cm}p{1.5cm}}
      \centering{$A$} & $\dessin{1.16cm}{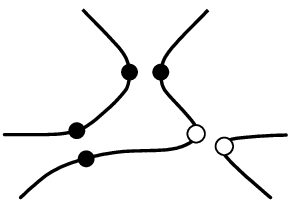}$ 
    \end{tabular}}
\end{tabular}}
\right.
$$

The closed connected components arising in $(1)$ and $(2')$ are of type $d$. This implies that $(1)=(1')$, $(2)=(2')$ and $(3)=(3')$ when summing their elements to evaluate the star-like Kauffman bracket. Moreover, the configurations of $h$--circles in $\R^2$ are the same.\\

The left hand side and the right hand side of the second move give respectively the following contributions:
$$
\hspace{-.2cm}
(1)
\left\{\textrm{
    \begin{tabular}{p{2.1cm}p{.8cm}p{2.1cm}}
      \begin{tabular}{p{.3cm}p{1.5cm}}
        \centering{$A^{-1}$} & $\dessin{1.16cm}{III11}$ 
      \end{tabular}
      &&
      \begin{tabular}{p{.3cm}p{1.5cm}}
        \centering{$A$} & $\dessin{1.16cm}{III12}$ 
      \end{tabular}
      \\[.6cm]
      \begin{tabular}{p{.3cm}p{1.5cm}}
        \centering{$A^{-1}$} & $\dessin{1.16cm}{III13}$ 
      \end{tabular}
      &&
      \begin{tabular}{p{.3cm}p{1.5cm}}
        \centering{$A$} & $\dessin{1.16cm}{III14}$ 
      \end{tabular}
    \end{tabular}}
\right.
\hspace{1.2cm}
(1')
\left\{\textrm{
    \begin{tabular}{p{5.8cm}}
      \centering{\begin{tabular}{p{.3cm}p{1.5cm}}
          \centering{$A$} & $\dessin{1.16cm}{III1s1}$ 
        \end{tabular}}
    \end{tabular}}   
\right.
$$
$$
\hspace{-.2cm}
(2)
\left\{\textrm{
    \begin{tabular}{p{5.8cm}}
      \centering{\begin{tabular}{p{.3cm}p{1.5cm}}
          \centering{$A$} & $\dessin{1.16cm}{III21}$ 
        \end{tabular}}
    \end{tabular}}   
\right.
\hspace{1.2cm}
(2')
\left\{\textrm{
    \begin{tabular}{p{2.1cm}p{.8cm}p{2.1cm}}
      \begin{tabular}{p{.3cm}p{1.5cm}}
        \centering{$A^{-1}$} & $\dessin{1.16cm}{III2s1}$ 
      \end{tabular}
      &&
      \begin{tabular}{p{.3cm}p{1.5cm}}
        \centering{$A$} & $\dessin{1.16cm}{III2s2}$ 
      \end{tabular}
      \\[.6cm]
      \begin{tabular}{p{.3cm}p{1.5cm}}
        \centering{$A$} & $\dessin{1.16cm}{III2s3}$ 
      \end{tabular}
      &&
      \begin{tabular}{p{.3cm}p{1.5cm}}
        \centering{$A^{-3}$} & $\dessin{1.16cm}{III2s4}$ 
      \end{tabular}
    \end{tabular}}
\right.
$$
\vspace{.3cm}
$$
\hspace{-.2cm}
(3)
\left\{\textrm{
\begin{tabular}{p{2.1cm}p{.8cm}p{2.1cm}}
  \begin{tabular}{p{.3cm}p{1.5cm}}
    \centering{$A^3$} & \centering{$\dessin{1.16cm}{III31}$} 
  \end{tabular}
  &&
  \begin{tabular}{p{.3cm}p{1.5cm}}
    \centering{$A^{-1}$} & \centering{$\dessin{1.16cm}{III32}$} 
  \end{tabular}
  \\[.6cm]
  \multicolumn{3}{c}{\begin{tabular}{p{.3cm}p{1.5cm}}
      \centering{$A^{-1}$} & $\dessin{1.16cm}{III33}$ 
    \end{tabular}}
\end{tabular}}
\right.
\hspace{1.2cm}
(3')
\left\{\textrm{
\begin{tabular}{p{2.1cm}p{.8cm}p{2.1cm}}
  \begin{tabular}{p{.3cm}p{1.5cm}}
    \centering{$A^3$} & \centering{$\dessin{1.16cm}{III3s1}$} 
  \end{tabular}
  &&
  \begin{tabular}{p{.3cm}p{1.5cm}}
    \centering{$A^{-1}$} & \centering{$\dessin{1.16cm}{III3s2}$} 
  \end{tabular}
  \\[.6cm]
  \multicolumn{3}{c}{\begin{tabular}{p{.3cm}p{1.5cm}}
      \centering{$A^{-1}$} & $\dessin{1.16cm}{III3s3}$ 
    \end{tabular}}
\end{tabular}}
\right.
$$
We conclude like in the first case.
\end{proof}

\subsubsection{Properties}
\label{par:StarLikeJonesProp}

If giving a sense at adding Seifert points on diagrams, Proposition \ref{StarLikeSkein} (\S\ref{par:StarLikeJonesDef}) would lead to the following identity:
$$
A^4 V_{st}\left( \dessin{.45cm}{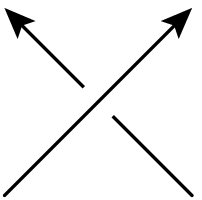} \right) - A^{-4} V_{st} \left( \dessin{.45cm}{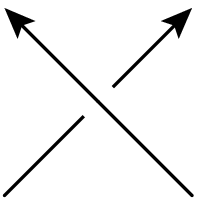} \right) = (A^2 - A^{-2})V_{st}\left( \dessin{.45cm}{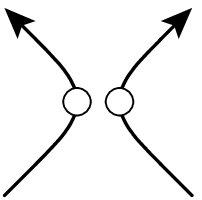} \right).
$$
But $V_{st}\left( \dessin{.45cm}{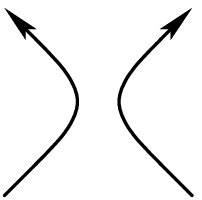} \right) \neq V_{st}\left( \dessin{.45cm}{skeinliss} \right)$ and there is no chance that $V_{st}$ does satisfy Jones skein relation.\\

Nevertheless, it happens that the star-like Jones polynomial is invariant under more Reidemeister moves.
\begin{prop}\label{PairsOfI}
  The polynomial $V_{st}$ is invariant under two simultaneous and adjacent Reidemeister I moves.
\end{prop}
\begin{proof}
  We will consider the two following cases only:
$$
\dessin{1.2cm}{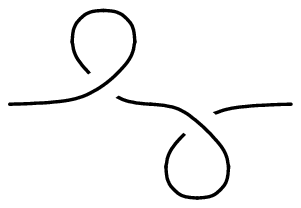} \  \longleftrightarrow  \ \dessin{1.2cm}{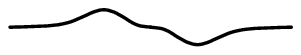}
$$
$$
\dessin{1.2cm}{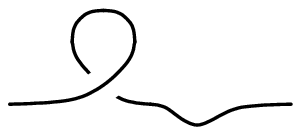} \ \longleftrightarrow  \ \dessin{1.2cm}{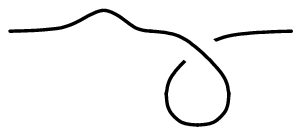}
$$
All other cases are analoguous.\\

The left hand side of the first case gives the following contributions:
\begin{eqnarray*}
\dessin{1.1cm}{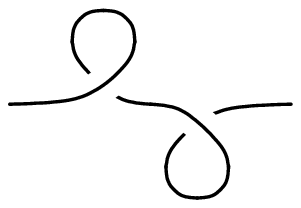} & = & A^2 \dessin{1.1cm}{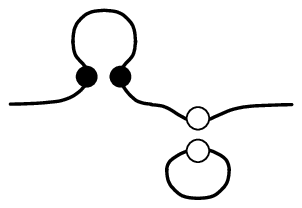} + \dessin{1.1cm}{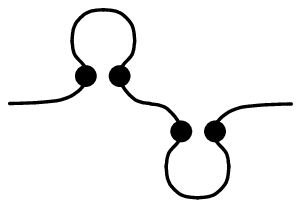} + \dessin{1.1cm}{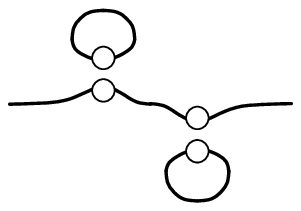} + A^{-2} \dessin{1.1cm}{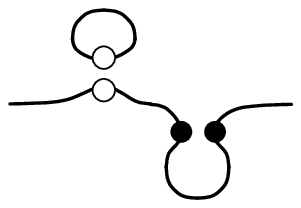}\\
& = & (A^2d + 1 + d^2 + A^{-2}d) \dessin{1.1cm}{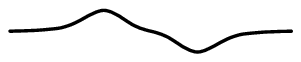}\\
& = & \dessin{1.1cm}{Istr2}.
\end{eqnarray*}

The left hand side of the second gives the contribution:
\begin{eqnarray*}
-A^3\left( A \dessin{1.1cm}{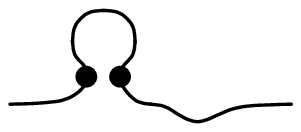} + A^{-1}\dessin{1.1cm}{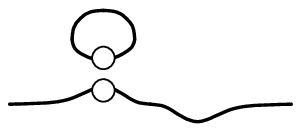}\right) & = & -A^3\left( (A + A^{-1}d) \dessin{1.1cm}{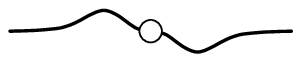}\right)\\
& = & \dessin{1.1cm}{I25},
\end{eqnarray*}
whereas its right hand side gives:
\begin{eqnarray*}
-A^{-3}\left( A \dessin{1.1cm}{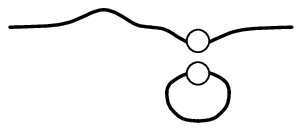} + A^{-1}\dessin{1.1cm}{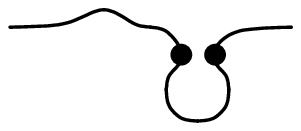}\right) & = & -A^{-3}\left( (Ad + A^{-1}) \dessin{1.1cm}{I25}\right)\\
& = & \dessin{1.1cm}{I25}.
\end{eqnarray*}
\end{proof}

% Transition
\vspace{2cm}
Jones polynomial have been refined to suit the restricted cases. It is now natural to wonder how its Khovanov categorication behaves.

%%% Local Variables: 
%%% mode: latex
%%% TeX-master: "These"
%%% End: 

% Invariants de Khovanov pour entrelacs restreints
\chapter{Khovanov invariants for restricted links}
\label{chap:Khovanov}

This chapter uses the basics in algebra given in the first section of the first chapter. However, for convenience, we require here that chain maps commutes with differentials instead of anti-commuting.\\

Defined as elements of $\Gamma$, $V_{br}$ and $V_{st}$ resist all attempt of categorification. Nevertheless, even if they may contain less information, one can define lightened versions which can be more easily categorified.

\section{Definition of homologies for restricted links}
\label{sec:DefinitionKhovanov}

Let $D$ be a link diagram.

\subsection{Lightening and reformulation}
\label{ssec:LighteningReformulation}

\subsubsection{Lightening}
\label{par:Lightening}

First, we define a map
$$
\func{\chi}{\C}{\Z[H,H^{-1}]}
$$
\index{chi@$\chi$}
which sends any configuration of circles $c$ to $(-H^2-H^{-2})^{|c|}$ where $|c|$ denotes the number of connected components of $c$. Then we extend it to $\Gamma$ by $\Z[A,A^{-1}]$--linearity.\\

Now we define the \emph{lightened restricted Jones polynomial}\index{polynomial!Jones!lightened restricted} as the image of $V_{br}$ or $V_{st}$ by $\chi$.

\subsubsection{Enhanced  Kauffman states}
\label{par:EnhancedStates}

A \emph{enhanced state}\index{smoothing!Kauffman state!enhanced state}\index{enhanced state|see{smoothing}} $S$ of $D$ is a Kauffman state $s$ of $D$ enhanced by an assignment of a plus or a minus sign to each of the connected components of $s$.\\

Now, as stated in \cite{Viro}, $\chi(V_{br})$ and $\chi(V_{st})$ can be given the following alternative definitions
$$
\left. \begin{array}{c}
  \chi(V_{br})\\[2mm]
  \chi(V_{st})
\end{array} \right\}
 = \sum_{\substack{S\textrm{ enhanced}\\ \textrm{state of }D}} (-1)^{\frac{\sigma(S) - w(K)}{2}}(-A^2)^{\frac{\sigma(S)-3w(K)+2\tau_d(S)}{2}}(-H^2)^{\tau_h(S)},
$$
where $\tau_d$ (resp. $\tau_h$) is the difference between the number of plus and minus assigned to the $d$--circles (resp. $h$--circles).

\subsection{Restricted Khovanov homology}
\label{par:DefBraidLikeKhovanov}

As soon as the definitions of $h$ and $d$--circles have been set down, the constructions for braid-like and star-like Khovanov homologies are strictly identical. Hence, we give the details for braid-like construction only and leave to the reader the care of changing $br$--indices for $st$--ones.

\subsubsection{Trigradings}
\label{par:KhChainCplx}

For any enhanced state $S$ of $D$, we define
$$
\begin{array}{l}
i(S) = \frac{\sigma(S) - w(K)}{2}\\[2mm]
j(S) = \frac{\sigma(S)-3w(K)+2\tau_d(S)}{2}\\[2mm]
k(S) = \tau_h(S).
\end{array}
$$

It is easily checked, using the fact that any link can be unknotted by switching a finite number of crossings that they are all integers.\\
This defines three gradings on $C(D)$ the module generated by all enhanced states of $D$. So we can consider the trigraded module $C_{br}(D):={\big(C_i^{j,k}(D)\big)}_{i,j,k\in\Z}$.

\subsubsection{Braid-like Khovanov differential}
\label{par:BraidLikeDiff}

First, we assign an arbitrary ordering on the crossings of $D$.\\

For any two enhanced states $S$ and $S'$ of $D$ and for any $v$ crossing of $D$, we define an incidence number $[S:S']_v$ by
\begin{itemize}
\item[-] $[S:S']_v=1$ if the following four conditions are satisfied:
\begin{enumerate}
\item[i)] $v$ is $A$--smoothed in $S$ but $A^{-1}$--smoothed in $S'$;
\item[ii)] all the other crossings are smoothed the same way in $S$ and $S'$;
\item[iii)] common connected components of $S$ and $S'$ are labeled the same way;
\item[iv)] $j(S) = j(S')$ and $k(S) = k(S')$;
\end{enumerate}
\item[-] otherwise, $[S:S']_v=0$.
\end{itemize}

Now we can set $\func{d_{br}^D}{C_{br}(D)}{C_{br}(D)}$ as the morphism of modules defined on generators by
$$
d_{br}^D(S) =  \sum_{\substack{v\textrm{ crossing}\\A\textrm{--smoothed in }S}} d_v(S)
$$
where the partial differential $d_v$ is defined by
$$
d_v(S) = (-1)^{t^-_{v,S}} \sum_{S'\textrm{enhanced state of }D} [S,S']_v S',
$$\index{dv@$d_v$}
with $t^-_{v,S}$ the number of $A^{-1}$--smoothed crossings in $S$ labeled with numbers greater than the label of $v$.\\
\index{homology!Khovanov!braid-like}
\index{homology!Khovanov!star-like}

For any enhanced state $S$, its image by $d_{br}^D$ is the alternating sum of enhanced states obtained by switching one $A$--smoothed crossing of $S$ into a $A^{-1}$--smoothed one and by locally relabeling connected components in such a way that $j$ and $k$ are preserved and $i$ is decreased by $1$. The map $d_{br}^D$ is then graded of tridegree $(-1,0,0)$.\\
Moreover the merge of two $d$--circles or two $h$--circles always gives a $d$--circle and the merge of a $d$--circle with an $h$--circle always an $h$--circle. Thus, one can explicit the action of $d_v$, as done in Figure \ref{fig:BraidDiff}.\\

\begin{figure}
$$
\begin{array}{ccc}
\begin{array}{ccc}
\dessin{1cm}{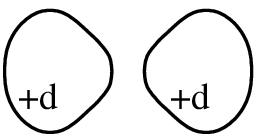}& \longrightarrow &zero\\[5mm]
\dessin{1cm}{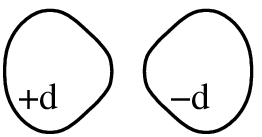}& \longrightarrow  &\dessin{1cm}{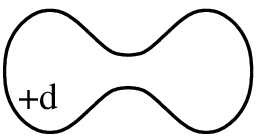}\\[5mm]
\dessin{1cm}{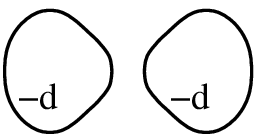}& \longrightarrow  &\dessin{1cm}{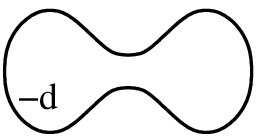}\\[5mm]
\dessin{1cm}{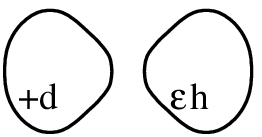}&  \longrightarrow  &zero\\[5mm]
\end{array}& \ \ &
\begin{array}{ccc}
\dessin{1cm}{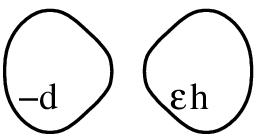}& \longrightarrow  &\dessin{1cm}{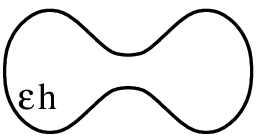}\\[5mm]
\dessin{1cm}{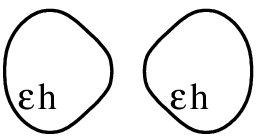}& \longrightarrow &zero\\[5mm]
\dessin{1cm}{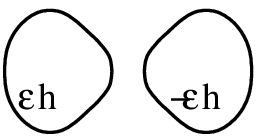}& \longrightarrow  &\dessin{1cm}{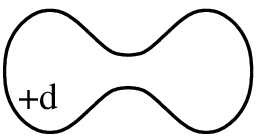}\\[5mm]
\end{array}
\end{array}
$$
$$
\hspace{1.5cm}
\begin{array}{ccl}
\dessin{1cm}{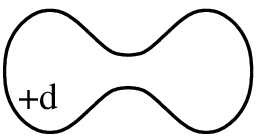}& \longrightarrow  &
\left\{\begin{array}{cl}
\dessin{1cm}{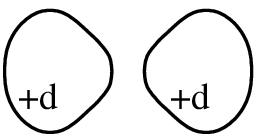}&\footnotemark[1]\\[5mm]
zero&
\end{array}\right.
\\[10mm]
\dessin{1cm}{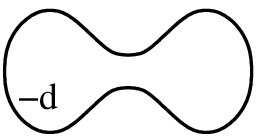}& \longrightarrow &
\left\{ \begin{array}{cl}
\dessin{1cm}{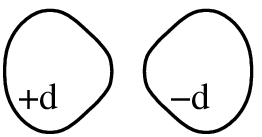}+\dessin{1cm}{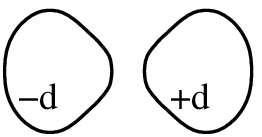}&\footnotemark[1]\\[5mm]
\dessin{1cm}{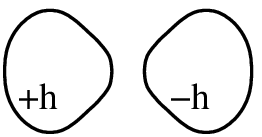}+\dessin{1cm}{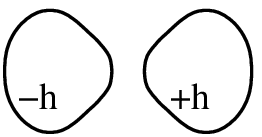}&\\
\end{array} \right.
\\[12mm]
\dessin{1cm}{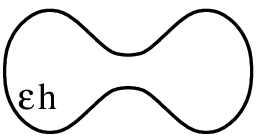}& \longrightarrow  &\dessin{1cm}{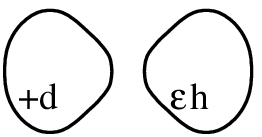}\\
\end{array}
$$
\caption{Definition of the partial differential $d_v$}
\label{fig:BraidDiff}
\end{figure}
\footnotetext[1]{depending on the type of the resulting circles}

\begin{prop}
  Every two distinct partial differentials anti-commute. As a consequence, $d_{br}^D$ is a differential. 
\end{prop}

\begin{proof}[Sketch of proof]
  To prove the proposition, all the possible configurations of labeled $h$ or $d$--circles linked by two crossings need to be listed. This corresponds to cases $1A$---$1D$ to $5A$---$5D$ in the proof of lemma 4.4 in \cite{Asaeda}.
\end{proof}

\begin{prop}
  The homology $\HH_{br}$ associated to $(C_{br}(D),d_{br}^D)$ does not depend on the ordering of the crossings.
\end{prop}
\begin{proof}[Sketch of proof]
It is sufficient to prove the invariance when permuting the order of two crossings $v$ and $v'$ which are consecutive for the considered orderings. But then, the values $(-1)^{t^-_{w,S}}[S,S']_w$ associated to each ordering differ only when one of the two following cases occurs:
\begin{itemize}
\item[i)] $w=v$ is $A$--smoothed in $S$ and $A^{-1}$--smoothed in $S'$ while $v'$ is $A^{-1}$--smoothed in both cases;
\item[ii)] $w=v'$ is $A$--smoothed in $S$ and $A^{-1}$--smoothed in $S'$ while $v$ is $A^{-1}$--smoothed in both cases. 
\end{itemize}
Then they differ from a minus sign.\\
On that account, the morphism of modules which sends an enhanced state $S$ to $-S$ if $v$ and $v'$ are $A^{-1}$--smoothed in $S$ and to $S$ otherwise, commutes with the two differentials and is therefore a quasi-isomorphism.
\end{proof}

By construction of the bigraded chain complex $C_{br}(D)$, its bigraded Euler characteristic is equal to $V_{br}$. According to \ref{Euler} (\S\ref{par:Homologies}), this remains true for its homology $\HH_{br}$

\subsubsection{Case of closed braids}
\label{par:APS}

As a corollary of the correspondence given in the proof of Proposition \ref{ClosedBraid} (\S\ref{par:BraidLikeJonesProp}), the merging rules between $d$ and $h$--circles are identical to the merging rules for annulus given in \cite{Asaeda}. This means that in the case of closed braids in a solid torus, the braid-like homology $\HH_{br}$ corresponds to homology developped in \cite{Asaeda}.

\section{Invariance under restricted isotopies}
\label{sec:InvRestricted}
\subsection{Braid-like invariance}
\label{ssec:InvBraidLikeHomology}

\begin{theo}
  The homology groups $\HH_{br}$ are invariant under braid-like isotopies.
\end{theo}

This section is devoted to the proof of this theorem. We adopt notation inspired from those described in Appendix \ref{appendix:Subcomplexes}. Modules are depicted by sets of generators and sets of generators by diagrams partially smoothed. A given set is obtained by considering all the choices for smoothing the unsmoothed crossings and labeling the unlabeled connected components.

\subsubsection{Invariance under braid-like move of type II}
\label{par:InvBraidReidII}

We consider two diagrams which differ only locally by a braid-like move of type II.
$$
\begin{array}{ccc}
D & \hspace{2cm} & D'\\
\dessin{1.1cm}{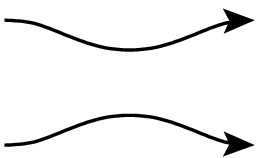} & & \dessin{1.1cm}{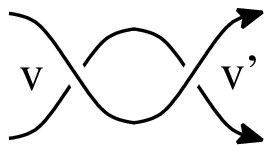}.
\end{array}
$$
We order the crossings of $D$ and $D'$ in the same fashion, letting $v$ and $v'$ be the last two ones of $D'$. Then we consider the following four diagrams which are partial smoothings of $D'$:
$$
\begin{array}{ccc}
D'_{00} & \hspace{1.5cm} & D'_{10}\\
\dessin{1cm}{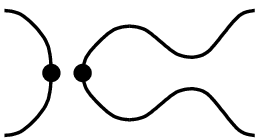} & & \dessin{1cm}{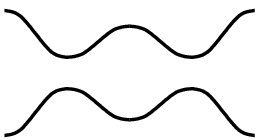}\\[5mm]
D'_{01} & & D'_{11}\\
\dessin{1cm}{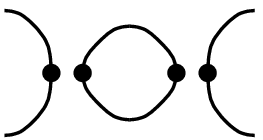} & & \dessin{1cm}{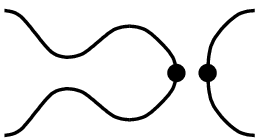},\\
\end{array}
$$
and two maps:
$$
\begin{array}{c}
  \func{d_{\star 1}}{\left\{\dessin{.8cm}{II01}\right\}}{\left\{\dessin{.8cm}{II11}\right\}},\\[.5cm]
  \func{d_{0\star}}{\left\{\dessin{.8cm}{II00}\right\}}{\left\{\dessin{.8cm}{II01}\right\}}.
\end{array}
$$
which are restriction of partial differentials.

\begin{lemme}
The morphisms $d_{\star 1}$ and $d_{0\star}$ are respectively surjective and injective.
\end{lemme}

\begin{proof}
This can be checked by direct calculus on generators.
\end{proof}

\begin{lemme}
The mapping cone of
$$
\xymatrix@C=2cm{ \left\{\ \dessin{.8cm}{II00}\ \right\}
  \ar[r]^{d_v+d_{0\star}}_\sim \ar@{}|{}="Nya"
  \ar@(dl,ul)"Nya"!<-1.35cm,-.3cm>;"Nya"!<-1.35cm,.3cm>^{d_{br}^{D'_{00}}}
  & \Im(d_v+d_{0\star}) \ar@{}|{}="Nya2"
  \ar@(dr,ur)"Nya2"!<1.1cm,-.3cm>;"Nya2"!<1.1cm,.3cm>_{d_{br}^{D'}}}
$$
is an acyclic subcomplex of $C_{gr}(D')$, which we denote by $C_1$.
\end{lemme}
\begin{proof}
First, we need to prove that the left and the right chain complexes are indeed left invariant by their respective differential.\\
Stability for the left summand is clear. Now we focus on the right one. As $d_{0\star}$ is injective, the map $(d_v+d_{0\star})$ has a left inverse on its image, and by anti-commutativity of distinct partial differentials we have for any element $S\in\Im(d_v+d_{0\star})$
$$
d_{gr}^{D'}(S) = -(d_v+d_{0\star})\circ d_{br}^{D'_{00}} \circ (d_v+d_{0\star})^{-1}(S) \in Im(d_v+d_{0\star}).
$$
It is now straighforward to check that it is actually a subcomplex of $C_{gr}(D')$. The acyclicity follows from the fact that $(d_v+d_{0\star})$ is injective and thus an isomorphism on its image.
\end{proof}

\begin{lemme}
The mapping cone of
$$
\xymatrix@C=2cm{ \left\{\ \dessin{.8cm}{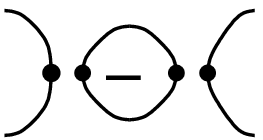}\ \right\}
  \ar[r]^{d_{\star 1}}_\sim \ar@{}|{}="Nya"
  \ar@(dl,ul)"Nya"!<-1.35cm,-.3cm>;"Nya"!<-1.35cm,.3cm>^{d_{br}^{D'_{01}}}
  & \left\{\ \dessin{.8cm}{II11}\ \right\} \ar@{}|{}="Nya2"
  \ar@(dr,ur)"Nya2"!<1.25cm,-.3cm>;"Nya2"!<1.25cm,.3cm>_{d_{br}^{D'}}}
$$
is an acyclic subcomplex of $C_{gr}(D')$, which we denote by $C_2$.
\end{lemme}

\begin{proof}
The stability of the right summand is clear. Since the middle $d$--circle of the generators of the left summand are labeled by $-$, the partial differential $d_v$ involved in $d_{br}^{D'_{01}}$ sends these generators to zero. Consequently, this $d$--circle is let unchanged by $d_{br}^{D'_{01}}$ and the left summand is stable.\\
Finally, the left summand is a supplementary space for $\Ker d_{\star 1}$ which is already surjective. It is then an isomorphism.
\end{proof}

We define the morphism of modules
$$
\begin{array}{rccc}
.\otimes v_-: & \left\{\ \dessin{.8cm}{II11}\ \right\} & \longrightarrow & \left\{\ \dessin{.8cm}{II01}\ \right\}\\[4mm]
& \dessin{.65cm}{II11} & \mapsto & -\dessin{.65cm}{Inj1}.
\end{array}
$$
It is grading-preserving and a right inverse for $d_{\star 1}$.

\begin{lemme}
  The chain complex
  $$
  \xymatrix@C=2cm{ \left\{\ \dessin{.8cm}{II10}-\big( d_{1 \star}(\dessin{.8cm}{II10}) \big)\otimes v_-\ \right\}
    \ar@{}|{}="Nya"
    \ar@(dl,ul)"Nya"!<-3.05cm,-.3cm>;"Nya"!<-3.05cm,.3cm>^{d_{br}^{D'}}}
  $$
is a subcomplex of $C_{gr}(D')$, which we denote by $C_3$, isomorphic to $C_{gr}(D)$ via the morphism of modules $\psi_{II}$ defined on generators by
$$
\psi_{II}\left(\ \dessin{.65cm}{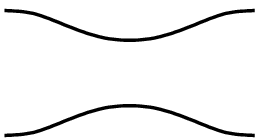}\ \right)=\dessin{.65cm}{II10}-\big( \ d_{1 \star}(\dessin{.65cm}{II10})\ \big)\otimes v_-.
$$
\end{lemme}
\begin{proof}
To prove this lemma, we split the differential into the sum of the underlying partial differentials. First we consider $d_c$ where $c \neq v,v'$. Since $d_c$ lets unchanged the middle $d$--circle in $\dessin{.6cm}{Inj1}$ and since $.\otimes v_-$ removes one $A^{-1}$--smoothed crossing labeled with a number greater than the label of $c$, we have
$$
\begin{array}{rcl}
d_c \Big( \big( d_{1 \star}(\dessin{.65cm}{II10}) \big)\otimes v_- \Big) & = & -\Big( d_c \big( d_{1 \star}(\dessin{.65cm}{II10}) \big)\Big) \otimes v_-\\[5mm]
& = & \Big( d_{1 \star} \big( d_c (\dessin{.65cm}{II10}) \big) \Big) \otimes v_-.
\end{array}
$$
And hence we have
$$
d_c \Big( \dessin{.65cm}{II10}-\big( d_{1 \star}(\dessin{.65cm}{II10}) \big)\otimes v_-  \Big) = d_c\Big(\dessin{.65cm}{II10}\Big)-\Big( d_{1 \star}\big(d_c(\dessin{.65cm}{II10})\big) \Big)\otimes v_-.
$$
Then, $d_{v'}$ and $d_v$ act respectively on  $\dessin{.6cm}{II10}$ and $\big(\ d_{1 \star}(\dessin{.6cm}{II10})\ \big)\otimes v_-$. But since $.\otimes v_-$ is a right inverse for $d_{\star 1}$, we have
$$
\begin{array}{rcl}
d_{1 \star}\Big( \dessin{.65cm}{II10} \Big) - d_{\star 1}\Big(  \big( d_{1 \star}(\dessin{.65cm}{II10}) \big)\otimes v_-\Big) & = & d_{1 \star}( \dessin{.65cm}{II10} ) - d_{1 \star}(\dessin{.65cm}{II10})\\[5mm]
& = & 0.
\end{array}
$$
These two points show not only that the subcomplex is stable, but also that $\psi_{II}$ is a chain map. Actually, for every $c \neq v,v'$, $d_c$ has the same definition when considered on $\dessin{.6cm}{isomII2}$ and on $\dessin{.6cm}{II10}$. Then, it is clearly a grading-preserving isomorphism since $\dessin{.6cm}{II10}$ has only two crossings more than $\dessin{.6cm}{isomII2}$, one which is $A$--smoothed whereas the other is $A^{-1}$--smoothed.
\end{proof}

Now, the invariance under braid-like moves of type II is reduced to proving that $C_{br}(D')$ is isomorphic to $C_1 \oplus C_2 \oplus C_3$. This can be done at the level of modules.\\
The subcomplex $C_2$ contains all the elements of the form  $\dessin{.6cm}{Inj1}$, thus we can clean all parasite terms in $\C_3$ and get
$$
C_2 \oplus C_3 \simeq \left\{\ \dessin{.8cm}{II10} \ ,\ \dessin{.8cm}{Inj1} \ ,\ \dessin{.8cm}{II11}\ \right\}.
$$
Moreover,
$$
\Im(d_v+d_{0\star}) = \left\{\ \dessin{.8cm}{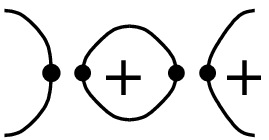} + \alpha_1\ ,\ \dessin{.8cm}{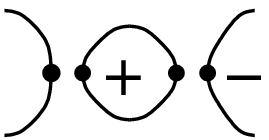} + \dessin{.8cm}{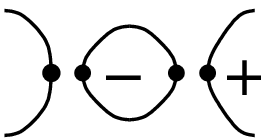} + \alpha_2\ ,\ \dessin{.8cm}{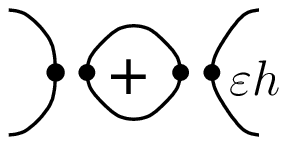} + \alpha_3 \ \right\},
$$
with $\alpha_i \in \dessin{.6cm}{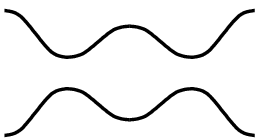}$ for $i=1,2,3$. Then, with $C_2 \oplus C_3$ we can again clean and get
$$
\Im(d_v+d_{0\star})\oplus \C_2 \oplus \C_3 \simeq \left\{\ \dessin{.8cm}{II10} \ ,\ \dessin{.8cm}{II01}\ ,\ \dessin{.8cm}{II11}\ \right\}.
$$
Finally $C_1 = \dessin{.5cm}{II00}\oplus \Im(d_v+d_{0\star})$ and this concludes the proof.

\subsubsection{Invariance under braid-like move of type III}
\label{par:InvBraidReidIII}

According to Lemma \ref{jeudemove} (\S\ref{par:BraidLikeLinks}), it is sufficient to prove the invariance between two diagrams which differ only locally by the following braid-like move of type III.
$$
\begin{array}{ccc}
D & \hspace{2cm} & D'\\
\dessin{1.7cm}{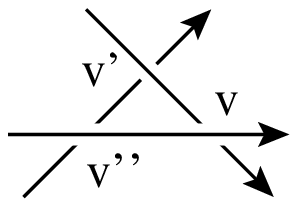} && \dessin{1.7cm}{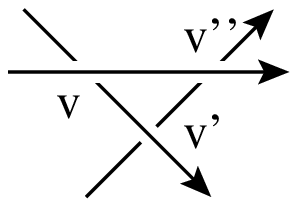}.
\end{array}
$$
We order the crossings of $D$ and $D'$ in the same fashion, letting the three crossings involved in the Reidemeister move be the last three ones. We can now consider the following ten partial smoothings of $D$ or $D'$:
$$
\begin{array}{ccc}
\begin{array}{ccc}
D_{000}& \hspace{.6cm} & D_{001}\\
\dessin{1.5cm}{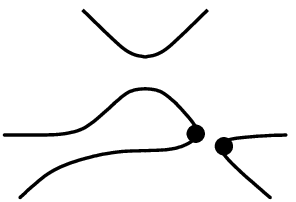} &  & \dessin{1.5cm}{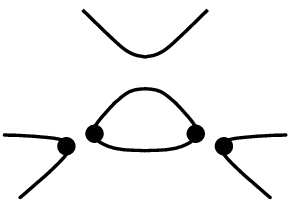}
\end{array}
& \hspace{1cm} &
\begin{array}{ccc}
D'_{000}& \hspace{.6cm} & D'_{001}\\
\dessin{1.5cm}{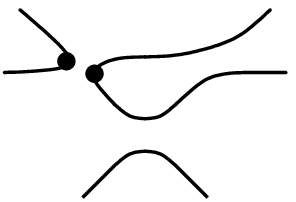} & & \dessin{1.5cm}{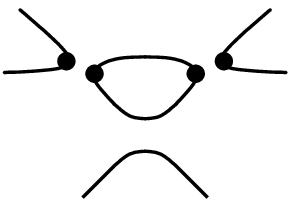}
\end{array}\\[10mm]
\begin{array}{ccc}
D_{010}& \hspace{.6cm} & D_{011}\\
\dessin{1.5cm}{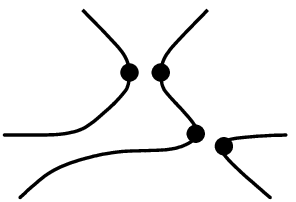} & & \dessin{1.5cm}{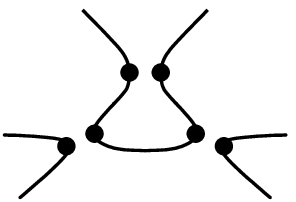}
\end{array}
& &
\begin{array}{ccc}
D'_{010}& \hspace{.6cm} & D'_{011}\\
\dessin{1.5cm}{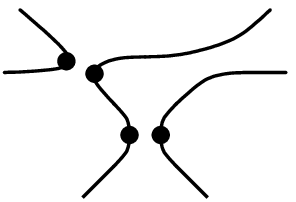} & & \dessin{1.5cm}{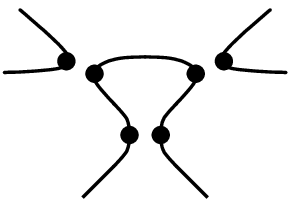}
\end{array}\\[10mm]
D_{1\bullet \bullet}&  & D'_{1\bullet \bullet}\\
\dessin{1.5cm}{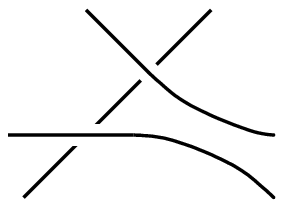} & & \dessin{1.5cm}{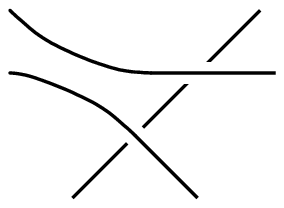}.
\end{array}
$$
It is immediate to note that $D_{1\bullet \bullet}\simeq D'_{1\bullet \bullet}$ and $D_{010}\simeq D'_{010}$. With respect to $D$, we consider also two maps
$$
\begin{array}{c}
  \func{d_{0\star 1}}{\left\{\dessin{1.16cm}{III01b}\right\}}{\left\{\dessin{1.16cm}{cad2}\right\}},\\[.7cm]
  \func{d_{00\star}}{\left\{\dessin{1.16cm}{III00b}\right\}}{\left\{\dessin{1.16cm}{III01b}\right\}}.
\end{array}
$$
which are restriction of partial differentials.\\

Now, the proof goes on similarly to the case of braid-like moves of type II, using the same arguments.

\begin{lemme}
The maps $d_{0\star 1}$ and $d_{00\star 1}$ are, respectively, surjective and injective.
\end{lemme}

\begin{lemme}
The mapping cone of
$$
\xymatrix@C=2cm{ \left\{\ \dessin{1.16cm}{III00b}\ \right\}
  \ar[r]^(.46){d_v+d_{v'}+d_{00\star}}_(.46)\sim \ar@{}|{}="Nya"
  \ar@(dl,ul)"Nya"!<-1.45cm,-.3cm>;"Nya"!<-1.45cm,.3cm>^{d_{br}^{D_{000}}}
  & \Im(d_v+d_{v'}+d_{00\star}) \ar@{}|{}="Nya2"
  \ar@(dr,ur)"Nya2"!<1.6cm,-.3cm>;"Nya2"!<1.6cm,.3cm>_{d_{br}^D}}
$$
is an acyclic subcomplex of $C_{gr}(D)$, which we denote by $C_1$.
\end{lemme}

\begin{lemme}
The mapping cone of
$$
\xymatrix@C=2cm{ \left\{\ \dessin{1.16cm}{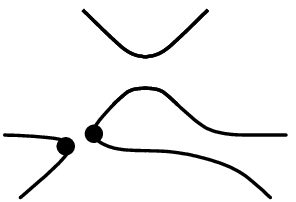}\ \right\}
  \ar[r]^(.41){d_v+d_{0\star 1}}_(.41)\sim \ar@{}|{}="Nya"
  \ar@(dl,ul)"Nya"!<-1.45cm,-.3cm>;"Nya"!<-1.45cm,.3cm>^{d_{br}^{D_{001}}}
  & \left\{\ \dessin{1.16cm}{cad2}+\dessin{1.16cm}{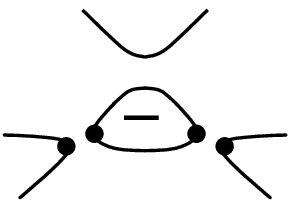}\ \right\} \ar@{}|{}="Nya2"
  \ar@(dr,ur)"Nya2"!<2.5cm,-.3cm>;"Nya2"!<2.5cm,.3cm>_{d_{br}^D}}
$$
is an acyclic subcomplex of $C_{gr}(D')$, which we denote by $C_2$.
\end{lemme}

Now, we define the morphism of modules
$$
\begin{array}{rccc}
.\otimes v_-: & \dessin{1.16cm}{cad2} & \longrightarrow & \dessin{1.16cm}{III01b}\\[6mm]
& \dessin{.9425cm}{cad2} & \mapsto & -\dessin{.9425cm}{surjIII1}.
\end{array}
$$
It is grading-preserving and a right inverse for $d_{0\star 1}$.

\begin{lemme}
  The chain complex
  $$
  \xymatrix@C=2cm{ \left\{\ \dessin{1.16cm}{cad1}-\big( d_{1 \star}(\dessin{1.16cm}{cad1}) \big)\otimes v_-\ ,\ \dessin{1.16cm}{passeb}\ \right\}
    \ar@{}|{}="Nya"
    \ar@(dl,ul)"Nya"!<-4.5cm,-.3cm>;"Nya"!<-4.5cm,.3cm>^{d_{br}^D}}
  $$
is a subcomplex of $C_{gr}(D )$, which we denote by $C_3$
\end{lemme}

Moreover, at the level of modules, $C(D)$ is isomorphic to $ C_1 \oplus C_2 \oplus C_3$.\\

With the same reasoning on $D'$, we get:
$$
C(D') \simeq C'_1 \oplus C'_2 \oplus C'_3
$$
where $C'_1$ and $C'_2$ are acyclic and
$$
C'_3 = \left\{\ \dessin{1.16cm}{cae1}-\big( d_{01 \star}(\dessin{1.16cm}{cae1}) \big)\otimes v_-\ \right\}\oplus \dessin{1.16cm}{passe}.
$$

Finally, we conclude using to the following lemma:
\begin{lemme}
The chain complexes $C_3$ and $C'_3$ are isomorphic via the morphism of modules $\psi_{III}$ which is the identity on $\dessin{.85cm}{passeb}$ and is defined otherwise by
$$
\psi\left(\ \dessin{1.16cm}{cad1}-\big( d_{01 \star}(\dessin{1.16cm}{cad1}) \big)\otimes v_-\ \right)=\dessin{1.16cm}{cae1}-\big( d_{01 \star}(\dessin{1.16cm}{cae1}) \big)\otimes v_-.
$$
\end{lemme}

\subsection{Star-like invariance}
\label{ssec:InvStarLikeHomology}

\begin{theo}
  The homology groups $\HH_{st}$ are invariant under star-like isotopies.
\end{theo}

The star-like case is similar in most respects to the braid-like one.

\subsubsection{Invariance under star-like moves}
\label{par:InvStarReid}

With the proviso of adding Seifert points, the proof for invariance under star-like moves of type II can be copied out from the braid-like one.\\

This is also true for the following star-like move of type III:
$$
\dessin{1.7cm}{IIIb2}\ \longleftrightarrow \ \dessin{1.7cm}{IIIb1}
$$
as soon as we have noticed that
$$
 \dessin{1.16cm}{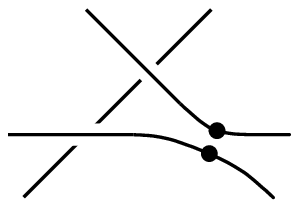}\ = \ \vcenter{\hbox{\xymatrix{\dessin{.9425cm}{III3s1} \ar[r] \ar[d] & \dessin{.9425cm}{III1s1} \ar[d]\\\dessin{.9425cm}{III2s3} \ar[r] & \dessin{.9425cm}{III3s2}}}} \ \ \cong \ \ \vcenter{\hbox{\xymatrix{\dessin{.9425cm}{III31} \ar[r] \ar[d] & \dessin{.9425cm}{III14} \ar[d]\\\dessin{.9425cm}{III21} \ar[r] & \dessin{.9425cm}{III32}}}} \ = \ \dessin{1.16cm}{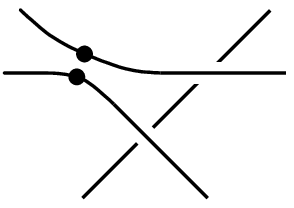}.
$$

Unfortunately, the invariance under the second star-like move of type III cannot be acheived in this way. Moreover, there is no way to connect it to the other one using only star-like moves of type II. However, these two moves are mirror image one to the other. The proof can thus be completed by using the following proposition.
\begin{prop}\label{KhCohomologie}
  For any diagram $D$ and any integers $i$, $j$ and $k$, the homology groups $\HH_i^{j,k}(D)$ and $\big(\HH_{-i}^{-j,-k}(\overline{D})/\T_{-i}^{-j,-k}(\overline{D})\big)\oplus \T_{-i+1}^{-j,-k}(\overline{D})$ are isomorphic where $\T(D)$  denotes the torsion part of $\HH_{st}(D)$ and $\overline{D}$ the mirror image of $D$.
\end{prop}

The proof is given in paragraphs \ref{par:KhoCodiff}---\ref{par:KhoDualMirror}.

\subsubsection{Invariance under pairs of Reidemeister moves of type I}
\label{par:InvStarReidI}

Proposition \ref{PairsOfI} (\S\ref{par:StarLikeJonesProp}) states that star-like Jones polynomial is invariant under pairs of adjacent Reidemeister moves of type I. This property remains true for star-like Khovanov homology.\\

To prove it, we consider the three following diagrams which differ only locally by Reidemeister moves of type I and the occurence of a Seifert point.
$$
\begin{array}{ccccc}
D' & \hspace{1.5cm} & D & \hspace{1.5cm} & D''\\
\dessin{1.5cm}{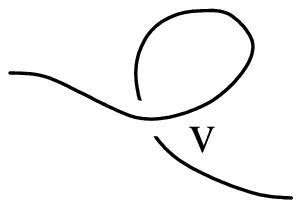} & & \dessin{1.5cm}{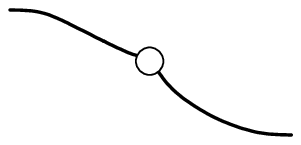} & & \dessin{1.5cm}{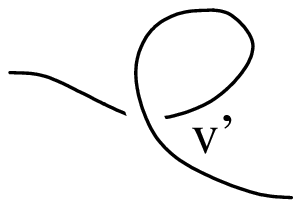}.
\end{array}
$$
We order the crossings of $D$, $D'$ and $D''$ in the same fashion, letting $v$ and $v'$ be the last ones of $D'$ and $D''$. Now we consider the following four partial smoothings of $D$ or $D''$:
$$
\begin{array}{cccl}
D'_{0} & \hspace{1.5cm}  & D''_{0}&\\
\dessin{1.36cm}{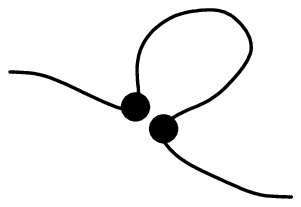} & & \dessin{1.36cm}{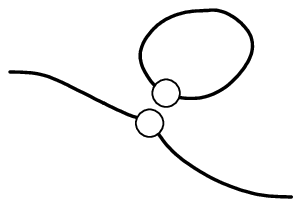}&\\[5mm]
D'_{1} & & D''_{1}&\\
\dessin{1.36cm}{RI11} & & \dessin{1.36cm}{RI10}&,\\
\end{array}
$$
and the two following restrictions of partial differentials:
$$
\begin{array}{c}
  \func{d^{D'}_\star}{\left\{\dessin{1cm}{RI10}\right\}}{\left\{\dessin{1cm}{RI11}\right\}},\\[.5cm]
  \func{d^{D''}_\star}{\left\{\dessin{1cm}{RI11}\right\}}{\left\{\dessin{1cm}{RI10}\right\}}.
\end{array}
$$

\begin{lemme}
The morphisms $d^{D'}_{\star}$ and $d^{D''}_{\star}$ are respectively injective and surjective.
\end{lemme}

Here again, all proofs can be argued like in the proofs of paragraph \ref{par:InvBraidReidII}.

\begin{lemme}
The mapping cone of
$$
\xymatrix@C=2cm{ \left\{\ \dessin{1cm}{RI10}\ \right\}
  \ar[r]^{d^{D'}_{\star}}_\sim \ar@{}|{}="Nya"
  \ar@(dl,ul)"Nya"!<-1.4cm,-.3cm>;"Nya"!<-1.4cm,.3cm>^{d_{br}^{D'_0}}
  & \left\{\ \dessin{1cm}{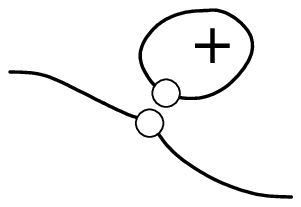}\ \right\} \ar@{}|{}="Nya2"
  \ar@(dr,ur)"Nya2"!<1.4cm,-.3cm>;"Nya2"!<1.4cm,.3cm>_{d_{br}^{D'}}}
$$
is an acyclic subcomplex of $C_{gr}(D')$, which we denote by $C'_1$.
\end{lemme}

\begin{lemme}
The mapping cone of
$$
\xymatrix@C=2cm{ \left\{\ \dessin{1cm}{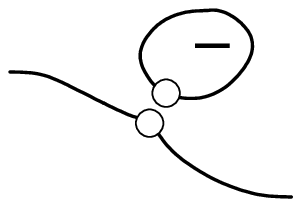}\ \right\}
  \ar[r]^{d^{D'}_{\star}}_\sim \ar@{}|{}="Nya"
  \ar@(dl,ul)"Nya"!<-1.4cm,-.3cm>;"Nya"!<-1.4cm,.3cm>^{d_{br}^{D''_0}}
  & \left\{\ \dessin{1cm}{RI10}\ \right\} \ar@{}|{}="Nya2"
  \ar@(dr,ur)"Nya2"!<1.4cm,-.3cm>;"Nya2"!<1.4cm,.3cm>_{d_{br}^{D''}}}
$$
is an acyclic subcomplex of $C_{gr}(D')$, which we denote by $C''_1$.
\end{lemme}

\begin{lemme}
The chain complexes
$$
\xymatrix@C=2cm{ \left\{\ \dessin{1cm}{RI11m}\ \right\}
   \ar@{}|{}="Nya"
  \ar@(dl,ul)"Nya"!<-1.4cm,-.3cm>;"Nya"!<-1.4cm,.3cm>^{d_{br}^{D'}}}
$$
and
$$
\xymatrix@C=2cm{ \left\{\ \dessin{1cm}{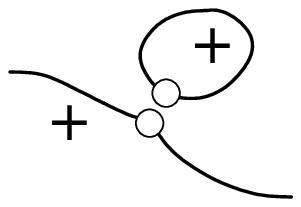}\ ,\ \dessin{1cm}{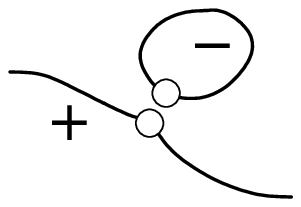} + \dessin{1cm}{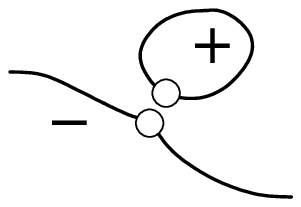}\ \right\}
   \ar@{}|{}="Nya"
  \ar@(dl,ul)"Nya"!<-3.45cm,-.3cm>;"Nya"!<-3.5cm,.3cm>^{d_{br}^{D''}}}
$$
are subcomplexes of, respectively, $C(D')$ and $C(D'')$, which we denote by $C'_2$ and $C''_2$. Both of them are isomorphic to $C(D)$.
\end{lemme}

Moreover, at the level of modules, $C(D')$ and $C(D'')$ are respectively isomorphic to $C'_1 \oplus C'_2$ and $C''_1 \oplus C''_2$.\\

Finally, the proof is completed by noticing that the operation of adding two adjacent Seifert points on a diagram is a blind operation for the Khovanov construction.

\subsection{Restricted Khovanov cohomology}
\label{ssec:KhCohomologie}

All the statements in this section hold straighforwardly in the braid-like context.\\

\subsubsection{Khovanov codifferential}
\label{par:KhoCodiff}

Just like the differential $d_{st}^D$ was defined for every diagram $D$, we can set a codifferential $\func{d^{st}_D}{C_{st}(D)}{C_{st}(D)}$ of tridegree $(1,0,0)$ defined on generators by
$$
d^{st}_D(S) =  \sum_{\substack{v\textrm{ crossing}\\A\textrm{--smoothed in }S}} d^v(S)
$$
where the partial differential $d^v$ is defined by
$$
d^v(S) = (-1)^{t^+_{v,S}} \sum_{S'\textrm{enhanced state of }D} [S',S]_v S',
$$
with $t^+_{v,S}$ the number of $A$--smoothed crossings in $S$ labeled with numbers greater than the label of $v$.

\subsubsection{Duality}
\label{par:KhoDuality}

\begin{lemme}
  The bigraded cochain complex $(C_{st}(D),d^{st}_D)$ is isomorphic to the dual cochain complex $(C^*_{st}(D),d_{st}^*)$ associated to $(C_{st}(D),d_{st}^D)$.
\end{lemme}
\begin{proof}
  We set the map $\func{\psi}{C_{st}(D)}{C^*_{st}(D)}$ defined on generators by $\psi(S)=(-1)^{u(S)}S^*$ where $S^*$ is the dual element of $S$ associated to the canonical base and
$$
u(S)=\sum_{\substack{v_i\textrm{ crossing}\\A\textrm{--smoothed in }S}} n-i.
$$
Since $\psi$ is obviously an isomorphism of modules, we only need to check that it commutes with the differentials.\\
Let $S$ and $S_0$ be enhanced states. In one hand, we have

\begin{eqnarray*}
\psi(d^{st}_D(S))(S_0) & = & \sum_{\substack{v \textrm{ crossing in }D\\S' \textrm{ enhanced state}}} (-1)^{t^+_{S,v}+u(S')}[S',S]_v S'^*(S_0)\\
& = & \sum_{v \textrm{ crossing in }D}(-1)^{t^+_{S,v}+u(S_0)}[S_0,S]_v,
\end{eqnarray*}
and in the other
\begin{eqnarray*}
d_{st}^*(\psi(S))(S_0) & = & \sum_{\substack{v \textrm{ crossing in}D\\S' \textrm{ enhanced state}}} (-1)^{t^-_{S_0,v}+u(S)}[S_0,S']_v S^*(S')\\
& = & \sum_{v \textrm{ crossing in }D}(-1)^{t^-_{S_0,v}+u(S)}[S_0,S]_v.
\end{eqnarray*}

Moreover, if $[S_0,S]_v\neq 0$, then $S$ and $S_0$ differ only on $v$. We have thus $u(S_0)-u(S)=n-i$ for some $i$. But on the other side
$$
t^+_{S,v}+t^-_{S_0,v}=t^+_{S,v}+t^-_{S,v}=n-i.
$$
As a matter of fact, $(t^+_{S,v}+u(S_0))+(t^-_{S_0,v}+u(S))=2(n-i+u(S))$ and $t^+_{S,v}+u(S_0)$ and $t^-_{S_0,v}+u(S)$ have the same parity.
\end{proof}

\subsubsection{Duality for mirror image}
\label{par:KhoDualMirror}

\begin{lemme}
The map
$$
\phi : \vcenter{\hbox{\xymatrix{
\cdots \ar[r]^(.37){d^D_{st}} & C_i^{j,k}(D) \ar[r]^(.45){d^D_{st}} \ar[d]^\phi & \C_{i-1}^{j,k}(D) \ar[r]^(.58){d^D_{st}} \ar[d]^\phi& \cdots\\
\cdots \ar[r]^(.32){d_{\overline{D}}^{st}} & \C^{-i}_{-j,-k}(\overline{D}) \ar[r]^(.47){d_{\overline{D}}^{st}} & \C'^{-i+1}_{-j,-k}(\overline{D}) \ar[r]^(.66){d_{\overline{D}}^{st}} & \cdots,\\
}}}
$$
defined on any enhanced state $S$ by inversing the labels of all circles, is a chain complex isomorphism.
\end{lemme}
\begin{proof}
  Any enhanced state $S$ of $D$ can be seen as an enhanced state
  for $\overline{D}$ with opposed smoothing at every crossing.
  Moreover, since $w(\overline{D})=-w(D)$ and since the map $\phi$
  inverses all the labels, this map is well defined as a graded
  isomorphism of modules. Thus, it is sufficient to verify that the diagram is
  commutative.  Indeed, for any enhanced state $S$
  $$
  \phi(d^D_{st}(S)) = \sum_{\substack{v \textrm{ crossing in }D\\S' \textrm{
        enhanced state of }D}} (-1)^{t^-_{S,v}}[S,S']_v \phi(S'),
  $$
  and
  \begin{eqnarray*}
    d_{\overline{D}}^{st}(\phi(S)) & = & \sum_{\substack{v \textrm{ crossing in }\overline{D}\\S' \textrm{ enhanced state of }\overline{D}}} (-1)^{t^+_{\phi(S),v}}[S',\phi(S)]_v S'\\
    & = & \sum_{\substack{v \textrm{ crossing in }\overline{D}\\\phi(S') \textrm{ enhanced state of }\overline{D}}} (-1)^{t^+_{\phi(S),v}}[\phi(S'),\phi(S)]_v \phi(S'),
  \end{eqnarray*}
  since $\phi$ is a bijection between enhanced states of $D$ and
  $\overline{D}$. But, by definition, $t^-_{S,v}=t^+_{\phi(S),v}$ and by
  construction $[\phi(S'),\phi(S)]_v = [S,S']_v$.
\end{proof}

Finally Proposition \ref{KhCohomologie} (\S\ref{par:InvStarReid}) is proved by using the Proposition \ref{UCT} (\S\ref{par:Cohomologies}).

%%% Local Variables: 
%%% mode: latex
%%% TeX-master: "These"
%%% End: 

% Conclusion
\chapter*{Conclusion}
\label{chap:Conclusion}
\addcontentsline{toc}{part}{Conclusion}

The work initiated in this thesis leaves open many questions.
As a conclusion, we try to draw up a non exhaustive list of them.\\

Concerning the generalization of link Floer homology, the property
\begin{eqnarray}
\widehat{HF}\left(\ \dessin{1.1cm}{Torsion}\right)\cong 0
\label{eq:Torsion}
\end{eqnarray}
should be proved as a general fact.\\
In a second time, a possible finite type behavior should be examined. In particular, the filtrated inverse of degree $1$ for  the switch morphism should lead to the extraction of acyclic subcomplexes. Hopes for full acyclicity under any simple conditions are certainly too optimistic.
Computations lead to the conjecture that a fully singular link, \ie a link which has a diagram with only double points as crossings, is acyclic when computed with the orientation induced on the double points by the planar orientation.\\
More generally, computations seems to indicate that some orientations for doubles points are more relevant than other.
Finally, it should be interesting to understand the possible links between our construction and the construction given in \cite{OSSing}. Notably, the latter leads to a description of link Floer homology in the spirit of Khovanov-Rozansky homology \cite{OSCube}. {\it A priori}, a third generalization of link Floer homology to singular links can thus be defined using the approach of N. Shirokova and B. Webster \cite{ShiroWeb}. Nevertheless, as shown in section \ref{sec:Sum}, being the mapping cone of a map between Seifert smoothing can exclude the (\ref{eq:Torsion}) relation.\\

Concerning the second part, there are many questions on restricted links which are still open. Above all, the relations between braid-like and usual links should be clarified; by exhibiting, in particular, a minimal set of moves which generates, associated with the braid-like moves, all the diagrams isotopies. We can wonder if the Reidemeister moves of type I are sufficient or, equivalently, if the Reidemeister moves corresponding to the singularity $\dessin{.5cm}{sing1bis}$ are really necessary. This would lead to a new diagrammatical approach of links which may be adapted, for instance, to genus questions and minimality of Seifert surfaces since braid-like moves are precisely moves which preserve the Seifert state.\\
In this thesis, we only categorify a lightened version of the braid-like (and star-like) Jones polynomial. A possible continuation would be a categorification of the full bracket. Finally, Khovanov homology can be seen as a particular case of Khovanov-Rozansky. Nevertheless, it differs from other cases by its invariance under non braid-like Reidemeister moves of type II. Our refinement precisely insists on the internal mechanisms of Khovanov homology at the time of these moves. It is thus likely to throw some light on this peculiar behavior.

%%% Local Variables: 
%%% mode: latex
%%% TeX-master: "These"
%%% End: 

\part*{Appendices}
\addcontentsline{toc}{part}{Appendices}
\appendix

% Decomposition en sous complexes
\chapter{Decomposition in subcomplexes}
\label{appendix:Subcomplexes}

This appendix is devoted to the details of the proof of Proposition \ref{AcyclicSubcomplexes} (\S\ref{par:LastFiltrationExit}).\\

The strategy is to enumerate subcomplexes described as mapping cones of isomorphisms. They are thus all acyclic.\\

We use the following notation:
\begin{itemize}
\item[-] sets of generators are described by the points they have in common; forbidden location for points are crossed;
\item[-] when summing sets of dots, we mean the sums of generators obtained by using the same complementary set of dots; such elements should be understood as a leading term and a remainder;
\item[-] minus signs should be replaced by $-\e(r_1)\e(r_2)$ where $r_1$ and $r_2$ are the polygons connecting the two summands to a same generator;
\item[-] modules are described by a set of generators;
\item[-] differentials are described by curved arrows illustrated by the rectangles we are summing over;
\item[-] chain isomorphisms simply associate generators of each side obtained by using the same complementary set of points; however, they are illustrated by the polygon which connects the two leading terms.
\end{itemize}

Now, using the description of $\widetilde{d}_\gr$ initiated in Proposition \ref{AcyclicDiff} (\S\ref{par:AnotherCrushingFiltration}) and completed in paragraph \ref{par:LastFiltrationExit}, it is straighforward to check that they are all subcomplexes of $\big( \widetilde{CL}(G_s),\widetilde{d}_\gr\big)$.\\

Then, to prove that their direct sum is isomorphic to $\widetilde{CL}^-(G_s)$, it is sufficient to work at the level of modules. If the usual generators of $\widetilde{CL}^-(G_s)$ are indexed by their order of appearance in the subcomplexes below, then the linear map which sends every summand to its leading term has a unit upper triangular matrix. It is then invertible. To prove it, subcomplexes are indexed by integers, and for all of them, we give between parenthesis the indices of the subcomplexes where the terms of the remainder appear for the first time.\\
Finally, there is no difficulty in checking that every generator appears exactly once as a leading term.\\

\vspace{-.5cm}

$$
\xymatrix@C=2cm{ \left\{\dessin{1.95cm}{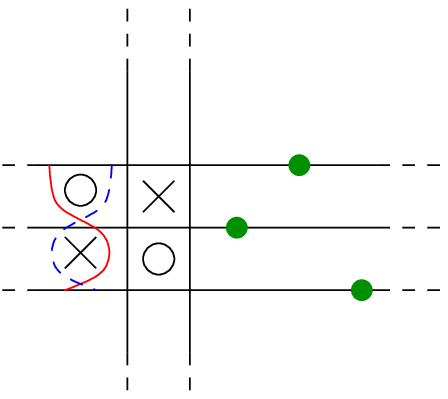}\right\}
  \ar[r]^{\dessinH{1.2cm}{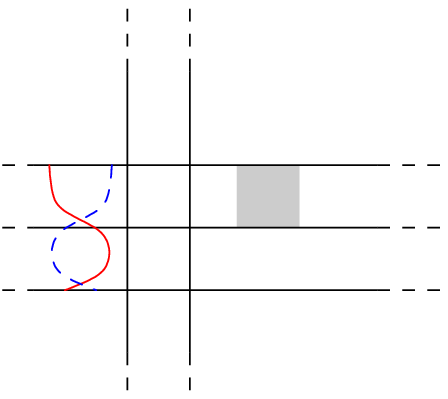}}_\sim \ar@{}|{}="Nya"
  \ar@(dl,ul)"Nya"!<-1.8cm,-.2cm>;"Nya"!<-1.8cm,.4cm>^{\dessin{1.2cm}{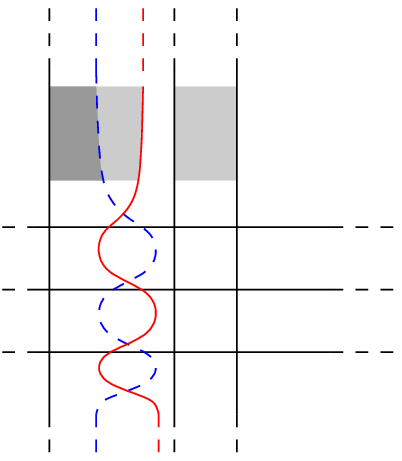}}
  & \left\{\dessin{1.95cm}{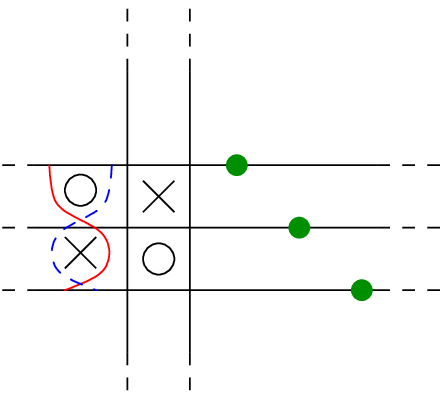}\right\}\ar@{}|{}="Nya2"
  \ar@(dr,ur)"Nya2"!<1.7cm,-.2cm>;"Nya2"!<1.7cm,.4cm>_{\dessin{1.2cm}{Diff1}}}
$$
\vspace{-1.3cm}
\begin{flushright}
  1 {\footnotesize (-)}
\end{flushright}
\vspace{.55cm}
$$
\xymatrix@C=2cm{ \left\{\dessin{1.95cm}{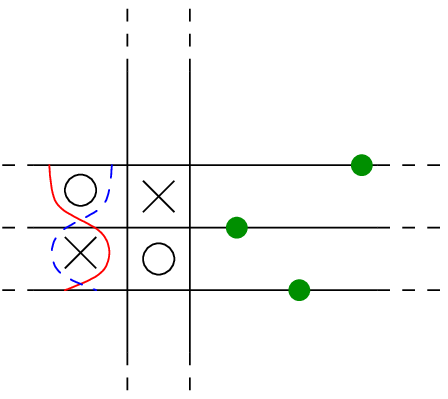}\right\}
  \ar[r]^(.4){\dessinH{1.2cm}{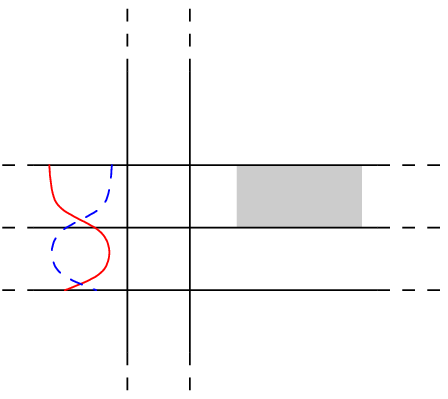}}_(.4)\sim \ar@{}|{}="Nya"
  \ar@(dl,ul)"Nya"!<-1.8cm,-.2cm>;"Nya"!<-1.8cm,.4cm>^{\dessin{1.2cm}{Diff1}}
  &
  \left\{\dessin{1.95cm}{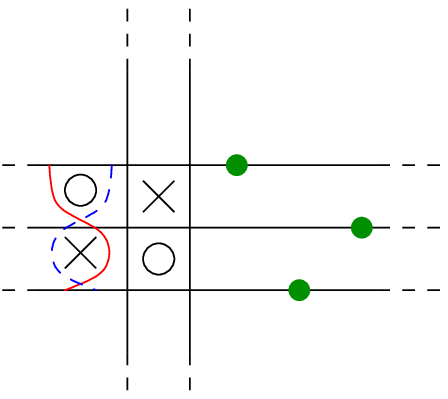}+\dessin{1.95cm}{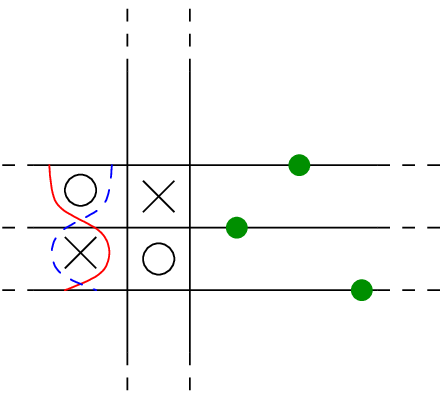}\right\}\ar@{}|{}="Nya2"
  \ar@(dr,ur)"Nya2"!<3.1cm,-.2cm>;"Nya2"!<3.1cm,.4cm>_{\dessin{1.2cm}{Diff1}}}
$$
\vspace{-1cm}
\begin{flushright}
  2 {\footnotesize (1)}
\end{flushright}
\vspace{.25cm}
$$
\xymatrix@C=2cm{ \left\{\dessin{1.95cm}{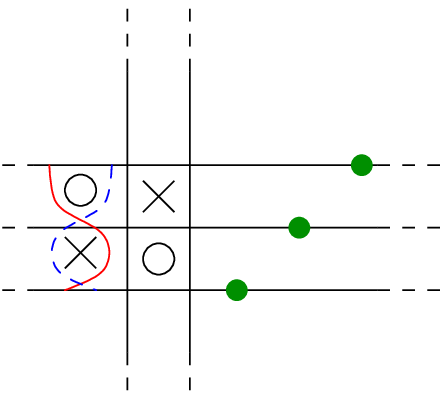}\right\}
  \ar[r]^(.4){\dessinH{1.2cm}{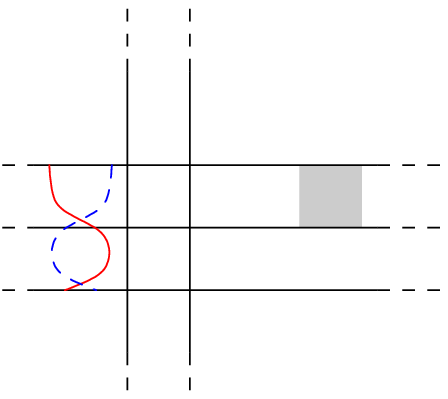}}_(.4)\sim \ar@{}|{}="Nya"
  \ar@(dl,ul)"Nya"!<-1.8cm,-.2cm>;"Nya"!<-1.8cm,.4cm>^{\dessin{1.2cm}{Diff1}}
  &
  \left\{\dessin{1.95cm}{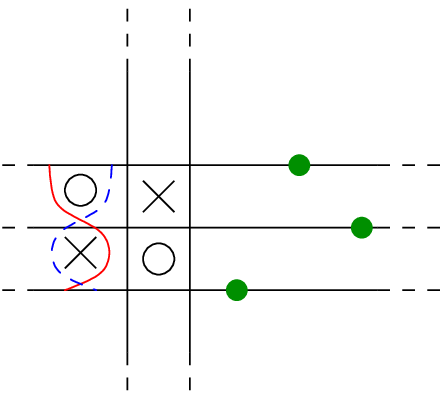}+\dessin{1.95cm}{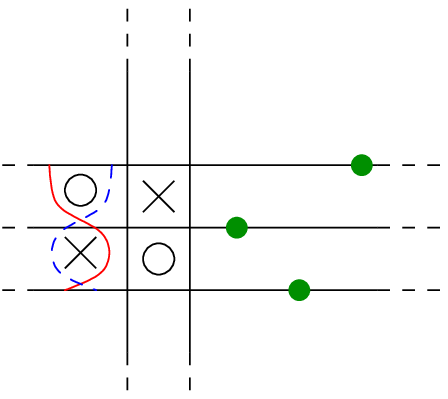}\right\}\ar@{}|{}="Nya2"
  \ar@(dr,ur)"Nya2"!<3.1cm,-.2cm>;"Nya2"!<3.1cm,.4cm>_{\dessin{1.2cm}{Diff1}}}
$$
\vspace{-1cm}
\begin{flushright}
  3 {\footnotesize (2)}
\end{flushright}
\vspace{.25cm}

$$
\xymatrix@C=2cm{ \left\{\dessin{1.95cm}{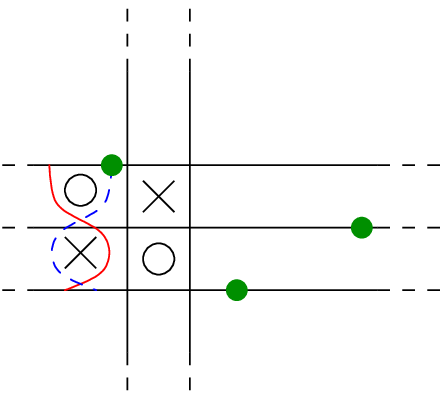}\right\}
  \ar[r]^{\dessinH{1.2cm}{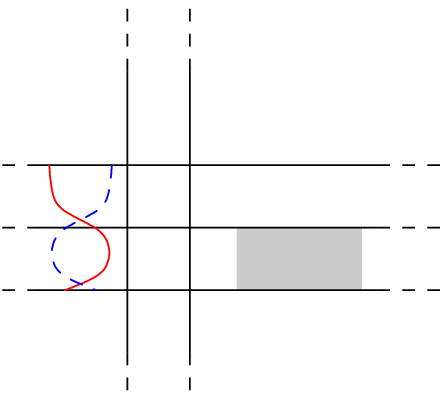}}_\sim \ar@{}|{}="Nya"
  \ar@(dl,ul)"Nya"!<-1.8cm,-.2cm>;"Nya"!<-1.8cm,.4cm>^{\dessin{1.2cm}{Diff1}}
  & \left\{\dessin{1.95cm}{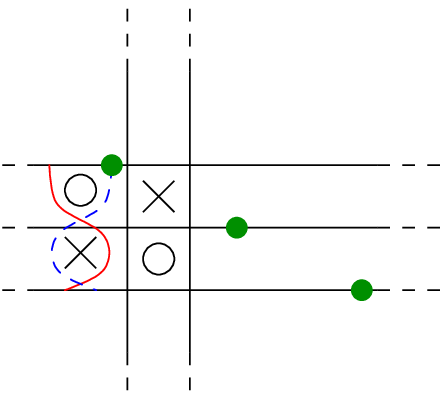}\right\}\ar@{}|{}="Nya2"
  \ar@(dr,ur)"Nya2"!<1.7cm,-.2cm>;"Nya2"!<1.7cm,.4cm>_{\dessin{1.2cm}{Diff1}}}
$$ 
\vspace{-1.3cm}
\begin{flushright}
  4 {\footnotesize (-)}
\end{flushright}
\vspace{.55cm}
$$
\xymatrix@C=2cm{ \left\{\dessin{1.95cm}{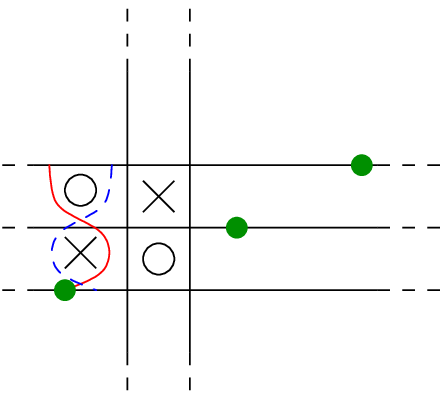}\right\}
  \ar[r]^{\dessinH{1.2cm}{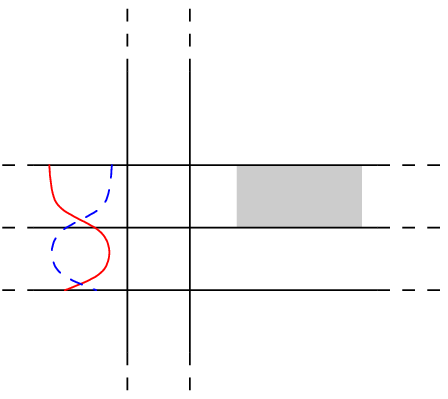}}_\sim \ar@{}|{}="Nya"
  \ar@(dl,ul)"Nya"!<-1.8cm,-.2cm>;"Nya"!<-1.8cm,.4cm>^{\dessin{1.2cm}{Diff1}}
  & \left\{\dessin{1.95cm}{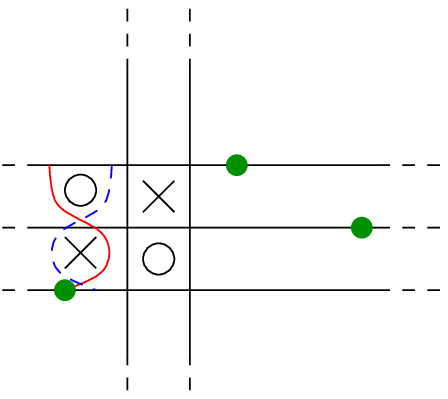}\right\}\ar@{}|{}="Nya2"
  \ar@(dr,ur)"Nya2"!<1.7cm,-.2cm>;"Nya2"!<1.7cm,.4cm>_{\dessin{1.2cm}{Diff1}}}
$$ 
\vspace{-1.3cm} 
\begin{flushright}
  5 {\footnotesize (-)}
\end{flushright}
\vspace{.55cm}
$$
\xymatrix@C=2cm{ \left\{\dessin{1.95cm}{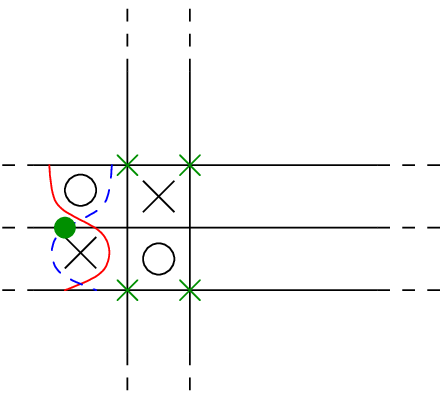}\right\}
  \ar[r]^{\dessinH{1.2cm}{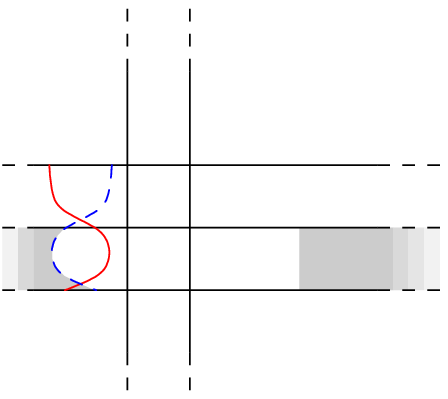}}_\sim \ar@{}|{}="Nya"
  \ar@(dl,ul)"Nya"!<-1.8cm,-.2cm>;"Nya"!<-1.8cm,.4cm>^{\dessin{1.2cm}{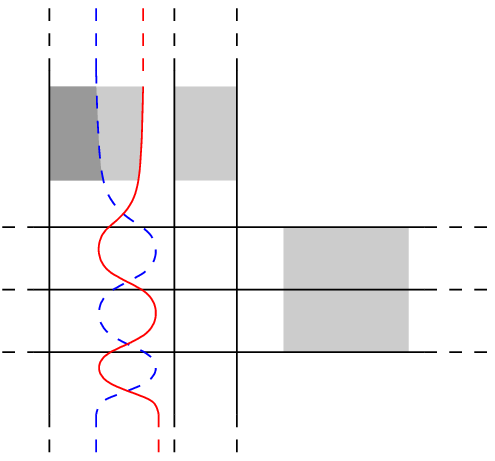}}
  & \left\{\dessin{1.95cm}{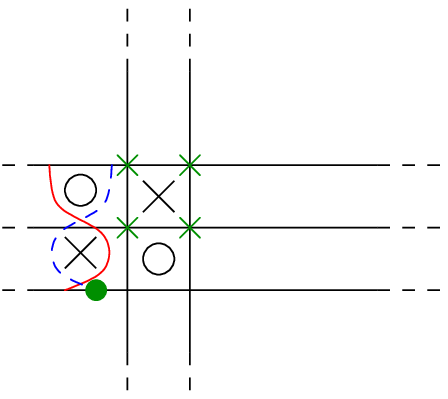}\right\}\ar@{}|{}="Nya2"
  \ar@(dr,ur)"Nya2"!<1.7cm,-.2cm>;"Nya2"!<1.7cm,.4cm>_{\dessin{1.2cm}{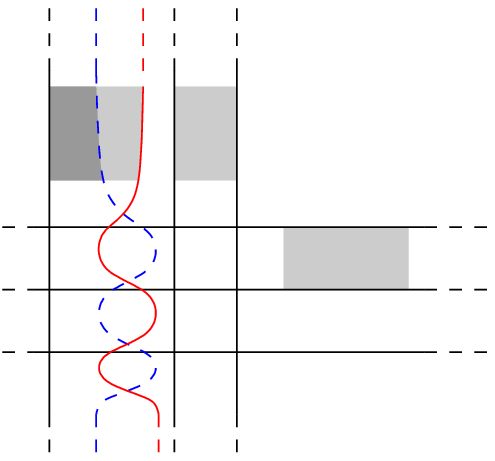}}}
$$ 
\vspace{-1.3cm} 
\begin{flushright}
  6 {\footnotesize (-)}
\end{flushright}
\vspace{.55cm}
$$
\hspace{-1cm}
\xymatrix@C=2cm{ \left\{\dessin{1.95cm}{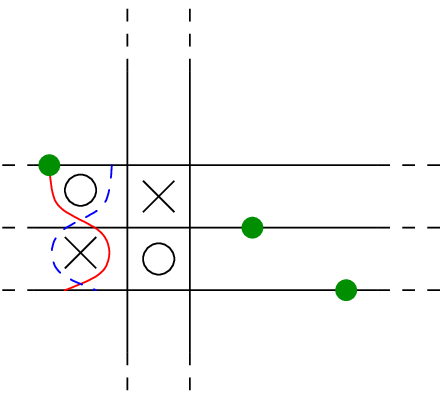}\right\}
  \ar[r]^(.32){\dessinH{1.2cm}{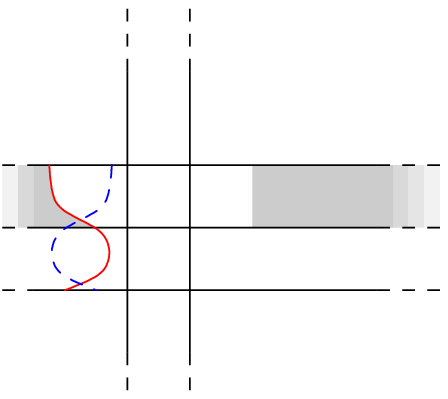}}_(.32)\sim \ar@{}|{}="Nya"
  \ar@(dl,ul)"Nya"!<-1.8cm,-.2cm>;"Nya"!<-1.8cm,.4cm>^{\dessin{1.2cm}{Diff1}}
  &
  \left\{\dessin{1.95cm}{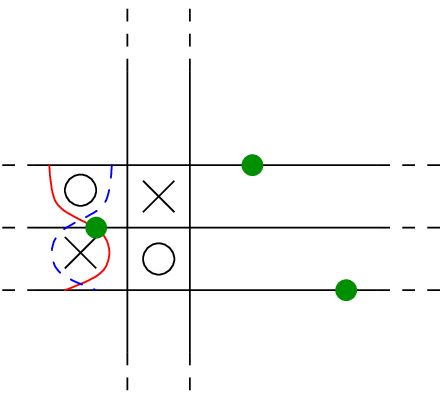}+\dessin{1.95cm}{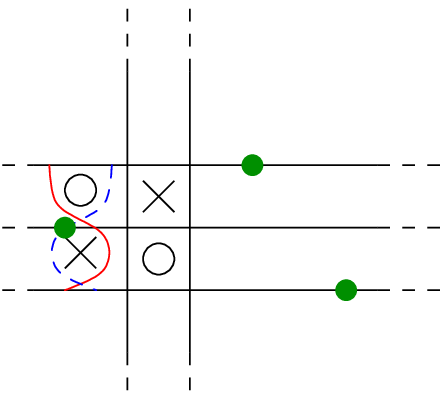}+\dessin{1.95cm}{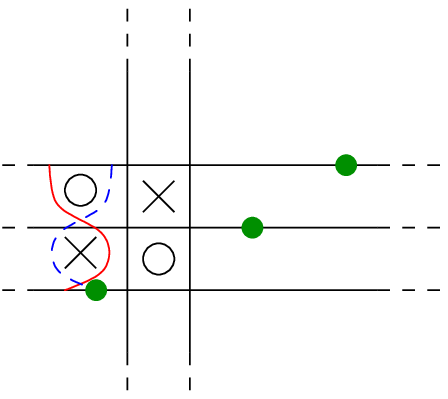}\right\}\ar@{}|{}="Nya2"
  \ar@(dr,ur)"Nya2"!<4.5cm,-.2cm>;"Nya2"!<4.5cm,.4cm>_{\dessin{1.2cm}{Diff1}}}
$$ 
\vspace{-1cm} 
\begin{flushright}
  7 {\footnotesize (6)}
\end{flushright}
\vspace{.25cm}
$$
\hspace{-1cm}
\xymatrix@C=2cm{ \left\{\dessin{1.95cm}{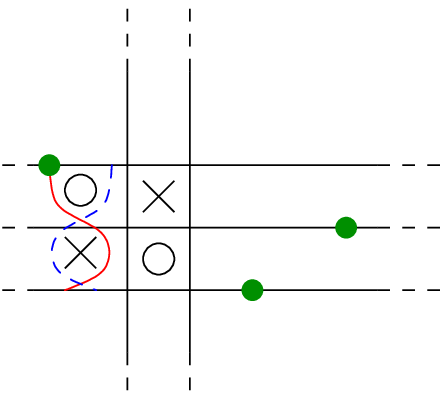}\right\}
  \ar[r]^(.32){\dessinH{1.2cm}{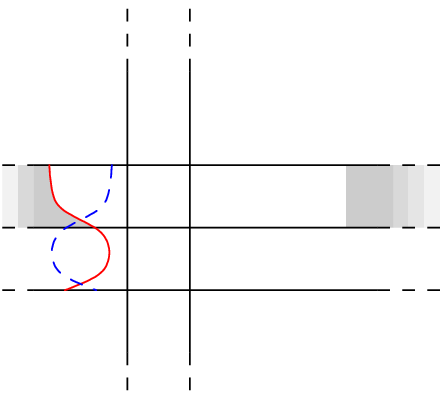}}_(.32)\sim \ar@{}|{}="Nya"
  \ar@(dl,ul)"Nya"!<-1.8cm,-.2cm>;"Nya"!<-1.8cm,.4cm>^{\dessin{1.2cm}{Diff1}}
  &
  \left\{\dessin{1.95cm}{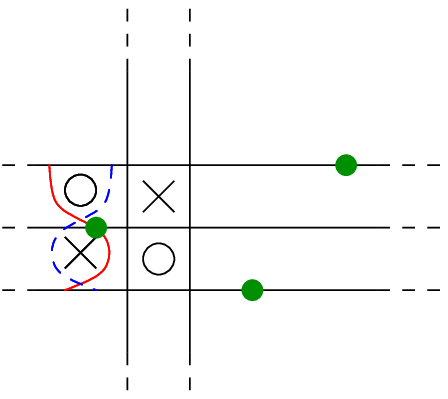}+\dessin{1.95cm}{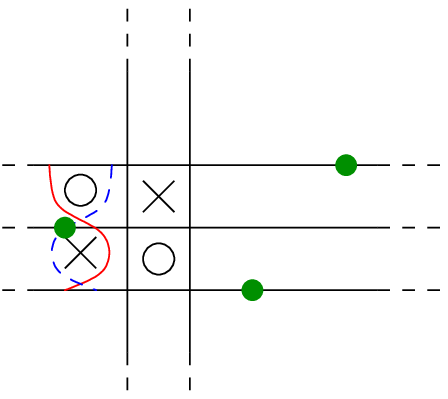}+\dessin{1.95cm}{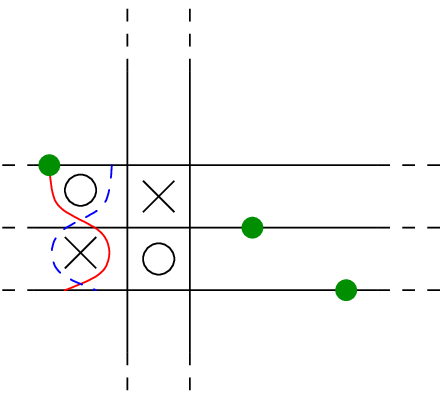}\right\}\ar@{}|{}="Nya2"
  \ar@(dr,ur)"Nya2"!<4.5cm,-.2cm>;"Nya2"!<4.5cm,.4cm>_{\dessin{1.2cm}{Diff1}}}
$$ 
\vspace{-.6cm} 
\begin{flushright}
  8 {\footnotesize (6,7)}
\end{flushright}
\vspace{.25cm}

\newpage

\vspace{-.5cm}

$$
\xymatrix@C=2cm{
  \left\{\dessin{1.95cm}{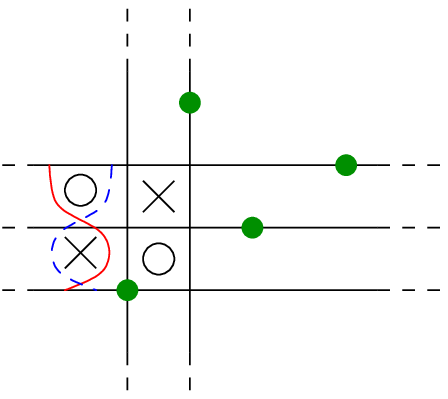}\right\} \ar[r]^{\dessinH{1.2cm}{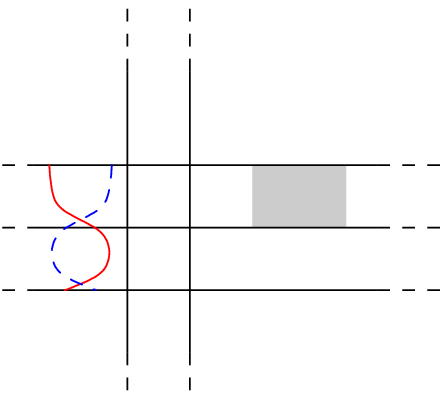}}_\sim \ar@{}|{}="Nya"  \ar@(dl,ul)"Nya"!<-1.8cm,-.2cm>;"Nya"!<-1.8cm,.4cm>^{\dessin{1.2cm}{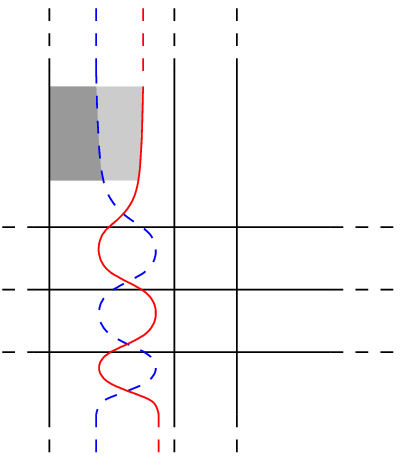}} & \left\{\dessin{1.95cm}{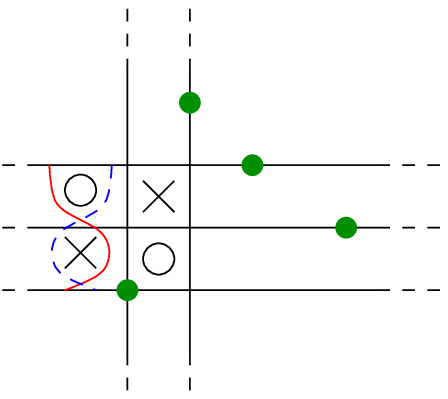}\right\}\ar@{}|{}="Nya2"  \ar@(dr,ur)"Nya2"!<1.7cm,-.2cm>;"Nya2"!<1.7cm,.4cm>_{\dessin{1.2cm}{Diff2}}}
$$ 
\vspace{-1.3cm} 
\begin{flushright}
  9 {\footnotesize (-)}
\end{flushright}
\vspace{.55cm}

$$
\xymatrix@C=2cm{
  \left\{\dessin{1.95cm}{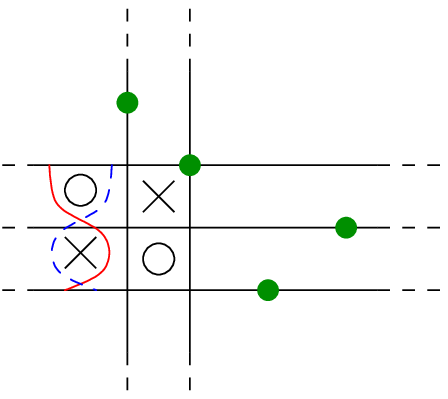}\right\} \ar[r]^{\dessinH{1.2cm}{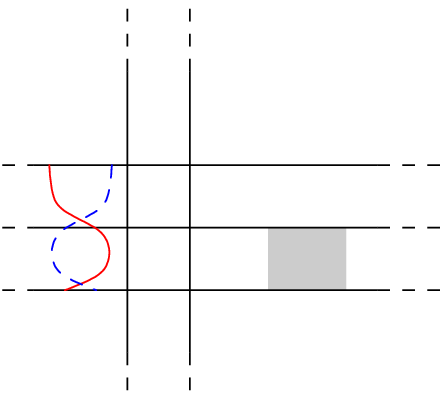}}_\sim \ar@{}|{}="Nya"  \ar@(dl,ul)"Nya"!<-1.8cm,-.2cm>;"Nya"!<-1.8cm,.4cm>^{\dessin{1.2cm}{Diff2}} & \left\{\dessin{1.95cm}{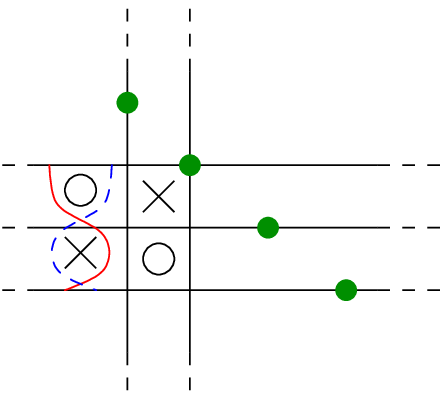}\right\}\ar@{}|{}="Nya2"  \ar@(dr,ur)"Nya2"!<1.7cm,-.2cm>;"Nya2"!<1.7cm,.4cm>_{\dessin{1.2cm}{Diff2}}}
$$ 
\vspace{-1.3cm} 
\begin{flushright}
  10 {\footnotesize (-)}
\end{flushright}
\vspace{.55cm}
$$
\hspace{-1cm}
\xymatrix@C=2cm{
  \left\{\dessin{1.95cm}{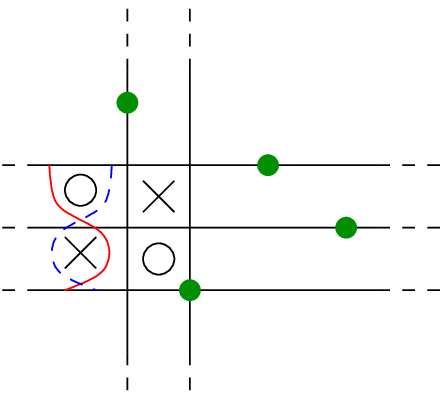}\right\} \ar[r]^(.32){\dessinH{1.2cm}{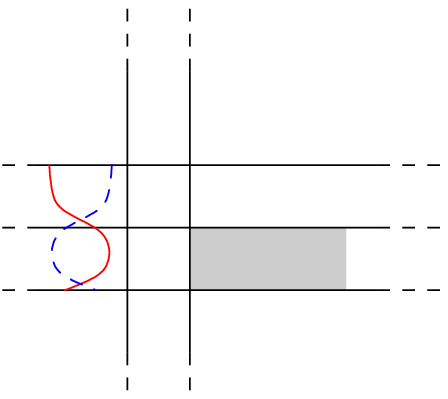}}_(.32)\sim \ar@{}|{}="Nya"  \ar@(dl,ul)"Nya"!<-1.8cm,-.2cm>;"Nya"!<-1.8cm,.4cm>^{\dessin{1.2cm}{Diff2}} & \left\{\dessin{1.95cm}{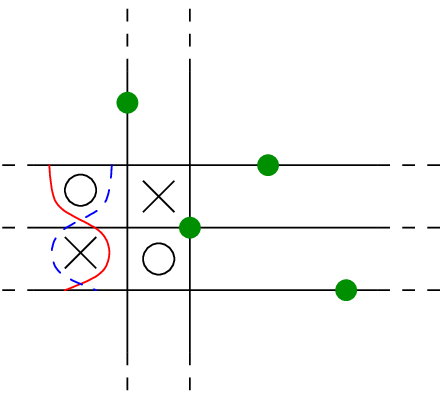}+\dessin{1.95cm}{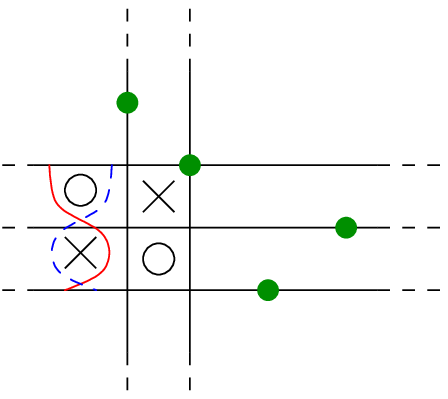}+\dessin{1.95cm}{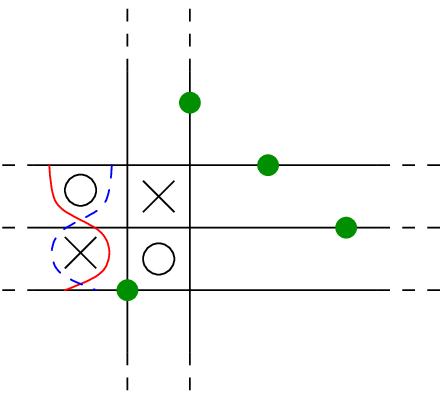}\right\}\ar@{}|{}="Nya2"  \ar@(dr,ur)"Nya2"!<4.5cm,-.2cm>;"Nya2"!<4.5cm,.4cm>_{\dessin{1.2cm}{Diff2}}}
$$ 
\vspace{-.7cm} 
\begin{flushright}
  11 {\footnotesize (9,10)}
\end{flushright}
\vspace{.25cm}
$$
\hspace{-1cm}
\xymatrix@C=2cm{
  \left\{\dessin{1.95cm}{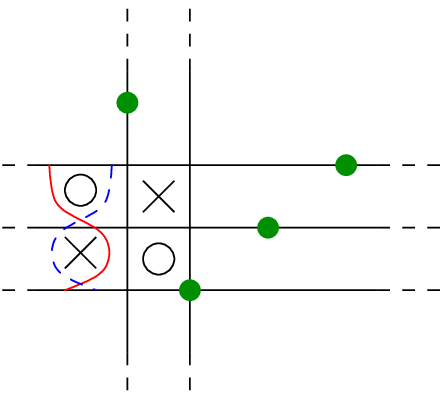}\right\} \ar[r]^(.32){\dessinH{1.2cm}{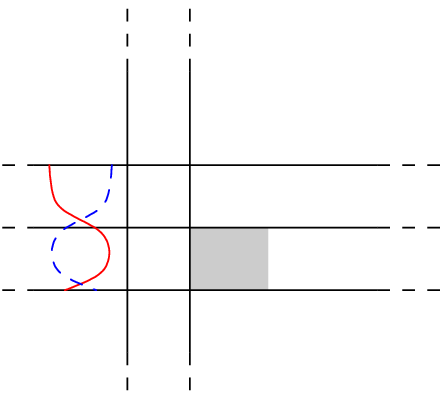}}_(.32)\sim \ar@{}|{}="Nya"  \ar@(dl,ul)"Nya"!<-1.8cm,-.2cm>;"Nya"!<-1.8cm,.4cm>^{\dessin{1.2cm}{Diff2}} & \left\{\dessin{1.95cm}{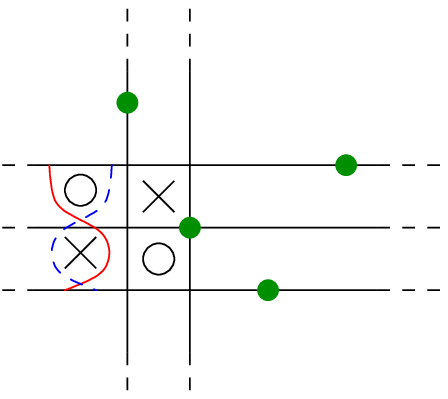}+\dessin{1.95cm}{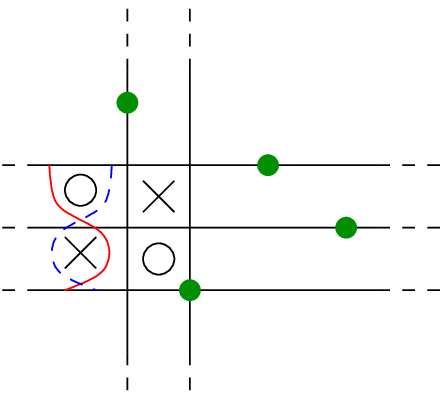}+\dessin{1.95cm}{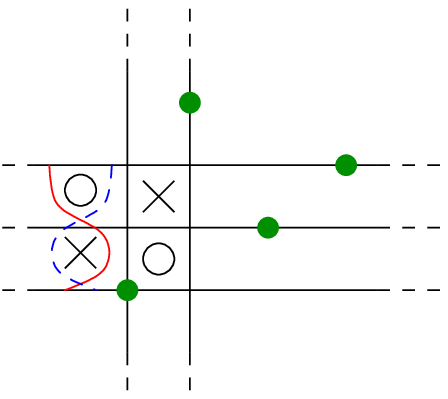}\right\}\ar@{}|{}="Nya2"  \ar@(dr,ur)"Nya2"!<4.5cm,-.2cm>;"Nya2"!<4.5cm,.4cm>_{\dessin{1.2cm}{Diff2}}}
$$ 
\vspace{-.7cm} 
\begin{flushright}
  12 {\footnotesize (9,11)}
\end{flushright}
\vspace{.25cm}
$$
\xymatrix@C=2cm{
  \left\{\dessin{1.95cm}{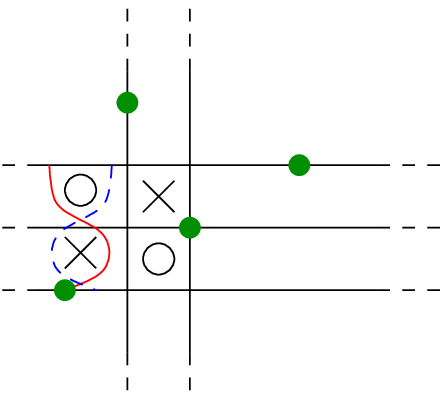}\right\} \ar[r]^{\dessinH{1.2cm}{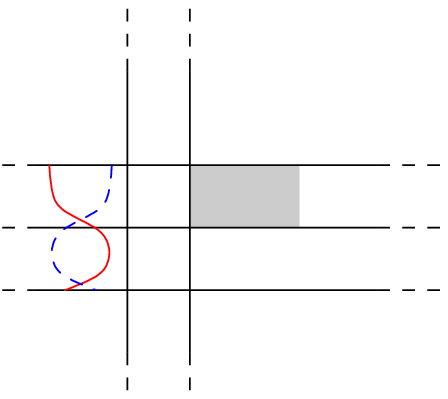}}_(.42)\sim \ar@{}|{}="Nya"  \ar@(dl,ul)"Nya"!<-1.8cm,-.2cm>;"Nya"!<-1.8cm,.4cm>^{0} & \left\{\dessin{1.95cm}{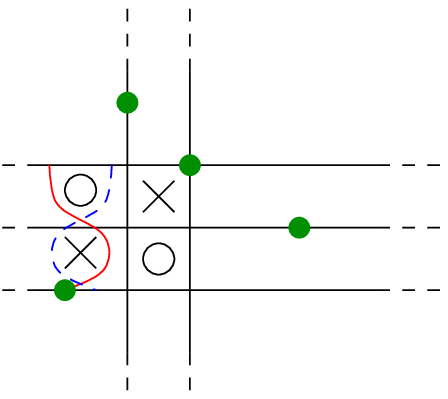}\right\}\ar@{}|{}="Nya2"  \ar@(dr,ur)"Nya2"!<1.8cm,-.2cm>;"Nya2"!<1.8cm,.4cm>_{0}}
$$ 
\vspace{-1.3cm} 
\begin{flushright}
  13 {\footnotesize (-)}
\end{flushright}
\vspace{.55cm}
$$
\xymatrix@C=2cm{
  \left\{\dessin{1.95cm}{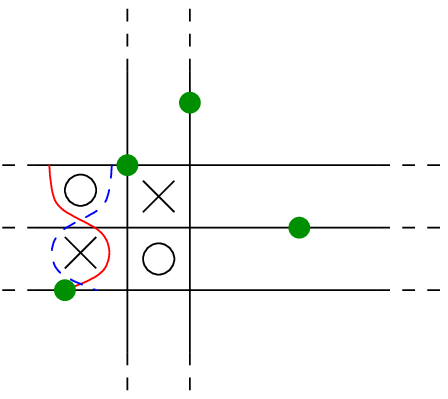}-\dessin{1.95cm}{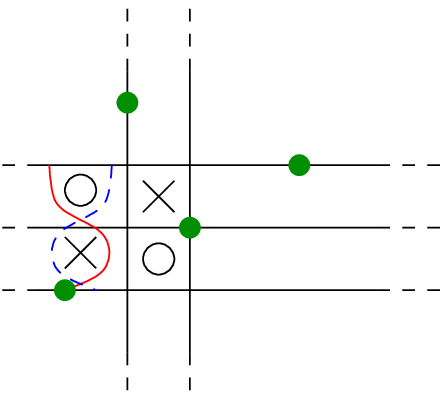}\right\} \ar[r]^(.6){\dessinH{1.2cm}{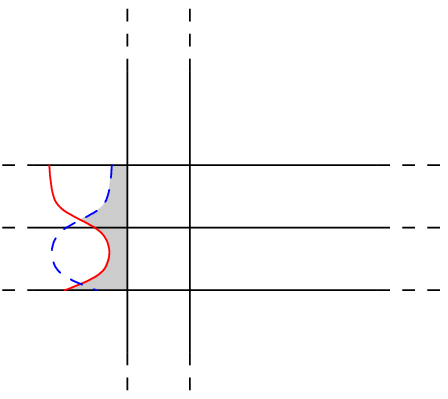}}_(.6)\sim \ar@{}|{}="Nya"  \ar@(dl,ul)"Nya"!<-3.1cm,-.2cm>;"Nya"!<-3.1cm,.4cm>^{0} & \left\{\dessin{1.95cm}{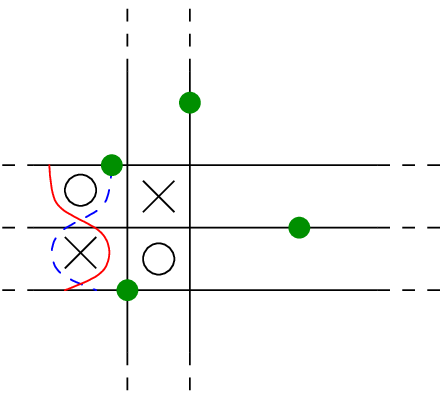}\right\}\ar@{}|{}="Nya2"  \ar@(dr,ur)"Nya2"!<1.7cm,-.2cm>;"Nya2"!<1.7cm,.4cm>_{0}}
$$ 
\vspace{-1.3cm} 
\begin{flushright}
  14 {\footnotesize (13)}
\end{flushright}
\vspace{.55cm}
$$
\xymatrix@C=2cm{
  \left\{\dessin{1.95cm}{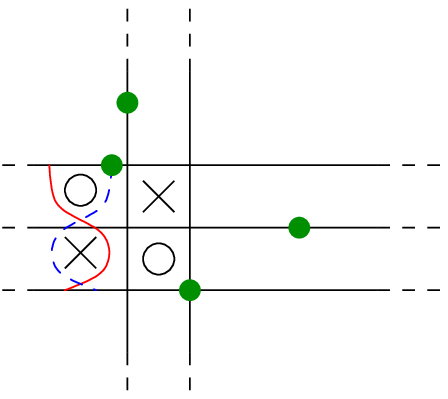}\right\} \ar[r]^(.42){\dessinH{1.2cm}{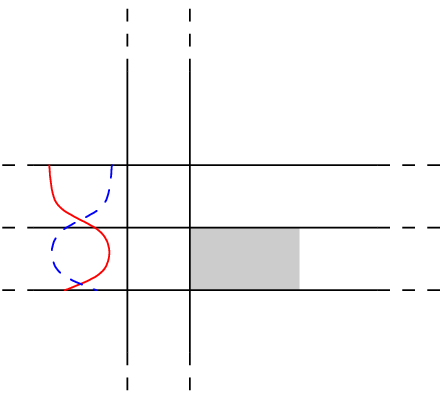}}_(.42)\sim \ar@{}|{}="Nya"  \ar@(dl,ul)"Nya"!<-1.8cm,-.2cm>;"Nya"!<-1.8cm,.4cm>^{0} & \left\{\dessin{1.95cm}{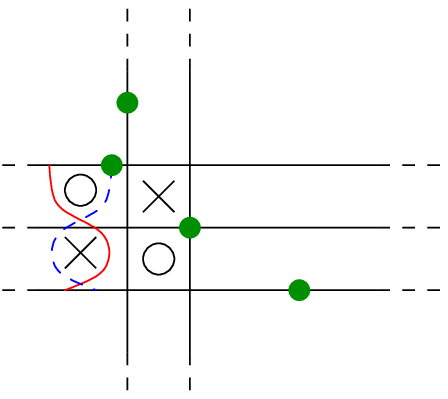}+\dessin{1.95cm}{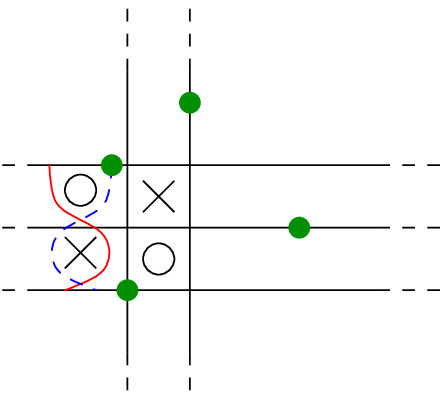}\right\}\ar@{}|{}="Nya2"  \ar@(dr,ur)"Nya2"!<3.1cm,-.2cm>;"Nya2"!<3.1cm,.4cm>_{0}}
$$ 
\vspace{-1.3cm} 
\begin{flushright}
  15 {\footnotesize (14)}
\end{flushright}
\vspace{.55cm}

\newpage

\vspace{-.5cm}

$$
\xymatrix@C=2cm{
  \left\{\dessin{1.95cm}{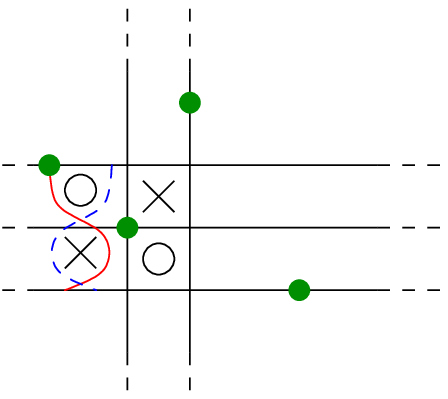}\right\} \ar[r]^{\dessinH{1.2cm}{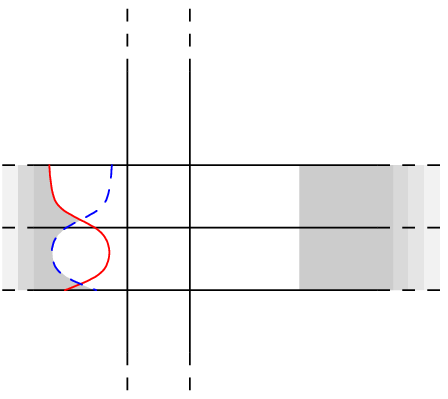}}_\sim \ar@{}|{}="Nya"  \ar@(dl,ul)"Nya"!<-1.8cm,-.2cm>;"Nya"!<-1.8cm,.4cm>^{0} & \left\{\dessin{1.95cm}{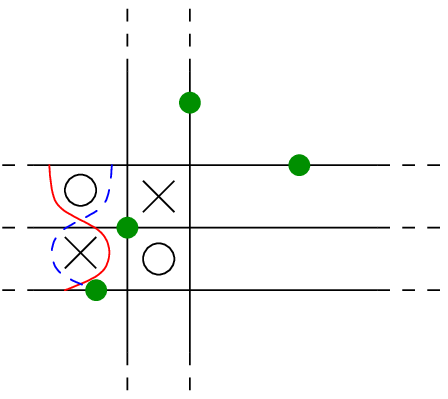}\right\}\ar@{}|{}="Nya2"  \ar@(dr,ur)"Nya2"!<1.7cm,-.2cm>;"Nya2"!<1.7cm,.4cm>_{0}}
$$ 
\vspace{-1.3cm} 
\begin{flushright}
  16 {\footnotesize (-)}
\end{flushright}
\vspace{.55cm}
$$
\xymatrix@C=2cm{
  \left\{\dessin{1.95cm}{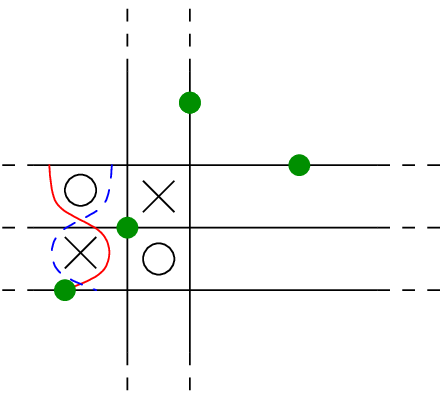}\right\} \ar[r]^{\dessinH{1.2cm}{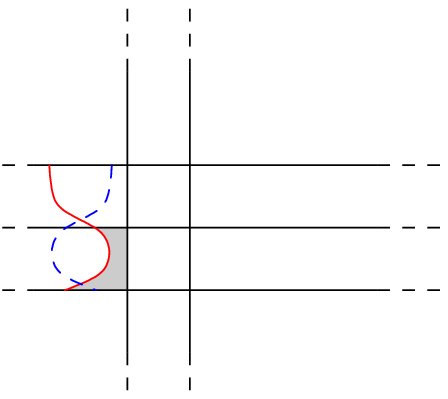}}_\sim \ar@{}|{}="Nya"  \ar@(dl,ul)"Nya"!<-1.8cm,-.2cm>;"Nya"!<-1.8cm,.4cm>^{0} & \left\{\dessin{1.95cm}{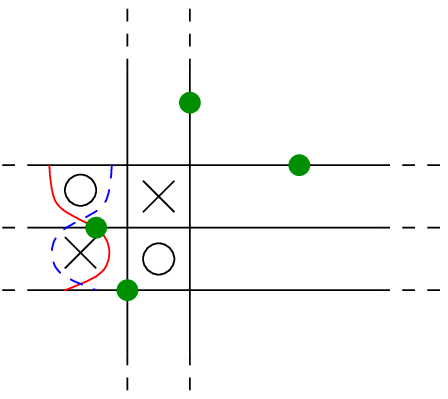}\right\}\ar@{}|{}="Nya2"  \ar@(dr,ur)"Nya2"!<1.7cm,-.2cm>;"Nya2"!<1.7cm,.4cm>_{0}}
$$ 
\vspace{-1.3cm} 
\begin{flushright}
  17 {\footnotesize (-)}
\end{flushright}
\vspace{.55cm}
$$
\xymatrix@C=2cm{
  \left\{\dessin{1.95cm}{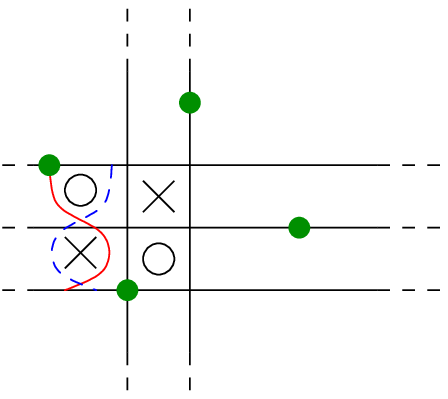}\right\} \ar[r]^(.42){\dessinH{1.2cm}{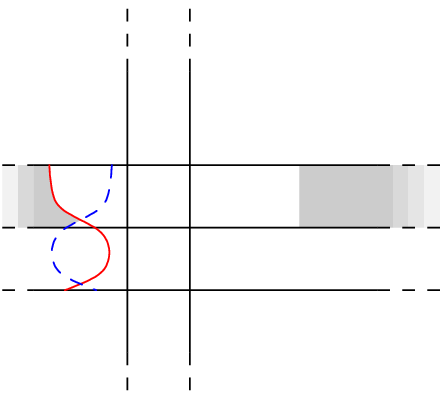}}_(.42)\sim \ar@{}|{}="Nya"  \ar@(dl,ul)"Nya"!<-1.8cm,-.2cm>;"Nya"!<-1.8cm,.4cm>^{0} & \left\{\dessin{1.95cm}{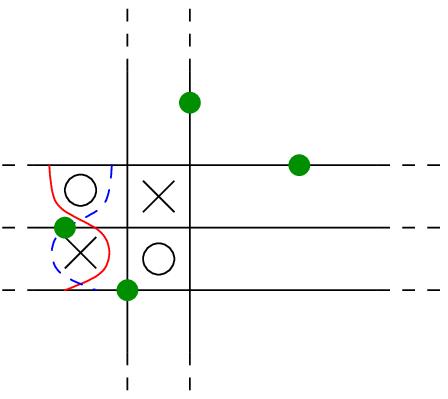}+\dessin{1.95cm}{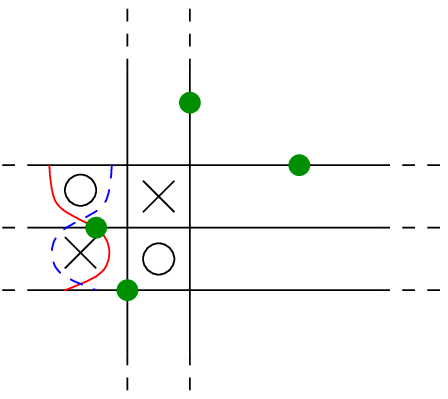}\right\}\ar@{}|{}="Nya2"  \ar@(dr,ur)"Nya2"!<3.1cm,-.2cm>;"Nya2"!<3.1cm,.4cm>_{0}}
$$ 
\vspace{-1.3cm} 
\begin{flushright}
  18 {\footnotesize (17)}
\end{flushright}
\vspace{.55cm}

$$
\xymatrix@C=2cm{
  \left\{\dessin{1.95cm}{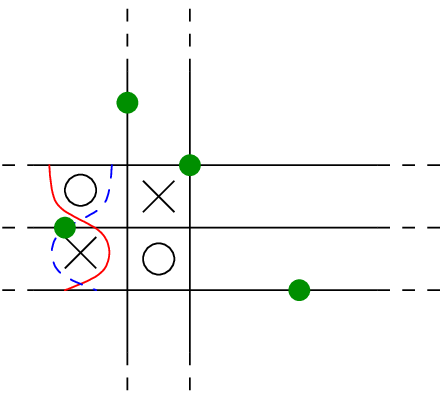}\right\} \ar[r]^{\dessinH{1.2cm}{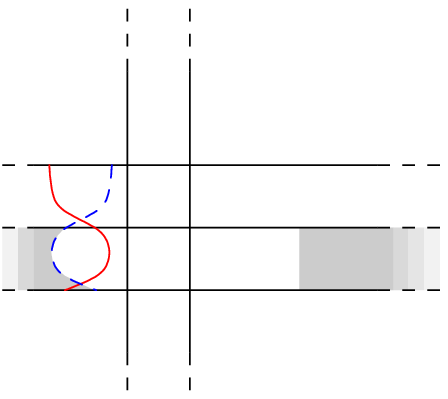}}_\sim \ar@{}|{}="Nya"  \ar@(dl,ul)"Nya"!<-1.8cm,-.2cm>;"Nya"!<-1.8cm,.4cm>^{0} & \left\{\dessin{1.95cm}{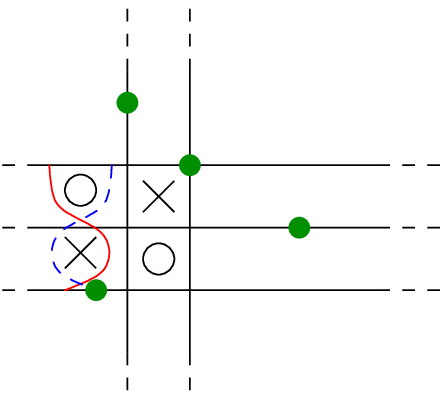}\right\}\ar@{}|{}="Nya2"  \ar@(dr,ur)"Nya2"!<1.7cm,-.2cm>;"Nya2"!<1.7cm,.4cm>_{0}}
$$ 
\vspace{-1.3cm} 
\begin{flushright}
  19 {\footnotesize (-)}
\end{flushright}
\vspace{.55cm}
$$
\xymatrix@C=2cm{
  \left\{\dessin{1.95cm}{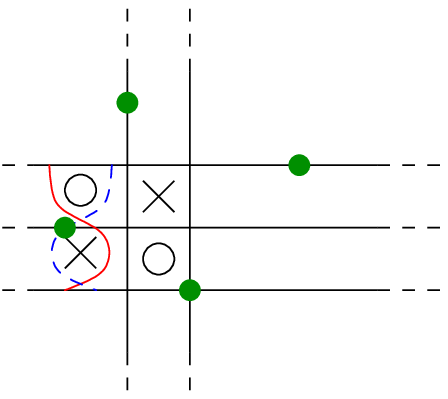}\right\} \ar[r]^(.32){\dessinH{1.2cm}{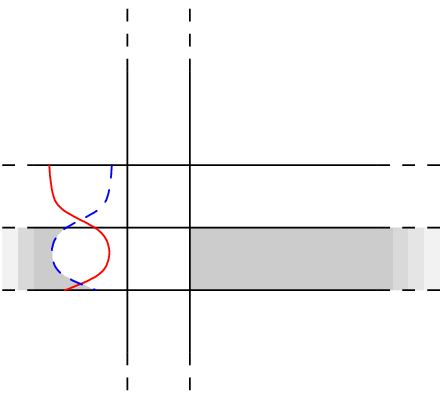}}_(.32)\sim \ar@{}|{}="Nya"  \ar@(dl,ul)"Nya"!<-1.8cm,-.2cm>;"Nya"!<-1.8cm,.4cm>^{0} & \left\{\dessin{1.95cm}{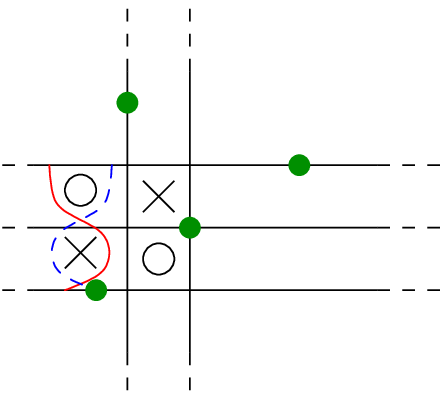}+\dessin{1.95cm}{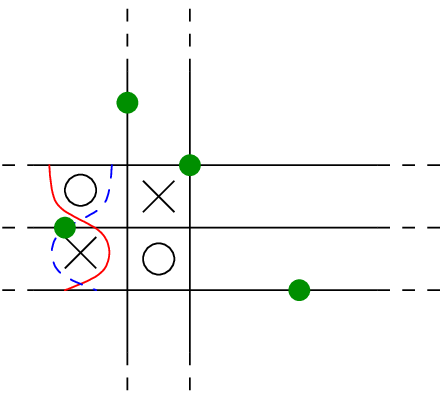}+\dessin{1.95cm}{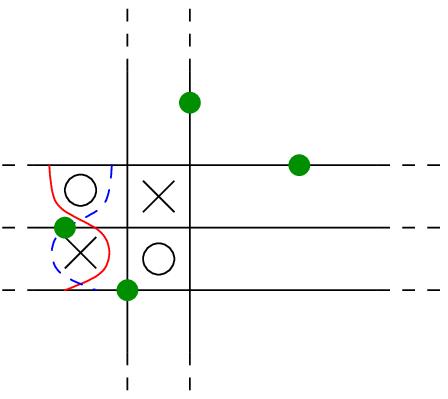}\right\}\ar@{}|{}="Nya2"  \ar@(dr,ur)"Nya2"!<4.5cm,-.2cm>;"Nya2"!<4.5cm,.4cm>_{0}}
$$ 
\vspace{-.7cm} 
\begin{flushright}
  20 {\footnotesize (18,19)}
\end{flushright}
\vspace{-.05cm}
$$
\hspace{-1.3cm}
\xymatrix@C=2cm{
  \left\{\dessin{1.95cm}{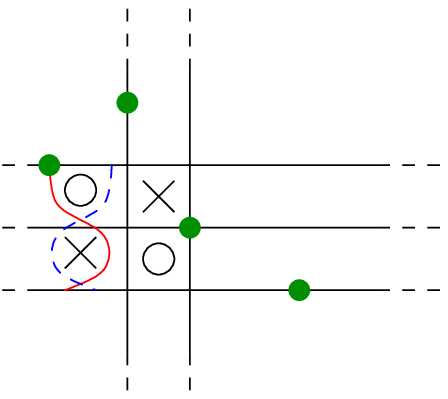}\right\} \ar[r]^(.28){\dessinH{1.2cm}{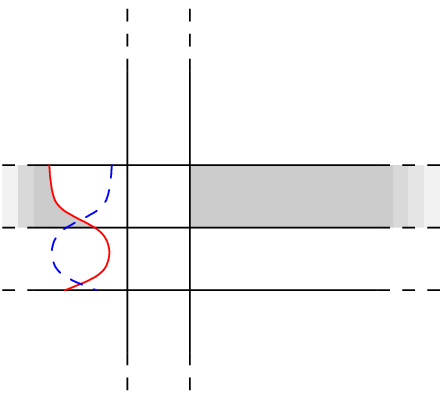}}_(.28)\sim \ar@{}|{}="Nya"  \ar@(dl,ul)"Nya"!<-1.8cm,-.2cm>;"Nya"!<-1.8cm,.4cm>^{0} & \left\{\dessin{1.95cm}{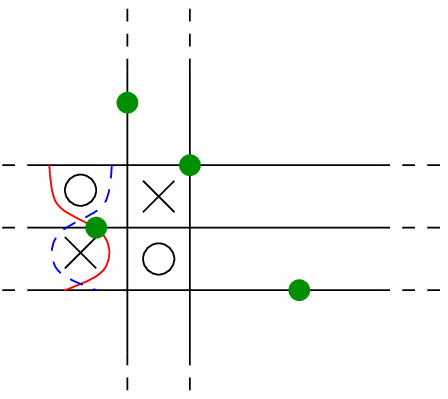}+\dessin{1.95cm}{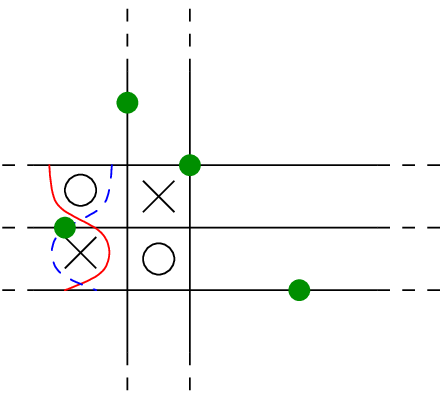}+\dessin{1.95cm}{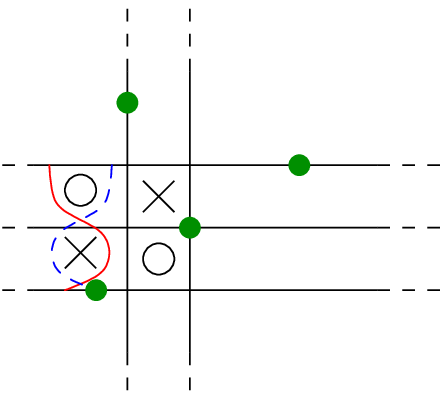}+\dessin{1.95cm}{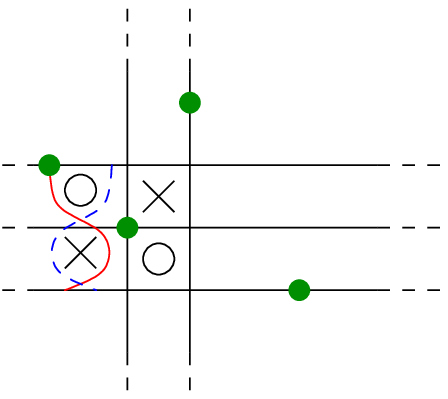}\right\}\ar@{}|{}="Nya2"  \ar@(dr,ur)"Nya2"!<5.8cm,-.2cm>;"Nya2"!<5.8cm,.4cm>_{0}}
$$ 
\vspace{-.7cm} 
\begin{flushright}
  21 {\footnotesize (16,19,20)}
\end{flushright}
\vspace{-.05cm}
$$
\hspace{-1.3cm}
\xymatrix@C=2cm{
  \left\{\dessin{1.95cm}{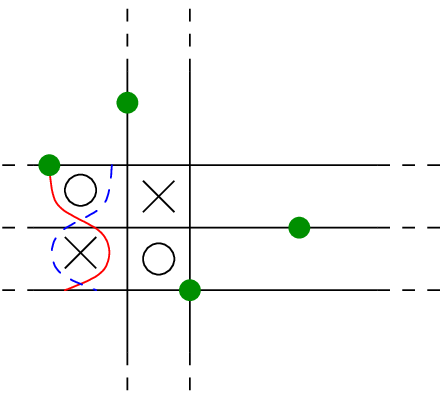}\right\} \ar[r]^(.28){\dessinH{1.2cm}{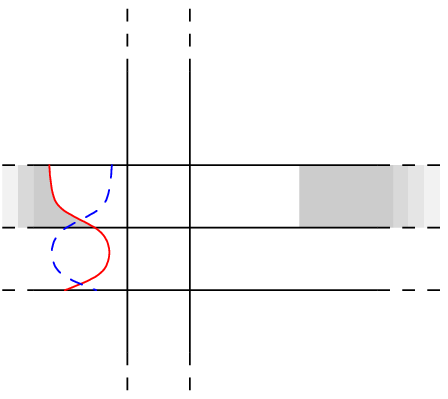}}_(.28)\sim \ar@{}|{}="Nya"  \ar@(dl,ul)"Nya"!<-1.8cm,-.2cm>;"Nya"!<-1.8cm,.4cm>^{0} & \left\{\dessin{1.95cm}{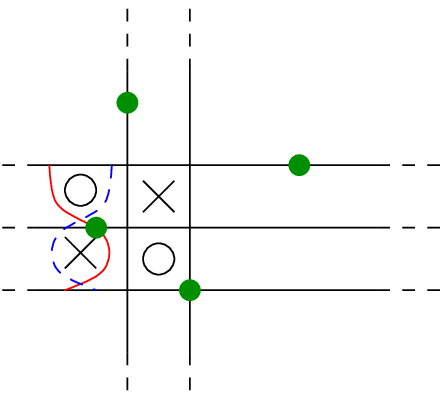}+\dessin{1.95cm}{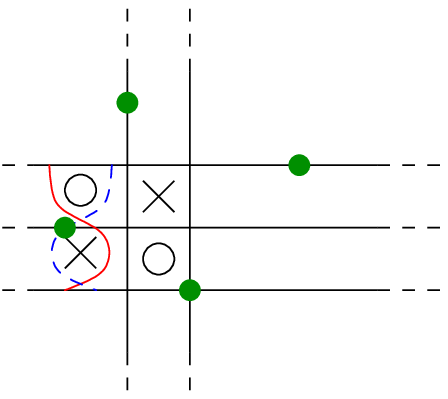}+\dessin{1.95cm}{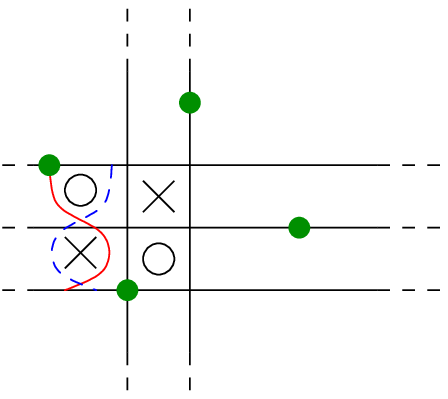}+\dessin{1.95cm}{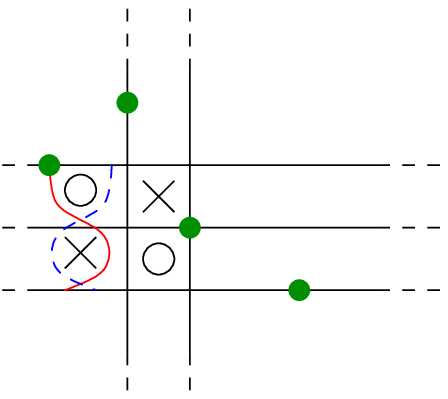}\right\}\ar@{}|{}="Nya2"  \ar@(dr,ur)"Nya2"!<5.8cm,-.2cm>;"Nya2"!<5.8cm,.4cm>_{0}}
$$ 
\vspace{-.7cm} 
\begin{flushright}
  22 {\footnotesize (18,20,21)}
\end{flushright}
\vspace{-.05cm}

$$
\xymatrix@C=2cm{
  \left\{\dessin{1.95cm}{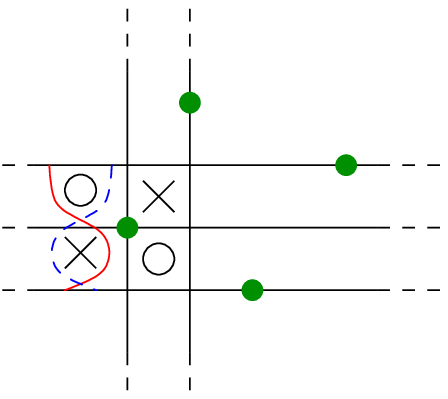}\right\} \ar[r]^{\dessinH{1.2cm}{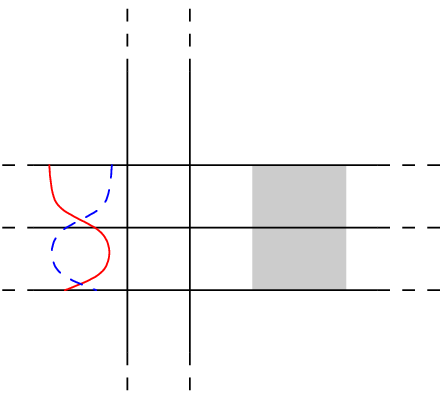}}_\sim \ar@{}|{}="Nya"  \ar@(dl,ul)"Nya"!<-1.8cm,-.2cm>;"Nya"!<-1.8cm,.4cm>^{\dessin{1.2cm}{Diff2}} & \left\{\dessin{1.95cm}{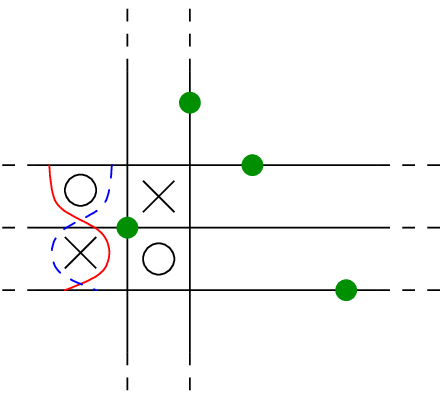}\right\}\ar@{}|{}="Nya2"  \ar@(dr,ur)"Nya2"!<1.7cm,-.2cm>;"Nya2"!<1.7cm,.4cm>_{\dessin{1.2cm}{Diff2}}}
$$ 
\vspace{-1.3cm} 
\begin{flushright}
  23 {\footnotesize (-)}
\end{flushright}
\vspace{.55cm}

\newpage

\vspace{-.5cm}

$$
\hspace{-.75cm}
\xymatrix@C=2cm{
  \left\{\dessin{1.95cm}{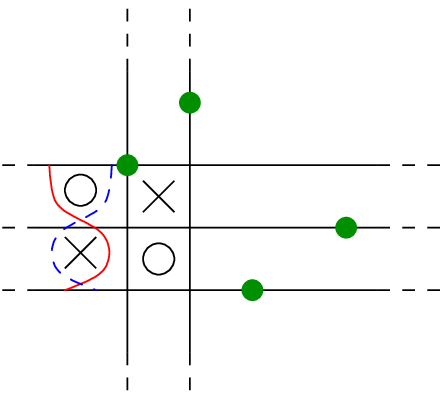}-\dessin{1.95cm}{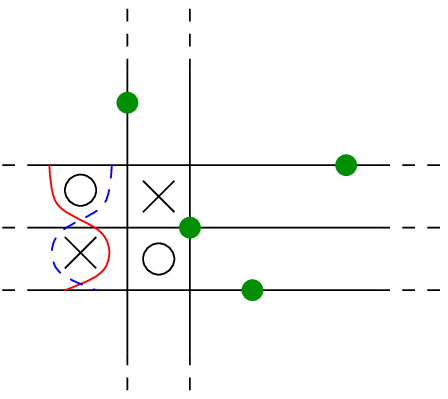}\right\} \ar[r]^{\dessinH{1.2cm}{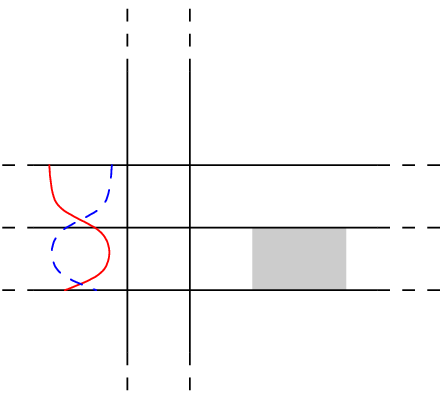}}_\sim \ar@{}|{}="Nya"  \ar@(dl,ul)"Nya"!<-3.1cm,-.2cm>;"Nya"!<-3.1cm,.4cm>^{\dessin{1.2cm}{Diff2}} & \left\{\dessin{1.95cm}{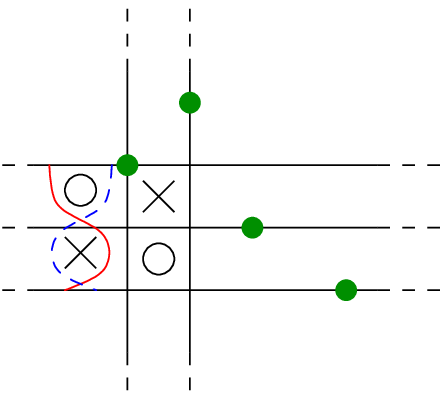}+\dessin{1.95cm}{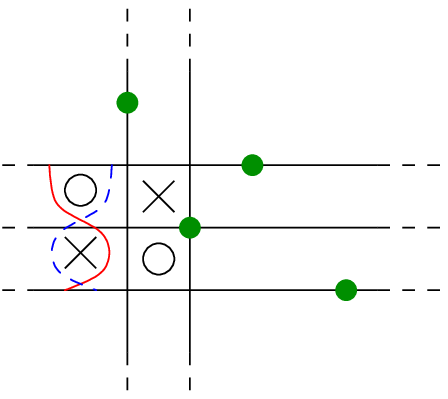}\right\}\ar@{}|{}="Nya2"  \ar@(dr,ur)"Nya2"!<3.1cm,-.2cm>;"Nya2"!<3.1cm,.4cm>_{\dessin{1.2cm}{Diff2}}}
$$ 
\vspace{-.7cm} 
\begin{flushright}
  24 {\footnotesize (11,12)}
\end{flushright}
\vspace{-.05cm}
$$
\xymatrix@C=2cm{
  \left\{\dessin{1.95cm}{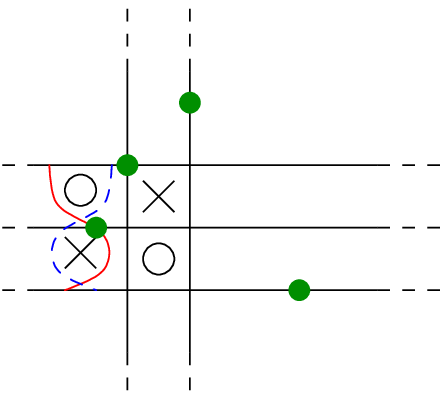}-\dessin{1.95cm}{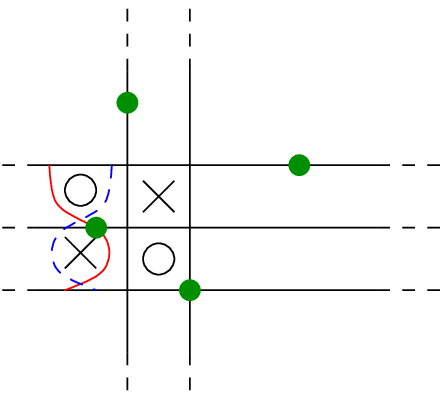}\right\} \ar[r]^(.6){\dessinH{1.2cm}{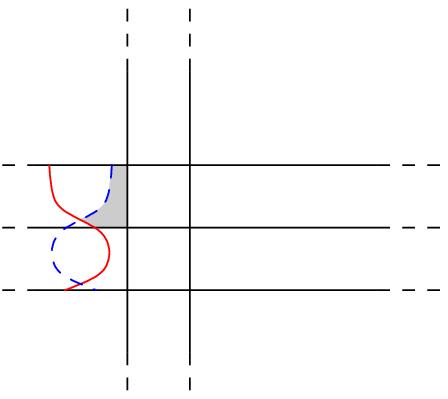}}_(.6)\sim \ar@{}|{}="Nya"  \ar@(dl,ul)"Nya"!<-3.1cm,-.2cm>;"Nya"!<-3.1cm,.4cm>^{0} & \left\{\dessin{1.95cm}{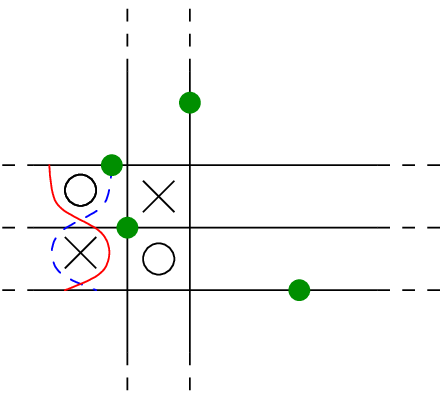}\right\}\ar@{}|{}="Nya2"  \ar@(dr,ur)"Nya2"!<1.7cm,-.2cm>;"Nya2"!<1.7cm,.4cm>_{0}}
$$ 
\vspace{-1.3cm} 
\begin{flushright}
  25 {\footnotesize (22)}
\end{flushright}
\vspace{.55cm}
$$
\hspace{-1.2cm}
\xymatrix@C=2cm{
  \left\{\dessin{1.95cm}{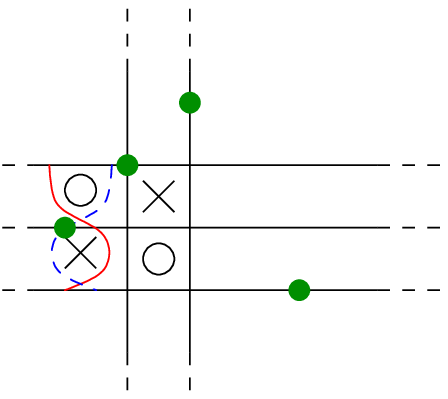}-\dessin{1.95cm}{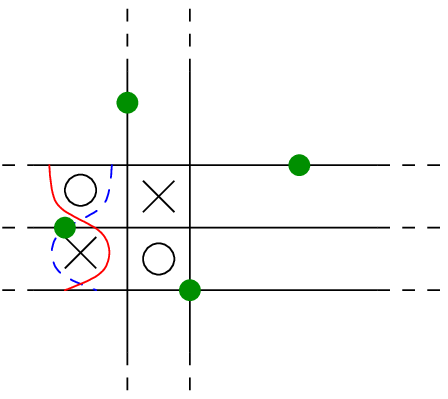}\right\} \ar[r]^(.42){\dessinH{1.2cm}{Isom7}}_(.42)\sim \ar@{}|{}="Nya"  \ar@(dl,ul)"Nya"!<-3.1cm,-.2cm>;"Nya"!<-3.1cm,.4cm>^{0} & \left\{\dessin{1.95cm}{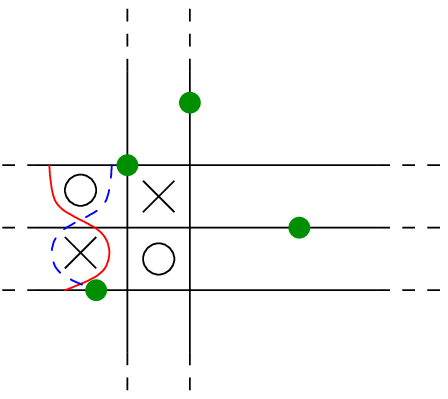}+\dessin{1.95cm}{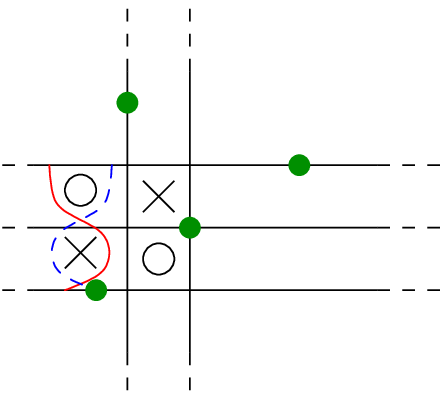}+\dessin{1.95cm}{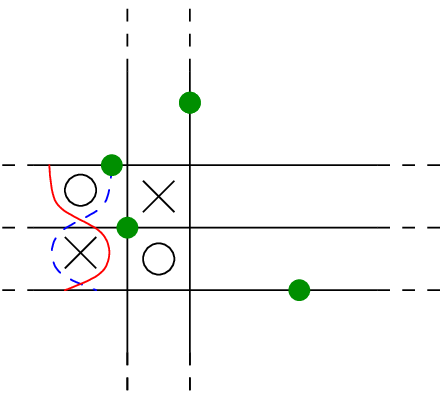}\right\}\ar@{}|{}="Nya2"  \ar@(dr,ur)"Nya2"!<4.5cm,-.2cm>;"Nya2"!<4.5cm,.4cm>_{0}}
$$ 
\vspace{-.6cm} 
\begin{flushright}
  26 {\footnotesize (20,25)}
\end{flushright}
\vspace{-.05cm}

$$
\xymatrix@C=2cm{
  \left\{\dessin{1.95cm}{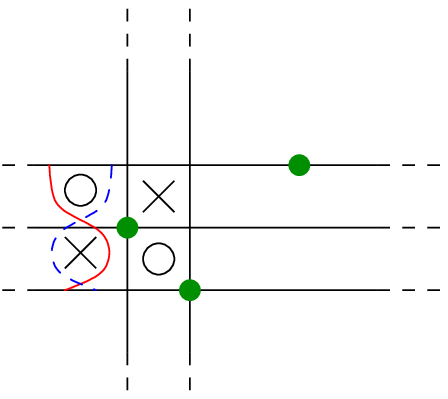}\right\} \ar[r]^{\dessinH{1.2cm}{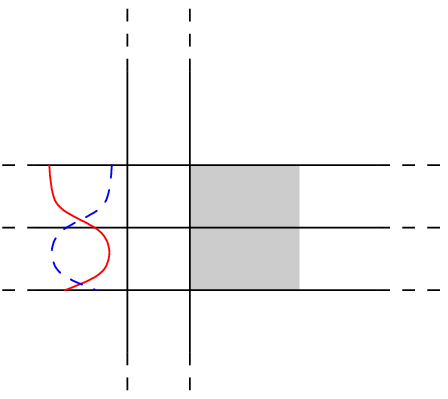}}_\sim \ar@{}|{}="Nya"  \ar@(dl,ul)"Nya"!<-1.8cm,-.2cm>;"Nya"!<-1.8cm,.4cm>^{\dessin{1.2cm}{Diff2}} & \left\{\dessin{1.95cm}{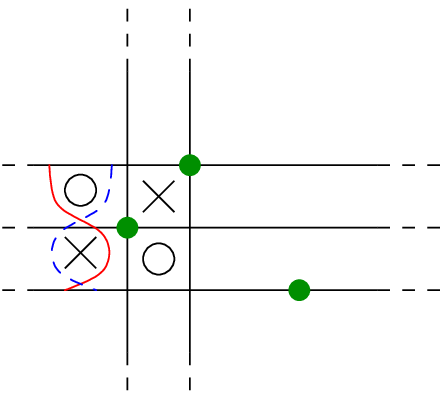}\right\}\ar@{}|{}="Nya2"  \ar@(dr,ur)"Nya2"!<1.7cm,-.2cm>;"Nya2"!<1.7cm,.4cm>_{\dessin{1.2cm}{Diff2}}}
$$ 
\vspace{-1.3cm} 
\begin{flushright}
  27 {\footnotesize (-)}
\end{flushright}
\vspace{.55cm}
$$
\xymatrix@C=2cm{
  \left\{\dessin{1.95cm}{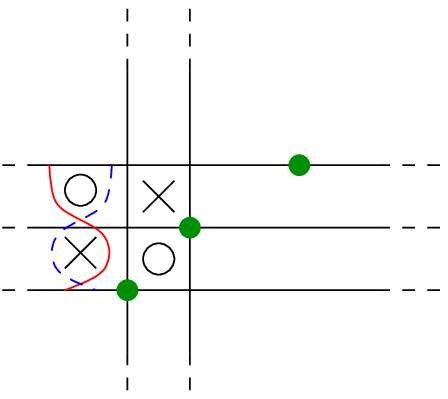}\right\} \ar[r]^{\dessinH{1.2cm}{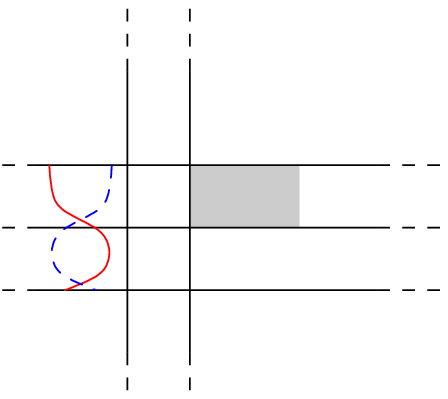}}_\sim \ar@{}|{}="Nya"  \ar@(dl,ul)"Nya"!<-1.8cm,-.2cm>;"Nya"!<-1.8cm,.4cm>^{\dessin{1.2cm}{Diff2}} & \left\{\dessin{1.95cm}{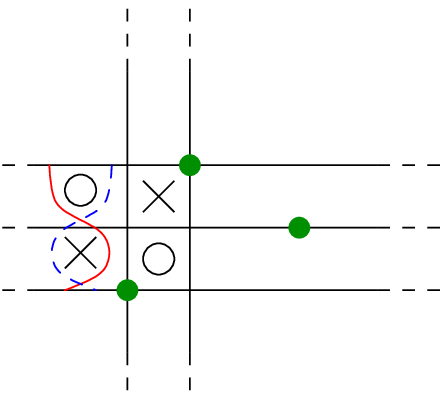}\right\}\ar@{}|{}="Nya2"  \ar@(dr,ur)"Nya2"!<1.7cm,-.2cm>;"Nya2"!<1.7cm,.4cm>_{\dessin{1.2cm}{Diff2}}}
$$ 
\vspace{-1.3cm} 
\begin{flushright}
  28 {\footnotesize (-)}
\end{flushright}
\vspace{.55cm}
$$
\xymatrix@C=2cm{
  \left\{\dessin{1.95cm}{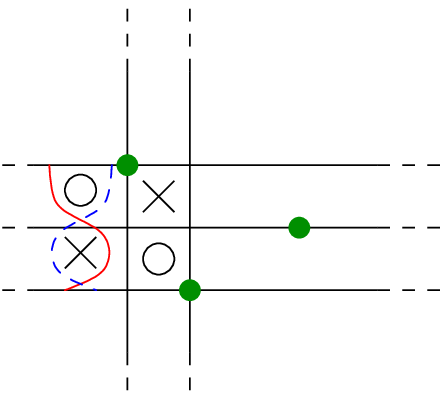}\right\} \ar[r]^(.42){\dessinH{1.2cm}{Isom20}}_(.42)\sim \ar@{}|{}="Nya"  \ar@(dl,ul)"Nya"!<-1.8cm,-.2cm>;"Nya"!<-1.8cm,.4cm>^{\dessin{1.2cm}{Diff2}} & \left\{\dessin{1.95cm}{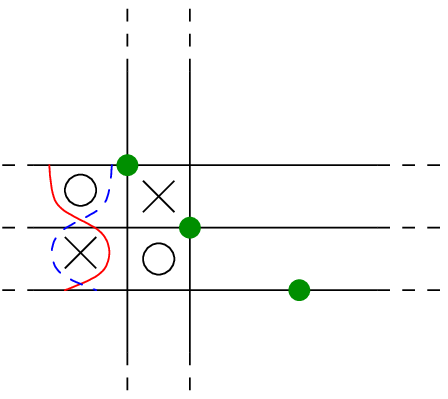}+\dessin{1.95cm}{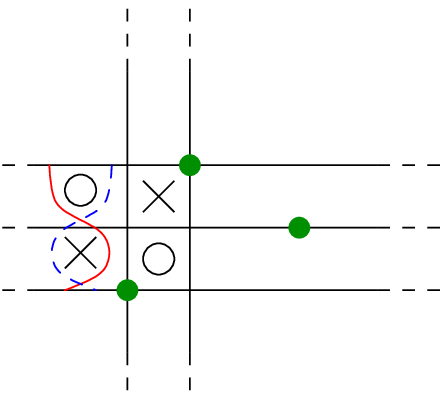}\right\}\ar@{}|{}="Nya2"  \ar@(dr,ur)"Nya2"!<3.1cm,-.2cm>;"Nya2"!<3.1cm,.4cm>_{\dessin{1.2cm}{Diff2}}}
$$ 
\vspace{-1cm} 
\begin{flushright}
  29 {\footnotesize (28)}
\end{flushright}
\vspace{.25cm}

$$
\xymatrix@C=2cm{
  \left\{\dessin{1.95cm}{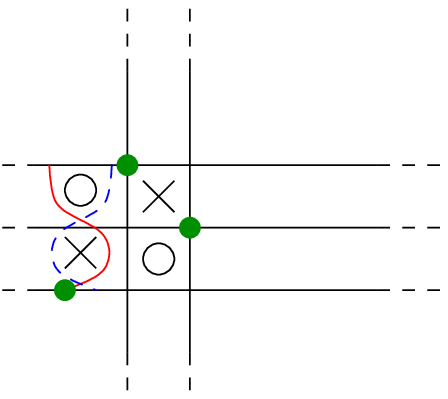}\right\} \ar[r]^{\dessinH{1.2cm}{Isom12}}_\sim \ar@{}|{}="Nya"  \ar@(dl,ul)"Nya"!<-1.8cm,-.2cm>;"Nya"!<-1.8cm,.4cm>^{0} & \left\{\dessin{1.95cm}{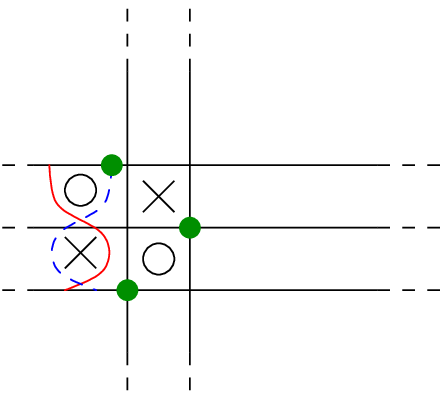}\right\}\ar@{}|{}="Nya2"  \ar@(dr,ur)"Nya2"!<1.7cm,-.2cm>;"Nya2"!<1.7cm,.4cm>_{0}}
$$ 
\vspace{-1.3cm} 
\begin{flushright}
  30 {\footnotesize (-)}
\end{flushright}
\vspace{.55cm}
$$
\xymatrix@C=2cm{
  \left\{\dessin{1.95cm}{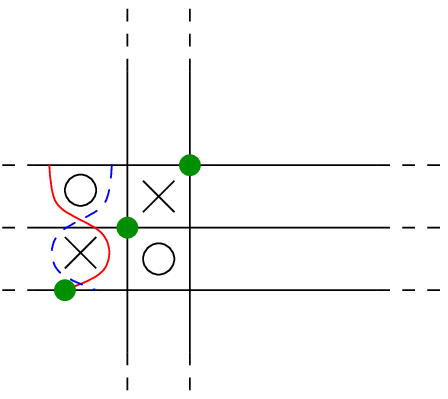}\right\} \ar[r]^(.5){\dessinH{1.2cm}{Isom15}}_(.5)\sim \ar@{}|{}="Nya"  \ar@(dl,ul)"Nya"!<-1.8cm,-.2cm>;"Nya"!<-1.8cm,.4cm>^{0} & \left\{\dessin{1.95cm}{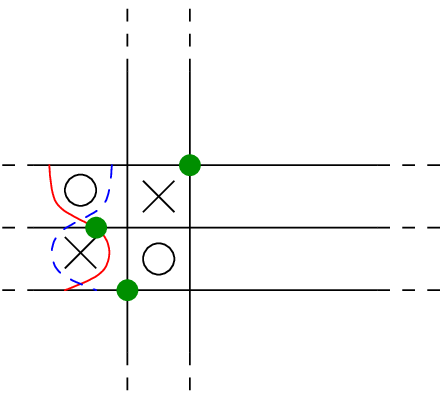}\right\}\ar@{}|{}="Nya2"  \ar@(dr,ur)"Nya2"!<1.8cm,-.2cm>;"Nya2"!<1.8cm,.4cm>_{0}}
$$ 
\vspace{-1.3cm} 
\begin{flushright}
  31 {\footnotesize (-)}
\end{flushright}
\vspace{.55cm}

\newpage

\vspace{-.5cm}

$$
\xymatrix@C=2cm{
  \left\{\dessin{1.95cm}{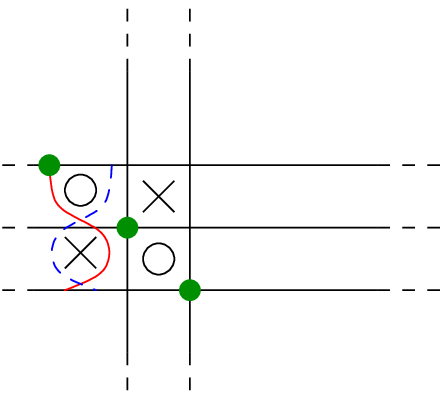}\right\} \ar[r]^(.5){\dessinH{1.2cm}{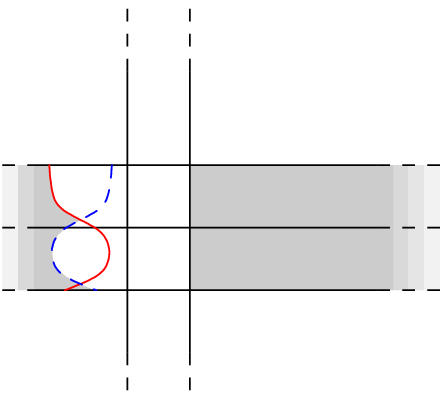}}_(.5)\sim \ar@{}|{}="Nya"  \ar@(dl,ul)"Nya"!<-1.8cm,-.2cm>;"Nya"!<-1.8cm,.4cm>^{0} & \left\{\dessin{1.95cm}{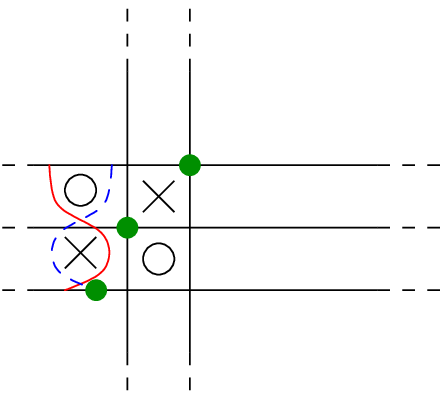}\right\}\ar@{}|{}="Nya2"  \ar@(dr,ur)"Nya2"!<1.8cm,-.2cm>;"Nya2"!<1.8cm,.4cm>_{0}}
$$ 
\vspace{-1.3cm} 
\begin{flushright}
  32 {\footnotesize (-)}
\end{flushright}
\vspace{.55cm}
$$
\xymatrix@C=2cm{
  \left\{\dessin{1.95cm}{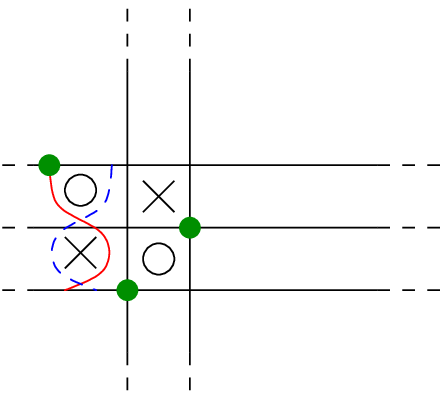}\right\} \ar[r]^(.5){\dessinH{1.2cm}{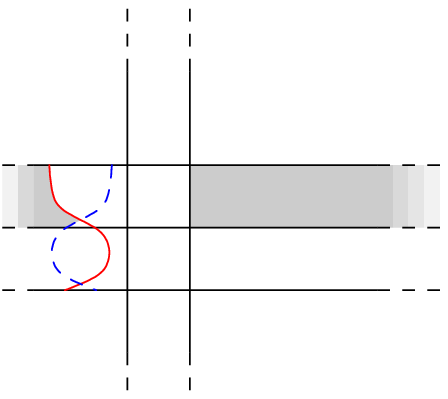}}_(.5)\sim \ar@{}|{}="Nya"  \ar@(dl,ul)"Nya"!<-1.8cm,-.2cm>;"Nya"!<-1.8cm,.4cm>^{0} & \left\{\dessin{1.95cm}{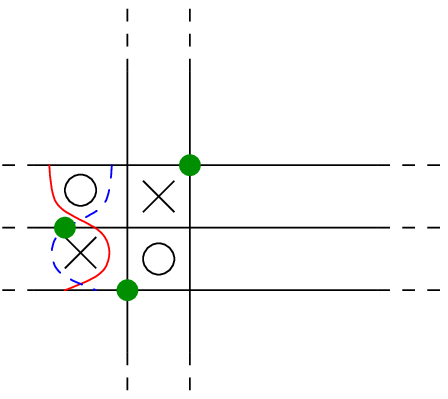}\right\}\ar@{}|{}="Nya2"  \ar@(dr,ur)"Nya2"!<1.8cm,-.2cm>;"Nya2"!<1.8cm,.4cm>_{0}}
$$ 
\vspace{-1.3cm} 
\begin{flushright}
  33 {\footnotesize (-)}
\end{flushright}
\vspace{.55cm}
$$
\xymatrix@C=2cm{
  \left\{\dessin{1.95cm}{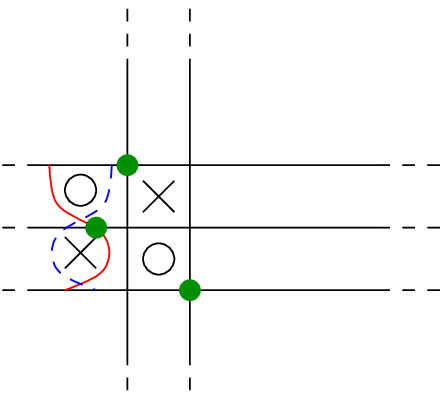}\right\} \ar[r]^(.42){\dessinH{1.2cm}{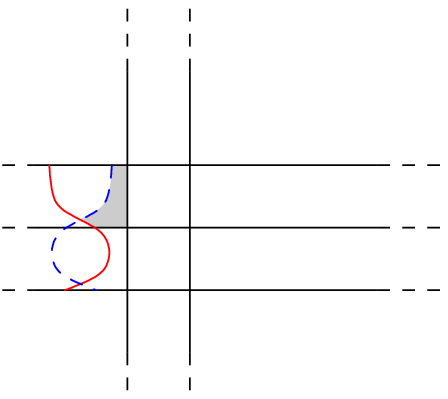}}_(.42)\sim \ar@{}|{}="Nya"  \ar@(dl,ul)"Nya"!<-1.8cm,-.2cm>;"Nya"!<-1.8cm,.4cm>^{0} & \left\{\dessin{1.95cm}{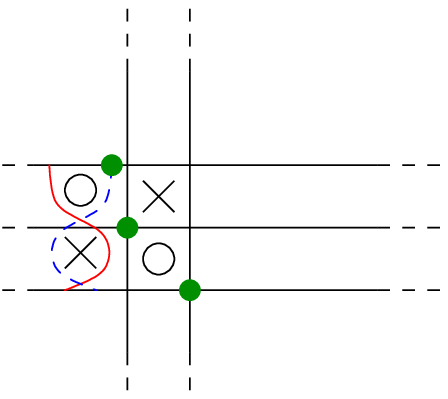}+\dessin{1.95cm}{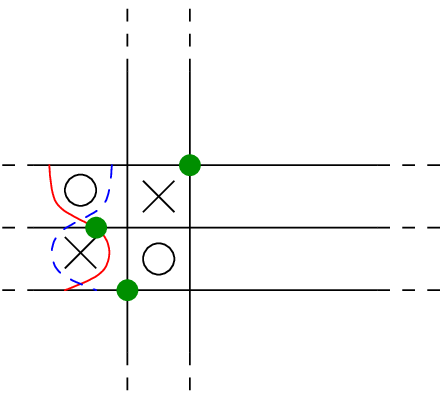}\right\}\ar@{}|{}="Nya2"  \ar@(dr,ur)"Nya2"!<3.1cm,-.2cm>;"Nya2"!<3.1cm,.4cm>_{0}}
$$ 
\vspace{-1.3cm} 
\begin{flushright}
  34 {\footnotesize (31)}
\end{flushright}
\vspace{.55cm}
$$
\xymatrix@C=2cm{
  \left\{\dessin{1.95cm}{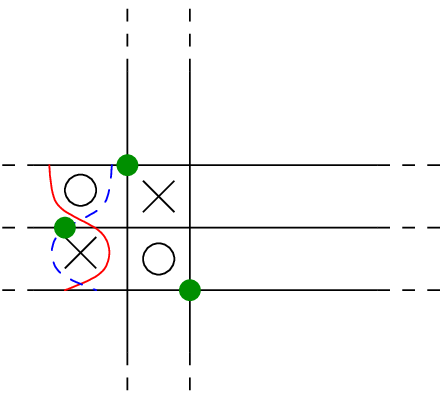}\right\} \ar[r]^(.32){\dessinH{1.2cm}{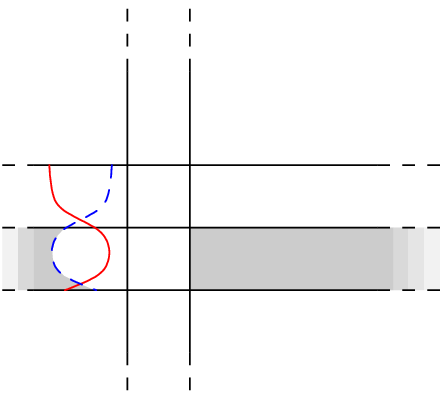}}_(.32)\sim \ar@{}|{}="Nya"  \ar@(dl,ul)"Nya"!<-1.8cm,-.2cm>;"Nya"!<-1.8cm,.4cm>^{0} & \left\{\dessin{1.95cm}{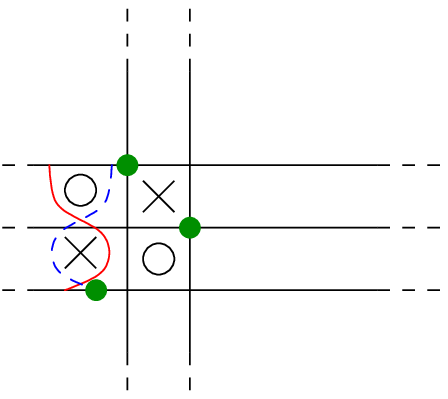}+\dessin{1.95cm}{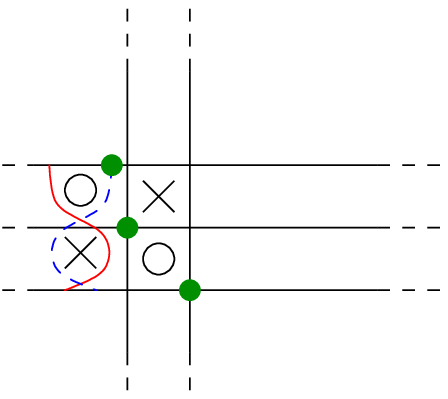}+\dessin{1.95cm}{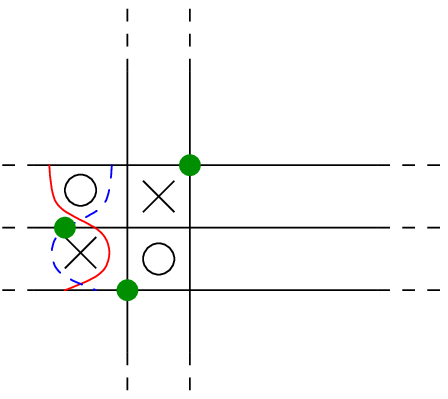}\right\}\ar@{}|{}="Nya2"  \ar@(dr,ur)"Nya2"!<4.5cm,-.2cm>;"Nya2"!<4.5cm,.4cm>_{0}}
$$
\vspace{-.7cm} 
\begin{flushright}
  35 {\footnotesize (33,34)}
\end{flushright}

%%% Local Variables: 
%%% mode: latex
%%% TeX-master: "These"

%%% End: 

% Simplifications par grilles mixtes
\chapter{Construction for mixed grids}
\label{appendix:MixedSimplification}

This appendix completes and achieves the construction given in this thesis.
As a matter of fact, if authorizing singular columns and rows at the same time make the set of grid diagrams larger, it simplifies the elementary moves.
The proofs of invariance are then more intelligible.\\
Moreover, it enables more operations on grids which clarify some symetric properties of the associated link Floer homologies.

\section{Construction}
\label{sec:MixedConstruction}

\subsubsection{Row desingularization}
\label{par:RowDesing}

All the machinery developped for the desingularization of a singular column can be adapted to singular rows by rotating all the pictures $90°$.
But then, vertical and horizontal strands are swapped and so are the nature of crossings.
Thus, for a given singular row in a standard configuration, \ie with its two $X$'s on the left of its $O$'s, we define its $0$ and $1$-resolution as shown in Figure \ref{fig:DesingRow}\index{resolution!Aresolution@$0$--resolution}\index{resolution!Bresolution@$1$--resolution}.

\begin{figure}[h]
  $$
\begin{array}{ccc}
  \dessin{2.5cm}{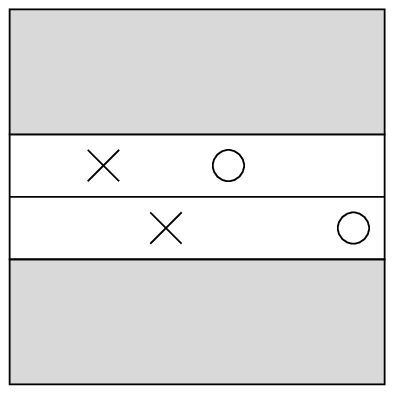} & \hspace{1cm} & \dessin{2.5cm}{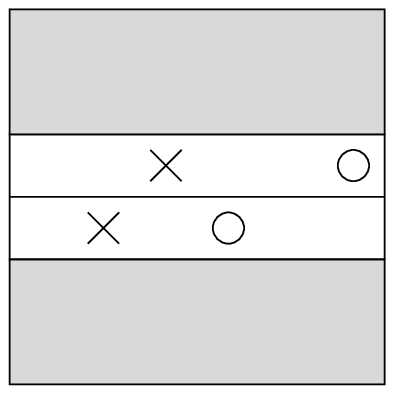}\\[1cm]
0\textrm{-resolution} && 1\textrm{-resolution}
\end{array}
$$  
  \caption{Resolution for a singular row in a standard configuration}
  \label{fig:DesingRow}
\end{figure}

They correspond, respectively, to a positive and a negative resolution of the associated double point.
We extend this definition to any singular row, using the fact that, up to cyclic permutations of the columns, any singular row can be set in a standard configuration.\\

Figure \ref{fig:rotCombinedGrid} shows how these two desingularizations can be depicted on the singular grid with some winding arcs, denoted $\alpha$ and $\beta$.

\begin{figure}[h]
$$
  \hspace{.5cm} \dessin{3.8cm}{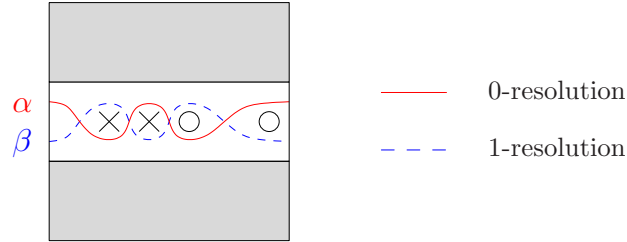} \hspace{.5cm}
\begin{array}{cl}
  \xymatrix{\ar@*{[red]}@{-}[r]&} & 0\textrm{-resolution}\\[3mm]
  \xymatrix{\ar@*{[blue]}@{--}[r]&} & 1\textrm{-resolution}
\end{array}
$$
\caption{Combined grid for a singular row}
\label{fig:rotCombinedGrid}
\end{figure}

The arcs $\alpha$ and $\beta$ meet in four points which can be identified according to their position with regard to the decorations of the singular row.
By analogy with the case of singular columns, we denote such an intersection by the picture of the two decorations surrounding it.
We are essentially interested in the intersections of type $\dessinH{.4cm}{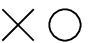}$\index{zzzzaaa@$\dessinH{.4cm}{rotCstyle}$} and \emph{$\dessinH{.4cm}{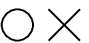}$}\index{zzzzaab@$\dessinH{.4cm}{rotCpstyle}$}.

\subsubsection{Grid polygons}
\label{par:GridPolygons}

Now that the singular grids are, in a sense, more sophisticated, the definition of the differential requires more polygons.
Here, we give a general definition for \emph{grid polygons}.\\
Let $G$ be a singular grid with $k\in\N$ singular rows and columns.
For convenience, we label these singular elements with integer from $1$ to $k$.
For $I=(i_1,\cdots,i_k)\in\{0,1\}^k$, we denote by $G_I$ the regular grid obtained by performing a $i_j$-resolution to the $j^\textrm{th}$ singular element for all $j\in \llbracket 1,k\rrbracket$.\\
We set
$$
C^-(G)= \bigoplus_{I\in \{0,1\}^k}C^-(G_I)[\# 0(I)].
$$
Now, let $P$ be a set of peaks \ie a choice, for each singular row or column, of an arcs intersection of type $\dessin{.4cm}{Cstyle}$, $\dessin{.4cm}{Cpstyle}$, $\dessinH{.4cm}{rotCstyle}$ or $\dessinH{.4cm}{rotCpstyle}$.
Let $\T$ the torus obtained by gluing the borders of $G$.\\

  Let $x$ and $y$ be generators of  $C^-(G)$. A \emph{grid polygon connecting $x$ to $y$}\index{polygon!grid polygon} is an embedded polygon $\pi$ in $\T$, with set of corners $C(\pi)$, which satisfies:
  \begin{itemize}
  \item[-] $\p \pi$ is embedded in the grid lines (including the winding arcs);
  \item[-] $C(\pi)$ is a subset of $x\cup y \cup P$, with $C(\pi) \cap x \neq \emptyset$;
  \item[-] starting at a element of $C(\pi)\cap x$ and running positively along $\p \pi$, according to the orientation of $\pi$ inherited from the one of $\T$, we first meet a horizontal arc and, then, the corners of $\pi$ are successively and alternatively points of $x$ and $y$ with, possibly, an element of $P$ inserted between them ;
  \item[-] except on $\p \pi$, the sets $x$ and $y$ coincide;
  \item[-] for all element $c\in C(\pi)\cap P$, the interior of $\pi$ does not intersect the winding arcs in a neighborhood of $c$.
  \end{itemize}
The set $C(\pi)\cap P$ is called the \emph{set of peaks of $\pi$}\index{polygon!peak}\index{peak|see{polygon}}.\\
It is said to be \emph{empty}\index{polygon!empty} if $\Int (\pi)\cap x=\emptyset$.\\

\begin{figure}[h]
 $$
  \begin{array}{ccccc}
    \dessin{3.7cm}{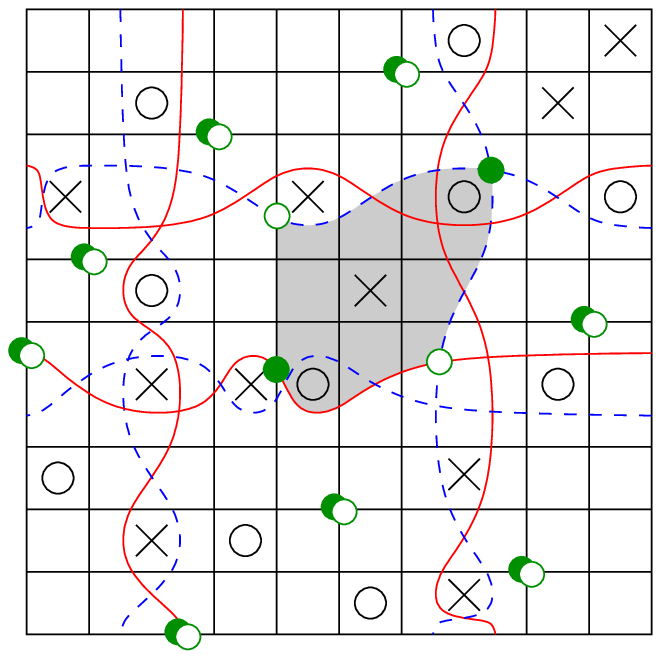} & \hspace{.8cm} & \dessin{3.7cm}{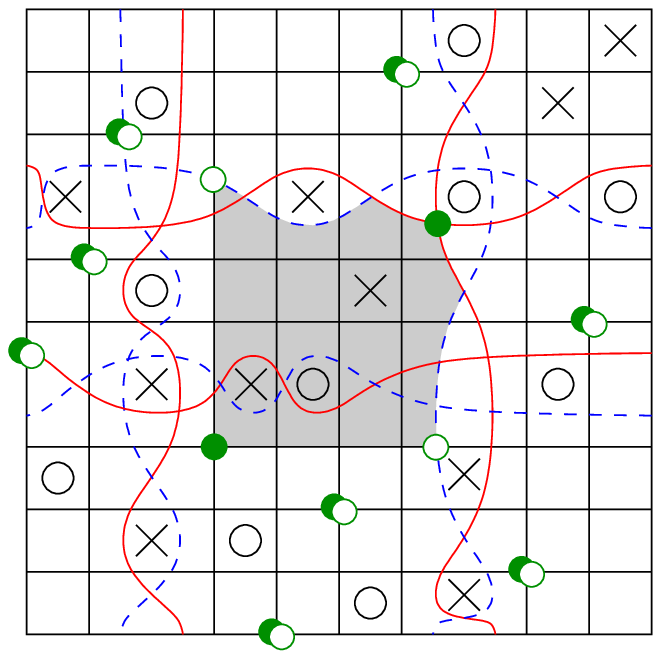} & \hspace{.8cm} & \dessin{3.7cm}{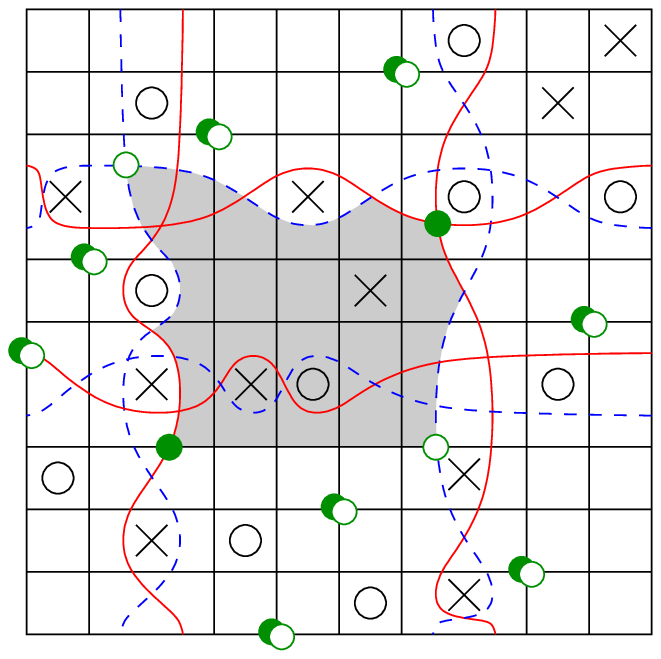}\\[1cm]
    \textrm{an empty rectangle} && \textrm{an empty hexagon} && \textrm{an empty heptagon}\\[1cm]
    \dessin{3.7cm}{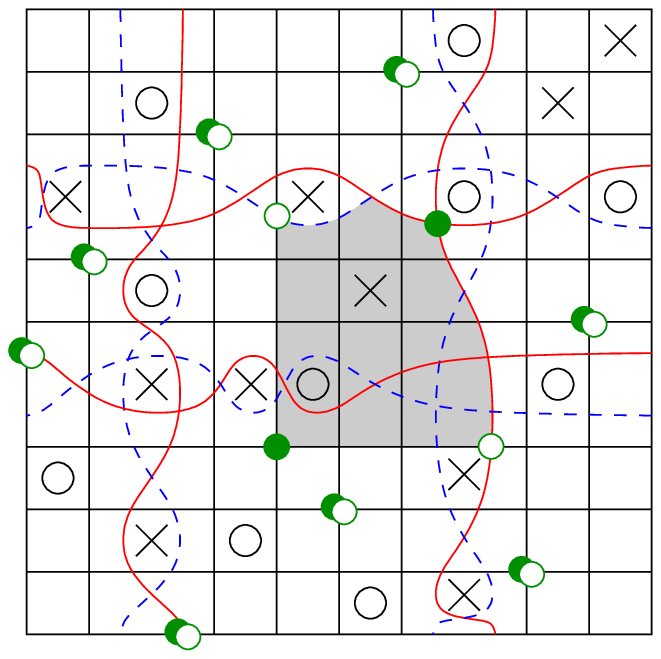} & \hspace{.8cm} & \dessin{3.7cm}{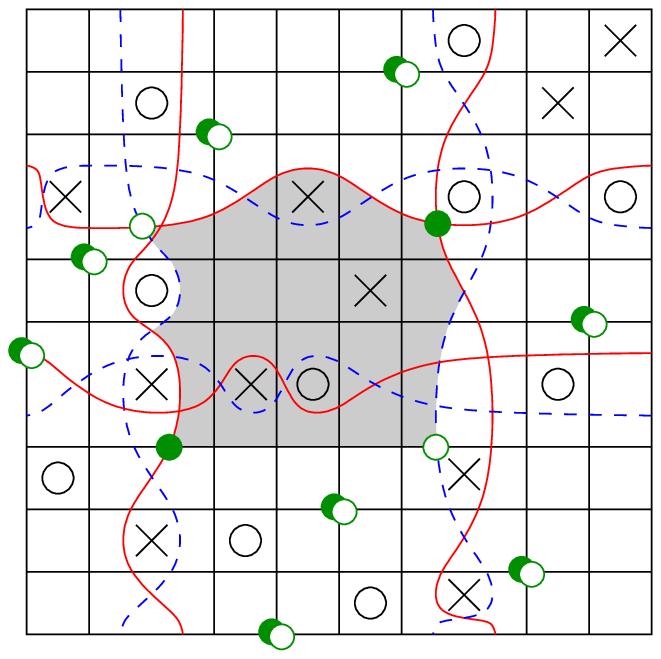} & \hspace{.8cm} & \dessin{3.7cm}{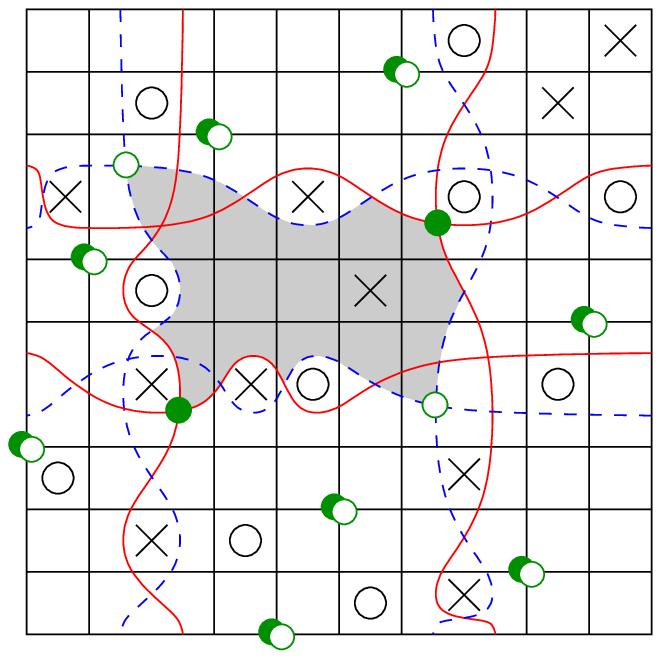}\\[1cm]
    \textrm{an empty pentagon} && \textrm{an empty hexagon} && \textrm{an empty octogon}
  \end{array}
$$ 
  \caption{Examples of grid polygons: {\footnotesize dark dots describe the generator $x$ while hollow ones describe $y$. Grid polygons are depicted by shading. Depending on the number of peaks, there are five kind of grid polygons.}}
    \label{fig:GridPolygons}
\end{figure}

We denote by $\Pol^\circ (G)$ the set of all empty grid polygons on $G$ and by $\Pol^\circ (x,y)$ the set of those which connect $x$ to $y$.

\subsubsection{Signs for grid polygons}
\label{par:SignsGridPolygons}

Following Figure \ref{fig:Hex->Rect} (\S \ref{par:SpiHex}), we define a map
$$
\func{\phi}{\Pol^\circ (G)}{\Rect (G_{0\cdots 0})}
$$
which embank the peaks of a grid polygon using spikes.\\
Then, we define the map $\func{\e}{\Pol^\circ (G)}{\{\pm 1\}}$ by
$$
\e(\pi)=(-1)^{\mu(\pi)}\e\big(\phi(\pi)\big)
$$
where $\mu(\pi)$ counts the number of peaks in $\p \pi$ which point toward the left or toward the top.

\subsubsection{Differential}
\label{par:FinalDiff}

We can now define the map
$$
\func{\p^-_G}{C^-(G)}{C^-(G)}
$$
as the morphism of $\Z[U_{O_1},\cdots,U_{O_{n+k}}]$--modules defined on the generators by
$$
\p^-_G(x) = \sum_{\substack{y \textrm{ generator}\textcolor{white}{R}\\ \textrm{of }C^-(G)}\textcolor{white}{I}} \sum_{\pi \in \Pol^\circ (x,y)} \e(\pi)U_{O_1}^{O_1(\pi)}\cdots U_{O_{n+k}}^{O_{n+k}(\pi)}\cdot y.
$$

\begin{prop}
    The couple $\big(C^-(G),\p^-_G\big)$ is a filtrated chain complex.
\end{prop}

The proof is totally similar to the proof of Proposition \ref{Cchain} (\S \ref{par:Consistency}).

\subsubsection{Invariance theorem}
\label{par:FinalStatement}

Let $G$ be a singular grid diagram.
By convention, we draw the associated singular link diagram $D$ by bending the singular strands to the right or to the top.
Now, we consider a set of winding arcs for every singular row or column.
Then, we choose an orientation for every double point \ie an orientation of the plan spanned by the two transverse tangent vectors at the double point.
Subsequently, we define a set of peaks $P$ by choosing, for a singular row (resp. singular column), the arc intersection of type $\dessin{.4cm}{Cstyle}$ (resp. $\dessinH{.4cm}{rotCstyle}$) if the orientation for the associated double point coincides with the planar orientation inherited from the one of $D$.
Otherwise, we choose the other arc intersection.\\

Finally, we define the chain complex $(C^-(G),\p^-_G)$ as above.

\begin{theo}\label{FinalStatement}
  The homology $H_*(C^-(G),\p^-_G)$ depends only on the associated singular link and on the choice of orientation of the double points.
\end{theo}

\section{Invariance under rotation}
\label{sec:MixedInvariance}

What is left to prove Theorem \ref{FinalStatement} is to deal with the rotation moves.
As a matter of fact, the proofs of invariance under other elementary grid moves and others convention choices given earlier straighforwardly extend to the mixed grid case.

\subsubsection{Presentation of the rotation move}
\label{par:PresentationRotation}

Let $G_h$ and $G_v$ be two grids which differ only from a rotation move:
$$
\begin{array}{ccc}
  \dessin{2.2cm}{RotD1} & \ \longleftrightarrow \  & \dessin{2.2cm}{RotD2}.\\[1.3cm]
G_h && G_v
\end{array}
$$

We consider the four desingularization of the involved singular row and column.
They all can be obtained by removing all winding arcs but two in the following grid $G$ (see (\ref{eq:RotInvDiag})):
$$
\dessin{5cm}{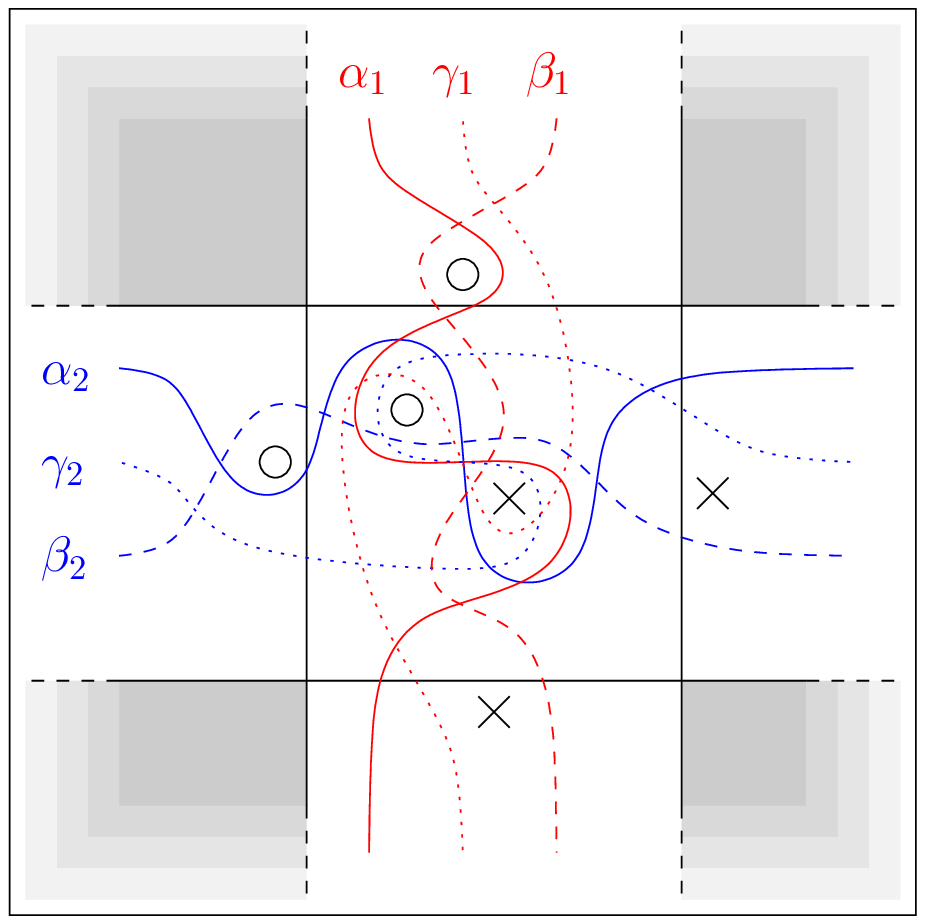}.
$$

We denote by $G_{XY}$ where $(X,Y)\in\{\alpha_1,\beta_1,\gamma_1\}\times\{\alpha_2,\beta_2,\gamma_2\}$ the grid obtained by considering in $G$ only the arcs $X$ and $Y$.\\

The chain complexes $C^-(G_h)$ (resp. $C^-(G_v)$) corresponds to the mapping cones of 
$$
\func{f_{\alpha_1\beta_1}^{\dessin{.2cm}{Cstyle}}}{C^-(G_{\alpha_1\beta_2})}{C^-(G_{\beta_1\beta_2})}\textrm{ (resp. }\func{f_{\alpha_2 \beta_2}^{\dessinH{.2cm}{rotCstyle}}}{C^-(G_{\beta_1\alpha_2})}{C^-(G_{\beta_1\beta_2})}\textrm{)}
$$
which counts the grid polygons containing as a peak the element of $\alpha_1\cup\beta_1$ (resp. $\alpha_2\cup\beta_2$) of type $\dessin{.4cm}{Cstyle}$ (resp. $\dessinH{.4cm}{rotCstyle}$).\\
Moreover, the positive desingularization of $G_h$ and $G_v$ can be linked by a sequence of two regular commutations.
Quasi-isomorphisms $\func{\phi_{\beta_2\gamma_2}}{C^-(G_{\alpha_1\beta_2})}{C^-(G_{\alpha_1\gamma_2})}$ and $\func{\phi_{\beta_1\gamma_1}}{C^-(G_{\gamma_1\alpha_2})}{C^-(G_{\beta_1\alpha_2})}$ have already been defined for these moves.\\

The following diagram sums up all the grids and the chain maps:
\begin{eqnarray}
  \xymatrix@!0@C=3.5cm@R=3.5cm{
    \dessin{2cm}{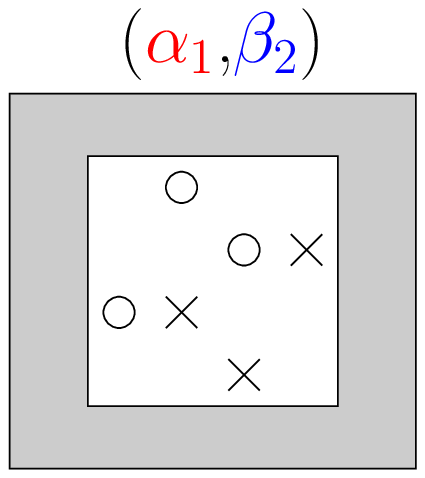} \ar[r]^{\phi_{\beta_2\gamma_2}} \ar[d]_{f_{\alpha_1\beta_1}^{\dessin{.2cm}{Cstyle}}} \ar@/_.55cm/@{-->}[drrr]^\phi & \dessin{2cm}{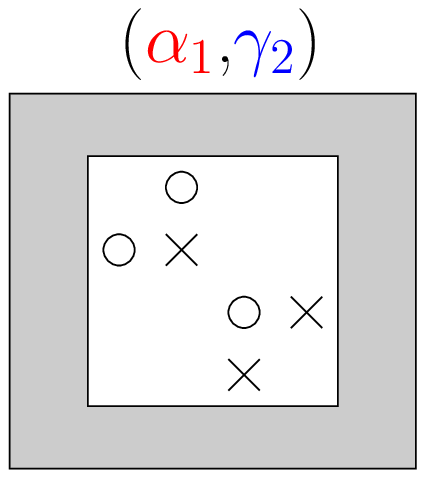} \ar@{-->}[r]^\psi & \dessin{2cm}{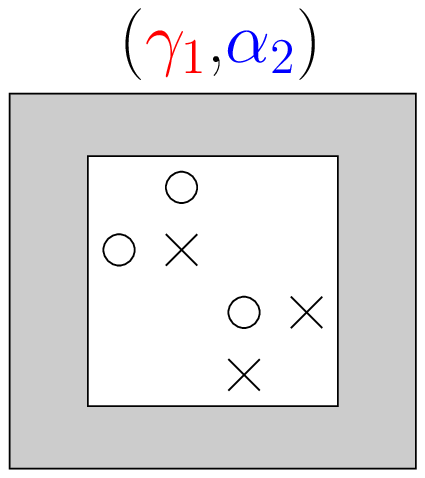} \ar[r]^{\phi_{\gamma_1\beta_1}} & \dessin{2cm}{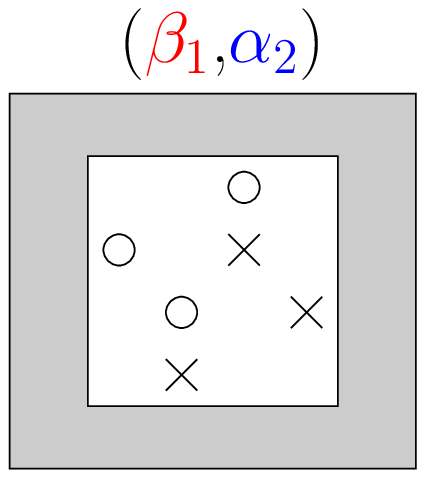} \ar[d]_{f_{\alpha_2\beta_2}^{\dessinH{.2cm}{rotCstyle}}}\\
    \dessin{2cm}{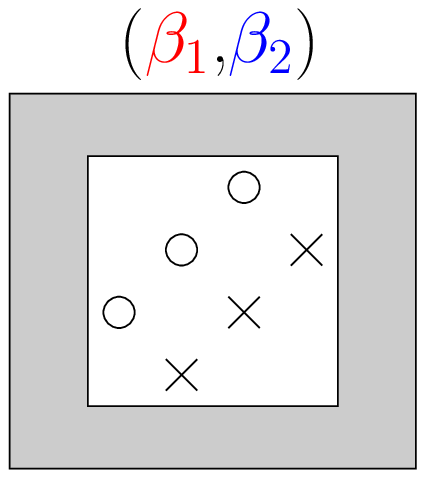} \ar[rrr]^\sim_{\Id} &&& \dessin{2cm}{MIT5}}
  \label{eq:RotInvDiag}
\end{eqnarray}

\subsubsection{Morphisms}
\label{par:RotMorph}

Now, to complete the diagram and make it anti-commuting, we define two maps
\begin{gather*}
  \func{\psi}{C^-(G_{\alpha_1\gamma_2})}{C^-(G_{\gamma_1\alpha_2})}\\
  \func{\phi}{C^-(G_{\alpha_1\beta_2})}{C^-(G_{\beta_1\beta_2})}.
\end{gather*}

The map $\psi$ send a generator $x\in C^-(G_{\alpha_1\gamma_2})$ to the unique generator $y\in C^-(G_{\gamma_1\alpha_2})$ connected to $x$ by a pair of disjoint spikes, containing no $O$ or $X$.
When $\alpha_1\cap\gamma_2\in x$, the spikes are then degenerated and $x=y$ as sets of dots on $G$.

\begin{figure}[h]
  $$
  \begin{array}{ccc}
    \dessin{4cm}{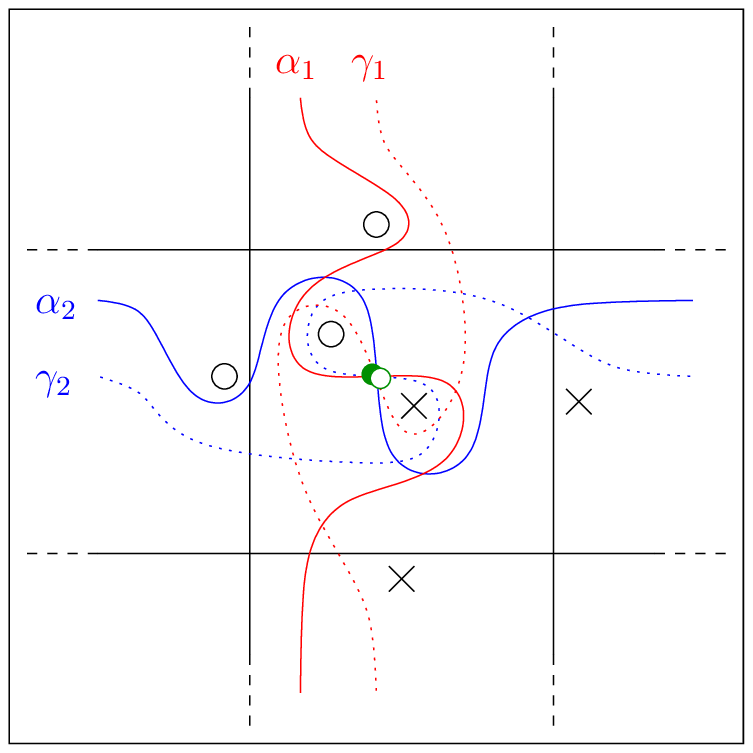} & \hspace{1cm} & \dessin{4cm}{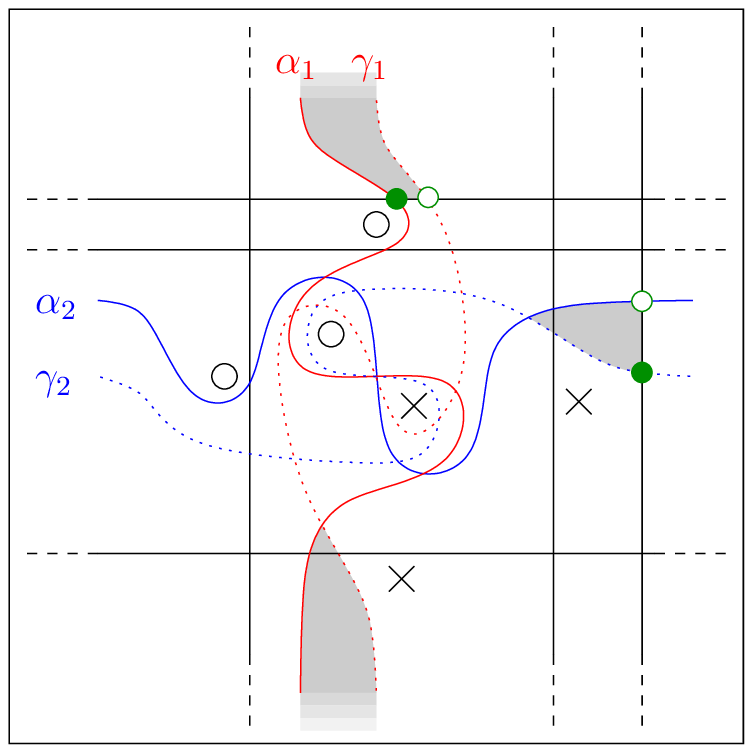} \\[2cm]
    \textrm{when }\alpha_1\cap\gamma_2\in x &&  \textrm{when }\alpha_1\cap\gamma_2\notin x
  \end{array}
  $$
  \caption{A picture for $\psi$:  {\footnotesize the generators $x$ and $y$ coincide except concerning the dark dot which belong to $x$ and the hollow one which belong to $y$. Spikes are depicted by shading.}}
  \label{fig:Psi}
\end{figure}

With generators seen as permutations, the map $\psi$ is nothing more than the identity map.\\ 

To define $\phi$, we need to introduce a new kind of polygons.
For any generators $x$ and $y$ in $C^-(G_{\alpha_1\beta_2})$, a polygon $\Pi=\overline{\pi \setminus B}$ is a \emph{$\beta\gamma\beta$-polygons} connecting $x$ to $y$ if
\begin{itemize}
\item[-] $\pi$ is a grid polygon in $\Pol^\circ (G_{\alpha_1\beta_2})$ connecting $x$ to $y$;
\item[-] $B$ is one of the two bigons delimited by the arcs $\beta_2$ and $\gamma_2$;
\item[-] $B\cap\beta_2\subset\p \pi \cap \p B$.
\end{itemize}
It is said to be \emph{empty} if $\Int (\Pi)\cap x=\emptyset$.

\begin{figure}[h]
 $$
  \begin{array}{ccc}
    \dessin{4cm}{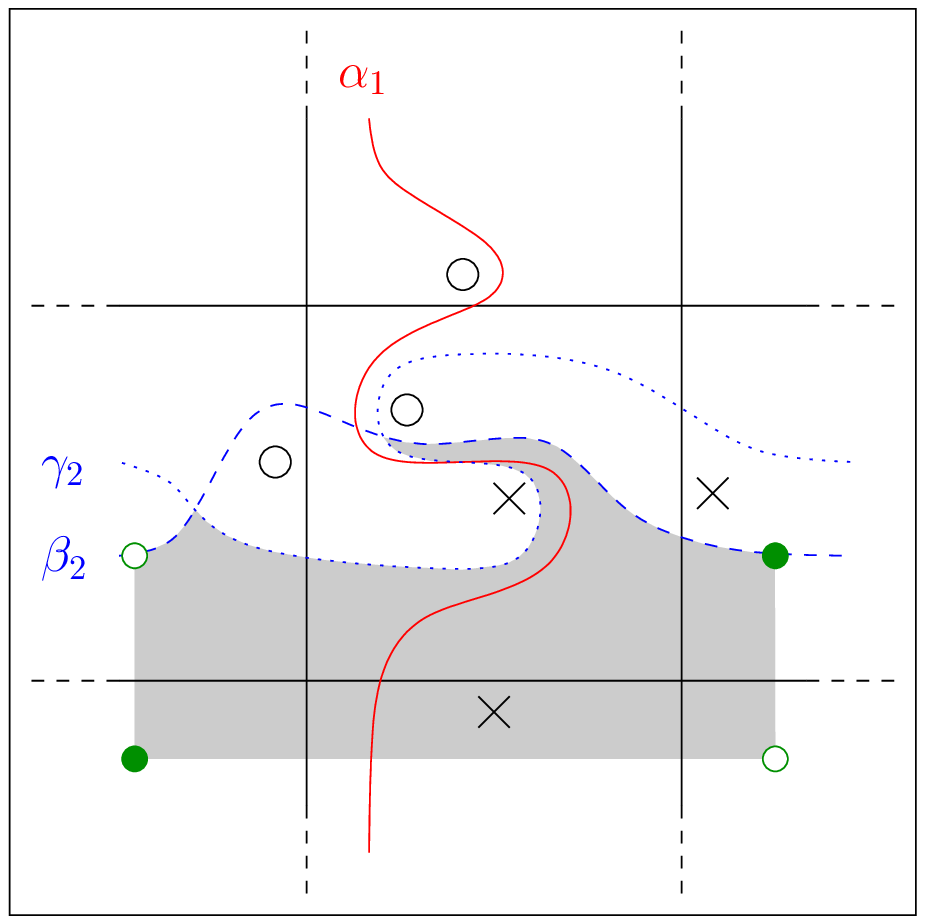} & \hspace{1cm} & \dessin{4cm}{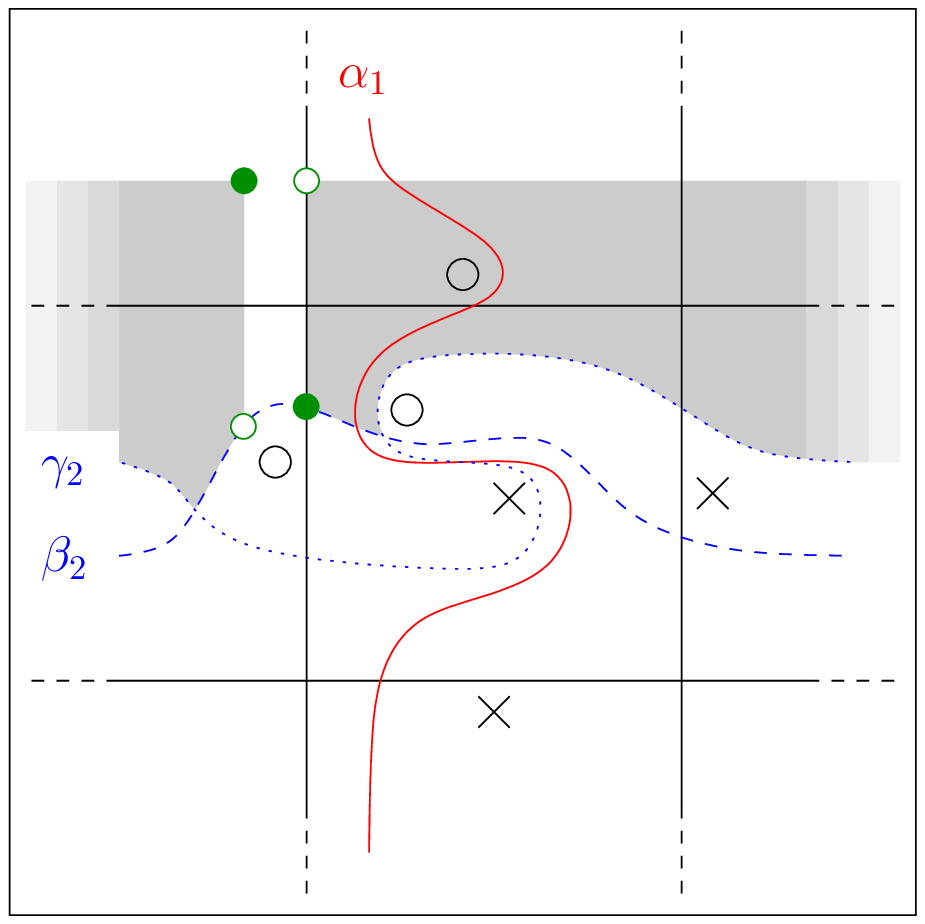}
  \end{array}
$$ 
  \caption{Examples of $\beta\gamma\beta$-polygons: {\footnotesize dark dots describe the generator $x$ while hollow ones describe $y$. Grid polygons are depicted by shading.}}
    \label{fig:bgbPolygons}
\end{figure}

We denote by $\beta\gamma\beta$-$\Pol^\circ (x,y)$ the set of all empty $\beta\gamma\beta$-polygons connecting $x$ to $y$.\\

Now, we define the map $\func{\varphi_1}{C^-(G_{\alpha_1\beta_2})}{C^-(G_{\alpha_1\beta_2})}$ as the morphism of $\Z[U_{O_1},\cdots,U_{O_{n+k}}]$--modules defined on the generators by
$$
\varphi_1(x) = \sum_{\substack{y \textrm{ generator}\textcolor{white}{R}\\ \textrm{of }C^-(G_{\alpha_1\beta_2})}\textcolor{white}{I}} \sum_{\Pi \in \beta\gamma\beta\textrm{-}\Pol^\circ (x,y)} \e(\pi)U_{O_1}^{O_1(\Pi)}\cdots U_{O_{n+k}}^{O_{n+k}(\Pi)}\cdot y
$$
where $\pi$ is the element of $\Pol^\circ (G_{\alpha_1\beta_2})$ such that $\Pi=\overline{\pi \setminus B}$ for a bigon $B$.\\

We define also $\func{\varphi_2}{C^-(G_{\alpha_1\gamma_2})}{C^-(G_{\beta_1\beta_2})}$ by
$$
\varphi_2(x) = \sum_{\substack{y \textrm{ generator}\textcolor{white}{R}\\ \textrm{of }C^-(G_{\beta_1\beta_2})}\textcolor{white}{I}} \sum_{\pi \in \Pol^\circ (x,y)} \e(\pi)U_{O_1}^{O_1(\pi)}\cdots U_{O_{n+k}}^{O_{n+k}(\pi)}\cdot y
$$
where $x$ is a generator of $C^-(G_{\alpha_1\gamma_2})$ and $\Pol^\circ (x,y)$ is the set of empty grid polygons connecting $x$ to $y$. The latter correspond to grid polygons which contain the topmost element of $\alpha_1\cap\beta_1$ and the leftmost of $\beta_2\cap\gamma_2$ as peaks.\\

Then we can set $\phi=f_{\alpha_1\beta_1}^{\dessin{.2cm}{Cstyle}}\circ \varphi_1+\varphi_2\circ\phi_{\beta_2\gamma_2}$.
Since $C^-(G_v)$ and $C^-(G_h)$ can, respectively, be seen as the mapping cones of $f_{\alpha_1\beta_1}^{\dessin{.2cm}{Cstyle}}$ and $f_{\alpha_2\beta_2}^{\dessinH{.2cm}{rotCstyle}}$, we finally define $\func{\Phi}{C^-(G_v)}{C^-(G_h)}$ by
$$
\Phi(x)=\left\{
  \begin{array}{ll}
    \phi_{\gamma_1\beta_1}\circ\psi\circ\phi_{\beta_2\gamma_2}(x) + \phi(x) & \textrm{if }x\textrm{ is a generator of }C^-(G_{\alpha_1\beta_2})\\
    x & \textrm{if }x\textrm{ is a generator of }C^-(G_{\beta_1\beta_2})
  \end{array}
\right..
$$

\subsubsection{Quasi-isomorphism}
\label{par:RotQuasiIsom}

\begin{prop}
  The map $\Phi$ commutes with the differentials.
\end{prop}
\begin{proof}
It is sufficient to prove that

\begin{eqnarray}
f_{\alpha_1\beta_1}^{\dessin{.2cm}{Cstyle}}+\Phi\circ\p^-_{G_{\alpha_1\beta_2}} \equiv f_{\alpha_2\beta_2}^{\dessinH{.2cm}{rotCstyle}}\circ\phi_{\gamma_1\beta_1}\circ\psi\circ\phi_{\beta_2\gamma_2} + \p^-_{G_{\beta_1\beta_2}}\circ\Phi.
\label{eq:Formule}
\end{eqnarray}

Let $x$ be a generator of $C^-(G_{\alpha_1\gamma_2})$, then it is clear from all the constructions made earlier that
\begin{eqnarray}
f_{\alpha_2\beta_2}^{\dessinH{.2cm}{rotCstyle}}\circ\phi_{\gamma_1\beta_1}\circ\psi(x) +  f_{\alpha_1\beta_1}^{\dessin{.2cm}{Cstyle}}\circ\phi_{\gamma_2\beta_2}(x)+\p^-_{G_{\beta_1\beta_2}}\circ\varphi_2(x) + \varphi_2\circ\p^-_{G_{\alpha_1\gamma_2}}(x)=0.
\label{eq:Magie}
\end{eqnarray}

As a matter of fact, Figure \ref{fig:Magie} shows how, thanks to $\psi$, the polygons involved in $\phi_{\gamma_1\beta_1}$ (resp. $f_{\alpha_2\beta_2}^{\dessinH{.2cm}{rotCstyle}}$) can be replaced by the polygons involved in $f_{\alpha_1\beta_1}^{\dessin{.2cm}{Cstyle}}$ (resp. $\phi_{\gamma_2\beta_2}$).\\

\begin{figure}[p]
\vspace{-1cm}
  \centering
    \subfigure[]{\fbox{\xymatrix@R=.6cm@C=1.45cm{
          \dessin{5.5cm}{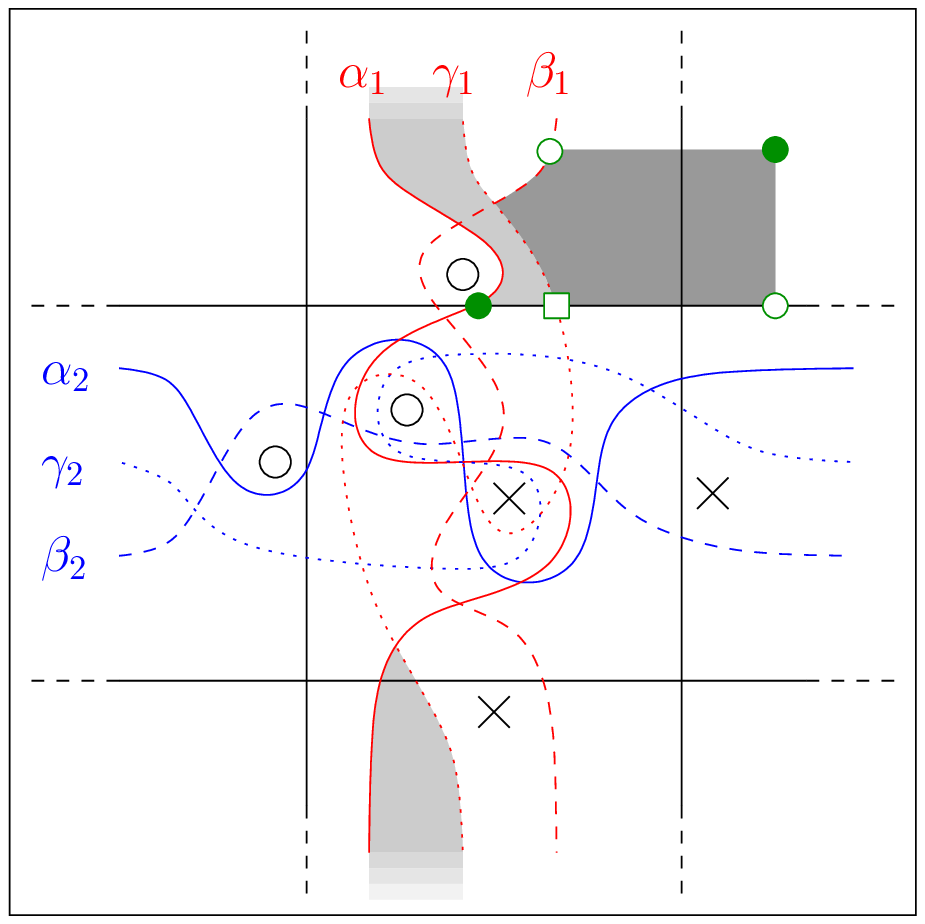} \ar@{^<-_>}[r]_(.47){\phi_{\gamma_1\beta1}\circ \psi}^(.55){f_{\alpha_1\beta_1}^{\dessin{.2cm}{Cstyle}}} & \dessin{5.5cm}{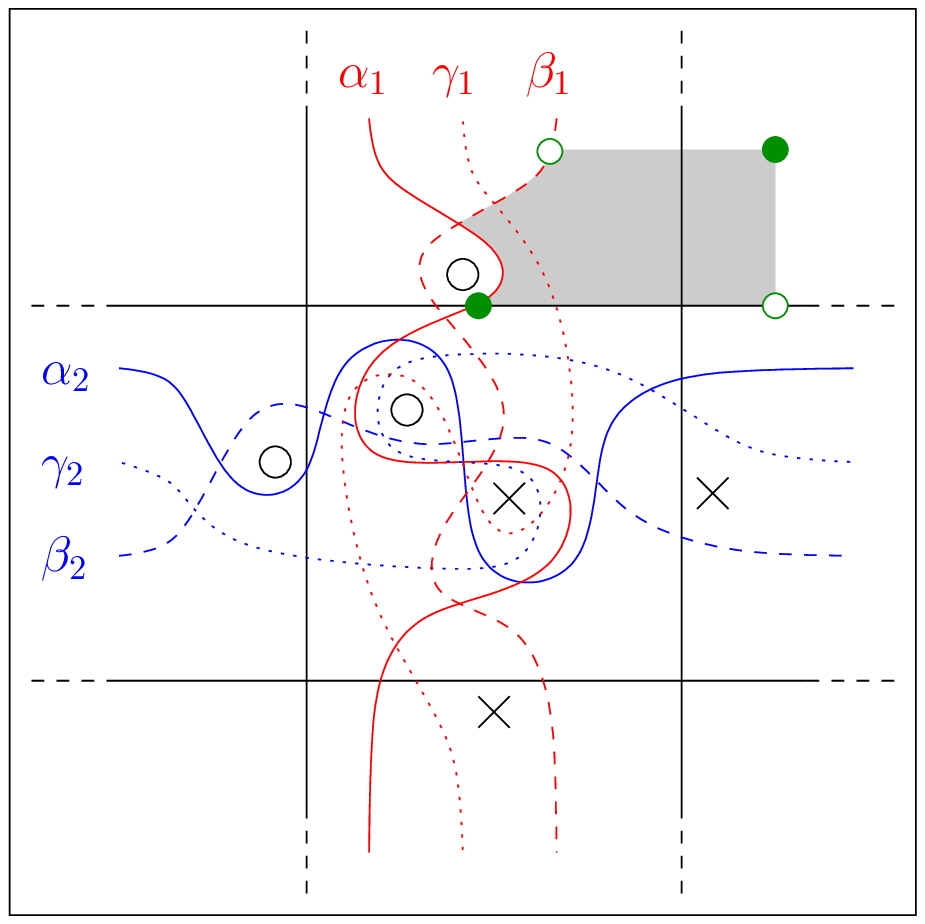}}}}\\
    \subfigure[]{\fbox{\xymatrix@R=.6cm@C=1.45cm{
          \dessin{5.5cm}{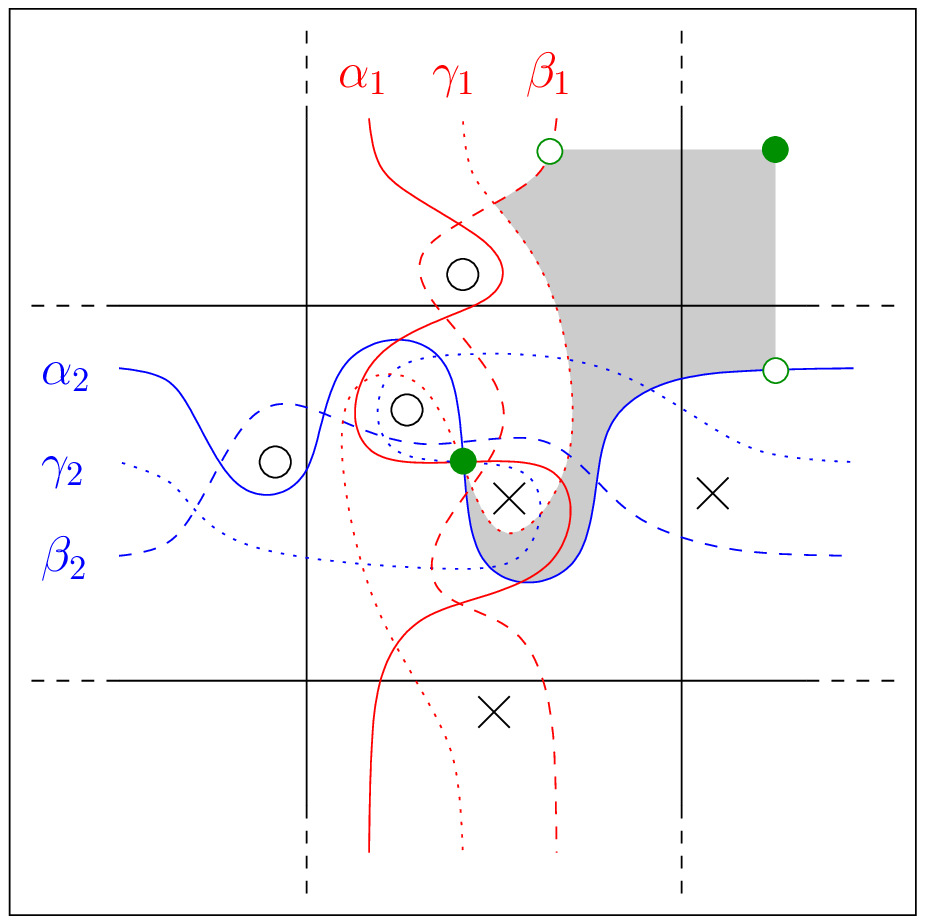} \ar@{^<-_>}[r]_(.47){\phi_{\gamma_1\beta1}\circ \psi}^(.55){f_{\alpha_1\beta_1}^{\dessin{.2cm}{Cstyle}}} & \dessin{5.5cm}{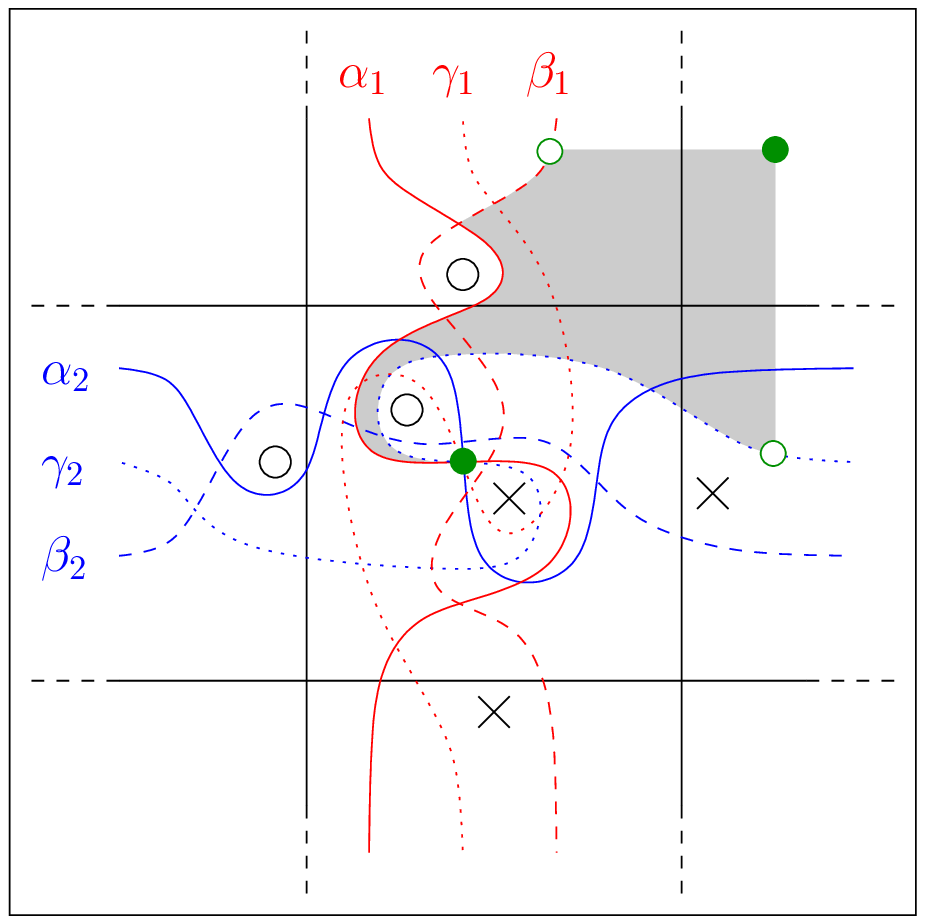}}}}\\
    \subfigure[]{\fbox{\xymatrix@R=.6cm@C=1.45cm{
          \dessin{5.5cm}{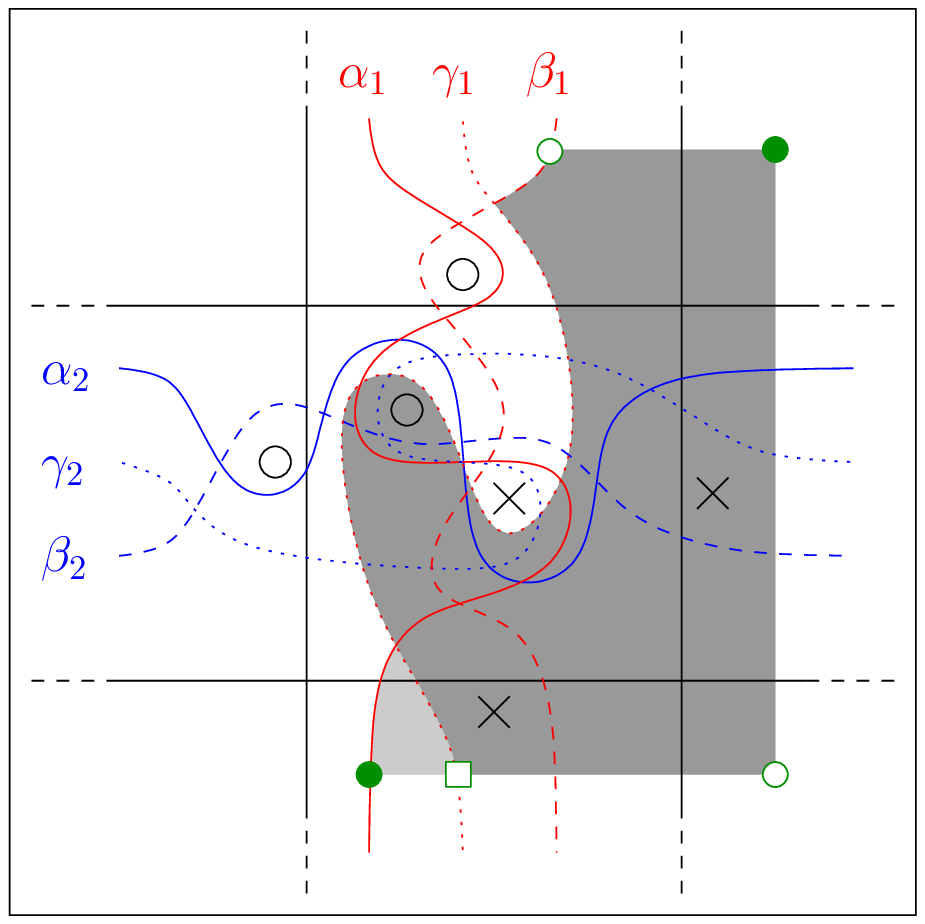} \ar@{^<-_>}[r]_(.47){\phi_{\gamma_1\beta1}\circ \psi}^(.55){f_{\alpha_1\beta_1}^{\dessin{.2cm}{Cstyle}}} & \dessin{5.5cm}{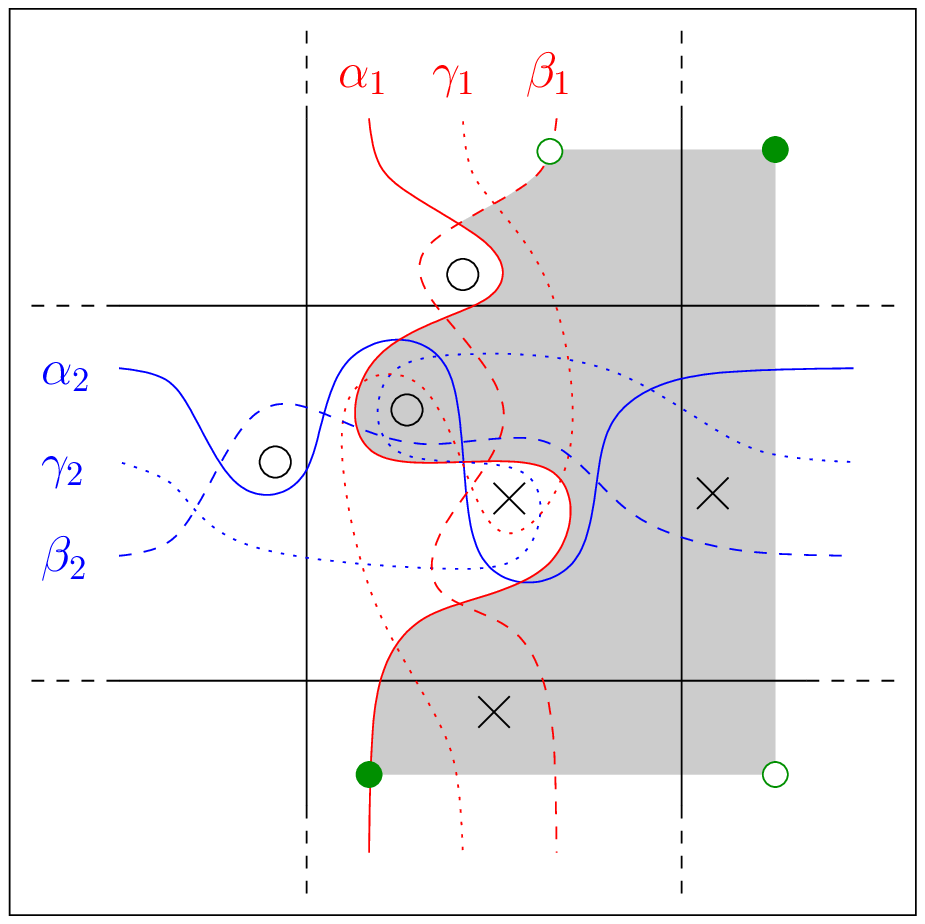}}}}
  \caption{Correspondence between grid polygons: {\footnotesize Dark dots
      describe the initial generator while hollow ones describe the
      final one. Squares describe intermediate states. Polygons are
      depicted by shading. The lightest is the first to occur
      whereas the darkest one is the last. For each polygon(s), we indicate to which map it belongs. Note that associated polygons share the same initial and final generators and the same sign rule. Moreover, they contain the same decorations. Other cases are similar.}}
  \label{fig:Magie}
\end{figure}

Moreover, it follows from the proof of propositions 3.2 and 4.24 in \cite{MOST} that
$$
\Id + \phi_{\gamma_2\beta_2}\circ\phi_{\beta_2\gamma_2} + \p^-_{G_{\alpha_1\beta_2}}\circ\varphi_1 + \varphi_1\circ\p^-_{G_{\alpha_1\beta_2}}\equiv 0.
$$
Then
\begin{eqnarray*}
  0 & \equiv & f_{\alpha_1\beta_1}^{\dessin{.2cm}{Cstyle}} + f_{\alpha_1\beta_1}^{\dessin{.2cm}{Cstyle}}\circ\phi_{\gamma_2\beta_2}\circ\phi_{\beta_2\gamma_2} + f_{\alpha_1\beta_1}^{\dessin{.2cm}{Cstyle}}\circ\p^-_{G_{\alpha_1\beta_2}}\circ\varphi_1 + f_{\alpha_1\beta_1}^{\dessin{.2cm}{Cstyle}}\circ\varphi_1\circ\p^-_{G_{\alpha_1\beta_2}}\\[.3cm]
& \equiv & f_{\alpha_1\beta_1}^{\dessin{.2cm}{Cstyle}} + f_{\alpha_1\beta_1}^{\dessin{.2cm}{Cstyle}}\circ\phi_{\gamma_2\beta_2}\circ\phi_{\beta_2\gamma_2} - \p^-_{G_{\beta_1\beta_2}}\circ f_{\alpha_1\beta_1}^{\dessin{.2cm}{Cstyle}}\circ\varphi_1 + f_{\alpha_1\beta_1}^{\dessin{.2cm}{Cstyle}}\circ\varphi_1\circ\p^-_{G_{\alpha_1\beta_2}}\\
\end{eqnarray*}
Then, from (\ref{eq:Magie}), we obtain
\begin{eqnarray*}
  0 & \equiv & f_{\alpha_1\beta_1}^{\dessin{.2cm}{Cstyle}} - f_{\alpha_2\beta_2}^{\dessinH{.2cm}{rotCstyle}}\circ\phi_{\gamma_1\beta_1}\circ\psi\circ\phi_{\beta_2\gamma_2} - \p^-_{G_{\beta_1\beta_2}}\circ\varphi_2\circ\phi_{\beta_2\gamma_2} - \varphi_2\circ\p^-_{G_{\alpha_1\gamma_2}}\circ\phi_{\beta_2\gamma_2}\\
&& \hspace{6cm} - \p^-_{G_{\beta_1\beta_2}}\circ f_{\alpha_1\beta_1}^{\dessin{.2cm}{Cstyle}}\circ\varphi_1 + f_{\alpha_1\beta_1}^{\dessin{.2cm}{Cstyle}}\circ\varphi_1\circ\p^-_{G_{\alpha_1\beta_2}}\\[.3cm]
& \equiv & f_{\alpha_1\beta_1}^{\dessin{.2cm}{Cstyle}} + f_{\alpha_1\beta_1}^{\dessin{.2cm}{Cstyle}}\circ\varphi_1\circ\p^-_{G_{\alpha_1\beta_2}} + \varphi_2\circ\phi_{\beta_2\gamma_2}\circ\p^-_{G_{\alpha_1\beta_2}}  - f_{\alpha_2\beta_2}^{\dessinH{.2cm}{rotCstyle}}\circ\phi_{\gamma_1\beta_1}\circ\psi\circ\phi_{\beta_2\gamma_2}\\
&& \hspace{6cm}  - \p^-_{G_{\beta_1\beta_2}}\circ\varphi_2\circ\phi_{\beta_2\gamma_2} - \p^-_{G_{\beta_1\beta_2}}\circ f_{\alpha_1\beta_1}^{\dessin{.2cm}{Cstyle}}\circ\varphi_1.
\end{eqnarray*}

This proves (\ref{eq:Formule}).
\end{proof}

Now we consider the filtration which send a generator $x$ of $C^-(G_v)$ (resp $C^-(G_h)$) to $0$ or $1$ depending on the resolution of the considered singular column (resp. singular row).
The graded part associated to $\Phi$ is then a composition of quasi-isomorphisms.
Acoording to Corollary \ref{QuasiIso} (\S \ref{par:MapCones}), the map $\Phi$ induces thus an isomorphism in homology.\\

By symetry, the proof can be adapted to fit the four rotation moves given in Figure \ref{fig:FourRotations} which correspond, up to equivalency, to the two rotation moves and the two choice of arcs intersection.

\begin{figure}[h]
  $$
\xymatrix{
\dessin{2.5cm}{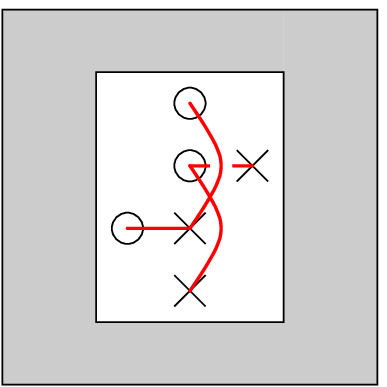} \ar@{^<-_>}[r]_(.42){\dessin{.3cm}{Cstyle}}^(.57){\dessinH{.3cm}{rotCstyle}} & \dessin{2.5cm}{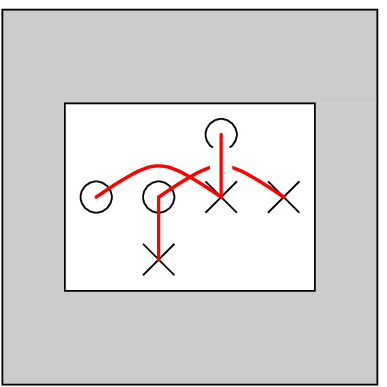} & & \dessin{2.5cm}{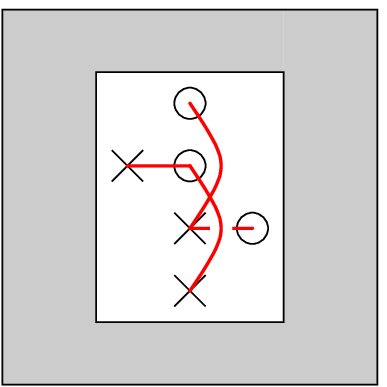} \ar@{^<-_>}[r]_(.42){\dessin{.3cm}{Cstyle}}^(.57){\dessinH{.3cm}{rotCpstyle}} & \dessin{2.5cm}{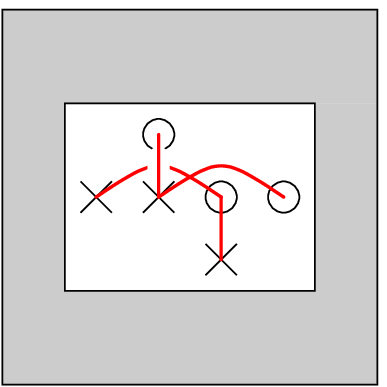}\\
 \dessin{2.5cm}{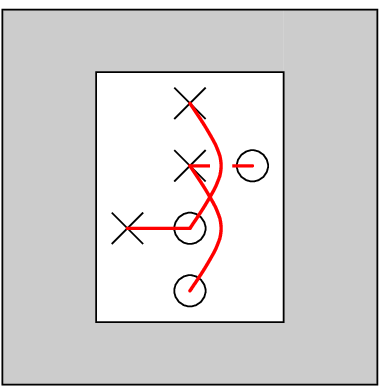} \ar@{^<-_>}[r]_(.42){\dessin{.3cm}{Cpstyle}}^(.57){\dessinH{.3cm}{rotCpstyle}} & \dessin{2.5cm}{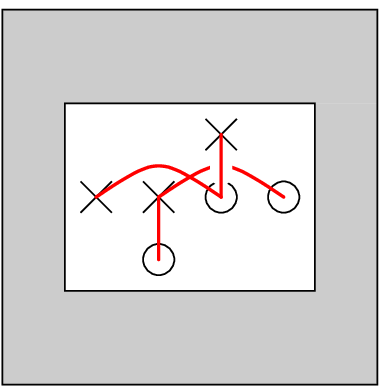} & &\dessin{2.5cm}{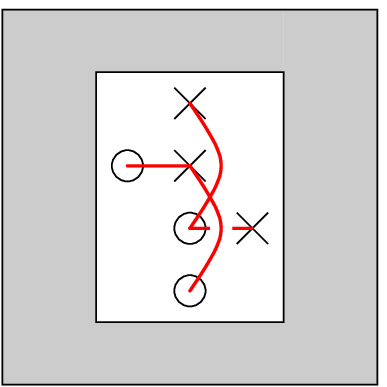} \ar@{^<-_>}[r]_(.42){\dessin{.3cm}{Cpstyle}}^(.57){\dessinH{.3cm}{rotCpstyle}} & \dessin{2.5cm}{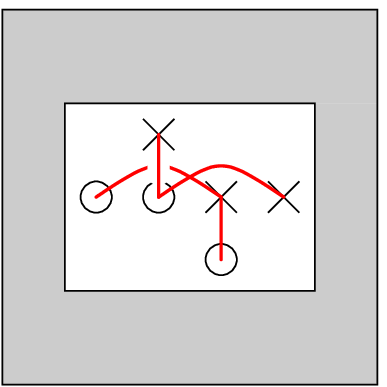}
}
  $$
  \caption{Four rotation moves: {\footnotesize For each move, we indicate the convention for arcs intersection which is considered in the proof of invariance.}}
  \label{fig:FourRotations}
\end{figure}

Finally we conclude by noticing that these four moves simultaneously change or preserve the convention for the choice of the arcs intersection and the orientation of the double point induced by the planar diagram.

\section{Symetry properties}
\label{sec:SymetryProperties}

Finally, we conclude with symetry properties which reflect the properties of the Alexander polynomial.

\subsubsection{Behavior of $HF^-$ with regard to usual knots operations}
\label{par:SymProp}

Let $L$ be an oriented link with $\ell$ components.
We denote by $L^!$ its mirror image and by $-L$ the link obtained by reversing the orientation.  
Moreover, we denote by $P$ a choice of orientation for all double points of $L$ and by $-P$ the opposite choice for all of them.

\begin{prop}
  If we denote by $\widehat{HF}(K,S)$ the singular link Floer homologies for a link $K$ obtained by considering a set of peaks $S$, then
  \begin{itemize}
  \item[-] $\widehat{HF}(L,P)\simeq  \widehat{HF}(-L,P)$;
  \item[-] $\forall i,j\in\N,\widehat{HF}_i^j(L,P)\simeq  \widehat{HF}_{i-2j}^{-j}(L,-P)$;
  \item[-] $\forall i,j\in\N,\widehat{HF}_i^j(L^!,P)\simeq\widehat{HF}_{-j}^{\circ}{}^{\hspace{-.2cm}-i+1-\ell+k}(L,P)$
  \end{itemize}
where $\widehat{HF}_j^\circ{}^i$ stands for the Maslov $i^\textrm{th}$ and Alexander $j^\textrm{th}$ group of cohomology associated to $\widehat{HF}$.
\end{prop}
\begin{proof}
  Let $G$ be a grid diagram for $L$.
  Let $n$ be the size of the regular grids obtained by desingularizing all the singular rows and columns.\\
  
  Flipping $G$ along the line $x=y$ gives a grid $-G$ for $-L$.
  The same operation induces a bijection between the generators of $C^-(G)$ and those of $C^-(-G)$ seen as set of dots on the grids.
  Since it sends the $0$--resolution (resp. $1$--resolution) of a singular row or column to the $0$--resolution (resp. $1$--resolution) of its image; and since it doesn't affect the relative positions of dots and decorations, this bijection preserves the Maslov and Alexander degrees.
  Moreover, the orientation on double points inherited from the one of $-G$ is reversed compared with the one of $G$.
  So, as soon as we substitute $-U_{O_i}$ to $U_{O_i}$ in the image; and because of the antisymetry of the convention used with regard to the choice of orientation for the double points, it clearly commutes with the differential.\\
  This proves the first statement.\\
  
  The second is obtained by switching the role of $\O$ and $\X$.
  The new grid, denoted by $-G'$, describe then $-L$.
  We denote by $M_1$, $M_2$, $A_1$ and $A_2$ the Maslov and Alexander degrees associated to each grid.
  Because of their definition and since we have switched the $O$'s and the $X$'s, there is no difficulty in checking that, for a given generator $x$ seen as a set of dots,
  $$
  M_2(x)-2A_2(x)=M_1(x) + n -\ell
  $$
  $$
  -A_2(x)=A_1(x) + n -\ell
  $$
  Moreover, since they avoid the polygons containing any decoration, the differentials, when sending all the $U_O$ to zero, clearly coincide.
  Then, according to the Proposition \ref{Dunealautre} (\S \ref{par:RelationGraded}), for all integers $i$ and $j$, we have defined an isomorphism
  
  \begin{eqnarray}
    \left(\widehat{HF}(G,P)\otimes V^{\otimes (n-\ell)}\right)_i^j \simeq \left(\widehat{HF}(-G',-P)\otimes V^{\otimes (n-\ell)}\right)_{i-2j-n+\ell}^{-j - n +\ell}
    \label{eq:Orient}
  \end{eqnarray}
  
  As a matter of fact, switching the decorations also switch the conventions for the orientations of double points.\\
  Now, we denote by $v_+$ and $v_-$ the generators of, respectively, highest and lowest degree in $V^{\otimes (n-\ell)}$.
  Then, by identifying $\widehat{HF}$ in $\widehat{HF}\otimes V^{\otimes (n-\ell)}$ with $\widehat{HF}\otimes v_+$ in the left-hand side of (\ref{eq:Orient}), and with $\widehat{HF}\otimes v_-$ in its right-hand side, we obtain
  $$
  \widehat{HF}_i^j(G,P) \simeq \widehat{HF}_{i-2j}^{-j}(-G',-P).
  $$
  Finally, we use the first statement to conclude.\\
  
  For the last statement, we rotate $G$ $90°$ and get a grid $G^!$ for $-L^!$.
  Nature of desingularization for singular rows and columns are then swapped.
  We denote by $M_1$, $M'_2$, $A_1$ and $A'_2$ the Maslov and Alexander degrees associated to each grid.
  Since all dots and decorations are on distincts vertical and horizontal lines, we can check that, for a given generator $x$,
  $$
  M_1(x)+M_2(x)= 1 - n + k
  $$
  $$
  A_1(x) + A_1(x) = \ell - n
  $$
  The differential induced on $G$ by $G^!$ counts the preimages of a generator under the differential given by $G$.
  Then, after having substituing $-U_{O_i}$ to $U_{O_i}$, it corresponds to the codifferential defined by $G$ on the dual basis of the usual generators.
  Moreover, the changes on double points orientation and arcs intersections conventions correspond.
Finally, we obtain
$$
\left(\widehat{HF}(G^!,P)\otimes V^{\otimes (n-\ell)}\right)_i^j \simeq \left(\widehat{HF}^\circ(G,P)\otimes V^{\otimes (n-\ell)}\right)_{-j-n+\ell}^{-i+1-n+k}.
$$
As above, it induces then
$$
\widehat{HF}_i^j(G^!,P) \simeq \widehat{HF}_{-j}^\circ{}^{-i+1-\ell+k}(G,P).
$$
\end{proof}

\begin{remarque}
  Since a global reversing of the orientation of double points may modify the homology, we loose the symetry satisfied by link Floer homologies in the regular case. A counter-example is given in Appendix \ref{sec:Computations}. It corresponds to the knot $90_{44}$ with one singularized crossing.
\end{remarque}

%%% Local Variables: 
%%% mode: latex
%%% TeX-master: "These"

%%% End: 

% Calculs
\chapter{Computations}
\label{sec:Computations}

We have written a program in oCaml which computes both graded singular link Floer homologies $\widehat{HL}$ and  $\widehat{HL}_{\dessin{.2cm}{Cpstyle}}$ with $\FF_2$--coefficients for any given singular grid. But regrettably, times of computation make it unusable as soon as the grid exceeds $9$ rows.\\

In this last part, we gather results of computations. For every link, we also give the grid we have used. Maslov grading is given horizontally and Alexander one vertically. When grids become too large, we give the bigraded dimension only.\\
As far as possible, we have tried to give only computations which cannot be deduced from the other.\\

For knots with less than six crossings, orientation for double points is denoted by a $+$ sign if it corresponds to the planar orientation and by a $-$ sign otherwise.
For knots for more crossings, the computations are all made using the $\dessin{.4cm}{Cstyle}$ convention.

\section*{Trivial knot}
\label{sec:Trivial}
$$
\fbox{$\dessin{1.86cm}{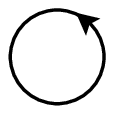}$}
$$

\vspace{.5cm}

\begin{center}
  \begin{tabular}{p{.3\columnwidth}p{.52\columnwidth}p{.15\columnwidth}}
    \centering{$\dessin{.8cm}{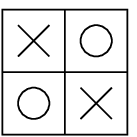}$} & 
    \centering{
      \begin{tabular}{|c|c|}
        \cline{2-2}
        \multicolumn{1}{c|}{} & 0 \\
        \hline
        {\Large \strut} 0 & $\2Z$ \\ 
        \hline 
      \end{tabular}}
    & \centering{$\dessin{.84cm}{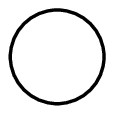}$}
  \end{tabular}
\end{center}

\newpage{}

\section*{Trefoil \& its singularizations}
\label{sec:Trefoil}
$$
\fbox{$\dessin{3.1cm}{3_1}$}
$$

\vspace{.5cm}

\begin{center}
  \begin{tabular}{p{.3\columnwidth}p{.52\columnwidth}p{.15\columnwidth}}
    \centering{$\dessin{2cm}{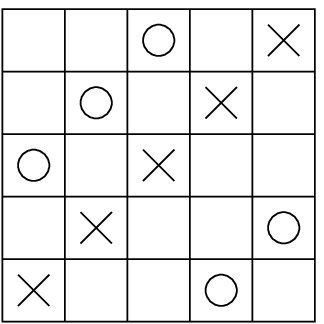}$} & 
    \centering{
      \begin{tabular}{|c|c|c|c|}
        \cline{2-4}
        \multicolumn{1}{c|}{} & 0 & 1 & 2\\
        \hline
        {\Large \strut} 1 & & &$\2Z$ \\ 
        \hline 
        {\Large \strut} 0 & & $\2Z$ &\\ 
        \hline
        {\Large \strut} -1 & $\2Z$ & &\\ 
        \hline 
      \end{tabular}}
    & \centering{$\dessin{1.4cm}{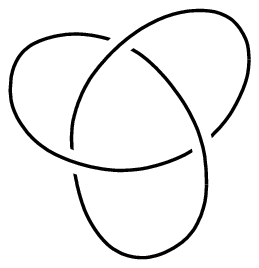}$}
  \end{tabular}
\end{center}

\begin{center}
  \begin{tabular}{p{.3\columnwidth}p{.52\columnwidth}p{.15\columnwidth}}
    \centering{$\dessin{2cm}{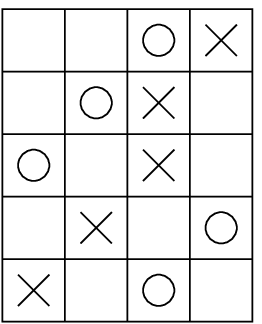}$} & 
    \centering{
      \begin{tabular}{|c|c|c|c|}
        \cline{2-4}
        \multicolumn{1}{c|}{} & 0 & 1 & 2\\
        \hline
        {\Large \strut} 1 & & &$\2Z$ \\ 
        \hline 
        {\Large \strut} 0 & & $\2Z^2$ &\\ 
        \hline
        {\Large \strut} -1 & $\2Z$ & &\\ 
        \hline 
      \end{tabular}}
    & \centering{$\dessin{1.4cm}{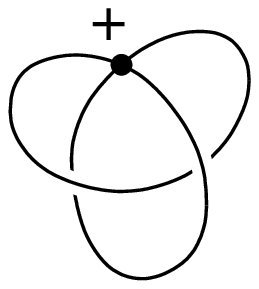}$}
  \end{tabular}
\end{center}

\begin{center}
  \begin{tabular}{p{.3\columnwidth}p{.52\columnwidth}p{.15\columnwidth}}
    \centering{$\dessin{2cm}{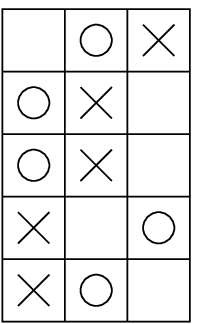}$} & 
    \centering{
      \begin{tabular}{|c|c|c|c|}
        \cline{2-4}
        \multicolumn{1}{c|}{} & 0 & 1 & 2\\
        \hline
        {\Large \strut} 1 & & & $\2Z$ \\ 
        \hline 
        {\Large \strut} 0 & & $\2Z^2$ & \\ 
        \hline
        {\Large \strut} -1 & $\2Z$ & & \\ 
        \hline 
      \end{tabular}}
    & \centering{$\begin{array}{c}
        \dessin{1.4cm}{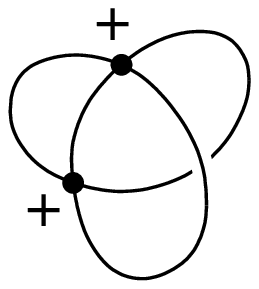} \\ \dessin{1.4cm}{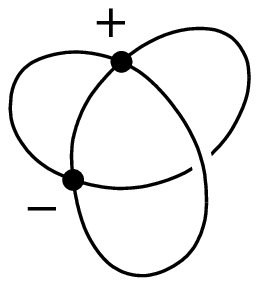}
      \end{array}$}
  \end{tabular}
\end{center}

\begin{center}
  \begin{tabular}{p{.3\columnwidth}p{.52\columnwidth}p{.15\columnwidth}}
    \centering{$\dessin{2.4cm}{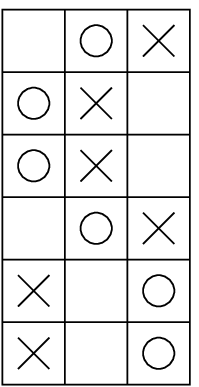}$} & 
    \centering{Acyclic}
    & \centering{$\dessin{1.4cm}{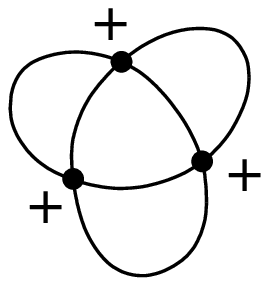}$}
  \end{tabular}
\end{center}

\begin{center}
  \begin{tabular}{p{.3\columnwidth}p{.52\columnwidth}p{.15\columnwidth}}
    \centering{$\dessin{2.4cm}{3_1sssG}$} & 
    \centering{
      \begin{tabular}{|c|c|c|c|c|}
        \cline{2-5}
        \multicolumn{1}{c|}{} & 0 & 1 & 2 & 3\\
        \hline
        {\Large \strut} 1 & & & $\2Z$ & $\2Z$ \\ 
        \hline 
        {\Large \strut} 0 & & $\2Z^2$ & $\2Z^2$ & \\ 
        \hline
        {\Large \strut} -1 & $\2Z$ &  $\2Z$ & & \\ 
        \hline 
      \end{tabular}}
    & \centering{$\dessin{1.4cm}{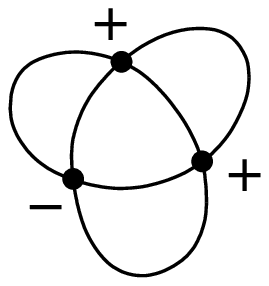}$}
  \end{tabular}
\end{center}

\newpage{}

\section*{Figure eight \& its singularizations}
\label{sec:Fig8}
$$
\fbox{$\dessin{3.1cm}{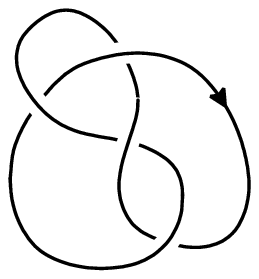}$}
$$

\vspace{.5cm}

\begin{center}
  \begin{tabular}{p{.3\columnwidth}p{.52\columnwidth}p{.15\columnwidth}}
    \centering{$\dessin{2.4cm}{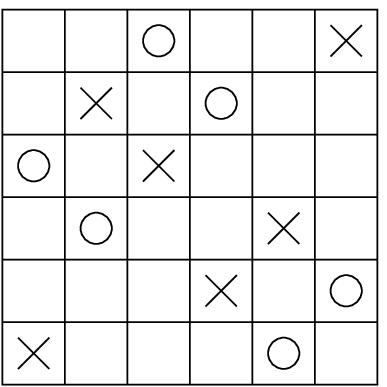}$} & 
    \centering{
      \begin{tabular}{|c|c|c|c|}
        \cline{2-4}
        \multicolumn{1}{c|}{} & -1 & 0 & 1\\
        \hline
        {\Large \strut} 1 & & &$\2Z$ \\ 
        \hline 
        {\Large \strut} 0 & & $\2Z^3$ &\\ 
        \hline
        {\Large \strut} -1 & $\2Z$ & &\\ 
        \hline 
      \end{tabular}}
    & \centering{$\dessin{1.4cm}{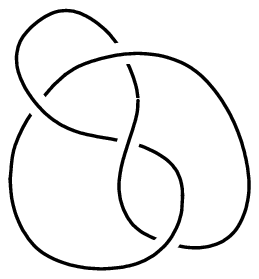}$}
  \end{tabular}
\end{center}

\begin{center}
  \begin{tabular}{p{.3\columnwidth}p{.52\columnwidth}p{.15\columnwidth}}
    \centering{$\dessin{2.4cm}{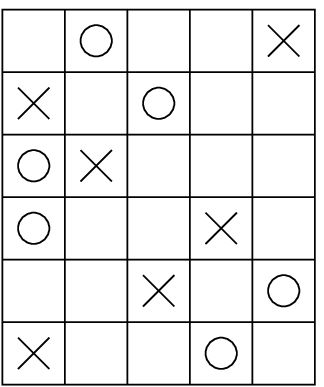}$} & 
    \centering{
      \begin{tabular}{|c|c|c|c|}
        \cline{2-4}
        \multicolumn{1}{c|}{} & 0 & 1 & 2\\
        \hline
        {\Large \strut} 1 & & &$\2Z$ \\ 
        \hline 
        {\Large \strut} 0 & & $\2Z^2$ &\\ 
        \hline
        {\Large \strut} -1 & $\2Z$ & &\\ 
        \hline 
      \end{tabular}}
    & \centering{$\dessin{1.4cm}{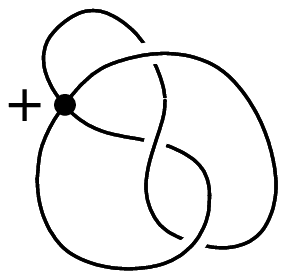}$}
  \end{tabular}
\end{center}

\begin{center}
  \begin{tabular}{p{.3\columnwidth}p{.52\columnwidth}p{.15\columnwidth}}
    \centering{$\dessin{2.4cm}{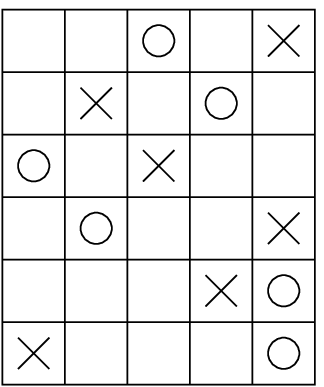}$} & 
    \centering{
      \begin{tabular}{|c|c|c|c|}
        \cline{2-4}
        \multicolumn{1}{c|}{} & -1 & 0 & 1\\
        \hline
        {\Large \strut} 1 & & &$\2Z$ \\ 
        \hline 
        {\Large \strut} 0 & & $\2Z^2$ &\\ 
        \hline
        {\Large \strut} -1 & $\2Z$ & &\\ 
        \hline 
      \end{tabular}}
    & \centering{$\dessin{1.4cm}{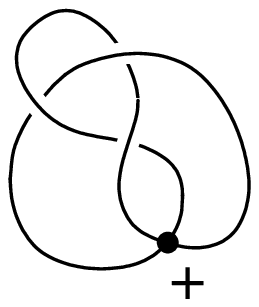}$}
  \end{tabular}
\end{center}

\begin{center}
  \begin{tabular}{p{.3\columnwidth}p{.52\columnwidth}p{.15\columnwidth}}
    \centering{$\dessin{2.4cm}{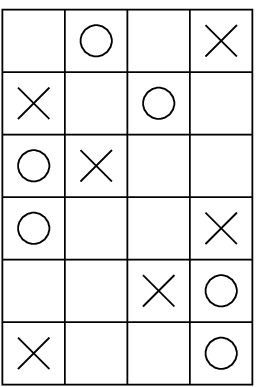}$} & 
    \centering{
      \begin{tabular}{|c|c|c|c|}
        \cline{2-4}
        \multicolumn{1}{c|}{} & 0 & 1 & 2 \\
        \hline
        {\Large \strut} 1 & & & $\2Z$ \\ 
        \hline 
        {\Large \strut} 0 & & $\2Z^2$ & \\ 
        \hline
        {\Large \strut} -1 & $\2Z$ &  & \\ 
        \hline 
      \end{tabular}}
    & \centering{$
      \begin{array}{c}
        \dessin{1.4cm}{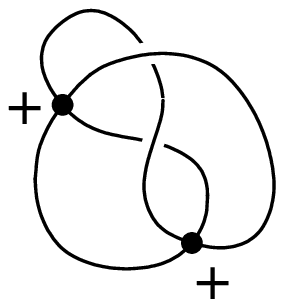} \\ \dessin{1.4cm}{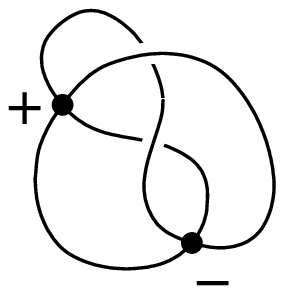}
      \end{array}
      $}
  \end{tabular}
\end{center}

\begin{center}
  \begin{tabular}{p{.3\columnwidth}p{.52\columnwidth}p{.15\columnwidth}}
    \centering{$\dessin{2.4cm}{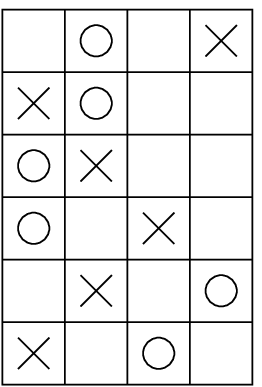}$} & 
    \centering{Acyclic}
    & \centering{$\dessin{1.4cm}{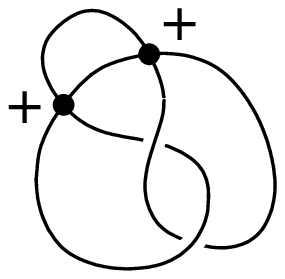}$}
  \end{tabular}
\end{center}

\begin{center}
  \begin{tabular}{p{.3\columnwidth}p{.52\columnwidth}p{.15\columnwidth}}
    \centering{$\dessin{2.4cm}{4_1ss2G}$} & 
    \centering{
      \begin{tabular}{|c|c|c|c|c|}
        \cline{2-5}
        \multicolumn{1}{c|}{} & 0 & 1 & 2 & 3\\
        \hline
        {\Large \strut} 1 & & & $\2Z$ & $\2Z$ \\ 
        \hline 
        {\Large \strut} 0 & & $\2Z^2$ & $\2Z^2$ & \\ 
        \hline
        {\Large \strut} -1 & $\2Z$ &  $\2Z$ & & \\ 
        \hline 
      \end{tabular}}
    & \centering{$\dessin{1.4cm}{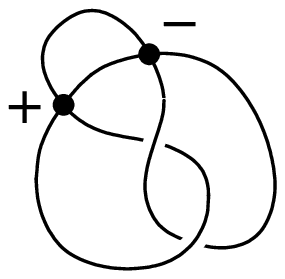}$}
  \end{tabular}
\end{center}

\begin{center}
  \begin{tabular}{p{.3\columnwidth}p{.52\columnwidth}p{.15\columnwidth}}
    \centering{$\dessin{3.2cm}{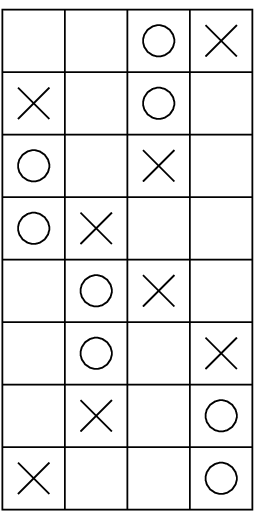}$} & 
    \centering{
      \begin{tabular}{|c|c|c|c|c|c|}
        \cline{2-6}
        \multicolumn{1}{c|}{} & 0 & 1 & 2 & 3 & 4 \\
        \hline
        {\Large \strut} 1 & & & $\2Z$ & $\2Z^2$ & $\2Z$ \\ 
        \hline 
        {\Large \strut} 0 & & $\2Z^2$ & $\2Z^4$ & $\2Z^2$ & \\ 
        \hline
        {\Large \strut} -1 & $\2Z$ & $\2Z^2$ & $\2Z$ & & \\ 
        \hline 
      \end{tabular}}
    & \centering{$
      \begin{array}{c}
        \dessin{1.4cm}{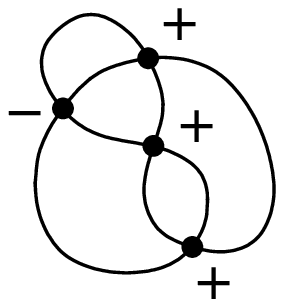} \\ \dessin{1.4cm}{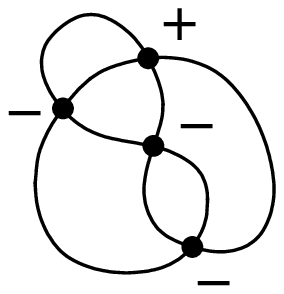}
      \end{array}
      $}
  \end{tabular}
\end{center}

\begin{center}
  \begin{tabular}{p{.3\columnwidth}p{.52\columnwidth}p{.15\columnwidth}}
    \centering{$\dessin{3.2cm}{4_1ssssG}$} & 
    \centering{
      \begin{tabular}{|c|c|c|c|c|c|}
        \cline{2-6}
        \multicolumn{1}{c|}{} & 0 & 1 & 2 & 3 & 4 \\
        \hline
        {\Large \strut} 1 & & & $\2Z$ & $\2Z^2$ & $\2Z$ \\ 
        \hline 
        {\Large \strut} 0 & & $\2Z^2$ & $\2Z^4$ & $\2Z^2$ & \\ 
        \hline
        {\Large \strut} -1 & $\2Z$ & $\2Z^2$ & $\2Z$ & & \\ 
        \hline 
      \end{tabular}}
    & \centering{$
      \begin{array}{c}
        \dessin{1.4cm}{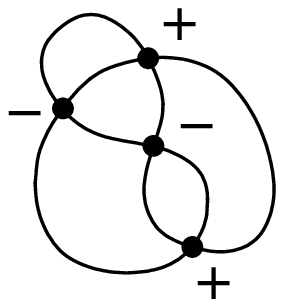} \\ \dessin{1.4cm}{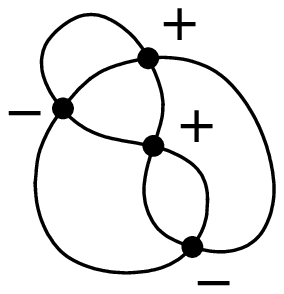}
      \end{array}
      $}
  \end{tabular}
\end{center}

\newpage{}

\section*{Knot $5_1$ \& its singularizations}
\label{sec:5.1}
$$
\fbox{$\dessin{3.41cm}{5_1}$}
$$

\vspace{.5cm}

\begin{center}
  \begin{tabular}{p{.3\columnwidth}p{.52\columnwidth}p{.15\columnwidth}}
    \centering{$\dessin{2.8cm}{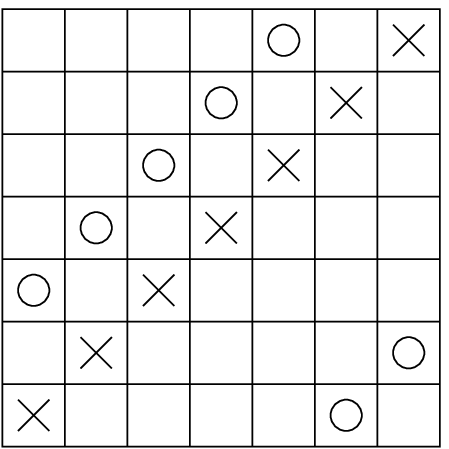}$} & 
    \centering{
      \begin{tabular}{|c|c|c|c|c|c|}
        \cline{2-6}
        \multicolumn{1}{c|}{} & 0 & 1 & 2 & 3 & 4 \\
        \hline
        {\Large \strut} 2 & & & & & $\2Z$ \\ 
        \hline
        {\Large \strut} 1 & & & & $\2Z$ & \\ 
        \hline 
        {\Large \strut} 0 & & & $\2Z$ & & \\ 
        \hline
        {\Large \strut} -1 & & $\2Z$ & & & \\ 
        \hline 
        {\Large \strut} -2 & $\2Z$ & & & & \\ 
        \hline 
      \end{tabular}}
    & \centering{$\dessin{1.54cm}{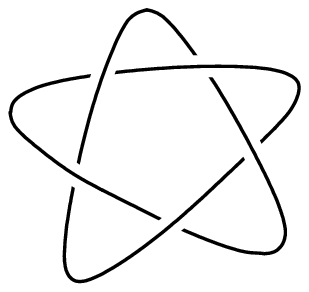}$}
  \end{tabular}
\end{center}

\begin{center}
  \begin{tabular}{p{.3\columnwidth}p{.52\columnwidth}p{.15\columnwidth}}
    \centering{$\dessin{2.8cm}{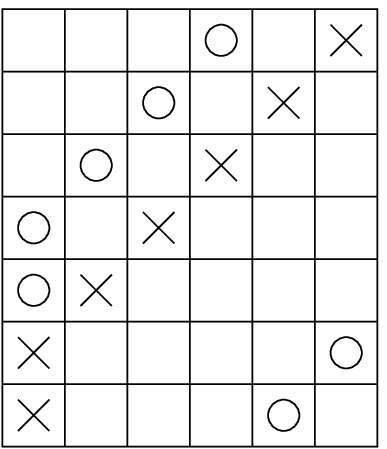}$} & 
    \centering{
      \begin{tabular}{|c|c|c|c|c|c|}
        \cline{2-6}
        \multicolumn{1}{c|}{} & 0 & 1 & 2 & 3 & 4 \\
        \hline
        {\Large \strut} 2 & & & & & $\2Z$ \\ 
        \hline
        {\Large \strut} 1 & & & & $\2Z^2$ & \\ 
        \hline 
        {\Large \strut} 0 & & & $\2Z^2$ & & \\ 
        \hline
        {\Large \strut} -1 & & $\2Z^2$ & & & \\ 
        \hline 
        {\Large \strut} -2 & $\2Z$ & & & & \\ 
        \hline 
      \end{tabular}}
    & \centering{$\dessin{1.54cm}{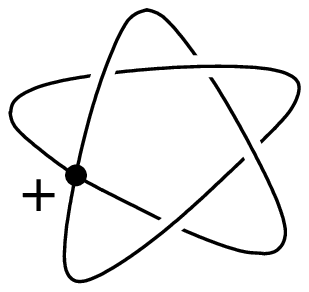}$}
  \end{tabular}
\end{center}

\begin{center}
  \begin{tabular}{p{.3\columnwidth}p{.52\columnwidth}p{.15\columnwidth}}
    \centering{$\dessin{2.8cm}{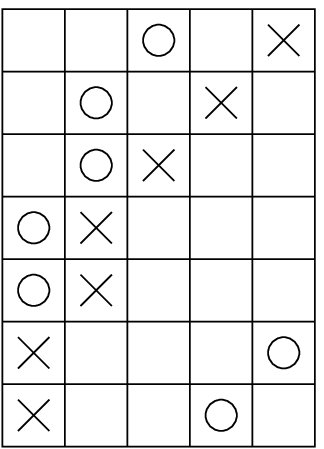}$} & 
    \centering{
      \begin{tabular}{|c|c|c|c|c|c|}
        \cline{2-6}
        \multicolumn{1}{c|}{} & 0 & 1 & 2 & 3 & 4 \\
        \hline
        {\Large \strut} 2 & & & & & $\2Z$ \\ 
        \hline
        {\Large \strut} 1 & & & & $\2Z^3$ & \\ 
        \hline 
        {\Large \strut} 0 & & & $\2Z^4$ & & \\ 
        \hline
        {\Large \strut} -1 & & $\2Z^3$ & & & \\ 
        \hline 
        {\Large \strut} -2 & $\2Z$ & & & & \\ 
        \hline 
      \end{tabular}}
    & \centering{$
      \begin{array}{c}
        \dessin{1.54cm}{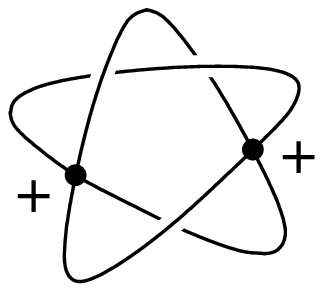} \\ \dessin{1.54cm}{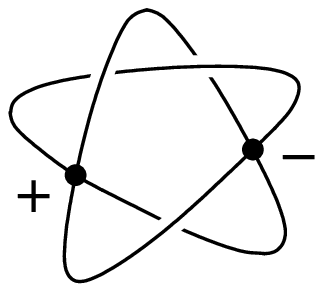}
      \end{array}
      $}
  \end{tabular}
\end{center}

\begin{center}
  \begin{tabular}{p{.3\columnwidth}p{.52\columnwidth}p{.15\columnwidth}}
    \centering{$\dessin{2.8cm}{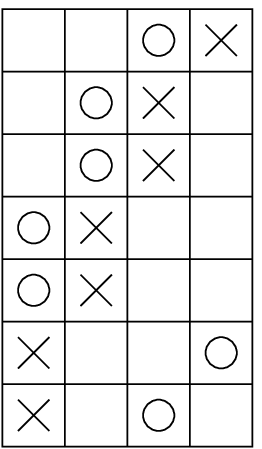}$} & 
    \centering{
      \begin{tabular}{|c|c|c|c|c|c|}
        \cline{2-6}
        \multicolumn{1}{c|}{} & 0 & 1 & 2 & 3 & 4 \\
        \hline
        {\Large \strut} 2 & & & & & $\2Z$ \\ 
        \hline
        {\Large \strut} 1 & & & & $\2Z^4$ & \\ 
        \hline 
        {\Large \strut} 0 & & & $\2Z^6$ & & \\ 
        \hline
        {\Large \strut} -1 & & $\2Z^4$ & & & \\ 
        \hline 
        {\Large \strut} -2 & $\2Z$ & & & & \\ 
        \hline 
      \end{tabular}}
    & \centering{$
      \begin{array}{c}
        \dessin{1.54cm}{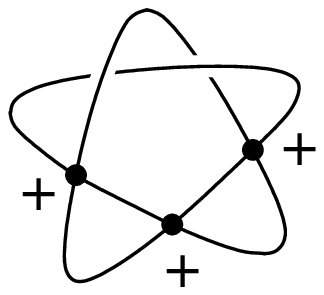} \\ \dessin{1.54cm}{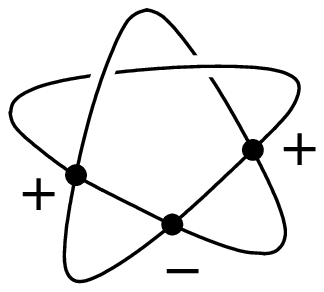}
      \end{array}
      $}
  \end{tabular}
\end{center}

\begin{center}
  \begin{tabular}{p{.3\columnwidth}p{.52\columnwidth}p{.15\columnwidth}}
    \centering{$\dessin{3.2cm}{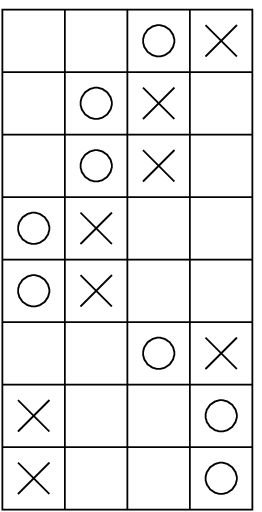}$} & 
    \centering{
      \begin{tabular}{|c|c|c|c|c|c|}
        \cline{2-6}
        \multicolumn{1}{c|}{} & 0 & 1 & 2 & 3 & 4 \\
        \hline
        {\Large \strut} 2 & & & & & $\2Z$ \\ 
        \hline
        {\Large \strut} 1 & & & & $\2Z^4$ & \\ 
        \hline 
        {\Large \strut} 0 & & & $\2Z^6$ & & \\ 
        \hline
        {\Large \strut} -1 & & $\2Z^4$ & & & \\ 
        \hline 
        {\Large \strut} -2 & $\2Z$ & & & & \\ 
        \hline 
      \end{tabular}}
    & \centering{$\dessin{1.54cm}{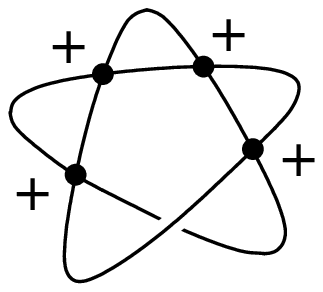}$}
  \end{tabular}
\end{center}

\begin{center}
  \begin{tabular}{p{.3\columnwidth}p{.52\columnwidth}p{.15\columnwidth}}
    \centering{$\dessin{3.2cm}{5_1ssssG}$} & 
    \centering{
      \begin{tabular}{|c|c|c|c|c|c|}
        \cline{2-6}
        \multicolumn{1}{c|}{} & 0 & 1 & 2 & 3 & 4 \\
        \hline
        {\Large \strut} 2 & & & & & $\2Z$ \\ 
        \hline
        {\Large \strut} 1 & & & & $\2Z^5$ & $\2Z$ \\ 
        \hline 
        {\Large \strut} 0 & & & $\2Z^8$ & $\2Z^2$ & \\ 
        \hline
        {\Large \strut} -1 & & $\2Z^5$ & $\2Z$ & & \\ 
        \hline 
        {\Large \strut} -2 & $\2Z$ & & & & \\ 
        \hline 
      \end{tabular}}
    & \centering{$\dessin{1.54cm}{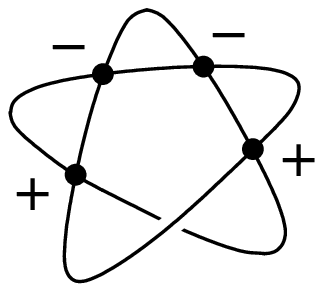}$}
  \end{tabular}
\end{center}

\newpage{}

\section*{Knot $5_2$ \& its singularizations}
\label{sec:5.2}
$$
\fbox{$\dessin{3.41cm}{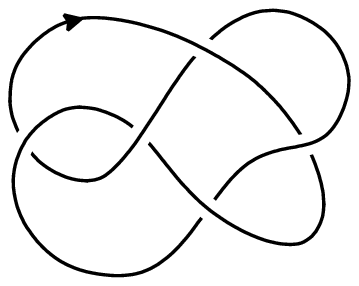}$}
$$

\vspace{.5cm}

\begin{center}
  \begin{tabular}{p{.3\columnwidth}p{.52\columnwidth}p{.15\columnwidth}}
    \centering{$\dessin{2.8cm}{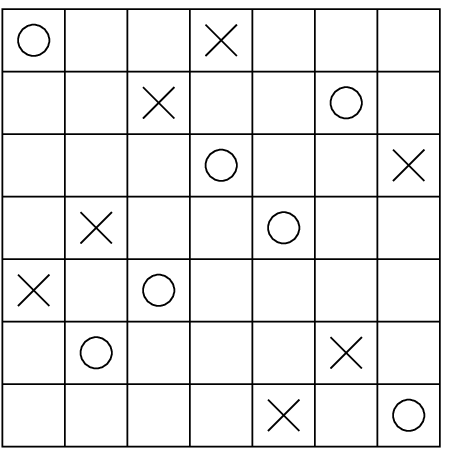}$} &
    \centering{ 
      \begin{tabular}{|c|c|c|c|}
        \cline{2-4}
        \multicolumn{1}{c|}{} & 0 & 1 & 2 \\
        \hline
        {\Large \strut} 1 & & & $\2Z^2$ \\ 
        \hline 
        {\Large \strut} 0 & & $\2Z^3$ & \\ 
        \hline
        {\Large \strut} -1 & $\2Z^2$ &  & \\ 
        \hline 
      \end{tabular}}
    & \centering{$\dessin{1.54cm}{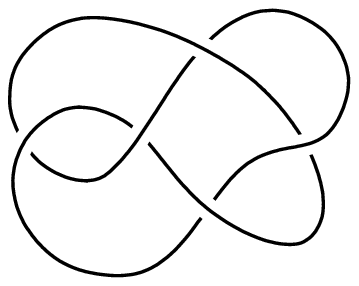}$}
  \end{tabular}
\end{center}

\begin{center}
  \begin{tabular}{p{.3\columnwidth}p{.52\columnwidth}p{.15\columnwidth}}
    \centering{$\dessin{2.8cm}{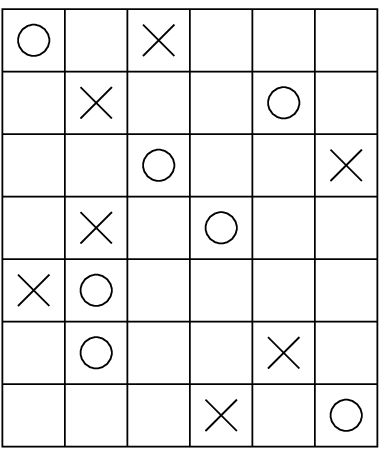}$} & 
    \centering{
      \begin{tabular}{|c|c|c|c|}
        \cline{2-4}
        \multicolumn{1}{c|}{} & 0 & 1 & 2 \\
        \hline
        {\Large \strut} 1 & & & $\2Z^2$ \\ 
        \hline 
        {\Large \strut} 0 & & $\2Z^4$ & \\ 
        \hline
        {\Large \strut} -1 & $\2Z^2$ &  & \\ 
        \hline 
      \end{tabular}}
    & \centering{$\dessin{1.54cm}{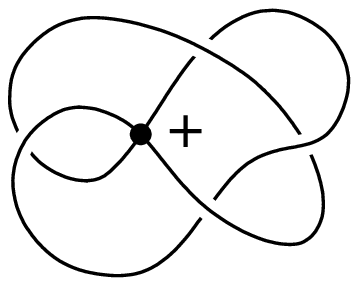}$}
  \end{tabular}
\end{center}

\begin{center}
  \begin{tabular}{p{.3\columnwidth}p{.52\columnwidth}p{.15\columnwidth}}
    \centering{$\dessin{2.8cm}{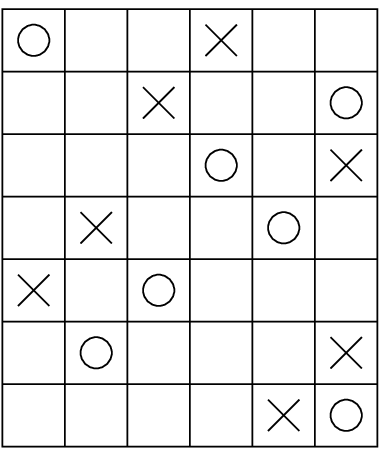}$} & 
    \centering{
      \begin{tabular}{|c|c|c|c|}
        \cline{2-4}
        \multicolumn{1}{c|}{} & 0 & 1 & 2 \\
        \hline
        {\Large \strut} 1 & & & $\2Z$ \\ 
        \hline 
        {\Large \strut} 0 & & $\2Z^2$ & \\ 
        \hline
        {\Large \strut} -1 & $\2Z$ &  & \\ 
        \hline 
      \end{tabular}}
    & \centering{$\dessin{1.54cm}{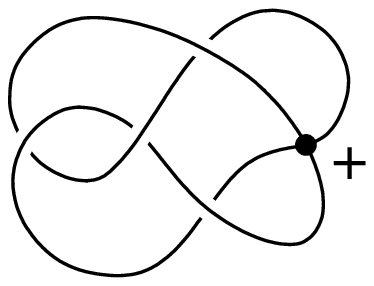}$}
  \end{tabular}
\end{center}

\begin{center}
  \begin{tabular}{p{.3\columnwidth}p{.52\columnwidth}p{.15\columnwidth}}
    \centering{$\dessin{2.8cm}{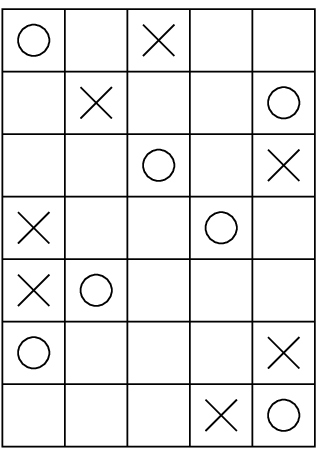}$} & 
    \centering{
      \begin{tabular}{|c|c|c|c|}
        \cline{2-4}
        \multicolumn{1}{c|}{} & 0 & 1 & 2 \\
        \hline
        {\Large \strut} 1 & & & $\2Z$ \\ 
        \hline 
        {\Large \strut} 0 & & $\2Z^2$ & \\ 
        \hline
        {\Large \strut} -1 & $\2Z$ &  & \\ 
        \hline 
      \end{tabular}}
    & \centering{$
      \begin{array}{c}
        \dessin{1.54cm}{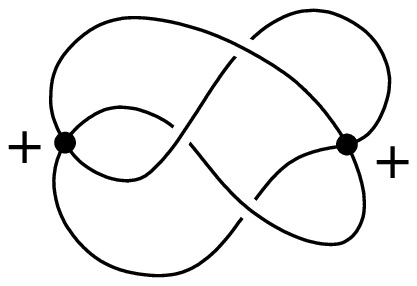} \\ \dessin{1.54cm}{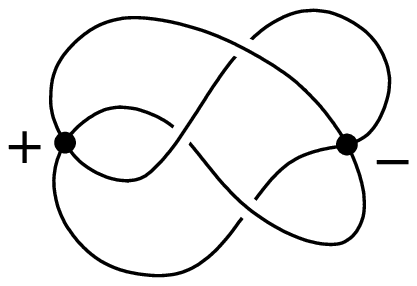}
      \end{array}
      $}
  \end{tabular}
\end{center}

\begin{center}
  \begin{tabular}{p{.3\columnwidth}p{.52\columnwidth}p{.15\columnwidth}}
    \centering{$\dessin{2.8cm}{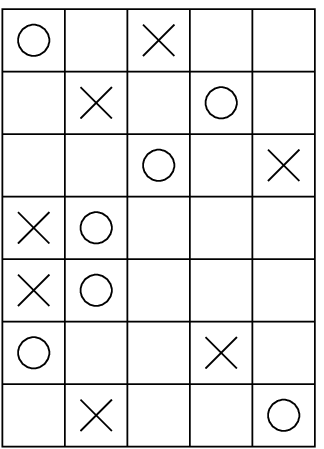}$} & 
    \centering{Acyclic}
    & \centering{$\dessin{1.54cm}{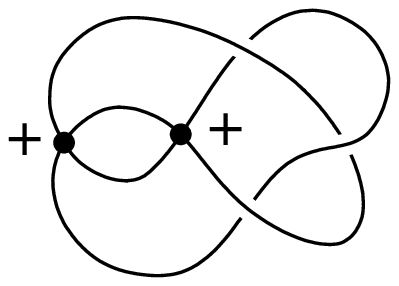}$}
  \end{tabular}
\end{center}

\begin{center}
  \begin{tabular}{p{.3\columnwidth}p{.52\columnwidth}p{.15\columnwidth}}
    \centering{$\dessin{2.8cm}{5_2ss2G}$} & 
    \centering{
      \begin{tabular}{|c|c|c|c|c|}
        \cline{2-5}
        \multicolumn{1}{c|}{} & 0 & 1 & 2 & 3 \\
        \hline
        {\Large \strut} 1 & & & $\2Z^2$ & $\2Z$ \\ 
        \hline 
        {\Large \strut} 0 & & $\2Z^4$ & $\2Z^2$ & \\ 
        \hline
        {\Large \strut} -1 & $\2Z^2$ & $\2Z$ & & \\ 
        \hline 
      \end{tabular}}
    & \centering{$\dessin{1.54cm}{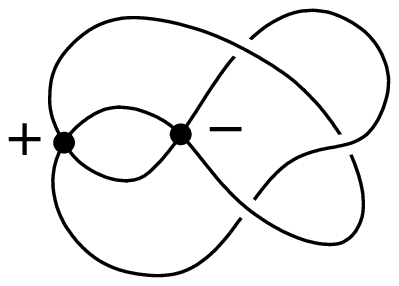}$}
  \end{tabular}
\end{center}

\begin{center}
  \begin{tabular}{p{.3\columnwidth}p{.52\columnwidth}p{.15\columnwidth}}
    \centering{$\dessin{2.8cm}{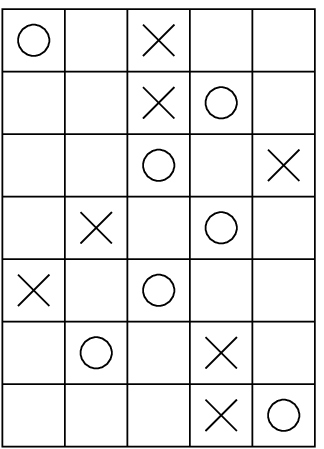}$} & 
    \centering{Acyclic}
    & \centering{$\dessin{1.54cm}{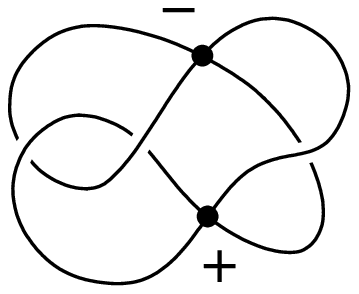}$}
  \end{tabular}
\end{center}

\begin{center}
  \begin{tabular}{p{.3\columnwidth}p{.52\columnwidth}p{.15\columnwidth}}
    \centering{$\dessin{2.8cm}{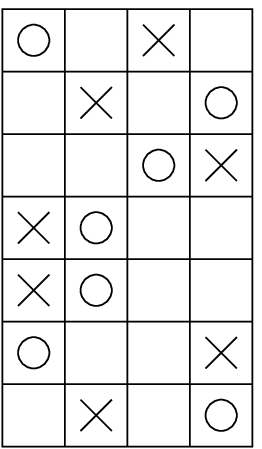}$} & 
    \centering{Acyclic}
    & \centering{$\dessin{1.54cm}{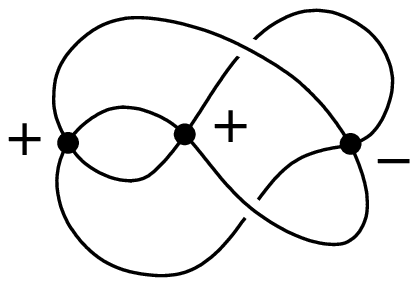}$}
  \end{tabular}
\end{center}

\begin{center}
  \begin{tabular}{p{.3\columnwidth}p{.52\columnwidth}p{.15\columnwidth}}
    \centering{$\dessin{2.8cm}{5_2sss3G}$} & 
    \centering{
      \begin{tabular}{|c|c|c|c|c|}
        \cline{2-5}
        \multicolumn{1}{c|}{} & 0 & 1 & 2 & 3 \\
        \hline
        {\Large \strut} 1 & & & $\2Z$ & $\2Z$ \\ 
        \hline 
        {\Large \strut} 0 & & $\2Z^2$ & $\2Z^2$ & \\ 
        \hline
        {\Large \strut} -1 & $\2Z$ & $\2Z$ & & \\ 
        \hline 
      \end{tabular}}
    & \centering{$
      \begin{array}{c}
         \dessin{1.54cm}{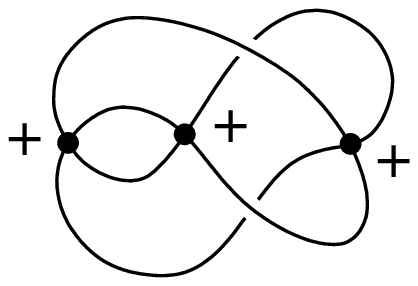} \\ \dessin{1.54cm}{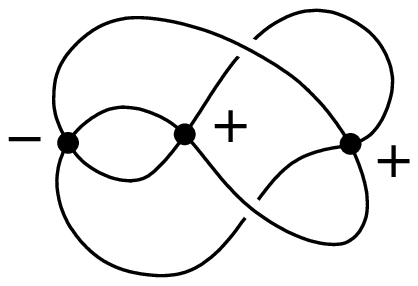} \\ \dessin{1.54cm}{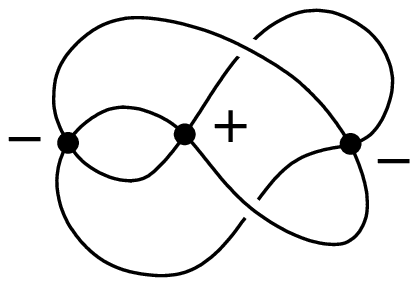}
      \end{array}
      $}
  \end{tabular}
\end{center}

\begin{center}
  \begin{tabular}{p{.3\columnwidth}p{.52\columnwidth}p{.15\columnwidth}}
    \centering{$\dessin{3.2cm}{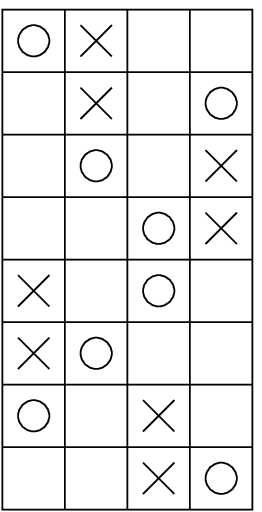}$} & 
    \centering{Acyclic}
    & \centering{$\dessin{1.54cm}{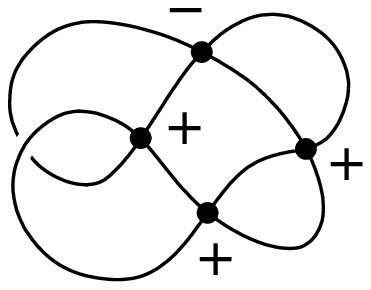}$}
  \end{tabular}
\end{center}

\begin{center}
  \begin{tabular}{p{.3\columnwidth}p{.52\columnwidth}p{.15\columnwidth}}
    \centering{$\dessin{3.2cm}{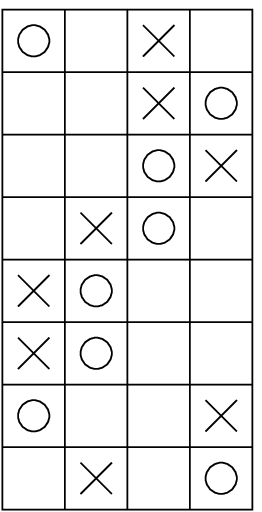}$} & 
    \centering{Acyclic}
    & \centering{$
      \begin{array}{c}
       \dessin{1.54cm}{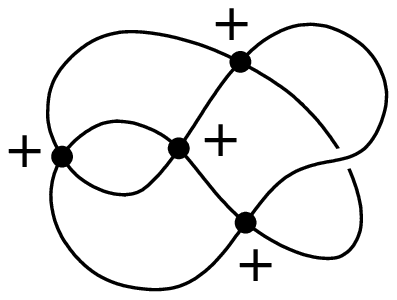} \\ \dessin{1.54cm}{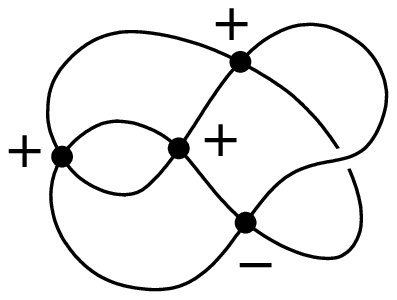} \\ \dessin{1.54cm}{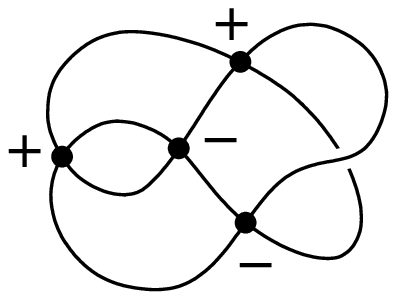}
      \end{array}
      $}
  \end{tabular}
\end{center}

\begin{center}
  \begin{tabular}{p{.3\columnwidth}p{.52\columnwidth}p{.15\columnwidth}}
    \centering{$\dessin{3.2cm}{5_2ssss2G}$} & 
    \centering{
      \begin{tabular}{|c|c|c|c|c|c|}
        \cline{2-6}
        \multicolumn{1}{c|}{} & 0 & 1 & 2 & 3 & 4 \\
        \hline
        {\Large \strut} 1 & & & $\2Z$ & $\2Z^2$ & $\2Z$ \\ 
        \hline 
        {\Large \strut} 0 & $\2Z^2$ & $\2Z^4$ & $\2Z^2$ & & \\ 
        \hline
        {\Large \strut} -1 & $\2Z$ & $\2Z^2$ & $\2Z$ & & \\ 
        \hline 
      \end{tabular}}
    & \centering{$\dessin{1.54cm}{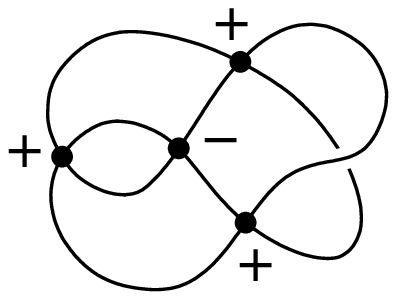}$}
  \end{tabular}
\end{center}

\newpage

\section*{Knot $8_{19}$ \& its singularizations}
\label{sec:8.19}
$$
\fbox{$\dessin{3.72cm}{8_19}$}
$$

\vspace{.5cm}

\begin{center}
  \begin{tabular}{p{.3\columnwidth}p{.67\columnwidth}}
    \centering{\begin{tabular}{c}$\dessin{1.68cm}{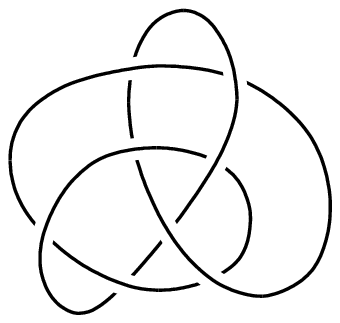}$\\[.9cm]$\dessin{2.8cm}{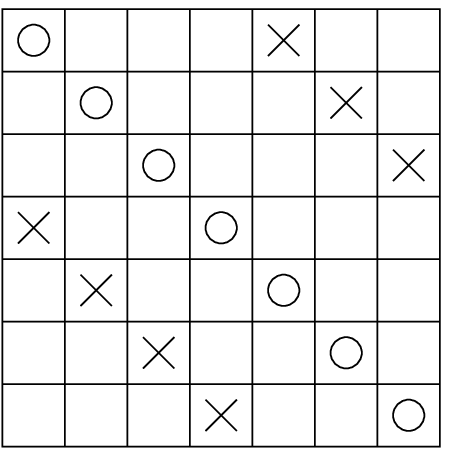}$\end{tabular}} & 
    \centering{
      \begin{tabular}{|c|c|c|c|c|c|c|c|}
        \cline{2-8}
        \multicolumn{1}{c|}{} & -6 & -5 & -4 & -3 & -2 & -1 & 0 \\
        \hline
        {\Large \strut} 3 & & & & & & & $\2Z$ \\ 
        \hline 
        {\Large \strut} 2 & & & & & & $\2Z$ & \\ 
        \hline 
        {\Large \strut} 1 & & & & & & & \\ 
        \hline 
        {\Large \strut} 0 & & & & & $\2Z$ & & \\ 
        \hline 
        {\Large \strut} -1 & & & & & & & \\ 
        \hline 
        {\Large \strut} -2 & & $\2Z$ & & & & & \\ 
        \hline
        {\Large \strut} -3 & $\2Z$ & & & & & & \\ 
        \hline 
      \end{tabular}}
  \end{tabular}
\end{center}

\begin{center}
  \begin{tabular}{p{.3\columnwidth}p{.67\columnwidth}}
    \centering{\begin{tabular}{c}$\dessin{1.68cm}{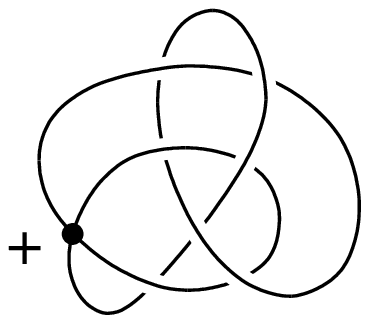}$\\[.9cm]$\dessin{2.8cm}{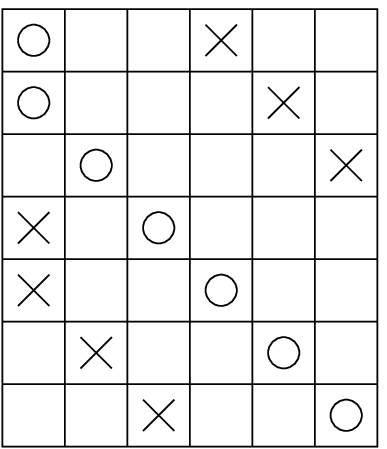}$\end{tabular}} & 
    \centering{
      \begin{tabular}{|c|c|c|c|c|c|c|c|}
        \cline{2-8}
        \multicolumn{1}{c|}{} & -5 & -4 & -3 & -2 & -1 & 0 & 1 \\
        \hline
        {\Large \strut} 3 & & & & & & & $\2Z$ \\ 
        \hline 
        {\Large \strut} 2 & & & & & & $\2Z^2$ & \\ 
        \hline 
        {\Large \strut} 1 & & & & & $\2Z$ & & \\ 
        \hline 
        {\Large \strut} 0 & & & & & & & \\ 
        \hline 
        {\Large \strut} -1 & & & $\2Z$ & & & & \\ 
        \hline 
        {\Large \strut} -2 & & $\2Z^2$ & & & & & \\ 
        \hline
        {\Large \strut} -3 & $\2Z$ & & & & & & \\ 
        \hline 
      \end{tabular}}
  \end{tabular}
\end{center}

\begin{center}
  \begin{tabular}{p{.3\columnwidth}p{.67\columnwidth}}
    \centering{\begin{tabular}{c}$\dessin{1.68cm}{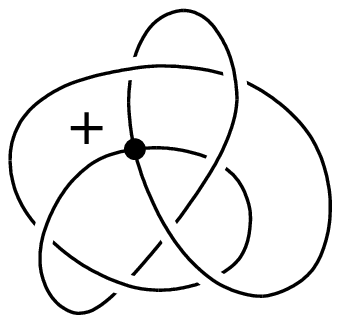}$\\[.9cm]$\dessin{3.2cm}{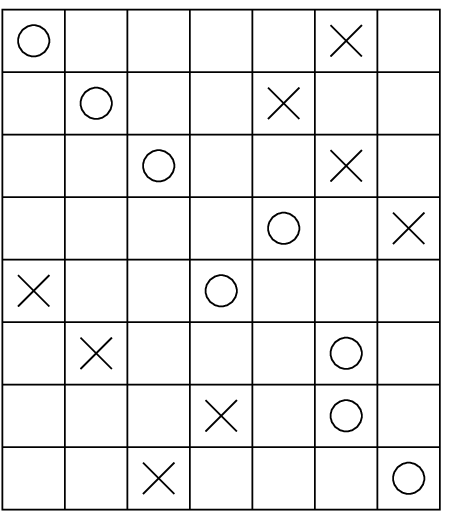}$\end{tabular}} & 
    \centering{
      \begin{tabular}{|c|c|c|c|c|c|c|c|}
        \cline{2-8}
        \multicolumn{1}{c|}{} & -5 & -4 & -3 & -2 & -1 & 0 & 1 \\
        \hline
        {\Large \strut} 3 & & & & & & & $\2Z$ \\ 
        \hline 
        {\Large \strut} 2 & & & & & & $\2Z^2$ & \\ 
        \hline 
        {\Large \strut} 1 & & & & & $\2Z$ & & \\ 
        \hline 
        {\Large \strut} 0 & & & & & & & \\ 
        \hline 
        {\Large \strut} -1 & & & $\2Z$ & & & & \\ 
        \hline 
        {\Large \strut} -2 & & $\2Z^2$ & & & & & \\ 
        \hline
        {\Large \strut} -3 & $\2Z$ & & & & & & \\ 
        \hline 
      \end{tabular}}
  \end{tabular}
\end{center}

\begin{center}
  \begin{tabular}{p{.3\columnwidth}p{.67\columnwidth}}
    \centering{\begin{tabular}{c}$\dessin{1.68cm}{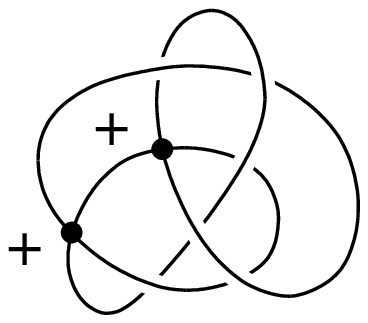}$\\[.9cm]$\dessin{3.2cm}{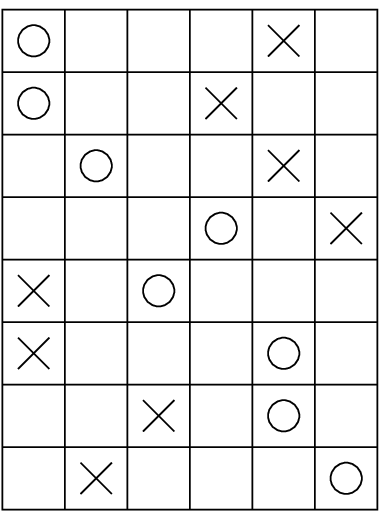}$\end{tabular}} & 
    \centering{
      \begin{tabular}{|c|c|c|c|c|c|c|c|}
        \cline{2-8}
        \multicolumn{1}{c|}{} & -4 & -3 & -2 & -1 & 0 & 1 & 2 \\
        \hline
        {\Large \strut} 3 & & & & & & & $\2Z$ \\ 
        \hline 
        {\Large \strut} 2 & & & & & & $\2Z^3$ & \\ 
        \hline 
        {\Large \strut} 1 & & & & & $\2Z^4$ & & \\ 
        \hline 
        {\Large \strut} 0 & & & & $\2Z^4$ & & & \\ 
        \hline 
        {\Large \strut} -1 & & & $\2Z^4$ & & & & \\ 
        \hline 
        {\Large \strut} -2 & & $\2Z^3$ & & & & & \\ 
        \hline
        {\Large \strut} -3 & $\2Z$ & & & & & & \\ 
        \hline 
      \end{tabular}}
  \end{tabular}
\end{center}

\begin{center}
  \begin{tabular}{p{.3\columnwidth}p{.67\columnwidth}}
    \centering{\begin{tabular}{c}$\dessin{1.68cm}{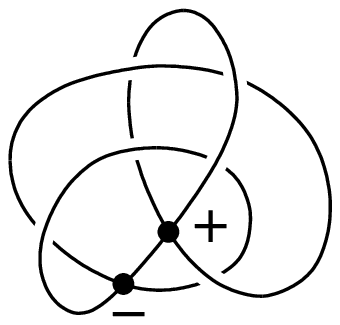}$\\[.9cm]$\dessin{2.8cm}{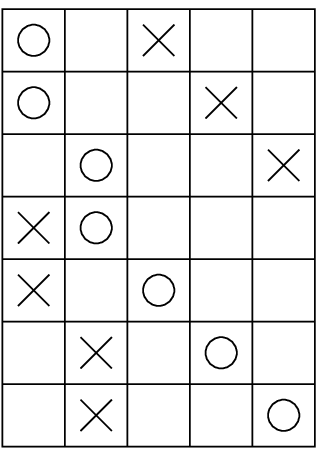}$\end{tabular}} & 
    \centering{
      \begin{tabular}{|c|c|c|c|c|c|c|c|}
        \cline{2-8}
        \multicolumn{1}{c|}{} & -4 & -3 & -2 & -1 & 0 & 1 & 2 \\
        \hline
        {\Large \strut} 3 & & & & & & & $\2Z$ \\ 
        \hline 
        {\Large \strut} 2 & & & & & & $\2Z^3$ & \\ 
        \hline 
        {\Large \strut} 1 & & & & & $\2Z^3$ & & \\ 
        \hline 
        {\Large \strut} 0 & & & & $\2Z^2$ & & & \\ 
        \hline 
        {\Large \strut} -1 & & & $\2Z^3$ & & & & \\ 
        \hline 
        {\Large \strut} -2 & & $\2Z^3$ & & & & & \\ 
        \hline
        {\Large \strut} -3 & $\2Z$ & & & & & & \\ 
        \hline 
      \end{tabular}}
  \end{tabular}
\end{center}

\begin{center}
  \begin{tabular}{p{.3\columnwidth}p{.67\columnwidth}}
    \centering{\begin{tabular}{c}$\dessin{1.68cm}{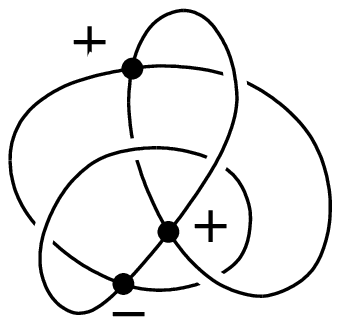}$\\[.9cm]$\dessin{2.8cm}{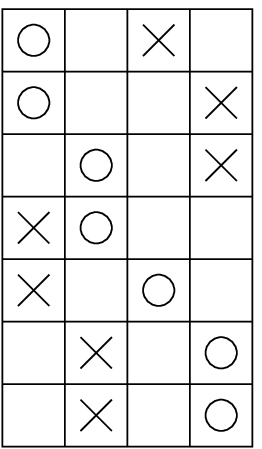}$\end{tabular}} & 
    \centering{
      \begin{tabular}{|c|c|c|c|c|c|c|c|}
        \cline{2-8}
        \multicolumn{1}{c|}{} & -3 & -2 & -1 & 0 & 1 & 2 & 3 \\
        \hline
        {\Large \strut} 3 & & & & & & & $\2Z$ \\ 
        \hline 
        {\Large \strut} 2 & & & & & & $\2Z^4$ & \\ 
        \hline 
        {\Large \strut} 1 & & & & & $\2Z^7$ & & \\ 
        \hline 
        {\Large \strut} 0 & & & & $\2Z^8$ & & & \\ 
        \hline 
        {\Large \strut} -1 & & & $\2Z^7$ & & & & \\ 
        \hline 
        {\Large \strut} -2 & & $\2Z^4$ & & & & & \\ 
        \hline
        {\Large \strut} -3 & $\2Z$ & & & & & & \\ 
        \hline 
      \end{tabular}}
  \end{tabular}
\end{center}

\begin{center}
  \begin{tabular}{p{.3\columnwidth}p{.67\columnwidth}}
    \centering{\begin{tabular}{c}$\dessin{1.68cm}{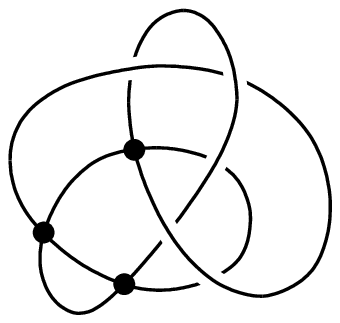}$\\[.9cm]$\dessin{3.2cm}{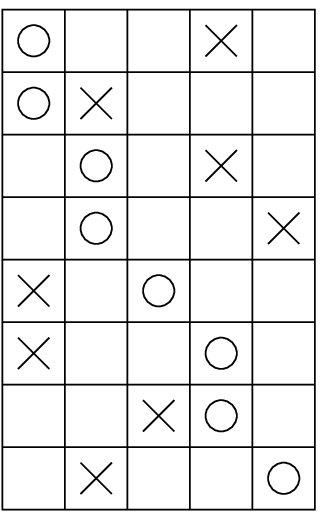}$\end{tabular}} & 
    \centering{
      \begin{tabular}{|c|c|c|c|c|c|c|c|}
        \cline{2-8}
        \multicolumn{1}{c|}{} & -3 & -2 & -1 & 0 & 1 & 2 & 3 \\
        \hline
        {\Large \strut} 3 & & & & & & & $\2Z$ \\ 
        \hline 
        {\Large \strut} 2 & & & & & & $\2Z^4$ & \\ 
        \hline 
        {\Large \strut} 1 & & & & & $\2Z^7$ & & \\ 
        \hline 
        {\Large \strut} 0 & & & & $\2Z^8$ & & & \\ 
        \hline 
        {\Large \strut} -1 & & & $\2Z^7$ & & & & \\ 
        \hline 
        {\Large \strut} -2 & & $\2Z^4$ & & & & & \\ 
        \hline
        {\Large \strut} -3 & $\2Z$ & & & & & & \\ 
        \hline 
      \end{tabular}}
  \end{tabular}
\end{center}

\begin{center}
  \begin{tabular}{p{.3\columnwidth}p{.67\columnwidth}}
    \centering{\begin{tabular}{c}$\dessin{1.68cm}{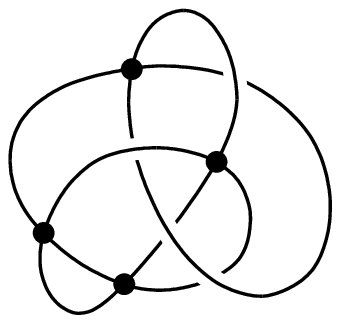}$\\[.9cm]$\dessin{3.2cm}{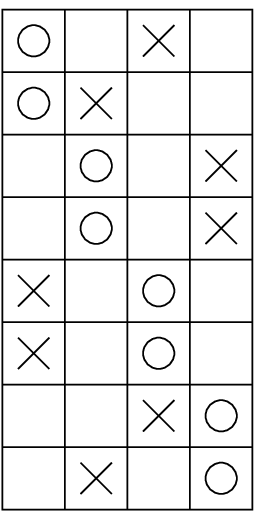}$\end{tabular}} & 
    \centering{
      \begin{tabular}{|c|c|c|c|c|c|c|c|}
        \cline{2-8}
        \multicolumn{1}{c|}{} & -2 & -1 & 0 & 1 & 2 & 3 & 4 \\
        \hline
        {\Large \strut} 3 & & & & & & & $\2Z$ \\ 
        \hline 
        {\Large \strut} 2 & & & & & & $\2Z^5$ & \\ 
        \hline 
        {\Large \strut} 1 & & & & $\2Z$ & $\2Z^{12}$ & & \\ 
        \hline 
        {\Large \strut} 0 & & & $\2Z^2$ & $\2Z^{16}$ & & & \\ 
        \hline 
        {\Large \strut} -1 & & $\2Z$ & $\2Z^{12}$ & & & & \\ 
        \hline 
        {\Large \strut} -2 & & $\2Z^5$ & & & & & \\ 
        \hline
        {\Large \strut} -3 & $\2Z$ & & & & & & \\ 
        \hline 
      \end{tabular}}
  \end{tabular}
\end{center}

\begin{center}
  \begin{tabular}{p{.3\columnwidth}p{.67\columnwidth}}
    \centering{\begin{tabular}{c}$\dessin{1.68cm}{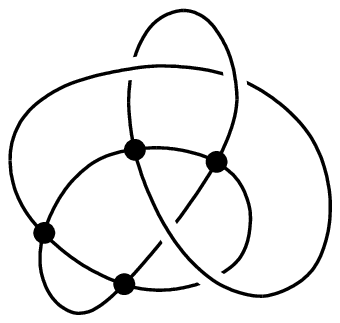}$\\[.9cm]$\dessin{3.2cm}{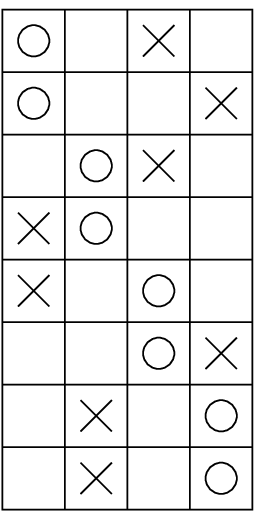}$\end{tabular}} & 
    \centering{
      \begin{tabular}{|c|c|c|c|c|c|c|c|}
        \cline{2-8}
        \multicolumn{1}{c|}{} & -2 & -1 & 0 & 1 & 2 & 3 & 4 \\
        \hline
        {\Large \strut} 3 & & & & & & & $\2Z$ \\ 
        \hline 
        {\Large \strut} 2 & & & & & & $\2Z^5$ & \\ 
        \hline 
        {\Large \strut} 1 & & & & & $\2Z^{11}$ & & \\ 
        \hline 
        {\Large \strut} 0 & & & & $\2Z^{14}$ & & & \\ 
        \hline 
        {\Large \strut} -1 & & & $\2Z^{11}$ & & & & \\ 
        \hline 
        {\Large \strut} -2 & & $\2Z^5$ & & & & & \\ 
        \hline
        {\Large \strut} -3 & $\2Z$ & & & & & & \\ 
        \hline 
      \end{tabular}}
  \end{tabular}
\end{center}

\newpage

\section*{Knot $9_{44}$ \& its singularizations}
\label{sec:9.44}
$$
\fbox{$\dessin{3.72cm}{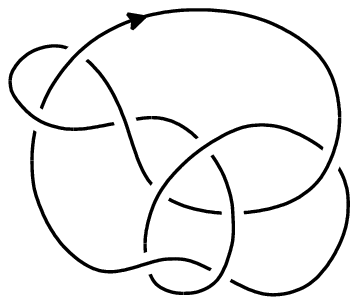}$}
$$

\vspace{.5cm}

\begin{center}
  \begin{tabular}{p{.3\columnwidth}p{.52\columnwidth}p{.15\columnwidth}}
    \centering{$\dessin{3cm}{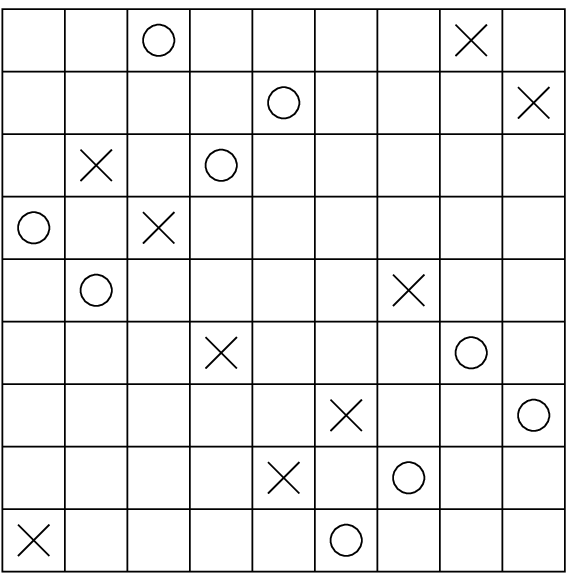}$} & 
    \centering{
      \begin{tabular}{|c|c|c|c|c|c|}
        \cline{2-6}
        \multicolumn{1}{c|}{} & -2 & -1 & 0 & 1 & 2 \\
        \hline
        {\Large \strut} 2 & & & & & $\2Z$ \\ 
        \hline 
        {\Large \strut} 1 & & & & $\2Z^4$ & \\ 
        \hline 
        {\Large \strut} 0 & & & $\2Z^7$ & & \\ 
        \hline 
        {\Large \strut} -1 & & $\2Z^4$ & & & \\ 
        \hline 
        {\Large \strut} -2 & $\2Z$ & & & & \\ 
        \hline 
      \end{tabular}}
    & \centering{$\dessin{1.68cm}{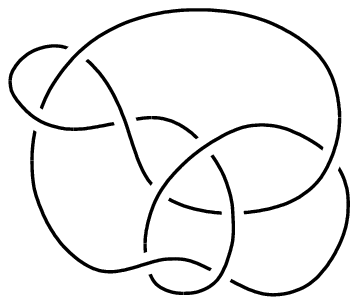}$}
  \end{tabular}
\end{center}

\begin{center}
  \begin{tabular}{p{.3\columnwidth}p{.52\columnwidth}p{.15\columnwidth}}
    \centering{$\dessin{3cm}{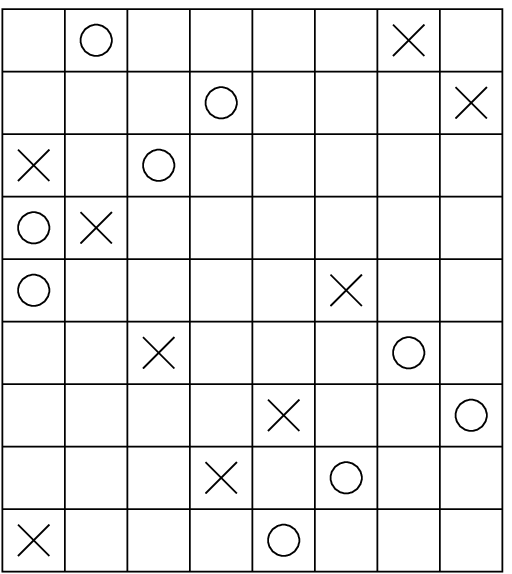}$} & 
    \centering{
      \begin{tabular}{|c|c|c|c|c|c|}
        \cline{2-6}
        \multicolumn{1}{c|}{} & -1 & 0 & 1 & 2 & 3 \\
        \hline
        {\Large \strut} 2 & & & & & $\2Z$ \\ 
        \hline 
        {\Large \strut} 1 & & & & $\2Z^4$ & \\ 
        \hline 
        {\Large \strut} 0 & & & $\2Z^6$ & & \\ 
        \hline 
        {\Large \strut} -1 & & $\2Z^4$ & & & \\ 
        \hline 
        {\Large \strut} -2 & $\2Z$ & & & & \\ 
        \hline 
      \end{tabular}}
    & \centering{$\dessin{1.68cm}{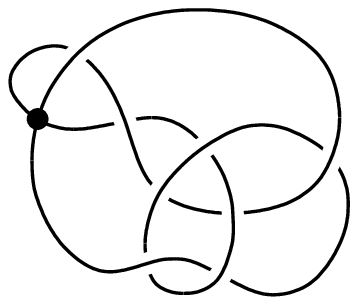}$}
  \end{tabular}
\end{center}

\begin{center}
  \begin{tabular}{p{.3\columnwidth}p{.52\columnwidth}p{.15\columnwidth}}
    \centering{$\dessin{3cm}{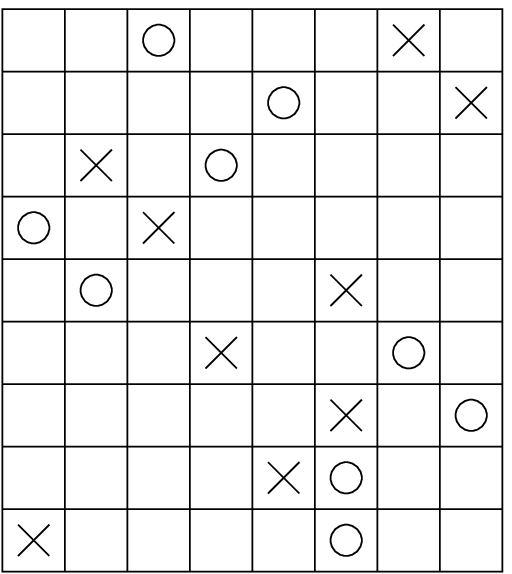}$} & 
    \centering{
      \begin{tabular}{|c|c|c|c|c|c|}
        \cline{2-6}
        \multicolumn{1}{c|}{} & -2 & -1 & 0 & 1 & 2 \\
        \hline
        {\Large \strut} 2 & & & & & $\2Z$ \\ 
        \hline 
        {\Large \strut} 1 & & & & $\2Z^4$ & \\ 
        \hline 
        {\Large \strut} 0 & & & $\2Z^6$ & & \\ 
        \hline 
        {\Large \strut} -1 & & $\2Z^4$ & & & \\ 
        \hline 
        {\Large \strut} -2 & $\2Z$ & & & & \\ 
        \hline 
      \end{tabular}}
    & \centering{$\dessin{1.68cm}{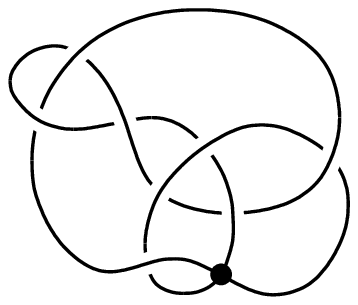}$}
  \end{tabular}
\end{center}

\begin{center}
  \begin{tabular}{p{.3\columnwidth}p{.52\columnwidth}p{.15\columnwidth}}
    \centering{$\dessin{3cm}{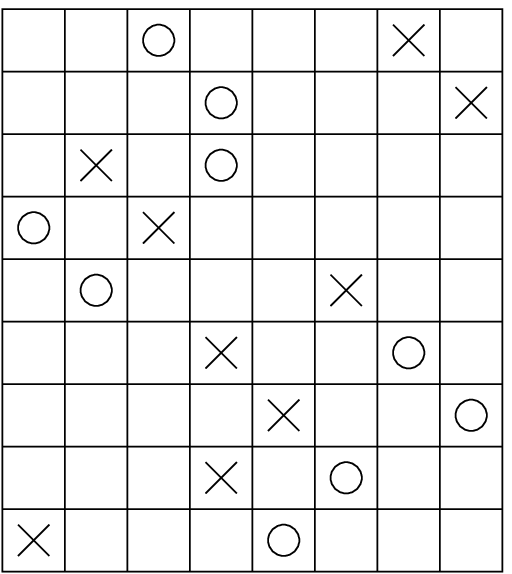}$} & 
    \centering{
      \begin{tabular}{|c|c|c|c|c|c|}
        \cline{2-6}
        \multicolumn{1}{c|}{} & -1 & 0 & 1 & 2 & 3 \\
        \hline
        {\Large \strut} 2 & & & & & $\2Z^2$ \\ 
        \hline 
        {\Large \strut} 1 & & & & $\2Z^9$ & \\ 
        \hline 
        {\Large \strut} 0 & & & $\2Z^{14}$ & & \\ 
        \hline 
        {\Large \strut} -1 & & $\2Z^9$ & & & \\ 
        \hline 
        {\Large \strut} -2 & $\2Z^2$ & & & & \\ 
        \hline 
      \end{tabular}}
    & \centering{$\dessin{1.68cm}{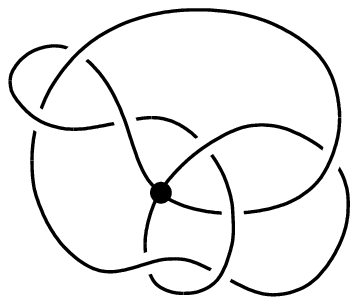}$}
  \end{tabular}
\end{center}

\begin{center}
  \begin{tabular}{p{.3\columnwidth}p{.52\columnwidth}p{.15\columnwidth}}
    \centering{$\dessin{3cm}{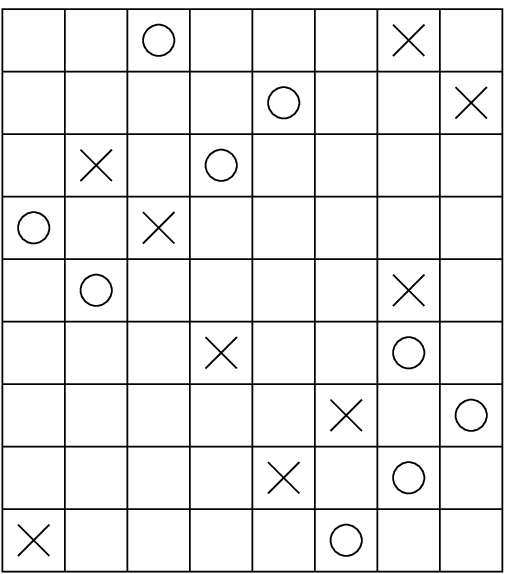}$} & 
    \centering{
      \begin{tabular}{|c|c|c|c|c|}
        \cline{2-5}
        \multicolumn{1}{c|}{} & -2 & -1 & 0 & 1 \\
        \hline
        {\Large \strut} 1 & & & & $\2Z$ \\ 
        \hline 
        {\Large \strut} 0 & & & $\2Z^3$ & $\2Z$ \\ 
        \hline 
        {\Large \strut} -1 & & $\2Z^3$ & $\2Z^2$ & \\ 
        \hline 
        {\Large \strut} -2 & $\2Z$ & $\2Z$ & & \\ 
        \hline 
      \end{tabular}}
    & \centering{$\dessin{1.68cm}{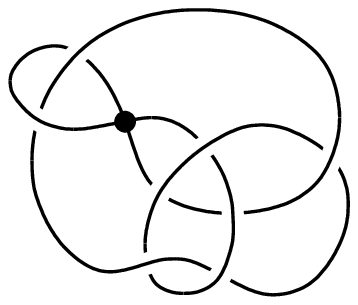}$}
  \end{tabular}
\end{center}

\begin{center}
  \begin{tabular}{p{.3\columnwidth}p{.52\columnwidth}p{.15\columnwidth}}
    \centering{$\dessin{3cm}{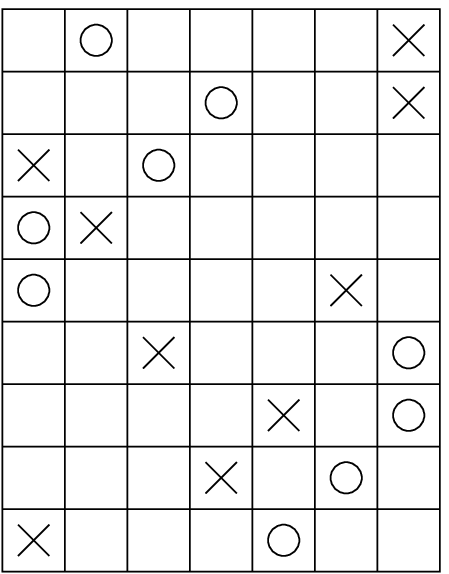}$} & 
    \centering{
      \begin{tabular}{|c|c|c|c|c|c|}
        \cline{2-6}
        \multicolumn{1}{c|}{} & 0 & 1 & 2 & 3 & 4 \\
        \hline
        {\Large \strut} 2 & & & & & $\2Z^2$ \\ 
        \hline 
        {\Large \strut} 1 & & & & $\2Z^7$ & \\ 
        \hline 
        {\Large \strut} 0 & & & $\2Z^{10}$ & & \\ 
        \hline 
        {\Large \strut} -1 & & $\2Z^7$ & & & \\ 
        \hline 
        {\Large \strut} -2 & $\2Z^2$ & & & & \\ 
        \hline 
      \end{tabular}}
    & \centering{$\dessin{1.68cm}{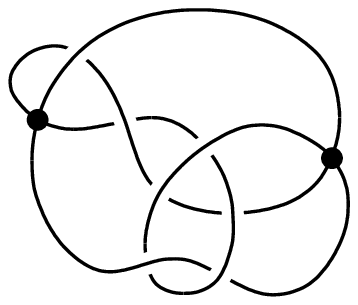}$}
  \end{tabular}
\end{center}

\begin{center}
  \begin{tabular}{p{.3\columnwidth}p{.52\columnwidth}p{.15\columnwidth}}
    \centering{$\dessin{3cm}{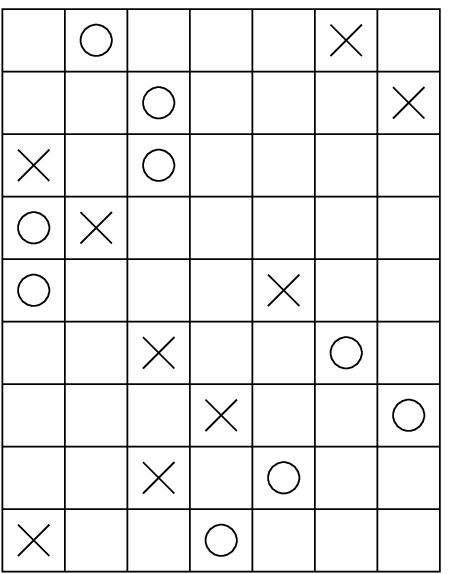}$} & 
    \centering{
      \begin{tabular}{|c|c|c|c|c|c|}
        \cline{2-6}
        \multicolumn{1}{c|}{} & 0 & 1 & 2 & 3 & 4 \\
        \hline
        {\Large \strut} 2 & & & & & $\2Z^2$ \\ 
        \hline 
        {\Large \strut} 1 & & & & $\2Z^7$ & \\ 
        \hline 
        {\Large \strut} 0 & & & $\2Z^{10}$ & & \\ 
        \hline 
        {\Large \strut} -1 & & $\2Z^7$ & & & \\ 
        \hline 
        {\Large \strut} -2 & $\2Z^2$ & & & & \\ 
        \hline 
      \end{tabular}}
    & \centering{$\dessin{1.68cm}{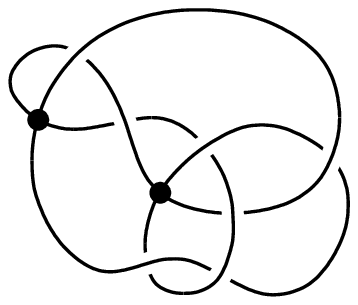}$}
  \end{tabular}
\end{center}

\begin{center}
  \begin{tabular}{p{.3\columnwidth}p{.52\columnwidth}p{.15\columnwidth}}
    \centering{$\dessin{3cm}{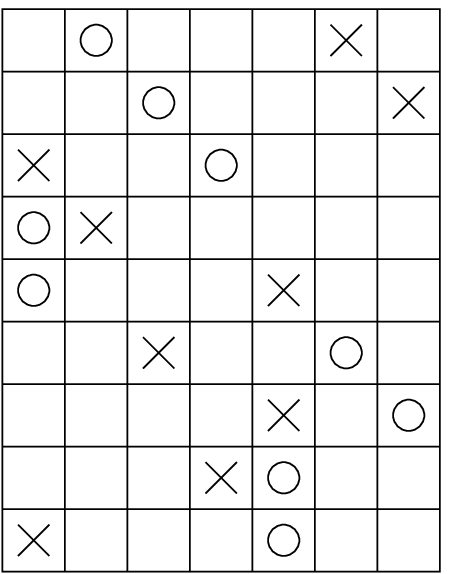}$} & 
    \centering{
      \begin{tabular}{|c|c|c|c|c|c|}
        \cline{2-6}
        \multicolumn{1}{c|}{} & -1 & 0 & 1 & 2 & 3 \\
        \hline
        {\Large \strut} 2 & & & & & $\2Z$ \\ 
        \hline 
        {\Large \strut} 1 & & & & $\2Z^4$ & \\ 
        \hline 
        {\Large \strut} 0 & & & $\2Z^6$ & & \\ 
        \hline 
        {\Large \strut} -1 & & $\2Z^4$ & & & \\ 
        \hline 
        {\Large \strut} -2 & $\2Z$ & & & & \\ 
        \hline 
      \end{tabular}}
    & \centering{$\dessin{1.68cm}{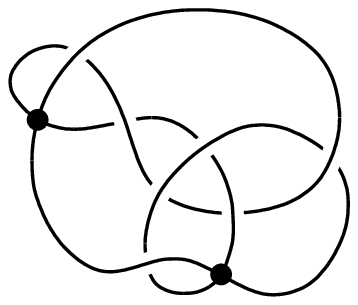}$}
  \end{tabular}
\end{center}

\begin{center}
  \begin{tabular}{p{.3\columnwidth}p{.52\columnwidth}p{.15\columnwidth}}
    \centering{$\dessin{3cm}{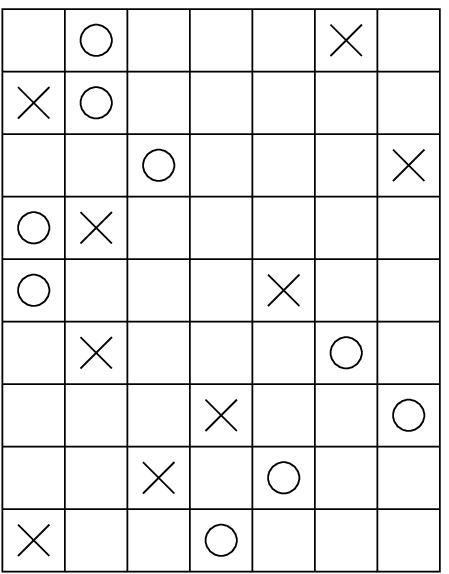}$} & 
    \centering{
      \begin{tabular}{|c|c|c|c|c|c|c|}
        \cline{2-7}
        \multicolumn{1}{c|}{} & -1 & 0 & 1 & 2 & 3 & 4 \\
        \hline
        {\Large \strut} 2 & & & & & $\2Z$ & $\2Z$ \\ 
        \hline 
        {\Large \strut} 1 & & & & $\2Z^4$ & $\2Z^4$ & \\ 
        \hline 
        {\Large \strut} 0 & & & $\2Z^6$ & $\2Z^6$ & & \\ 
        \hline 
        {\Large \strut} -1 & & $\2Z^4$ & $\2Z^4$ & & & \\ 
        \hline 
        {\Large \strut} -2 & $\2Z$ & $\2Z$ & & & & \\ 
        \hline 
      \end{tabular}}
    & \centering{$\dessin{1.68cm}{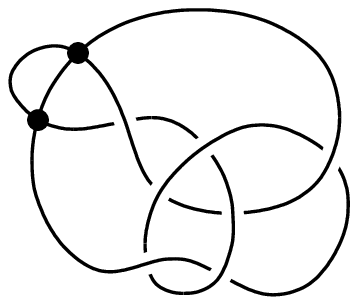}$}
  \end{tabular}
\end{center}

\begin{center}
  \begin{tabular}{p{.3\columnwidth}p{.52\columnwidth}p{.15\columnwidth}}
    \centering{$\dessin{3cm}{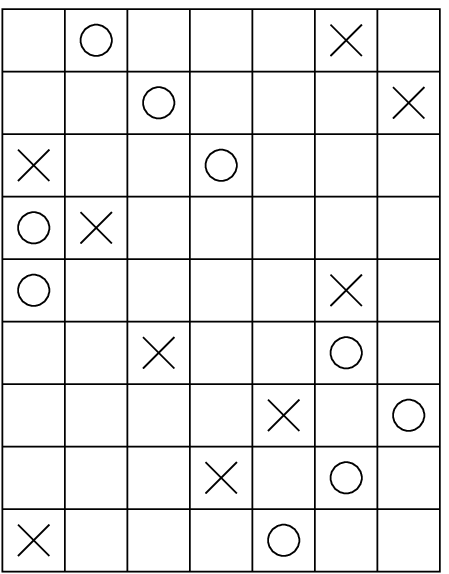}$} & 
    \centering{
      \begin{tabular}{|c|c|c|c|c|}
        \cline{2-5}
        \multicolumn{1}{c|}{} & -1 & 0 & 1 & 2 \\
        \hline
        {\Large \strut} 1 & & & & $\2Z$ \\ 
        \hline 
        {\Large \strut} 0 & & & $\2Z^3$ & $\2Z$ \\ 
        \hline 
        {\Large \strut} -1 & & $\2Z^3$ & $\2Z^2$ & \\ 
        \hline 
        {\Large \strut} -2 & $\2Z$ & $\2Z$ & & \\ 
        \hline 
      \end{tabular}}
    & \centering{$\dessin{1.68cm}{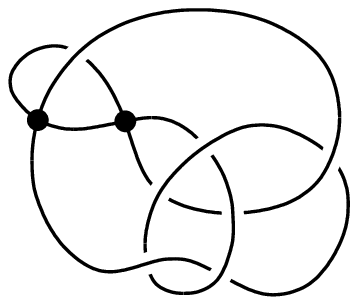}$}
  \end{tabular}
\end{center}

\newpage

\section*{Knot $10_{124}$ \& its singularizations}
\label{sec:10.124}
$$
\fbox{$\dessin{3.72cm}{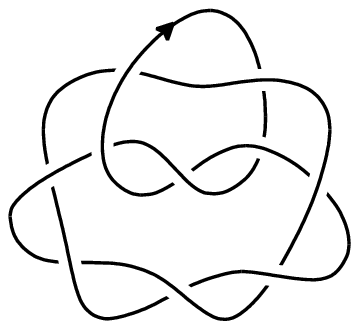}$}
$$

\vspace{.5cm}

\begin{center}
  \begin{tabular}{p{.3\columnwidth}p{.52\columnwidth}p{.15\columnwidth}}
    \centering{$\dessin{3cm}{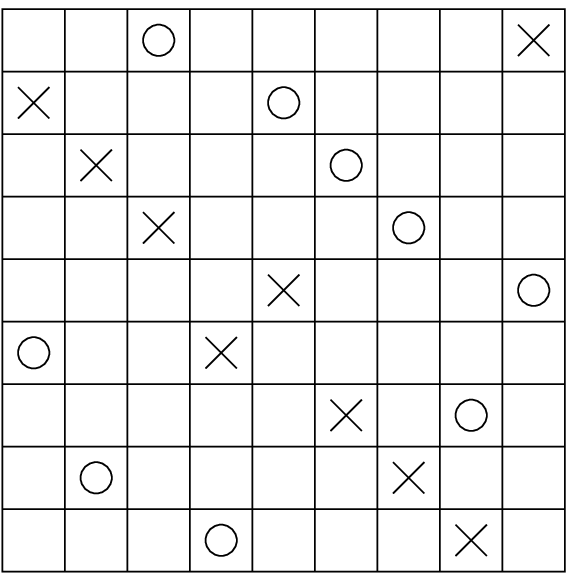}$} & 
    \centering{$t^{-8}q^{-4}+t^{-7}q^{-3}+t^{-4}q^{-1}+t^{-3}+t^{-2}q +t^{-1}q^3+q^4$}&
    \centering{$\dessin{1.68cm}{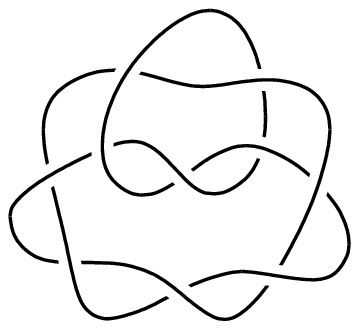}$}
  \end{tabular}
\end{center}

\vspace{0cm}

\begin{center}
  \begin{tabular}{p{.3\columnwidth}p{.52\columnwidth}p{.15\columnwidth}}
    \centering{$\dessin{3cm}{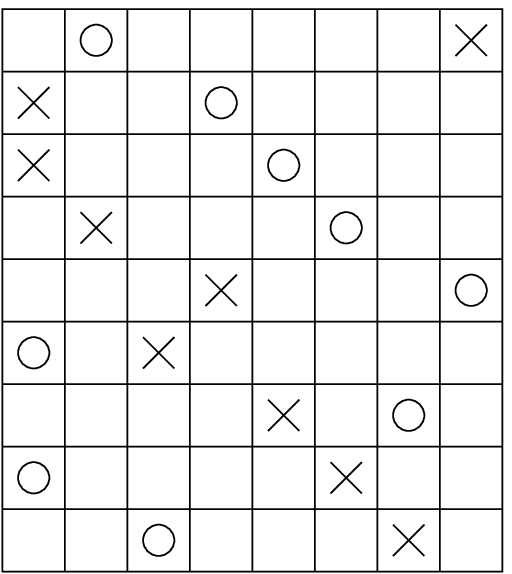}$} & 
    \centering{$(1+tq)^2(t^{-7}q^{-4}+t^{-3}q^{-1}+t^{-1}q^2)$}&
    \centering{$\dessin{1.68cm}{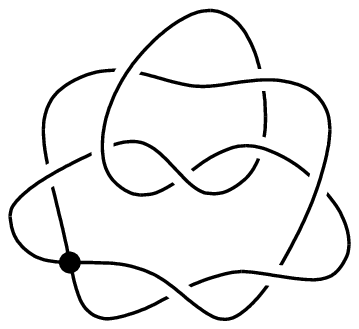}$}
  \end{tabular}
\end{center}

\vspace{0cm}

\begin{center}
  \begin{tabular}{p{.3\columnwidth}p{.52\columnwidth}p{.15\columnwidth}}
    \centering{$\dessin{3cm}{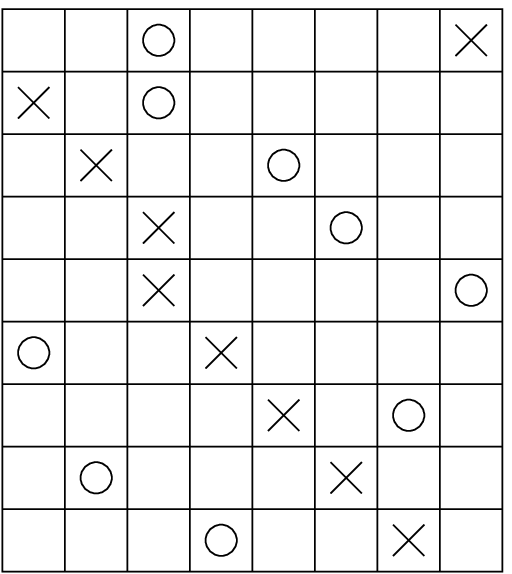}$} & 
    \centering{$(1+tq)^2(t^{-7}q^{-4}+t^{-1}q^2)$}&
    \centering{$\dessin{1.68cm}{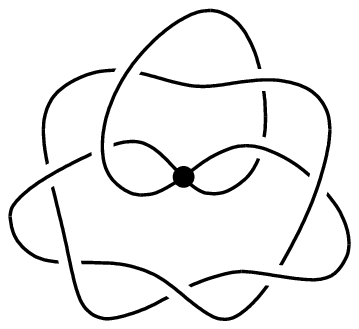}$}
  \end{tabular}
\end{center}

\vspace{0cm}

\begin{center}
  \begin{tabular}{p{.3\columnwidth}p{.52\columnwidth}p{.15\columnwidth}}
    \centering{$\dessin{3cm}{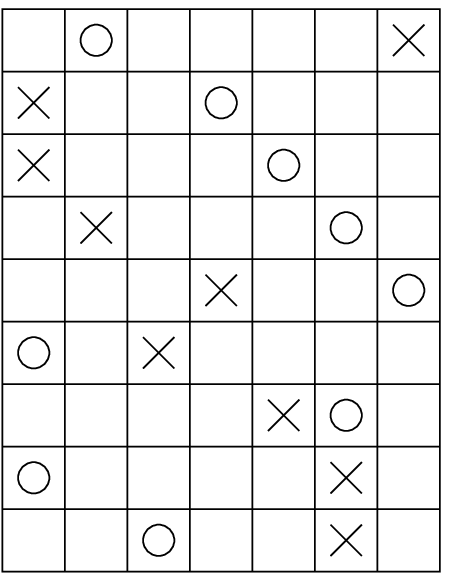}$} & 
    \centering{$(1+tq)^3(t^{-6}q^{-4}+t^{-1}q)+t^{-2}q^{-1}(1+tq)^2$}&
    \centering{$\dessin{1.68cm}{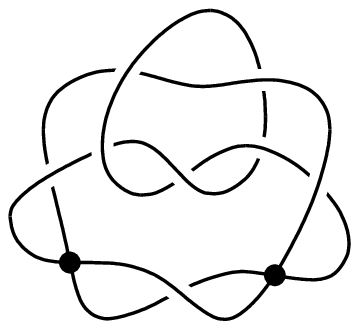}$}
  \end{tabular}
\end{center}

\vspace{0cm}

\begin{center}
  \begin{tabular}{p{.3\columnwidth}p{.52\columnwidth}p{.15\columnwidth}}
    \centering{$\dessin{3cm}{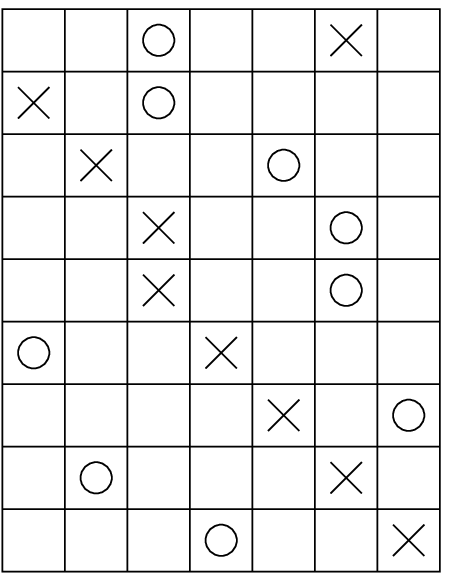}$} & 
    \centering{$(1+tq)^3(t^{-6}q^{-4}+t^{-4}q^{-3}+t^{-1}q)+t^{-2}(1+tq)^2$}&
    \centering{$\dessin{1.68cm}{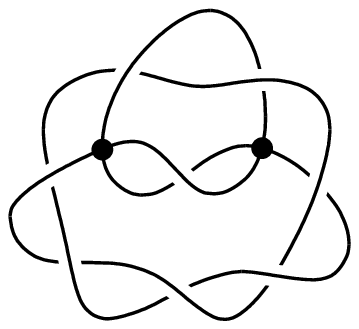}$}
  \end{tabular}
\end{center}

\vspace{0cm}

\begin{center}
  \begin{tabular}{p{.3\columnwidth}p{.52\columnwidth}p{.15\columnwidth}}
    \centering{$\dessin{3cm}{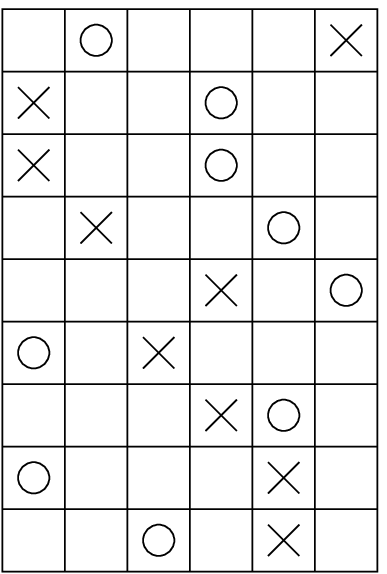}$} & 
    \centering{$(1+tq)^4(t^{-5}q^{-4}+t^{-1})+(1+tq)^2(t^{-2}q^{-1}+t^{-1}q^{-1})$}&
    \centering{$\dessin{1.68cm}{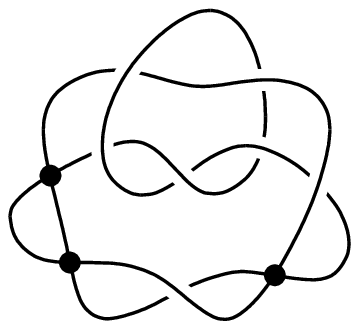}$}
  \end{tabular}
\end{center}

\vspace{0cm}

\begin{center}
  \begin{tabular}{p{.3\columnwidth}p{.52\columnwidth}p{.15\columnwidth}}
    \centering{$\dessin{3cm}{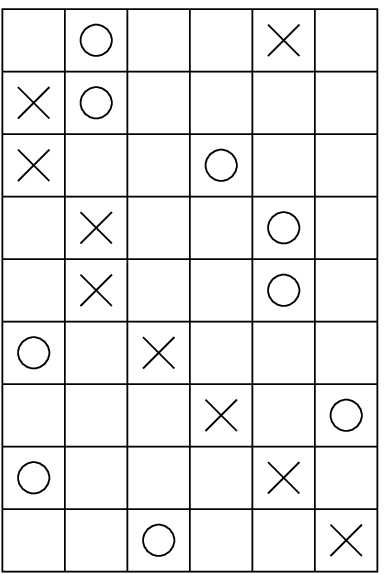}$} & 
    \centering{$(1+tq)^4(t^{-5}q^{-4}t^{-3}q^{-2}+t^{-1})$}&
    \centering{$\dessin{1.68cm}{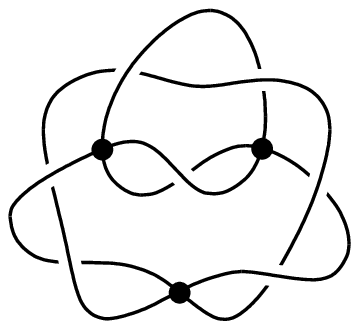}$}
  \end{tabular}
\end{center}

% \newpage

% \section*{Modèle}
% \label{sec:modele}
% $$
% \fbox{noeud $\times$ 0.62}
% $$

% \vspace{.5cm}

% \begin{center}
%   \begin{tabular}{p{.3\columnwidth}p{.52\columnwidth}p{.15\columnwidth}}
%     \centering{grille $\times$ 0.4} & 
%     \centering{
%       \begin{tabular}{|c|c|c|c|c|c|c|c|}
%         \cline{2-8}
%         \multicolumn{1}{c|}{} & -6 & -5 & -4 & -3 & -2 & -1 & 0 \\
%         \hline
%         {\Large \strut} 3 & & & & & & & $\2Z^2$ \\ 
%         \hline 
%         {\Large \strut} 2 & & & & & & & \\ 
%         \hline 
%         {\Large \strut} 1 & & & & & & & \\ 
%         \hline 
%         {\Large \strut} 0 & & & & & & & \\ 
%         \hline 
%         {\Large \strut} -1 & & & & & & & \\ 
%         \hline 
%         {\Large \strut} -2 & & & & & & & \\ 
%         \hline
%         {\Large \strut} -3 & & & & & & & \\ 
%         \hline 
%       \end{tabular}}
%     & \centering{noeud $\times$ 0.28}
%   \end{tabular}
% \end{center}

%%% Local Variables: 
%%% mode: latex
%%% TeX-master: "These"

%%% End: 

% Index
\printindex 
\addcontentsline{toc}{part}{Index}

% Bibliographie
\nocite{*}
\bibliographystyle{amsalpha}
\bibliography{These}

\providecommand{\bysame}{\leavevmode\hbox to3em{\hrulefill}\thinspace}
\providecommand{\MR}{\relax\ifhmode\unskip\space\fi MR }
\providecommand{\MRhref}[2]{%
  \href{http://www.ams.org/mathscinet-getitem?mr=#1}{#2}
}
\providecommand{\href}[2]{#2}
\begin{thebibliography}{MOST06}

\bibitem[AF05]{Moi1}
Benjamin Audoux and Thomas Fiedler, \emph{A {J}ones polynomial for braid-like
  isotopies of oriented links and its categorification}, Algebr. Geom. Topol.
  \textbf{5} (2005), 1535--1553 (electronic). \MR{MR2186108 (2006h:57008)}

\bibitem[APS04]{Asaeda}
Marta~M. Asaeda, J{\'o}zef~H. Przytycki, and Adam~S. Sikora,
  \emph{Categorification of the {K}auffman bracket skein module of
  {$I$}-bundles over surfaces}, Algebr. Geom. Topol. \textbf{4} (2004),
  1177--1210 (electronic). \MR{MR2113902 (2006a:57010)}

\bibitem[Aud06]{Moi2}
Benjamin Audoux, \emph{Khovanov homology and star-like isotopies}, Submitted
  paper, 2006.

\bibitem[Aud07]{Moi3}
\bysame, \emph{Heegaard--floer homology for singular knots},
  \textsf{arXiv:math.GT/0705.2377}, 2007.

\bibitem[Bir74]{Birman2}
Joan~S. Birman, \emph{Braids, links, and mapping class groups}, Princeton
  University Press, Princeton, N.J., 1974, Annals of Mathematics Studies, No.
  82. \MR{MR0375281 (51 \#11477)}

\bibitem[BL93]{Birman}
Joan~S. Birman and Xiao-Song Lin, \emph{Knot polynomials and {V}assiliev's
  invariants}, Invent. Math. \textbf{111} (1993), no.~2, 225--270.
  \MR{MR1198809 (94d:57010)}

\bibitem[BN95]{BarNatanVass}
Dror Bar-Natan, \emph{On the {V}assiliev knot invariants}, Topology \textbf{34}
  (1995), no.~2, 423--472. \MR{MR1318886 (97d:57004)}

\bibitem[BN05]{BarNatanKho}
\bysame, \emph{Khovanov's homology for tangles and cobordisms}, Geom. Topol.
  \textbf{9} (2005), 1443--1499 (electronic). \MR{MR2174270 (2006g:57017)}

\bibitem[BZ03]{Burde}
Gerhard Burde and Heiner Zieschang, \emph{Knots}, second ed., de Gruyter
  Studies in Mathematics, vol.~5, Walter de Gruyter \& Co., Berlin, 2003.
  \MR{MR1959408 (2003m:57005)}

\bibitem[CE99]{Cartan}
Henri Cartan and Samuel Eilenberg, \emph{Homological algebra}, Princeton
  Landmarks in Mathematics, Princeton University Press, Princeton, NJ, 1999,
  With an appendix by David A. Buchsbaum, Reprint of the 1956 original.
  \MR{MR1731415 (2000h:18022)}

\bibitem[Cro95]{Cromwell}
Peter~R. Cromwell, \emph{Embedding knots and links in an open book. {I}.
  {B}asic properties}, Topology Appl. \textbf{64} (1995), no.~1, 37--58.
  \MR{MR1339757 (96g:57006)}

\bibitem[Dyn06]{Dynnikov}
I.~A. Dynnikov, \emph{Arc-presentations of links: monotonic simplification},
  Fund. Math. \textbf{190} (2006), 29--76. \MR{MR2232855 (2007e:57006)}

\bibitem[Fie93]{FiedlerArt}
Thomas Fiedler, \emph{A small state sum for knots}, Topology \textbf{32}
  (1993), no.~2, 281--294. \MR{MR1217069 (94c:57006)}

\bibitem[Fie01]{FiedlerBook}
\bysame, \emph{Gauss diagram invariants for knots and links}, Mathematics and
  its Applications, vol. 532, Kluwer Academic Publishers, Dordrecht, 2001.
  \MR{MR1948012 (2003m:57031)}

\bibitem[Gal07]{Gallais}
Etienne Gallais, \emph{Sign refinment for combinatorial link floer homology},
  \textsf{arXiv:0706.0089}, 2007.

\bibitem[GM03]{Gelfand}
Sergei~I. Gelfand and Yuri~I. Manin, \emph{Methods of homological algebra},
  second ed., Springer Monographs in Mathematics, Springer-Verlag, Berlin,
  2003. \MR{MR1950475 (2003m:18001)}

\bibitem[HP89]{Hoste}
Jim Hoste and J{\'o}zef~H. Przytycki, \emph{An invariant of dichromatic links},
  Proc. Amer. Math. Soc. \textbf{105} (1989), no.~4, 1003--1007. \MR{MR989100
  (90d:57002)}

\bibitem[Kau83]{Formal}
Louis~H. Kauffman, \emph{Formal knot theory}, Mathematical Notes, vol.~30,
  Princeton University Press, Princeton, NJ, 1983. \MR{MR712133 (85b:57006)}

\bibitem[Kau87]{StateJones}
\bysame, \emph{State models and the {J}ones polynomial}, Topology \textbf{26}
  (1987), no.~3, 395--407. \MR{MR899057 (88f:57006)}

\bibitem[Kau89]{Kauffman}
\bysame, \emph{Invariants of graphs in three-space}, Trans. Amer. Math. Soc.
  \textbf{311} (1989), no.~2, 697--710. \MR{MR946218 (89f:57007)}

\bibitem[Kho00]{Khovanov}
Mikhail Khovanov, \emph{A categorification of the {J}ones polynomial}, Duke
  Math. J. \textbf{101} (2000), no.~3, 359--426. \MR{MR1740682 (2002j:57025)}

\bibitem[Kon93]{Kontsevich}
Maxim Kontsevich, \emph{Vassiliev's knot invariants}, I. M. Gelfand Seminar,
  Adv. Soviet Math., vol.~16, Amer. Math. Soc., Providence, RI, 1993,
  pp.~137--150. \MR{MR1237836 (94k:57014)}

\bibitem[MOS06]{MOS}
Ciprian Manolescu, Peter Ozsv{\'a}th, and Sucharit Sarkar, \emph{A
  combinatorial description of knot floer homology},
  \textsf{arXiv:math.GT/0607691}, 2006.

\bibitem[MOST06]{MOST}
Ciprian Manolescu, Peter Ozsv{\'a}th, Zolt{\'a}n Szab{\'o}, and Dylan Thurston,
  \emph{On combinatorial link floer homology}, \textsf{arXiv:math.GT/0610559},
  2006.

\bibitem[MOY98]{Yamada}
Hitoshi Murakami, Tomotada Ohtsuki, and Shuji Yamada, \emph{Homfly polynomial
  via an invariant of colored plane graphs}, Enseign. Math. (2) \textbf{44}
  (1998), no.~3-4, 325--360. \MR{MR1659228 (2000a:57023)}

\bibitem[OS04a]{OS2}
Peter Ozsv{\'a}th and Zolt{\'a}n Szab{\'o}, \emph{Holomorphic disks and knot
  invariants}, Adv. Math. \textbf{186} (2004), no.~1, 58--116. \MR{MR2065507
  (2005e:57044)}

\bibitem[OS04b]{OS1}
\bysame, \emph{Holomorphic disks and topological invariants for closed
  three-manifolds}, Ann. of Math. (2) \textbf{159} (2004), no.~3, 1027--1158.
  \MR{MR2113019 (2006b:57016)}

\bibitem[OS05]{OS3}
\bysame, \emph{Holomorphic disks and link invariants},
  \textsf{arXiv:math/0512286}, 2005.

\bibitem[OS07a]{OSCube}
\bysame, \emph{A cube of resolutions for knot floer homology},
  \textsf{arXiv:math.GT/0705.3852}, 2007.

\bibitem[OS07b]{OSSkein}
\bysame, \emph{On the skein exact sequence for knot floer homology},
  \textsf{arXiv:math.GT/0707.1165}, 2007.

\bibitem[OSS07]{OSSing}
Peter Ozsv{\'a}th, Andr{\'a}s~I. Stipsicz, and Zolt{\'a}n Szab{\'o},
  \emph{Floer homology and singular knots}, \textsf{arXiv:math.GT/0705.2661},
  2007.

\bibitem[Ras03]{Rasmussen}
Jacob Rasmussen, \emph{Floer homology and knot complements}, Ph.D. thesis,
  Harvard University, 2003.

\bibitem[Rei72]{Reidemeister}
Kurt Reidemeister, \emph{Einf\"uhrung in die kombinatorische {T}opologie},
  Wissenschaftliche Buchgesellschaft, Darmstadt, 1972, Unver\"anderter
  reprografischer Nachdruck der Ausgabe Braunschweig 1951. \MR{MR0345088 (49
  \#9827)}

\bibitem[Shi07]{Shiro}
Nadya Shirokova, \emph{On the classification of floer-type theories},
  \textsf{arXiv:0704.1330}, 2007.

\bibitem[SW07]{ShiroWeb}
Nadya Shirokova and Ben Webster, \emph{Wall-crossing morphisms in
  khovanov-rozansky homology}, \textsf{arXiv:0706.1388}, 2007.

\bibitem[Vas90]{Vassiliev}
V.~A. Vassiliev, \emph{Cohomology of knot spaces}, Theory of singularities and
  its applications, Adv. Soviet Math., vol.~1, Amer. Math. Soc., Providence,
  RI, 1990, pp.~23--69. \MR{MR1089670 (92a:57016)}

\bibitem[Vir02]{Viro}
Oleg Viro, \emph{Remarks on definition of khovanov homology},
  \textsf{arXiv:math.GT/0202199}, 2002.

\bibitem[Wei94]{Weibel}
Charles~A. Weibel, \emph{An introduction to homological algebra}, Cambridge
  Studies in Advanced Mathematics, vol.~38, Cambridge University Press,
  Cambridge, 1994. \MR{MR1269324 (95f:18001)}

\end{thebibliography}
\addcontentsline{toc}{part}{Bibliography}

\end{document}